\documentclass[a4paper,12pt]{article}

\usepackage{amsmath}
\usepackage{amssymb}
\usepackage{amsthm}

\usepackage{amsfonts}
\usepackage{mathscinet}
\usepackage{enumerate}
\usepackage{etaremune}
\usepackage{pdfsync}
\usepackage{setspace}
\usepackage{mathrsfs}
\usepackage{stmaryrd}
\usepackage{tikz}
\usetikzlibrary{arrows,chains,matrix,positioning,scopes}
\usepackage{cite}
 \usepackage{enumitem}
 \usepackage{mathtools}

\usepackage{bookmark}
\usepackage{authblk}

\usepackage{multicol}

\newcommand {\new} {\newcommand}
\newcommand \oper[2] {\new #1 {\operatorname{#2}}}

\newcommand \gcode[1] {\ulcorner\! #1 \!\urcorner}	
\newcommand \se [1] { \{ #1 \}}			
\newcommand\set [2]{ \{#1:#2\} }			
\newcommand\res {\!\upharpoonright\!}		

\newcommand\elem {\prec}				
\newcommand\eqiv {\leftrightarrow}			
\newcommand\corner[1] {
  \langle #1 \rangle}		

\newcommand\power {{\mathscr{P}}}			

\newcommand\coll{{\text{Coll}}}

\newcommand{\card}{\operatorname{card}}

\newcommand\concat {{^\frown}}   

\oper{\Ord}{Ord}				
\oper{\hull}{Hull}				
\oper{\ppt}{ppt} 				
\newcommand{\ord} {\Ord}
\oper{\ZFC}{ZFC}				
\oper{\rank}{rank}				
\oper{\crit}{cr}					
\oper{\crt}{crt}					
\oper{\cf}{cf}					

\oper{\height}{ht}				
\oper{\wfcore}{wfcore}					
\oper{\core}{core}					

\oper{\Ult}{Ult}				
\oper{\Cone}{Cone}				
\oper{\dirlim}{dirlim}
\oper{\rud}{rud}		
\oper{\const}{const}
\oper{\OD}{OD}
\oper{\final}{final}
\oper{\HYP}{HYP}
\oper{\wfp}{wfp}
\oper{\Hom}{Hom}
\new{\ult} {\Ult}
\oper{\dom}{dom}
\oper{\rep}{rep}

\newcommand{\bardom}{\dom^{*}}
\newcommand{\exdesc}{\desc^{*}}
\newcommand{\exexdesc}{\desc^{**}}

\oper{\suc}{succ}
\oper{\fac}{fac}

\oper{\Code}{Code}

\oper{\ran}{ran}
\oper{\maxdom}{maxdom}
\oper{\maxran}{maxran}

\oper{\lp}{Lp} 

\newcommand{\Los}{{\L}o\'{s}}
\newcommand\iniseg {\vartriangleleft}		
\newcommand\inisegeq {\trianglelefteq}	

 \newtheorem{mytheorem}{Theorem}[section]
 
 \newtheorem{mylemma}[mytheorem]{Lemma}

 \newtheorem{mycorollary}[mytheorem]{Corollary}
  \theoremstyle{definition}\newtheorem{myexample}[mytheorem]{Example}
 \newtheorem{myquestion}[mytheorem]{Question}

 \newtheorem{myclaim}[mytheorem]{Claim}
 \theoremstyle{mydefinition}\newtheorem{remark}[mytheorem]{Remark}
 \newtheorem{mydefinition}[mytheorem]{Definition}

\newcommand{\game}{{\Game}}
\newcommand{\boldpi}[1]{{\boldsymbol{\Pi}^1_{#1}}}
\newcommand{\boldsigma}[1]{{\boldsymbol{\Sigma}^1_{#1}}}

\newcommand{\bolddelta}[1]{{\boldsymbol{{\delta}}^1_{#1}}}
\newcommand{\boldDelta}[1]{{\boldsymbol{{\Delta}}^1_{#1}}}
\newcommand{\admistwo}[1]{L_{\kappa_3^{#1}}[T_2,#1]}
\newcommand{\admisfour}[1]{L_{\kappa_5^{#1}}[T_4,#1]}
\newcommand{\admistwobold}{{  \mathbb{L}_{\boldsymbol{{\delta}}^1_3}[T_2] }}

\newcommand{\admistwoboldextra}[1]{{  \mathbb{L}_{\boldsymbol{{\delta}}^1_3}[ {#1} (T_2 ) ]}}

\newcommand{\admistwoboldultratwo}[2]{{  \mathbb{L}_{\boldsymbol{{\delta}}^1_3}(j^{#1}\circ j^{#2}[T_2] ) }}

\newcommand{\WO}{{\textrm{WO}}}
\newcommand{\LO}{{\textrm{LO}}}
\newcommand{\DEF}{=_{{\textrm{DEF}}}}

\newcommand{\desc}{{\operatorname{desc}}}

\newcommand{\Diff}{\operatorname{Diff}}

\newcommand{\id}{\operatorname{id}}
\newcommand{\tree}{\operatorname{tree}}
\newcommand{\node}{\operatorname{node}}
\newcommand{\seed}{\operatorname{seed}}
\newcommand{\pred}{\operatorname{pred}}

\newcommand{\ot}{\mbox{o.t.}}

\newcommand{\sharpcode}[1]{  \left|   #1  \right|}

\newcommand{\absvalue}[1]{  \left|   #1  \right|}

\newcommand{\wocode}[1]{  \|   #1  \|}

\newcommand{\lessthanshort}[1]{{<\! #1}}

\newcommand{\comm}[1]{{}}
\newcommand{\comp}[2]{{}^{#1}\!#2}
\oper{\ucf}{ucf}					
\oper{\sign}{sign}					
\oper{\lh}{lh}
\oper{\com}{com}

\newcommand{\ooo}{^{(3)}_{\omega^{\omega^{\omega}}}}

\begin{document}

\title{The higher sharp III: an EM blueprint of $0^{3\#}$ and the level-4 Kechris-Martin}

\author{Yizheng Zhu}

\affil{Institut f\"{u}r mathematische Logik und Grundlagenforschung \\
Fachbereich Mathematik und Informatik\\
Universit\"{a}t M\"{u}nster\\
Einsteinstr. 62 \\
48149 M\"{u}nster, Germany}

\maketitle

\begin{abstract}
We establish the descriptive set theoretic representation of the mouse $M_n^{\#}$, which is called $0^{(n+1)\#}$. This part partially finishes the case $n=2$ by establishing the higher level analog of the EM blueprint definition of $0^{\#}$. From this, we prove the level-4 Kechris-Martin Theorem and deal with the case $n=3$.
\end{abstract}

\section{Introduction}
\label{sec:introduction}

This is the third part of a series starting with \cite{sharpI}. It gives an EM blueprint axiomatization of $0^{3\#}$, which is defined in \cite{sharpII}. We know that the existence of $0^{\#}$ is a large cardinal axiom. The same works for $0^{3\#}$: the existence of $0^{3\#}$ is a purely syntactical large cardinal axiom under the base theory  of  $\boldDelta{2}$-determinacy. From the EM blueprint formulation of $0^{3\#}$, we define the level-4 Martin-Solovay tree $T_4$ and prove the level-4 Kechris-Martin theorem. 

\section{A guide to Jackson's analysis}
\label{sec:guide-jacks-analys}

The main ingredient of the axiomatization of $0^{3\#}$ is a generalization of Jackson's analysis \cite{jackson_delta15,jackson_proj_ord,cabal2_intro_jackson}. This generalization clarifies the main idea of Jackson's analysis. The idea of the axiomatization of $0^{3\#}$ arises from the interplay between embeddings of mice and ultrapower maps of measures on $\bolddelta{3}$. The focus is not on the cardinality of $\bolddelta{5}$, but rather on the set-theoretic structures of different ultrapowers. 

We try to spread the idea of the generalization of Jackson's analysis which probably has not appeared in the literature. 
 In one sentence,
\begin{quote}
  Jackson's analysis is the analysis of equivalence classes of continuous, order-preserving embeddings between wellorders. 
\end{quote}
This paper continues on this track towards a translation between
\begin{center}
  continuous, order-preserving embeddings between wellorders
\end{center}
and
\begin{center}
  embeddings between mice. 
\end{center}

We shall explain with examples. Assume AD throughout this section. 
Fix some notations throughout this section:
\begin{itemize}
\item If $Z$ is finite a level-1, level $\leq 2$ or level-3 tree, then
  \begin{displaymath}
    \theta_Z : \rep(Z) \to \ot(\rep(Z)) 
  \end{displaymath}
is the associated order-preserving bijection. 
\end{itemize}

\subsection{Level-(1,2,1) factoring maps}
\label{sec:level-1-2}

Let $S = \se{(3),(3,0),(3,0,0),(2),(1),(1,2),(1,2,0),(1,1),(1,0),(0)}$ be a level-1 tree. Let $Q$ be a level $\leq 2$ tree where $\comp{1}{Q} = \emptyset$, $\dom(\comp{2}{Q}) = \se{\emptyset, ((1)),  ((0)), ((0),(0))}$, $\comp{2}{Q}(((1)))= \comp{2}{Q}(((0)))= (\se{(0)}, (0,0))$, $\comp{2}{Q}(((0),(0))) = (\se{(0),(0,0)}, -1)$. Then $\rep(S)$ consists of $(s)$ and $(s,k)$ for $s \in S$ and $k<\omega$. $\rep(Q)$ consists of:
\begin{itemize}
\item $(2,\emptyset)$,
\item $(2,(\alpha, (1)))$ for $\alpha < \omega_1$ limit,
\item $(2,(\alpha, (1), \beta, -1))$ for $\alpha < \omega_1$ limit and $\beta < \alpha$. 
\item $(2,(\alpha, (0)))$ for $\alpha < \omega_1$ limit,
\item $(2,(\alpha, (0), \beta, (0)))$ for $\beta < \alpha < \omega_1$ limits,
\item $(2,(\alpha, (0), \beta, (0), k, -1))$ for $\beta < \alpha < \omega_1$ limits and $k<\omega$,
\item $(2,(\alpha, (0), \beta, -1))$ for $\alpha < \omega_1$ limit and $\beta < \alpha$,
\item $(2,(\alpha, -1))$ for $\alpha < \omega_1$,
\end{itemize}
We would like to characterize the behaviour of
\begin{displaymath}
 \theta_{SQ} =  (\theta_Q)^{-1} \circ \theta_S : \rep(S) \to \rep(Q)
\end{displaymath}
up to an equivalence. We say that two \emph{continuous, order-preserving} maps $\theta,\theta' : \rep(S) \to \rep(Q)$ are $S$-equivalent iff $\theta(s) = \theta'(s)$ for any $s \in S$. Of course this equivalence relation is basically trivial, the equivalence classes depending only on $(\theta(s))_{s \in S}$, but its higher level analogs will  become meaningful. Then a continuous, order-preserving $\theta$ is $S$-equivalent to $\theta_{SQ}$ iff
\begin{itemize}
\item $\theta((3)) = (2,(\omega 3, -1 ))$.
\item $\theta((3,0)) = (2,(\omega 2, (1)))$,
\item $\theta((3,0,0)) = (2,(\omega 2, (1), \omega, -1))$, 
\item $\theta((2)) = (2,(\omega 2, (0)))$,
\item $\theta((1)) = (2,(\omega 2, (0), \omega, (0)))$,
\item $\theta((1,2)) = (2,(\omega 2, (0), \omega, -1))$,
\item $\theta((1,2,0)) = (2,(\omega 2,-1))$,
\item $\theta((1,1)) = (2,(\omega, (1)))$,
\item $\theta((1,0)) = (2,(\omega, (0)))$,
\item $\theta((0)) = (2,(\omega, -1))$,
\end{itemize}
%
Collect all the ordinals showing up in this list: $\omega 3, \omega 2, \omega$. Compose a level-1 tree $W = \se{(2),(1),(0)}$ that is respected by these ordinals: $\vec{\alpha}$ respects $W$, where $\alpha_{(2)}= \omega 3$, $\alpha_{(1)} = \omega 2$, $\alpha_{(0)} = \omega$. 
%
Now the $S$-equivalence class of $\theta_{SQ}$ is decided by $W$ and $\vec{\alpha}$: $\theta$ is $S$-equivalent to $\theta_{SQ}$ iff $\theta$ is continuous, order-preserving and 
\begin{itemize}
\item $\theta((3)) = (2,(\alpha_{(2)}, -1 ))$.
\item $\theta((3,0)) = (2,(\alpha_{(1)}, (1)))$,
\item $\theta((3,0,0)) = (2,(\alpha_{(1)}, (1), \alpha_{(0)}, -1))$, 
\item $\theta((2)) = (2,(\alpha_{(1)}, (0)))$,
\item $\theta((1)) = (2,(\alpha_{(1)}, (0), \alpha_{(0)}, (0)))$,
\item $\theta((1,2)) = (2,(\alpha_{(1)}, (0), \alpha_{(0)}, -1))$,
\item $\theta((1,2,0)) = (2,(\alpha_{(1)},-1))$,
\item $\theta((1,1)) = (2,(\alpha_{(0)}, (1)))$,
\item $\theta((1,0)) = (2,(\alpha_{(0)}, (0)))$,
\item $\theta((0)) = (2,(\alpha_{(0)}, -1))$,
\end{itemize}
This is a combinatorial characterization of the $S$-equivalence class of $\theta_{SQ}$ depending only on $W$ and $\vec{\alpha}$. 
If $\vec{\alpha}'$ is replaced by another tuple $\vec{\alpha}'$ respecting $W$, the same expressions will define another $S$-equivalence class. 
For instance, if $\alpha'_{(0)} = \omega 8$, $\alpha'_{(1)} = \omega 9$, $\alpha'_{(2)} = \omega^2$, a member of the $S$-equivalence class decided by $\vec{\alpha}'$ is $\theta'$, where $\theta'$ is continuous, order-preserving and 
\begin{itemize}
\item $\theta((3)) = (2,(\omega^2, -1 ))$.
\item $\theta((3,0)) = (2,(\omega 9, (1)))$,
\item $\theta((3,0,0)) = (2,(\omega 9, (1), \omega 8, -1))$, 
\item $\theta((2)) = (2,(\omega 9, (0)))$,
\item $\theta((1)) = (2,(\omega 9, (0), \omega 8, (0)))$,
\item $\theta((1,2)) = (2,(\omega 9, (0), \omega 8, -1))$,
\item $\theta((1,2,0)) = (2,(\omega 9,-1))$,
\item $\theta((1,1)) = (2,(\omega 8, (1)))$,
\item $\theta((1,0)) = (2,(\omega 8, (0)))$,
\item $\theta((0)) = (2,(\omega 8, -1))$,
\end{itemize}
This template of generating $S$-equivalence classes depending on a tuple respecting $W$ will 
%
be abstracted into a finitary object
\begin{displaymath}
\tau
\end{displaymath}
which factors $(S,Q,W)$, to be defined in Section~\ref{sec:level-2-description}. 
 The right hand side of each equation will be abstracted into a finitary object called a $(Q,W)$-description, to be defined Section~\ref{sec:level-2-description}. 
In terms of ultrapowers, the main property of $\tau$ is that it induces a map
\begin{displaymath}
\tau^{Q,W} : \mathbb{L}_{j^S(\omega_1)} \to \mathbb{L}_{j^Q \circ j^W(\omega_1)}
\end{displaymath}
such that $\tau^{QW} \circ j^S = j^Q \circ j^W$ and $\tau^{QW}$ is elementary on any submodel of ZFC. 


For illustration purposes, the tensor product of trees, $Q \otimes W$, to be defined in Section~\ref{sec:level-2-description}, is the level-1 tree $U$ 
(up to an isomorphism) such that $\theta_{UQ}$ is universal 
in the sense that
\begin{enumerate}
\item the $U$-equivalence class of $\theta_{UQ}$ has a combinatorial characterization decided by $(W, \vec{\alpha},\tau^{*})$, where $\vec{\alpha}$ respects $W$, $\tau^{*}$ factors $(U,Q,W)$, and
\item if $U'$ is another level-1 tree, $\theta' : \rep(U') \to \rep(Q)$ is continuous, order-preserving, and the $U'$-equivalence class of $\theta'$ has a combinatorial characterization decided by $(W, \vec{\alpha}', \tau')$, where $\vec{\alpha}'$ respects $W$, $\tau'$ factors $(U',Q,W)$, $\theta''$ a representative of the $U$-equivalence class decided by $(W, \vec{\alpha}', \tau^{*})$, then there is a level-1 tree factoring map $\psi$ of $U'$ into $U$ such that $\theta' (s) = \theta''(\psi(s))$ for any $s \in U'$. 
\end{enumerate}
We describe the algorithm of producing $U$. Write down all the different forms of expressions
\begin{displaymath}
  (\alpha_{w_0},  z_0,  \dots, \alpha_{w_k}, z_k)
\end{displaymath}
that belongs to $\rep(Q)$ for some $w_0,\dots, w_k \in W$ and $\vec{\alpha} = (\alpha_w)_{w \in W}$ respecting $W$. Order these expressions according to their positions in $\rep(Q)$. They are: 
\begin{itemize}
\begin{multicols}{2}
\item $ (2,(\alpha_{(2)}, (1)))$,
\item $ (2,(\alpha_{(2)}, (1), \alpha_{(1)}, -1))$,
\item $ (2,(\alpha_{(2)}, (1), \alpha_{(0)}, -1))$,
\item $ (2,(\alpha_{(2)}, (0), \alpha_{(1)}, (0)))$,
\item $ (2,(\alpha_{(2)}, (0), \alpha_{(1)}, -1))$,
\item $ (2,(\alpha_{(2)}, (0), \alpha_{(0)}, (0)))$,
\item $ (2,(\alpha_{(2)}, (0), \alpha_{(0)}, -1 ))$,
\item $ (2,(\alpha_{(2)}, (0)))$, 
\item $ (2,(\alpha_{(2)}, -1))$,
\item $ (2,(\alpha_{(1)}, (1)))$,
\item $ (2,(\alpha_{(1)}, (1), \alpha_{(0)}, -1))$,
\item $ (2,(\alpha_{(1)}, (0)))$,
\item $ (2,(\alpha_{(1)}, (0), \alpha_{(0)}, (0)))$,
\item $ (2,(\alpha_{(1)}, (0), \alpha_{(0)}, -1))$,
\item $ (2,(\alpha_{(1)}, -1))$, 
\item $ (2,(\alpha_{(0)}, (1)))$,
\item $ (2,(\alpha_{(0)}, (0)))$,
\item $ (2,(\alpha_{(0)}, -1))$,
\end{multicols}
\end{itemize}
There are 18 expressions. We let $U = \se{(0), \dots, (17)}$ be a level-1 tree of cardinality 18. Let $\alpha_{(i)} = \omega(i+1)$. Then $\theta$ is $U$-equivalent to $\theta_{UQ}$ iff
\begin{itemize}
\begin{multicols}{2}
\item $\theta((17)) =  (2,(\alpha_{(2)}, (1)))$,
\item $\theta((16)) =  (2,(\alpha_{(2)}, (1), \alpha_{(1)}, -1))$,
\item $\theta((15)) =  (2,(\alpha_{(2)}, (1), \alpha_{(0)}, -1))$,
\item $\theta((14)) =  (2,(\alpha_{(2)}, (0), \alpha_{(1)}, (0)))$,
\item $\theta((13)) =  (2,(\alpha_{(2)}, (0), \alpha_{(1)}, -1))$,
\item $\theta((12)) =  (2,(\alpha_{(2)}, (0), \alpha_{(0)}, (0)))$,
\item $\theta((11)) =  (2,(\alpha_{(2)}, (0), \alpha_{(0)}, -1 ))$,
\item $\theta((10)) =  (2,(\alpha_{(2)}, (0)))$, 
\item $\theta((9)) =  (2,(\alpha_{(2)}, -1))$,
\item $\theta((8)) =  (2,(\alpha_{(1)}, (1)))$,
\item $\theta((7)) =  (2,(\alpha_{(1)}, (1), \alpha_{(0)}, -1))$,
\item $\theta((6)) =  (2,(\alpha_{(1)}, (0)))$,
\item $\theta((5)) =  (2,(\alpha_{(1)}, (0), \alpha_{(0)}, (0)))$,
\item $\theta((4)) =  (2,(\alpha_{(1)}, (0), \alpha_{(0)}, -1))$,
\item $\theta((3)) =  (2,(\alpha_{(1)}, -1))$, 
\item $\theta((2)) =  (2,(\alpha_{(0)}, (1)))$,
\item $\theta((1)) =  (2,(\alpha_{(0)}, (0)))$,
\item $\theta((0)) =  (2,(\alpha_{(0)}, -1))$.
\end{multicols}
\end{itemize}
It is easy to verify clause (2) for this $U$. Apply universality to $S$, $\theta_{SQ}$, $\vec{\alpha}$ and $\tau$. The resulting factoring map $\psi$ of $S$ into $U$ is 
\begin{itemize}
  \begin{multicols}{3}
  \item $\psi((3)) = (9)$. 
  \item $\psi((3,0)) = (8)$,
  \item $\psi((3,0,0)) = (7)$,
  \item $\psi((2)) = (6)$,
  \item $\psi((1)) = (5)$,
  \item $\psi((1,2)) = (4)$,
  \item $\psi((1,2,0)) = (3)$,
  \item $\psi((1,1)) = (2)$,
  \item $\psi((1,0)) = (1)$,
  \item $\psi((0)) = (0)$,
  \end{multicols}
\end{itemize}

In Section~\ref{sec:level-2-description}, we will show that the equality
\begin{displaymath}
  U = Q \otimes W
\end{displaymath}
corresponds to the equality
\begin{displaymath}
  \mathbb{L}_{j^U(\omega_1)} = \mathbb{L}_{j^Q\circ j^W(\omega_1)}.
\end{displaymath}
For $U', \theta', \vec{\alpha}', \tau' , \psi$ as in (2), there is an induced map between ultrapowers
\begin{displaymath}
  \psi^U : \mathbb{L}_{j^{U'}(\omega_1)} \to \mathbb{L}_{j^U(\omega_1)},
\end{displaymath}
and we will show that
\begin{displaymath}
  \psi^U = (\tau')^{Q,W}.
\end{displaymath}


To summarize, we have given an example of a general fact on level-(1,2,1) factoring maps. Given a finite level-1 tree $S$ and a finite level-2 tree $Q$, there is another level-1 tree $W$, a tuple of ordinals $\vec{\alpha}$ respecting $W$ such that the $S$-equivalence class of $\theta_{SQ}$ has a purely combinatorial characterization, called a map $\tau$ which factors $(S,Q,W)$, and $S$ embeds into $U = Q \otimes W$ via a level-1 tree factoring map $\psi$. 
%
In terms of ultrapowers, $\tau$ induces $\tau^{Q,W} : \mathbb{L}_{j^S(\omega_1)} \to \mathbb{L}_{j^Q \circ j^W(\omega_1)} = \mathbb{L}_{j^U(\omega_1)}$ such that $\tau^{Q,W} \circ j^S = j^Q \circ j^W = j^U$, and $\tau^{Q,W}$ is just the embedding induced by the measure projection of $\mu^U $ to $\mu^S$ which is generated by $\psi$. 

 \subsection{Level-(2,2,2) factoring maps}
\label{sec:level-2-2}

 Let $X$ be a level $\leq 2$ tree where $\comp{1}{X} = \se{(0),(1)}$, 
$\comp{2}{X}$ has domain $\se{\emptyset, ((0)), ((0),(0)), ((0),(1))}$ and $\comp{2}{X}(((0))) = \comp{2}{T}(((1)))$, $\comp{2}{X}(((0),(0))) = \comp{2}{X}(((0),(1)))=(\se{(0),(0,0)}, (0,0,0))$. $\rep(X)$ consists of
\begin{itemize}
\item $(2,\emptyset)$,
\item $(2,(\alpha_{(0)}, (0)))$ for $\vec{\alpha}$ respecting $\se{(0)}$, 
\item $(2,(\alpha_{(0)}, (0), \alpha_{(0,0)}, (1)))$ for $\vec{\alpha}$ respecting $\se{(0),(0,0)}$,
\item $(2,(\alpha_{(0)}, (0),  \alpha_{(0,0)}, (1), \beta, -1))$ for $\vec{\alpha}$ respecting $\se{(0),(0,0)}$ and $\beta < \alpha_{(0,0)}$ limit,
\item $(2,(\alpha_{(0)}, (0),  \alpha_{(0,0)}, (0)))$ for $\vec{\alpha}$ respecting $\se{(0),(0,0)}$,
\item $(2,(\alpha_{(0)}, (0),  \alpha_{(0,0)}, (0), \beta, -1))$ for $\vec{\alpha}$ respecting $\se{(0),(0,0)}$ and $\beta < \alpha_{(0,0)}$ limit,
\item $(2,(\alpha_{(0)}, (0), \beta, -1 ))$ for $\vec{\alpha}$ respecting $\se{(0)}$, and $\beta < \alpha_{(0)}$ limit, 
\item $(2,(\beta, -1))$ for $\beta < \omega_1$ limit, 
\item $(1, ((1)))$,
\item $(1, ((1), k))$ for $k<\omega$,
\item $(1, ((0)))$, and
\item $(1, ((0), k))$ for $k<\omega$.
\end{itemize}
If $C$ is a club in $\omega_1$, recall that  $\rep(X) \res C$ is the subset of $\rep(X)$ consisting of all the tuples in $\rep(X)$ whose ordinal entries are in $C \cup \omega$. 
We say that a continuous, order preserving $\theta$ is $X$-equivalent to $\theta'$ iff there is a club $D$ in $\omega_1$ such that
\begin{displaymath}
  \theta \res (\rep(X) \res D ) =  \theta' \res (\rep(X) \res D ).
\end{displaymath}
The $X$-equivalence class of $\theta_X$ can be characterized by: $\theta$ is $X$-equivalent to $\theta_X$ iff there is a club $D \subseteq \omega_1$ such that 
\begin{itemize}
\item $\theta(2,(\alpha_{(0)}, ((0)) )) = \alpha_{(0)} 2 $, for $\vec{\alpha} \in [D]^{\se{(0)}\uparrow}$, i.e.\ $\alpha_{(0)} \in D$, 
\item $\theta(2,(\alpha_{(0)}, ((0)), \alpha_{(0,0)}, ((0)) )) = \alpha_{(0)}+\alpha_{(0,0)} 2$, for $\vec{\alpha} \in [D]^{\se{(0)}\uparrow}$, i.e.\ $\alpha_{(0,0)}<\alpha_{(0)}$, both in $D$, 
\item $\theta(2,(\alpha_{(0)}, ((0)), \alpha_{(0,0)}, ((0)) )) = \alpha_{(0)}+\alpha_{(0,0)} $, for $\vec{\alpha} \in [D]^{\se{(0)}\uparrow}$, i.e.\ $\alpha_{(0,0)}<\alpha_{(0)}$, both in $D$, 
\item $\theta(1, ((1))) = \omega 2$.
\item $\theta(1, ((0))) = \omega$.
\end{itemize}
The ordinals showing up in the above list forms a tuple that respects $X$: $\vec{\xi}$ respects $X$, where 
$\comp{2}{\xi}_{((0))} = \omega_1 2 = [\vec{\alpha} \mapsto \alpha_{(0)} 2]_{\mu^{\se{(0)}}}$, 
$\comp{2}{\xi}_{((0),(1))} = \omega_2+\omega_1 2 = [\vec{\alpha} \mapsto \alpha_{(0)}+\alpha_{(0,0)}2]_{\mu^{\se{(0),(0,0)}}}$, 
$\comp{2}{\xi}_{((0,(0)))} = \omega_2+\omega_1  = [\vec{\alpha} \mapsto \alpha_{(0)}+\alpha_{(0,0)}]_{\mu^{\se{(0),(0,0)}}}$, 
$\comp{1}{\xi}_{(1)} = \omega$,
$\comp{1}{\xi}_{(0)} = \omega$. (Recall that $\se{(0)}$, $\se{(0),(0,0)}$  are level-1 trees and $\mu^{\se{(0)}}$ is the club measure on $\omega_1$, $\mu^{\se{(0),(0,0)}}$ is the two-fold product measure of $\mu^{\se{(0)}}$.)

So far, everything is basically a restatement of the theory of level $\leq 2$ trees and their representations. 
We now introduce another tree $T$ and investigate the embeddings between $\rep(X)$ and $\rep(T)$.

Let $T$ be a level $\leq 2$ tree where $\comp{1}{T} = \se{(0)}$, $\comp{2}{T}$ has domain $\se{\emptyset, ((1)), ((1),(0)), ((0))}$ and $\comp{2}{T}(((1))) =  (\se{(0)}, (0,0))$, $\comp{2}{T}((1),(0)) = (\se{(0),(0,0)}, (0,1))$, $\comp{2}{T}((0)) = (\se{(0)}, -1)$. $\rep(T)$ consists of
\begin{itemize}
\item $(2,\emptyset)$,
\item $(2,(\alpha_{(0)}, (1)))$ for $\vec{\alpha} $ respecting $\se{(0)}$, i.e.\ $\alpha_{(0)} < \omega_1$ is a limit ordinal, 
\item $(2,(\alpha_{(0)}, (1), \alpha_{(0,0)}, (0)))$ for $\vec{\alpha}$ respecting $\se{(0),(0,0)}$, i.e. $\alpha_{(0,0)}<\alpha_{(0)} < \omega_1$ are limit ordinals, 
\item $(2,(\alpha_{(0)}, (1), \alpha_{(0,0)}, (0), \beta, -1))$ for $\vec{\alpha}$ respecting $\se{(0),(0,0)}$ and $\beta < \alpha_{(0)}$ limit,
\item $(2,(\alpha_{(0)}, (1), \beta, -1))$ for $\vec{\alpha} $ respecting $\se{(0)}$ and $\beta<\alpha_{(0)}$ limit,
\item $(2,(\alpha, (0)))$ for $\vec{\alpha} $ respecting $\se{(0)}$,
\item $(2,(\alpha, (0), k, -1))$ for $\vec{\alpha} $ respecting $\se{(0)}$ and $k<\omega$,
\item $(2,(\beta, -1))$ for $\beta < \omega_1$ limit, 
\item $(1, ((0)))$, and
\item $(1, ((0),k))$ for $k<\omega$.
\end{itemize}
Let
\begin{displaymath}
  \theta_{XT} : \rep(X) \to \rep(T)
\end{displaymath}
be the order-preserving bijection. There is a concrete way of characterizing the $X$-equivalence class of $\theta_{XT}$: A continuous, order-preserving $\theta$ is $X$-equivalent to $\theta_{XT}$ iff there is a club $D \subseteq \omega_1$ such that
\begin{itemize}
\item $\theta(2,(\alpha_{(0)}, (0))) = (2, (\alpha_{(0)}, (1), \omega, (0)))$ for $\vec{\alpha} \in [D]^{\se{(0)}\uparrow}$,
\item $\theta(2,(\alpha_{(0)}, (0), \alpha_{(0,0)}, (1))) = (2, (\alpha_{(0)}, (1), \omega, (0), \alpha_{(0,0)} 2, -1))$ for $\vec{\alpha} \in [D]^{\se{(0),(0,0)}\uparrow}$,
\item $\theta(2,(\alpha_{(0)}, (0), \alpha_{(0,0)}, (0))) = (2, (\alpha_{(0)}, (1), \omega, (0), \alpha_{(0,0)} , -1))$ for $\vec{\alpha} \in [D]^{\se{(0),(0,0)}\uparrow}$,
\item $\theta (1, ((1))) = (2, (\omega, -1))$, and 
\item $\theta(1, ((0))) = (1, ((0)))$,
\end{itemize}
In general, if
\begin{displaymath}
  \theta' : \rep(X) \to \rep(T)
\end{displaymath}
is an \emph{order-preserving, continuous} map, the $X$-equivalence class of $\theta'$ is uniquely decided by another finite level-2 tree $Q'$, a finite tuple of ordinals $\vec{\gamma}'$ respecting $Q'$ and  a finite combinatorial factoring map $\pi$. The general form is Lemma~\ref{lem:XT_division_Q} and Theorem~\ref{thm:factor_ordertype_embed_equivalent_lv2}. 

Return to our example. 
In the characterization of the $X$-equivalence class of $\theta_{XT}$, the ordinals 
 $\alpha$,
 $\omega$,
 $\beta 2$, and
 $\beta$
show up. The actual information about the occurrences of these ordinals are functions
\begin{itemize}
\item $\vec{\alpha} \mapsto \alpha_{(0)}$ for $\vec{\alpha} \in [D]^{\se{(0)}\uparrow}$,
\item $\vec{\alpha} \mapsto \omega$ for $\vec{\alpha} \in [D]^{\se{(0)}\uparrow}$,
\item $\vec{\alpha} \mapsto \alpha_{(0)}$ for $\vec{\alpha} \in [D]^{\se{(0),(0,0)}\uparrow}$,
\item $\vec{\alpha} \mapsto \omega$ for $\vec{\alpha} \in [D]^{\se{(0),(0,0)}\uparrow}$,
\item $\vec{\alpha} \mapsto \alpha_{(0,0)} 2$ for $\vec{\alpha} \in [D]^{\se{(0),(0,0)}\uparrow}$, 
\item $\vec{\alpha}\mapsto \alpha_{(0,0)}$ for $\vec{\alpha} \in [D]^{\se{(0),(0,0)}\uparrow}$, and
\item $* \mapsto \omega$.
\end{itemize}
Throw away the redundant coordinates in the variables of these functions and remove duplicates. The remaining is
\begin{itemize}
\item $\alpha_{(0)} \mapsto \alpha_{(0)}$,
\item $\alpha_{(0,0)} \mapsto \alpha_{(0,0)} 2$,
\item $* \mapsto \omega$.
\end{itemize}
The function $\alpha_{(0)} \mapsto \alpha_{(0)}$ is of continuous type, so we remove the continuous coordinate $\alpha_{(0)}$ by taking sup over it. The resulting function becomes 
$* \mapsto \sup_{\alpha_{(0)}<\omega_1} \alpha_{(0)} = \omega_1$.

These functions generate a tuple that respects a level $\leq 2$ tree. Let $Q$ be a level $\leq 2$ tree where $\comp{1}{Q} = \se{(0)}$, $\comp{2}{Q}$ has domain $\se{\emptyset, ((0))}$, $\comp{2}{Q}(((0))) = (\se{(0)}, (0,0))$. $\vec{\xi}$ respects $Q$ where $\comp{1}{\xi}_{(0)} = \omega$, $\comp{2}{\xi}_{\emptyset} = \omega_1$, $\comp{2}{\xi}_{((0))} = \omega_1 2 = [\vec{\alpha} \mapsto \alpha_{(0)} 2]_{\mu^{\se{(0)}}}$. 

Now the $X$-equivalence class of $\theta_{XT}$ is decided by $Q$ and $\vec{\xi}$. Fix $g \in [\omega_1]^{Q \uparrow}$ with $[g]^Q = \vec{\xi}$. Then a continuous, order-preserving $\theta$ is $X$-equivalent to $\theta_{XT}$ iff there is a club $C$ in $\omega_1$ such that 
\begin{itemize}
\item $ \theta(2, (\alpha_{(0)},(0))) =(2, (   g(2,   (   \alpha_{(0)}   , -1))     , (1),   g(1,(0)) ,  (0)  ))$ for $\vec{\alpha}\in [C]^{\se{(0)} \uparrow}$,
\item $ \theta(2, (\alpha_{(0)},(0), \alpha_{(0,0)}, (1))) =(2, (   g(2,   (   \alpha_{(0)}   , -1))     , (1),   g(1,(0)) ,  (0) ,  g(2, (\alpha_{(0)}, 0))  ,-1 ))$ for $\vec{\alpha}\in [C]^{\se{(0)} \uparrow}$,
\item $ \theta(2, (\alpha_{(0)},(0), \alpha_{(0,0)}, (0))) =(2, (   g(2,   (   \alpha_{(0)}   , -1))     , (1),   g(1,(0)) ,  (0) ,  g(2, (\alpha_{(0,0)}, -1))  ,-1 ))$ for $\vec{\alpha}\in [C]^{\se{(0)} \uparrow}$,
\item $\theta(1, ((1))) = (2, ( g(1,(0)) ,-1))$, and
\item $\theta (1, ((0))) = (1, ((0)))$. 
\end{itemize}
This is a combinatorial characterization of the $X$-equivalence class of $\theta_{XT}$ depending only on $Q$ and $\vec{\xi}$. This template of generating $X$-equivalence classes depending on a tuple respecting $Q$ will be abstracted into a finitary object
\begin{displaymath}
\pi
\end{displaymath}
which factors $(X,T,Q)$, 
%
to be defined in Section~\ref{sec:more-level-2}. The main property of $\pi$ is that it induces a map
\begin{displaymath}
  \pi^{T,Q} : \mathbb{L}_{\bolddelta{3}}[j^X(T_2)] \to \mathbb{L}_{\bolddelta{3}}[j^T \circ j^Q(T_2)]
\end{displaymath}
such that $\pi^{T,Q} \circ j^X = j^T \circ j^Q$ and $\pi^{T,Q}$ is elementary on any submodel of ZFC.

An important ingredient in this paper is a simplification to the original Jackson's analysis. It merges the iteration of two ultrapowers into a single ultrapower by  defining the tensor product of trees. For illustration purposes, the tensor product, $T \otimes Q$, will be a tree $U$ (up to an isomorphism) such that $\theta_{UT} = \theta_T^{-1} \circ \theta_U $ is universal in the sense that
\begin{enumerate}
\item the $U$-equivalence class of $\theta_{UT} $ has a combinatorial characterization decided by $(Q, \vec{\xi}, \pi^{*})$, where $\vec{\xi}$ respects $Q$, $\tau^{*}$ factors $(U,T,Q)$, and
\item given $(U',\theta',\vec{\xi}',\pi')$, where $U'$ is another level $\leq 2$ tree, $\theta' : \rep(U') \to \rep(T)$ is continuous, order-preserving, and the $U'$-equivalence class of $\theta'$ has a combinatorial characterization decided by $(X, \vec{\xi}', \pi')$, where $\vec{\xi}'$ respects $Q$, $\pi'$ factors $(U',T,Q)$, $\theta'$ is a representative of the $U$-equivalence class decided by $(Q, \vec{\xi}', \pi')$, then there is a level $\leq 2$ tree factoring map $\psi$ of $U'$ into $U$ such that
  \begin{displaymath}
    \theta' (1, (s)) = \theta(1,(\comp{1}{\psi}(s))) 
  \end{displaymath}
for any $s \in \comp{1}{U'}$ and 
  \begin{displaymath}
    [\vec{\alpha} \mapsto \theta'(2, \vec{\alpha}  \oplus_{\comp{2}{U'}} s)]_{\comp{2}{U'}_{\tree}(s)}    =    [\vec{\alpha} \mapsto \theta(2, \vec{\alpha}  \oplus_{\comp{2}{U}} \comp{2}{\psi}(s))]_{\comp{2}{U}_{\tree}(s)}
  \end{displaymath}
for any $s \in \dom(\comp{2}{U}')$. 
\end{enumerate}
We describe the algorithm of producing $U$. Take an arbitrary $g \in [\omega_1]^{Q \uparrow}$. Observe that in the characterization of the $X$-equivalent class of $\theta_{XT}$, the right hand side of the equations are always of the form
\begin{displaymath}
  (2, (g ( f_0(\vec{\alpha})), t(0), \dots, g(f_k(\vec{\alpha})),t(k)))
\end{displaymath}
where $\vec{\alpha} = (\alpha_w)_{w \in W}$ for some level-1 tree $W$, 
$f_0(\vec{\alpha}),\dots,f_k(\vec{\alpha})$ are functions into $\rep(T)$ so that the ordinal entries of $f_i(\vec{\alpha})$ are always some $\alpha_w$, 
$t = (t(0),\dots,t(k)) \in \dom^{*}(\comp{2}{T})$, or
\begin{displaymath}
  (1, (t))
\end{displaymath}
where $t \in \comp{1}{T}$. 
We now write down all the possible expressions of such kind modulo a level-1 tree isomorphism of $W$, with the obvious restrictions:
\begin{enumerate}
\item The whole expression must be in $\rep(T)$, i.e., the ordering of $(f_0(\vec{\alpha}),\dots,f_0(\vec{\alpha}))$ is fixed according to $T$ and $(t(0),\dots,t(k))$.
\item For every $w \in W$, $\alpha_w$ must show up at least once in the expression of $(f_0(\vec{\alpha}),\dots,f_0(\vec{\alpha}))$. Otherwise we would throw away the unused coordinates.
\item The whole expression must be of discontinuous type, in the sense of \cite{sharpII}. That is, there is a club $C$ in $\omega_1$ such that for $\vec{\alpha} \in [C]^{W \uparrow}$,
  \begin{displaymath}
     (2, (g ( f_0(\vec{\alpha})), t(0), \dots, g(f_k(\vec{\alpha})),t(k))) 
  \end{displaymath}
is bigger than the $<^T$ sup of 
\begin{displaymath}
  (2, (g ( f_0(\vec{\beta})), t(0), \dots, g(f_k(\vec{\beta})),t(k))) 
\end{displaymath}
that are $<^T (2, (g ( f_0(\vec{\alpha})), t(0), \dots, g(f_k(\vec{\alpha})),t(k))) $.
Otherwise, one of the coordinates $w_{*}$ would be redundant information, and the information in the whole expression would be equivalent to the following expression depending only on $\vec{\alpha} \in [\omega_1]^{W \setminus \se{w_{*}}\uparrow}$:
\begin{displaymath}
 \sup_{<_T}  \set{(2, (g ( f_0(\vec{\beta})), t(0), \dots, g(f_k(\vec{\beta})),t(k)))}{\vec{\beta} \in [\omega_1]^{W \uparrow}\text{ extends } \vec{\alpha} }
\end{displaymath}
\end{enumerate}

These expressions naturally form a level $\leq 2$ tree $U$. 
An expression of the form $(1,(t))$ corresponds to a node $\tilde{t} \in \comp{1}{U}$.
An expression of the form $h(\vec{\alpha}) = (2,(g(f_0(\vec{\alpha})),t(0),\dots))$ corresponds to a node $\tilde{h} \in\dom(\comp{2}{U})$. The value of $\comp{2}{U}(\tilde{h})$ 
is decided by the uniform cofinality of $h$. 
If $h(\vec{\alpha})$ is expression where $\vec{\alpha}$ ranges over $[\omega_1]^{W \uparrow}$ and $h'(\vec{\alpha})$ is another expression where $\vec{\alpha}$ ranges over $[\omega_1]^{W' \uparrow}$, $W'$ is a one-node extension of $W$,  then
$\tilde{h'}$ is a one-node extension of $\tilde{h}$ iff there is a club $C$ in $\omega_1$ such that for any $\vec{\alpha} \in [C]^{W \uparrow}$,
\begin{displaymath}
  h(\vec{\alpha}) = \sup_{<^T} \set{h'(\vec{\alpha}')}{\vec{\alpha}' \text{ extends }\vec{\alpha}}.
\end{displaymath}

A representation of the isomorphism type of this level $\leq 2$ tree and the corresponding expressions are given as follows. The level $\leq 2$ tree is $U$, where $\comp{1}{U}$ has cardinality 4 and $\comp{2}{U}$ has cardinality 156.

\begin{etaremune} 
\item $\comp{2}{U}(\emptyset) = (\se{\emptyset}, (0))$

(2, $\emptyset$)

\item $\comp{2}{U_{\node}}(((11)))= (0, 0)$

(2, ($g$(2, ($\alpha_{(0)}$, (0))), (1)))

\item $\comp{2}{U_{\node}}(((11), (0)))= (0, 1)$

(2, ($g$(2, ($\alpha_{(0)}$, (0))), (1), $g$(2, ($\alpha_{(0)}$, (0), $\alpha_{(0, 0)}$, $-1$)), (0)))

\item $\comp{2}{U_{\node}}(((10)))= (0, 0)$

(2, ($g$(2, ($\alpha_{(0)}$, (0))), (1), $g$(2, ($\alpha_{(0)}$, $-1$)), (0)))

\item $\comp{2}{U_{\node}}(((9)))= (0, 0)$

(2, ($g$(2, ($\alpha_{(0)}$, (0))), (1), $g$(2, ($\alpha_{(0)}$, $-1$)), $-1$))

\item $\comp{2}{U_{\node}}(((9), (7)))= (0, 1)$

(2, ($g$(2, ($\alpha_{(0)}$, (0))), (1), $g$(2, ($\alpha_{(0, 0)}$, (0))), (0)))

\item $\comp{2}{U_{\node}}(((9), (6)))= (0, 0, 0)$

(2, ($g$(2, ($\alpha_{(0)}$, (0))), (1), $g$(2, ($\alpha_{(0, 0)}$, (0))), (0), $g$(2, ($\alpha_{(0)}$, (0), $\alpha_{(0, 0)}$, $-1$)), $-1$))

\item $\comp{2}{U_{\node}}(((9), (5)))= (0, 1)$

(2, ($g$(2, ($\alpha_{(0)}$, (0))), (1), $g$(2, ($\alpha_{(0, 0)}$, (0))), (0), $g$(2, ($\alpha_{(0)}$, $-1$)), $-1$))

\item $\comp{2}{U_{\node}}(((9), (5), (1)))= (0, 1, 0)$

(2, ($g$(2, ($\alpha_{(0)}$, (0))), (1), $g$(2, ($\alpha_{(0, 0)}$, (0))), (0), $g$(2, ($\alpha_{(0, 1)}$, (0))), $-1$))

\item $\comp{2}{U_{\node}}(((9), (5), (0)))= (0, 0, 0)$

(2, ($g$(2, ($\alpha_{(0)}$, (0))), (1), $g$(2, ($\alpha_{(0, 0)}$, (0))), (0), $g$(2, ($\alpha_{(0, 1)}$, (0), $\alpha_{(0, 0)}$, $-1$)), $-1$))

\item $\comp{2}{U_{\node}}(((9), (4)))= (0, 0, 0)$

(2, ($g$(2, ($\alpha_{(0)}$, (0))), (1), $g$(2, ($\alpha_{(0, 0)}$, (0))), $-1$))

\item $\comp{2}{U_{\node}}(((9), (4), (4)))= (0, 1)$

(2, ($g$(2, ($\alpha_{(0)}$, (0))), (1), $g$(2, ($\alpha_{(0, 0)}$, (0), $\alpha_{(0, 0, 0)}$, $-1$)), (0)))

\item $\comp{2}{U_{\node}}(((9), (4), (3)))= (0, 0, 1)$

(2, ($g$(2, ($\alpha_{(0)}$, (0))), (1), $g$(2, ($\alpha_{(0, 0)}$, (0), $\alpha_{(0, 0, 0)}$, $-1$)), (0), $g$(2, ($\alpha_{(0)}$, (0), $\alpha_{(0, 0)}$, $-1$)), $-1$))

\item $\comp{2}{U_{\node}}(((9), (4), (2)))= (0, 0, 0, 0)$

(2, ($g$(2, ($\alpha_{(0)}$, (0))), (1), $g$(2, ($\alpha_{(0, 0)}$, (0), $\alpha_{(0, 0, 0)}$, $-1$)), (0), $g$(2, ($\alpha_{(0)}$, (0), $\alpha_{(0, 0, 0)}$, $-1$)), $-1$))

\item $\comp{2}{U_{\node}}(((9), (4), (1)))= (0, 1)$

(2, ($g$(2, ($\alpha_{(0)}$, (0))), (1), $g$(2, ($\alpha_{(0, 0)}$, (0), $\alpha_{(0, 0, 0)}$, $-1$)), (0), $g$(2, ($\alpha_{(0)}$, $-1$)), $-1$))

\item $\comp{2}{U_{\node}}(((9), (4), (1), (2)))= (0, 1, 0)$

(2, ($g$(2, ($\alpha_{(0)}$, (0))), (1), $g$(2, ($\alpha_{(0, 0)}$, (0), $\alpha_{(0, 0, 0)}$, $-1$)), (0), $g$(2, ($\alpha_{(0, 1)}$, (0))), $-1$))

\item $\comp{2}{U_{\node}}(((9), (4), (1), (1)))= (0, 0, 1)$

(2, ($g$(2, ($\alpha_{(0)}$, (0))), (1), $g$(2, ($\alpha_{(0, 0)}$, (0), $\alpha_{(0, 0, 0)}$, $-1$)), (0), $g$(2, ($\alpha_{(0, 1)}$, (0), $\alpha_{(0, 0)}$, $-1$)), $-1$))

\item $\comp{2}{U_{\node}}(((9), (4), (1), (0)))= (0, 0, 0, 0)$

(2, ($g$(2, ($\alpha_{(0)}$, (0))), (1), $g$(2, ($\alpha_{(0, 0)}$, (0), $\alpha_{(0, 0, 0)}$, $-1$)), (0), $g$(2, ($\alpha_{(0, 1)}$, (0), $\alpha_{(0, 0, 0)}$, $-1$)), $-1$))

\item $\comp{2}{U_{\node}}(((9), (4), (0)))= (0, 0, 1)$

(2, ($g$(2, ($\alpha_{(0)}$, (0))), (1), $g$(2, ($\alpha_{(0, 0)}$, (0), $\alpha_{(0, 0, 0)}$, $-1$)), (0), $g$(2, ($\alpha_{(0, 0)}$, (0))), $-1$))

\item $\comp{2}{U_{\node}}(((9), (3)))= (0, 1)$

(2, ($g$(2, ($\alpha_{(0)}$, (0))), (1), $g$(2, ($\alpha_{(0, 0)}$, $-1$)), (0)))

\item $\comp{2}{U_{\node}}(((9), (2)))= (0, 0, 0)$

(2, ($g$(2, ($\alpha_{(0)}$, (0))), (1), $g$(2, ($\alpha_{(0, 0)}$, $-1$)), (0), $g$(2, ($\alpha_{(0)}$, (0), $\alpha_{(0, 0)}$, $-1$)), $-1$))

\item $\comp{2}{U_{\node}}(((9), (1)))= (0, 1)$

(2, ($g$(2, ($\alpha_{(0)}$, (0))), (1), $g$(2, ($\alpha_{(0, 0)}$, $-1$)), (0), $g$(2, ($\alpha_{(0)}$, $-1$)), $-1$))

\item $\comp{2}{U_{\node}}(((9), (1), (1)))= (0, 1, 0)$

(2, ($g$(2, ($\alpha_{(0)}$, (0))), (1), $g$(2, ($\alpha_{(0, 0)}$, $-1$)), (0), $g$(2, ($\alpha_{(0, 1)}$, (0))), $-1$))

\item $\comp{2}{U_{\node}}(((9), (1), (0)))= (0, 0, 0)$

(2, ($g$(2, ($\alpha_{(0)}$, (0))), (1), $g$(2, ($\alpha_{(0, 0)}$, $-1$)), (0), $g$(2, ($\alpha_{(0, 1)}$, (0), $\alpha_{(0, 0)}$, $-1$)), $-1$))

\item $\comp{2}{U_{\node}}(((9), (0)))= (0, 0, 0)$

(2, ($g$(2, ($\alpha_{(0)}$, (0))), (1), $g$(2, ($\alpha_{(0, 0)}$, $-1$)), (0), $g$(2, ($\alpha_{(0, 0)}$, (0))), $-1$))

\item $\comp{2}{U_{\node}}(((8)))= (0, 0)$

(2, ($g$(2, ($\alpha_{(0)}$, (0))), (1), $g$(1, ((0))), (0)))

\item $\comp{2}{U_{\node}}(((7)))= (0, 0)$

(2, ($g$(2, ($\alpha_{(0)}$, (0))), (1), $g$(1, ((0))), (0), $g$(2, ($\alpha_{(0)}$, $-1$)), $-1$))

\item $\comp{2}{U_{\node}}(((7), (0)))= (0, 0, 0)$

(2, ($g$(2, ($\alpha_{(0)}$, (0))), (1), $g$(1, ((0))), (0), $g$(2, ($\alpha_{(0, 0)}$, (0))), $-1$))

\item $\comp{2}{U_{\node}}(((6)))= -1$

(2, ($g$(2, ($\alpha_{(0)}$, (0))), (1), $g$(1, ((0))), $-1$))

\item $\comp{2}{U_{\node}}(((5)))= -1$

(2, ($g$(2, ($\alpha_{(0)}$, (0))), (0)))

\item $\comp{2}{U_{\node}}(((4)))= (0, 0)$

(2, ($g$(2, ($\alpha_{(0)}$, (0))), $-1$))

\item $\comp{2}{U_{\node}}(((4), (15)))= (0, 0, 0)$

(2, ($g$(2, ($\alpha_{(0)}$, (0), $\alpha_{(0, 0)}$, $-1$)), (1)))

\item $\comp{2}{U_{\node}}(((4), (15), (0)))= (0, 0, 1)$

(2, ($g$(2, ($\alpha_{(0)}$, (0), $\alpha_{(0, 0)}$, $-1$)), (1), $g$(2, ($\alpha_{(0)}$, (0), $\alpha_{(0, 0, 0)}$, $-1$)), (0)))

\item $\comp{2}{U_{\node}}(((4), (14)))= (0, 0, 0)$

(2, ($g$(2, ($\alpha_{(0)}$, (0), $\alpha_{(0, 0)}$, $-1$)), (1), $g$(2, ($\alpha_{(0)}$, $-1$)), (0)))

\item $\comp{2}{U_{\node}}(((4), (13)))= (0, 1)$

(2, ($g$(2, ($\alpha_{(0)}$, (0), $\alpha_{(0, 0)}$, $-1$)), (1), $g$(2, ($\alpha_{(0)}$, $-1$)), $-1$))

\item $\comp{2}{U_{\node}}(((4), (13), (10)))= (0, 0, 0)$

(2, ($g$(2, ($\alpha_{(0)}$, (0), $\alpha_{(0, 0)}$, $-1$)), (1), $g$(2, ($\alpha_{(0, 1)}$, (0))), (0)))

\item $\comp{2}{U_{\node}}(((4), (13), (9)))= (0, 2)$

(2, ($g$(2, ($\alpha_{(0)}$, (0), $\alpha_{(0, 0)}$, $-1$)), (1), $g$(2, ($\alpha_{(0, 1)}$, (0))), (0), $g$(2, ($\alpha_{(0)}$, $-1$)), $-1$))

\item $\comp{2}{U_{\node}}(((4), (13), (9), (2)))= (0, 2, 0)$

(2, ($g$(2, ($\alpha_{(0)}$, (0), $\alpha_{(0, 0)}$, $-1$)), (1), $g$(2, ($\alpha_{(0, 1)}$, (0))), (0), $g$(2, ($\alpha_{(0, 2)}$, (0))), $-1$))

\item $\comp{2}{U_{\node}}(((4), (13), (9), (1)))= (0, 1, 0)$

(2, ($g$(2, ($\alpha_{(0)}$, (0), $\alpha_{(0, 0)}$, $-1$)), (1), $g$(2, ($\alpha_{(0, 1)}$, (0))), (0), $g$(2, ($\alpha_{(0, 2)}$, (0), $\alpha_{(0, 1)}$, $-1$)), $-1$))

\item $\comp{2}{U_{\node}}(((4), (13), (9), (0)))= (0, 0, 0)$

(2, ($g$(2, ($\alpha_{(0)}$, (0), $\alpha_{(0, 0)}$, $-1$)), (1), $g$(2, ($\alpha_{(0, 1)}$, (0))), (0), $g$(2, ($\alpha_{(0, 2)}$, (0), $\alpha_{(0, 0)}$, $-1$)), $-1$))

\item $\comp{2}{U_{\node}}(((4), (13), (8)))= (0, 1, 0)$

(2, ($g$(2, ($\alpha_{(0)}$, (0), $\alpha_{(0, 0)}$, $-1$)), (1), $g$(2, ($\alpha_{(0, 1)}$, (0))), $-1$))

\item $\comp{2}{U_{\node}}(((4), (13), (8), (2)))= (0, 0, 0)$

(2, ($g$(2, ($\alpha_{(0)}$, (0), $\alpha_{(0, 0)}$, $-1$)), (1), $g$(2, ($\alpha_{(0, 1)}$, (0), $\alpha_{(0, 1, 0)}$, $-1$)), (0)))

\item $\comp{2}{U_{\node}}(((4), (13), (8), (1)))= (0, 2)$

(2, ($g$(2, ($\alpha_{(0)}$, (0), $\alpha_{(0, 0)}$, $-1$)), (1), $g$(2, ($\alpha_{(0, 1)}$, (0), $\alpha_{(0, 1, 0)}$, $-1$)), (0), $g$(2, ($\alpha_{(0)}$, $-1$)), $-1$))

\item $\comp{2}{U_{\node}}(((4), (13), (8), (1), (3)))= (0, 2, 0)$

(2, ($g$(2, ($\alpha_{(0)}$, (0), $\alpha_{(0, 0)}$, $-1$)), (1), $g$(2, ($\alpha_{(0, 1)}$, (0), $\alpha_{(0, 1, 0)}$, $-1$)), (0), $g$(2, ($\alpha_{(0, 2)}$, (0))), $-1$))

\item $\comp{2}{U_{\node}}(((4), (13), (8), (1), (2)))= (0, 1, 1)$

(2, ($g$(2, ($\alpha_{(0)}$, (0), $\alpha_{(0, 0)}$, $-1$)), (1), $g$(2, ($\alpha_{(0, 1)}$, (0), $\alpha_{(0, 1, 0)}$, $-1$)), (0), $g$(2, ($\alpha_{(0, 2)}$, (0), $\alpha_{(0, 1)}$, $-1$)), $-1$))

\item $\comp{2}{U_{\node}}(((4), (13), (8), (1), (1)))= (0, 1, 0, 0)$

(2, ($g$(2, ($\alpha_{(0)}$, (0), $\alpha_{(0, 0)}$, $-1$)), (1), $g$(2, ($\alpha_{(0, 1)}$, (0), $\alpha_{(0, 1, 0)}$, $-1$)), (0), $g$(2, ($\alpha_{(0, 2)}$, (0), $\alpha_{(0, 1, 0)}$, $-1$)), $-1$))

\item $\comp{2}{U_{\node}}(((4), (13), (8), (1), (0)))= (0, 0, 0)$

(2, ($g$(2, ($\alpha_{(0)}$, (0), $\alpha_{(0, 0)}$, $-1$)), (1), $g$(2, ($\alpha_{(0, 1)}$, (0), $\alpha_{(0, 1, 0)}$, $-1$)), (0), $g$(2, ($\alpha_{(0, 2)}$, (0), $\alpha_{(0, 0)}$, $-1$)), $-1$))

\item $\comp{2}{U_{\node}}(((4), (13), (8), (0)))= (0, 1, 1)$

(2, ($g$(2, ($\alpha_{(0)}$, (0), $\alpha_{(0, 0)}$, $-1$)), (1), $g$(2, ($\alpha_{(0, 1)}$, (0), $\alpha_{(0, 1, 0)}$, $-1$)), (0), $g$(2, ($\alpha_{(0, 1)}$, (0))), $-1$))

\item $\comp{2}{U_{\node}}(((4), (13), (7)))= (0, 0, 0)$

(2, ($g$(2, ($\alpha_{(0)}$, (0), $\alpha_{(0, 0)}$, $-1$)), (1), $g$(2, ($\alpha_{(0, 1)}$, (0), $\alpha_{(0, 0)}$, $-1$)), (0)))

\item $\comp{2}{U_{\node}}(((4), (13), (6)))= (0, 2)$

(2, ($g$(2, ($\alpha_{(0)}$, (0), $\alpha_{(0, 0)}$, $-1$)), (1), $g$(2, ($\alpha_{(0, 1)}$, (0), $\alpha_{(0, 0)}$, $-1$)), (0), $g$(2, ($\alpha_{(0)}$, $-1$)), $-1$))

\item $\comp{2}{U_{\node}}(((4), (13), (6), (2)))= (0, 2, 0)$

(2, ($g$(2, ($\alpha_{(0)}$, (0), $\alpha_{(0, 0)}$, $-1$)), (1), $g$(2, ($\alpha_{(0, 1)}$, (0), $\alpha_{(0, 0)}$, $-1$)), (0), $g$(2, ($\alpha_{(0, 2)}$, (0))), $-1$))

\item $\comp{2}{U_{\node}}(((4), (13), (6), (1)))= (0, 1, 0)$

(2, ($g$(2, ($\alpha_{(0)}$, (0), $\alpha_{(0, 0)}$, $-1$)), (1), $g$(2, ($\alpha_{(0, 1)}$, (0), $\alpha_{(0, 0)}$, $-1$)), (0), $g$(2, ($\alpha_{(0, 2)}$, (0), $\alpha_{(0, 1)}$, $-1$)), $-1$))

\item $\comp{2}{U_{\node}}(((4), (13), (6), (0)))= (0, 0, 0)$

(2, ($g$(2, ($\alpha_{(0)}$, (0), $\alpha_{(0, 0)}$, $-1$)), (1), $g$(2, ($\alpha_{(0, 1)}$, (0), $\alpha_{(0, 0)}$, $-1$)), (0), $g$(2, ($\alpha_{(0, 2)}$, (0), $\alpha_{(0, 0)}$, $-1$)), $-1$))

\item $\comp{2}{U_{\node}}(((4), (13), (5)))= (0, 1, 0)$

(2, ($g$(2, ($\alpha_{(0)}$, (0), $\alpha_{(0, 0)}$, $-1$)), (1), $g$(2, ($\alpha_{(0, 1)}$, (0), $\alpha_{(0, 0)}$, $-1$)), (0), $g$(2, ($\alpha_{(0, 1)}$, (0))), $-1$))

\item $\comp{2}{U_{\node}}(((4), (13), (4)))= (0, 0, 0)$

(2, ($g$(2, ($\alpha_{(0)}$, (0), $\alpha_{(0, 0)}$, $-1$)), (1), $g$(2, ($\alpha_{(0, 1)}$, (0), $\alpha_{(0, 0)}$, $-1$)), $-1$))

\item $\comp{2}{U_{\node}}(((4), (13), (4), (4)))= (0, 0, 1)$

(2, ($g$(2, ($\alpha_{(0)}$, (0), $\alpha_{(0, 0)}$, $-1$)), (1), $g$(2, ($\alpha_{(0, 1)}$, (0), $\alpha_{(0, 0, 0)}$, $-1$)), (0)))

\item $\comp{2}{U_{\node}}(((4), (13), (4), (3)))= (0, 0, 0, 0)$

(2, ($g$(2, ($\alpha_{(0)}$, (0), $\alpha_{(0, 0)}$, $-1$)), (1), $g$(2, ($\alpha_{(0, 1)}$, (0), $\alpha_{(0, 0, 0)}$, $-1$)), (0), $g$(2, ($\alpha_{(0)}$, (0), $\alpha_{(0, 0, 0)}$, $-1$)), $-1$))

\item $\comp{2}{U_{\node}}(((4), (13), (4), (2)))= (0, 2)$

(2, ($g$(2, ($\alpha_{(0)}$, (0), $\alpha_{(0, 0)}$, $-1$)), (1), $g$(2, ($\alpha_{(0, 1)}$, (0), $\alpha_{(0, 0, 0)}$, $-1$)), (0), $g$(2, ($\alpha_{(0)}$, $-1$)), $-1$))

\item $\comp{2}{U_{\node}}(((4), (13), (4), (2), (3)))= (0, 2, 0)$

(2, ($g$(2, ($\alpha_{(0)}$, (0), $\alpha_{(0, 0)}$, $-1$)), (1), $g$(2, ($\alpha_{(0, 1)}$, (0), $\alpha_{(0, 0, 0)}$, $-1$)), (0), $g$(2, ($\alpha_{(0, 2)}$, (0))), $-1$))

\item $\comp{2}{U_{\node}}(((4), (13), (4), (2), (2)))= (0, 1, 0)$

(2, ($g$(2, ($\alpha_{(0)}$, (0), $\alpha_{(0, 0)}$, $-1$)), (1), $g$(2, ($\alpha_{(0, 1)}$, (0), $\alpha_{(0, 0, 0)}$, $-1$)), (0), $g$(2, ($\alpha_{(0, 2)}$, (0), $\alpha_{(0, 1)}$, $-1$)), $-1$))

\item $\comp{2}{U_{\node}}(((4), (13), (4), (2), (1)))= (0, 0, 1)$

(2, ($g$(2, ($\alpha_{(0)}$, (0), $\alpha_{(0, 0)}$, $-1$)), (1), $g$(2, ($\alpha_{(0, 1)}$, (0), $\alpha_{(0, 0, 0)}$, $-1$)), (0), $g$(2, ($\alpha_{(0, 2)}$, (0), $\alpha_{(0, 0)}$, $-1$)), $-1$))

\item $\comp{2}{U_{\node}}(((4), (13), (4), (2), (0)))= (0, 0, 0, 0)$

(2, ($g$(2, ($\alpha_{(0)}$, (0), $\alpha_{(0, 0)}$, $-1$)), (1), $g$(2, ($\alpha_{(0, 1)}$, (0), $\alpha_{(0, 0, 0)}$, $-1$)), (0), $g$(2, ($\alpha_{(0, 2)}$, (0), $\alpha_{(0, 0, 0)}$, $-1$)), $-1$))

\item $\comp{2}{U_{\node}}(((4), (13), (4), (1)))= (0, 1, 0)$

(2, ($g$(2, ($\alpha_{(0)}$, (0), $\alpha_{(0, 0)}$, $-1$)), (1), $g$(2, ($\alpha_{(0, 1)}$, (0), $\alpha_{(0, 0, 0)}$, $-1$)), (0), $g$(2, ($\alpha_{(0, 1)}$, (0))), $-1$))

\item $\comp{2}{U_{\node}}(((4), (13), (4), (0)))= (0, 0, 1)$

(2, ($g$(2, ($\alpha_{(0)}$, (0), $\alpha_{(0, 0)}$, $-1$)), (1), $g$(2, ($\alpha_{(0, 1)}$, (0), $\alpha_{(0, 0, 0)}$, $-1$)), (0), $g$(2, ($\alpha_{(0, 1)}$, (0), $\alpha_{(0, 0)}$, $-1$)), $-1$))

\item $\comp{2}{U_{\node}}(((4), (13), (3)))= (0, 0, 0)$

(2, ($g$(2, ($\alpha_{(0)}$, (0), $\alpha_{(0, 0)}$, $-1$)), (1), $g$(2, ($\alpha_{(0, 1)}$, $-1$)), (0)))

\item $\comp{2}{U_{\node}}(((4), (13), (2)))= (0, 2)$

(2, ($g$(2, ($\alpha_{(0)}$, (0), $\alpha_{(0, 0)}$, $-1$)), (1), $g$(2, ($\alpha_{(0, 1)}$, $-1$)), (0), $g$(2, ($\alpha_{(0)}$, $-1$)), $-1$))

\item $\comp{2}{U_{\node}}(((4), (13), (2), (2)))= (0, 2, 0)$

(2, ($g$(2, ($\alpha_{(0)}$, (0), $\alpha_{(0, 0)}$, $-1$)), (1), $g$(2, ($\alpha_{(0, 1)}$, $-1$)), (0), $g$(2, ($\alpha_{(0, 2)}$, (0))), $-1$))

\item $\comp{2}{U_{\node}}(((4), (13), (2), (1)))= (0, 1, 0)$

(2, ($g$(2, ($\alpha_{(0)}$, (0), $\alpha_{(0, 0)}$, $-1$)), (1), $g$(2, ($\alpha_{(0, 1)}$, $-1$)), (0), $g$(2, ($\alpha_{(0, 2)}$, (0), $\alpha_{(0, 1)}$, $-1$)), $-1$))

\item $\comp{2}{U_{\node}}(((4), (13), (2), (0)))= (0, 0, 0)$

(2, ($g$(2, ($\alpha_{(0)}$, (0), $\alpha_{(0, 0)}$, $-1$)), (1), $g$(2, ($\alpha_{(0, 1)}$, $-1$)), (0), $g$(2, ($\alpha_{(0, 2)}$, (0), $\alpha_{(0, 0)}$, $-1$)), $-1$))

\item $\comp{2}{U_{\node}}(((4), (13), (1)))= (0, 1, 0)$

(2, ($g$(2, ($\alpha_{(0)}$, (0), $\alpha_{(0, 0)}$, $-1$)), (1), $g$(2, ($\alpha_{(0, 1)}$, $-1$)), (0), $g$(2, ($\alpha_{(0, 1)}$, (0))), $-1$))

\item $\comp{2}{U_{\node}}(((4), (13), (0)))= (0, 0, 0)$

(2, ($g$(2, ($\alpha_{(0)}$, (0), $\alpha_{(0, 0)}$, $-1$)), (1), $g$(2, ($\alpha_{(0, 1)}$, $-1$)), (0), $g$(2, ($\alpha_{(0, 1)}$, (0), $\alpha_{(0, 0)}$, $-1$)), $-1$))

\item $\comp{2}{U_{\node}}(((4), (12)))= (0, 0, 0)$

(2, ($g$(2, ($\alpha_{(0)}$, (0), $\alpha_{(0, 0)}$, $-1$)), (1), $g$(2, ($\alpha_{(0, 0)}$, (0))), (0)))

\item $\comp{2}{U_{\node}}(((4), (11)))= (0, 1)$

(2, ($g$(2, ($\alpha_{(0)}$, (0), $\alpha_{(0, 0)}$, $-1$)), (1), $g$(2, ($\alpha_{(0, 0)}$, (0))), (0), $g$(2, ($\alpha_{(0)}$, $-1$)), $-1$))

\item $\comp{2}{U_{\node}}(((4), (11), (1)))= (0, 1, 0)$

(2, ($g$(2, ($\alpha_{(0)}$, (0), $\alpha_{(0, 0)}$, $-1$)), (1), $g$(2, ($\alpha_{(0, 0)}$, (0))), (0), $g$(2, ($\alpha_{(0, 1)}$, (0))), $-1$))

\item $\comp{2}{U_{\node}}(((4), (11), (0)))= (0, 0, 0)$

(2, ($g$(2, ($\alpha_{(0)}$, (0), $\alpha_{(0, 0)}$, $-1$)), (1), $g$(2, ($\alpha_{(0, 0)}$, (0))), (0), $g$(2, ($\alpha_{(0, 1)}$, (0), $\alpha_{(0, 0)}$, $-1$)), $-1$))

\item $\comp{2}{U_{\node}}(((4), (10)))= (0, 0, 0)$

(2, ($g$(2, ($\alpha_{(0)}$, (0), $\alpha_{(0, 0)}$, $-1$)), (1), $g$(2, ($\alpha_{(0, 0)}$, (0))), $-1$))

\item $\comp{2}{U_{\node}}(((4), (10), (3)))= (0, 0, 1)$

(2, ($g$(2, ($\alpha_{(0)}$, (0), $\alpha_{(0, 0)}$, $-1$)), (1), $g$(2, ($\alpha_{(0, 0)}$, (0), $\alpha_{(0, 0, 0)}$, $-1$)), (0)))

\item $\comp{2}{U_{\node}}(((4), (10), (2)))= (0, 0, 0, 0)$

(2, ($g$(2, ($\alpha_{(0)}$, (0), $\alpha_{(0, 0)}$, $-1$)), (1), $g$(2, ($\alpha_{(0, 0)}$, (0), $\alpha_{(0, 0, 0)}$, $-1$)), (0), $g$(2, ($\alpha_{(0)}$, (0), $\alpha_{(0, 0, 0)}$, $-1$)), $-1$))

\item $\comp{2}{U_{\node}}(((4), (10), (1)))= (0, 1)$

(2, ($g$(2, ($\alpha_{(0)}$, (0), $\alpha_{(0, 0)}$, $-1$)), (1), $g$(2, ($\alpha_{(0, 0)}$, (0), $\alpha_{(0, 0, 0)}$, $-1$)), (0), $g$(2, ($\alpha_{(0)}$, $-1$)), $-1$))

\item $\comp{2}{U_{\node}}(((4), (10), (1), (2)))= (0, 1, 0)$

(2, ($g$(2, ($\alpha_{(0)}$, (0), $\alpha_{(0, 0)}$, $-1$)), (1), $g$(2, ($\alpha_{(0, 0)}$, (0), $\alpha_{(0, 0, 0)}$, $-1$)), (0), $g$(2, ($\alpha_{(0, 1)}$, (0))), $-1$))

\item $\comp{2}{U_{\node}}(((4), (10), (1), (1)))= (0, 0, 1)$

(2, ($g$(2, ($\alpha_{(0)}$, (0), $\alpha_{(0, 0)}$, $-1$)), (1), $g$(2, ($\alpha_{(0, 0)}$, (0), $\alpha_{(0, 0, 0)}$, $-1$)), (0), $g$(2, ($\alpha_{(0, 1)}$, (0), $\alpha_{(0, 0)}$, $-1$)), $-1$))

\item $\comp{2}{U_{\node}}(((4), (10), (1), (0)))= (0, 0, 0, 0)$

(2, ($g$(2, ($\alpha_{(0)}$, (0), $\alpha_{(0, 0)}$, $-1$)), (1), $g$(2, ($\alpha_{(0, 0)}$, (0), $\alpha_{(0, 0, 0)}$, $-1$)), (0), $g$(2, ($\alpha_{(0, 1)}$, (0), $\alpha_{(0, 0, 0)}$, $-1$)), $-1$))

\item $\comp{2}{U_{\node}}(((4), (10), (0)))= (0, 0, 1)$

(2, ($g$(2, ($\alpha_{(0)}$, (0), $\alpha_{(0, 0)}$, $-1$)), (1), $g$(2, ($\alpha_{(0, 0)}$, (0), $\alpha_{(0, 0, 0)}$, $-1$)), (0), $g$(2, ($\alpha_{(0, 0)}$, (0))), $-1$))

\item $\comp{2}{U_{\node}}(((4), (9)))= (0, 0, 0)$

(2, ($g$(2, ($\alpha_{(0)}$, (0), $\alpha_{(0, 0)}$, $-1$)), (1), $g$(2, ($\alpha_{(0, 0)}$, $-1$)), (0)))

\item $\comp{2}{U_{\node}}(((4), (8)))= (0, 1)$

(2, ($g$(2, ($\alpha_{(0)}$, (0), $\alpha_{(0, 0)}$, $-1$)), (1), $g$(2, ($\alpha_{(0, 0)}$, $-1$)), (0), $g$(2, ($\alpha_{(0)}$, $-1$)), $-1$))

\item $\comp{2}{U_{\node}}(((4), (8), (1)))= (0, 1, 0)$

(2, ($g$(2, ($\alpha_{(0)}$, (0), $\alpha_{(0, 0)}$, $-1$)), (1), $g$(2, ($\alpha_{(0, 0)}$, $-1$)), (0), $g$(2, ($\alpha_{(0, 1)}$, (0))), $-1$))

\item $\comp{2}{U_{\node}}(((4), (8), (0)))= (0, 0, 0)$

(2, ($g$(2, ($\alpha_{(0)}$, (0), $\alpha_{(0, 0)}$, $-1$)), (1), $g$(2, ($\alpha_{(0, 0)}$, $-1$)), (0), $g$(2, ($\alpha_{(0, 1)}$, (0), $\alpha_{(0, 0)}$, $-1$)), $-1$))

\item $\comp{2}{U_{\node}}(((4), (7)))= (0, 0, 0)$

(2, ($g$(2, ($\alpha_{(0)}$, (0), $\alpha_{(0, 0)}$, $-1$)), (1), $g$(2, ($\alpha_{(0, 0)}$, $-1$)), (0), $g$(2, ($\alpha_{(0, 0)}$, (0))), $-1$))

\item $\comp{2}{U_{\node}}(((4), (6)))= (0, 0, 0)$

(2, ($g$(2, ($\alpha_{(0)}$, (0), $\alpha_{(0, 0)}$, $-1$)), (1), $g$(2, ($\alpha_{(0, 0)}$, $-1$)), $-1$))

\item $\comp{2}{U_{\node}}(((4), (6), (13)))= (0, 0, 1)$

(2, ($g$(2, ($\alpha_{(0)}$, (0), $\alpha_{(0, 0)}$, $-1$)), (1), $g$(2, ($\alpha_{(0, 0, 0)}$, (0))), (0)))

\item $\comp{2}{U_{\node}}(((4), (6), (12)))= (0, 0, 0, 0)$

(2, ($g$(2, ($\alpha_{(0)}$, (0), $\alpha_{(0, 0)}$, $-1$)), (1), $g$(2, ($\alpha_{(0, 0, 0)}$, (0))), (0), $g$(2, ($\alpha_{(0)}$, (0), $\alpha_{(0, 0, 0)}$, $-1$)), $-1$))

\item $\comp{2}{U_{\node}}(((4), (6), (11)))= (0, 1)$

(2, ($g$(2, ($\alpha_{(0)}$, (0), $\alpha_{(0, 0)}$, $-1$)), (1), $g$(2, ($\alpha_{(0, 0, 0)}$, (0))), (0), $g$(2, ($\alpha_{(0)}$, $-1$)), $-1$))

\item $\comp{2}{U_{\node}}(((4), (6), (11), (2)))= (0, 1, 0)$

(2, ($g$(2, ($\alpha_{(0)}$, (0), $\alpha_{(0, 0)}$, $-1$)), (1), $g$(2, ($\alpha_{(0, 0, 0)}$, (0))), (0), $g$(2, ($\alpha_{(0, 1)}$, (0))), $-1$))

\item $\comp{2}{U_{\node}}(((4), (6), (11), (1)))= (0, 0, 1)$

(2, ($g$(2, ($\alpha_{(0)}$, (0), $\alpha_{(0, 0)}$, $-1$)), (1), $g$(2, ($\alpha_{(0, 0, 0)}$, (0))), (0), $g$(2, ($\alpha_{(0, 1)}$, (0), $\alpha_{(0, 0)}$, $-1$)), $-1$))

\item $\comp{2}{U_{\node}}(((4), (6), (11), (0)))= (0, 0, 0, 0)$

(2, ($g$(2, ($\alpha_{(0)}$, (0), $\alpha_{(0, 0)}$, $-1$)), (1), $g$(2, ($\alpha_{(0, 0, 0)}$, (0))), (0), $g$(2, ($\alpha_{(0, 1)}$, (0), $\alpha_{(0, 0, 0)}$, $-1$)), $-1$))

\item $\comp{2}{U_{\node}}(((4), (6), (10)))= (0, 0, 1)$

(2, ($g$(2, ($\alpha_{(0)}$, (0), $\alpha_{(0, 0)}$, $-1$)), (1), $g$(2, ($\alpha_{(0, 0, 0)}$, (0))), (0), $g$(2, ($\alpha_{(0, 0)}$, (0))), $-1$))

\item $\comp{2}{U_{\node}}(((4), (6), (9)))= (0, 0, 0, 0)$

(2, ($g$(2, ($\alpha_{(0)}$, (0), $\alpha_{(0, 0)}$, $-1$)), (1), $g$(2, ($\alpha_{(0, 0, 0)}$, (0))), (0), $g$(2, ($\alpha_{(0, 0)}$, (0), $\alpha_{(0, 0, 0)}$, $-1$)), $-1$))

\item $\comp{2}{U_{\node}}(((4), (6), (8)))= (0, 0, 1)$

(2, ($g$(2, ($\alpha_{(0)}$, (0), $\alpha_{(0, 0)}$, $-1$)), (1), $g$(2, ($\alpha_{(0, 0, 0)}$, (0))), (0), $g$(2, ($\alpha_{(0, 0)}$, $-1$)), $-1$))

\item $\comp{2}{U_{\node}}(((4), (6), (8), (1)))= (0, 0, 1, 0)$

(2, ($g$(2, ($\alpha_{(0)}$, (0), $\alpha_{(0, 0)}$, $-1$)), (1), $g$(2, ($\alpha_{(0, 0, 0)}$, (0))), (0), $g$(2, ($\alpha_{(0, 0, 1)}$, (0))), $-1$))

\item $\comp{2}{U_{\node}}(((4), (6), (8), (0)))= (0, 0, 0, 0)$

(2, ($g$(2, ($\alpha_{(0)}$, (0), $\alpha_{(0, 0)}$, $-1$)), (1), $g$(2, ($\alpha_{(0, 0, 0)}$, (0))), (0), $g$(2, ($\alpha_{(0, 0, 1)}$, (0), $\alpha_{(0, 0, 0)}$, $-1$)), $-1$))

\item $\comp{2}{U_{\node}}(((4), (6), (7)))= (0, 0, 0, 0)$

(2, ($g$(2, ($\alpha_{(0)}$, (0), $\alpha_{(0, 0)}$, $-1$)), (1), $g$(2, ($\alpha_{(0, 0, 0)}$, (0))), $-1$))

\item $\comp{2}{U_{\node}}(((4), (6), (7), (8)))= (0, 0, 1)$

(2, ($g$(2, ($\alpha_{(0)}$, (0), $\alpha_{(0, 0)}$, $-1$)), (1), $g$(2, ($\alpha_{(0, 0, 0)}$, (0), $\alpha_{(0, 0, 0, 0)}$, $-1$)), (0)))

\item $\comp{2}{U_{\node}}(((4), (6), (7), (7)))= (0, 0, 0, 1)$

(2, ($g$(2, ($\alpha_{(0)}$, (0), $\alpha_{(0, 0)}$, $-1$)), (1), $g$(2, ($\alpha_{(0, 0, 0)}$, (0), $\alpha_{(0, 0, 0, 0)}$, $-1$)), (0), $g$(2, ($\alpha_{(0)}$, (0), $\alpha_{(0, 0, 0)}$, $-1$)), $-1$))

\item $\comp{2}{U_{\node}}(((4), (6), (7), (6)))= (0, 0, 0, 0, 0)$

(2, ($g$(2, ($\alpha_{(0)}$, (0), $\alpha_{(0, 0)}$, $-1$)), (1), $g$(2, ($\alpha_{(0, 0, 0)}$, (0), $\alpha_{(0, 0, 0, 0)}$, $-1$)), (0), $g$(2, ($\alpha_{(0)}$, (0), $\alpha_{(0, 0, 0, 0)}$, $-1$)), $-1$))

\item $\comp{2}{U_{\node}}(((4), (6), (7), (5)))= (0, 1)$

(2, ($g$(2, ($\alpha_{(0)}$, (0), $\alpha_{(0, 0)}$, $-1$)), (1), $g$(2, ($\alpha_{(0, 0, 0)}$, (0), $\alpha_{(0, 0, 0, 0)}$, $-1$)), (0), $g$(2, ($\alpha_{(0)}$, $-1$)), $-1$))

\item $\comp{2}{U_{\node}}(((4), (6), (7), (5), (3)))= (0, 1, 0)$

(2, ($g$(2, ($\alpha_{(0)}$, (0), $\alpha_{(0, 0)}$, $-1$)), (1), $g$(2, ($\alpha_{(0, 0, 0)}$, (0), $\alpha_{(0, 0, 0, 0)}$, $-1$)), (0), $g$(2, ($\alpha_{(0, 1)}$, (0))), $-1$))

\item $\comp{2}{U_{\node}}(((4), (6), (7), (5), (2)))= (0, 0, 1)$

(2, ($g$(2, ($\alpha_{(0)}$, (0), $\alpha_{(0, 0)}$, $-1$)), (1), $g$(2, ($\alpha_{(0, 0, 0)}$, (0), $\alpha_{(0, 0, 0, 0)}$, $-1$)), (0), $g$(2, ($\alpha_{(0, 1)}$, (0), $\alpha_{(0, 0)}$, $-1$)), $-1$))

\item $\comp{2}{U_{\node}}(((4), (6), (7), (5), (1)))= (0, 0, 0, 1)$

(2, ($g$(2, ($\alpha_{(0)}$, (0), $\alpha_{(0, 0)}$, $-1$)), (1), $g$(2, ($\alpha_{(0, 0, 0)}$, (0), $\alpha_{(0, 0, 0, 0)}$, $-1$)), (0), $g$(2, ($\alpha_{(0, 1)}$, (0), $\alpha_{(0, 0, 0)}$, $-1$)), $-1$))

\item $\comp{2}{U_{\node}}(((4), (6), (7), (5), (0)))= (0, 0, 0, 0, 0)$

(2, ($g$(2, ($\alpha_{(0)}$, (0), $\alpha_{(0, 0)}$, $-1$)), (1), $g$(2, ($\alpha_{(0, 0, 0)}$, (0), $\alpha_{(0, 0, 0, 0)}$, $-1$)), (0), $g$(2, ($\alpha_{(0, 1)}$, (0), $\alpha_{(0, 0, 0, 0)}$, $-1$)), $-1$))

\item $\comp{2}{U_{\node}}(((4), (6), (7), (4)))= (0, 0, 1)$

(2, ($g$(2, ($\alpha_{(0)}$, (0), $\alpha_{(0, 0)}$, $-1$)), (1), $g$(2, ($\alpha_{(0, 0, 0)}$, (0), $\alpha_{(0, 0, 0, 0)}$, $-1$)), (0), $g$(2, ($\alpha_{(0, 0)}$, (0))), $-1$))

\item $\comp{2}{U_{\node}}(((4), (6), (7), (3)))= (0, 0, 0, 1)$

(2, ($g$(2, ($\alpha_{(0)}$, (0), $\alpha_{(0, 0)}$, $-1$)), (1), $g$(2, ($\alpha_{(0, 0, 0)}$, (0), $\alpha_{(0, 0, 0, 0)}$, $-1$)), (0), $g$(2, ($\alpha_{(0, 0)}$, (0), $\alpha_{(0, 0, 0)}$, $-1$)), $-1$))

\item $\comp{2}{U_{\node}}(((4), (6), (7), (2)))= (0, 0, 0, 0, 0)$

(2, ($g$(2, ($\alpha_{(0)}$, (0), $\alpha_{(0, 0)}$, $-1$)), (1), $g$(2, ($\alpha_{(0, 0, 0)}$, (0), $\alpha_{(0, 0, 0, 0)}$, $-1$)), (0), $g$(2, ($\alpha_{(0, 0)}$, (0), $\alpha_{(0, 0, 0, 0)}$, $-1$)), $-1$))

\item $\comp{2}{U_{\node}}(((4), (6), (7), (1)))= (0, 0, 1)$

(2, ($g$(2, ($\alpha_{(0)}$, (0), $\alpha_{(0, 0)}$, $-1$)), (1), $g$(2, ($\alpha_{(0, 0, 0)}$, (0), $\alpha_{(0, 0, 0, 0)}$, $-1$)), (0), $g$(2, ($\alpha_{(0, 0)}$, $-1$)), $-1$))

\item $\comp{2}{U_{\node}}(((4), (6), (7), (1), (2)))= (0, 0, 1, 0)$

(2, ($g$(2, ($\alpha_{(0)}$, (0), $\alpha_{(0, 0)}$, $-1$)), (1), $g$(2, ($\alpha_{(0, 0, 0)}$, (0), $\alpha_{(0, 0, 0, 0)}$, $-1$)), (0), $g$(2, ($\alpha_{(0, 0, 1)}$, (0))), $-1$))

\item $\comp{2}{U_{\node}}(((4), (6), (7), (1), (1)))= (0, 0, 0, 1)$

(2, ($g$(2, ($\alpha_{(0)}$, (0), $\alpha_{(0, 0)}$, $-1$)), (1), $g$(2, ($\alpha_{(0, 0, 0)}$, (0), $\alpha_{(0, 0, 0, 0)}$, $-1$)), (0), $g$(2, ($\alpha_{(0, 0, 1)}$, (0), $\alpha_{(0, 0, 0)}$, $-1$)), $-1$))

\item $\comp{2}{U_{\node}}(((4), (6), (7), (1), (0)))= (0, 0, 0, 0, 0)$

(2, ($g$(2, ($\alpha_{(0)}$, (0), $\alpha_{(0, 0)}$, $-1$)), (1), $g$(2, ($\alpha_{(0, 0, 0)}$, (0), $\alpha_{(0, 0, 0, 0)}$, $-1$)), (0), $g$(2, ($\alpha_{(0, 0, 1)}$, (0), $\alpha_{(0, 0, 0, 0)}$, $-1$)), $-1$))

\item $\comp{2}{U_{\node}}(((4), (6), (7), (0)))= (0, 0, 0, 1)$

(2, ($g$(2, ($\alpha_{(0)}$, (0), $\alpha_{(0, 0)}$, $-1$)), (1), $g$(2, ($\alpha_{(0, 0, 0)}$, (0), $\alpha_{(0, 0, 0, 0)}$, $-1$)), (0), $g$(2, ($\alpha_{(0, 0, 0)}$, (0))), $-1$))

\item $\comp{2}{U_{\node}}(((4), (6), (6)))= (0, 0, 1)$

(2, ($g$(2, ($\alpha_{(0)}$, (0), $\alpha_{(0, 0)}$, $-1$)), (1), $g$(2, ($\alpha_{(0, 0, 0)}$, $-1$)), (0)))

\item $\comp{2}{U_{\node}}(((4), (6), (5)))= (0, 0, 0, 0)$

(2, ($g$(2, ($\alpha_{(0)}$, (0), $\alpha_{(0, 0)}$, $-1$)), (1), $g$(2, ($\alpha_{(0, 0, 0)}$, $-1$)), (0), $g$(2, ($\alpha_{(0)}$, (0), $\alpha_{(0, 0, 0)}$, $-1$)), $-1$))

\item $\comp{2}{U_{\node}}(((4), (6), (4)))= (0, 1)$

(2, ($g$(2, ($\alpha_{(0)}$, (0), $\alpha_{(0, 0)}$, $-1$)), (1), $g$(2, ($\alpha_{(0, 0, 0)}$, $-1$)), (0), $g$(2, ($\alpha_{(0)}$, $-1$)), $-1$))

\item $\comp{2}{U_{\node}}(((4), (6), (4), (2)))= (0, 1, 0)$

(2, ($g$(2, ($\alpha_{(0)}$, (0), $\alpha_{(0, 0)}$, $-1$)), (1), $g$(2, ($\alpha_{(0, 0, 0)}$, $-1$)), (0), $g$(2, ($\alpha_{(0, 1)}$, (0))), $-1$))

\item $\comp{2}{U_{\node}}(((4), (6), (4), (1)))= (0, 0, 1)$

(2, ($g$(2, ($\alpha_{(0)}$, (0), $\alpha_{(0, 0)}$, $-1$)), (1), $g$(2, ($\alpha_{(0, 0, 0)}$, $-1$)), (0), $g$(2, ($\alpha_{(0, 1)}$, (0), $\alpha_{(0, 0)}$, $-1$)), $-1$))

\item $\comp{2}{U_{\node}}(((4), (6), (4), (0)))= (0, 0, 0, 0)$

(2, ($g$(2, ($\alpha_{(0)}$, (0), $\alpha_{(0, 0)}$, $-1$)), (1), $g$(2, ($\alpha_{(0, 0, 0)}$, $-1$)), (0), $g$(2, ($\alpha_{(0, 1)}$, (0), $\alpha_{(0, 0, 0)}$, $-1$)), $-1$))

\item $\comp{2}{U_{\node}}(((4), (6), (3)))= (0, 0, 1)$

(2, ($g$(2, ($\alpha_{(0)}$, (0), $\alpha_{(0, 0)}$, $-1$)), (1), $g$(2, ($\alpha_{(0, 0, 0)}$, $-1$)), (0), $g$(2, ($\alpha_{(0, 0)}$, (0))), $-1$))

\item $\comp{2}{U_{\node}}(((4), (6), (2)))= (0, 0, 0, 0)$

(2, ($g$(2, ($\alpha_{(0)}$, (0), $\alpha_{(0, 0)}$, $-1$)), (1), $g$(2, ($\alpha_{(0, 0, 0)}$, $-1$)), (0), $g$(2, ($\alpha_{(0, 0)}$, (0), $\alpha_{(0, 0, 0)}$, $-1$)), $-1$))

\item $\comp{2}{U_{\node}}(((4), (6), (1)))= (0, 0, 1)$

(2, ($g$(2, ($\alpha_{(0)}$, (0), $\alpha_{(0, 0)}$, $-1$)), (1), $g$(2, ($\alpha_{(0, 0, 0)}$, $-1$)), (0), $g$(2, ($\alpha_{(0, 0)}$, $-1$)), $-1$))

\item $\comp{2}{U_{\node}}(((4), (6), (1), (1)))= (0, 0, 1, 0)$

(2, ($g$(2, ($\alpha_{(0)}$, (0), $\alpha_{(0, 0)}$, $-1$)), (1), $g$(2, ($\alpha_{(0, 0, 0)}$, $-1$)), (0), $g$(2, ($\alpha_{(0, 0, 1)}$, (0))), $-1$))

\item $\comp{2}{U_{\node}}(((4), (6), (1), (0)))= (0, 0, 0, 0)$

(2, ($g$(2, ($\alpha_{(0)}$, (0), $\alpha_{(0, 0)}$, $-1$)), (1), $g$(2, ($\alpha_{(0, 0, 0)}$, $-1$)), (0), $g$(2, ($\alpha_{(0, 0, 1)}$, (0), $\alpha_{(0, 0, 0)}$, $-1$)), $-1$))

\item $\comp{2}{U_{\node}}(((4), (6), (0)))= (0, 0, 0, 0)$

(2, ($g$(2, ($\alpha_{(0)}$, (0), $\alpha_{(0, 0)}$, $-1$)), (1), $g$(2, ($\alpha_{(0, 0, 0)}$, $-1$)), (0), $g$(2, ($\alpha_{(0, 0, 0)}$, (0))), $-1$))

\item $\comp{2}{U_{\node}}(((4), (5)))= (0, 0, 0)$

(2, ($g$(2, ($\alpha_{(0)}$, (0), $\alpha_{(0, 0)}$, $-1$)), (1), $g$(1, ((0))), (0)))

\item $\comp{2}{U_{\node}}(((4), (4)))= (0, 1)$

(2, ($g$(2, ($\alpha_{(0)}$, (0), $\alpha_{(0, 0)}$, $-1$)), (1), $g$(1, ((0))), (0), $g$(2, ($\alpha_{(0)}$, $-1$)), $-1$))

\item $\comp{2}{U_{\node}}(((4), (4), (1)))= (0, 1, 0)$

(2, ($g$(2, ($\alpha_{(0)}$, (0), $\alpha_{(0, 0)}$, $-1$)), (1), $g$(1, ((0))), (0), $g$(2, ($\alpha_{(0, 1)}$, (0))), $-1$))

\item $\comp{2}{U_{\node}}(((4), (4), (0)))= (0, 0, 0)$

(2, ($g$(2, ($\alpha_{(0)}$, (0), $\alpha_{(0, 0)}$, $-1$)), (1), $g$(1, ((0))), (0), $g$(2, ($\alpha_{(0, 1)}$, (0), $\alpha_{(0, 0)}$, $-1$)), $-1$))

\item $\comp{2}{U_{\node}}(((4), (3)))= (0, 0, 0)$

(2, ($g$(2, ($\alpha_{(0)}$, (0), $\alpha_{(0, 0)}$, $-1$)), (1), $g$(1, ((0))), (0), $g$(2, ($\alpha_{(0, 0)}$, (0))), $-1$))

\item $\comp{2}{U_{\node}}(((4), (2)))= (0, 0, 0)$

(2, ($g$(2, ($\alpha_{(0)}$, (0), $\alpha_{(0, 0)}$, $-1$)), (1), $g$(1, ((0))), (0), $g$(2, ($\alpha_{(0, 0)}$, $-1$)), $-1$))

\item $\comp{2}{U_{\node}}(((4), (2), (0)))= (0, 0, 0, 0)$

(2, ($g$(2, ($\alpha_{(0)}$, (0), $\alpha_{(0, 0)}$, $-1$)), (1), $g$(1, ((0))), (0), $g$(2, ($\alpha_{(0, 0, 0)}$, (0))), $-1$))

\item $\comp{2}{U_{\node}}(((4), (1)))= -1$

(2, ($g$(2, ($\alpha_{(0)}$, (0), $\alpha_{(0, 0)}$, $-1$)), (1), $g$(1, ((0))), $-1$))

\item $\comp{2}{U_{\node}}(((4), (0)))= -1$

(2, ($g$(2, ($\alpha_{(0)}$, (0), $\alpha_{(0, 0)}$, $-1$)), (0)))

\item $\comp{2}{U_{\node}}(((3)))= (0, 0)$

(2, ($g$(2, ($\alpha_{(0)}$, $-1$)), (1)))

\item $\comp{2}{U_{\node}}(((3), (3)))= (0, 1)$

(2, ($g$(2, ($\alpha_{(0)}$, $-1$)), (1), $g$(2, ($\alpha_{(0, 0)}$, (0))), (0)))

\item $\comp{2}{U_{\node}}(((3), (3), (1)))= (0, 1, 0)$

(2, ($g$(2, ($\alpha_{(0)}$, $-1$)), (1), $g$(2, ($\alpha_{(0, 0)}$, (0))), (0), $g$(2, ($\alpha_{(0, 1)}$, (0))), $-1$))

\item $\comp{2}{U_{\node}}(((3), (3), (0)))= (0, 0, 0)$

(2, ($g$(2, ($\alpha_{(0)}$, $-1$)), (1), $g$(2, ($\alpha_{(0, 0)}$, (0))), (0), $g$(2, ($\alpha_{(0, 1)}$, (0), $\alpha_{(0, 0)}$, $-1$)), $-1$))

\item $\comp{2}{U_{\node}}(((3), (2)))= (0, 0, 0)$

(2, ($g$(2, ($\alpha_{(0)}$, $-1$)), (1), $g$(2, ($\alpha_{(0, 0)}$, (0))), $-1$))

\item $\comp{2}{U_{\node}}(((3), (2), (1)))= (0, 1)$

(2, ($g$(2, ($\alpha_{(0)}$, $-1$)), (1), $g$(2, ($\alpha_{(0, 0)}$, (0), $\alpha_{(0, 0, 0)}$, $-1$)), (0)))

\item $\comp{2}{U_{\node}}(((3), (2), (1), (2)))= (0, 1, 0)$

(2, ($g$(2, ($\alpha_{(0)}$, $-1$)), (1), $g$(2, ($\alpha_{(0, 0)}$, (0), $\alpha_{(0, 0, 0)}$, $-1$)), (0), $g$(2, ($\alpha_{(0, 1)}$, (0))), $-1$))

\item $\comp{2}{U_{\node}}(((3), (2), (1), (1)))= (0, 0, 1)$

(2, ($g$(2, ($\alpha_{(0)}$, $-1$)), (1), $g$(2, ($\alpha_{(0, 0)}$, (0), $\alpha_{(0, 0, 0)}$, $-1$)), (0), $g$(2, ($\alpha_{(0, 1)}$, (0), $\alpha_{(0, 0)}$, $-1$)), $-1$))

\item $\comp{2}{U_{\node}}(((3), (2), (1), (0)))= (0, 0, 0, 0)$

(2, ($g$(2, ($\alpha_{(0)}$, $-1$)), (1), $g$(2, ($\alpha_{(0, 0)}$, (0), $\alpha_{(0, 0, 0)}$, $-1$)), (0), $g$(2, ($\alpha_{(0, 1)}$, (0), $\alpha_{(0, 0, 0)}$, $-1$)), $-1$))

\item $\comp{2}{U_{\node}}(((3), (2), (0)))= (0, 0, 1)$

(2, ($g$(2, ($\alpha_{(0)}$, $-1$)), (1), $g$(2, ($\alpha_{(0, 0)}$, (0), $\alpha_{(0, 0, 0)}$, $-1$)), (0), $g$(2, ($\alpha_{(0, 0)}$, (0))), $-1$))

\item $\comp{2}{U_{\node}}(((3), (1)))= (0, 1)$

(2, ($g$(2, ($\alpha_{(0)}$, $-1$)), (1), $g$(2, ($\alpha_{(0, 0)}$, $-1$)), (0)))

\item $\comp{2}{U_{\node}}(((3), (1), (1)))= (0, 1, 0)$

(2, ($g$(2, ($\alpha_{(0)}$, $-1$)), (1), $g$(2, ($\alpha_{(0, 0)}$, $-1$)), (0), $g$(2, ($\alpha_{(0, 1)}$, (0))), $-1$))

\item $\comp{2}{U_{\node}}(((3), (1), (0)))= (0, 0, 0)$

(2, ($g$(2, ($\alpha_{(0)}$, $-1$)), (1), $g$(2, ($\alpha_{(0, 0)}$, $-1$)), (0), $g$(2, ($\alpha_{(0, 1)}$, (0), $\alpha_{(0, 0)}$, $-1$)), $-1$))

\item $\comp{2}{U_{\node}}(((3), (0)))= (0, 0, 0)$

(2, ($g$(2, ($\alpha_{(0)}$, $-1$)), (1), $g$(2, ($\alpha_{(0, 0)}$, $-1$)), (0), $g$(2, ($\alpha_{(0, 0)}$, (0))), $-1$))

\item $\comp{2}{U_{\node}}(((2)))= (0, 0)$

(2, ($g$(2, ($\alpha_{(0)}$, $-1$)), (1), $g$(1, ((0))), (0)))

\item $\comp{2}{U_{\node}}(((2), (0)))= (0, 0, 0)$

(2, ($g$(2, ($\alpha_{(0)}$, $-1$)), (1), $g$(1, ((0))), (0), $g$(2, ($\alpha_{(0, 0)}$, (0))), $-1$))

\item $\comp{2}{U_{\node}}(((1)))= -1$

(2, ($g$(2, ($\alpha_{(0)}$, $-1$)), (1), $g$(1, ((0))), $-1$))

\item $\comp{2}{U_{\node}}(((0)))= -1$

(2, ($g$(2, ($\alpha_{(0)}$, $-1$)), (0)))

\item $\comp{1}{U}$ has the node $(3)$

(2, ($\alpha_{(0)}$, (1)))

\item $\comp{1}{U}$ has the node $(2)$

(2, ($\alpha_{(0)}$, (0)))

\item $\comp{1}{U}$ has the node $(1)$

(2, ($\alpha_{(0)}$, $-1$))

\item $\comp{1}{U}$ has the node $(0)$

(1, ((0)))
\end{etaremune}

This algorithm of producing $U$ can be abstracted into a combinatorial one, without referring to clubs in $\omega_1$ at all. If one tries to apply universality of $U$  to ($X$, $\theta_{XT}$, $\vec{\beta}$, $\pi$), the resulting factoring map $\psi$ of $X$ into $U$ is
\begin{itemize}
\item $\psi(2,((1))) = (2,((3)))$,
\item $\psi(2,((1),(0))) = (2, ((3),(1)))$,
\item $\psi(2,((0))) = (2, ((0)))$,
\item $\psi(1,(0)) = (1,(0))$.
\end{itemize}

In Section~\ref{sec:more-level-2}, we will show that the equality
\begin{displaymath}
  U = T \otimes Q
\end{displaymath}
correspond to the equalities
\begin{displaymath}
  \mathbb{L}_{\bolddelta{3}}[j^U(T_2)] = \mathbb{L}_{\bolddelta{3}}[j^T \circ j^Q (T_2)]
\end{displaymath}
and
\begin{displaymath}
j^U = j^Q \circ j^T
\end{displaymath}

To summarize, we have given an example of a general fact on the level-($\leq 2$, $\leq 2$, $\leq 2$) factoring maps. Given finite level $\leq 2$ trees $X,T$, 
there is another level $\leq 2$ tree $Q$, a tuple of ordinals $\vec{\beta}$ respecting $Q$ so that the $X$-equivalence class of $\theta_{XT}$ has a purely combinatorial characterization, called a map $\pi$ which factors $(X,T,Q)$, and $X$ embeds into $U = T \otimes Q$ via a level $\leq 2$ tree factoring map $\psi$. 
In terms of ultrapowers, $\pi$ induces $\pi^{T,Q} : \mathbb{L}_{\bolddelta{3}}[j^X(T_2)] \to \mathbb{L}_{\bolddelta{3}}[j^T \circ j^Q (T_2)] = \mathbb{L}_{\bolddelta{3}}[j^U(T_2)]$ such that $\pi^{T,Q} \circ j^X = j^T  \circ j^Q = j^U$, and $\pi^{T,Q}$ is just the embedding induced by the measure projection of $\mu^U$ to $\mu^X$, which is  generated by $\psi$. 

\subsection{The entire scenario up to level-3}
\label{sec:entire-scenario-up}

The rigorous definitions and proofs of this level-2 scenario and the level-3 scenario will be in Section~\ref{sec:level-2-analysis}. The central definitions form a stack, illustrated in Fig.~\ref{fig:1} In this figure, $P,W,S$ are level-1 trees, $Q,T,X$ are level $\leq 2$ trees, and $R,Y$ are level-3 trees. An arrow stands for a factoring map. A solid line stands for membership, that is, $P$ is the tree-component of an entry of $Q$, $S$ is the tree-component of an entry of $T$, etc. The stack of definitions consists of:
\begin{itemize}
\item $W$-description (defined in \cite{sharpII}), the set of $W$-descriptions is $W \cup \se{\emptyset}$, 
\item $(P,W)$-factoring map (defined in \cite{sharpII}) , a $(P,W)$-factoring map is a map $\sigma$ from $P \cup \se{\emptyset}$ to the set of $W$-descriptions,
\item $(Q,W)$-description (Definition~\ref{def:level-2_W_Q_description}), a $(Q,W)$-description $\mathbf{D}$ typically consists of a node $(2,q) \in \dom(Q)$ and a $(P,W)$-factoring map, where $P = \comp{2}{Q}_{\tree}(q)$,
\item $(S,Q,W)$-factoring map (Definition~\ref{def:factoring_2}), an $(S,Q,W)$-factoring map is a map $\tau$ from $S \cup \se{\emptyset}$ to the set of $(Q,W)$-descriptions,
\item $(T,Q,W)$-description (Definition~\ref{def:description_TQW}), a $(T,Q,W)$-description $\mathbf{C}$ typically consists of a node $(2,t) \in \dom(T)$ and a $(S,Q,W)$-factoring map, where $S = \comp{2}{T}_{\tree}(t)$,
\item $(X,T,Q)$-factoring map (Definition~\ref{def:factoring_3}), a $(X,T,Q)$-factoring map is a map $\pi$ on $\dom(X)$ which sends each $(2,x) \in \dom(X)$ to a $(T,Q,W)$-description, where $W = \comp{2}{X}_{\tree}(x)$,
\item $(Y,T,Q)$-description (Definition~\ref{def:description_TQY}), a $(Y,T,Q)$-description $\mathbf{B}$ typically consists of a node $y \in \dom(Y)$ and a $(X,T,Q)$-factoring map, where $X= Y_{\tree}(y)$,
\item $(R,Y,T)$-factoring map (Definition~\ref{def:factoring_4}), a $(R,Y,T)$-factoring map is a map $\rho$ on $\dom(R) \cup \se{\emptyset}$ which sends each $r \in \dom(R)$ to a $(Y,T,Q)$-description, where $Q = R_{\tree}(r)$. 
\end{itemize}

  \begin{figure}
    \centering
  \begin{tikzpicture}
    \matrix  (m)  [matrix of math nodes, row sep=1.5em, column sep = 4em]
{R & & Y \\
     & T &  \\
Q & & X \\
  & S & \\
P & & W \\
};
\draw[->, very thick](m-1-1) edge node[auto] {$\rho$} (m-1-3) 
(m-5-1) edge  node[auto] {$\sigma$}(m-5-3) 
(m-3-3) edge  node[auto] {$\pi$}(m-2-2) 
(m-4-2) edge  node[auto] {$\tau$}(m-3-1);
\draw[-] (m-1-1) edge (m-3-1)
 (m-3-1) edge (m-5-1)
 (m-1-3) edge (m-3-3)
 (m-3-3) edge (m-5-3)
 (m-2-2) edge (m-4-2);
\draw[dotted]  (-3,1.5) -- (6 , 1.5); 
\draw[dotted]  (-3,-.5) -- (6 , -.5); 
\node at (4.5,2) {level-3};
\node at (4.5,.4) {level-2};
\node at (4.5,-1.3) {level-1};
  \end{tikzpicture}
\caption{The stack of definitions}
\label{fig:1}
     \end{figure}
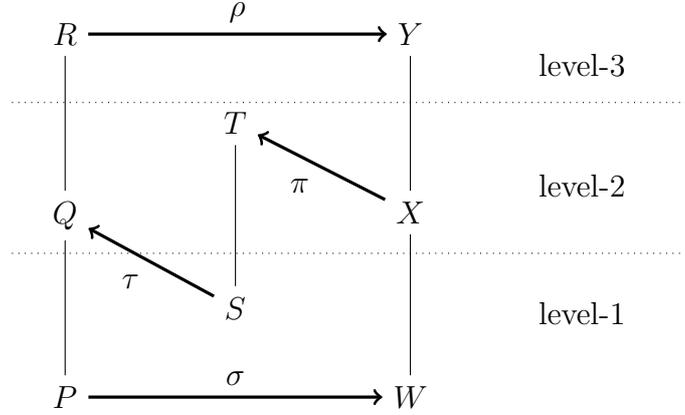

\subsection{Level-(3,3,2) factoring maps}
\label{sec:level-3-3}

We give an example on the level-(3,3,2) factoring map. 
Define two level-3 trees $R,Y$:

$\dom(R) = \se{((0)),((0),(1)),((0),(0)),((0),(0),(0))}$, 
$R(((0))) = (Q^0, (2, (0), \se{(0)}))$,
$Q^0$ is the unique level $\leq 2$ with cardinality 1, $R(((0),(1))) = (Q^{21} , (2,((0)),\se{(0)}))$,
$\comp{1}{Q}^{21} = \emptyset$, 
$\dom(\comp{2}{Q}^{21} )= \se{\emptyset, ((0))}$, 
$\comp{2}{Q}^{21}(((0))) = ( \se{(0)}, (0,0))$. 
$R(((0),(0))) = (Q^{20} , (1,((0)),\emptyset))$,
$\comp{1}{Q}^{20} = \emptyset$, 
$\dom(\comp{2}{Q}^{20} )= \se{\emptyset, ((0))}$, 
$\comp{2}{Q}^{20}(((0))) = ( \se{(0)}, -1)$. 
$R(((0),(0),(0))) = (Q^{*}, (0,-1,\emptyset)$,
 $Q^{*}$ is the unique completion of $R(((0),(0)))$. 

$\dom(Y) = \se{((3)),((2)),((1)), ((1),(0)),((0))}$, 
$Y(((3))) = Y(((0))) = (Q^0, (0,-1,\emptyset))$,
$Y(((2))) = (Q^0, (1, (0), \emptyset))$,
$Y(((1)))=(Q^0, (2, ((0)), \se{(0)}))$, 
$Y(((1),(0))) = (Q^{21}, (2, ((0,0)), \se{(0),(0,0)}))$.

$\rep(R)$ consists of
\begin{itemize}
\item $((0))$, 
 \item $((0), \comp{2}{\beta}_{((0))}, (1))$ for $\vec{\beta}$ respecting $Q^{21}$,  i.e.\ $\omega_1<\comp{2}{\beta}_{((0))}<\omega_2$ and $\cf(\comp{2}{\beta}_{((0))}) = \omega_1$,
 \item $((0), \comp{2}{\beta}_{((0))}, (1), \gamma, -1)$ for $\vec{\beta}$ respecting $Q^{21}$,  $ \sup(\sigma^W)''\beta < \gamma < \sigma^W(\beta)$, $\gamma$ limit, where $W = \se{(0),(0,0)}$, $\sigma$ factors $(\se{0},W)$, $\sigma$ is the inclusion map $\se{0} \subseteq W$,
 i.e.\ $\sigma^W ( \tau^{L[z]}(z,u_1)) = \tau^{L[z]}(z,u_2)$ for $z \in \mathbb{R}$,
 \item $((0), \comp{2}{\beta}_{((0))}, (0))$ for $\vec{\beta}$ respecting $Q^{20}$, i.e.\ $\omega_1<\comp{2}{\beta}_{((0))}<\omega_2$ and $\cf(\comp{2}{\beta}_{((0))}) = \omega$,
 \item $((0), \comp{2}{\beta}_{((0))}, (0), \comp{1}{\beta}_{(0)}, (0))$ for $\vec{\beta}$ respecting $Q^{*}$, i.e., $\omega_1<\comp{2}{\beta}_{((0))}<\omega_2$ and $\cf(\comp{2}{\beta}_{((0))}) = \omega$, $\comp{1}{\beta}_{(0)}<\omega_1$ limit,
 \item $((0), \comp{2}{\beta}_{((0))}, (0), \comp{1}{\beta}_{(0)}, (0),k, -1)$ for $\vec{\beta}$ respecting $Q^{*}$ and $k<\omega$,
 \item $((0), \comp{2}{\beta}_{((0))}, (0), \gamma, -1)$ for $\vec{\beta}$ respecting $Q^{20}$ and $\gamma < \omega_1$ limit, 
\item $((0),\gamma, -1)$ for $\omega_1<\gamma < u_2$, $\gamma$ limit.
\end{itemize}
$\rep(Y)$ consists of 
\begin{itemize}
\item $((3))$,
\item $((3), k, -1)$ for $k<\omega$,
\item $((2))$,
\item $((2),\gamma,-1)$ for $\gamma < \omega_1$ limit,
\item $((1))$,
\item $((1),  \comp{2}{\beta}_{((0))}, (0))$ for $\vec{\beta}$ respecting $Q^{21}$,  
\item $((1), \comp{2}{\beta}_{((0))}, (0), \gamma, -1)$ for $\vec{\beta}$ respecting $Q^{21}$,  $ \sup(\sigma^W)''\comp{2}{\beta}_{((0))} < \gamma < \sigma^W(\comp{2}{\beta}_{((0))})$, $\gamma$ limit,  where $W = \se{(0),(0,0)}$, $\sigma$ factors $(\se{0},W)$, $\sigma$ is the inclusion map $\se{0} \subseteq W$,
i.e.\ $\sigma^W ( \tau^{L[z]}(z,u_1)) = \tau^{L[z]}(z,u_2)$ for $z \in \mathbb{R}$,
\item $((1),\gamma, -1)$ for $\omega_1<\gamma < u_2$, $\gamma$ limit,
\item $((0))$,
\item $((0), k, -1)$ for $k<\omega$.
\end{itemize}
Then
\begin{align*}
  \theta_R &  : \rep(R) \to u_3\\
    \theta_Y&:\rep(Y) \to u_3 + \omega_1 + \omega
\end{align*}
are the order preserving bijections. Let
\begin{displaymath}
  \theta_{RY} = \theta_Y^{-1} \circ \theta_R.
\end{displaymath}
If $E$ is a club in $\omega_1$ consisting of limit ordinals, recall that  $\rep(R) \res E$ consists of 
\begin{itemize}
\item $((0))$, 
 \item $((0), \comp{2}{\beta}_{((0))}, (1))$ for  $\vec{\beta} \in [E]^{Q^{21} \uparrow}$, 
 \item $((0), \comp{2}{\beta}_{((0))}, (1), \gamma, -1)$ for  $\vec{\beta} \in [E]^{Q^{21} \uparrow}$,  $ \sup(\sigma^W)''\beta < \gamma < \sigma^W(\beta)$,  $\gamma \in j^{\se{(0),(0,0)}}(E)$, 
 \item $((0), \comp{2}{\beta}_{((0))}, (0))$  for $\vec{\beta} \in [E]^{Q^{20} \uparrow}$, 
 \item $((0), \comp{2}{\beta}_{((0))}, (0), \comp{1}{\beta}_{(0)}, (0))$ for $\vec{\beta} \in [E]^{Q^{*} \uparrow}$,
 \item $((0), \comp{2}{\beta}_{((0))}, (0), \comp{1}{\beta}_{(0)}, (0),k, -1)$ for for $\vec{\beta} \in [E]^{Q^{*} \uparrow}$ and $k<\omega$,
 \item $((0), \comp{2}{\beta}_{((0))}, (0), \gamma, -1)$ for $\vec{\beta} \in [E]^{Q^{*} \uparrow}$ and $\gamma \in E$,
\item $((0),\gamma, -1)$ for $\omega_1<\gamma < u_2$, $\gamma\in E$.
\end{itemize}
 $\rep(R) \res E$ is a closed subset of $\rep(R)$. 
We say that a continuous, order preserving $\theta$ is equivalent to $\theta'$ iff there is a club $E$ in $\omega_1$ such that $\theta \res (\rep(R) \res E) = \theta' \res (\rep(R) \res E) $. 
Thus, $\theta$ is equivalent to $\theta_{RY}$ iff there is a club $E$ in $\omega_1$ such that $\theta \res (\rep(R) \res E)  $ is continuous, order-preserving and
\begin{itemize}
\item $\theta(((0))) = ((1))$,
\item $\theta(((0), \comp{2}{\beta}_{((0))}, (1))) = ((1), \comp{2}{\beta}_{((0))}, (0))$ for $\vec{\beta} \in [E]^{Q^{20} \uparrow}$,
\item $\theta(((0), \comp{2}{\beta}_{((0))}, (0))) = ((1),  \comp{2}{\beta}_{((0))} + \omega_1,  -1 )$  for $\vec{\beta} \in [E]^{Q^{20} \uparrow}$,
\item $\theta(((0), \comp{2}{\beta}_{((0))} , (0), \comp{1}{\beta}_{(0)}, (0))) = ((1),  \comp{2}{\beta}_{((0))} + \comp{1}{\beta}_{(0)} + \omega , -1 )$  for $\vec{\beta} \in [E]^{Q^{*} \uparrow}$,
\end{itemize}
In this characterization, the following functions show up:
\begin{itemize}
\item $\vec{\beta} \mapsto \comp{2}{\beta}_{((0))}$ for  $\vec{\beta} \in [E]^{Q^{20} \uparrow}$,
\item $\vec{\beta} \mapsto \comp{2}{\beta}_{((0))}+\omega_1$ for  $\vec{\beta} \in [E]^{Q^{20} \uparrow}$,
\item $\vec{\beta} \mapsto \comp{2}{\beta}_{((0))} + \comp{1}{\beta}_{(0)} + \omega$ for  $\vec{\beta} \in [E]^{Q^{*} \uparrow}$,
\end{itemize}
Throw away the redundant coordinates in the variables of these functions and remove duplicates. In this example, nothing is thrown away. 
 We do the following operation on the these three functions:
\begin{itemize}
\item 
The function  $\vec{\beta} \mapsto \comp{2}{\beta}_{((0))}$ represents $u_2$ modulo $\mu^{Q^{20}}$, a fact to be proved in Section~\ref{sec:level-2-analysis} using the analysis of level ($\leq 2$, 1)-descriptions (the same for the other two functions).
The representative function for $u_2$ under $\mu^{\se{(1),(0)}}$ is $\vec{\alpha} \mapsto \alpha_{(1)}$. Remove the unused coordinate $\alpha_{(0)}$. The resulting function becomes $\vec{\alpha} \mapsto \alpha_{(0)}$ for $\alpha \in [E]^{\se{(0)\uparrow}}$. It is of continuous type. Further remove the continuous coordinate $\alpha_{(0)}$ by taking sup over it. The resulting function becomes $* \mapsto \sup_{\alpha_{(0)}<\omega_1}\alpha_{(0)} = \omega_1$.
\item The function  $\vec{\beta} \mapsto \comp{2}{\beta}_{((0))}+\omega_1$ represents $u_2 + \omega_1$ modulo $\mu^{Q^{20}}$ for  $\vec{\beta} \in [E]^{Q^{20} \uparrow}$. 
The representative function for $u_2 + \omega_1$  under $\mu^{\se{(1),(0)}}$ is $\vec{\alpha} \mapsto \alpha_{(1)}+\alpha_{(0)}$. It is of continuous type. Further remove the continuous coordinate $\alpha_{(0)}$ by taking sup over it. The resulting function becomes $\vec{\alpha} \mapsto \sup_{\beta<\alpha_{(0)}} (\alpha_{(0)} + \beta) = \alpha_{(0)} 2$ for $\vec{\alpha} \in [E]^{\se{(0)\uparrow}}$. It represents $\omega_1 2$ in the $\mu^{\se{(0)}}$-ultrapower.
\item The function  $\vec{\beta} \mapsto \comp{2}{\beta}_{((0))}+  \comp{1}{\beta}_{(0)} + \omega$ represents $u_4 + \omega_1 + \omega$  modulo $\mu^{Q^{21}}$. The representative function for $u_4 + \omega_1+\omega$  under $\mu^{\se{(3),(2),(1),(0)}}$ is $\vec{\alpha} \mapsto \alpha_{(3)}+\alpha_{(0)} + \omega$. Remove the unused coordinates $\alpha_{(2)},\alpha_{(1)}$. The resulting function becomes $\vec{\alpha} \mapsto \alpha_{(1)} + \alpha_{(0)} + \omega$ for $\alpha \in [E]^{\se{(1),(0)\uparrow}}$.
It is of dis continuous type.  It represents $\omega_2 + \omega_1 + \omega $ in the $\mu^{\se{(0),(1)}}$-ultrapower.
\end{itemize}
 After all these operations, we end up with a tuple that respects a level $\leq 2$ tree. Let $T$ be a level $\leq 2$ tree where $\comp{1}{T} = \emptyset$, $\comp{2}{T}$ has domain $\se{\emptyset, ((0)), ((0),(0))}$, $\comp{2}{T}(((0))) = (\se{(0)}, (0,0))$, $\comp{2}{T}(((0),(0))) = (\se{(0),(0,0)}, -1)$. $\vec{\xi}$ respects $T$, where $\comp{2}{\xi}_{\emptyset} = \omega_1$, $\comp{2}{\xi}_{((0))} = \omega_1 2$, $\comp{2}{\xi}_{((0),(0))} = u_2+\omega_1+\omega$.

Now the $R$-equivalence class of $\theta_{RY}$ is decided by $T$ and $\vec{\xi}$. Fix $h \in [\omega_1]^{T \uparrow}$ with $[h]^T = \vec{\xi}$. Then a continuous, order-preserving $\theta$ is $R$-equivalent to $\theta_{RY}$ iff there is a club $E$ in $\omega_1$ such that 
Then, an order-preserving, continuous  $\theta$ is $R$-equivalent to $\theta_{RY}$ iff
\begin{itemize}
\item $\theta(((0))) = ((1))$,
\item $\theta(((0), \comp{2}{\beta}_{((0))}, (1))) =$
 ((1), $[ \vec{\alpha} \mapsto h$(2, ($g$(2, ($\alpha_{(0)}$, (0))), $-1$))$]_{\mu_{\vec{\alpha}}}$, (0))
 for $\vec{\beta}  = [g]^{Q^{21}}$, $g \in E^{Q^{21} \uparrow}$,
\item $\theta(((0), \comp{2}{\beta}_{((0))}, (0))) =$ 
((1), $[ \vec{\alpha} \mapsto h$(2, ($g$(2, ($\alpha_{(0)}$, (0))), (0), $g$(2, ($\alpha_{(0)}$, $-1$)), $-1$))$]_{\mu_{\vec{\alpha}}}$, $-1$)
 for $\vec{\beta}  = [g]^{Q^{20}}$, $g \in E^{Q^{20} \uparrow}$,
\item $\theta(((0), \comp{2}{\beta}_{((0))} , (0), \comp{1}{\beta}_{(0)}, (0))) = $
((1), $[ \vec{\alpha} \mapsto h$(2, ($g$(2, ($\alpha_{(0)}$, (0))), (0), $g$(1, ((0))), (0)))$]_{\mu_{\vec{\alpha}}}$, $-1$)
 for $\vec{\beta}  = [g]^{Q^{*}}$, $g \in E^{Q^{*} \uparrow}$,
\end{itemize}
Here $\mu_{\vec{\alpha}}$ is a shorthand for $\mu^W$ where $W$ is a level-1 tree consisting of nodes that appear in subscripts of $\alpha$. In all of the above examples, only $\alpha_{(0)}$ show up in each expression, so $\mu_{\vec{\alpha}} = \mu^{\se{(0)}}$ in each expression.  
This combinatorial characterization of $\theta_{RY}$ will be abstracted into a finitary object
\begin{displaymath}
\rho
\end{displaymath}
which factors $(R,Y,T)$, to be defined in Section~\ref{sec:fact-betw-level}. The main property of $\rho$ is that it induces a map
\begin{displaymath}
  \rho^{Y,T} : \mathbb{L}_{j^R(\bolddelta{3})}[j^R(T_3)] \to \mathbb{L}_{j^Y \circ j^T(\bolddelta{3})}[j^Y \circ j^T (T_3)]
\end{displaymath}
such that $\rho^{Y,T} \circ j^R = j^Y \circ j^T$ and $\rho^{Y,T}$ is elementary on any submodel of ZFC.

For illustration purposes, the tensor product, $Y \otimes T$, will be a level-3 tree $U$ (up to an isomorphism) such that $\theta_{UY}$ is universal in the sense that
\begin{enumerate}
\item the $U$-equivalence class of $\theta_{UY}$ has a combinatorial characterization decided by $(T, \vec{\xi}, \rho^{*})$, where $\vec{\beta}$ respects $Q$, $\rho^{*}$ factors $(U,Y,T)$, and
\item given $(U',\theta',\vec{\xi}',\rho)$, where $U'$ is another level-3 tree, $\theta' : \rep(U') \to \rep(T)$ is continuous, order-preserving, and the $U'$-equivalence class of $\theta'$ has a combinatorial characterization decided by $(\vec{\xi}', \rho)$, where $\vec{\xi}'$ respects $T$, $\rho$ factors $(U',Y,T)$, $\theta'$ is a representative of the $U$-equivalence class decided by $(\vec{\xi}', \rho')$, then there is a level-3 tree factoring map $\psi$ of $U'$ into $U$ such that
  \begin{displaymath}
    [s \mapsto \theta'(s)]_{U_{\tree}(s)}    =    [s \mapsto \theta'(\psi(s))]_{U_{\tree}(s)}
  \end{displaymath}
for any $s \in \dom(U')$. 
\end{enumerate}
In Section~\ref{sec:boldface-level-3_sharp}, we will show that
\begin{displaymath}
   \mathbb{L}_{j^Y \circ j^T(\bolddelta{3})}[j^Y \circ j^T (T_3)] = \mathbb{L}_{j^U(\bolddelta{3})}[j^{U \oplus T}(T_3)]
\end{displaymath}
and
\begin{displaymath}
  j^{U \oplus T} = j^Y \circ j^T,
\end{displaymath}
where $j^{U \oplus T}$ is the ultrapower map of $\mu^{U \oplus T}$, the product measure of $\mu^U$ and $\mu^T$. 

The algorithm of producing such a $U$ is quite similar to the tensor product of level ($\leq 2$, $\leq 2$) trees. Observe that in the characterization of the $R$-equivalent class of $\theta_{RY}$, the right hand side of the equation are always of the form
\begin{displaymath}
  (y(0), [\vec{\alpha} \mapsto h(g_0(\vec{\alpha}))]_{\mu_{\vec{\alpha}}}, y(1),\dots, [\vec{\alpha} \mapsto h(g_k(\vec{\alpha}))]_{\mu_{\vec{\alpha}}}, y(k))
\end{displaymath}
where
\begin{enumerate}
\item the $\mu_{\vec{\alpha}}$ in $ [\vec{\alpha} \mapsto h(g_i(\vec{\alpha}))]_{\mu_{\vec{\alpha}}}$ stands for $\mu^W$ where $W$ is the set of indices of $\alpha$ that shows up in $g_i(\vec{\alpha})$, or equivalently, $W$ is the third component of $Y_{\node}(y \res i+1)$;
\item The subscript $i$ in $g_i$ transforms $g : \rep(Q) \to \omega_1$ to $g_i : [\omega_1]^{W \uparrow} \to \rep(Y)$ (Recall that $g$ was the function with $[g]^Q = \vec{\beta}$ on the left hand side to start with). $g_i$ takes the form of either
  \begin{displaymath}
    g_i(\vec{\alpha}) = (1, (g_{i0}(\vec{\alpha})))
  \end{displaymath}
(which doesn't happen in this particular example as $\comp{1}{T} = \emptyset$)
or
  \begin{displaymath}
    g_i(\vec{\alpha}) = (2, (g_{i0}(\vec{\alpha}), t(0), \dots, g_{ik_i}(\vec{\alpha}), t(k_i)  ))
  \end{displaymath}
where each $g_{ij}$ takes the form of either 
\begin{displaymath}
  g_{ij}(\vec{\alpha}) = g (1, (q_{ij}))
\end{displaymath}
or
\begin{displaymath}
  g_{ij}(\vec{\alpha}) = g (2, (\alpha_{w_{ij0}}, q_{ij0},\dots)).
\end{displaymath}
\end{enumerate}
We need to write down all the legal expressions of such kind modulo a level $\leq 2$ isomorphism of $Q$. Legal means that:
\begin{enumerate}
\item The whole expression must be in $\rep(Y)$.
\item For every $(1,q) \in \dom(Q)$, $g(1,(q))$ must show up at least once in one of the $g_{ij}$'s. For every $(2,q) \in \dom(Q)$, there must be some $(2,q') \in \dom(Q)$ with $q \subseteq q'$, $\lh(q') = l$ such that $g(2, (\alpha_{*},q'(0),\dots, \alpha_{*},q'(l)))$ shows up at least once in one of the $g_{ij}$'s.
\item The whole expression must be of discontinuous type. That is, there is a club $E$ in $\omega_1$ such that for any $g \in E^{Q \uparrow}$,
  \begin{displaymath}
  (y(0), [\vec{\alpha} \mapsto h(g_0(\vec{\alpha}))]_{\mu_{\vec{\alpha}}}, y(1),\dots, [\vec{\alpha} \mapsto h(g_k(\vec{\alpha}))]_{\mu_{\vec{\alpha}}}, y(k))
  \end{displaymath}
is bigger than the sup of 
  \begin{displaymath}
  (y(0), [\vec{\alpha} \mapsto h(g'_0(\vec{\alpha}))]_{\mu_{\vec{\alpha}}}, y(1),\dots, [\vec{\alpha} \mapsto h(g'_k(\vec{\alpha}))]_{\mu_{\vec{\alpha}}}, y(k))
  \end{displaymath}
that are $<^Y 
  (y(0), [\vec{\alpha} \mapsto h(g_0(\vec{\alpha}))]_{\mu_{\vec{\alpha}}}, y(1),\dots, [\vec{\alpha} \mapsto h(g_k(\vec{\alpha}))]_{\mu_{\vec{\alpha}}}, y(k)) $.
\end{enumerate}

These expressions naturally form a level-3 tree $U$. Every expression $h([g]^Q)$ corresponds to a node $\tilde{h}$ in the level-3 tree $U$. 
If the domain of $h$ consists of tuples of the form $[g]^Q$, the tree component of $U(\tilde{h})$ is isomorphic to $Q$. The node component of $U(\tilde{h})$ is decided by the uniform cofinality of $h$. 
As in the level $\leq 2$ case, if $h([g]^Q),h'([g]^{Q'})$ are expressions, where $Q'$ is a one-node extension of $Q'$, then $\tilde{h'}$ is a one-node extension of $\tilde{h}$ iff there is a club $E$ in $\omega_1$ such that for any $g \in E^{Q\uparrow}$,
\begin{displaymath}
  h([g]^Q) = \sup_{<^Y} \set{h'([g']^{Q'})}{g' \text{ extends }g}.
\end{displaymath}

To rigorously prove that this works, one would need the theory of descriptions in Section~\ref{sec:level-2-analysis}. For the moment, the reader may skip the verification part only to get a feeling of what is going on and come back later after finishing Section~\ref{sec:level-2-analysis}. 

Here is a representation $U$ of $Y \otimes T$. $U$ has cardinality 147. The entries of $U$ and corresponding expressions are as follows. 

\begin{etaremune}
\item $U(((6)))$  has degree 0

( (3) )

\item $U(((5)))$  has degree 1

( (2) )

\item $U_{\tree}(((5), (0)))=$ the unique completion of $U(((5)))$

\noindent$U_{\node}((5), (0))= (1, (0, 0))$

( (2), $h$(2, ($g$(1, (0)), (0))), $-1$ )

\item $U_{\tree}(((5), (0), (0)))=$ the unique completion of $U(((5), (0)))$

\noindent$U_{\node}((5), (0), (0))= (0, -1)$

( (2), $h$(2, ($g$(1, (0)), (0), $g$(1, (0, 0)), (0))), $-1$ )

\item $U(((4)))$  has degree 2

( (1) )

\item $U_{\tree}(((4), (23)))=$ the completion of $U(((4)))$ that sends $(2, ((0)))$ to the completion of $U_{\tree}(((4)))(2, \emptyset)$  whose node component is $(0, 0)$

\noindent$U_{\node}((4), (23))= (2, ((0), (0)))$

( (1), $[ \vec{\alpha} \mapsto h$(2, ($g$(2, ($\alpha_{(0)}$, (0))), (0)))$]_{\mu_{\vec{\alpha}}}$, (0) )

\item $U_{\tree}(((4), (23), (2)))=$ the completion of $U(((4), (23)))$ that sends $(2, ((0), (0)))$ to the completion of $U_{\tree}(((4), (23)))(2, ((0)))$  whose node component is $(0, 1)$

\noindent$U_{\node}((4), (23), (2))= (0, -1)$

( (1), $[ \vec{\alpha} \mapsto h$(2, ($g$(2, ($\alpha_{(0)}$, (0))), (0)))$]_{\mu_{\vec{\alpha}}}$, (0), $[ \vec{\alpha} \mapsto h$(2, ($g$(2, ($\alpha_{(0)}$, (0))), (0), $g$(2, ($\alpha_{(0)}$, (0), $\alpha_{(0, 0)}$, (0))), (0)))$]_{\mu_{\vec{\alpha}}}$, $-1$ )

\item $U_{\tree}(((4), (23), (1)))=$ the completion of $U(((4), (23)))$ that sends $(2, ((0), (0)))$ to the completion of $U_{\tree}(((4), (23)))(2, ((0)))$  whose node component is $(0, 0, 0)$

\noindent$U_{\node}((4), (23), (1))= (0, -1)$

( (1), $[ \vec{\alpha} \mapsto h$(2, ($g$(2, ($\alpha_{(0)}$, (0))), (0)))$]_{\mu_{\vec{\alpha}}}$, (0), $[ \vec{\alpha} \mapsto h$(2, ($g$(2, ($\alpha_{(0)}$, (0))), (0), $g$(2, ($\alpha_{(0)}$, (0), $\alpha_{(0, 0)}$, (0))), (0)))$]_{\mu_{\vec{\alpha}}}$, $-1$ )

\item $U_{\tree}(((4), (23), (0)))=$ the completion of $U(((4), (23)))$ that sends $(2, ((0), (0)))$ to the completion of $U_{\tree}(((4), (23)))(2, ((0)))$  whose node component is $-1$

\noindent$U_{\node}((4), (23), (0))= (0, -1)$

( (1), $[ \vec{\alpha} \mapsto h$(2, ($g$(2, ($\alpha_{(0)}$, (0))), (0)))$]_{\mu_{\vec{\alpha}}}$, (0), $[ \vec{\alpha} \mapsto h$(2, ($g$(2, ($\alpha_{(0)}$, (0))), (0), $g$(2, ($\alpha_{(0)}$, (0), $\alpha_{(0, 0)}$, (0))), (0)))$]_{\mu_{\vec{\alpha}}}$, $-1$ )

\item $U_{\tree}(((4), (22)))=$ the completion of $U(((4)))$ that sends $(2, ((0)))$ to the completion of $U_{\tree}(((4)))(2, \emptyset)$  whose node component is $(0, 0)$

\noindent$U_{\node}((4), (22))= (0, -1)$

( (1), $[ \vec{\alpha} \mapsto h$(2, ($g$(2, ($\alpha_{(0)}$, (0))), (0)))$]_{\mu_{\vec{\alpha}}}$, (0), $[ \vec{\alpha} \mapsto h$(2, ($g$(2, ($\alpha_{(0)}$, (0))), (0), $g$(2, ($\alpha_{(0)}$, (0), $\alpha_{(0, 0)}$, $-1$)), (0)))$]_{\mu_{\vec{\alpha}}}$, $-1$ )

\item $U_{\tree}(((4), (21)))=$ the completion of $U(((4)))$ that sends $(2, ((0)))$ to the completion of $U_{\tree}(((4)))(2, \emptyset)$  whose node component is $(0, 0)$

\noindent$U_{\node}((4), (21))= (2, ((0, 0)))$

( (1), $[ \vec{\alpha} \mapsto h$(2, ($g$(2, ($\alpha_{(0)}$, (0))), (0)))$]_{\mu_{\vec{\alpha}}}$, $-1$ )

\item $U_{\tree}(((4), (21), (3)))=$ the completion of $U(((4), (21)))$ that sends $(2, ((0, 0)))$ to the completion of $U_{\tree}(((4), (21)))(2, \emptyset)$  whose node component is $(0, 0)$

\noindent$U_{\node}((4), (21), (3))= (0, -1)$

( (1), $[ \vec{\alpha} \mapsto h$(2, ($g$(2, ($\alpha_{(0)}$, (0))), (0), $g$(2, ($\alpha_{(0)}$, (0, 0))), (0)))$]_{\mu_{\vec{\alpha}}}$, $-1$ )

\item $U_{\tree}(((4), (21), (2)))=$ the completion of $U(((4), (21)))$ that sends $(2, ((0, 0)))$ to the completion of $U_{\tree}(((4), (21)))(2, \emptyset)$  whose node component is $(0, 0)$

\noindent$U_{\node}((4), (21), (2))= (2, ((0, 0), (0)))$

( (1), $[ \vec{\alpha} \mapsto h$(2, ($g$(2, ($\alpha_{(0)}$, (0))), (0), $g$(2, ($\alpha_{(0)}$, (0, 0))), $-1$))$]_{\mu_{\vec{\alpha}}}$, (0) )

\item $U_{\tree}(((4), (21), (2), (2)))=$ the completion of $U(((4), (21), (2)))$ that sends $(2, ((0, 0), (0)))$ to the completion of $U_{\tree}(((4), (21), (2)))(2, ((0, 0)))$  whose node component is $(0, 1)$

\noindent$U_{\node}((4), (21), (2), (2))= (0, -1)$

( (1), $[ \vec{\alpha} \mapsto h$(2, ($g$(2, ($\alpha_{(0)}$, (0))), (0), $g$(2, ($\alpha_{(0)}$, (0, 0))), $-1$))$]_{\mu_{\vec{\alpha}}}$, (0), $[ \vec{\alpha} \mapsto h$(2, ($g$(2, ($\alpha_{(0)}$, (0))), (0), $g$(2, ($\alpha_{(0)}$, (0, 0), $\alpha_{(0, 0)}$, (0))), (0)))$]_{\mu_{\vec{\alpha}}}$, $-1$ )

\item $U_{\tree}(((4), (21), (2), (1)))=$ the completion of $U(((4), (21), (2)))$ that sends $(2, ((0, 0), (0)))$ to the completion of $U_{\tree}(((4), (21), (2)))(2, ((0, 0)))$  whose node component is $(0, 0, 0)$

\noindent$U_{\node}((4), (21), (2), (1))= (0, -1)$

( (1), $[ \vec{\alpha} \mapsto h$(2, ($g$(2, ($\alpha_{(0)}$, (0))), (0), $g$(2, ($\alpha_{(0)}$, (0, 0))), $-1$))$]_{\mu_{\vec{\alpha}}}$, (0), $[ \vec{\alpha} \mapsto h$(2, ($g$(2, ($\alpha_{(0)}$, (0))), (0), $g$(2, ($\alpha_{(0)}$, (0, 0), $\alpha_{(0, 0)}$, (0))), (0)))$]_{\mu_{\vec{\alpha}}}$, $-1$ )

\item $U_{\tree}(((4), (21), (2), (0)))=$ the completion of $U(((4), (21), (2)))$ that sends $(2, ((0, 0), (0)))$ to the completion of $U_{\tree}(((4), (21), (2)))(2, ((0, 0)))$  whose node component is $-1$

\noindent$U_{\node}((4), (21), (2), (0))= (0, -1)$

( (1), $[ \vec{\alpha} \mapsto h$(2, ($g$(2, ($\alpha_{(0)}$, (0))), (0), $g$(2, ($\alpha_{(0)}$, (0, 0))), $-1$))$]_{\mu_{\vec{\alpha}}}$, (0), $[ \vec{\alpha} \mapsto h$(2, ($g$(2, ($\alpha_{(0)}$, (0))), (0), $g$(2, ($\alpha_{(0)}$, (0, 0), $\alpha_{(0, 0)}$, (0))), (0)))$]_{\mu_{\vec{\alpha}}}$, $-1$ )

\item $U_{\tree}(((4), (21), (1)))=$ the completion of $U(((4), (21)))$ that sends $(2, ((0, 0)))$ to the completion of $U_{\tree}(((4), (21)))(2, \emptyset)$  whose node component is $(0, 0)$

\noindent$U_{\node}((4), (21), (1))= (0, -1)$

( (1), $[ \vec{\alpha} \mapsto h$(2, ($g$(2, ($\alpha_{(0)}$, (0))), (0), $g$(2, ($\alpha_{(0)}$, (0, 0))), $-1$))$]_{\mu_{\vec{\alpha}}}$, (0), $[ \vec{\alpha} \mapsto h$(2, ($g$(2, ($\alpha_{(0)}$, (0))), (0), $g$(2, ($\alpha_{(0)}$, (0, 0), $\alpha_{(0, 0)}$, $-1$)), (0)))$]_{\mu_{\vec{\alpha}}}$, $-1$ )

\item $U_{\tree}(((4), (21), (0)))=$ the completion of $U(((4), (21)))$ that sends $(2, ((0, 0)))$ to the completion of $U_{\tree}(((4), (21)))(2, \emptyset)$  whose node component is $-1$

\noindent$U_{\node}((4), (21), (0))= (0, -1)$

( (1), $[ \vec{\alpha} \mapsto h$(2, ($g$(2, ($\alpha_{(0)}$, (0))), (0), $g$(2, ($\alpha_{(0)}$, (0, 0))), (0)))$]_{\mu_{\vec{\alpha}}}$, $-1$ )

\item $U_{\tree}(((4), (20)))=$ the completion of $U(((4)))$ that sends $(2, ((0)))$ to the completion of $U_{\tree}(((4)))(2, \emptyset)$  whose node component is $(0, 0)$

\noindent$U_{\node}((4), (20))= (0, -1)$

( (1), $[ \vec{\alpha} \mapsto h$(2, ($g$(2, ($\alpha_{(0)}$, (0))), (0), $g$(2, ($\alpha_{(0)}$, $-1$)), (0)))$]_{\mu_{\vec{\alpha}}}$, $-1$ )

\item $U_{\tree}(((4), (19)))=$ the completion of $U(((4)))$ that sends $(2, ((0)))$ to the completion of $U_{\tree}(((4)))(2, \emptyset)$  whose node component is $(0, 0)$

\noindent$U_{\node}((4), (19))= (2, ((1)))$

( (1), $[ \vec{\alpha} \mapsto h$(2, ($g$(2, ($\alpha_{(0)}$, (0))), (0), $g$(2, ($\alpha_{(0)}$, $-1$)), $-1$))$]_{\mu_{\vec{\alpha}}}$, (0) )

\item $U_{\tree}(((4), (19), (1)))=$ the completion of $U(((4), (19)))$ that sends $(2, ((1)))$ to the completion of $U_{\tree}(((4), (19)))(2, \emptyset)$  whose node component is $(0, 0)$

\noindent$U_{\node}((4), (19), (1))= (0, -1)$

( (1), $[ \vec{\alpha} \mapsto h$(2, ($g$(2, ($\alpha_{(0)}$, (0))), (0), $g$(2, ($\alpha_{(0)}$, $-1$)), $-1$))$]_{\mu_{\vec{\alpha}}}$, (0), $[ \vec{\alpha} \mapsto h$(2, ($g$(2, ($\alpha_{(0)}$, (0))), (0), $g$(2, ($\alpha_{(0, 0)}$, (1))), (0)))$]_{\mu_{\vec{\alpha}}}$, $-1$ )

\item $U_{\tree}(((4), (19), (0)))=$ the completion of $U(((4), (19)))$ that sends $(2, ((1)))$ to the completion of $U_{\tree}(((4), (19)))(2, \emptyset)$  whose node component is $-1$

\noindent$U_{\node}((4), (19), (0))= (0, -1)$

( (1), $[ \vec{\alpha} \mapsto h$(2, ($g$(2, ($\alpha_{(0)}$, (0))), (0), $g$(2, ($\alpha_{(0)}$, $-1$)), $-1$))$]_{\mu_{\vec{\alpha}}}$, (0), $[ \vec{\alpha} \mapsto h$(2, ($g$(2, ($\alpha_{(0)}$, (0))), (0), $g$(2, ($\alpha_{(0, 0)}$, (1))), (0)))$]_{\mu_{\vec{\alpha}}}$, $-1$ )

\item $U_{\tree}(((4), (18)))=$ the completion of $U(((4)))$ that sends $(2, ((0)))$ to the completion of $U_{\tree}(((4)))(2, \emptyset)$  whose node component is $(0, 0)$

\noindent$U_{\node}((4), (18))= (0, -1)$

( (1), $[ \vec{\alpha} \mapsto h$(2, ($g$(2, ($\alpha_{(0)}$, (0))), (0), $g$(2, ($\alpha_{(0)}$, $-1$)), $-1$))$]_{\mu_{\vec{\alpha}}}$, (0), $[ \vec{\alpha} \mapsto h$(2, ($g$(2, ($\alpha_{(0)}$, (0))), (0), $g$(2, ($\alpha_{(0, 0)}$, (0))), (0)))$]_{\mu_{\vec{\alpha}}}$, $-1$ )

\item $U_{\tree}(((4), (17)))=$ the completion of $U(((4)))$ that sends $(2, ((0)))$ to the completion of $U_{\tree}(((4)))(2, \emptyset)$  whose node component is $(0, 0)$

\noindent$U_{\node}((4), (17))= (2, ((0, 0)))$

( (1), $[ \vec{\alpha} \mapsto h$(2, ($g$(2, ($\alpha_{(0)}$, (0))), (0), $g$(2, ($\alpha_{(0)}$, $-1$)), $-1$))$]_{\mu_{\vec{\alpha}}}$, (0), $[ \vec{\alpha} \mapsto h$(2, ($g$(2, ($\alpha_{(0)}$, (0))), (0), $g$(2, ($\alpha_{(0, 0)}$, (0))), $-1$))$]_{\mu_{\vec{\alpha}}}$, $-1$ )

\item $U_{\tree}(((4), (17), (1)))=$ the completion of $U(((4), (17)))$ that sends $(2, ((0, 0)))$ to the completion of $U_{\tree}(((4), (17)))(2, \emptyset)$  whose node component is $(0, 0)$

\noindent$U_{\node}((4), (17), (1))= (0, -1)$

( (1), $[ \vec{\alpha} \mapsto h$(2, ($g$(2, ($\alpha_{(0)}$, (0))), (0), $g$(2, ($\alpha_{(0)}$, $-1$)), $-1$))$]_{\mu_{\vec{\alpha}}}$, (0), $[ \vec{\alpha} \mapsto h$(2, ($g$(2, ($\alpha_{(0)}$, (0))), (0), $g$(2, ($\alpha_{(0, 0)}$, (0, 0))), (0)))$]_{\mu_{\vec{\alpha}}}$, $-1$ )

\item $U_{\tree}(((4), (17), (0)))=$ the completion of $U(((4), (17)))$ that sends $(2, ((0, 0)))$ to the completion of $U_{\tree}(((4), (17)))(2, \emptyset)$  whose node component is $-1$

\noindent$U_{\node}((4), (17), (0))= (0, -1)$

( (1), $[ \vec{\alpha} \mapsto h$(2, ($g$(2, ($\alpha_{(0)}$, (0))), (0), $g$(2, ($\alpha_{(0)}$, $-1$)), $-1$))$]_{\mu_{\vec{\alpha}}}$, (0), $[ \vec{\alpha} \mapsto h$(2, ($g$(2, ($\alpha_{(0)}$, (0))), (0), $g$(2, ($\alpha_{(0, 0)}$, (0, 0))), (0)))$]_{\mu_{\vec{\alpha}}}$, $-1$ )

\item $U_{\tree}(((4), (16)))=$ the completion of $U(((4)))$ that sends $(2, ((0)))$ to the completion of $U_{\tree}(((4)))(2, \emptyset)$  whose node component is $(0, 0)$

\noindent$U_{\node}((4), (16))= (0, -1)$

( (1), $[ \vec{\alpha} \mapsto h$(2, ($g$(2, ($\alpha_{(0)}$, (0))), (0), $g$(2, ($\alpha_{(0)}$, $-1$)), $-1$))$]_{\mu_{\vec{\alpha}}}$, (0), $[ \vec{\alpha} \mapsto h$(2, ($g$(2, ($\alpha_{(0)}$, (0))), (0), $g$(2, ($\alpha_{(0, 0)}$, $-1$)), (0)))$]_{\mu_{\vec{\alpha}}}$, $-1$ )

\item $U_{\tree}(((4), (15)))=$ the completion of $U(((4)))$ that sends $(2, ((0)))$ to the completion of $U_{\tree}(((4)))(2, \emptyset)$  whose node component is $(0, 0)$

\noindent$U_{\node}((4), (15))= (1, (0))$

( (1), $[ \vec{\alpha} \mapsto h$(2, ($g$(2, ($\alpha_{(0)}$, (0))), (0), $g$(2, ($\alpha_{(0)}$, $-1$)), $-1$))$]_{\mu_{\vec{\alpha}}}$, $-1$ )

\item $U_{\tree}(((4), (15), (0)))=$ the unique completion of $U(((4), (15)))$

\noindent$U_{\node}((4), (15), (0))= (0, -1)$

( (1), $[ \vec{\alpha} \mapsto h$(2, ($g$(2, ($\alpha_{(0)}$, (0))), (0), $g$(1, ((0))), (0)))$]_{\mu_{\vec{\alpha}}}$, $-1$ )

\item $U_{\tree}(((4), (14)))=$ the completion of $U(((4)))$ that sends $(2, ((0)))$ to the completion of $U_{\tree}(((4)))(2, \emptyset)$  whose node component is $(0, 0)$

\noindent$U_{\node}((4), (14))= (2, ((0), (0)))$

( (1), $[ \vec{\alpha} \mapsto h$(2, ($g$(2, ($\alpha_{(0)}$, (0))), $-1$))$]_{\mu_{\vec{\alpha}}}$, (0) )

\item $U_{\tree}(((4), (14), (26)))=$ the completion of $U(((4), (14)))$ that sends $(2, ((0), (0)))$ to the completion of $U_{\tree}(((4), (14)))(2, ((0)))$  whose node component is $(0, 1)$

\noindent$U_{\node}((4), (14), (26))= (2, ((0), (0, 0)))$

( (1), $[ \vec{\alpha} \mapsto h$(2, ($g$(2, ($\alpha_{(0)}$, (0))), $-1$))$]_{\mu_{\vec{\alpha}}}$, (0), $[ \vec{\alpha} \mapsto h$(2, ($g$(2, ($\alpha_{(0)}$, (0), $\alpha_{(0, 0)}$, (0))), (0)))$]_{\mu_{\vec{\alpha}}}$, $-1$ )

\item $U_{\tree}(((4), (14), (26), (2)))=$ the completion of $U(((4), (14), (26)))$ that sends $(2, ((0), (0, 0)))$ to the completion of $U_{\tree}(((4), (14), (26)))(2, ((0)))$  whose node component is $(0, 1)$

\noindent$U_{\node}((4), (14), (26), (2))= (0, -1)$

( (1), $[ \vec{\alpha} \mapsto h$(2, ($g$(2, ($\alpha_{(0)}$, (0))), $-1$))$]_{\mu_{\vec{\alpha}}}$, (0), $[ \vec{\alpha} \mapsto h$(2, ($g$(2, ($\alpha_{(0)}$, (0), $\alpha_{(0, 0)}$, (0))), (0), $g$(2, ($\alpha_{(0)}$, (0), $\alpha_{(0, 0)}$, (0, 0))), (0)))$]_{\mu_{\vec{\alpha}}}$, $-1$ )

\item $U_{\tree}(((4), (14), (26), (1)))=$ the completion of $U(((4), (14), (26)))$ that sends $(2, ((0), (0, 0)))$ to the completion of $U_{\tree}(((4), (14), (26)))(2, ((0)))$  whose node component is $(0, 0, 0)$

\noindent$U_{\node}((4), (14), (26), (1))= (0, -1)$

( (1), $[ \vec{\alpha} \mapsto h$(2, ($g$(2, ($\alpha_{(0)}$, (0))), $-1$))$]_{\mu_{\vec{\alpha}}}$, (0), $[ \vec{\alpha} \mapsto h$(2, ($g$(2, ($\alpha_{(0)}$, (0), $\alpha_{(0, 0)}$, (0))), (0), $g$(2, ($\alpha_{(0)}$, (0), $\alpha_{(0, 0)}$, (0, 0))), (0)))$]_{\mu_{\vec{\alpha}}}$, $-1$ )

\item $U_{\tree}(((4), (14), (26), (0)))=$ the completion of $U(((4), (14), (26)))$ that sends $(2, ((0), (0, 0)))$ to the completion of $U_{\tree}(((4), (14), (26)))(2, ((0)))$  whose node component is $-1$

\noindent$U_{\node}((4), (14), (26), (0))= (0, -1)$

( (1), $[ \vec{\alpha} \mapsto h$(2, ($g$(2, ($\alpha_{(0)}$, (0))), $-1$))$]_{\mu_{\vec{\alpha}}}$, (0), $[ \vec{\alpha} \mapsto h$(2, ($g$(2, ($\alpha_{(0)}$, (0), $\alpha_{(0, 0)}$, (0))), (0), $g$(2, ($\alpha_{(0)}$, (0), $\alpha_{(0, 0)}$, (0, 0))), (0)))$]_{\mu_{\vec{\alpha}}}$, $-1$ )

\item $U_{\tree}(((4), (14), (25)))=$ the completion of $U(((4), (14)))$ that sends $(2, ((0), (0)))$ to the completion of $U_{\tree}(((4), (14)))(2, ((0)))$  whose node component is $(0, 1)$

\noindent$U_{\node}((4), (14), (25))= (0, -1)$

( (1), $[ \vec{\alpha} \mapsto h$(2, ($g$(2, ($\alpha_{(0)}$, (0))), $-1$))$]_{\mu_{\vec{\alpha}}}$, (0), $[ \vec{\alpha} \mapsto h$(2, ($g$(2, ($\alpha_{(0)}$, (0), $\alpha_{(0, 0)}$, (0))), (0), $g$(2, ($\alpha_{(0)}$, (0), $\alpha_{(0, 0)}$, $-1$)), (0)))$]_{\mu_{\vec{\alpha}}}$, $-1$ )

\item $U_{\tree}(((4), (14), (24)))=$ the completion of $U(((4), (14)))$ that sends $(2, ((0), (0)))$ to the completion of $U_{\tree}(((4), (14)))(2, ((0)))$  whose node component is $(0, 1)$

\noindent$U_{\node}((4), (14), (24))= (2, ((0, 0)))$

( (1), $[ \vec{\alpha} \mapsto h$(2, ($g$(2, ($\alpha_{(0)}$, (0))), $-1$))$]_{\mu_{\vec{\alpha}}}$, (0), $[ \vec{\alpha} \mapsto h$(2, ($g$(2, ($\alpha_{(0)}$, (0), $\alpha_{(0, 0)}$, (0))), (0), $g$(2, ($\alpha_{(0)}$, (0), $\alpha_{(0, 0)}$, $-1$)), $-1$))$]_{\mu_{\vec{\alpha}}}$, $-1$ )

\item $U_{\tree}(((4), (14), (24), (3)))=$ the completion of $U(((4), (14), (24)))$ that sends $(2, ((0, 0)))$ to the completion of $U_{\tree}(((4), (14), (24)))(2, \emptyset)$  whose node component is $(0, 0)$

\noindent$U_{\node}((4), (14), (24), (3))= (0, -1)$

( (1), $[ \vec{\alpha} \mapsto h$(2, ($g$(2, ($\alpha_{(0)}$, (0))), $-1$))$]_{\mu_{\vec{\alpha}}}$, (0), $[ \vec{\alpha} \mapsto h$(2, ($g$(2, ($\alpha_{(0)}$, (0), $\alpha_{(0, 0)}$, (0))), (0), $g$(2, ($\alpha_{(0)}$, (0, 0))), (0)))$]_{\mu_{\vec{\alpha}}}$, $-1$ )

\item $U_{\tree}(((4), (14), (24), (2)))=$ the completion of $U(((4), (14), (24)))$ that sends $(2, ((0, 0)))$ to the completion of $U_{\tree}(((4), (14), (24)))(2, \emptyset)$  whose node component is $(0, 0)$

\noindent$U_{\node}((4), (14), (24), (2))= (2, ((0, 0), (0)))$

( (1), $[ \vec{\alpha} \mapsto h$(2, ($g$(2, ($\alpha_{(0)}$, (0))), $-1$))$]_{\mu_{\vec{\alpha}}}$, (0), $[ \vec{\alpha} \mapsto h$(2, ($g$(2, ($\alpha_{(0)}$, (0), $\alpha_{(0, 0)}$, (0))), (0), $g$(2, ($\alpha_{(0)}$, (0, 0))), $-1$))$]_{\mu_{\vec{\alpha}}}$, $-1$ )

\item $U_{\tree}(((4), (14), (24), (2), (2)))=$ the completion of $U(((4), (14), (24), (2)))$ that sends $(2, ((0, 0), (0)))$ to the completion of $U_{\tree}(((4), (14), (24), (2)))(2, ((0, 0)))$  whose node component is $(0, 1)$

\noindent$U_{\node}((4), (14), (24), (2), (2))= (0, -1)$

( (1), $[ \vec{\alpha} \mapsto h$(2, ($g$(2, ($\alpha_{(0)}$, (0))), $-1$))$]_{\mu_{\vec{\alpha}}}$, (0), $[ \vec{\alpha} \mapsto h$(2, ($g$(2, ($\alpha_{(0)}$, (0), $\alpha_{(0, 0)}$, (0))), (0), $g$(2, ($\alpha_{(0)}$, (0, 0), $\alpha_{(0, 0)}$, (0))), (0)))$]_{\mu_{\vec{\alpha}}}$, $-1$ )

\item $U_{\tree}(((4), (14), (24), (2), (1)))=$ the completion of $U(((4), (14), (24), (2)))$ that sends $(2, ((0, 0), (0)))$ to the completion of $U_{\tree}(((4), (14), (24), (2)))(2, ((0, 0)))$  whose node component is $(0, 0, 0)$

\noindent$U_{\node}((4), (14), (24), (2), (1))= (0, -1)$

( (1), $[ \vec{\alpha} \mapsto h$(2, ($g$(2, ($\alpha_{(0)}$, (0))), $-1$))$]_{\mu_{\vec{\alpha}}}$, (0), $[ \vec{\alpha} \mapsto h$(2, ($g$(2, ($\alpha_{(0)}$, (0), $\alpha_{(0, 0)}$, (0))), (0), $g$(2, ($\alpha_{(0)}$, (0, 0), $\alpha_{(0, 0)}$, (0))), (0)))$]_{\mu_{\vec{\alpha}}}$, $-1$ )

\item $U_{\tree}(((4), (14), (24), (2), (0)))=$ the completion of $U(((4), (14), (24), (2)))$ that sends $(2, ((0, 0), (0)))$ to the completion of $U_{\tree}(((4), (14), (24), (2)))(2, ((0, 0)))$  whose node component is $-1$

\noindent$U_{\node}((4), (14), (24), (2), (0))= (0, -1)$

( (1), $[ \vec{\alpha} \mapsto h$(2, ($g$(2, ($\alpha_{(0)}$, (0))), $-1$))$]_{\mu_{\vec{\alpha}}}$, (0), $[ \vec{\alpha} \mapsto h$(2, ($g$(2, ($\alpha_{(0)}$, (0), $\alpha_{(0, 0)}$, (0))), (0), $g$(2, ($\alpha_{(0)}$, (0, 0), $\alpha_{(0, 0)}$, (0))), (0)))$]_{\mu_{\vec{\alpha}}}$, $-1$ )

\item $U_{\tree}(((4), (14), (24), (1)))=$ the completion of $U(((4), (14), (24)))$ that sends $(2, ((0, 0)))$ to the completion of $U_{\tree}(((4), (14), (24)))(2, \emptyset)$  whose node component is $(0, 0)$

\noindent$U_{\node}((4), (14), (24), (1))= (0, -1)$

( (1), $[ \vec{\alpha} \mapsto h$(2, ($g$(2, ($\alpha_{(0)}$, (0))), $-1$))$]_{\mu_{\vec{\alpha}}}$, (0), $[ \vec{\alpha} \mapsto h$(2, ($g$(2, ($\alpha_{(0)}$, (0), $\alpha_{(0, 0)}$, (0))), (0), $g$(2, ($\alpha_{(0)}$, (0, 0), $\alpha_{(0, 0)}$, $-1$)), (0)))$]_{\mu_{\vec{\alpha}}}$, $-1$ )

\item $U_{\tree}(((4), (14), (24), (0)))=$ the completion of $U(((4), (14), (24)))$ that sends $(2, ((0, 0)))$ to the completion of $U_{\tree}(((4), (14), (24)))(2, \emptyset)$  whose node component is $-1$

\noindent$U_{\node}((4), (14), (24), (0))= (0, -1)$

( (1), $[ \vec{\alpha} \mapsto h$(2, ($g$(2, ($\alpha_{(0)}$, (0))), $-1$))$]_{\mu_{\vec{\alpha}}}$, (0), $[ \vec{\alpha} \mapsto h$(2, ($g$(2, ($\alpha_{(0)}$, (0), $\alpha_{(0, 0)}$, (0))), (0), $g$(2, ($\alpha_{(0)}$, (0, 0))), (0)))$]_{\mu_{\vec{\alpha}}}$, $-1$ )

\item $U_{\tree}(((4), (14), (23)))=$ the completion of $U(((4), (14)))$ that sends $(2, ((0), (0)))$ to the completion of $U_{\tree}(((4), (14)))(2, ((0)))$  whose node component is $(0, 1)$

\noindent$U_{\node}((4), (14), (23))= (0, -1)$

( (1), $[ \vec{\alpha} \mapsto h$(2, ($g$(2, ($\alpha_{(0)}$, (0))), $-1$))$]_{\mu_{\vec{\alpha}}}$, (0), $[ \vec{\alpha} \mapsto h$(2, ($g$(2, ($\alpha_{(0)}$, (0), $\alpha_{(0, 0)}$, (0))), (0), $g$(2, ($\alpha_{(0)}$, $-1$)), (0)))$]_{\mu_{\vec{\alpha}}}$, $-1$ )

\item $U_{\tree}(((4), (14), (22)))=$ the completion of $U(((4), (14)))$ that sends $(2, ((0), (0)))$ to the completion of $U_{\tree}(((4), (14)))(2, ((0)))$  whose node component is $(0, 1)$

\noindent$U_{\node}((4), (14), (22))= (2, ((1)))$

( (1), $[ \vec{\alpha} \mapsto h$(2, ($g$(2, ($\alpha_{(0)}$, (0))), $-1$))$]_{\mu_{\vec{\alpha}}}$, (0), $[ \vec{\alpha} \mapsto h$(2, ($g$(2, ($\alpha_{(0)}$, (0), $\alpha_{(0, 0)}$, (0))), (0), $g$(2, ($\alpha_{(0)}$, $-1$)), $-1$))$]_{\mu_{\vec{\alpha}}}$, $-1$ )

\item $U_{\tree}(((4), (14), (22), (1)))=$ the completion of $U(((4), (14), (22)))$ that sends $(2, ((1)))$ to the completion of $U_{\tree}(((4), (14), (22)))(2, \emptyset)$  whose node component is $(0, 0)$

\noindent$U_{\node}((4), (14), (22), (1))= (0, -1)$

( (1), $[ \vec{\alpha} \mapsto h$(2, ($g$(2, ($\alpha_{(0)}$, (0))), $-1$))$]_{\mu_{\vec{\alpha}}}$, (0), $[ \vec{\alpha} \mapsto h$(2, ($g$(2, ($\alpha_{(0)}$, (0), $\alpha_{(0, 0)}$, (0))), (0), $g$(2, ($\alpha_{(0, 0)}$, (1))), (0)))$]_{\mu_{\vec{\alpha}}}$, $-1$ )

\item $U_{\tree}(((4), (14), (22), (0)))=$ the completion of $U(((4), (14), (22)))$ that sends $(2, ((1)))$ to the completion of $U_{\tree}(((4), (14), (22)))(2, \emptyset)$  whose node component is $-1$

\noindent$U_{\node}((4), (14), (22), (0))= (0, -1)$

( (1), $[ \vec{\alpha} \mapsto h$(2, ($g$(2, ($\alpha_{(0)}$, (0))), $-1$))$]_{\mu_{\vec{\alpha}}}$, (0), $[ \vec{\alpha} \mapsto h$(2, ($g$(2, ($\alpha_{(0)}$, (0), $\alpha_{(0, 0)}$, (0))), (0), $g$(2, ($\alpha_{(0, 0)}$, (1))), (0)))$]_{\mu_{\vec{\alpha}}}$, $-1$ )

\item $U_{\tree}(((4), (14), (21)))=$ the completion of $U(((4), (14)))$ that sends $(2, ((0), (0)))$ to the completion of $U_{\tree}(((4), (14)))(2, ((0)))$  whose node component is $(0, 1)$

\noindent$U_{\node}((4), (14), (21))= (0, -1)$

( (1), $[ \vec{\alpha} \mapsto h$(2, ($g$(2, ($\alpha_{(0)}$, (0))), $-1$))$]_{\mu_{\vec{\alpha}}}$, (0), $[ \vec{\alpha} \mapsto h$(2, ($g$(2, ($\alpha_{(0)}$, (0), $\alpha_{(0, 0)}$, (0))), (0), $g$(2, ($\alpha_{(0, 0)}$, (0))), (0)))$]_{\mu_{\vec{\alpha}}}$, $-1$ )

\item $U_{\tree}(((4), (14), (20)))=$ the completion of $U(((4), (14)))$ that sends $(2, ((0), (0)))$ to the completion of $U_{\tree}(((4), (14)))(2, ((0)))$  whose node component is $(0, 1)$

\noindent$U_{\node}((4), (14), (20))= (2, ((0, 0)))$

( (1), $[ \vec{\alpha} \mapsto h$(2, ($g$(2, ($\alpha_{(0)}$, (0))), $-1$))$]_{\mu_{\vec{\alpha}}}$, (0), $[ \vec{\alpha} \mapsto h$(2, ($g$(2, ($\alpha_{(0)}$, (0), $\alpha_{(0, 0)}$, (0))), (0), $g$(2, ($\alpha_{(0, 0)}$, (0))), $-1$))$]_{\mu_{\vec{\alpha}}}$, $-1$ )

\item $U_{\tree}(((4), (14), (20), (1)))=$ the completion of $U(((4), (14), (20)))$ that sends $(2, ((0, 0)))$ to the completion of $U_{\tree}(((4), (14), (20)))(2, \emptyset)$  whose node component is $(0, 0)$

\noindent$U_{\node}((4), (14), (20), (1))= (0, -1)$

( (1), $[ \vec{\alpha} \mapsto h$(2, ($g$(2, ($\alpha_{(0)}$, (0))), $-1$))$]_{\mu_{\vec{\alpha}}}$, (0), $[ \vec{\alpha} \mapsto h$(2, ($g$(2, ($\alpha_{(0)}$, (0), $\alpha_{(0, 0)}$, (0))), (0), $g$(2, ($\alpha_{(0, 0)}$, (0, 0))), (0)))$]_{\mu_{\vec{\alpha}}}$, $-1$ )

\item $U_{\tree}(((4), (14), (20), (0)))=$ the completion of $U(((4), (14), (20)))$ that sends $(2, ((0, 0)))$ to the completion of $U_{\tree}(((4), (14), (20)))(2, \emptyset)$  whose node component is $-1$

\noindent$U_{\node}((4), (14), (20), (0))= (0, -1)$

( (1), $[ \vec{\alpha} \mapsto h$(2, ($g$(2, ($\alpha_{(0)}$, (0))), $-1$))$]_{\mu_{\vec{\alpha}}}$, (0), $[ \vec{\alpha} \mapsto h$(2, ($g$(2, ($\alpha_{(0)}$, (0), $\alpha_{(0, 0)}$, (0))), (0), $g$(2, ($\alpha_{(0, 0)}$, (0, 0))), (0)))$]_{\mu_{\vec{\alpha}}}$, $-1$ )

\item $U_{\tree}(((4), (14), (19)))=$ the completion of $U(((4), (14)))$ that sends $(2, ((0), (0)))$ to the completion of $U_{\tree}(((4), (14)))(2, ((0)))$  whose node component is $(0, 1)$

\noindent$U_{\node}((4), (14), (19))= (0, -1)$

( (1), $[ \vec{\alpha} \mapsto h$(2, ($g$(2, ($\alpha_{(0)}$, (0))), $-1$))$]_{\mu_{\vec{\alpha}}}$, (0), $[ \vec{\alpha} \mapsto h$(2, ($g$(2, ($\alpha_{(0)}$, (0), $\alpha_{(0, 0)}$, (0))), (0), $g$(2, ($\alpha_{(0, 0)}$, $-1$)), (0)))$]_{\mu_{\vec{\alpha}}}$, $-1$ )

\item $U_{\tree}(((4), (14), (18)))=$ the completion of $U(((4), (14)))$ that sends $(2, ((0), (0)))$ to the completion of $U_{\tree}(((4), (14)))(2, ((0)))$  whose node component is $(0, 1)$

\noindent$U_{\node}((4), (14), (18))= (1, (0))$

( (1), $[ \vec{\alpha} \mapsto h$(2, ($g$(2, ($\alpha_{(0)}$, (0))), $-1$))$]_{\mu_{\vec{\alpha}}}$, (0), $[ \vec{\alpha} \mapsto h$(2, ($g$(2, ($\alpha_{(0)}$, (0), $\alpha_{(0, 0)}$, (0))), (0), $g$(2, ($\alpha_{(0, 0)}$, $-1$)), $-1$))$]_{\mu_{\vec{\alpha}}}$, $-1$ )

\item $U_{\tree}(((4), (14), (18), (0)))=$ the unique completion of $U(((4), (14), (18)))$

\noindent$U_{\node}((4), (14), (18), (0))= (0, -1)$

( (1), $[ \vec{\alpha} \mapsto h$(2, ($g$(2, ($\alpha_{(0)}$, (0))), $-1$))$]_{\mu_{\vec{\alpha}}}$, (0), $[ \vec{\alpha} \mapsto h$(2, ($g$(2, ($\alpha_{(0)}$, (0), $\alpha_{(0, 0)}$, (0))), (0), $g$(1, ((0))), (0)))$]_{\mu_{\vec{\alpha}}}$, $-1$ )

\item $U_{\tree}(((4), (14), (17)))=$ the completion of $U(((4), (14)))$ that sends $(2, ((0), (0)))$ to the completion of $U_{\tree}(((4), (14)))(2, ((0)))$  whose node component is $(0, 0, 0)$

\noindent$U_{\node}((4), (14), (17))= (2, ((0), (0, 0)))$

( (1), $[ \vec{\alpha} \mapsto h$(2, ($g$(2, ($\alpha_{(0)}$, (0))), $-1$))$]_{\mu_{\vec{\alpha}}}$, (0), $[ \vec{\alpha} \mapsto h$(2, ($g$(2, ($\alpha_{(0)}$, (0), $\alpha_{(0, 0)}$, (0))), (0)))$]_{\mu_{\vec{\alpha}}}$, $-1$ )

\item $U_{\tree}(((4), (14), (17), (2)))=$ the completion of $U(((4), (14), (17)))$ that sends $(2, ((0), (0, 0)))$ to the completion of $U_{\tree}(((4), (14), (17)))(2, ((0)))$  whose node component is $(0, 1)$

\noindent$U_{\node}((4), (14), (17), (2))= (0, -1)$

( (1), $[ \vec{\alpha} \mapsto h$(2, ($g$(2, ($\alpha_{(0)}$, (0))), $-1$))$]_{\mu_{\vec{\alpha}}}$, (0), $[ \vec{\alpha} \mapsto h$(2, ($g$(2, ($\alpha_{(0)}$, (0), $\alpha_{(0, 0)}$, (0))), (0), $g$(2, ($\alpha_{(0)}$, (0), $\alpha_{(0, 0)}$, (0, 0))), (0)))$]_{\mu_{\vec{\alpha}}}$, $-1$ )

\item $U_{\tree}(((4), (14), (17), (1)))=$ the completion of $U(((4), (14), (17)))$ that sends $(2, ((0), (0, 0)))$ to the completion of $U_{\tree}(((4), (14), (17)))(2, ((0)))$  whose node component is $(0, 0, 0)$

\noindent$U_{\node}((4), (14), (17), (1))= (0, -1)$

( (1), $[ \vec{\alpha} \mapsto h$(2, ($g$(2, ($\alpha_{(0)}$, (0))), $-1$))$]_{\mu_{\vec{\alpha}}}$, (0), $[ \vec{\alpha} \mapsto h$(2, ($g$(2, ($\alpha_{(0)}$, (0), $\alpha_{(0, 0)}$, (0))), (0), $g$(2, ($\alpha_{(0)}$, (0), $\alpha_{(0, 0)}$, (0, 0))), (0)))$]_{\mu_{\vec{\alpha}}}$, $-1$ )

\item $U_{\tree}(((4), (14), (17), (0)))=$ the completion of $U(((4), (14), (17)))$ that sends $(2, ((0), (0, 0)))$ to the completion of $U_{\tree}(((4), (14), (17)))(2, ((0)))$  whose node component is $-1$

\noindent$U_{\node}((4), (14), (17), (0))= (0, -1)$

( (1), $[ \vec{\alpha} \mapsto h$(2, ($g$(2, ($\alpha_{(0)}$, (0))), $-1$))$]_{\mu_{\vec{\alpha}}}$, (0), $[ \vec{\alpha} \mapsto h$(2, ($g$(2, ($\alpha_{(0)}$, (0), $\alpha_{(0, 0)}$, (0))), (0), $g$(2, ($\alpha_{(0)}$, (0), $\alpha_{(0, 0)}$, (0, 0))), (0)))$]_{\mu_{\vec{\alpha}}}$, $-1$ )

\item $U_{\tree}(((4), (14), (16)))=$ the completion of $U(((4), (14)))$ that sends $(2, ((0), (0)))$ to the completion of $U_{\tree}(((4), (14)))(2, ((0)))$  whose node component is $(0, 0, 0)$

\noindent$U_{\node}((4), (14), (16))= (0, -1)$

( (1), $[ \vec{\alpha} \mapsto h$(2, ($g$(2, ($\alpha_{(0)}$, (0))), $-1$))$]_{\mu_{\vec{\alpha}}}$, (0), $[ \vec{\alpha} \mapsto h$(2, ($g$(2, ($\alpha_{(0)}$, (0), $\alpha_{(0, 0)}$, (0))), (0), $g$(2, ($\alpha_{(0)}$, (0), $\alpha_{(0, 0)}$, $-1$)), (0)))$]_{\mu_{\vec{\alpha}}}$, $-1$ )

\item $U_{\tree}(((4), (14), (15)))=$ the completion of $U(((4), (14)))$ that sends $(2, ((0), (0)))$ to the completion of $U_{\tree}(((4), (14)))(2, ((0)))$  whose node component is $(0, 0, 0)$

\noindent$U_{\node}((4), (14), (15))= (2, ((0, 0)))$

( (1), $[ \vec{\alpha} \mapsto h$(2, ($g$(2, ($\alpha_{(0)}$, (0))), $-1$))$]_{\mu_{\vec{\alpha}}}$, (0), $[ \vec{\alpha} \mapsto h$(2, ($g$(2, ($\alpha_{(0)}$, (0), $\alpha_{(0, 0)}$, (0))), (0), $g$(2, ($\alpha_{(0)}$, (0), $\alpha_{(0, 0)}$, $-1$)), $-1$))$]_{\mu_{\vec{\alpha}}}$, $-1$ )

\item $U_{\tree}(((4), (14), (15), (3)))=$ the completion of $U(((4), (14), (15)))$ that sends $(2, ((0, 0)))$ to the completion of $U_{\tree}(((4), (14), (15)))(2, \emptyset)$  whose node component is $(0, 0)$

\noindent$U_{\node}((4), (14), (15), (3))= (0, -1)$

( (1), $[ \vec{\alpha} \mapsto h$(2, ($g$(2, ($\alpha_{(0)}$, (0))), $-1$))$]_{\mu_{\vec{\alpha}}}$, (0), $[ \vec{\alpha} \mapsto h$(2, ($g$(2, ($\alpha_{(0)}$, (0), $\alpha_{(0, 0)}$, (0))), (0), $g$(2, ($\alpha_{(0)}$, (0, 0))), (0)))$]_{\mu_{\vec{\alpha}}}$, $-1$ )

\item $U_{\tree}(((4), (14), (15), (2)))=$ the completion of $U(((4), (14), (15)))$ that sends $(2, ((0, 0)))$ to the completion of $U_{\tree}(((4), (14), (15)))(2, \emptyset)$  whose node component is $(0, 0)$

\noindent$U_{\node}((4), (14), (15), (2))= (2, ((0, 0), (0)))$

( (1), $[ \vec{\alpha} \mapsto h$(2, ($g$(2, ($\alpha_{(0)}$, (0))), $-1$))$]_{\mu_{\vec{\alpha}}}$, (0), $[ \vec{\alpha} \mapsto h$(2, ($g$(2, ($\alpha_{(0)}$, (0), $\alpha_{(0, 0)}$, (0))), (0), $g$(2, ($\alpha_{(0)}$, (0, 0))), $-1$))$]_{\mu_{\vec{\alpha}}}$, $-1$ )

\item $U_{\tree}(((4), (14), (15), (2), (2)))=$ the completion of $U(((4), (14), (15), (2)))$ that sends $(2, ((0, 0), (0)))$ to the completion of $U_{\tree}(((4), (14), (15), (2)))(2, ((0, 0)))$  whose node component is $(0, 1)$

\noindent$U_{\node}((4), (14), (15), (2), (2))= (0, -1)$

( (1), $[ \vec{\alpha} \mapsto h$(2, ($g$(2, ($\alpha_{(0)}$, (0))), $-1$))$]_{\mu_{\vec{\alpha}}}$, (0), $[ \vec{\alpha} \mapsto h$(2, ($g$(2, ($\alpha_{(0)}$, (0), $\alpha_{(0, 0)}$, (0))), (0), $g$(2, ($\alpha_{(0)}$, (0, 0), $\alpha_{(0, 0)}$, (0))), (0)))$]_{\mu_{\vec{\alpha}}}$, $-1$ )

\item $U_{\tree}(((4), (14), (15), (2), (1)))=$ the completion of $U(((4), (14), (15), (2)))$ that sends $(2, ((0, 0), (0)))$ to the completion of $U_{\tree}(((4), (14), (15), (2)))(2, ((0, 0)))$  whose node component is $(0, 0, 0)$

\noindent$U_{\node}((4), (14), (15), (2), (1))= (0, -1)$

( (1), $[ \vec{\alpha} \mapsto h$(2, ($g$(2, ($\alpha_{(0)}$, (0))), $-1$))$]_{\mu_{\vec{\alpha}}}$, (0), $[ \vec{\alpha} \mapsto h$(2, ($g$(2, ($\alpha_{(0)}$, (0), $\alpha_{(0, 0)}$, (0))), (0), $g$(2, ($\alpha_{(0)}$, (0, 0), $\alpha_{(0, 0)}$, (0))), (0)))$]_{\mu_{\vec{\alpha}}}$, $-1$ )

\item $U_{\tree}(((4), (14), (15), (2), (0)))=$ the completion of $U(((4), (14), (15), (2)))$ that sends $(2, ((0, 0), (0)))$ to the completion of $U_{\tree}(((4), (14), (15), (2)))(2, ((0, 0)))$  whose node component is $-1$

\noindent$U_{\node}((4), (14), (15), (2), (0))= (0, -1)$

( (1), $[ \vec{\alpha} \mapsto h$(2, ($g$(2, ($\alpha_{(0)}$, (0))), $-1$))$]_{\mu_{\vec{\alpha}}}$, (0), $[ \vec{\alpha} \mapsto h$(2, ($g$(2, ($\alpha_{(0)}$, (0), $\alpha_{(0, 0)}$, (0))), (0), $g$(2, ($\alpha_{(0)}$, (0, 0), $\alpha_{(0, 0)}$, (0))), (0)))$]_{\mu_{\vec{\alpha}}}$, $-1$ )

\item $U_{\tree}(((4), (14), (15), (1)))=$ the completion of $U(((4), (14), (15)))$ that sends $(2, ((0, 0)))$ to the completion of $U_{\tree}(((4), (14), (15)))(2, \emptyset)$  whose node component is $(0, 0)$

\noindent$U_{\node}((4), (14), (15), (1))= (0, -1)$

( (1), $[ \vec{\alpha} \mapsto h$(2, ($g$(2, ($\alpha_{(0)}$, (0))), $-1$))$]_{\mu_{\vec{\alpha}}}$, (0), $[ \vec{\alpha} \mapsto h$(2, ($g$(2, ($\alpha_{(0)}$, (0), $\alpha_{(0, 0)}$, (0))), (0), $g$(2, ($\alpha_{(0)}$, (0, 0), $\alpha_{(0, 0)}$, $-1$)), (0)))$]_{\mu_{\vec{\alpha}}}$, $-1$ )

\item $U_{\tree}(((4), (14), (15), (0)))=$ the completion of $U(((4), (14), (15)))$ that sends $(2, ((0, 0)))$ to the completion of $U_{\tree}(((4), (14), (15)))(2, \emptyset)$  whose node component is $-1$

\noindent$U_{\node}((4), (14), (15), (0))= (0, -1)$

( (1), $[ \vec{\alpha} \mapsto h$(2, ($g$(2, ($\alpha_{(0)}$, (0))), $-1$))$]_{\mu_{\vec{\alpha}}}$, (0), $[ \vec{\alpha} \mapsto h$(2, ($g$(2, ($\alpha_{(0)}$, (0), $\alpha_{(0, 0)}$, (0))), (0), $g$(2, ($\alpha_{(0)}$, (0, 0))), (0)))$]_{\mu_{\vec{\alpha}}}$, $-1$ )

\item $U_{\tree}(((4), (14), (14)))=$ the completion of $U(((4), (14)))$ that sends $(2, ((0), (0)))$ to the completion of $U_{\tree}(((4), (14)))(2, ((0)))$  whose node component is $(0, 0, 0)$

\noindent$U_{\node}((4), (14), (14))= (0, -1)$

( (1), $[ \vec{\alpha} \mapsto h$(2, ($g$(2, ($\alpha_{(0)}$, (0))), $-1$))$]_{\mu_{\vec{\alpha}}}$, (0), $[ \vec{\alpha} \mapsto h$(2, ($g$(2, ($\alpha_{(0)}$, (0), $\alpha_{(0, 0)}$, (0))), (0), $g$(2, ($\alpha_{(0)}$, $-1$)), (0)))$]_{\mu_{\vec{\alpha}}}$, $-1$ )

\item $U_{\tree}(((4), (14), (13)))=$ the completion of $U(((4), (14)))$ that sends $(2, ((0), (0)))$ to the completion of $U_{\tree}(((4), (14)))(2, ((0)))$  whose node component is $(0, 0, 0)$

\noindent$U_{\node}((4), (14), (13))= (2, ((1)))$

( (1), $[ \vec{\alpha} \mapsto h$(2, ($g$(2, ($\alpha_{(0)}$, (0))), $-1$))$]_{\mu_{\vec{\alpha}}}$, (0), $[ \vec{\alpha} \mapsto h$(2, ($g$(2, ($\alpha_{(0)}$, (0), $\alpha_{(0, 0)}$, (0))), (0), $g$(2, ($\alpha_{(0)}$, $-1$)), $-1$))$]_{\mu_{\vec{\alpha}}}$, $-1$ )

\item $U_{\tree}(((4), (14), (13), (1)))=$ the completion of $U(((4), (14), (13)))$ that sends $(2, ((1)))$ to the completion of $U_{\tree}(((4), (14), (13)))(2, \emptyset)$  whose node component is $(0, 0)$

\noindent$U_{\node}((4), (14), (13), (1))= (0, -1)$

( (1), $[ \vec{\alpha} \mapsto h$(2, ($g$(2, ($\alpha_{(0)}$, (0))), $-1$))$]_{\mu_{\vec{\alpha}}}$, (0), $[ \vec{\alpha} \mapsto h$(2, ($g$(2, ($\alpha_{(0)}$, (0), $\alpha_{(0, 0)}$, (0))), (0), $g$(2, ($\alpha_{(0, 0)}$, (1))), (0)))$]_{\mu_{\vec{\alpha}}}$, $-1$ )

\item $U_{\tree}(((4), (14), (13), (0)))=$ the completion of $U(((4), (14), (13)))$ that sends $(2, ((1)))$ to the completion of $U_{\tree}(((4), (14), (13)))(2, \emptyset)$  whose node component is $-1$

\noindent$U_{\node}((4), (14), (13), (0))= (0, -1)$

( (1), $[ \vec{\alpha} \mapsto h$(2, ($g$(2, ($\alpha_{(0)}$, (0))), $-1$))$]_{\mu_{\vec{\alpha}}}$, (0), $[ \vec{\alpha} \mapsto h$(2, ($g$(2, ($\alpha_{(0)}$, (0), $\alpha_{(0, 0)}$, (0))), (0), $g$(2, ($\alpha_{(0, 0)}$, (1))), (0)))$]_{\mu_{\vec{\alpha}}}$, $-1$ )

\item $U_{\tree}(((4), (14), (12)))=$ the completion of $U(((4), (14)))$ that sends $(2, ((0), (0)))$ to the completion of $U_{\tree}(((4), (14)))(2, ((0)))$  whose node component is $(0, 0, 0)$

\noindent$U_{\node}((4), (14), (12))= (0, -1)$

( (1), $[ \vec{\alpha} \mapsto h$(2, ($g$(2, ($\alpha_{(0)}$, (0))), $-1$))$]_{\mu_{\vec{\alpha}}}$, (0), $[ \vec{\alpha} \mapsto h$(2, ($g$(2, ($\alpha_{(0)}$, (0), $\alpha_{(0, 0)}$, (0))), (0), $g$(2, ($\alpha_{(0, 0)}$, (0))), (0)))$]_{\mu_{\vec{\alpha}}}$, $-1$ )

\item $U_{\tree}(((4), (14), (11)))=$ the completion of $U(((4), (14)))$ that sends $(2, ((0), (0)))$ to the completion of $U_{\tree}(((4), (14)))(2, ((0)))$  whose node component is $(0, 0, 0)$

\noindent$U_{\node}((4), (14), (11))= (2, ((0, 0)))$

( (1), $[ \vec{\alpha} \mapsto h$(2, ($g$(2, ($\alpha_{(0)}$, (0))), $-1$))$]_{\mu_{\vec{\alpha}}}$, (0), $[ \vec{\alpha} \mapsto h$(2, ($g$(2, ($\alpha_{(0)}$, (0), $\alpha_{(0, 0)}$, (0))), (0), $g$(2, ($\alpha_{(0, 0)}$, (0))), $-1$))$]_{\mu_{\vec{\alpha}}}$, $-1$ )

\item $U_{\tree}(((4), (14), (11), (1)))=$ the completion of $U(((4), (14), (11)))$ that sends $(2, ((0, 0)))$ to the completion of $U_{\tree}(((4), (14), (11)))(2, \emptyset)$  whose node component is $(0, 0)$

\noindent$U_{\node}((4), (14), (11), (1))= (0, -1)$

( (1), $[ \vec{\alpha} \mapsto h$(2, ($g$(2, ($\alpha_{(0)}$, (0))), $-1$))$]_{\mu_{\vec{\alpha}}}$, (0), $[ \vec{\alpha} \mapsto h$(2, ($g$(2, ($\alpha_{(0)}$, (0), $\alpha_{(0, 0)}$, (0))), (0), $g$(2, ($\alpha_{(0, 0)}$, (0, 0))), (0)))$]_{\mu_{\vec{\alpha}}}$, $-1$ )

\item $U_{\tree}(((4), (14), (11), (0)))=$ the completion of $U(((4), (14), (11)))$ that sends $(2, ((0, 0)))$ to the completion of $U_{\tree}(((4), (14), (11)))(2, \emptyset)$  whose node component is $-1$

\noindent$U_{\node}((4), (14), (11), (0))= (0, -1)$

( (1), $[ \vec{\alpha} \mapsto h$(2, ($g$(2, ($\alpha_{(0)}$, (0))), $-1$))$]_{\mu_{\vec{\alpha}}}$, (0), $[ \vec{\alpha} \mapsto h$(2, ($g$(2, ($\alpha_{(0)}$, (0), $\alpha_{(0, 0)}$, (0))), (0), $g$(2, ($\alpha_{(0, 0)}$, (0, 0))), (0)))$]_{\mu_{\vec{\alpha}}}$, $-1$ )

\item $U_{\tree}(((4), (14), (10)))=$ the completion of $U(((4), (14)))$ that sends $(2, ((0), (0)))$ to the completion of $U_{\tree}(((4), (14)))(2, ((0)))$  whose node component is $(0, 0, 0)$

\noindent$U_{\node}((4), (14), (10))= (0, -1)$

( (1), $[ \vec{\alpha} \mapsto h$(2, ($g$(2, ($\alpha_{(0)}$, (0))), $-1$))$]_{\mu_{\vec{\alpha}}}$, (0), $[ \vec{\alpha} \mapsto h$(2, ($g$(2, ($\alpha_{(0)}$, (0), $\alpha_{(0, 0)}$, (0))), (0), $g$(2, ($\alpha_{(0, 0)}$, $-1$)), (0)))$]_{\mu_{\vec{\alpha}}}$, $-1$ )

\item $U_{\tree}(((4), (14), (9)))=$ the completion of $U(((4), (14)))$ that sends $(2, ((0), (0)))$ to the completion of $U_{\tree}(((4), (14)))(2, ((0)))$  whose node component is $(0, 0, 0)$

\noindent$U_{\node}((4), (14), (9))= (1, (0))$

( (1), $[ \vec{\alpha} \mapsto h$(2, ($g$(2, ($\alpha_{(0)}$, (0))), $-1$))$]_{\mu_{\vec{\alpha}}}$, (0), $[ \vec{\alpha} \mapsto h$(2, ($g$(2, ($\alpha_{(0)}$, (0), $\alpha_{(0, 0)}$, (0))), (0), $g$(2, ($\alpha_{(0, 0)}$, $-1$)), $-1$))$]_{\mu_{\vec{\alpha}}}$, $-1$ )

\item $U_{\tree}(((4), (14), (9), (0)))=$ the unique completion of $U(((4), (14), (9)))$

\noindent$U_{\node}((4), (14), (9), (0))= (0, -1)$

( (1), $[ \vec{\alpha} \mapsto h$(2, ($g$(2, ($\alpha_{(0)}$, (0))), $-1$))$]_{\mu_{\vec{\alpha}}}$, (0), $[ \vec{\alpha} \mapsto h$(2, ($g$(2, ($\alpha_{(0)}$, (0), $\alpha_{(0, 0)}$, (0))), (0), $g$(1, ((0))), (0)))$]_{\mu_{\vec{\alpha}}}$, $-1$ )

\item $U_{\tree}(((4), (14), (8)))=$ the completion of $U(((4), (14)))$ that sends $(2, ((0), (0)))$ to the completion of $U_{\tree}(((4), (14)))(2, ((0)))$  whose node component is $-1$

\noindent$U_{\node}((4), (14), (8))= (2, ((0), (0, 0)))$

( (1), $[ \vec{\alpha} \mapsto h$(2, ($g$(2, ($\alpha_{(0)}$, (0))), $-1$))$]_{\mu_{\vec{\alpha}}}$, (0), $[ \vec{\alpha} \mapsto h$(2, ($g$(2, ($\alpha_{(0)}$, (0), $\alpha_{(0, 0)}$, (0))), (0)))$]_{\mu_{\vec{\alpha}}}$, $-1$ )

\item $U_{\tree}(((4), (14), (8), (2)))=$ the completion of $U(((4), (14), (8)))$ that sends $(2, ((0), (0, 0)))$ to the completion of $U_{\tree}(((4), (14), (8)))(2, ((0)))$  whose node component is $(0, 1)$

\noindent$U_{\node}((4), (14), (8), (2))= (0, -1)$

( (1), $[ \vec{\alpha} \mapsto h$(2, ($g$(2, ($\alpha_{(0)}$, (0))), $-1$))$]_{\mu_{\vec{\alpha}}}$, (0), $[ \vec{\alpha} \mapsto h$(2, ($g$(2, ($\alpha_{(0)}$, (0), $\alpha_{(0, 0)}$, (0))), (0), $g$(2, ($\alpha_{(0)}$, (0), $\alpha_{(0, 0)}$, (0, 0))), (0)))$]_{\mu_{\vec{\alpha}}}$, $-1$ )

\item $U_{\tree}(((4), (14), (8), (1)))=$ the completion of $U(((4), (14), (8)))$ that sends $(2, ((0), (0, 0)))$ to the completion of $U_{\tree}(((4), (14), (8)))(2, ((0)))$  whose node component is $(0, 0, 0)$

\noindent$U_{\node}((4), (14), (8), (1))= (0, -1)$

( (1), $[ \vec{\alpha} \mapsto h$(2, ($g$(2, ($\alpha_{(0)}$, (0))), $-1$))$]_{\mu_{\vec{\alpha}}}$, (0), $[ \vec{\alpha} \mapsto h$(2, ($g$(2, ($\alpha_{(0)}$, (0), $\alpha_{(0, 0)}$, (0))), (0), $g$(2, ($\alpha_{(0)}$, (0), $\alpha_{(0, 0)}$, (0, 0))), (0)))$]_{\mu_{\vec{\alpha}}}$, $-1$ )

\item $U_{\tree}(((4), (14), (8), (0)))=$ the completion of $U(((4), (14), (8)))$ that sends $(2, ((0), (0, 0)))$ to the completion of $U_{\tree}(((4), (14), (8)))(2, ((0)))$  whose node component is $-1$

\noindent$U_{\node}((4), (14), (8), (0))= (0, -1)$

( (1), $[ \vec{\alpha} \mapsto h$(2, ($g$(2, ($\alpha_{(0)}$, (0))), $-1$))$]_{\mu_{\vec{\alpha}}}$, (0), $[ \vec{\alpha} \mapsto h$(2, ($g$(2, ($\alpha_{(0)}$, (0), $\alpha_{(0, 0)}$, (0))), (0), $g$(2, ($\alpha_{(0)}$, (0), $\alpha_{(0, 0)}$, (0, 0))), (0)))$]_{\mu_{\vec{\alpha}}}$, $-1$ )

\item $U_{\tree}(((4), (14), (7)))=$ the completion of $U(((4), (14)))$ that sends $(2, ((0), (0)))$ to the completion of $U_{\tree}(((4), (14)))(2, ((0)))$  whose node component is $-1$

\noindent$U_{\node}((4), (14), (7))= (0, -1)$

( (1), $[ \vec{\alpha} \mapsto h$(2, ($g$(2, ($\alpha_{(0)}$, (0))), $-1$))$]_{\mu_{\vec{\alpha}}}$, (0), $[ \vec{\alpha} \mapsto h$(2, ($g$(2, ($\alpha_{(0)}$, (0), $\alpha_{(0, 0)}$, (0))), (0), $g$(2, ($\alpha_{(0)}$, (0), $\alpha_{(0, 0)}$, $-1$)), (0)))$]_{\mu_{\vec{\alpha}}}$, $-1$ )

\item $U_{\tree}(((4), (14), (6)))=$ the completion of $U(((4), (14)))$ that sends $(2, ((0), (0)))$ to the completion of $U_{\tree}(((4), (14)))(2, ((0)))$  whose node component is $-1$

\noindent$U_{\node}((4), (14), (6))= (2, ((0, 0)))$

( (1), $[ \vec{\alpha} \mapsto h$(2, ($g$(2, ($\alpha_{(0)}$, (0))), $-1$))$]_{\mu_{\vec{\alpha}}}$, (0), $[ \vec{\alpha} \mapsto h$(2, ($g$(2, ($\alpha_{(0)}$, (0), $\alpha_{(0, 0)}$, (0))), (0), $g$(2, ($\alpha_{(0)}$, (0), $\alpha_{(0, 0)}$, $-1$)), $-1$))$]_{\mu_{\vec{\alpha}}}$, $-1$ )

\item $U_{\tree}(((4), (14), (6), (3)))=$ the completion of $U(((4), (14), (6)))$ that sends $(2, ((0, 0)))$ to the completion of $U_{\tree}(((4), (14), (6)))(2, \emptyset)$  whose node component is $(0, 0)$

\noindent$U_{\node}((4), (14), (6), (3))= (0, -1)$

( (1), $[ \vec{\alpha} \mapsto h$(2, ($g$(2, ($\alpha_{(0)}$, (0))), $-1$))$]_{\mu_{\vec{\alpha}}}$, (0), $[ \vec{\alpha} \mapsto h$(2, ($g$(2, ($\alpha_{(0)}$, (0), $\alpha_{(0, 0)}$, (0))), (0), $g$(2, ($\alpha_{(0)}$, (0, 0))), (0)))$]_{\mu_{\vec{\alpha}}}$, $-1$ )

\item $U_{\tree}(((4), (14), (6), (2)))=$ the completion of $U(((4), (14), (6)))$ that sends $(2, ((0, 0)))$ to the completion of $U_{\tree}(((4), (14), (6)))(2, \emptyset)$  whose node component is $(0, 0)$

\noindent$U_{\node}((4), (14), (6), (2))= (2, ((0, 0), (0)))$

( (1), $[ \vec{\alpha} \mapsto h$(2, ($g$(2, ($\alpha_{(0)}$, (0))), $-1$))$]_{\mu_{\vec{\alpha}}}$, (0), $[ \vec{\alpha} \mapsto h$(2, ($g$(2, ($\alpha_{(0)}$, (0), $\alpha_{(0, 0)}$, (0))), (0), $g$(2, ($\alpha_{(0)}$, (0, 0))), $-1$))$]_{\mu_{\vec{\alpha}}}$, $-1$ )

\item $U_{\tree}(((4), (14), (6), (2), (2)))=$ the completion of $U(((4), (14), (6), (2)))$ that sends $(2, ((0, 0), (0)))$ to the completion of $U_{\tree}(((4), (14), (6), (2)))(2, ((0, 0)))$  whose node component is $(0, 1)$

\noindent$U_{\node}((4), (14), (6), (2), (2))= (0, -1)$

( (1), $[ \vec{\alpha} \mapsto h$(2, ($g$(2, ($\alpha_{(0)}$, (0))), $-1$))$]_{\mu_{\vec{\alpha}}}$, (0), $[ \vec{\alpha} \mapsto h$(2, ($g$(2, ($\alpha_{(0)}$, (0), $\alpha_{(0, 0)}$, (0))), (0), $g$(2, ($\alpha_{(0)}$, (0, 0), $\alpha_{(0, 0)}$, (0))), (0)))$]_{\mu_{\vec{\alpha}}}$, $-1$ )

\item $U_{\tree}(((4), (14), (6), (2), (1)))=$ the completion of $U(((4), (14), (6), (2)))$ that sends $(2, ((0, 0), (0)))$ to the completion of $U_{\tree}(((4), (14), (6), (2)))(2, ((0, 0)))$  whose node component is $(0, 0, 0)$

\noindent$U_{\node}((4), (14), (6), (2), (1))= (0, -1)$

( (1), $[ \vec{\alpha} \mapsto h$(2, ($g$(2, ($\alpha_{(0)}$, (0))), $-1$))$]_{\mu_{\vec{\alpha}}}$, (0), $[ \vec{\alpha} \mapsto h$(2, ($g$(2, ($\alpha_{(0)}$, (0), $\alpha_{(0, 0)}$, (0))), (0), $g$(2, ($\alpha_{(0)}$, (0, 0), $\alpha_{(0, 0)}$, (0))), (0)))$]_{\mu_{\vec{\alpha}}}$, $-1$ )

\item $U_{\tree}(((4), (14), (6), (2), (0)))=$ the completion of $U(((4), (14), (6), (2)))$ that sends $(2, ((0, 0), (0)))$ to the completion of $U_{\tree}(((4), (14), (6), (2)))(2, ((0, 0)))$  whose node component is $-1$

\noindent$U_{\node}((4), (14), (6), (2), (0))= (0, -1)$

( (1), $[ \vec{\alpha} \mapsto h$(2, ($g$(2, ($\alpha_{(0)}$, (0))), $-1$))$]_{\mu_{\vec{\alpha}}}$, (0), $[ \vec{\alpha} \mapsto h$(2, ($g$(2, ($\alpha_{(0)}$, (0), $\alpha_{(0, 0)}$, (0))), (0), $g$(2, ($\alpha_{(0)}$, (0, 0), $\alpha_{(0, 0)}$, (0))), (0)))$]_{\mu_{\vec{\alpha}}}$, $-1$ )

\item $U_{\tree}(((4), (14), (6), (1)))=$ the completion of $U(((4), (14), (6)))$ that sends $(2, ((0, 0)))$ to the completion of $U_{\tree}(((4), (14), (6)))(2, \emptyset)$  whose node component is $(0, 0)$

\noindent$U_{\node}((4), (14), (6), (1))= (0, -1)$

( (1), $[ \vec{\alpha} \mapsto h$(2, ($g$(2, ($\alpha_{(0)}$, (0))), $-1$))$]_{\mu_{\vec{\alpha}}}$, (0), $[ \vec{\alpha} \mapsto h$(2, ($g$(2, ($\alpha_{(0)}$, (0), $\alpha_{(0, 0)}$, (0))), (0), $g$(2, ($\alpha_{(0)}$, (0, 0), $\alpha_{(0, 0)}$, $-1$)), (0)))$]_{\mu_{\vec{\alpha}}}$, $-1$ )

\item $U_{\tree}(((4), (14), (6), (0)))=$ the completion of $U(((4), (14), (6)))$ that sends $(2, ((0, 0)))$ to the completion of $U_{\tree}(((4), (14), (6)))(2, \emptyset)$  whose node component is $-1$

\noindent$U_{\node}((4), (14), (6), (0))= (0, -1)$

( (1), $[ \vec{\alpha} \mapsto h$(2, ($g$(2, ($\alpha_{(0)}$, (0))), $-1$))$]_{\mu_{\vec{\alpha}}}$, (0), $[ \vec{\alpha} \mapsto h$(2, ($g$(2, ($\alpha_{(0)}$, (0), $\alpha_{(0, 0)}$, (0))), (0), $g$(2, ($\alpha_{(0)}$, (0, 0))), (0)))$]_{\mu_{\vec{\alpha}}}$, $-1$ )

\item $U_{\tree}(((4), (14), (5)))=$ the completion of $U(((4), (14)))$ that sends $(2, ((0), (0)))$ to the completion of $U_{\tree}(((4), (14)))(2, ((0)))$  whose node component is $-1$

\noindent$U_{\node}((4), (14), (5))= (0, -1)$

( (1), $[ \vec{\alpha} \mapsto h$(2, ($g$(2, ($\alpha_{(0)}$, (0))), $-1$))$]_{\mu_{\vec{\alpha}}}$, (0), $[ \vec{\alpha} \mapsto h$(2, ($g$(2, ($\alpha_{(0)}$, (0), $\alpha_{(0, 0)}$, (0))), (0), $g$(2, ($\alpha_{(0)}$, $-1$)), (0)))$]_{\mu_{\vec{\alpha}}}$, $-1$ )

\item $U_{\tree}(((4), (14), (4)))=$ the completion of $U(((4), (14)))$ that sends $(2, ((0), (0)))$ to the completion of $U_{\tree}(((4), (14)))(2, ((0)))$  whose node component is $-1$

\noindent$U_{\node}((4), (14), (4))= (2, ((1)))$

( (1), $[ \vec{\alpha} \mapsto h$(2, ($g$(2, ($\alpha_{(0)}$, (0))), $-1$))$]_{\mu_{\vec{\alpha}}}$, (0), $[ \vec{\alpha} \mapsto h$(2, ($g$(2, ($\alpha_{(0)}$, (0), $\alpha_{(0, 0)}$, (0))), (0), $g$(2, ($\alpha_{(0)}$, $-1$)), $-1$))$]_{\mu_{\vec{\alpha}}}$, $-1$ )

\item $U_{\tree}(((4), (14), (4), (1)))=$ the completion of $U(((4), (14), (4)))$ that sends $(2, ((1)))$ to the completion of $U_{\tree}(((4), (14), (4)))(2, \emptyset)$  whose node component is $(0, 0)$

\noindent$U_{\node}((4), (14), (4), (1))= (0, -1)$

( (1), $[ \vec{\alpha} \mapsto h$(2, ($g$(2, ($\alpha_{(0)}$, (0))), $-1$))$]_{\mu_{\vec{\alpha}}}$, (0), $[ \vec{\alpha} \mapsto h$(2, ($g$(2, ($\alpha_{(0)}$, (0), $\alpha_{(0, 0)}$, (0))), (0), $g$(2, ($\alpha_{(0, 0)}$, (1))), (0)))$]_{\mu_{\vec{\alpha}}}$, $-1$ )

\item $U_{\tree}(((4), (14), (4), (0)))=$ the completion of $U(((4), (14), (4)))$ that sends $(2, ((1)))$ to the completion of $U_{\tree}(((4), (14), (4)))(2, \emptyset)$  whose node component is $-1$

\noindent$U_{\node}((4), (14), (4), (0))= (0, -1)$

( (1), $[ \vec{\alpha} \mapsto h$(2, ($g$(2, ($\alpha_{(0)}$, (0))), $-1$))$]_{\mu_{\vec{\alpha}}}$, (0), $[ \vec{\alpha} \mapsto h$(2, ($g$(2, ($\alpha_{(0)}$, (0), $\alpha_{(0, 0)}$, (0))), (0), $g$(2, ($\alpha_{(0, 0)}$, (1))), (0)))$]_{\mu_{\vec{\alpha}}}$, $-1$ )

\item $U_{\tree}(((4), (14), (3)))=$ the completion of $U(((4), (14)))$ that sends $(2, ((0), (0)))$ to the completion of $U_{\tree}(((4), (14)))(2, ((0)))$  whose node component is $-1$

\noindent$U_{\node}((4), (14), (3))= (0, -1)$

( (1), $[ \vec{\alpha} \mapsto h$(2, ($g$(2, ($\alpha_{(0)}$, (0))), $-1$))$]_{\mu_{\vec{\alpha}}}$, (0), $[ \vec{\alpha} \mapsto h$(2, ($g$(2, ($\alpha_{(0)}$, (0), $\alpha_{(0, 0)}$, (0))), (0), $g$(2, ($\alpha_{(0, 0)}$, (0))), (0)))$]_{\mu_{\vec{\alpha}}}$, $-1$ )

\item $U_{\tree}(((4), (14), (2)))=$ the completion of $U(((4), (14)))$ that sends $(2, ((0), (0)))$ to the completion of $U_{\tree}(((4), (14)))(2, ((0)))$  whose node component is $-1$

\noindent$U_{\node}((4), (14), (2))= (2, ((0, 0)))$

( (1), $[ \vec{\alpha} \mapsto h$(2, ($g$(2, ($\alpha_{(0)}$, (0))), $-1$))$]_{\mu_{\vec{\alpha}}}$, (0), $[ \vec{\alpha} \mapsto h$(2, ($g$(2, ($\alpha_{(0)}$, (0), $\alpha_{(0, 0)}$, (0))), (0), $g$(2, ($\alpha_{(0, 0)}$, (0))), $-1$))$]_{\mu_{\vec{\alpha}}}$, $-1$ )

\item $U_{\tree}(((4), (14), (2), (1)))=$ the completion of $U(((4), (14), (2)))$ that sends $(2, ((0, 0)))$ to the completion of $U_{\tree}(((4), (14), (2)))(2, \emptyset)$  whose node component is $(0, 0)$

\noindent$U_{\node}((4), (14), (2), (1))= (0, -1)$

( (1), $[ \vec{\alpha} \mapsto h$(2, ($g$(2, ($\alpha_{(0)}$, (0))), $-1$))$]_{\mu_{\vec{\alpha}}}$, (0), $[ \vec{\alpha} \mapsto h$(2, ($g$(2, ($\alpha_{(0)}$, (0), $\alpha_{(0, 0)}$, (0))), (0), $g$(2, ($\alpha_{(0, 0)}$, (0, 0))), (0)))$]_{\mu_{\vec{\alpha}}}$, $-1$ )

\item $U_{\tree}(((4), (14), (2), (0)))=$ the completion of $U(((4), (14), (2)))$ that sends $(2, ((0, 0)))$ to the completion of $U_{\tree}(((4), (14), (2)))(2, \emptyset)$  whose node component is $-1$

\noindent$U_{\node}((4), (14), (2), (0))= (0, -1)$

( (1), $[ \vec{\alpha} \mapsto h$(2, ($g$(2, ($\alpha_{(0)}$, (0))), $-1$))$]_{\mu_{\vec{\alpha}}}$, (0), $[ \vec{\alpha} \mapsto h$(2, ($g$(2, ($\alpha_{(0)}$, (0), $\alpha_{(0, 0)}$, (0))), (0), $g$(2, ($\alpha_{(0, 0)}$, (0, 0))), (0)))$]_{\mu_{\vec{\alpha}}}$, $-1$ )

\item $U_{\tree}(((4), (14), (1)))=$ the completion of $U(((4), (14)))$ that sends $(2, ((0), (0)))$ to the completion of $U_{\tree}(((4), (14)))(2, ((0)))$  whose node component is $-1$

\noindent$U_{\node}((4), (14), (1))= (0, -1)$

( (1), $[ \vec{\alpha} \mapsto h$(2, ($g$(2, ($\alpha_{(0)}$, (0))), $-1$))$]_{\mu_{\vec{\alpha}}}$, (0), $[ \vec{\alpha} \mapsto h$(2, ($g$(2, ($\alpha_{(0)}$, (0), $\alpha_{(0, 0)}$, (0))), (0), $g$(2, ($\alpha_{(0, 0)}$, $-1$)), (0)))$]_{\mu_{\vec{\alpha}}}$, $-1$ )

\item $U_{\tree}(((4), (14), (0)))=$ the completion of $U(((4), (14)))$ that sends $(2, ((0), (0)))$ to the completion of $U_{\tree}(((4), (14)))(2, ((0)))$  whose node component is $-1$

\noindent$U_{\node}((4), (14), (0))= (1, (0))$

( (1), $[ \vec{\alpha} \mapsto h$(2, ($g$(2, ($\alpha_{(0)}$, (0))), $-1$))$]_{\mu_{\vec{\alpha}}}$, (0), $[ \vec{\alpha} \mapsto h$(2, ($g$(2, ($\alpha_{(0)}$, (0), $\alpha_{(0, 0)}$, (0))), (0), $g$(2, ($\alpha_{(0, 0)}$, $-1$)), $-1$))$]_{\mu_{\vec{\alpha}}}$, $-1$ )

\item $U_{\tree}(((4), (14), (0), (0)))=$ the unique completion of $U(((4), (14), (0)))$

\noindent$U_{\node}((4), (14), (0), (0))= (0, -1)$

( (1), $[ \vec{\alpha} \mapsto h$(2, ($g$(2, ($\alpha_{(0)}$, (0))), $-1$))$]_{\mu_{\vec{\alpha}}}$, (0), $[ \vec{\alpha} \mapsto h$(2, ($g$(2, ($\alpha_{(0)}$, (0), $\alpha_{(0, 0)}$, (0))), (0), $g$(1, ((0))), (0)))$]_{\mu_{\vec{\alpha}}}$, $-1$ )

\item $U_{\tree}(((4), (13)))=$ the completion of $U(((4)))$ that sends $(2, ((0)))$ to the completion of $U_{\tree}(((4)))(2, \emptyset)$  whose node component is $(0, 0)$

\noindent$U_{\node}((4), (13))= (2, ((0, 0)))$

( (1), $[ \vec{\alpha} \mapsto h$(2, ($g$(2, ($\alpha_{(0)}$, (0))), $-1$))$]_{\mu_{\vec{\alpha}}}$, (0), $[ \vec{\alpha} \mapsto h$(2, ($g$(2, ($\alpha_{(0)}$, (0), $\alpha_{(0, 0)}$, $-1$)), (0)))$]_{\mu_{\vec{\alpha}}}$, $-1$ )

\item $U_{\tree}(((4), (13), (3)))=$ the completion of $U(((4), (13)))$ that sends $(2, ((0, 0)))$ to the completion of $U_{\tree}(((4), (13)))(2, \emptyset)$  whose node component is $(0, 0)$

\noindent$U_{\node}((4), (13), (3))= (0, -1)$

( (1), $[ \vec{\alpha} \mapsto h$(2, ($g$(2, ($\alpha_{(0)}$, (0))), $-1$))$]_{\mu_{\vec{\alpha}}}$, (0), $[ \vec{\alpha} \mapsto h$(2, ($g$(2, ($\alpha_{(0)}$, (0), $\alpha_{(0, 0)}$, $-1$)), (0), $g$(2, ($\alpha_{(0)}$, (0, 0))), (0)))$]_{\mu_{\vec{\alpha}}}$, $-1$ )

\item $U_{\tree}(((4), (13), (2)))=$ the completion of $U(((4), (13)))$ that sends $(2, ((0, 0)))$ to the completion of $U_{\tree}(((4), (13)))(2, \emptyset)$  whose node component is $(0, 0)$

\noindent$U_{\node}((4), (13), (2))= (2, ((0, 0), (0)))$

( (1), $[ \vec{\alpha} \mapsto h$(2, ($g$(2, ($\alpha_{(0)}$, (0))), $-1$))$]_{\mu_{\vec{\alpha}}}$, (0), $[ \vec{\alpha} \mapsto h$(2, ($g$(2, ($\alpha_{(0)}$, (0), $\alpha_{(0, 0)}$, $-1$)), (0), $g$(2, ($\alpha_{(0)}$, (0, 0))), $-1$))$]_{\mu_{\vec{\alpha}}}$, $-1$ )

\item $U_{\tree}(((4), (13), (2), (2)))=$ the completion of $U(((4), (13), (2)))$ that sends $(2, ((0, 0), (0)))$ to the completion of $U_{\tree}(((4), (13), (2)))(2, ((0, 0)))$  whose node component is $(0, 1)$

\noindent$U_{\node}((4), (13), (2), (2))= (0, -1)$

( (1), $[ \vec{\alpha} \mapsto h$(2, ($g$(2, ($\alpha_{(0)}$, (0))), $-1$))$]_{\mu_{\vec{\alpha}}}$, (0), $[ \vec{\alpha} \mapsto h$(2, ($g$(2, ($\alpha_{(0)}$, (0), $\alpha_{(0, 0)}$, $-1$)), (0), $g$(2, ($\alpha_{(0)}$, (0, 0), $\alpha_{(0, 0)}$, (0))), (0)))$]_{\mu_{\vec{\alpha}}}$, $-1$ )

\item $U_{\tree}(((4), (13), (2), (1)))=$ the completion of $U(((4), (13), (2)))$ that sends $(2, ((0, 0), (0)))$ to the completion of $U_{\tree}(((4), (13), (2)))(2, ((0, 0)))$  whose node component is $(0, 0, 0)$

\noindent$U_{\node}((4), (13), (2), (1))= (0, -1)$

( (1), $[ \vec{\alpha} \mapsto h$(2, ($g$(2, ($\alpha_{(0)}$, (0))), $-1$))$]_{\mu_{\vec{\alpha}}}$, (0), $[ \vec{\alpha} \mapsto h$(2, ($g$(2, ($\alpha_{(0)}$, (0), $\alpha_{(0, 0)}$, $-1$)), (0), $g$(2, ($\alpha_{(0)}$, (0, 0), $\alpha_{(0, 0)}$, (0))), (0)))$]_{\mu_{\vec{\alpha}}}$, $-1$ )

\item $U_{\tree}(((4), (13), (2), (0)))=$ the completion of $U(((4), (13), (2)))$ that sends $(2, ((0, 0), (0)))$ to the completion of $U_{\tree}(((4), (13), (2)))(2, ((0, 0)))$  whose node component is $-1$

\noindent$U_{\node}((4), (13), (2), (0))= (0, -1)$

( (1), $[ \vec{\alpha} \mapsto h$(2, ($g$(2, ($\alpha_{(0)}$, (0))), $-1$))$]_{\mu_{\vec{\alpha}}}$, (0), $[ \vec{\alpha} \mapsto h$(2, ($g$(2, ($\alpha_{(0)}$, (0), $\alpha_{(0, 0)}$, $-1$)), (0), $g$(2, ($\alpha_{(0)}$, (0, 0), $\alpha_{(0, 0)}$, (0))), (0)))$]_{\mu_{\vec{\alpha}}}$, $-1$ )

\item $U_{\tree}(((4), (13), (1)))=$ the completion of $U(((4), (13)))$ that sends $(2, ((0, 0)))$ to the completion of $U_{\tree}(((4), (13)))(2, \emptyset)$  whose node component is $(0, 0)$

\noindent$U_{\node}((4), (13), (1))= (0, -1)$

( (1), $[ \vec{\alpha} \mapsto h$(2, ($g$(2, ($\alpha_{(0)}$, (0))), $-1$))$]_{\mu_{\vec{\alpha}}}$, (0), $[ \vec{\alpha} \mapsto h$(2, ($g$(2, ($\alpha_{(0)}$, (0), $\alpha_{(0, 0)}$, $-1$)), (0), $g$(2, ($\alpha_{(0)}$, (0, 0), $\alpha_{(0, 0)}$, $-1$)), (0)))$]_{\mu_{\vec{\alpha}}}$, $-1$ )

\item $U_{\tree}(((4), (13), (0)))=$ the completion of $U(((4), (13)))$ that sends $(2, ((0, 0)))$ to the completion of $U_{\tree}(((4), (13)))(2, \emptyset)$  whose node component is $-1$

\noindent$U_{\node}((4), (13), (0))= (0, -1)$

( (1), $[ \vec{\alpha} \mapsto h$(2, ($g$(2, ($\alpha_{(0)}$, (0))), $-1$))$]_{\mu_{\vec{\alpha}}}$, (0), $[ \vec{\alpha} \mapsto h$(2, ($g$(2, ($\alpha_{(0)}$, (0), $\alpha_{(0, 0)}$, $-1$)), (0), $g$(2, ($\alpha_{(0)}$, (0, 0))), (0)))$]_{\mu_{\vec{\alpha}}}$, $-1$ )

\item $U_{\tree}(((4), (12)))=$ the completion of $U(((4)))$ that sends $(2, ((0)))$ to the completion of $U_{\tree}(((4)))(2, \emptyset)$  whose node component is $(0, 0)$

\noindent$U_{\node}((4), (12))= (0, -1)$

( (1), $[ \vec{\alpha} \mapsto h$(2, ($g$(2, ($\alpha_{(0)}$, (0))), $-1$))$]_{\mu_{\vec{\alpha}}}$, (0), $[ \vec{\alpha} \mapsto h$(2, ($g$(2, ($\alpha_{(0)}$, (0), $\alpha_{(0, 0)}$, $-1$)), (0), $g$(2, ($\alpha_{(0)}$, $-1$)), (0)))$]_{\mu_{\vec{\alpha}}}$, $-1$ )

\item $U_{\tree}(((4), (11)))=$ the completion of $U(((4)))$ that sends $(2, ((0)))$ to the completion of $U_{\tree}(((4)))(2, \emptyset)$  whose node component is $(0, 0)$

\noindent$U_{\node}((4), (11))= (2, ((1)))$

( (1), $[ \vec{\alpha} \mapsto h$(2, ($g$(2, ($\alpha_{(0)}$, (0))), $-1$))$]_{\mu_{\vec{\alpha}}}$, (0), $[ \vec{\alpha} \mapsto h$(2, ($g$(2, ($\alpha_{(0)}$, (0), $\alpha_{(0, 0)}$, $-1$)), (0), $g$(2, ($\alpha_{(0)}$, $-1$)), $-1$))$]_{\mu_{\vec{\alpha}}}$, $-1$ )

\item $U_{\tree}(((4), (11), (1)))=$ the completion of $U(((4), (11)))$ that sends $(2, ((1)))$ to the completion of $U_{\tree}(((4), (11)))(2, \emptyset)$  whose node component is $(0, 0)$

\noindent$U_{\node}((4), (11), (1))= (0, -1)$

( (1), $[ \vec{\alpha} \mapsto h$(2, ($g$(2, ($\alpha_{(0)}$, (0))), $-1$))$]_{\mu_{\vec{\alpha}}}$, (0), $[ \vec{\alpha} \mapsto h$(2, ($g$(2, ($\alpha_{(0)}$, (0), $\alpha_{(0, 0)}$, $-1$)), (0), $g$(2, ($\alpha_{(0, 0)}$, (1))), (0)))$]_{\mu_{\vec{\alpha}}}$, $-1$ )

\item $U_{\tree}(((4), (11), (0)))=$ the completion of $U(((4), (11)))$ that sends $(2, ((1)))$ to the completion of $U_{\tree}(((4), (11)))(2, \emptyset)$  whose node component is $-1$

\noindent$U_{\node}((4), (11), (0))= (0, -1)$

( (1), $[ \vec{\alpha} \mapsto h$(2, ($g$(2, ($\alpha_{(0)}$, (0))), $-1$))$]_{\mu_{\vec{\alpha}}}$, (0), $[ \vec{\alpha} \mapsto h$(2, ($g$(2, ($\alpha_{(0)}$, (0), $\alpha_{(0, 0)}$, $-1$)), (0), $g$(2, ($\alpha_{(0, 0)}$, (1))), (0)))$]_{\mu_{\vec{\alpha}}}$, $-1$ )

\item $U_{\tree}(((4), (10)))=$ the completion of $U(((4)))$ that sends $(2, ((0)))$ to the completion of $U_{\tree}(((4)))(2, \emptyset)$  whose node component is $(0, 0)$

\noindent$U_{\node}((4), (10))= (0, -1)$

( (1), $[ \vec{\alpha} \mapsto h$(2, ($g$(2, ($\alpha_{(0)}$, (0))), $-1$))$]_{\mu_{\vec{\alpha}}}$, (0), $[ \vec{\alpha} \mapsto h$(2, ($g$(2, ($\alpha_{(0)}$, (0), $\alpha_{(0, 0)}$, $-1$)), (0), $g$(2, ($\alpha_{(0, 0)}$, (0))), (0)))$]_{\mu_{\vec{\alpha}}}$, $-1$ )

\item $U_{\tree}(((4), (9)))=$ the completion of $U(((4)))$ that sends $(2, ((0)))$ to the completion of $U_{\tree}(((4)))(2, \emptyset)$  whose node component is $(0, 0)$

\noindent$U_{\node}((4), (9))= (2, ((0, 0)))$

( (1), $[ \vec{\alpha} \mapsto h$(2, ($g$(2, ($\alpha_{(0)}$, (0))), $-1$))$]_{\mu_{\vec{\alpha}}}$, (0), $[ \vec{\alpha} \mapsto h$(2, ($g$(2, ($\alpha_{(0)}$, (0), $\alpha_{(0, 0)}$, $-1$)), (0), $g$(2, ($\alpha_{(0, 0)}$, (0))), $-1$))$]_{\mu_{\vec{\alpha}}}$, $-1$ )

\item $U_{\tree}(((4), (9), (1)))=$ the completion of $U(((4), (9)))$ that sends $(2, ((0, 0)))$ to the completion of $U_{\tree}(((4), (9)))(2, \emptyset)$  whose node component is $(0, 0)$

\noindent$U_{\node}((4), (9), (1))= (0, -1)$

( (1), $[ \vec{\alpha} \mapsto h$(2, ($g$(2, ($\alpha_{(0)}$, (0))), $-1$))$]_{\mu_{\vec{\alpha}}}$, (0), $[ \vec{\alpha} \mapsto h$(2, ($g$(2, ($\alpha_{(0)}$, (0), $\alpha_{(0, 0)}$, $-1$)), (0), $g$(2, ($\alpha_{(0, 0)}$, (0, 0))), (0)))$]_{\mu_{\vec{\alpha}}}$, $-1$ )

\item $U_{\tree}(((4), (9), (0)))=$ the completion of $U(((4), (9)))$ that sends $(2, ((0, 0)))$ to the completion of $U_{\tree}(((4), (9)))(2, \emptyset)$  whose node component is $-1$

\noindent$U_{\node}((4), (9), (0))= (0, -1)$

( (1), $[ \vec{\alpha} \mapsto h$(2, ($g$(2, ($\alpha_{(0)}$, (0))), $-1$))$]_{\mu_{\vec{\alpha}}}$, (0), $[ \vec{\alpha} \mapsto h$(2, ($g$(2, ($\alpha_{(0)}$, (0), $\alpha_{(0, 0)}$, $-1$)), (0), $g$(2, ($\alpha_{(0, 0)}$, (0, 0))), (0)))$]_{\mu_{\vec{\alpha}}}$, $-1$ )

\item $U_{\tree}(((4), (8)))=$ the completion of $U(((4)))$ that sends $(2, ((0)))$ to the completion of $U_{\tree}(((4)))(2, \emptyset)$  whose node component is $(0, 0)$

\noindent$U_{\node}((4), (8))= (0, -1)$

( (1), $[ \vec{\alpha} \mapsto h$(2, ($g$(2, ($\alpha_{(0)}$, (0))), $-1$))$]_{\mu_{\vec{\alpha}}}$, (0), $[ \vec{\alpha} \mapsto h$(2, ($g$(2, ($\alpha_{(0)}$, (0), $\alpha_{(0, 0)}$, $-1$)), (0), $g$(2, ($\alpha_{(0, 0)}$, $-1$)), (0)))$]_{\mu_{\vec{\alpha}}}$, $-1$ )

\item $U_{\tree}(((4), (7)))=$ the completion of $U(((4)))$ that sends $(2, ((0)))$ to the completion of $U_{\tree}(((4)))(2, \emptyset)$  whose node component is $(0, 0)$

\noindent$U_{\node}((4), (7))= (1, (0))$

( (1), $[ \vec{\alpha} \mapsto h$(2, ($g$(2, ($\alpha_{(0)}$, (0))), $-1$))$]_{\mu_{\vec{\alpha}}}$, (0), $[ \vec{\alpha} \mapsto h$(2, ($g$(2, ($\alpha_{(0)}$, (0), $\alpha_{(0, 0)}$, $-1$)), (0), $g$(2, ($\alpha_{(0, 0)}$, $-1$)), $-1$))$]_{\mu_{\vec{\alpha}}}$, $-1$ )

\item $U_{\tree}(((4), (7), (0)))=$ the unique completion of $U(((4), (7)))$

\noindent$U_{\node}((4), (7), (0))= (0, -1)$

( (1), $[ \vec{\alpha} \mapsto h$(2, ($g$(2, ($\alpha_{(0)}$, (0))), $-1$))$]_{\mu_{\vec{\alpha}}}$, (0), $[ \vec{\alpha} \mapsto h$(2, ($g$(2, ($\alpha_{(0)}$, (0), $\alpha_{(0, 0)}$, $-1$)), (0), $g$(1, ((0))), (0)))$]_{\mu_{\vec{\alpha}}}$, $-1$ )

\item $U_{\tree}(((4), (6)))=$ the completion of $U(((4)))$ that sends $(2, ((0)))$ to the completion of $U_{\tree}(((4)))(2, \emptyset)$  whose node component is $-1$

\noindent$U_{\node}((4), (6))= (2, ((0, 0)))$

( (1), $[ \vec{\alpha} \mapsto h$(2, ($g$(2, ($\alpha_{(0)}$, (0))), (0)))$]_{\mu_{\vec{\alpha}}}$, $-1$ )

\item $U_{\tree}(((4), (6), (3)))=$ the completion of $U(((4), (6)))$ that sends $(2, ((0, 0)))$ to the completion of $U_{\tree}(((4), (6)))(2, \emptyset)$  whose node component is $(0, 0)$

\noindent$U_{\node}((4), (6), (3))= (0, -1)$

( (1), $[ \vec{\alpha} \mapsto h$(2, ($g$(2, ($\alpha_{(0)}$, (0))), (0), $g$(2, ($\alpha_{(0)}$, (0, 0))), (0)))$]_{\mu_{\vec{\alpha}}}$, $-1$ )

\item $U_{\tree}(((4), (6), (2)))=$ the completion of $U(((4), (6)))$ that sends $(2, ((0, 0)))$ to the completion of $U_{\tree}(((4), (6)))(2, \emptyset)$  whose node component is $(0, 0)$

\noindent$U_{\node}((4), (6), (2))= (2, ((0, 0), (0)))$

( (1), $[ \vec{\alpha} \mapsto h$(2, ($g$(2, ($\alpha_{(0)}$, (0))), (0), $g$(2, ($\alpha_{(0)}$, (0, 0))), $-1$))$]_{\mu_{\vec{\alpha}}}$, (0) )

\item $U_{\tree}(((4), (6), (2), (2)))=$ the completion of $U(((4), (6), (2)))$ that sends $(2, ((0, 0), (0)))$ to the completion of $U_{\tree}(((4), (6), (2)))(2, ((0, 0)))$  whose node component is $(0, 1)$

\noindent$U_{\node}((4), (6), (2), (2))= (0, -1)$

( (1), $[ \vec{\alpha} \mapsto h$(2, ($g$(2, ($\alpha_{(0)}$, (0))), (0), $g$(2, ($\alpha_{(0)}$, (0, 0))), $-1$))$]_{\mu_{\vec{\alpha}}}$, (0), $[ \vec{\alpha} \mapsto h$(2, ($g$(2, ($\alpha_{(0)}$, (0))), (0), $g$(2, ($\alpha_{(0)}$, (0, 0), $\alpha_{(0, 0)}$, (0))), (0)))$]_{\mu_{\vec{\alpha}}}$, $-1$ )

\item $U_{\tree}(((4), (6), (2), (1)))=$ the completion of $U(((4), (6), (2)))$ that sends $(2, ((0, 0), (0)))$ to the completion of $U_{\tree}(((4), (6), (2)))(2, ((0, 0)))$  whose node component is $(0, 0, 0)$

\noindent$U_{\node}((4), (6), (2), (1))= (0, -1)$

( (1), $[ \vec{\alpha} \mapsto h$(2, ($g$(2, ($\alpha_{(0)}$, (0))), (0), $g$(2, ($\alpha_{(0)}$, (0, 0))), $-1$))$]_{\mu_{\vec{\alpha}}}$, (0), $[ \vec{\alpha} \mapsto h$(2, ($g$(2, ($\alpha_{(0)}$, (0))), (0), $g$(2, ($\alpha_{(0)}$, (0, 0), $\alpha_{(0, 0)}$, (0))), (0)))$]_{\mu_{\vec{\alpha}}}$, $-1$ )

\item $U_{\tree}(((4), (6), (2), (0)))=$ the completion of $U(((4), (6), (2)))$ that sends $(2, ((0, 0), (0)))$ to the completion of $U_{\tree}(((4), (6), (2)))(2, ((0, 0)))$  whose node component is $-1$

\noindent$U_{\node}((4), (6), (2), (0))= (0, -1)$

( (1), $[ \vec{\alpha} \mapsto h$(2, ($g$(2, ($\alpha_{(0)}$, (0))), (0), $g$(2, ($\alpha_{(0)}$, (0, 0))), $-1$))$]_{\mu_{\vec{\alpha}}}$, (0), $[ \vec{\alpha} \mapsto h$(2, ($g$(2, ($\alpha_{(0)}$, (0))), (0), $g$(2, ($\alpha_{(0)}$, (0, 0), $\alpha_{(0, 0)}$, (0))), (0)))$]_{\mu_{\vec{\alpha}}}$, $-1$ )

\item $U_{\tree}(((4), (6), (1)))=$ the completion of $U(((4), (6)))$ that sends $(2, ((0, 0)))$ to the completion of $U_{\tree}(((4), (6)))(2, \emptyset)$  whose node component is $(0, 0)$

\noindent$U_{\node}((4), (6), (1))= (0, -1)$

( (1), $[ \vec{\alpha} \mapsto h$(2, ($g$(2, ($\alpha_{(0)}$, (0))), (0), $g$(2, ($\alpha_{(0)}$, (0, 0))), $-1$))$]_{\mu_{\vec{\alpha}}}$, (0), $[ \vec{\alpha} \mapsto h$(2, ($g$(2, ($\alpha_{(0)}$, (0))), (0), $g$(2, ($\alpha_{(0)}$, (0, 0), $\alpha_{(0, 0)}$, $-1$)), (0)))$]_{\mu_{\vec{\alpha}}}$, $-1$ )

\item $U_{\tree}(((4), (6), (0)))=$ the completion of $U(((4), (6)))$ that sends $(2, ((0, 0)))$ to the completion of $U_{\tree}(((4), (6)))(2, \emptyset)$  whose node component is $-1$

\noindent$U_{\node}((4), (6), (0))= (0, -1)$

( (1), $[ \vec{\alpha} \mapsto h$(2, ($g$(2, ($\alpha_{(0)}$, (0))), (0), $g$(2, ($\alpha_{(0)}$, (0, 0))), (0)))$]_{\mu_{\vec{\alpha}}}$, $-1$ )

\item $U_{\tree}(((4), (5)))=$ the completion of $U(((4)))$ that sends $(2, ((0)))$ to the completion of $U_{\tree}(((4)))(2, \emptyset)$  whose node component is $-1$

\noindent$U_{\node}((4), (5))= (0, -1)$

( (1), $[ \vec{\alpha} \mapsto h$(2, ($g$(2, ($\alpha_{(0)}$, (0))), (0), $g$(2, ($\alpha_{(0)}$, $-1$)), (0)))$]_{\mu_{\vec{\alpha}}}$, $-1$ )

\item $U_{\tree}(((4), (4)))=$ the completion of $U(((4)))$ that sends $(2, ((0)))$ to the completion of $U_{\tree}(((4)))(2, \emptyset)$  whose node component is $-1$

\noindent$U_{\node}((4), (4))= (2, ((1)))$

( (1), $[ \vec{\alpha} \mapsto h$(2, ($g$(2, ($\alpha_{(0)}$, (0))), (0), $g$(2, ($\alpha_{(0)}$, $-1$)), $-1$))$]_{\mu_{\vec{\alpha}}}$, (0) )

\item $U_{\tree}(((4), (4), (1)))=$ the completion of $U(((4), (4)))$ that sends $(2, ((1)))$ to the completion of $U_{\tree}(((4), (4)))(2, \emptyset)$  whose node component is $(0, 0)$

\noindent$U_{\node}((4), (4), (1))= (0, -1)$

( (1), $[ \vec{\alpha} \mapsto h$(2, ($g$(2, ($\alpha_{(0)}$, (0))), (0), $g$(2, ($\alpha_{(0)}$, $-1$)), $-1$))$]_{\mu_{\vec{\alpha}}}$, (0), $[ \vec{\alpha} \mapsto h$(2, ($g$(2, ($\alpha_{(0)}$, (0))), (0), $g$(2, ($\alpha_{(0, 0)}$, (1))), (0)))$]_{\mu_{\vec{\alpha}}}$, $-1$ )

\item $U_{\tree}(((4), (4), (0)))=$ the completion of $U(((4), (4)))$ that sends $(2, ((1)))$ to the completion of $U_{\tree}(((4), (4)))(2, \emptyset)$  whose node component is $-1$

\noindent$U_{\node}((4), (4), (0))= (0, -1)$

( (1), $[ \vec{\alpha} \mapsto h$(2, ($g$(2, ($\alpha_{(0)}$, (0))), (0), $g$(2, ($\alpha_{(0)}$, $-1$)), $-1$))$]_{\mu_{\vec{\alpha}}}$, (0), $[ \vec{\alpha} \mapsto h$(2, ($g$(2, ($\alpha_{(0)}$, (0))), (0), $g$(2, ($\alpha_{(0, 0)}$, (1))), (0)))$]_{\mu_{\vec{\alpha}}}$, $-1$ )

\item $U_{\tree}(((4), (3)))=$ the completion of $U(((4)))$ that sends $(2, ((0)))$ to the completion of $U_{\tree}(((4)))(2, \emptyset)$  whose node component is $-1$

\noindent$U_{\node}((4), (3))= (0, -1)$

( (1), $[ \vec{\alpha} \mapsto h$(2, ($g$(2, ($\alpha_{(0)}$, (0))), (0), $g$(2, ($\alpha_{(0)}$, $-1$)), $-1$))$]_{\mu_{\vec{\alpha}}}$, (0), $[ \vec{\alpha} \mapsto h$(2, ($g$(2, ($\alpha_{(0)}$, (0))), (0), $g$(2, ($\alpha_{(0, 0)}$, (0))), (0)))$]_{\mu_{\vec{\alpha}}}$, $-1$ )

\item $U_{\tree}(((4), (2)))=$ the completion of $U(((4)))$ that sends $(2, ((0)))$ to the completion of $U_{\tree}(((4)))(2, \emptyset)$  whose node component is $-1$

\noindent$U_{\node}((4), (2))= (2, ((0, 0)))$

( (1), $[ \vec{\alpha} \mapsto h$(2, ($g$(2, ($\alpha_{(0)}$, (0))), (0), $g$(2, ($\alpha_{(0)}$, $-1$)), $-1$))$]_{\mu_{\vec{\alpha}}}$, (0), $[ \vec{\alpha} \mapsto h$(2, ($g$(2, ($\alpha_{(0)}$, (0))), (0), $g$(2, ($\alpha_{(0, 0)}$, (0))), $-1$))$]_{\mu_{\vec{\alpha}}}$, $-1$ )

\item $U_{\tree}(((4), (2), (1)))=$ the completion of $U(((4), (2)))$ that sends $(2, ((0, 0)))$ to the completion of $U_{\tree}(((4), (2)))(2, \emptyset)$  whose node component is $(0, 0)$

\noindent$U_{\node}((4), (2), (1))= (0, -1)$

( (1), $[ \vec{\alpha} \mapsto h$(2, ($g$(2, ($\alpha_{(0)}$, (0))), (0), $g$(2, ($\alpha_{(0)}$, $-1$)), $-1$))$]_{\mu_{\vec{\alpha}}}$, (0), $[ \vec{\alpha} \mapsto h$(2, ($g$(2, ($\alpha_{(0)}$, (0))), (0), $g$(2, ($\alpha_{(0, 0)}$, (0, 0))), (0)))$]_{\mu_{\vec{\alpha}}}$, $-1$ )

\item $U_{\tree}(((4), (2), (0)))=$ the completion of $U(((4), (2)))$ that sends $(2, ((0, 0)))$ to the completion of $U_{\tree}(((4), (2)))(2, \emptyset)$  whose node component is $-1$

\noindent$U_{\node}((4), (2), (0))= (0, -1)$

( (1), $[ \vec{\alpha} \mapsto h$(2, ($g$(2, ($\alpha_{(0)}$, (0))), (0), $g$(2, ($\alpha_{(0)}$, $-1$)), $-1$))$]_{\mu_{\vec{\alpha}}}$, (0), $[ \vec{\alpha} \mapsto h$(2, ($g$(2, ($\alpha_{(0)}$, (0))), (0), $g$(2, ($\alpha_{(0, 0)}$, (0, 0))), (0)))$]_{\mu_{\vec{\alpha}}}$, $-1$ )

\item $U_{\tree}(((4), (1)))=$ the completion of $U(((4)))$ that sends $(2, ((0)))$ to the completion of $U_{\tree}(((4)))(2, \emptyset)$  whose node component is $-1$

\noindent$U_{\node}((4), (1))= (0, -1)$

( (1), $[ \vec{\alpha} \mapsto h$(2, ($g$(2, ($\alpha_{(0)}$, (0))), (0), $g$(2, ($\alpha_{(0)}$, $-1$)), $-1$))$]_{\mu_{\vec{\alpha}}}$, (0), $[ \vec{\alpha} \mapsto h$(2, ($g$(2, ($\alpha_{(0)}$, (0))), (0), $g$(2, ($\alpha_{(0, 0)}$, $-1$)), (0)))$]_{\mu_{\vec{\alpha}}}$, $-1$ )

\item $U_{\tree}(((4), (0)))=$ the completion of $U(((4)))$ that sends $(2, ((0)))$ to the completion of $U_{\tree}(((4)))(2, \emptyset)$  whose node component is $-1$

\noindent$U_{\node}((4), (0))= (1, (0))$

( (1), $[ \vec{\alpha} \mapsto h$(2, ($g$(2, ($\alpha_{(0)}$, (0))), (0), $g$(2, ($\alpha_{(0)}$, $-1$)), $-1$))$]_{\mu_{\vec{\alpha}}}$, $-1$ )

\item $U_{\tree}(((4), (0), (0)))=$ the unique completion of $U(((4), (0)))$

\noindent$U_{\node}((4), (0), (0))= (0, -1)$

( (1), $[ \vec{\alpha} \mapsto h$(2, ($g$(2, ($\alpha_{(0)}$, (0))), (0), $g$(1, ((0))), (0)))$]_{\mu_{\vec{\alpha}}}$, $-1$ )

\item $U(((3)))$  has degree 2

( (1), $[ \vec{\alpha} \mapsto h$(2, ($g$(2, ($\alpha_{(0)}$, $-1$)), (0)))$]_{\mu_{\vec{\alpha}}}$, (0) )

\item $U_{\tree}(((3), (1)))=$ the completion of $U(((3)))$ that sends $(2, ((0)))$ to the completion of $U_{\tree}(((3)))(2, \emptyset)$  whose node component is $(0, 0)$

\noindent$U_{\node}((3), (1))= (0, -1)$

( (1), $[ \vec{\alpha} \mapsto h$(2, ($g$(2, ($\alpha_{(0)}$, $-1$)), (0)))$]_{\mu_{\vec{\alpha}}}$, (0), $[ \vec{\alpha} \mapsto h$(2, ($g$(2, ($\alpha_{(0)}$, $-1$)), (0), $g$(2, ($\alpha_{(0, 0)}$, (0))), (0)))$]_{\mu_{\vec{\alpha}}}$, $-1$ )

\item $U_{\tree}(((3), (0)))=$ the completion of $U(((3)))$ that sends $(2, ((0)))$ to the completion of $U_{\tree}(((3)))(2, \emptyset)$  whose node component is $-1$

\noindent$U_{\node}((3), (0))= (0, -1)$

( (1), $[ \vec{\alpha} \mapsto h$(2, ($g$(2, ($\alpha_{(0)}$, $-1$)), (0)))$]_{\mu_{\vec{\alpha}}}$, (0), $[ \vec{\alpha} \mapsto h$(2, ($g$(2, ($\alpha_{(0)}$, $-1$)), (0), $g$(2, ($\alpha_{(0, 0)}$, (0))), (0)))$]_{\mu_{\vec{\alpha}}}$, $-1$ )

\item $U(((2)))$  has degree 0

( (1), $[ \vec{\alpha} \mapsto h$(2, ($g$(2, ($\alpha_{(0)}$, $-1$)), (0)))$]_{\mu_{\vec{\alpha}}}$, (0), $[ \vec{\alpha} \mapsto h$(2, ($g$(2, ($\alpha_{(0)}$, $-1$)), (0), $g$(2, ($\alpha_{(0, 0)}$, $-1$)), (0)))$]_{\mu_{\vec{\alpha}}}$, $-1$ )

\item $U(((1)))$  has degree 1

( (1), $[ \vec{\alpha} \mapsto h$(2, ($g$(2, ($\alpha_{(0)}$, $-1$)), (0)))$]_{\mu_{\vec{\alpha}}}$, $-1$ )

\item $U_{\tree}(((1), (0)))=$ the unique completion of $U(((1)))$

\noindent$U_{\node}((1), (0))= (0, -1)$

( (1), $[ \vec{\alpha} \mapsto h$(2, ($g$(2, ($\alpha_{(0)}$, $-1$)), (0), $g$(1, ((0))), (0)))$]_{\mu_{\vec{\alpha}}}$, $-1$ )

\item $U(((0)))$  has degree 0

( (0) )

\end{etaremune}

This algorithm of producing $U$ can be abstracted into a combinatorial one, without referring to clubs in $\omega_1$ at all.

Applying universality of $U$ to $(R,\theta_{RY},\vec{\xi}, \rho)$ results in the level-3 factoring map $\psi$ of $R$ into $U$:
\begin{itemize}
\item $\psi(((0))) = ((4))$,
\item $\psi(((0),(1))) = ((4),(14))$,
\item $\psi(((0),(0))) = ((4),(0))$,
\item $\psi(((0),(0),(0))) = ((4),(0),(0))$.
\end{itemize}

To summarize, we have given an example of a general fact on the level-(3, $\leq 2$, $\leq 2$) factoring maps. Given finite level-3 trees $R,Y$, there is a level $\leq 2$ tree $T$, a tuple of ordinals $\vec{\xi}$ respecting $T$ so that the $R$-equivalence class of $\theta_{RY}$ has a purely combinatorial characterization, called a map $\rho$ which factors $(R,Y,T)$, and $R$ embeds into $U = Y \otimes T$ via a level-3 factoring map $\psi$. In terms of ultrapowers, $\rho$ induces $\rho^{Y,T}: \mathbb{L}_{j^R(\bolddelta{3})}[j^R(T_3)] \to \mathbb{L}_{j^Y(\bolddelta{3})}[j^Y \circ j^T(T_3)]$ such that $\rho^{Y,T} \circ j^R = j^Y \circ j^T = j^{U \oplus T}$, and $\rho^{Y,T}$ is just the embedding induced by the measure projection of $\mu^U$ to $\mu^R$, which is generated by $\psi$. 

\section{More on the level-1 analysis}
\label{sec:level-1-analysis}

In this section, we assume $\boldpi{1}$-determinacy. Recall some useful lemmas in \cite{sharpII} on the continuity and uniform cofinality of functions in level-1 ultrapowers. 

\begin{mylemma}\label{lem:j_sigma_continuity}
Suppose $\sigma : \se{1,\ldots,n} \to \se{1,\ldots,n'}$ is order preserving, $\beta< u_{n+1}$. Put $\sigma(0) = 0$. Then
  $j^{\sigma}(\beta) \neq j^{\sigma}_{\sup}(\beta)$ iff for some $k$, $\cf^{\mathbb{L}}(\beta) = u_k$ and $\sigma(k)>\sigma(k-1)+1$.  If $\cf^{\mathbb{L}}({\beta}) =u_{k}$ and $\sigma(k)>\sigma(k-1)+1$, then $j^{\sigma}_{\sup}(\beta) = j ^{\sigma_k} \circ j^{\tau_k}_{\sup}(\beta)$, where $\sigma=\sigma_k \circ \tau_k$, $\sigma_k(i) = \sigma(i)$ for $1 \leq i < k$, $\sigma_k(k) = \sigma(k-1)+1$, $\sigma_k(i) = \sigma(i-1)$ for $k < i \leq n+1$. 
\end{mylemma}

\begin{mylemma}
  \label{lem:level_2_uniform_cofinality}
  Suppose $(P^{-}, p)$ is a partial level $\leq 1$ tree  whose completion is $P$. $\sigma, \sigma'$ both factor $(P,W)$. $\sigma $ and $\sigma'$ agree on $P^{-}$, $\sigma'(p)$ is the $\prec^W$-predecessor of $\sigma(p)$. Then for any $\beta < j^{P^{-}}(\omega_1)$ such that $\cf^{\mathbb{L}}(\beta) = \seed^{P^{-}}_{p^{-}}$, 
\begin{displaymath}
  \sigma^W \circ j^{P^{-},P}_{\sup}(\beta) =   ({\sigma}')^W_{\sup} \circ j^{P^{-},P} (\beta).
\end{displaymath}
\end{mylemma} 

\begin{mylemma}
  \label{lem:level_2_uniform_cofinality_another}
  Suppose $(P, p)$ is a partial level $\leq 1$ tree, $\sigma$ factors $(P,W)$.  Suppose $\beta < j^{P^{-}}(\omega_1)$ and  either
  \begin{enumerate}
  \item $p=-1$, $P^{+} = P$, $\sigma'=\sigma$,  $\cf^{\mathbb{L}}(\beta) = \omega$, or 
  \item $p \neq -1$, $P^{+}$ is the completion of $(P,p)$, $\sigma'$ factors $(P^{+},W)$, $\sigma = \sigma' \res P$,  $\sigma'(p)$ is the $\prec^W$-predecessor of $\sigma(p^{-})$,  $\cf^{\mathbb{L}}(\beta) = \seed^{P}_{p^{-}}$.
  \end{enumerate}
Then
\begin{displaymath}
  \sigma^W (\beta) =   ({\sigma}')^W_{\sup} \circ j^{P,P^{+}} (\beta).
\end{displaymath}
\end{mylemma}

Suppose $P,W$ are level-1 trees, $\sigma$ factors $(P,W)$. Given a tuple $\vec{\alpha} =(\alpha_w)_{w \in W} \in [\omega_1]^{W \uparrow}$,  define
\begin{displaymath}
  \vec{\alpha}_{\sigma} = (\alpha_{\sigma,p})_{p \in P } \in [\omega_1]^{P \uparrow}
\end{displaymath}
where  $\alpha_{\sigma,p} = \alpha_{\sigma(p)}$. If $W$ is finite, put $\seed_{\sigma} ^W =   (\seed^W_{\sigma(p)})_{p \in P}$, i.e., $\seed_{\sigma}^W$ is represented modulo $\mu^W$ by the function  $\vec{\alpha}\mapsto \vec{\alpha}_{\sigma}$.

If $P,W$ are finite, $\sigma$ factors $(P,W)$,  then for any $A \in \mu^P$, for $\mu^W$-a.e.\ $\vec{\alpha}$, $\vec{\alpha}_{\sigma} \in A$. Thus, for any $A \in \mu^P$, 
$ \seed_{\sigma}^W \in j^W(A)$. 
Thus, we can unambiguously define
  \begin{displaymath}
    \sigma^{W} : \mathbb{L} \to \mathbb{L}
  \end{displaymath}
by $\sigma^{W} ( j^P(h)(\seed^P)) = j^W (h)(  \seed^{W}_{\sigma} )$. $\sigma^{W}$ is the unique map such that for any $z \in \mathbb{R}$, $\sigma^{W}$ is elementary from $L[z]$ to $L[z]$, $\sigma^W \circ j^P = j^W$, and
for any $ p \in P$, $ \sigma^{W}\circ p^P = \sigma(p)^W$. Define $\sigma^W_{\sup} (\beta) = \sup (\sigma^W)''\beta$. 

Suppose  the signature of $\beta$ is $(u_{l_i})_{i < m}$, the approximation sequence of $\beta$ is $( \gamma_i)_{i \leq  m} $. For $i \leq m$, let $\tau_{i,m} : \se{ 1, \dots,i+1} \to \se{l_0,\dots,l_i}$ be order preserving. For $i < k < m$, let  $\tau_{i,k} = \tau_{k,m}^{-1} \circ \tau_{i,m}$. A straightforward analysis on the representative functions of $\beta$ yields the following:
\begin{enumerate}
\item For $i < k  < m$, $j^{\tau_{i,k}}_{\sup} ( \gamma_i ) < \gamma_k < j^{\tau_{i,k}}(\gamma_i)$. 
\item For $i < m$, $j^{\tau_{i,m}} _{\sup}(\gamma_i) \leq  \gamma_m < j^{\tau_{i,m}}(\gamma_i)$.
\item For $i < m$, $j^{\tau_i,m}_{\sup}(\gamma_i) = \gamma_m$ iff  $i=m-1$ and $\beta$ is essentially continuous.
\item $\beta = j^{\tau_{m,m}}(\gamma_{m})$. 
\end{enumerate}
 The next lemma is  a version of the ``converse direction''. In its statement, the inequality $j^{\pi}_{\sup}(\gamma) < \gamma' < j^{\pi}(\gamma)$ forces $\pi$ to move  the signature of $\gamma$ to a proper initial segment of that of $\gamma'$, and forces the approximation sequence of $\gamma$ to be a proper initial segment of that of $\gamma'$.  
 The proof is again based on an analysis of the representative function of $\gamma$ and $\gamma'$.

\begin{mylemma}
  \label{lem:ordinal_division_blocks}
Suppose $A$ is a finite subset of $\omega$. 
Let $\pi : \se{1,\ldots,\card(A)} \to A$ be order preserving. 
Suppose that $ \gamma < u_{\card(A)+1}$ and $j^{\pi}_{\sup} (\gamma) < \gamma' < j^{\pi}(\gamma)$. 
Let $(u_{l_k})_{k<v}, (\gamma_k)_{k\leq v}$, $(P,\vec{p})$ be the signature, approximation sequence and potential partial level $\leq 1$ tower induced by $\gamma$ respectively. Let  $(u_{l'_k})_{k<v'}, (\gamma'_k)_{k\leq v'}$, $(P',\vec{p}')$ be the signature, approximation sequence and potential partial level $\leq 1$ tower induced by $\gamma'$ respectively.  Let $\cf^{\mathbb{L}}(\gamma) = u_{l_{*}}$. 
Then
\begin{enumerate}
\item $v < v'$, $(\pi(l_k), \gamma_k) = (l_k',\gamma_k')$. $\gamma$ is essentially discontinuous $\to \gamma_v = \gamma_{v'}$.  $\gamma$ is essentially continuous$\to\gamma_v < \gamma_{v'}$. 
\item $l'_{k} \notin  A$ for $v\leq k < v'$.
\item For any $k <v$,  $l'_{v} < \pi(l_{k}) \eqiv l_{*} \leq  l_{k}$.
\item $P$ is a proper subtree of $P'$ and $\vec{p}$ is an initial segment of $\vec{p}'$. 
\end{enumerate}
Moreover, if $ \gamma' < \gamma'' < j^{\pi}(\gamma)$ and $(\gamma_k'')_{k \leq  v''}$ is the approximation sequence of $\gamma''$, then $\gamma_{v}'<\gamma_{v}''$.
\end{mylemma}

\section{More on the level-2 analysis}
\label{sec:level-2-analysis}

We recall a useful corollary to Silver's dichotomy on $\boldpi{3}$-equivalence relations in \cite{sharpII}. 

Recall the lemmas in \cite{sharpII} on the definability of respecting a level $\leq 2$ tree. 

\begin{mylemma}\label{lem:respect_lv2_measure_one_set}
  Suppose $Q$ is a finite level $\leq 2$ tree,  $C \in \mu_{\mathbb{L}}$ is a club. Then $\vec{\beta} \in [C]^{Q \uparrow}$ iff there exist $f \in \omega_1^{Q \uparrow}$ and $E \in \mu_{\mathbb{L}}$ such that $\vec{\beta} = [f]^Q$ and  for any $q \in \comp{1}{Q}$, $\comp{1}{f}(q)$ is a limit point of $C$;
 for any $q \in \dom(\comp{2}{Q})$, for any $\vec{\alpha} \in [E]^{\comp{2}{Q}_{\tree}(q)\uparrow}$, $\comp{2}{f}_q(\vec{\alpha})$ is a limit point of $C$.
\end{mylemma}
 In the proof of Lemma \ref{lem:respect_lv2_measure_one_set}, the following claim is used. Suppose $f \in \omega_1^{Q \uparrow}$ and  $E \in \mu_{\mathbb{L}}$ are as given. For $q \in \dom(\comp{2}{Q})\setminus\se{\emptyset}$,  let $\comp{2}{Q}(q)= (P_q,p_q)$, and let $q^{*}$ be the $<_{BK}$-maximum of $\comp{2}{Q}\{q,-\}$.  
\begin{myclaim}
  \label{claim:respect_lv2_distance}
  There is $E' \in \mu_{\mathbb{L}}$ such that $E' \subseteq E$ and for any $q \in \dom(\comp{2}{Q})\setminus\se{\emptyset}$, 
for any $\vec{\alpha} \in [E']^{P_q\uparrow}$, if $p_q \neq -1$ then 
  $C \cap (\comp{2}{f}_{q^{*}}(\vec{\alpha}), \comp{2}{f}_q (\vec{\alpha}))$ has order type $\alpha_{p_q^{-}}$.
\end{myclaim}

\begin{mylemma}
  \label{lem:level_2_Q_respecting_function_sort}
  Suppose $Q$ is a finite level $\leq 2$ tree, $\comp{2}{Q}(q) = (P_q, p_q)$ for $q \in \dom(Q)$, $E \in \mu_{\mathbb{L}}$ is a club. Suppose $f : \rep(Q) \res E \to \omega_1+1$ satisfies
  \begin{enumerate}
  \item $f \res( \se{1} \times  \rep(\comp{1}{Q}))$ is continuous, order preserving;
  \item if $q \in \dom(\comp{2}{Q})$, then the potential partial level $\leq 1$ tower induced by $\comp{2}{f}_{q}$ is $\comp{2}{Q}[q]$, the approximation sequence of $\comp{2}{f}_q$ is $(\comp{2}{f}_{q \res i})_{i \leq \lh(q)}$,
and the uniform cofinality of $\comp{2}{f}_q$ on $[E]^{P_q\uparrow}$ is witnessed by $\comp{2}{f}_{q \concat (-1)}$, i.e., if $\vec{\alpha} \in [E]^{P_q \uparrow}$, then $\comp{2}{f}_q(\vec{\alpha}) = \sup\{ \comp{2}{f}_{q \concat (-1)}(\vec{\alpha} \concat (\beta)) : \vec{\alpha} \concat (\beta) \in \rep(\comp{2}{Q}) \res E \}$, and the map $\vec{\beta} \mapsto \comp{2}{f}_{q \concat (-1)}(\vec{\alpha} \concat (\beta))$ is continuous, order preserving;
  \item if $a,b \in \comp{2}{Q}\se{q}$ and $a < _{BK} b$, then $[f_{q \concat (a)}]_{\mu^{P_{q \concat (a)}}} < [f_{q \concat (b)}]_{\mu^{P_{q \concat (b)}}}$. 
  \end{enumerate}
Then there is $E' \in \mu_{\mathbb{L}}$ such that $E' \subseteq E$ and $f \res (\rep(Q) \res E')$ is order preserving.
\end{mylemma}

\begin{mylemma}
  \label{lem:Q_respecting}
  Suppose that $Q$ is a finite level $\leq 2$ tree and $\vec{\beta}= (\comp{d}{\beta}_q)_{(d,q)\in \dom(Q)}$ is a tuple of ordinals in $u_{\omega}$. Then $\vec{\beta}$ respects $Q$ iff all of the following holds:
  \begin{enumerate}
  \item $(\comp{1}{\beta}_q)_{q \in \comp{1}{Q}} $ respects $\comp{1}{Q}$. 
  \item For any $q \in \dom(\comp{2}{Q})$, the potential partial level $\leq 1$ tower induced by $\comp{2}{\beta}_q$ is $Q[q]$, and the  approximation sequence of $\comp{2}{\beta}_q$ is  $(\comp{2}{\beta}_{q \res l})_{  l \leq \lh(q)}$.
    \item If $a,b\in \comp{2}{Q}\se{q}$ and $a<_{BK}b$ then $\comp{2}{\beta}_{q \concat (a)} < \comp{2}{\beta}_{q \concat (b)}$.
  \end{enumerate}
Moreover, if $C \in \mu_{\mathbb{L}}$ is a club, then $\vec{\beta} \in [C]^{Q\uparrow}$ iff $\vec{\beta}$ respects $Q$ and letting $C'$ be the set of limit points of $C$, then $\comp{1}{\beta}_q \in C'$ for $ q \in \comp{1}{Q}$, $\comp{2}{\beta}_q \in j^{\comp{2}{Q}_{\tree}(q)} (C')$ for $q \in \dom(\comp{2}{Q})$.
\end{mylemma}

\begin{mylemma}
  \label{lem:respecting_level_2_tree_is_Delta13}
  The relation ``$Q$ is a finite level $\leq 2$ tree $\wedge$ $\vec{\beta}$ respects $Q$'' is $\Delta^1_3$.
\end{mylemma}

\begin{mylemma}
  \label{lem:unique_level_2_tree_represent}
  Suppose $Q$ and $Q'$ are  level $\leq 2$ trees with the same domain.  Suppose $\vec{\beta}$ respects both $Q$ and $Q'$. Then $Q = Q'$. 
\end{mylemma}

\begin{mylemma}\label{lem:beta_q_unambiguous}
Suppose $Q$ is a level $\leq 2$ tree. 
Suppose $\vec{\beta} = (\comp{d}{\beta}_q)_{(d,q) \in \dom(Q)} \in [\omega_1]^{Q \uparrow}$, $(2,\mathbf{q})=(2,(q,P, \vec{p})) \in \desc({Q})$ is of continuous type, $P^{-} = Q_{\tree}(q^{-})$,  then $\comp{2}{\beta}_{\mathbf{q}} =j^{P^{-},P}_{\sup}(\comp{2}{\beta}_{q^{-}})$.
\end{mylemma}

\begin{mycorollary}
  \label{coro:Delta13_pwo_computable_in_LT2}
      Assume $\boldDelta{2}$-determinacy. Let $x \in \mathbb{R}$. If $\leq^{*}$ is a $\Delta^1_3(x)$ prewellordering on $\mathbb{R}$ and $A $ is a $\Sigma^1_3(x)$ subset of $\mathbb{R}$, then $\sharpcode{\leq^{*}}$ and $\set{ \wocode{y}_{\leq^{*}}}{y \in A}$ are both in $\admistwo{M_1^{\#}(x)}$ and  $\Delta_1$-definable over $\admistwo{M_1^{\#}(x)}$ from parameters in $\se{T_2,M_1^{\#}(x)}$.
\end{mycorollary}


\begin{mylemma}
  \label{lem:Pi13_closed_under_club_measure}
   Assume $\boldDelta{2}$-determinacy. If $P$, $W$ are finite level-1 trees, $A \subseteq [\omega_1]^{W \uparrow}\times \mathbb{R}$ is $\Pi^1_3$ (or $\Sigma^1_3$, $\Delta^1_3$ resp.), then so are
   \begin{align*}
          B &= \set{x }{ \text{for $\mu^W$-a.e.\ }\vec{\alpha}, (\vec{\alpha},x) \in A }, \\
     C&= \set{ (\vec{\beta}, x) }{\vec{\beta} \in j^P(A_x)},
   \end{align*}
where $A_x = \set{\vec{\alpha}}{(\vec{\alpha},x) \in A}$.
\end{mylemma}
\begin{proof}
  $x \in B$ iff $\exists y ~\forall \vec{\alpha}\in [\omega_1]^{W\uparrow} (\vec{\alpha} \text{ are $y$-admissibles}\to (\vec{\alpha},x) \in A)$.   $x \notin B$ iff $\exists y~ \forall \vec{\alpha} \in [\omega_1]^{W \uparrow} (\vec{\alpha} \text{ are $y$-admissibles}\to (\vec{\alpha},x) \notin A)$. The quantifier $\forall \vec{\alpha} \in [\omega_1]^{W \uparrow}$ does not increase the complexity due to Kechris-Martin \cite{Kechris_Martin_II}. The complexity of $C$ follows from that of $B$ and \Los{}. 
\end{proof}
A purely descriptive set theoretic proof of Lemma~\ref{lem:Pi13_closed_under_club_measure} is given in \cite[Lemma 4.40]{jackson_handbook}.

\begin{mylemma}
  \label{lem:Pi13_lift_by_level_1_tree}
   Assume $\boldDelta{2}$-determinacy.  Suppose $Q$ is a finite level $\leq 2$ tree. If $A \subseteq \omega_1 \times \mathbb{R}$ is $\Pi^1_3$ (or $\Sigma^1_3$, $\Delta^1_3$ resp.), then so is 
\begin{displaymath}
     B = \set{ (\vec{\beta}, x) }{  \vec{\beta} \in [A_x]^{{Q}\uparrow}},
   \end{displaymath}
where $A_x = \set{\alpha}{(\alpha,x) \in A}$.
\end{mylemma}
\begin{proof}
%
%
Put $A_x'=$the set of limit points of $A_x$. 
By Lemma~\ref{lem:Q_respecting},  $\vec{\beta} \in [A_x]^{Q\uparrow}$ iff $\vec{\beta}$ respects $Q$ and for any $q \in \comp{1}{Q}$, $\comp{1}{\beta}_q \in A_x'$ , for any $q \in \dom(\comp{2}{Q})$, $\comp{2}{\beta}_q \in j^{Q_{\tree}(q)}(A_x')$. 

 Now apply Lemma~\ref{lem:Pi13_closed_under_club_measure} and Lemma~\ref{lem:respecting_level_2_tree_is_Delta13}.
\end{proof}

\begin{mylemma}
  \label{lem:Pi13_closed_under_level_2_measure}
Assume $\boldDelta{2}$-determinacy. Suppose $Q$ is a finite level $\leq 2$ tree. If $A \subseteq [\omega_1]^{Q \uparrow} \times \mathbb{R}$ is $\Pi^1_3$ (or $\Sigma^1_3$, $\Delta^1_3$ resp.), then so is 
  \begin{displaymath}
     B = \set{  x }{ \text{for $\mu^Q$-a.e.\ }\vec{\beta}, (\vec{\beta} ,x) \in A},
   \end{displaymath}
\end{mylemma}
\begin{proof}
Let $C = \set{(y,\alpha)}{\alpha <\omega_1 \wedge \alpha \text{ is $y$-admissible}}$. $C$ is $\Delta^1_3$. Then
$x \in B$ iff  $\exists y ~ \forall \vec{\beta}~ (\vec{\beta} \in [C_{y}]^{Q \uparrow} \to (\vec{\beta},x) \in A)$. $x \notin B$ iff  $\exists y ~ \forall \vec{\beta}~ (\vec{\beta} \in [C_{y}]^{Q \uparrow} \to (\vec{\beta},x) \notin A)$. 
Use  Lemma~\ref{lem:Pi13_lift_by_level_1_tree}.
\end{proof}

\begin{mycorollary}
  \label{coro:delta13_bounds_ultrapowers}
  Assume $\boldDelta{2}$-determinacy. Suppose $Q$ is a finite level $\leq 2$ tree. Then $j^Q(\alpha)< \bolddelta{3}$ for any $\alpha < \bolddelta{3}$. $j^Q(T_2) \in \admistwobold$.
\end{mycorollary}
\begin{proof}
  By Lemma~\ref{lem:Pi13_closed_under_level_2_measure}, for any $\alpha< \bolddelta{3}$, $j^Q(\alpha)$ is the length of a $\boldDelta{3}$ prewellordering on $\mathbb{R}$. $j^Q(T_2) \in \admistwo{M_1^{\#}}$ by Corollary~\ref{coro:Delta13_pwo_computable_in_LT2}.
\end{proof}

Corollary~\ref{coro:delta13_bounds_ultrapowers} is the effective version of \cite[Corollary 3.9]{sol_delta13coding}. Actually, $j^Q(T_2)$ is $\Delta^1_3$ in the sharp codes, a fact to be shown in Section~\ref{sec:approximation-LT3-LT2}.

For $Q$  a finite level $\leq 2$ tree, 
by Corollary~\ref{coro:delta13_bounds_ultrapowers}, $\admistwoboldextra{j^Q} = \ult(\admistwobold,\mu^Q)$.


Suppose $Q$ is a finite level $\leq 2$ tree. Suppose $(d,\mathbf{q})\in \exdesc(Q)$, and if $d=2$ then $\mathbf{q} = (q,P, \vec{p})$. Put
\begin{displaymath}
  \llbracket(d,\mathbf{q}) \rrbracket _Q =
  \begin{cases}
    \wocode{(1, (\mathbf{q}))}_{<^Q} & \text{ if } d=1, \\
    [\vec{\alpha} \mapsto \wocode{ (2, \vec{\alpha} \oplus_{\comp{2}{Q}} q)}_{<^Q}]_{\mu^P} & \text{ if }d=2.
  \end{cases}
\end{displaymath}
To save ink, put $\llbracket d, \mathbf{q} \rrbracket_Q  = \llbracket (d,\mathbf{q}) \rrbracket_Q$. 
If in addition, $d=2$ and $\mathbf{q}\in \desc(Q)$ of discontinuous type, put  $\llbracket 2,q \rrbracket_Q = \llbracket 2, \mathbf{q} \rrbracket_Q$. 
It is easy to compute $\llbracket d,\mathbf{q} \rrbracket_Q$ only from the syntactics.
\begin{mydefinition}\label{def:ordinal_assignment}
  To every ordinal $ \xi < \omega^{\omega^{\omega}}$ (ordinal arithmetic), we assign $\widehat{\xi}$ as follows:
  \begin{enumerate}
     \item $\widehat{0} = 0$.
  \item $\widehat{1} = \omega$.
  \item If $0 < \eta = \omega^{n_1}+ \dots + \omega^{n_k} < \omega^{\omega}$, $\omega>n_1 \geq \dots \geq n_k $ in the Cantor normal form, then $\widehat{\omega^{\eta}} = u_{n_1+1} \cdot \dots \cdot u_{n_k+1}$.
  \item If $0 < \xi = \omega^{\eta_1}+ \dots + \omega^{\eta_k}$, $\omega^{\omega} > \eta_1\geq \dots \geq \eta_k$ in the Cantor normal form, then $\widehat{\xi} = \widehat{\omega^{\eta_1}}+\dots+ \widehat{\omega^{\eta_k}}$.
  \end{enumerate}
\end{mydefinition}

Then
\begin{displaymath}
    \set{\widehat{\xi}}{0 < \xi < \omega^{\omega^{\omega}}} = \{   \llbracket d,\mathbf{q} \rrbracket_Q :~  Q \text{ finite level  $\leq 2$ tree} ,
(d,\mathbf{q}) \in \exdesc(Q) \}
\end{displaymath}
and the relation $\widehat{\xi} =  \llbracket d,\mathbf{q} \rrbracket_Q $ is effective. The ordering among different $\llbracket d,\mathbf{q} \rrbracket_Q$ can be computed in the following concrete way. Put $\corner{1,\mathbf{q}} = (1,\mathbf{q})$. For $\mathbf{q} = (q,P,\vec{p})$, $k=\lh(q)$, $\vec{p}=(p_i)_{i < \lh(\vec{p})}$, put 
\begin{displaymath}
  \corner{ 2,\mathbf{q} }=
  \begin{cases}
     (2,  \wocode{ p_0 }_{\prec^P}, q(0), \dots, \wocode{ p_{k-2}}_{\prec^P}, q(k-2), -1) \\
  \qquad\text{ if }\mathbf{q} \in \desc(\comp{2}{Q}) \text{ of continuous type},\\
    (2,  \wocode{ p_0} _{\prec^P}, q(0), \dots, \wocode{p_{k-1}}_{\prec^P}, q(k-1), -1) \\
 \qquad\text{ if }\mathbf{q} \in \desc(\comp{2}{Q}) \text{ of discontinuous type},\\
    (2,  \wocode{ p_0} _{\prec^P}, q(0), \dots, \wocode{p_{k-1}}_{\prec^P}, q(k-1), \wocode{p_k}_{\prec^P}) \\
 \qquad\text{ if } \mathbf{q} \notin \desc(\comp{2}{Q}) .
  \end{cases}
\end{displaymath}
Define 
\begin{displaymath}
  (d,\mathbf{q}) \prec (d',\mathbf{q}')
\end{displaymath}
iff $\corner{d,\mathbf{q}} <_{BK} \corner{d',\mathbf{q}'} $. Define 
\begin{displaymath}
  (d,\mathbf{q}) \sim (d',\mathbf{q}')
\end{displaymath}
iff $\corner{d,\mathbf{q}}  = \corner{d',\mathbf{q}'} $.
Then for any  finite level $\leq 2$ tree $Q$, $ \llbracket d,\mathbf{q} \rrbracket_Q <  \llbracket d',\mathbf{q}' \rrbracket_Q $ iff $  (d,\mathbf{q}) \prec (d',\mathbf{q}')$; $ \llbracket d,\mathbf{q} \rrbracket_Q =  \llbracket d',\mathbf{q}' \rrbracket_Q $ iff $  (d,\mathbf{q}) \sim (d',\mathbf{q}')$. In fact, $(d,\mathbf{q}) \sim (d',\mathbf{q}')$ iff either $(d,\mathbf{q}) = (d',\mathbf{q}')$ or 
 $\se{(d,\mathbf{q}), (d',\mathbf{q}')} = \se{(2,(\emptyset,\emptyset,(0))), (2, ((-1), \se{(0)}, ((0))))}$.

Define $\prec^Q = \prec \res \exdesc(Q)$, $\sim^Q = \sim \res \exdesc(Q)$. 
It is also easy to verify the next lemma on the order of the entries of $\vec{\beta}$ which respects $Q$.
\begin{mylemma}
  \label{lem:level_2_beta_q_order}
  Suppose $Q$ is a level $\leq 2$ tree and $\vec{\beta}$ respects $Q$. Suppose   $(d,\mathbf{q}), (d',\mathbf{q}')\in \exdesc(Q)$. 
  Then    $\comp{d}{\beta}_{\mathbf{q}} < \comp{d'}{\beta}_{\mathbf{q}'}$ iff $(d,\mathbf{q}) \prec^Q (d,\mathbf{q}')$;  $\comp{d}{\beta}_{\mathbf{q}} = \comp{d'}{\beta}_{\mathbf{q}'}$ iff $(d,\mathbf{q}) \sim^Q (d,\mathbf{q}')$.
\end{mylemma}

For level $\leq 2$ trees $Q,X$, we say that $\pi: \dom(X) \to \dom(Q)$ \emph{factors} $(X,Q)$ iff putting $(d, \comp{d}{\pi}(x)) = \pi(d,x)$ for $(d,x) \in \dom(X)$, 
\begin{enumerate}
\item $\comp{1}{\pi}$ factors $(\comp{1}{X},\comp{1}{Q})$;
\item if $x \in \dom(\comp{2}{X})$ then $\comp{2}{X}(x) = \comp{2}{Q}(\comp{2}{\pi}(x))$;
\item if $x ,x' \in \dom(\comp{2}{X})$, then $x <_{BK} x' \to \comp{2}{\pi}(x)<_{BK} \comp{2}{\pi}(x')$, $x \subseteq x' \to \comp{2}{\pi}(x) \subseteq \comp{2}{\pi}(x')$. 
\end{enumerate}
For $d \in \se{1,2}$, $\comp{d}{\pi}$ has this fixed meaning if $\pi$ factors $(Q,X)$. Extend the definition of $\comp{2}{\pi}$ on $\bardom(\comp{2}{Q})$ and $\desc(\comp{2}{Q})$ is the following natural way: if $q  \concat (-1)\in \bardom(\comp{2}{Q})$, define $\comp{2}{\pi}(q \concat (-1)) = \pi(q) \concat (-1)$;
if $ \mathbf{q} = (q,P,\vec{p}) \in \desc(\comp{2}{Q})$, define $\comp{2}{\pi} (\mathbf{q}) = (\comp{2}{\pi}(q), P,\vec{p})$. If $\vec{\beta} = (\comp{d}{\beta}_q)_{(d,q) \in \dom(Q)} \in [\omega_1]^{Q \uparrow}$, put $\vec{\beta}_{\pi} = (\comp{d}{\beta}_{\pi,x})_{(d,x) \in \dom(X)} \in [\omega_1]^{X \uparrow}$ where $\comp{d}{\beta}_{\pi,x} = \comp{d}{\beta}_{\comp{d}{\pi}(x)}$.


\subsection{Level-2 description analysis}
\label{sec:level-2-description}

If $Q$ is a level-2 tree, $\mathbf{q}=(q,P, \vec{p} ) \in \desc({Q})$, $\lh(q)=k$, $\vec{p}=(p_i)_{i < \lh(\vec{p})}$, 
$\sigma$ is a function whose domain contains $P$, we put
\begin{displaymath}
 \sigma \oplus \mathbf{q} =  \sigma \oplus_Q q = (\sigma(p_0), q(0), \ldots,\sigma(p_{k-1}), q(k-1)).
\end{displaymath}

\begin{mydefinition}
  \label{def:level-2_W_Q_description}
  Suppose $W$ is a finite level-1 tree  and suppose $Q$ is a level $\leq 2$ tree. A \emph{$(Q,W)$-description} is of the form
  \begin{displaymath}
    \mathbf{D} = (d,(\mathbf{q},\sigma)) 
  \end{displaymath}
such that either
\begin{enumerate}
\item $d= 1$, $\mathbf{q} \in\comp{1}{Q}$, $\sigma=\emptyset$, or
\item $d=2$, $\mathbf{q} =(q,P, \vec{p})\in \desc(\comp{2}{Q})$, $\sigma$ factors $(P,W)$. 
\end{enumerate}
 $\desc(Q,W)$ is the set of $(Q,W)$-descriptions.   A $(Q,*)$-description is a $(Q,W')$-description for some finite level-1 tree $W'$. $\desc(Q,*)$ is the set of $(Q,*)$-descriptions.  We sometimes abbreviate $(d, \mathbf{q}, \sigma) $ for  $(d, (\mathbf{q}, \sigma)) \in \desc(Q,W)$ without confusion. 

Suppose now $\mathbf{D}=(d,\mathbf{q},\sigma)$ and if $d=2$, then $\mathbf{q} = (q,P, \vec{p})$, $\vec{p} = (p_i)_{i < \lh(\vec{p})}$, $\lh(q) = k$. 
 The \emph{degree} of $\mathbf{D}$ is $d$. 
The \emph{signature} of $\mathbf{D}$ is 
\begin{displaymath}
  \sign(\mathbf{D})=
  \begin{cases}
    \emptyset & \text{ if }d = 1,\\
    (\sigma(p_i))_{i < k} & \text{ if } d=2.
  \end{cases}
\end{displaymath}
$\mathbf{D}$ is \emph{of continuous type} iff $d=2$ and $\mathbf{q}$ is of continuous type; otherwise, $\mathbf{D} $ is \emph{of discontinuous type}.
The \emph{uniform cofinality} of $\mathbf{D}$ is
\begin{displaymath}
  \ucf(\mathbf{D}) =
  \begin{cases}
    -1 &  \text{if } d=1 \vee (d=2 \wedge \ucf(P,\vec{p}) = -1),\\
   \sigma(\ucf(P,\vec{p})) & \text{if } d=2 \wedge \ucf(P,\vec{p}) \neq -1.
  \end{cases}
\end{displaymath}
The \emph{$*$-signature} of $\mathbf{D}$ is 
\begin{displaymath}
  \sign_{*}(\mathbf{D})=
  \begin{cases}
    ((1, \mathbf{q})) & \text{ if }d = 1,\\
    ((2, q \res i))_{1 \leq i \leq k-1} & \text{ if } d=2, q \text{ of continuous type},\\
    ((2, q \res i))_{1 \leq i \leq k} & \text{ if } d=2, q \text{ of discontinuous type}.
  \end{cases}
\end{displaymath}
 $\mathbf{D}$ is \emph{of  $*$-$W$-continuous type} iff either $d=1$ or ($d=2$ and ($\ucf(P,\vec{p}) \neq -1 \wedge \sigma(\ucf(P,\vec{p})) \neq \min(\prec^W)$) implies $ \pred_{\prec^W}(\sigma(\ucf(P,\vec{p}))) \in \ran(\sigma)$).
Otherwise, $\mathbf{D}$ is \emph{of $*$-$W$-discontinuous type}. 
The \emph{$*$-$W$-uniform cofinality} of $\mathbf{D}$ is
\begin{displaymath}
  \ucf_{*}^W(\mathbf{D})
\end{displaymath}
defined as follows. If $d=1$, then $\ucf_{*}^W (\mathbf{D}) = (1, \mathbf{q})$.
 If $d=2$, $q$ is of continuous type, 
    \begin{enumerate}
    \item if $\mathbf{D}$ is of $*$-$W$-continuous type, then $\ucf_{*}^W(\mathbf{D}) = (2,( q^{-}, P\setminus \se{p_{k-1}}, \vec{p}))$;
    \item if $\mathbf{D}$ is of $*$-$W$-discontinuous type, then $\ucf_{*}^W(\mathbf{D}) = (2, (q^{-}, P, \vec{p}))$.
    \end{enumerate}
If $d=2$,  $q$ is of discontinuous type,
    \begin{enumerate}
    \item if $\mathbf{D}$ is of $*$-$W$-continuous type, then $\ucf_{*}^W(\mathbf{D}) = (2,\mathbf{q})$;
    \item if $\mathbf{D}$ is of $*$-$W$-discontinuous type, then $\ucf_{*}^W(\mathbf{D}) = (2,(q, P\cup \se{p_k}, \vec{p})$.
    \end{enumerate}

 The \emph{constant $(Q,*)$-description} is $(2,(\emptyset,\emptyset, ((0))), \sigma_0)$ where $\sigma_0$ is the unique that factors $(\emptyset,*)$, i.e.,  $\sigma_0(\emptyset) = \emptyset$. 
\end{mydefinition}

Note that if $\mathbf{D} \in \desc(Q,W)$ and $W$ is a subtree of $W'$, then $\mathbf{D} \in \desc(Q,W')$, but $\ucf_{*}^W(\mathbf{D})$ could be different from $\ucf_{*}^{W'}(\mathbf{D})$.
If $Q$ is finite, there are in total
\begin{displaymath}
\card(\comp{1}{Q} )  + \sum_{q\in \dom(\comp{2}{Q})} {\card(W) \choose \lh(q)} +  \sum_{\comp{2}{Q}(q) \text{ of degree 1}}{\card(W) \choose \lh(q)+1} 
\end{displaymath}
many $(Q,W)$-descriptions.   We shall establish an exact correspondence between  $\desc(Q,W)$ and uniform indiscernibles $\leq j^Q \circ j^W (\omega_1)$.

Suppose $\mathbf{D} = (d,\mathbf{q},\sigma) \in \desc(Q,W)$, and if $d=2$, then  $\mathbf{q}=(q,P,\vec{p})$, $\vec{p}=(p_i)_{i < \lh(\vec{p})}$, $\lh(q) = k$. For $g \in\omega_1^{Q \uparrow}$, let
\begin{displaymath}
  g_{\mathbf{D}}^W : [\omega_1]^{W \uparrow} \to \omega_1+1
\end{displaymath}
be the function as follows:  if $d=1$, then $g^W_{\mathbf{D}}(\vec{\alpha}) = {}^1[g]^{{Q}}_{{\mathbf{q}}}$ when $\min(\vec{\alpha}) > {}^1[g]^{{Q}}_{{\mathbf{q}}}$, $g^W_{\mathbf{D}} (\vec{\alpha}) = \wocode{ (1, (q)) }_{<^Q} $ otherwise\footnote{the split in definition is insignificant, only to ensure Lemma~\ref{lem:level_2_desc_order}.}; if $d=2$, then $g^W_{\mathbf{D}} (\vec{\alpha}) = \comp{2}{g}_q (\vec{\alpha}_{\sigma} )$ (Recall the definition of $\vec{\alpha}_{\sigma}$ in Section~\ref{sec:level-1-analysis}).  In particular, 
if $\mathbf{D}$ is the constant $(Q,*)$-description, then $g_{\mathbf{D}}^W$ is the constant function with value $\omega_1$. Clearly, 
the signature of $g^W_{\mathbf{D}}$ is  $\sign(\mathbf{D})$; $\mathbf{D}$ is of continuous type iff $g^W_{\mathbf{D}}$ is essentially continuous; the  uniform cofinality of $g_{\mathbf{D}}^W$ is $\ucf(\mathbf{D})$. Suppose  additionally that $Q$ is finite. Let
\begin{displaymath}
  \id_{\mathbf{D}}^{Q,W}
\end{displaymath}
be the function  $[g]^Q\mapsto [g_{\mathbf{D}}^W]_{\mu^W}$, or equivalently,  $\vec{\beta}\mapsto \sigma^{W} (  \comp{d}{\beta}_{\mathbf{q}})$, where $\emptyset^W$ is interpreted as $j^W$. Clearly, the signature of $\id^{Q,W}_{\mathbf{D}}$ is $\sign_{*}^W(\mathbf{D})$; $\id^{Q,W}_{\mathbf{D}}$ is essentially continuous iff $\mathbf{D}$ is of $*$-$W$-continuous type; the uniform cofinality of $\id^{Q,W}_{\mathbf{D}}$ is $\ucf_{*}^W(\mathbf{D})$.
Let
\begin{displaymath}
  \seed_{\mathbf{D}}^{Q,W} \in \admistwoboldultratwo{Q}{W}
\end{displaymath}
be the element represented modulo $\mu^Q$ by $\id^{Q,W}_{\mathbf{D}}$. In particular, if $d=1$ then $\seed^{Q,W}_{\mathbf{D}} = \seed^Q_{(1,\mathbf{q})}$; if $d=2$, $P = W$ and $\sigma =\id_P$, then $\seed^{Q,W}_{\mathbf{D}} = \seed^Q_{(2, \mathbf{q})}$.
By \Los{}, if $\mathbf{D}$ is not the constant $(Q,*)$-description, for any $A \in \mu_{\mathbb{L}}$, $\seed_{\mathbf{D}}^{Q,W} \in j^Q\circ j^W(A)$. Thus, we can define
\begin{displaymath}
  \mathbf{D}^{Q,W}:  \admistwoboldextra{j_{\mu_{\mathbb{L}}}} \to \admistwoboldultratwo{Q}{W} 
\end{displaymath}
by sending  $j_{\mu_{\mathbb{L}}}(h)(\omega_1)$ to $j^Q\circ j^{W} (h) (\seed_{\mathbf{D}}^{Q,W})$. 

If $Q$ is a level-2 tree, $\mathbf{q} = (q,P,\vec{p}) \in \desc(Q)$, $l \leq \lh(q)$, define
\begin{displaymath}
  \mathbf{q} \res l = (q \res l,  \set{p_i}{i < l}, (p_i)_{i \leq l}).
\end{displaymath}
which is a $Q$-description. 
If $\mathbf{D} = (2, \mathbf{q}, \sigma) \in \desc(Q,*)$, $\mathbf{q} = (q, P, \vec{p})$, $l \leq \lh(q)$, define
\begin{displaymath}
  \lh(\mathbf{D}) = \lh (\mathbf{q})
\end{displaymath}
and
\begin{displaymath}
  \mathbf{D} \res l = (2, \mathbf{q} \res l ,  \sigma \res \set{p_i}{ i < l})
\end{displaymath}
which is a $(Q,*)$-description. 
Define
\begin{displaymath}
  \mathbf{D} \iniseg \mathbf{D}'
\end{displaymath}
 iff  $\mathbf{D} = \mathbf{D}' \res l$ for some $l < \lh(\mathbf{D}')$. Define $\mathbf{D}^{-} = \mathbf{D} \res \lh(\mathbf{D})-1$. Define $\iniseg^{Q,W} = \iniseg \res \desc(Q,W)$.

 The ordering of $\seed_{\mathbf{D}}^{Q,W}$ is definable in the following concrete way. 
Put 
\begin{displaymath}
  \corner{\mathbf{D}}=
  \begin{cases}
    (1,\mathbf{q}) & \text{if } d =1, \\
    (2, \sigma \oplus \mathbf{q} ) & \text{if } d = 2.
  \end{cases}
\end{displaymath}
Define
\begin{displaymath}
\mathbf{D} \prec \mathbf{D}'
\end{displaymath}
 iff $\corner{\mathbf{D}}<_{BK} \corner{\mathbf{D}'}$, the ordering on subcoordinates in $ \omega^{<\omega} $ again according to $<_{BK}$. 
For example, the constant $(Q,*)$-description $\mathbf{D}_0$ is the $\prec$-maximum, and we have $\corner{\mathbf{D}_0} = (2,\emptyset)$. 
When $1 \leq \card(\comp{1}{Q}) < \aleph_0$, the $\prec$-least $(Q,*)$-description is $( 1, q, \emptyset)$, where $q$ is the $<_{BK}$-least node in $\comp{1}{Q}$. 
When $W \neq \emptyset$,  the $\prec$-least $(Q,W)$-description of degree 2 is $\mathbf{D}_W=(2, ((-1),\se{(0)}, ( ( 0) )), \sigma_W)$, where $\sigma_W((0))=$the $<_{BK}$-least node in ${W}$, and we have $\corner{\mathbf{D}_W} =(2, (\sigma_W(1),-1))$. Define $\prec^{Q,W} = \prec \res \desc(Q,W)$. 
 $\prec^{Q,W}$ exactly determines  the order of the $\seed_{\mathbf{D}}^{Q,W}$'s, as in the following lemma. It is parallel to Lemma~\ref{lem:level_2_beta_q_order}. 
 \begin{mylemma}
   \label{lem:level_2_desc_order}
   Suppose   $\mathbf{D},\mathbf{D}' \in \desc(Q,W)$ and  $\mathbf{D} \prec ^{Q,W} \mathbf{D}'$. Then
   \begin{enumerate}
   \item For any $g \in \omega_1^{Q \uparrow}$, for any $\vec{\alpha} \in \omega_1^{W \uparrow}$, $g^W_{\mathbf{D}}(\vec{\alpha}) < g^W_{\mathbf{D}'}(\vec{\alpha})$.
   \item Suppose $Q$ is finite. Then $\seed_{\mathbf{D}}^{Q,W} < \seed_{\mathbf{D}'}^{Q,W}$. Moreover,  for any $\beta < u_2$, $\mathbf{D}^{Q,W}(\beta) < \seed_{\mathbf{D}'}^{Q,W}$.
   \end{enumerate}
 \end{mylemma}
 \begin{proof}
1. Simple computation. 

2.   Note that $\mathbf{D}^{Q,W}(\omega_1) = \seed_{\mathbf{D}}^{Q,W}$.  
We directly prove the ``moreover'' part. 
We are given $\beta = j_{\mu_{\mathbb{L}}}(h)(\omega_1)$, where $h$ is a function into $\omega_1$. 
Let $E \in \mu_{\mathbb{L}}$ such that for any $\alpha \in E$, $h(\alpha)< \min (E \setminus \alpha+1)$.
We have $\mathbf{D}^{Q,W}(\beta)= j^Q\circ j^W(h) (\seed_{\mathbf{D}}^{Q,W})$. By \Los{}, it suffices to show that for any $g \in E^{Q \uparrow}$,  $j^W(h) ( [g^W_{\mathbf{D}}]_{\mu^W}) <  [g_{\mathbf{D}'}^W]_{\mu^W}$.  
By \Los{} again, it suffices to show that for any $\vec{\alpha} \in [\omega_1]^{W \uparrow}$, $h ( g_{\mathbf{D}}^W  (  \vec{\alpha})) <  g _{\mathbf{D}'}^W(  \vec{\alpha}  )$. 
We already  know that 
$g_{\mathbf{D}}^W  (  \vec{\alpha}), g_{\mathbf{D}'}^W  (  \vec{\alpha})  \in E$. By our choice of $E$, it suffices to show that $g_{\mathbf{D}}^W  (  \vec{\alpha})<g_{\mathbf{D}'}^W  (  \vec{\alpha})$. This is exactly part 1.
 \end{proof}

 Suppose $W$ is a level-1 proper subtree of $W'$, $W'$ is finite, $w \in W\cup \se{\emptyset}$, $w' \in W'\setminus W$. Define
 \begin{displaymath}
   w = w' \res W
 \end{displaymath}
iff $w' <_{BK}w $ and $\set{w^{*} \in W}{ w' <_{BK} w^{*} <_{BK} w  } = \emptyset$. 

The $\res W$ operator inherits the following trivial continuity property.
\begin{mylemma}
\label{lem:W_desc_extension}
Suppose $W$ is a level-1 proper subtree of $W'$, $W'$ is finite, $w \in W \cup \se{\emptyset}$, $w' \in W'\setminus W$, $w= w'\res W$. 
Suppose $C \in \mu_{\mathbb{L}}$ is a club, $C'$ is the set of limit points of $C$. 
Then for any $\vec{\alpha} \in [C']^{W \uparrow}$,
\begin{displaymath}
  \alpha_w = \sup \set{\beta_{w'}}{ \vec{\beta} \in [C]^{W'\uparrow}, \vec{\beta} \text{ extends }\vec{\alpha}}.
\end{displaymath}
\end{mylemma}

Suppose $W$ is a proper level-1 subtree of $W'$.
For $\mathbf{D}  \in \desc(Q,W)$ and $\mathbf{D}' \in \desc(Q,W') \setminus \desc(Q,W)$,  define
\begin{displaymath}
  \mathbf{D} = \mathbf{D}' \res (Q,W)
\end{displaymath}
iff  $\mathbf{D}' \prec \mathbf{D}$ and $\{\mathbf{D}^{*} \in \desc(Q,W) : \mathbf{D}' \prec \mathbf{D}^{*} \prec \mathbf{D}\} = \emptyset$. Thus, $  \mathbf{D} = \mathbf{D}' \res (Q,W)$ iff both $\mathbf{D},\mathbf{D}'$ are of degree 2 and letting $\mathbf{D} = (2, (q,P, \vec{p}), \sigma)$, $\mathbf{D}' = (2, (q',P', \vec{p}'),\sigma')$, $\lh(q) = k$, $\vec{p}=(p_i)_{i < \lh(\vec{p})}$,  then either
\begin{enumerate}
\item $q$ is of continuous type (hence $\lh(\vec{p}) = k$), 
  $\mathbf{D}^{-} \iniseg \mathbf{D}'$, $\sigma(p_{k-1}) =  \sigma'(p_{k-1}) \res W$, or 
\item $q$ is of discontinuous type (hence $\lh(\vec{p}) = k+1$), $\mathbf{D} \iniseg \mathbf{D}'$, $\sigma(p_k^{-}) = \sigma'( p_k) \res W$. 
\end{enumerate}

As a corollary to Lemma~\ref{lem:W_desc_extension}, the $\res (Q,W)$ operator inherits the following continuity property. 
\begin{mylemma}
  \label{lem:QW_description_extension}
  Suppose  $W$ is a proper subtree of $W'$, $\mathbf{D} \in \desc(Q,W)$, $\mathbf{D}' \in \desc(Q,W')$,   $\mathbf{D}= \mathbf{D}' \res (Q,W)$.  
Suppose $C \in \mu_{\mathbb{L}}$ is a club, $C'$ is the set of limit points of $C$. 
Then for any $g \in \omega_1^{Q \uparrow}$, for any $\vec{\alpha} \in [C']^{W \uparrow}$,
\begin{displaymath}
  g^W_{\mathbf{D}}(\vec{\alpha}) = \sup \{g_{\mathbf{D}'}^{W'}(\vec{\beta}): \vec{\beta} \in [C]^{W' \uparrow} ,\vec{\beta} \text{ extends } \vec{\alpha}\} .
\end{displaymath}
\end{mylemma}

Suppose $Q$ is a proper subtree of $Q'$, both finite. For $(d,\mathbf{q}) \in \exdesc(Q)$, $(d',\mathbf{q}' )\in \exdesc(Q') \setminus \exdesc(Q)$, define 
\begin{displaymath}
 (d, \mathbf{q})= (d',\mathbf{q}') \res Q
\end{displaymath}
iff $(d',\mathbf{q}' )\prec (d, \mathbf{q})$ and $\{(d^{*}, \mathbf{q}^{*}) \in \exdesc(Q) : (d',\mathbf{q}' )\prec (d^{*}, \mathbf{q}^{*}) \prec (d,\mathbf{q})\} = \emptyset$. 
Thus, 
$ (d, \mathbf{q}) =  (d',\mathbf{q}') \res Q$ iff either 
\begin{enumerate}
\item $d=d'=1$, $\mathbf{q}= \mathbf{q}' \res {\comp{1}{Q}}$, or
\item  $d'=1$, $\emptyset  = \mathbf{q}'  \res {\comp{1}{Q}}$, $d=2$, $\mathbf{q} \in \{ ((-1), \se{(0)}, ((0))), (\emptyset,\emptyset,(0))\}$, or
\item $d=d'=2$, letting $\mathbf{q} = (q,P, \vec{p})$, $\vec{p} = (p_i)_{i < \lh(\vec{p})}$, $\mathbf{q} = (q',P', \vec{p}')$, $\vec{p}' =  (p_i')_{i < \lh(\vec{p}')}$,  $\lh(q) = k$, then either
  \begin{enumerate}
  \item $\mathbf{q} \in \desc(Q)$ is of continuous type, $k \geq 2$, $(P,\vec{p}) = (P', \vec{p}'\res k)$, $(q^{-})^{-} \subsetneq q'$, $q(k-2) = q'(k-2) \res \comp{2}{Q}\se{(q^{-})^{-}}$, or 
  \item $\mathbf{q} \in \desc(Q)$ is of discontinuous type, $k \geq 2$, $(P,\vec{p}\res k) = (P', \vec{p}'\res k)$, $(q^{-})^{-} =( q')^{-}$, $q(k-1) = q'(k-1) \res \comp{2}{Q}\se{q^{-}}$, or
  \item $\mathbf{q} \notin \desc(Q)$, $q \subsetneq q'$, $\emptyset  =  q'(k) \res \comp{2}{Q}\se{q^{-}}$.
  \end{enumerate}
\end{enumerate}

As a corollary to Lemma~\ref{lem:level_2_beta_q_order} and Lemma~\ref{lem:Q_respecting}, the $\res Q$ operator inherits the following continuity property. 

\begin{mylemma}
  \label{lem:Q_desc_extension}
  Suppose $C \in \mu_{\mathbb{L}}$ is a club.  Let $\eta \in C'$ iff  $C \cap \eta$ has order type $\eta$. 
Suppose $Q$ is a proper subtree of $Q'$, $Q,Q'$ are finite, $(d,\mathbf{q})\in \exdesc(Q)$ , $(d',\mathbf{q}')\in \exdesc(Q')$,   $(d,\mathbf{q}) = (d',\mathbf{q}') \res Q$. Then for any $\vec{\beta} \in [C']^{Q \uparrow}$,
      \begin{displaymath}
       {\comp{d}{\beta}}_{\mathbf{q}} = \sup \set{{\comp{d'}{\gamma}}_{\mathbf{q}'}}{ \vec{\gamma} \in [C]^{Q'\uparrow}, \vec{\gamma} \text{ extends } \vec{\beta}}.
      \end{displaymath}
\end{mylemma}

In the proof of Lemma~\ref{lem:Q_desc_extension}, the construction of $\vec{\gamma}$ that witnesses the $\leq$ direction relies on the assumption that $\eta \in C'$ iff $C\cap \eta$ has order type $\eta$.

Suppose $Q$ is a proper subtree of $Q'$, both finite. 
For
 $\mathbf{D} \in \desc(Q,W)$, $\mathbf{D}'  \in \desc(Q',W)\setminus \desc(Q,W)$.  Define  
\begin{displaymath}
  \mathbf{D} = \mathbf{D}' \res (Q,W)
\end{displaymath}
iff  $\mathbf{D}' \prec \mathbf{D}$ and $\{\mathbf{D}^{*} \in \desc(Q,W) : \mathbf{D}' \prec \mathbf{D}^{*} \prec \mathbf{D}\} = \emptyset$.   Thus, putting $\mathbf{D} = (d,\mathbf{q},\sigma)$, $\mathbf{D} = (d',\mathbf{q}',\sigma')$, 
$\mathbf{D}= \mathbf{D}' \res (Q,W)$ iff either 
  \begin{enumerate}
  \item $q$ is of continuous type (hence $\lh(\vec{p}) = k$),  and either
    \begin{enumerate}
    \item $\mathbf{D}$ is of $*$-$W$-discontinuous type, $\mathbf{D}^{-} \iniseg \mathbf{D}'$, $\sigma'(p_{k-1}) = \pred_{\prec^W}(\sigma(p_{k-1}))$, $q'(k-1) \res \comp{2}{Q}\se{q^{-}} = \emptyset$, or
    \item $\mathbf{D}$ is of $*$-$W$-continuous type, $(\mathbf{D}^{-})^{-} \iniseg \mathbf{D}' $, $q'(k-2) \res \comp{2}{Q}\se{(q^{-})^{-}} = q(k-2)$, or
    \end{enumerate}
  \item $q$ is of continuous type (hence $\lh(\vec{p}) = k+1$), and either 
    \begin{enumerate}
    \item $\mathbf{D}$ is of $*$-$W$-discontinuous type, $\mathbf{D} \iniseg \mathbf{D}'$, $\sigma'(p_{k}) = \pred_{\prec^W}(\sigma(p_{k}^{-}))$, $q'(k) \res \comp{2}{Q}\se{q} = \emptyset$, or
    \item $\mathbf{D}$ is of $*$-$W$-continuous type, $\mathbf{D}^{-} \iniseg \mathbf{D}' $, $q'(k-1) \res \comp{2}{Q}\se{q^{-}} = q(k-1)$.
    \end{enumerate}
\end{enumerate}

The $\res (Q,W)$ relation inherits the following continuity property.

\begin{mylemma}
  \label{lem:QW_description_extension_another}
Suppose $C \in \mu_{\mathbb{L}}$ is a club.  Let $\eta \in C'$ iff $\eta \in C$ and $C \cap \eta$ has order type $\eta$. 
Suppose $Q$ is a proper subtree of $Q'$, both finite, $\mathbf{D} =(d,\mathbf{q},\sigma)\in \desc(Q,W)$, $\mathbf{D}' =(d',\mathbf{q}',\sigma')\in \desc(Q',W)$,   $\mathbf{D} = \mathbf{D}' \res (Q,W)$. Then for any $\vec{\beta} \in [C']^{Q \uparrow}$,
      \begin{displaymath}
       \sigma^{W}({\comp{d}{\beta}}_{\mathbf{q}}) = \sup \set{({\sigma}')^{W}({\comp{d}{\gamma}}_{\mathbf{q}'})}{ \vec{\gamma} \in [C]^{Q'\uparrow}, \vec{\gamma} \text{ extends } \vec{\beta}}.
      \end{displaymath}
\end{mylemma}
\begin{proof}
The  $\geq$ direction follows from Lemma~\ref{lem:level_2_desc_order}. We show the $\leq $ direction.   When $d=d'=1$, both sides are equal to $\comp{d}{\beta}_{\mathbf{q}}$ by Lemma~\ref{lem:W_desc_extension}. When $d=2\wedge d'=1$, both sides are equal to $\omega_1$ by Lemma~\ref{lem:W_desc_extension} again. Suppose now $d=d'=2$. Let  $\mathbf{q} = (q,P, \vec{p})$, $\vec{p} = (p_i)_{i < \lh(\vec{p})}$, $\mathbf{q} = (q',P', \vec{p}')$, $\vec{p}' =  (p_i')_{i < \lh(\vec{p}')}$,  $\lh(q) = k$. 

Case 1: $q$ is of continuous type. 

Subcase 1.1: $\mathbf{D}$ is of $*$-$W$-discontinuous type.

Let $P^{-} =  P \setminus \se{ p_{k-1}}$. So  $ \comp{2}{Q}(q^{-}) =(P^{-}, p_{k-1}) $.
Let $q'' = q' \res k$,  $\sigma'' = \sigma' \res P$, $p''=\comp{2}{Q}_{\node}(q'')$. Then $(2, (q'', P, \vec{p}\concat (p''))) = (2, (q^{-},P, \vec{p})) \res Q$. By Lemma~\ref{lem:Q_desc_extension}, 
\begin{displaymath}
  j^{P^{-},P} (\comp{2}{\beta}_{q^{-}}) = \sup \set{\comp{2}{\gamma}_{{q}''}}{\vec{\gamma} \in [C]^{Q'\uparrow}, \vec{\gamma} \text{ extends } \vec{\beta}}.
\end{displaymath}


It suffices to show that
\begin{displaymath}
  \sigma^W \circ j^{P^{-},P}_{\sup}(\comp{2}{\beta}_{q^{-}}) =   ({\sigma}'')^W_{\sup} \circ j^{P^{-},P} (\comp{2}{\beta}_{q^{-}}).
\end{displaymath}
This is exactly Lemma~\ref{lem:level_2_uniform_cofinality}, using the fact   $\cf^{\mathbb{L}}(\comp{2}{\beta}_{q^{-}}) = \seed^{P^{-}}_{p_{k-1}^{-}}$. 

Subcase 1.2: $\mathbf{D}$ is of $*$-$W$-continuous type.

Simply use Lemma~\ref{lem:Q_desc_extension}.

Case 2. $q$ is of discontinuous type. 

Subcase 2.1:  $\mathbf{D}$ is of $*$-$W$-discontinuous type.
Let $P^{+}$ be the completion of $(P,p_k)$ if $p_k\neq -1$, $P^{+} = P$ if $p_k = -1$. 
Let $q''=q' \res k+1$, $\sigma'' = \sigma' \res P^{+}$. Then $(2, (q'', P^{+}, \vec{p}))=(2, (q, {P}^{+}, \vec{p})) \res Q$. By Lemma~\ref{lem:Q_desc_extension},
\begin{displaymath}
  j^{P,P^{+}} (\comp{2}{\beta}_q) = \sup \set{\comp{2}{\gamma}_{\mathbf{q}'}}{\vec{\gamma} \in [C]^{Q'\uparrow}, \vec{\gamma} \text{ extends } \vec{\beta}}.
\end{displaymath}
 It remains to show 
\begin{displaymath}
{\sigma}^W (\comp{2}{\beta}_{q}) =   ({\sigma}'')^W_{\sup} \circ j^{P,P^{+}} (\comp{2}{\beta}_{q}).
\end{displaymath}
This is exactly Lemma~\ref{lem:level_2_uniform_cofinality_another}, using the fact $\cf^{\mathbb{L}}(\comp{2}{\beta}_{q}) = \seed^P_{p_k^{-}}$ when $p_k\neq -1$, $\cf^{\mathbb{L}}(\comp{2}{\beta}_{q}) = \omega$ when $p_k = -1$.  

Subcase 2.2: $\mathbf{D}$ is of $*$-$W$-continuous type.

Simply use Lemma~\ref{lem:Q_desc_extension}.

\end{proof}

\begin{mydefinition}
  \label{def:factoring_2}
   Suppose $S$ is a finite regular level-1 tree and  $Q$ is a level $\leq 2$ tree.  Suppose $\tau : S \cup \se{\emptyset}\to \desc(Q,*)$ is a function. Then $\tau$ factors $(S,Q,*)$ iff
  \begin{enumerate}
  \item $\tau(\emptyset)$ is the constant $(Q,*)$-description.
  \item If $s \prec^{{S}} s'$, then ${\tau}(s) \prec {\tau}(s')$. 
  \end{enumerate}
For a level-1 tree $W$, $\tau$ factors $(S,Q,W)$ iff $\tau$ factors $(S,Q,*)$ and $\ran(\tau) \subseteq \desc(Q,{W})$. In particular, if every $\tau(s)$ is of degree 1, then $\tau$ factors $(S,Q,\emptyset)$.

 \begin{myexample}
   Let $S,Q,W$ be as in Section~\ref{sec:level-1-2}. There is a $(S,Q,W)$-factoring map $\tau$ characterizing the $S$-equivalence class of $\theta_{SQ}$. We have for instance,
   \begin{itemize}
   \item $\tau ((3)) = ( 2 ,  (-1,\se{(0)},((0))), \sigma )$, where $\sigma((0)) = (2)$.
   \item $\tau((1)) = (2, ( (0,0),  \se{(0),(0,0)}, ((0),(0,0))), \sigma'  ))$, where $\sigma'((0)) = (1)$, $\sigma'((0,0)) = (0)$.
   \end{itemize}
 \end{myexample}

If $S$ is a level-1 tree, then
\begin{displaymath}
  \id_{*,S}
\end{displaymath}
factors $(S,Q^0,S)$, where $\id_{*,S} (s)  =  (2, ( (-1), \se{(0)}, ((0))), \sigma_s)$, $\sigma_s(0) = s$. 

Suppose $\tau$ factors $(S,Q,W)$.
For  $g \in \omega_1^{Q \uparrow}$, let
\begin{displaymath}
  g_{\tau}^W : [\omega_1]^{W \uparrow} \to [\omega_1]^{S \uparrow}
\end{displaymath}
be the function sending $\vec{\alpha}$ to $ (g_{\tau(s)}^W(\vec{\alpha}))_{s \in \dom(S) } $.    Lemma~\ref{lem:level_2_desc_order} ensures that $g^W_{\tau}$ is indeed a function into $[\omega_1]^{S \uparrow}$. In particular, $g_{\id_{*,S}}^S$ is the identity map on $[\omega_1]^{S \uparrow}$. 
\begin{displaymath}
  \id^{Q,W}_{\tau} 
\end{displaymath}
is the map sending $[g]^Q$ to $[g^W_{\tau}]_{\mu^W}$. So $\id^{Q,W}_{\tau} (\vec{\beta}) = (\id^{Q,W}_{\tau (s )}(\vec{\beta}))_{s \in S}$.
Put 
\begin{displaymath}
  \seed_{\tau}^{Q,W} =  [\id^{Q,W}_{\tau}]_{\mu^Q}
\end{displaymath}
 By Lemma~\ref{lem:level_2_desc_order} and \Los{},  for any $A \in \mu^S$, $\seed^{Q,W}_{\tau} \in j^Q \circ j^W(A)$. Hence, we can unambiguously define
\begin{displaymath}
  \tau^{Q,W} : \admistwoboldextra{j^S} \to \admistwoboldextra{j^Q \circ j^W}
\end{displaymath}
by sending $j^S(h)(\seed^S) $ to $ j^Q\circ j^W (h)  ( \seed_{\tau}^{Q,W} )$. $\tau^{Q,W}$ is the unique map such that for any $z \in \mathbb{R}$, 
 $\tau^{Q,W}$ is elementary from $L_{\kappa_3^z}[j^S(T_2),z]$ into $ L_{\kappa_3^z}[j^Q\circ j^W(T_2), z]$ and
 for any $s\in S$, $ \tau^{Q,W} \circ s^S = \tau(s)^{Q,W}$.
\end{mydefinition}

\begin{mylemma}
  \label{lem:level_2_desc_cofinal_in_next}
Suppose $Q,W$ are finite. 
  \begin{enumerate}
  \item If  $\mathbf{D} = \min(\prec^{Q,W})$, then $\seed_{\mathbf{D}}^{Q,W} = \omega_1$. 
Hence $\mathbf{D}^{Q,W}$ is the identity on $\omega_1+1$.
  \item 
  If  $\mathbf{E}= \pred_{\prec^{Q,W}}(\mathbf{D})$, then $(\mathbf{E}^{Q,W})'' u_2 $ is a cofinal subset of  $\seed_{\mathbf{D}}^{Q,W}$. 
  \end{enumerate}
\end{mylemma}
\begin{proof}
We only prove the case when $\comp{1}{Q} = {\emptyset}$. The general case takes an analogous additional argument. 

Case 1: $W= \emptyset$. 

The only $(Q,W)$-description is the constant $(Q,*)$-description $\mathbf{D}_0$. 
We only have to prove part 1. 
 For any $x \in \mathbb{R}$,  $\mathbf{D}_0^{Q,W}=j^Q \circ j^W$ is elementary from $L_{\kappa_3^x}[T_2,x]$ into $L_{\kappa_3^x}[j^Q \circ j^W (T_2), x]$. It follows that $\mathbf{D}_0^{Q,W} \res \omega_1 $ is the identity map. It remains to show that $\seed_{\mathbf{D}_0}^{Q,W} = \omega_1$. 
We already know that $\seed_{\mathbf{D}_0}^{Q,W} =j^Q(\omega_1)$. 
Suppose $[g]_{\mu^Q} < j^Q(\omega_1)$ and we try to show that $[g]_{\mu^Q} \leq \omega_1$. 
Let $Q'$ be the completion of the partial level $\leq 2$ tree $(Q,(1,(0),\emptyset))$. 
Let $\mathbf{D}' = (1, (0), \emptyset)  \in \desc(Q',W)$. Then $\mathbf{D}_0 =  \mathbf{D}'  \res (Q,W)$. 
We partition functions $f \in \omega_1^{Q'\uparrow} $  according to whether ${}^1[f]^{Q'}_{(0)} \leq g ( [f \res \rep(Q)]^Q)$. By the level-2 partition property of $\omega_1$ in \cite{sharpII}, we obtain a club $C \in \mu_{\mathbb{L}}$ which is homogeneous for this property. Let  $\eta \in C'$ iff $\eta \in C$ and $C \cap \eta$ has order type $\eta$. 
If the homogeneous side satisfies ${}^1[f]^{Q'}_{(0)} > g([f \res \rep(Q)]^Q)$, we let $\alpha_0 =$the $\omega$-th element of $C$, and so 
every $f \in [C']^{Q \uparrow}$ is extendable to $f' \in C^{Q' \uparrow}$  so that ${}^1[f']^{Q'}_{(0)} = \alpha_0$. Therefore, for every $\vec{\xi} \in [C']^{Q \uparrow}$, $g(\vec{\xi}) < \alpha_0$. Hence by \Los{},  $[g]_{\mu^Q} < j^Q(\alpha_0) = \alpha_0$ and we are done. If the homogeneous side satisfies ${}^1[f]^{Q'}_{(0)} \leq g([f\res \rep(Q)]^Q)$, then by Lemma~\ref{lem:QW_description_extension_another}, $\omega_1 = {}^2[f\res \rep(Q)]^Q_{\emptyset} \leq g([f\res \rep(Q)]^Q)$,  
contradicting to the assumption on $g$. 

Case 2: $W \neq \emptyset$. 

 We firstly prove part 1.   The $\prec^{Q,W}$-minimum is 
 $\mathbf{D} _0= (2, \mathbf{q},  \sigma)$, where $\mathbf{q}=( (-1)  , \se{(0)}, ((0)))$, ${\sigma}((0)) $ is the $<_{BK}$-least node in $W$. 
$\seed_{\mathbf{D}_0}^{Q,W}$ is represented modulo $\mu^Q$ by the function that sends $\vec{\beta}$ to ${\sigma}^W(\beta_{\mathbf{q}}) = {\sigma}^W(\omega_1) = \omega_1$. Hence, $\seed_{\mathbf{D}_0}^{Q,W} = j^Q(\omega_1)$. 
 Work with the same $Q'$ as in Case 1 and argue with the same partition arguments. 

Next, we prove part 2. Let $\mathbf{D} = (2,\mathbf{q},\sigma)$, $\mathbf{q} = (q,P, \vec{p})$,  $\mathbf{E} = (2, \mathbf{r}, \tau)$, $\mathbf{r} = (r, Z, \vec{z})$. 
Then $q \neq (-1)$.

Subcase 2.1: $r$ is of discontinuous type. 

Let $Q'$ be the level $\leq 2$ tree extending $Q$ such that $\dom(Q') \setminus \dom(Q) = \se{(2, {r}')}$,  $r' = r ^{-}\concat (a)$, $\emptyset = a \res \comp{2}{Q}\se{r^{-}}$, $Q'(r') = Q(r)$.
 Let  $\mathbf{r}' = (r',Z,\vec{z})$, $\mathbf{E}'=(2,\mathbf{r}', \tau)$. Then $\mathbf{D}= \mathbf{E}' \res (Q,W)$. 
Our partition arguments will be based on $Q'$.

Suppose  $[g]_{\mu^Q} < \seed_{\mathbf{D}}^{Q,W}$ and we seek $\eta_0< u_2$ such that $[g]_{\mu^Q} < \mathbf{E}^{Q,W}  (\eta_0) $. 
We partition  functions $f\in \omega_1^{Q' \uparrow}$ according to whether $\tau^{W}({}^2 [f]^{Q'}_{r'})  \leq g([f \res \rep(Q)]^Q)$. By the level-2 partition property of $\omega_1$ in \cite{sharpII}, we obtain a club $C \in \mu_{\mathbb{L}}$ which is homogeneous for this property. 
Let  $\eta \in C'$ iff $\eta \in C$ and $C \cap \eta $ has order type $\eta$.  If the homogeneous side satisfies  $\tau^{W}({}^2 [f]^{Q'}_{r'})  > g([f \res \rep(Q)]^Q)$, we let $\eta_0 =j_{\mu_{\mathbb{L}}}(h_0)(\omega_1)$, where  $h_0(\alpha)= \min(C' \setminus (\alpha+1))$. This allows us to extend every $f \in (C')^{Q\uparrow}$ to $f' \in C^{Q' \uparrow}$ so that ${}^2[f']^{Q'}_{r'} = j^{P} (h_0)( [f]^Q_{{r}} )$. Therefore, for every $\vec{\xi} \in [C']^{Q \uparrow}$, $g(\vec{\xi}) < \tau ^W(  j^{Z}(h_0) ( \xi_{{r}})) = j^W(h_0) (\tau^W(\xi_{{r}}))$. Hence $[g]_{\mu^Q} < j^Q \circ j^W (h_0) (\seed_{\mathbf{E}}^{Q,W})  = \mathbf{E}^{Q,W} ( j_{\mu_{\mathbb{L}}}(h_0)(\omega_1) )$. Hence $[g]_{\mu^Q} < \mathbf{E}^{Q,W}(\eta_0)$. 
If the homogeneous side satisfies $\tau^{W}({}^2 [f]^{Q'}_{r'})  \leq g([f \res \rep(Q)]^Q)$, then by Lemma~\ref{lem:QW_description_extension_another}, $\sigma^W ({}^2[f]^Q_q) \leq g([f \res \rep(Q)]^Q)$. This contradicts our assumption on $g$. 

Subcase 2.2: $r$ is of continuous type.

Let $Q'$ be the level $\leq 2$ tree extending $Q$ such that $\dom(Q') \setminus \dom(Q) = \se{(2, {r}')}$,  
$r' = (r^{-})^{-} \concat (a)$, $\emptyset  = a \res \comp{2}{Q} \se{(r^{-})^{-}} $,  $Q'(r') = Q(r^{-})$.  Let  $\mathbf{r}' = (r' \concat (-1) ,Z,\vec{z})$, $\mathbf{E}'=(2,\mathbf{r}', \tau)$. Then $\mathbf{D} = \mathbf{E}' \res (Q,W)$. The rest is similar to Subcase 2.1. 
\end{proof}

At this point, it is convenient to label the nodes of a tree of uniform cofinalities using arbitrary sets instead of elements in $\omega^{<\omega}$ and $(\omega^{<\omega})^{<\omega}$. 
Suppose  $Q$ is a level $\leq 2$ tree and $W$ is a level-1 tree. \emph{A representation of $Q \otimes W$} is a pair $(S,\tau)$ such that $S$ is a level-1 tree, $\tau$ factors $(S,Q,W)$, $\ran(\tau) = \desc(Q,W)$, and $s \prec^S s'$ iff  $\tau(s) \prec^{Q,W} \tau(s')$. Representations of $Q \otimes W$ are clearly mutually isomorphic. We shall informally regard
\begin{displaymath}
Q \otimes W = \desc(Q,W) \setminus \se{\text{the constant $(Q,W)$-description}}
\end{displaymath}
as a ``level-1 tree'' by identifying it with $S$ via $\tau$. 
We put $\seed^{Q \otimes W}_{\mathbf{D}} = \seed^S_{\tau^{-1}(\mathbf{D})}$ for $\mathbf{D} \in \desc(Q,W)$. If $\tau'$ factors $(S',Q,W)$, then $\tau'$ also factors ``level-1 trees'' $(S',Q \otimes W)$, and $(\tau')^{Q \otimes W}$ makes sense. That is, $(\tau')^{Q \otimes W} = (\tau^{-1} \circ \tau')^S$, where $\tau^{-1} \circ \tau'$ factors $(S',S)$. The identity map $ \id_{Q \otimes W} : \mathbf{D} \mapsto \mathbf{D}$ factors $(Q \otimes W, Q,W)$. 
If $Q$ is a subtree of $Q'$ and $W$ is a subtree of $W'$, then $Q \otimes W$ is regarded as a subtree of $Q' \otimes W'$, and the map  $j^{Q \otimes W, Q' \otimes W'}$ makes sense. In other words, let $(S,\tau)$ be a representation of $Q \otimes W$ and $(S',\tau')$ be a representation of $Q' \otimes W'$ such that $S$ is a subtree of $S'$ and $\tau \subseteq \tau'$, then $j^{Q \otimes W, Q' \otimes W'} = j^{S,S'}$. If $\pi$ factors level $\leq 2$ trees $(Q,T)$, then
\begin{displaymath}
  \pi \otimes W
\end{displaymath}
factors level-1 trees $(Q \otimes W, T \otimes W)$, where $\pi (d, \mathbf{q}, \sigma) = (d, \comp{d}{\pi}(\mathbf{q}), \sigma)$.

\begin{mylemma}
  \label{lem:factor_SQW}
Suppose $Q$ is a finite level $\leq 2$ tree, $W$ is a  finite level-1 tree.
\begin{enumerate}
\item If $\mathbf{D} \in \desc(Q,W)$, then $\seed^{Q \otimes W}_{\mathbf{D}} = \seed^{Q,W}_{\mathbf{D}}$. 
\item  $(\id_{Q \otimes W})^{Q,W} $ is identity on $j^Q \circ j^W (\omega_1+1)$.
\item If $S$ is a level-1 tree, $\tau$ factors $(S,Q,W)$, then $\tau^{Q,W} = \tau^{Q\otimes W}$.
\end{enumerate}
\end{mylemma}
\begin{proof}
Let  $\mathbf{D}_0,\ldots,\mathbf{D}_m$ enumerate $\desc(Q,W)$ in the $\prec^{Q,W}$-ascending order. 
We prove by induction on $l \leq m$ that $\seed^{Q \otimes W}_{\mathbf{D}_i} = \seed^{Q,W}_{\mathbf{D}_i}$ for any $i \leq l$.

Suppose $l=0$. The fact $\seed_{\mathbf{D}_0}^{Q,W} = \omega_1$ follows  from  Lemma~\ref{lem:level_2_desc_cofinal_in_next}. 

Suppose the induction hypothesis holds at $l<m$. 
That is, $(\mathbf{D}_l)^{Q,W} (\omega_1) = \seed_{\mathbf{D}_l}^{Q,W} = u_{l+1}$ for $l<m$. By \Los{}, $(\id_{Q \otimes W})^{Q,W}  $ is identity on $u_{l+2}$.
But $((\mathbf{D}_l)^{Q,W})'' u_2$ is a cofinal subset of $\seed_{\mathbf{D}_{l+1}}^{Q,W}$ by Lemma~\ref{lem:level_2_desc_cofinal_in_next}. 
Hence, $\seed^{Q,W}_{\mathbf{D}_{l+1}} = u_{l+2}$. This proves part 1. Parts 2-3 are immediate corollaries. 
\end{proof}



$\mathbf{D} \in \desc(Q,W)$ is \emph{direct} iff either $\mathbf{D}$ is of degree 1 or $\mathbf{D} $ is of the form $(2,(q, P, \vec{p}) , \id_P)$. Lemma~\ref{lem:factor_SQW} has the following corollary on representations of uniform indiscernibles in the $\mu^Q$-ultrapower.
\begin{mylemma}
  \label{lem:Q_desc_uniform_indis}
  Suppose $Q$ is a finite level $\leq 2$ tree. Then
  \begin{displaymath}
    \set{u_n}{1 \leq n < \omega} = \set{\seed^{Q,W}_{\mathbf{D}}}{W \text{ finite, } \mathbf{D} \in \desc(Q,W) \text{ is direct}}.
  \end{displaymath}
\end{mylemma}

For $(d,q) \in \dom(Q)$, define
\begin{displaymath}
  \cf^Q(d,q) =
  \begin{cases}
    0  & \text{ if } d=1 \vee (d=2 \wedge \comp{2}{Q}(q) \text{ of degree 0}),\\
    1  & \text{ if } d=2 \wedge \ucf(\comp{2}{Q}(q)) = \min(\prec^{\comp{2}{Q}_{\tree}(q)}),\\
   2   & \text{ otherwise}. 
  \end{cases}
\end{displaymath}
By Lemma~\ref{lem:level_2_desc_cofinal_in_next}, if $\vec{\beta}$ respects $Q$, then
\begin{displaymath}
  \cf^{\admistwobold}(\comp{d}{\beta}_q) = u_{\cf^Q(d,q)}
\end{displaymath}
where $u_0 = \omega$.

\subsection{Approximations of $S_3$ in \texorpdfstring{$\admistwobold$}{}}
\label{sec:approximation-LT3-LT2}

\begin{mylemma}
  \label{lem:jQ_move_factor_maps}
  Suppose $Q$ is a level $\leq 2$ tree, $W$ is a level-1 subtree of $W'$, all trees are finite. 
 Then $j^Q( j^{W,W'}\res j^W(\omega_1+1)) = j^{Q \otimes W,Q \otimes W'}\res (j^{Q \otimes W}(\omega_1+1))$, and hence $j^Q (j^{W,W'}_{\sup} \res j^W(\omega_1+1)) = j^{Q \otimes W,Q \otimes W'}_{\sup} (j^{Q \otimes W}(\omega_1+1))$ by sufficient elementarity of $j^Q$.
\end{mylemma}
\begin{proof}
  By Lemma~\ref{lem:factor_SQW}, $\seed^{Q,W}_{\mathbf{D}} = \seed^{Q \otimes W}_{\mathbf{D}}$ for $\mathbf{D} \in \desc(Q,W)$, and similarly for $W'$. So $j^{Q \otimes W, Q \otimes W'}(\seed^{Q,W}_{\mathbf{D}}) = \seed^{Q,W'}_{\mathbf{D}}$ for $\mathbf{D} \in \desc(Q,W)$.
since $j^Q(j^{W,W'})$ is elementary from $L[z]$ to $L[z]$ for any $z \in \mathbb{R}$, it suffices to show that 
$j^Q(j^{W,W'} \res j^W(\omega_1+1)) (\seed_{\mathbf{D}}^{Q,W}) = j^{Q \otimes W,Q \otimes W'}(\seed_{\mathbf{D}}^{Q,W})$ for any $\mathbf{D} \in \desc(Q,W)$. Fix $\mathbf{D} \in \desc(Q,W)$. Suppose the typical case when $\mathbf{D} = (2, \mathbf{q}, \sigma)$ is of degree 2. 
  Then by \Los{}, 
\begin{align*}
  j^Q (j^{W,W'} \res j^W(\omega_1+1)) (\seed_{\mathbf{D}}^{Q,W}) & =  j^Q (j^{W,W'} \res j^W(\omega_1+1)) ([\vec{\xi} \mapsto {\sigma}^W(\comp{2}{\xi}_{\mathbf{q}})]_{\mu^Q}) \\
  & = [\vec{\xi} \mapsto j^{W,W'}\circ {\sigma}^W(\comp{2}{\xi}_{\mathbf{q}})]_{\mu^Q} \\
  & =  [\vec{\xi} \mapsto { \sigma}^{W'}(\comp{2}{\xi}_{\mathbf{q}})]_{\mu^Q} \\ 
  & = \seed_{\mathbf{D}}^{Q,W'} \\
  & = j^{Q \otimes W,Q \otimes W'}(\seed_{\mathbf{D}}^{Q,W}).
\end{align*}
\end{proof}

\begin{mylemma}
  \label{lem:jQ_move_factor_maps_another}
  Suppose $\pi$ factors finite level $\leq 2$ trees $(Q,T)$ and $W$ is a finite level-1 tree, 
  all trees are finite. 
Then $\pi^T \res j^Q\circ j^W(\omega_1+1) = (\pi \otimes W)^{T\otimes W} \res j^{Q \otimes W}(\omega_1+1)$.
\end{mylemma}
\begin{proof}
Apply Lemma~\ref{lem:factor_SQW} and use the fact that $\pi^T (\seed^{Q,W}_{\mathbf{D}}) = \seed^{T,W}_{\pi \otimes W(\mathbf{D})}$ for $\mathbf{D} \in \desc(Q,W)$. 
\end{proof}

\begin{mylemma}
  \label{lem:jQ_move_T_2}
Suppose $Q$ is a finite level $\leq 2$ tree. Then
  \begin{enumerate}
  \item $j^{Q} \res \set{u_n}{n < \omega}$ is $\Delta^1_3$, uniformly in $Q$.
  \item $j^Q (u_{\omega}) = u_{\omega}$.
  \item $j^{Q} \res u_{\omega} $ is $\Delta^1_3$, uniformly in $Q$. 
  \item Suppose $P,P'$ are finite level-1 trees and  $\pi$ factors $(P,P')$.  Then $j^{Q} ( \pi^{P'} \res u_{\omega})$ is $\Delta^1_3$, uniformly in $(Q,P,P',\pi)$.
  \item $j^{Q} ( T_2) $ is $\Delta^1_3$, uniformly in $Q$. 
  \end{enumerate}
\end{mylemma}
\begin{proof}
1 and 2. By Lemma~\ref{lem:factor_SQW}.

3. $j^{Q} (\tau^{L[z]} (u_1,\ldots,u_n)) = \tau^{L[z]} ( j^{Q}(u_1),\ldots,j^{Q}(u_n))$. 

4. By Lemma~\ref{lem:jQ_move_factor_maps}.

5. by 4. 
\end{proof}

The following lemma refines Corollary~\ref{coro:delta13_bounds_ultrapowers}.
\begin{mylemma}
  \label{lem:level_2_embedding_bounded_by_delta13}
  Assume $\boldDelta{2}$-determinacy. Suppose $x\in \mathbb{R}$. Then 
for any finite level $\leq 2$ tree $Q$, $j^Q ( \kappa_3^x,\lambda_3^x) = (\kappa_3^x, \lambda_3^x)$. Moreover,  $S_3 \res \kappa_3^x$ and $S_3 \res \lambda_3^x$ are both uniformly $\Delta_1$-definable over $\admistwo{x}$ from $\se{T_2,x}$. 
\end{mylemma}
\begin{proof}
By elementarity,  $j^Q(\kappa_3^x)$ is the least $\gamma$ for which  $L_{\gamma}[j^Q(T_2),x]$ is admissible. But 
$j^Q(T_2) \in \admistwo{x}$ by Lemma~\ref{lem:jQ_move_T_2}. Consequently, 
$L_{\kappa_3^x}[j^Q(T_2),x]$ is admissible. Since 
$j^Q$ is non-decreasing on ordinals, we must have $j^Q(\kappa_3^x) = \kappa_3^x$. Similarly, $\lambda_3^x$, being the sup of the ordinals $\Delta_1$-definable over $\admistwo{x}$ from $\se{T_2,x}$, is also fixed by $j^Q$.

To define $S_3 \res \kappa_3^x$, it is of course enough to establish a uniformly $\Delta_1$ definition of $j^{Q,Q'}\res \kappa_3^x$ over $\admistwo{x}$, for $Q$ a level $\leq 2$ subtree of $Q'$. 
 Note that every element of $\admistwo{x}$ is $\Delta_1$-definable over $\admistwo{x}$ from $u_{\omega}\cup\se{T_2,x}$, and hence by \Los{}, every ordinal in $j^Q(\admistwo{x})$ is $\Delta_1$-definable over $j^Q(\admistwo{x})$ from parameters in $u_{\omega}\cup\se{j^Q(T_2),x}$. The lemma  follows immediately from Lemma~\ref{lem:level_2_embedding_bounded_by_delta13} and \Los{}:
 \begin{displaymath}
   j^{Q,Q'}(\gamma) = \gamma'
 \end{displaymath}
iff for some $\xi<\kappa_3^x$, some $\Sigma_1$-formula $\varphi$, and some ordinal $\alpha<u_{\omega}$,
\begin{displaymath}
  L_{\xi}[j^{Q}(T_2),x] \models \forall \delta~ (\delta=\gamma \eqiv \varphi(\delta,j^Q(T_2),x,\alpha)),
\end{displaymath}
and 
\begin{displaymath}
  L_{\xi}[j^{Q'}(T_2),x] \models \forall \delta~ (\delta=\gamma' \eqiv \varphi(\delta,j^{Q'}(T_2),x,j^{Q,Q'}(\alpha))).
\end{displaymath}
This is a $\Sigma_1$ definition of $j^{Q,Q'}(\gamma) = \gamma'$ over $\admistwo{x}$ from $\se{T_2,x}$. In a similar way, we can write down a $\Sigma_1$ definition of $j^{Q,Q'}(\gamma) \neq \gamma'$.  The definition of $S_3 \res \lambda_3^x$ is similar. 
\end{proof}

In light of Lemma~\ref{lem:level_2_embedding_bounded_by_delta13}, $L_{\kappa_3^x}[S_3 \res \kappa_3^x]$ is regarded as the ``lightface core'' of $\admistwo{x}$, analogous to $L_{\omega_1^x}$ versus $L_{\omega_1^x}[x]$. In parallel to Guaspari-Kechris-Sacks in  \cite{guaspari_thesis,kechris_analytical_I,sacks_countable_adm_1976}, if  $C_3$ is the largest countable $\Pi^1_3$ set of reals, then $x \in C_3$ iff $x \in L_{\kappa_3^x}[S_3 \res \kappa_3^x]$ iff $x \in L_{\lambda_3^x}[S_3 \res \lambda_3^x]$. A related result about $C_3$ is in \cite{guaspari_harrington_C3} which follows the same line. 
An inner model theoretic characterization of $C_3$ is still unknown.

Recall that the set of  uncountable $\mathbb{L}$-regular cardinals below $u_{\omega}$ is $\set{u_n}{1 \leq n < \omega}$. 
The scenario in the AD world suggests that the set of uncountable $\admistwobold{}$-regular cardinals should be $\se{u_1,u_2}$. For  a  finite level $\leq 2$ tree $Q$, by Lemma~\ref{lem:jQ_move_T_2}, $\mathbb{L}_{u_{\omega}} \subseteq \admistwoboldextra{j^Q} \subseteq \admistwobold$, so the set of  uncountable $ \admistwoboldextra{j^Q}$-regular cardinals should be $\set{u_n}{n \in A}$ for some set $\se{1,2} \subseteq A \subseteq \omega\setminus 1$. Which $u_n$ is $\admistwoboldextra{j^Q}$-regular? The answer to this is an abstraction of Jackson's uniform cofinality analysis on functions $F: [\omega_1]^{Q \uparrow} \to \bolddelta{3}$ that lie in $\admistwobold{}$, originally in \cite{jackson_delta15}. In particular, we confirm that the set of uncountable $\admistwobold{}$-regular cardinals is indeed $\se{u_1,u_2}$.

\begin{mytheorem}
  \label{thm:jQ_LT2_regularity}
Suppose $Q$ is a finite level $\leq 2$ tree, $W$ is a finite level-1 tree. 
  Suppose $\mathbf{D}=(d,\mathbf{q},\sigma) \in \desc(Q,W)$. 
Then
\begin{displaymath}
\cf^{\admistwoboldextra{j^Q}} (\seed^{Q,W}_{\mathbf{D}})= {\seed}^{Q}_{\ucf^W_2(\mathbf{D})}.
\end{displaymath}
In particular, the set of  uncountable $\admistwoboldextra{j^Q}$-regular cardinals is exactly 
\begin{displaymath}
 \set{ \seed^Q_{(d,\mathbf{q})}}{(d,\mathbf{q}) \in \exdesc(Q) \text{ is regular}}.
\end{displaymath}
\end{mytheorem}
\begin{proof}
Put $(d,\mathbf{r}) = \ucf^W_2(\mathbf{D})$. 
Firstly, we prove  $\cf^{\admistwoboldextra{j^Q}} (\seed^{Q,W}_{\mathbf{D}})= \cf^{\admistwoboldextra{j^Q}} (\seed^{Q}_{(d,\mathbf{r})})$. There is nothing to prove for $d=1$. Suppose now $d=2$.

Case 1: $q$ is of continuous type. 

Subcase 1.1: $\mathbf{D}$ is of $*$-$W$--continuous type.

In this case, $\sigma$ is continuous at $\comp{2}{\beta}_{\mathbf{q}}$. For any $\vec{\beta} \in [\omega_1]^{Q \uparrow}$, $\sigma^W(\comp{2}{\beta}_{\mathbf{q}}) = \sup (\sigma^W \circ j^{P^{-},P})'' (\comp{2}{\beta}_{q^{-}})$. So $\cf^{\mathbb{L}_{\kappa_3}[T_2]}(\sigma^W(\comp{2}{\beta}_{\mathbf{q}})) = \cf^{\mathbb{L}_{\kappa_3}[T_2]}(\comp{2}{\beta}_{q^{-}}) $. Note that ${\seed}^{Q}_{(2,\mathbf{r})}$ is represented by the function $\vec{\beta} \mapsto \comp{2}{\beta}_{q^{-}}$. By \Los{}, $\cf^{\mathbb{L}_{\kappa_3}(j^Q(T_2))}(\seed^{Q,W}_{\mathbf{D}}) = \cf^{\mathbb{L}_{\kappa_3}(j^Q(T_2))}(\seed^{Q}_{(2,\mathbf{r})})$.

Subcase 1.2:  $\mathbf{D}$ is of $*$-$W$--discontinuous type.

Then $\pred_{\prec^W}(\sigma(p_{k-1}))$ exists and is not in $\ran( \sigma)$. Put $P^{-}  =  P \setminus \se{p_{k-1}}$. Let $\sigma'$ factor $(P,W)$ where $\sigma$ and $\sigma'$ agree on $P^{-}$ and $\sigma'(p_{k-1}) = \pred_{\prec^W}(\sigma(p_{k-1}))$. By Lemma~\ref{lem:level_2_uniform_cofinality}, $\sigma^W(\comp{2}{\beta}_{\mathbf{q}}) = (\sigma')^W_{\sup} (j^{P^{-}, P}(\comp{2}{\beta}_{q^{-}}))$. So \\
$\cf^{\mathbb{L}_{\kappa_3}[T_2]}(\sigma^W(\comp{2}{\beta}_{\mathbf{q}})) = \cf^{\mathbb{L}_{\kappa_3}[T_2]}(j^{P^{-}, P}(\comp{2}{\beta}_{q^{-}})) $. Note that ${\seed}^{Q}_{(2,\mathbf{r})}$ is represented by the function $\vec{\beta} \mapsto j^{P^{-},P}(\comp{2}{\beta}_{q^{-}})$. 
By \Los{}, $\cf^{\mathbb{L}_{\kappa_3}(j^Q(T_2))}(\seed^{Q,W}_{\mathbf{D}}) = \cf^{\mathbb{L}_{\kappa_3}(j^Q(T_2))}(\seed^{Q}_{(2,\mathbf{r})})$. 

Case 2: $q$ is of discontinuous type. 

Subcase 2.1: $\mathbf{D}$ is of $*$-$W$--continuous type. 

Then  $\sigma$ is continuous at $\comp{2}{\beta}_{\mathbf{q}}$. Proceed as in Subcase 1.1. 

Subcase 2.2: $\mathbf{D}$ is of $*$-$W$--discontinuous type. 

Then $\pred_{\prec^W}(\sigma(p_k^{-}))$ exists and is not in $\ran(\sigma)$. Put $P^{+} = P \cup \se{ p_k }$. 
Let $\sigma'$ factor $(P^{+},W)$ where $\sigma'\supseteq \sigma$ and $\sigma'(p_k)=\pred_{\prec^W}(\sigma(p_k^{-}))$. 
By Lemma~\ref{lem:level_2_uniform_cofinality_another}, 
$\sigma^W(\comp{2}{\beta}_{\mathbf{q}}) = (\sigma')^W_{\sup}(j^{P,P^{+}}(\comp{2}{\beta}_{\mathbf{q}}))$. Proceed as in Subcase 1.2. 

Note that by Lemma~\ref{lem:Q_desc_uniform_indis}, each $u_n$ ($1 \leq n<\omega$) is of the form $\seed^{Q,W}_{\mathbf{D}}$ for some finite $W$  and $\mathbf{D} \in \desc(Q,W)$. In summary, we have proved that every  $\admistwoboldextra{j^Q}$-regular cardinal must be of the form $\seed^Q_{(d,\mathbf{q})}$, where $(d,\mathbf{q})\in \desc^{*}(Q)$ is regular. 

Secondly, we prove that if  $(d,\mathbf{q})\in \desc^{*}(Q)$ is regular, then $\seed^Q_{(d,\mathbf{q})}$ is regular in $\admistwoboldextra{j^Q}$.  




Suppose towards a contradiction that $\cf^{\admistwoboldextra{j^Q}} (\seed^Q_{(d, \mathbf{q})})  =  \seed^Q_{(e, \mathbf{r})}$, where $(e, \mathbf{r}) \prec^Q (d, \mathbf{q})$. Let $g \in \admistwoboldextra{j^Q}$ be a cofinal map from $\seed^Q_{(e, \mathbf{r})}$ into $\seed^Q_{(d, \mathbf{q})}$.  Let $g = [G]_{\mu^Q}$, where $G \in \admistwobold$. By \Los{}, for $\mu^Q$-a.e.\ $\vec{\beta}$, $g(\vec{\beta}) \in \admistwobold$ is a cofinal map from $\comp{e}{\beta}_{\mathbf{r}}$ into $\comp{d}{\beta}_{\mathbf{q}}$. 

We prove the case when $d= e=2$, the other cases being similar. Put $\mathbf{r} = (r, Z, \vec{z})$. Let $Q_1$ be a level $\leq 2$ tree which extends $Q$ such that $\dom(Q') \setminus \dom(Q) = \se{(2, q')}$, and
\begin{enumerate}
\item if $(2, \mathbf{q}) \in \desc(Q)$, then $q' = q ^{-}\concat (a)$, $\emptyset = a \res \comp{2}{Q}\se{q^{-}}$, $Q'(q') = Q(q)$;
\item if $(2, \mathbf{q}) \notin \desc(Q)$, then $q' = q \concat (a)$, $\emptyset = a \res \comp{2}{Q}\se{q }$,  $Q'_{\tree}(q') = P$.
\end{enumerate}
 Let $Q_2$ be the level $\leq 2$ tree defined in a similar way with $(q,q')$ replaced by $(r,r')$. Let $Q'$ be the tree extending both $Q_1$ and $Q_2$ and $\dom(Q') = \dom(Q_1) \cup \dom(Q_2)$. 
Put $\mathbf{q}' = (q',P,\vec{p}) $, $\mathbf{r}' = (r', Z, \vec{z})$. So $\mathbf{q}' = \mathbf{q} \res Q_2$, $\mathbf{r}'  = \mathbf{r} \res Q_1$. 
We partition functions $f \in \omega_1^{Q' \uparrow}$ according to whether $g([f \res \rep(Q)]^Q) ( {}^2[f] ^{Q'}_{\mathbf{r}'}) < {}^2 [f]^{Q'}_{\mathbf{q}'}$. By the level-2 partition property of $\omega_1$ in \cite{sharpII}, we obtain a club $C \in \mu_{\mathbb{L}}$ which is homogeneous for this property. 
Let $\eta \in C'$ iff $\eta \in C$ and $C \cap \eta$ has order type $\eta$. Let $\eta \in C''$ iff $\eta \in C'$ and $C' \cap \eta$ has order type $\eta$. 
If the homogeneous side satisfies $g([f \res \rep(Q)]^Q) ( {}^2[f] ^Q_{\mathbf{r}'}) < {}^2[f]^Q_{\mathbf{q}'}$, then every function $f \in (C'')^{Q \uparrow}$ extends to some $f' \in (C')^{Q_1 \uparrow}$ by Lemma~\ref{lem:Q_respecting}, and $\set{{}^2[f'']^{Q'}_{\mathbf{r}'}}{\exists f'' \in C^{Q'\uparrow} (f' \subseteq f'')}$ is cofinal in ${}^2[f]^{Q}_{\mathbf{r}}$ by Lemma~\ref{lem:Q_desc_extension}. Hence,  $\sup (g([f \res \rep(Q)]^Q)) ''( {}^2[f] ^Q_{\mathbf{r}}) \leq  {}^2[f']^{Q_1}_{\mathbf{q}'}< {}^2[f]^{Q}_{\mathbf{q}}$, contradicting to our assumption. 
If the homogeneous side satisfies $g([f \res \rep(Q)]^Q) ( {}^2[f] ^Q_{\mathbf{r}'}) \geq [f]^Q_{\mathbf{q}'}$, a similar arguments yields  $\sup (g([f \res \rep(Q)]^Q)) ''( {}^2[f] ^Q_{\mathbf{r}}) > {}^2[f]^{Q}_{\mathbf{q}}$, contradiction again.
\end{proof}


It is easy to deduce the following corollary using \Los{}:
\begin{mycorollary}
  \label{coro:Q_uniform_cofinality}
  Suppose $\beta < \bolddelta{3}$ is a limit ordinal. Then $\beta$ has $Q$-uniform cofinality $(d,\mathbf{q})$ iff $\cf^{\admistwoboldextra{j^Q}}(\beta) = \seed^{Q}_{(d,\mathbf{q})}$. In particular, the $Q$-uniform cofinality of $\beta$ exists and is unique. 
\end{mycorollary}

If $\pi$ factors finite level $\leq 2$ trees $(Q,T)$, then $\pi^T(u_n) = u_m \to \pi^T_{\sup} (u_{n+1}) = u_{m+1} $. 
Therefore, the continuity of $\pi^T$ is decided by $\pi^T \res \set{u_n}{n<\omega}$. If $(d,\mathbf{q}) \in \exdesc(Q)$ is regular, $(\pi, T)$ is \emph{continuous at} $(d,\mathbf{q})$ 
iff one of the following holds:
\begin{enumerate}
\item $d=1$, either $\comp{1}{\pi}(\mathbf{q}) = \min (\prec^{\comp{1}{T}})$ or $\pred_{\prec^{\comp{1}{T}}} (\comp{1}{\pi}(q)) \in \ran(\comp{1}{\pi})$.
\item $d=2$, $\mathbf{q} = (\emptyset,\emptyset,((0)))$, either $\comp{1}{T} = \emptyset$ or $\max(\prec^{\comp{1}{T}}) \in \ran(\comp{1}{\pi})$.
\item $d=2$, $\mathbf{q} = (q,P,\vec{p}) \in \desc(\comp{2}{Q})$, $q \neq \emptyset$,  and letting $t' = \max_{<_{BK}} \comp{2}{T}\se{\comp{2}{\pi}(q), -}$, then either $t' = \comp{2}{\pi}(q^{-}) \concat (-1)$ or $t' \in \ran(\comp{2}{\pi})$.
\item $d=2$, $\mathbf{q} = (q,P,\vec{p}) \notin \desc(\comp{2}{Q})$, and letting $a = \max_{<_{BK}} (\comp{2}{T}\se{\comp{2}{\pi}(q)} \cup \se{-1})$, then either $a=-1$ or $\comp{2}{\pi}(q) \concat (a) \in \ran(\comp{2}{\pi})$.
\end{enumerate}
Thus, $(\pi, T)$ is continuous at $(d,\mathbf{q})$ iff $\pi^T$ is continuous at $\seed^Q_{(d,\mathbf{q})}$. We obtain the following lemma discussing the continuity behavior of $\pi^T$. It is the level-2 version of Lemma~\ref{lem:j_sigma_continuity}.

\begin{mylemma}
  \label{lem:level_2_factoring_continuity}
  Suppose $\pi$ factors finite level $\leq 2$ trees $(Q,T)$, $\gamma < \bolddelta{3}$ is a limit ordinal, $\cf^{\admistwoboldextra{j^Q}}(\gamma) = \seed^Q_{(d,\mathbf{q})}$, $(d,\mathbf{q}) \in \exdesc(Q)$ is regular. Then
  \begin{enumerate}
  \item $\pi^T(\gamma) = \pi^T_{\sup} (\gamma)$ iff $(\pi, T)$ is continuous at $(d,\mathbf{q})$.
  \item Suppose $(\pi, T)$ is not continuous at $(d,\mathbf{q})$. Let $Q^{+}$ be a level $\leq 2$ tree and let  $\pi^{+}$ factor $(Q^{+}, \pi)$ so that $Q^{+} $ extends $Q$, $\pi^{+} $ extends $\pi$, and 
    \begin{enumerate}
    \item if $d=1$, then $\dom(Q^{+}) \setminus \dom(Q) = \se{(1, \mathbf{q}^{+})}$, $\mathbf{q} = \mathbf{q}^{+} \res \comp{1}{Q} $,  
$\comp{1}{\pi}^{+} (\mathbf{q}^{+}) = \suc_{\prec^{\comp{1}{T}}} (\comp{1}{\pi}(\pred_{{\comp{1}{Q}}}(\mathbf{q})))$;
    \item if $d=2$ and $\mathbf{q} = (\emptyset,\emptyset, ((0)))$, then $\dom(Q^{+} ) \setminus \dom(Q)= \se{ (1, q^{+})}$, $\emptyset =q^{+} \res \comp{1}{Q}$, $\comp{1}{\pi}^{+}(q^{+}) = \min_{\prec^{\comp{1}{T}}} \set{a}{\forall r \in \dom(\comp{1}{Q})~\comp{1}{\pi}(r) \prec^{\comp{1}{T}} a}$;
    \item if $d=2$ and  $\mathbf{q} = (q,P,\vec{p}) \in \desc(\comp{2}{Q})$, then $\dom(Q^{+}) \setminus \dom(Q) = \se{(2, q^{+})}$, $q^{+} = \max_{<_{BK}} \comp{2}{Q}^{+}\se{q, -}$, and $\comp{2}{\pi}^{+}(q^{+}) = \comp{2}{\pi}(q^{-}) \concat (a)$, $a=
\min_{<_{BK}} 
\set{b}{ \forall r \in \comp{2}{Q}(q,-)\setminus \se{q^{-} \concat (-1)}~  \comp{2}{\pi}(r) <_{BK} \comp{2}{\pi}(q^{-}) \concat (b)}$;
    \item if $d=2$ and  $\mathbf{q} = (q,P,\vec{p}) \notin \desc(\comp{2}{Q})$, then $\dom(Q^{+}) \setminus \dom(Q) = \se{(2, q^{+})}$, $q^{+} = q \concat (
\max_{<_{BK}} \comp{2}{Q}^{+}\se{q})$, 
$\comp{2}{\pi}^{+}(q^{+}) = \comp{2}{\pi}(q) \concat (a)$, 
$a= \min_{<_{BK}} \set{b}{ \forall  c \in \comp{2}{Q}\se{q} ~ \comp{2}{\pi}(q \concat (c)) <_{BK} \comp{2}{\pi}(q) \concat (b)}$.
    \end{enumerate}
Then $\pi^T_{\sup}(\gamma) = (\pi^{+})^T \circ j^{Q,Q^{+}}_{\sup} (\gamma)$.
  \end{enumerate}
\end{mylemma}

If $\pi$ factors finite level $\leq 2$ trees $(Q,T)$ and $(\pi,T)$ is discontinuous at $(d, \mathbf{q})$, then $\pred(\pi, T, (d,\mathbf{q}))$ is a node in $\dom(T)$ defined as follows:
\begin{enumerate}
\item If $d=1$, then $\pred(\pi, T, (d,\mathbf{q})) = (1, \pred_{\prec^{\comp{1}{T}}}( \comp{1}{\pi}(\mathbf{q})))$.
\item If $d=2$ and $\mathbf{q} = (\emptyset,\emptyset, ((0)))$, then $\pred(\pi, T, (d,\mathbf{q})) = (1, \max_{<_{BK}}\comp{1}{T})$.
\item If $d=2$ and $\mathbf{q} = (q, P, \vec{p}) \in \desc(Q)$, $q \neq \emptyset$, then $\pred(\pi,{T},(d,\mathbf{q})) = (2, \max_{<_{BK}} \comp{2}{T}\se{\comp{2}{\pi}(q),-})$.
\item If $d=2$ and $\mathbf{q} = (q, P, \vec{p}) \notin \desc(Q)$, $q \neq \emptyset$, then $\pred(\pi, T, (d,\mathbf{q})) = (2, q \concat (a))$, $a = \max_{<_{BK}} \comp{2}{T}\se{\comp{2}{\pi}(q)}$.
\end{enumerate}
If $(2,\mathbf{q}) = (2, (q,P,\vec{p})) \in \desc(Q)$ then put $\pred(\pi,T, (2,q)) = \pred(\pi,T, (2, \mathbf{q}))$. The next lemma is the level-2 version of Lemma~\ref{lem:level_2_uniform_cofinality}, whose proof is similar. 

\begin{mylemma}
  \label{lem:continuous_decomp_level_2}
Suppose $(Q^{-}, (d,q,P))$ is a partial level $\leq 2$ tree, $T$ is a finite level $\leq 2$ tree, $Q$ is a completion of $(Q^{-}, (d,q,P))$, $\pi$ factors $(Q,T)$, and  $(\pi,T)$ is discontinuous at $(d,q)$. Let $\tau$ factor $(Q,T)$ where $\tau$ and $\pi$ agree on $\dom(Q) \setminus \se{(d,q)}$, $\tau(d,q) = \pred(\pi,T, (d,q))$. Suppose $\cf^{\admistwoboldextra{j^{Q^{-}}}}(\gamma) = \seed^{Q^{-}}_{\ucf(Q^{-},(d,q,P))}$. Then
  \begin{displaymath}
    \pi^T \circ j^{Q^{-},Q}_{\sup} (\gamma) = \tau^T_{\sup} \circ j^{Q^{-},Q} (\gamma) .
  \end{displaymath}
\end{mylemma}

The level-2 version of Lemma~\ref{lem:level_2_uniform_cofinality_another} is similarly proved. 

\begin{mylemma}
  \label{lem:continuous_decomp_level_2_another}
Suppose $(Q, (d,q,P))$ is a partial level $\leq 2$ tree, $\ucf(Q,(d,q,P)) = (d^{*}, \mathbf{q}^{*})$,
 $T$ is a finite level $\leq 2$ tree, $\pi$ factors $(Q,T)$, and  $(\pi,T)$ is discontinuous at $(d^{*},\mathbf{q}^{*})$. Let 
  $Q^{+}$ be a completion of $(Q, (d,q,P))$ and let
$\tau$ factor $(Q^{+},T)$ so that $\tau$ extends $\pi$,  $\tau(d,q) = \pred(\pi,T, (d^{*},\mathbf{q}^{*}))$. Suppose $\cf^{\admistwoboldextra{j^{Q}}}(\gamma) = \seed^{Q}_{(d^{*},\mathbf{q}^{*})}$. Then
  \begin{displaymath}
    \pi^T (\gamma) = \tau^T_{\sup} \circ j^{Q,Q^{+}} (\gamma) .
  \end{displaymath}
\end{mylemma}
Note that in Lemma~\ref{lem:continuous_decomp_level_2_another}, the completion $Q^{+}$ is decided by $\pred(\pi,T, (d^{*}, \mathbf{q}^{*}))$. There is no freedom in choosing $Q^{+}$.

In the same spirit as Lemmas~\ref{lem:respect_lv2_measure_one_set}-\ref{lem:unique_level_2_tree_represent} we will obtain a concrete way of deciding whether a tuple $\vec{\gamma}$ respects a level-3 tree $R$.

\begin{mydefinition}
  \label{def:level_2_function_sign_etc}
  Suppose $Q$ is a finite level $\leq 2$ tree, $\overrightarrow{(d,q)} = ((d_i, q_i))_{1 \leq i < k}$ is a distinct enumeration of a subset of $Q$ and such that $\set{q_i}{d_i = 2}\cup \se{\emptyset}$ forms a tree on $\omega^{<\omega}$. Suppose $F : [\omega_1]^{Q \uparrow} \to \bolddelta{3}$ is a function which lies is $\admistwobold$.
  \begin{enumerate}[series=F]
  \item The \emph{signature} of $F$ is $\overrightarrow{(d,q)}$ iff there is $E \in \mu_{\mathbb{L}}$ such that
    \begin{enumerate}
    \item for any $\vec{\beta}, \vec{\gamma} \in [E]^{Q \uparrow}$, if $(\comp{d_0}{\gamma}_{q_0}, \ldots, \comp{d_{k-1}}{\gamma}_{q_{k-1}}) <_{BK} (\comp{d_0}{\beta}_{q_0}, \ldots, \comp{d_{k-1}}{\beta}_{q_{k-1}})$ then $f(\vec{\beta}) < f(\vec{\gamma})$;
    \item for any $\vec{\beta}, \vec{\gamma} \in [E]^{Q \uparrow}$, if $(\comp{d_0}{\gamma}_{q_0}, \ldots, \comp{d_{k-1}}{\gamma}_{q_{k-1}}) = (\comp{d_0}{\beta}_{q_0}, \ldots, \comp{d_{k-1}}{\beta}_{q_{k-1}})$ then $f(\vec{\beta}) = f(\vec{\gamma})$.
    \end{enumerate}
  \end{enumerate}
Clearly the signature of $F$ exists and is unique. In particular, $F$ is constant on a $\mu^Q$-measure one set iff the signature of $F$ is $\emptyset$. 
Suppose the signature of $F$ is  $\overrightarrow{(d,q)} = ((d_i, q_i))_{1 \leq i < k}$.
\begin{enumerate}[resume*=F]
\item 
$F$ is \emph{essentially continuous} iff for $\mu^Q$-a.e.\ $\vec{\beta}$, $F(\vec{\beta}) = \sup\set{F(\vec{\gamma})}{\vec{\gamma} \in [\omega_1]^{Q \uparrow}, (\comp{d_0}{\gamma}_{q_0}, \ldots, \comp{d_{k-1}}{\gamma}_{q_{k-1}}) <_{BK} (\comp{d_0}{\beta}_{q_0}, \ldots, \comp{d_{k-1}}{\beta}_{q_{k-1}})}$. Otherwise, $F$ is \emph{essentially discontinuous}.
\end{enumerate}
Put $[\omega_1]^{Q \uparrow (0,-1)} =  [\omega_1]^{Q \uparrow} \times \omega $. For 
 $(d,\mathbf{q}) \in \exdesc(Q)$ regular, put $[\omega_1]^{Q \uparrow(d, \mathbf{q})}=
 \set{(\vec{\beta},\gamma)}{\vec{\beta} \in [\omega_1]^{Q \uparrow} , \gamma < \comp{d}{\beta}_{\mathbf{q}}}$.
 \begin{enumerate}[resume*=F]
 \item 
For $(d,\mathbf{q}) $ either $(0,-1)$ or in $ \exdesc(Q)$ regular, 
say that the \emph{uniform cofinality} of $F$ is $\ucf(F)  =  (d,\mathbf{q})$ iff 
 there is  $G : [\omega_1]^{Q\uparrow (d,\mathbf{q}) } \to \bolddelta{3}$ such that $G \in \admistwobold$ and for any for $\mu^Q$-a.e.\ $\vec{\beta}$,   $F(\vec{\beta}) = \sup \{G(\vec{\beta}, \gamma) :  (\vec{\beta}, \gamma) \in [\omega_1]^{Q \uparrow (d,\mathbf{q})}\}$ and the function $\gamma \mapsto G(\vec{\beta}, \gamma)$ is order preserving.
\item The \emph{cofinality} of $F$ is
 \begin{displaymath}
   \cf(F) =
   \begin{cases}
     0 &\text{if } \ucf(F) = (0,-1), \\
     1 &\text{if }\ucf(F) = (1, q), q = \min(\prec^{\comp{1}{Q}}), \\
     2 &\text{otherwise.}
   \end{cases}
 \end{displaymath}
\end{enumerate}
Let $(X_i, (d_i,x_i, W_i)) \concat (X_k)$ be the partial level $\leq 2$ tower of continuous type and let $\pi$ factor $(X_k, Q)$ such that $\pi(d_i, x_i) = (d_i, q_i)$ for each $1 \leq i < k$.
\begin{enumerate}[resume*=F]
\item The \emph{potential partial level $\leq 2$ tower} induced by $F$ is
  \begin{enumerate}
  \item $(X_k, (d_i,x_i, W_i)_{1 \leq i < k})$, if $F$ is essentially continuous;
  \item $(X_k, (d_i,x_i,W_i)_{1 \leq i < k} \concat (0,-1,\emptyset))$, if $F$ is essentially discontinuous and has uniform cofinality $(0,-1)$;
  \item $(X_k, (d_i,x_i,W_i)_{1 \leq i < k} \concat (1, x^{+},\emptyset))$, if $F$ is essentially discontinuous and has uniform cofinality $(1,q_{*})$, $(X_k, (1, x^{+}, \emptyset))$ is a partial level $\leq 2$ tree, $\pi(1, (x^{+})^{-}) = (1, q_{*})$;
  \item $(X_k, (d_i,x_i,W_i)_{1 \leq i < k}\concat (2,x^{+},P_{*}))$, if $F$ is essentially discontinuous and has uniform cofinality $(2, \mathbf{q}_{*})$, $\mathbf{q}_{*} = (q_{*}, P_{*}, \vec{p}_{*})$, $(X_k, (2, x^{+}, P_{*}))$ is a partial level $\leq 2$ tree with uniform cofinality $(d_{*},\mathbf{x}_{*})$ and $\comp{2}{\pi}(\mathbf{x}_{*}) = \mathbf{q}_{*}$.
  \end{enumerate}
\item   The \emph{approximation sequence} of $F$ is $( F_i)_{1 \leq i \leq k}$ where $F_i $ is a function on $[\omega_1]^{X_i\uparrow}$, $F_i(\vec{\beta}) = \sup \{F(\vec{\gamma}):$ $ \vec{\gamma} \in [\omega_1]^{Q \uparrow}, (\comp{d_1}{\gamma}_{q_1}, \ldots, \comp{d_{i-1}}{\gamma}_{q_{i-1}}) = (\comp{d_1}{\beta}_{x_1}, \ldots, \comp{d_{i-1}}{\beta}_{x_{i-1}}) \}$ for $1 \leq i \leq k$.
\end{enumerate}
This is the end of Definition~\ref{def:level_2_function_sign_etc}.
\end{mydefinition}

The existence and uniqueness of the uniform cofinality of $F$ will be proved in Section~\ref{sec:approximation-LT3-LT2}. In particular, if $R$ is a level-3 tree, $H \in (\bolddelta{3})^{R \uparrow}$, then for any $r \in \dom(R)$, $H_r$ has signature $(R_{\node}(r \res i))_{1 \leq i < \lh(r)}$, 
is essentially discontinuous,  has uniform cofinality $\ucf(R(r))$ and cofinality $\cf(R(r))$,  induces the potential partial level $\leq 2$ tower $R[r]$, and $(H_{r \res i})_{1 \leq i \leq \lh(r)}$ is the approximation sequence of $H_r$. Again, all the relevant properties of $F$ depends only on the value of $F$ on a $\mu^Q$-measure one set. We will thus be free to say the signature, etc.\ of $F$ when $F$ is defined on a $\mu^Q$-measure one set.

\begin{mydefinition}
  \label{def:ordinal_long_division_level_3}
  Suppose $\omega_1 \leq \gamma < \bolddelta{3}$ is a limit ordinal. Suppose $Q$ is a finite level $\leq 2$ tree,  $\gamma = [F]_{\mu^Q}$,  the signature of $F$ is $((d_i,q_i))_{1 \leq i < k}$, the approximation sequence of $F$ is $(F_i)_{1 \leq i \leq k}$. Then:
  \begin{enumerate}
  \item The \emph{$Q$-signature} of $\beta$ is $({(d_i,q_i)})_{1 \leq i < k}$.
  \item The \emph{$Q$-approximation sequence} of $\gamma$ is $([F_i]_{\mu^Q})_{1 \leq i \leq k}$.
  \item $\gamma$ is \emph{$Q$-essentially continuous} iff $F$ is essentially continuous.
  \item The \emph{$Q$-uniform cofinality} of $\gamma$ is $\omega$ if $F$ has uniform cofinality $(0,-1)$, $\seed^Q_{(d,\mathbf{q})}$ if $f$ has uniform cofinality $(d,\mathbf{q}) \in \exdesc(Q)$.
  \item The \emph{$Q$-potential partial level $\leq 2$ tower induced by $\gamma$} is the potential partial level $\leq 2$ tower induced by $F$.
  \end{enumerate}

\end{mydefinition}

In Section~\ref{sec:approximation-LT3-LT2}, we will show that all the relevant properties in Definition~\ref{def:ordinal_long_division_level_3} are  independent of the choice of $F$ (but depends on $Q$ of course). We will also show that the $Q$-uniform cofinality of $\gamma$ is exactly $\cf^{\admistwoboldextra{j^Q}}(\gamma)$, and we will show that $\cf^{\admistwobold}(\gamma) = u_{\cf(F)}$, where we set $u_0 = \omega$.

\begin{mydefinition}
  \label{def:typical_level_3_tree}
We fix the notations for all the level-3 trees of cardinality $1$. 
 For $d \in \se{0,1,2}$, $\dom(R^{d}) = \se{((0))}$ and $R^d ((0))$ is of degree $d$.
\end{mydefinition}
Therefore, the order type of $\rep(R^d)$ is $u_d$, where $u_0 = \omega$.

Recall that a \emph{partial level $\leq 2$ tower of discontinuous type} is a nonempty 
finite sequence $(Q_i,(d_i,q_i,P_i))_{1 \leq i \leq k}$ such that $\card(Q_1)=1 $, each $(Q_i,(d_i,q_i,P_i))$ is a partial level $\leq 2$ tree, and each $Q_{i+1}$ is a completion of $(Q_i,(d_i,q_i,P_i))$. We say that: Its \emph{signature} is $(d_i,q_i)_{1 \leq i < k}$. 
Its \emph{uniform cofinality} is $\ucf(Q_k, (d_k,q_k, P_k))$. 
Recall that a \emph{partial level $\leq 2$ tower of continuous type} is $(Q_i,(d_i,q_i,P_i))_{1 \leq i < k} \concat (Q_{*})$ such that either $k=0 \wedge Q_{*}$ is the level $\leq 2$ tree of cardinality 1 or $(Q_i,(d_i,q_i,P_i))_{1 \leq i < k}$ is a partial level $\leq 2$  tower of discontinuous type $\wedge Q_{*}$ is a completion of $(Q_{k-1}, (d_{k-1},q_{k-1},P_{k-1}))$. 
Its \emph{signature} is $(d_i,q_i)_{1 \leq i < k}$. If $k>0$, its \emph{uniform cofinality} is $(1, q_{k-1})$ if $d_{k-1}=1$,  $(2, (q_{k-1}, P, \vec{p}))$ if $d_{k-1} = 2$ and $\comp{2}{Q}[q_{k-1}] = (P,\vec{p})$. 
For notational convenience, the information of a partial level $\leq 2$ tower is compressed into a potential partial level $\leq 2$ tower. 
A \emph{potential partial level $\leq 2$ tower} is $(Q_{*},\overrightarrow{(d,q,P)}) = (Q_{*}, (d_i, q_i, P_i)_{1 \leq i \leq \lh(\vec{q})})$ such that for some level $\leq 2$ tower $\vec{Q} = (Q_i)_{1 \leq i \leq k}$, either $Q_{*} = Q_k$ $\wedge$ $(\vec{Q}, \overrightarrow{(d,q,P)})$ is a partial level $\leq 2$ tower of discontinuous type or $(\vec{Q}, \overrightarrow{(d,q,P)}) \concat (Q_{*})$ is a partial level $\leq 2$ tower of continuous type. The signature, (dis-)continuity type, uniform cofinality of $(Q_{*},\overrightarrow{(d,q,P)})$ are defined according to the partial level $\leq 2$ tree generating $(Q_{*}, \overrightarrow{(d,q,P)})$. 
\begin{displaymath}
  \ucf(Q_{*}, \overrightarrow{(d,q,P)})
\end{displaymath}
denotes the uniform cofinality of $(Q_{*}, \overrightarrow{(d,q,P)})$. Clearly,  a potential partial level $\leq 2$ tower $(Q_{*},\overrightarrow{(d,q,P)})$ is of continuous type iff $\card(Q_{*}) = \lh(\vec{q})$, of discontinuous type iff $\card(Q_{*}) = \lh(\vec{q})-1$.

The properties of a tuple  respecting $R$ is decided by the signature, approximation sequence and relative ordering of its entries, in a parallel way to the level-2 case. It is handled in \cite{jackson_delta15}. We state the results in our language. 

\begin{mylemma}
  \label{lem:respect_lv3_measure_one_set}
  Suppose $R$ is a finite level-3 tree, $B \in \admistwobold$ is a closed set of ordinals. Then $\vec{\gamma} \in [B]^{R \uparrow}$ iff there is $F\in (\bolddelta{3})^{R \uparrow}$ and $E \in \mu_{\mathbb{L}}$ such that $\vec{\gamma} = [F]^R$ and for any $r \in \dom(R)$, for any $\vec{\beta} \in [E]^{R(r)\uparrow}$, $F_r (\vec{\beta})$ is a limit point of $B$.
\end{mylemma}
\begin{proof}
The nontrivial direction is $\Leftarrow$.  Suppose $F \in (\bolddelta{3})^{R \uparrow}$ and  $E \in \mu_{\mathbb{L}}$ are as given. For $r \in \dom(R)$, let $R(r) = (Q_r, (d_r,q_r,P_r))$, and let $r^{*}$ be the $<_{BK}$-greatest member of $R\se{r,-}$.  In parallel to Claim~\ref{claim:respect_lv2_distance}, by Theorem~\ref{thm:jQ_LT2_regularity} and cofinality considerations in $\admistwoboldextra{j^{Q_r}}$, we have 
\begin{myclaim}
  \label{claim:respect_lv3_distance}
  There is $E' \in \mu_{\mathbb{L}}$ such that $E' \subseteq E$ and for any $r \in \dom(R)$, 
for any $\vec{\beta} \in [E']^{P_q\uparrow}$,
\begin{enumerate}
\item if $d_r = 1$ then $B \cap (F_{r^{*}}(\vec{\beta}), F_r (\vec{\beta}))$ has order type $\comp{1}{\beta}_{q_r^{-}}$;
\item if $d_r = 2$, $\ucf(R[r] )= (2,   \mathbf{q}_{r,*})$, $ \mathbf{q}_{r,*}=(q_{r,*}, P_{r,*},\vec{p}_{r,*})$, then $B \cap(F_{r^{*}}(\vec{\beta}), F_r (\vec{\beta})) $ has order type $\comp{2}{\beta}_{q_{r,*}}$. 
\end{enumerate}
\end{myclaim}
The rest proceeds as in the proof of Lemma~\ref{lem:respect_lv2_measure_one_set}. 
\end{proof}

\begin{mylemma}
  \label{lem:level_3_R_respecting_function_sort}
  Suppose $R$ is a finite level-3 tree, $R[r] = (Q_r, (d_r,q_r,P_r))$ for $r \in \dom(R)$, $E \in \mu_{\mathbb{L}}$ is a club. Suppose $f : \rep(R) \res E \to \bolddelta{3}$ satisfies
  \begin{enumerate}
  \item if $r \in \dom(R)$, then the $Q_r$-potential partial level $\leq 2$ tower induced by $F_r$ is $R[r]$, the approximation sequence of $F_r$ is $(F_{r \res i})_{1 \leq i \leq \lh(q)}$,
and the uniform cofinality of $F_r$ on $[E]^{Q_r\uparrow}$ is witnessed by $F_{r \concat (-1)}$, i.e., if $\vec{\beta} \in [E]^{Q_r \uparrow}$, then $F_r(\vec{\beta}) = \sup\{ F_{r \concat (-1)}(\vec{\beta} \concat (\gamma)) : \vec{\beta} \concat (\gamma) \in \rep(R) \res E \}$, and the map $\vec{\beta} \mapsto F_{r \concat (-1)}(\vec{\beta} \concat (\gamma))$ is continuous, order preserving.
  \item if $R_{\tree}(r \concat (a)) = R_{\tree}( r \concat (b))$, and $a < _{BK} b$, then $[F_{r \concat (a)}]_{\mu^{Q_{r \concat (a)}}} < [F_{r \concat (b)}]_{\mu^{Q_{r \concat (b)}}}$. 
  \end{enumerate}
Then there is $E' \in \mu_{\mathbb{L}}$ such that $E' \subseteq E$ and $f \res (\rep(R) \res E')$ is order preserving.
\end{mylemma}
\begin{proof}
Put $\ucf(R[r]) = (d_r, \mathbf{q}_{r,*})$, and if $d_r=2$ then $\mathbf{q}_{r,*} = (q_{r,*}, P_{r,*},\vec{p}_{r,*})$.
  We know  by assumption that  for $\mu^{Q_r}$-a.e.\ $\vec{\beta}$, $F_r ( \vec{\beta}) = \sup \{F_{r \concat (a)}(\vec{\beta} \concat (\gamma)) :  r \concat (a) \in \dom(R) ,  \gamma < \comp{d_r}{\beta}_{q_{r,*}}\}$. Fix for the moment $r$ such that $d_r \neq 0$. Similarly to the proof of Lemma~\ref{lem:level_2_Q_respecting_function_sort}, we need a club $E' \in \mu_{\mathbb{L}}$ such that for any $\vec{\beta} \in [E']^{Q_r\uparrow}$, for any $\gamma< \gamma'$ both in $j^{P_r}(E')$, if $R_{\tree}(r \concat (a)) = R _{\tree} (r \concat (b) )$ then $F_{r \concat (a)}(\vec{\beta} \concat (\gamma))< F_{r \concat (b)} (\beta \concat (\gamma'))$. 

If $\vec{\beta}$ respects $Q_r$ and  $\vec{\beta}\concat (\gamma)$ respects $R(r)$, let $g(\vec{\beta}\concat (\gamma))$ be the least $\gamma'$ satisfying that whenever $r \concat (a) , r \concat (b)\in \dom(R)$, $\delta\leq \gamma$, $\delta' \geq \gamma'$, $\vec{\beta} \concat (\delta)$ respects $R(r \concat (a))$, $\vec{\beta} \concat (\delta')$ respects $R(r \concat (b))$,  we have $F_{r \concat (a)}(\vec{\beta} \concat (\delta)) < F_{r \concat (b)}(\vec{\beta} \concat (\delta'))$. 
If  $Q^{+}$ is a completion of $R(r)$, then for $\mu^{Q^{+}}$-a.e.\ $\vec{\xi}$,  $g(\vec{\xi}) < \comp{d_r}{\xi}_{\mathbf{q}_{r,*}}$. 
By Lemma~\ref{lem:level_2_desc_cofinal_in_next}, there is $h^{Q^{+}} : \omega_1 \to \omega_1$ and $E^{Q^{+}} \in \mu_{\mathbb{L}}$ such that $h \in \mathbb{L}$ and for any $\vec{\xi} \in [E^{Q^{+}}]^{Q^{+}\uparrow}$, $g(\vec{\xi}) < j^{P_r}(h^{Q^{+}}) (\comp{d_r}{\xi}_{q_r}) $. There are only finitely many completions of $R(r)$. Let $h : \omega_1 \to \omega_1$ where $h(\alpha) = \sup \{h^{Q^{+}}(\alpha) : Q^{+}$ is a completion of $R(r)\}$. Let $E' =  \cap \{E^{Q^{+}}: Q^{+}$ is a completion of $R(r)\}$. Let $\eta \in E'' $ iff $h''(\eta \cap E')\subseteq \eta$. $E''$ is as desired. 
\end{proof}


As corollaries of Lemmas~\ref{lem:respect_lv3_measure_one_set} and~\ref{lem:level_3_R_respecting_function_sort}, we obtain:
\begin{mylemma}
  \label{lem:R_respect}
  Suppose that $R$ is a level-3 tree and $\vec{\gamma} = (\gamma_r)_{r \in \dom(R)}$ is a tuple of ordinals in $\bolddelta{3}$. Then $\vec{\gamma}$ respects $R$ iff the following holds:
  \begin{enumerate}
  \item For any $r \in \dom(R)$, the $R_{\tree}(r)$-potential partial level $\leq 2$ tower induced by $\gamma_r$ is $R[r]$, and the $R_{\tree}(r)$-approximation sequence of $\gamma_r$ is $(\gamma_{r \res l})_{1 \leq  l \leq  \lh(r)}$.
    \item If $R_{\tree}(r \concat (a)) = R_{\tree} (r \concat (b))$ and $a<_{BK}b$ then $\gamma_{r \concat (a)} <\gamma_{r \concat (b)}$.
  \end{enumerate}
Moreover, if $B \in \admistwobold$ is a closed set, $B'$ is the set of limit points of $B$, then $\vec{\gamma} \in [B]^{R \uparrow}$ iff $\vec{\gamma}$ respects $R$ and for each $r \in \dom(R)$, $\gamma_r \in j^{R_{\tree}(r)}(B')$.
\end{mylemma}
In particular, if $\vec{\gamma}$ respects $R$, then $\cf^{\admistwobold}(\gamma) = u_{\cf(R(r))}$ for $r \in \dom(R)$, where $u_0 = \omega$.

\begin{mylemma}
  \label{lem:unique_level_3_tree_represent}
  Suppose $R$ and $R'$ are  level-3 trees  with the same domain. Suppose $\vec{\gamma}$ respects both $R$ and $R'$. Then $R = R'$. 
\end{mylemma}


\subsection{Factoring maps between level-2 trees}
\label{sec:more-level-2}

\begin{mydefinition}
  \label{def:contraction}
  Suppose $I<\omega$. Suppose for each $i < I$, $\bar{J}_i \leq J_i < \omega$ and $A_i = (a_{i,j})_{\bar{J}_i \leq j < J_i}$ is a finite sequence of sets. Then the \emph{contraction} of $(A_i)_{i < I}$ is $(b_k)_{k < K}$ such that
  \begin{enumerate}
  \item $\set{a_{i,j}}{i<I, \bar{J}_i\leq j<J_i} = \set{ b_k}{ k < K}$.
  \item For each $k < K$, letting $(i^k,j^k)$ be the $<_{BK}$-least $(i,j)$ such that $a_{i,j} = b_k$, then the map $k \mapsto (i^k, j^k)$ is order preserving with respect to $<$ and $<_{BK}$.
  \end{enumerate}
\end{mydefinition}

\begin{mydefinition}
  \label{def:description_TQW}
  Suppose $T,Q$ are level $\leq 2$ trees. 
A \emph{$(T,Q,-1)$-description} is just a $(T, \comp{1}{Q})$-description.
 Suppose  $(W,\vec{w}) $ is a potential partial level $\leq 1$ tower of discontinuous type, $\vec{w} = (w_i)_{i \leq m}$.  If $m=0$, the only  \emph{$(T,Q,(W,\vec{w}))$-description} is $(2, ( (\emptyset,\emptyset,((0))), \tau))$, where $\tau$ factors $(\emptyset,Q,\emptyset)$, which is called \emph{the constant $(T,Q, *)$-description}. If $m>0$, a \emph{$(T,Q,({W},\vec{w}))$-description} is of the form
  \begin{displaymath}
    \mathbf{C} =(2,(\mathbf{t},\tau))
  \end{displaymath}
such that
\begin{enumerate}
\item $\mathbf{t}\in \desc(\comp{2}{T})$ and $\mathbf{t} \neq (\emptyset,\emptyset,(0))$. Let $\mathbf{t} = (t,S, \vec{s})$, $\lh(t)=k$, $\vec{s} = (s_i)_{i < \lh(\vec{s})}$.
\item $\tau$ factors $(S, Q, W)$.
\item The contraction of $(\sign(\tau(s_i)))_{i< k}$ is $(w_i)_{i < m}$.
\item If $t$ is of continuous type and $w_{m-1}$ does not appear in the contraction of $(\sign(\tau(s_i)))_{i < k-1} \concat (\sign(\tau(s_{k-1})^{-}))$ then $\tau(s_{k-1})$ is of discontinuous type. 
\item Either $\ucf(S,\vec{s}) = w_m = -1$ or $\ucf_1(\tau(\ucf(S,\vec{s}))) =  \ucf(W,\vec{w})$. 
\end{enumerate}

  A $(T,Q,*)$-description is either a $(T,Q,-1)$-description or a
$(T,Q,({W}',\vec{w}'))$-description for some potential partial level $\leq 1$ tower $(W',\vec{w}')$ of discontinuous type.
For a level-1 tree $W$, a $(T,Q, W)$-description is a $(T,Q,({W},\vec{w}'))$-description for some $\vec{w}'$.  $\desc(T,Q,-1)$, $\desc(T,Q,(W, \vec{w}))$,   $\desc(T,Q,*)$, $\desc(T,Q,W)$ denote the sets of relevant descriptions. 
 We sometimes abbreviate $(d, \mathbf{t}, \tau) $ for  $(d, (\mathbf{t}, \tau)) \in \desc(T,Q, *)$ without confusion. 

Recalling our notation of $Q \otimes W$, we may regard $\desc(T,Q,-1)\subseteq \desc(T, \emptyset)$, $\desc(T,Q,W) \subseteq \desc(T,Q \otimes W)$. $T \otimes (Q \otimes W)$ is also a ``level-1 tree'', whose nodes consist of non-constant $(T, Q \otimes W)$-descriptions, so that $\desc(T,Q \otimes W) = \desc(T \otimes (Q \otimes W))$. 
Every non-constant $(T,Q,W)$-description is a member of $T \otimes (Q \otimes W)$. The constant $(T,Q,*)$-description $\mathbf{C}_0$ is regarded as the constant $T \otimes (Q \otimes W)$-description, to make sense of $\seed^{T \otimes (Q \otimes W)}_{\mathbf{C}_0}$. 
The \emph{degree} of $(d,\mathbf{t},\tau) \in \desc(T,Q,*)$ is $d$. 
 In fact, if $\mathbf{C} \in \desc(T,Q,*)$ is of degree 2, then there is a unique potential partial level $\leq 1$ tower $(W,\vec{w})$ for which
$\mathbf{C} \in \desc(T,Q,({W},\vec{w}))$. 

Suppose now  $\mathbf{C} = (d, \mathbf{t}, \tau) \in \desc(T,Q, *)$, and if $d=2$, then $\mathbf{C} \in \desc(T, Q, (W, \vec{w}))$, $\mathbf{t} = (t, S, \vec{s})$, $\lh(t) = k$, $\vec{s} = (s_i)_{i < \lh(\vec{s})}$, $\vec{w} = (w_i)_{i \leq m}$. 
If $d=1$ and the signature of $\mathbf{C}$ regarded as a $(T, \comp{1}{Q})$ is $(q_i)_{i < l}$, then the \emph{signature} of $\mathbf{C}$ is $((1,q_i))_{i < l}$. 
If $d=2$, the  \emph{signature} of $\mathbf{C}$ is
\begin{displaymath}
\sign(\mathbf{C}) =
\text{the contraction of } (\sign_{*}^W(\tau(s_i)))_{i < k} .
\end{displaymath}
$\mathbf{C}$ is \emph{of continuous type} iff $d=2$, $t$ is of continuous type, and $\tau(s_{k-1})$ is of  $*$-$W$-continuous type; $\mathbf{C}$ is \emph{of discontinuous type} otherwise. 
The \emph{uniform cofinality} of $\mathbf{C}=(d,\mathbf{t},\tau)$ is
\begin{displaymath}
  \ucf(\mathbf{C})
\end{displaymath}
defined as follows. If $ d = 1 $ and $\ucf_{\mathbf{C}} = q_*$ regarding $\mathbf{C}$  as a $(T,\comp{1}{Q})$-description, then $\ucf_{\mathbf{C}} = (0,-1)$ when $q_{*} = -1$, $\ucf_{\mathbf{C}} = (1,q_{*})$ otherwise. 
 If $ d = 2$ then 
 \begin{enumerate}
 \item if $\ucf(S,\vec{s}) = -1$,  then $\ucf (\mathbf{C}) = (0,-1)$;
   \item if  $\ucf(S,\vec{s}) = s_{*} \neq -1$,  then $\ucf(\mathbf{C}) = \ucf_{*}^W(\tau(s_{*}))$.
   \end{enumerate}
If $w_m\neq -1$, $\mathbf{C}$ is said to be of \emph{plus-discontinuous type}, and put
\begin{displaymath}
  \ucf^{+}(\mathbf{C}) = \ucf_{*}^{W^{+}}(\tau(\ucf(S,\vec{s}))),
\end{displaymath}
where $W^{+}$ is the completion of $(W, w_m)$. 
The \emph{$*$-signature} of  $\mathbf{C}$ is
\begin{displaymath}
  \sign_{*}(\mathbf{C}) =
  \begin{cases}
    ((1, \mathbf{t})) & \text{ if } d=1,\\
    ((2, t \res i))_{1 \leq i \leq k-1} & \text{ if } d=2, t \text{ is of continuous type,}\\
    ((2, t \res i))_{1 \leq i \leq k} & \text{ if } d=2, t \text{ is of discontinuous type.}\\    
  \end{cases}
\end{displaymath}
If $d = 1$, $\mathbf{C}$ is \emph{of $*$-$Q$-continuous type} iff $\mathbf{C}$ is of $*$-$\comp{1}{Q}$-continuous type regarded as a $(T, \comp{1}{Q})$-description.
If $d = 2$, $\mathbf{C}$ is \emph{of $*$-$Q$-continuous type} iff $\ucf(S,\vec{s}) \neq -1\wedge \tau(\ucf(S,\vec{s})) \neq \min(\prec^{Q,W}))$ implies $\pred_{\prec^{Q,W}}(\tau(\ucf(S,\vec{s}))) \in \ran( \tau )$.  $\mathbf{C}$ is  \emph{of $*$-$Q$-discontinuous type} iff $\mathbf{C}$ is not of $*$-$Q$-continuous type. 
The \emph{$*$-$Q$-uniform cofinality} of $\mathbf{C}$ is
\begin{displaymath}
  \ucf^Q_{*}(\mathbf{C})
\end{displaymath}
defined as follows. If $d=1$, then $\ucf^Q_{*}(\mathbf{C})$ is defined just by treating $\mathbf{C}$ as a $(T, \comp{1}{Q})$-description. If $d=2$, $t$ is of continuous type,
\begin{enumerate}
\item if $\mathbf{C}$ is of $*$-$Q$-continuous type, then $\ucf^Q_{*}(\mathbf{C}) = (2, (t^{-}, S \setminus \se{s_{k-1}}, \vec{s}))$; 
\item if $\mathbf{C}$ is of $*$-$Q$-discontinuous type, then $\ucf^Q_{*}(\mathbf{C}) = (2, (t^{-}, S, \vec{s}))$.
\end{enumerate}
If $d=2$, $t$ is of discontinuous type,
\begin{enumerate}
\item if $\mathbf{C}$ is of $*$-$Q$-continuous type, then $\ucf^Q_{*}(\mathbf{C}) = (2, \mathbf{t})$; 
\item if $\mathbf{C}$ is of $*$-$Q$-discontinuous type, then $\ucf^Q_{*}(\mathbf{C}) = (2, (t, S\cup \se{s_k}, \vec{s}))$.
\end{enumerate}
\end{mydefinition}

For $h  \in \omega_1^{T \uparrow} $, if $\mathbf{C} = (1,\mathbf{t},\emptyset)$ is a $(T,Q,-1)$-description,  then
\begin{displaymath}
  h^{Q}_{\mathbf{C}} : [\omega_1]^{Q \uparrow} \to \omega_1+1
\end{displaymath}
is defined simply by regarding $\mathbf{C}$ as a $(T, \comp{1}{Q})$-description; if $W$ is a (possibly empty) level-1 tree, $\mathbf{C} = (2, \mathbf{t},\tau)$  is a 
 $(T,Q,W)$-description, $\mathbf{t} = (t,S, \vec{s})$, 
then
\begin{displaymath}
  h_{\mathbf{C}} ^{Q}: [\omega_1]^{Q \uparrow }\to j^{W}(\omega_1).
\end{displaymath}
is the function that sends $[g]^Q$ to $[\comp{2}{h}_{\mathbf{t}} \circ g^W_{\tau}]_{\mu^W}$. Note here that $\comp{2}{h}_{\mathbf{t}} \circ g^W_{\tau}$ has signature $\sign(W,\vec{w})$, is essentially discontinuous, and has uniform cofinality $\ucf(W,\vec{w})$. 
In either case, when $Q$ is finite,  we have the following: the signature of $h^Q_{\mathbf{C}}$ is $\sign(\mathbf{C})$; $h^Q_{\mathbf{C}}$ is essentially continuous iff $\mathbf{C}$ is of continuous type; the uniform cofinality of $h^Q_{\mathbf{C}}$ is $\ucf(\mathbf{C})$. If $W^{+}$ is the completion of $(W,w_m)$, then $j^{W,W^{+}}\circ h^Q_{\mathbf{C}}$ is of discontinuous type and has cofinality $\ucf^{+}(\mathbf{C})$. 
Moreover, $\ran(h^Q_{\mathbf{C}}) \subseteq \ran(h)$ if $d=1$, $\ran(h^Q_{\mathbf{C}}) \subseteq j^W(\ran(h))$ if $d=2$. 
When $T,Q$ are both finite, $\mathbf{C} = (d,\mathbf{t},\tau) \in \desc(T,Q,*)$,
\begin{displaymath}
  \id_{\mathbf{C}}^{T,Q}
\end{displaymath}
is the function  $[h]^T \mapsto [h_{\mathbf{C}}^{Q}]_{\mu^Q}$, or equivalently,  $\vec{\xi} \mapsto \comp{1}{\xi}_{\mathbf{t}}$ when $d=1$, $\vec{\xi} \mapsto \tau^{Q,{W}}(\comp{2}{\xi}_{\mathbf{t}})$ when $d=2$ and $\mathbf{C} \in \desc(T,Q, W)$. The signature of $\id^{T,Q}_{\mathbf{C}}$ is $\sign_{*}^Q(\mathbf{C})$; $\id^{T,Q}_{\mathbf{C}}$ is essentially continuous iff $\mathbf{C}$ is of $*$-$Q$-continuous type; the uniform cofinality of $\id^{T,Q}_{\mathbf{C}}$ is $\ucf_{*}^Q(\mathbf{C})$. 
\begin{displaymath}
  \seed_{\mathbf{C}}^{T,Q} \in \admistwoboldextra{j^T \circ j^Q}
\end{displaymath}
is the element represented modulo $\mu^T$ by $\id^{T,Q}_{\mathbf{C}}$. Using \Los{}, it is clear that if $d=1$, then for any $A \in \mu_{\mathbb{L}}$, $\seed^{T,Q}_{\mathbf{C}} \in j^T \circ j^Q(A)$; if $d=2$, then for any $A \in \mu^{(W,\vec{w})}$, $\seed_{\mathbf{C}}^{T,Q} \in j^T \circ j^Q(A)$. Using Lemma~\ref{lem:factor_SQW}, we can see that $\seed^{T,Q}_{\mathbf{C}} \in  \set{u_n}{n<\omega}$, and $\seed^{T,Q}_{\mathbf{C}}$ can be computed in the following concrete way:
\begin{itemize}
\item If $d=1$, then 
 $\seed^{T,Q}_{\mathbf{C}} = \seed^{T,\emptyset}_{\mathbf{C}} = \seed^{T \otimes \emptyset}_{\mathbf{C}}$.
\item If $d=2$ and $\mathbf{C} \in \desc(T,Q,W)$, then 
 $\seed^{T,Q}_{\mathbf{C}} = \seed^{T, Q \otimes W}_{\mathbf{C}} = \seed^{T \otimes (Q \otimes W)}_{\mathbf{C}}$. 
\end{itemize}

If $\mathbf{C} = (1, \mathbf{t}, \emptyset)$, let
\begin{displaymath}
  \mathbf{C}^{T,Q} : \admistwoboldextra{ j_{\mu_{\mathbb{L}}}} \to \admistwoboldextra{j^T \circ j^Q }
\end{displaymath}
where $\mathbf{C}^{T,Q} ( j_{\mu_{\mathbb{L}}}(h)(\omega_1) ) = j ^{ T } \circ j^Q (\seed^{T,Q}_{\mathbf{C}}) $. 
If $\mathbf{C} = (2,\mathbf{t},\tau)$, let
\begin{displaymath}
  \mathbf{C}^{T,Q} : \admistwoboldextra{ j ^{(W,\vec{w})}} \to \admistwoboldextra{j^T \circ j^Q }
\end{displaymath}
where  $\mathbf{C}^{T,Q} ( j^{(W,\vec{w})}(h)(\seed^{(W,\vec{w})}) ) = j ^{ T } \circ j^Q (\seed^{T,Q}_{\mathbf{C}}) $. 

If $\mathbf{C} = (2, \mathbf{t}, \tau)$, $\mathbf{t} = (t, S, \vec{s})$, define
\begin{displaymath}
  \lh( \mathbf{C}) = \lh( \mathbf{t} )
\end{displaymath}
and
\begin{displaymath}
  \mathbf{C} \res l = (2, \mathbf{t} \res l, \tau \res \set{s_i}{i < l}).
\end{displaymath}
Define
\begin{displaymath}
  \mathbf{C} \iniseg \mathbf{C}'
\end{displaymath}
iff $\mathbf{C} = \mathbf{C}' \res l$ for some $l < \lh(\mathbf{C}')$. Define $\mathbf{C}^{-} = \mathbf{C} \res \lh(\mathbf{C})-1$

Let 
\begin{displaymath}
  \corner{\mathbf{C}}   =
  \begin{cases}
    (1,\mathbf{t})& \text{ if } d=1, \\
    (2, \tau \oplus \mathbf{t})& \text{ if } d=2.
  \end{cases}
\end{displaymath}
Define
\begin{displaymath}
  \mathbf{C} \prec \mathbf{C}'
\end{displaymath}
iff $\corner{\mathbf{C}} <_{BK} \corner{\mathbf{C}'}$, the ordering on subcoordinates in $\desc(Q,*)\cup \desc(Q',*)$ according to $\prec$ acting on $\desc(Q,*)\cup \desc(Q',*)$.
The constant  $(T,Q,*)$-description,  $\mathbf{C}_0$,  is the $\prec$-greatest, and we have $\corner{\mathbf{C}_0} =(2,\emptyset)$. Define $\prec^{T,Q} = \prec \res \desc(T,Q, *)$. 
$\prec^{T,Q}$ inherits the following ordering property as a corollary  to Lemma~\ref{lem:level_2_desc_order}.
\begin{mylemma}
  \label{lem:TQW_desc_order}
Suppose $(\vec{W},\vec{w}) = (W_i,w_i)_{i \leq m}$ is a partial level $\leq 1$ tower. 
 Suppose $\mathbf{C}=(2,\mathbf{t},\tau)\in \desc(T,Q,W_k)$, $\mathbf{C}' =(2,\mathbf{t}',\tau')\in \desc(T,Q,W_{k'})$, $k\leq m$, $k' \leq m$, 
 $\mathbf{C} \prec^{T,Q} \mathbf{C}'$. Then
  for any $h \in \omega_1^{T \uparrow}$, for any $g \in \omega_1^{Q \uparrow}$, for any $\vec{\alpha} \in \omega_1^{W_m\uparrow}$, $h_{\mathbf{t}}\circ g^{W_m}_{\tau}(\vec{\alpha}) < h_{\mathbf{t}'} \circ g^{W_m}_{\tau'}(\vec{\alpha})$.
\end{mylemma}

Suppose $(\vec{W},\vec{w}) = (W_i,w_i)_{i \leq m}$ is a potential partial level-1 tower and $m>0$. If $\mathbf{C} = (2, \mathbf{t}, \tau)\in \desc(T,Q, (W_m,\vec{w}))$, $\mathbf{t} = (t, S,\vec{s})$, $\bar{m}< m$, then
\begin{displaymath}
  \mathbf{C} \res  (T, Q, W_{\bar{m}})  \in \desc(T,Q, (W_{\bar{m}}, (w_i)_{i \leq \bar{m}}))
\end{displaymath}
is defined by the following: letting $l $ be the least such that $ \tau (s_l) \notin \desc(Q, W_{\bar{m}})$, and letting $\mathbf{D} = \tau(s_l) \res (Q ,W_{\bar{m}})$,  then
\begin{enumerate}
\item if $\mathbf{D} \neq \tau (s_l^{-})$, then $ \mathbf{C} \res  (T, Q, W_{\bar{m}})= (2, \mathbf{t} \res l \concat (-1), \bar{\tau})$, where $\bar{\tau}$ and $\tau$ agree on $ \comp{2}{T}_{\tree}(t \res l)$, $\bar{\tau}(s_l) = \mathbf{D}$;
\item if $\mathbf{D} = \tau(s_l^{-})$, then $ \mathbf{C} \res  (T, Q, W_{\bar{m}})= (2, \mathbf{t} \res l, \tau \res \comp{2}{T}_{\tree}(t \res l))$.
\end{enumerate}
As a corollary to Lemma~\ref{lem:QW_description_extension}, the $\res  (T, Q, W_{\bar{m}})$ operator inherits  the following continuity property.
\begin{mylemma}
  \label{lem:TQW_description_extension}
Suppose $T,Q$ are finite level $\leq 2$ trees,  $W$ is a level-1 proper subtree of $W'$. 
Suppose $E \in \mu_{\mathbb{L}}$ is a club, $E'$ is the set of limit points of $E$. 
Suppose
  $\mathbf{C}=(2,\mathbf{t},\tau) \in \desc(T,Q, W)$,   $\mathbf{C}'=(2,\mathbf{t}',\tau') \in \desc(T,Q, W')$, $\mathbf{C}  = \mathbf{C}' \res  (T, Q, W_{\bar{m}})$.
  Then   for any $h \in \omega_1^{T \uparrow}$, for any $g \in \omega_1^{Q \uparrow}$, for any $\vec{\alpha} \in [E']^{W\uparrow}$,
    \begin{displaymath}
      h_{\mathbf{t}} \circ g_{\tau}^{W}(\vec{\alpha}) = \sup \{ h_{\mathbf{t}'} \circ g_{\tau'}^{W'}(\vec{\beta}) : \vec{\beta} \in [E]^{W'\uparrow}, \vec{\beta} \text{ extends } \vec{\alpha}\}.
    \end{displaymath}
Hence, the signature and approximation sequence of $h_{\mathbf{t}} \circ g^W_{\tau}$ are proper initial segments of those of  $h_{\mathbf{t}'} \circ g^{W'}_{\tau'}$ respectively. 
\end{mylemma}

\begin{mydefinition}
  \label{def:factoring_3}
Suppose $X,T,Q$ are level $\leq 2$ trees. Suppose $\pi : \dom(X) \to \desc(T,Q,*)$ is a function. $\pi$ is said to \emph{factor} $(X,T,{Q})$ iff
  \begin{enumerate}
  \item If $(1,x) \in \dom(X)$, then $\pi(1,x) \in \desc(T,Q,-1) \cup \desc(T,Q,\comp{2}{X}[\emptyset])$. 
  \item  If $(2,x) \in \dom({X})$, then  $\pi(2,x) \in \desc(T,Q,  \comp{2}{X}[x])$.   
  \item   $\pi(2,\emptyset)$ is the constant $(T,Q,*)$-description.  
  \item For any $(d,x), (d',x') \in \dom(X)$, if $(d,x) <_{BK}(d', x')$ then ${\pi(d,x)}\prec^{T,Q} {\pi(d',x')}$. 
  \item   For any $x \in \dom(\comp{2}{X})\setminus\se{\emptyset}$, $\pi(2,x^{-}) =  \pi(2,x) \res  (T, Q, \comp{2}{X}_{\tree}(x^{-}))$.
  \end{enumerate}
$\pi$ is said to \emph{factor} $(X,T,*)$ iff $\pi$ factors $(X,T,Q')$ for some level $\leq 2$ tree $Q'$. 
\end{mydefinition}

 \begin{myexample}
   Let $X,T,Q$ be as in Section~\ref{sec:level-2-2}. There is a $(X,T,Q)$-factoring map $\pi$ characterizing the $X$-equivalence class of $\theta_{XT}$. We have for instance,
   \begin{itemize}
   \item $\pi (2, ((0))) = ( 2 , \comp{2}{T}[((1),(0))], \tau )$, where $\tau((0)) = (2, (-1,\se{(0)},((0))), (0) \mapsto (0))$, $\tau(1) = (1, (0))$.
   \item $\pi(1,(1)) =  (2, (-1,\se{(0)},(0)), \tau' )$, where $\tau'((0)) = (1,(0))$.
   \item $\pi(1,(0)) = (1, ((0)))$.
   \end{itemize}
 \end{myexample}

Suppose $T$ is a level $\leq 2$ tree.
\begin{displaymath}
\id_{T,*}
\end{displaymath}
factors $(T,T,Q^0)$ where $\id_{T,*}(1,t) = (1,t,\emptyset)$, $\id_{T,*}(2,t) = (2,(t,S,\vec{s}), \id_{*,S})$ when $\comp{2}{T}[t] = (S,\vec{s})$. 
\begin{displaymath}
\id_{*, T}
\end{displaymath}
factors $(T,Q^0,T)$ where $\id_{*,T}(1,t) = (2, \mathbf{q}_0, \tau^1_t)$, $\mathbf{q} _0 = ((-1), \se{(0)}, ((0)))$, $\tau^1_t$ factors $(\se{(0)}, T, \emptyset)$, $\tau^1_t ((0)) = (1, t, \emptyset)$,  $\id_{*,T}(2, t) = (2, \mathbf{q}_0, \tau^2_t)$ when $\comp{2}{T}[t] = (S,\vec{s})$,  $\tau^2_t$ factors $(\se{(0)}, T, S)$, $\tau^2_t((0)) = (2, (t, S, \vec{s}), \id_S)$. 

If $\pi\neq \emptyset$ factors $(X,T,Q)$, $T$ is $\Pi^1_2$-wellfounded and $h \in \omega_1^{T \uparrow}$, let
\begin{displaymath}
  h_{\pi} ^{Q}: [\omega_1]^{Q \uparrow} \to [\omega_1]^{X \uparrow}
\end{displaymath}
be the function that sends $\vec{\beta}$ to $(h_{\pi(d,x)} ^{Q}(\vec{\beta}))_{(d,x)\in \dom(X)}$. 
The fact that $h^Q_{\pi}(\vec{\beta}) \in [\omega_1]^{X \uparrow}$ follows from Lemmas~\ref{lem:Q_respecting},~\ref{lem:TQW_desc_order}-\ref{lem:TQW_description_extension}. 
Moreover, for any $\vec{\beta} \in [\omega_1]^{Q\uparrow}$, $h^Q_{\pi}(\vec{\beta}) \in [\ran(h)]^{X \uparrow}$.
In particular, if $Q=X$ then $h^{Q}_{\id_{*,X}}$ is the identity function on $[\omega_1]^{X \uparrow}$.
 If $T$ is finite, let
 \begin{displaymath}
   \id^{T,Q}_{\pi} 
 \end{displaymath}
be the function  $[h]^T \mapsto [h^Q_{\pi}]_{\mu^Q}$, or equivalently,  $\vec{\xi} \mapsto (\id^{T,Q}_{\pi(d,x)} (\vec{\xi}))_{(d,x) \in \dom(X)}$. Let
\begin{displaymath}
  \seed_{\pi}^{T,Q} =[\id^{T,Q}_{\pi}]_{\mu^T}.
\end{displaymath}
By \Los{} and Lemmas~\ref{lem:Q_respecting},~\ref{lem:TQW_desc_order}-\ref{lem:TQW_description_extension},  it is clear that for any $A \in \mu^X$, $\seed_{\pi}^{T,Q} \in j^T \circ j^Q (A) $.
 Define
\begin{displaymath}
  \pi^{T,Q} : \admistwoboldextra{j^X} \to \admistwoboldextra{j^{T } \circ j^Q}
\end{displaymath}
by sending $ j^X(h)(\seed^X)$ to $j^T\circ j^Q (h) ( \seed^{T,Q}_{\pi})$.

Suppose $T,Q$ are both level $\leq 2$ trees.  \emph{A representation of $T \otimes Q$} is a pair $(X,\pi)$ such that
\begin{enumerate}
\item $X$ is a level $\leq 2$ tree;
\item $\pi$ factors $(X,T,Q)$;
\item $\ran(\pi) = \desc(T,Q,*)$;
\item $(d,x) <_{BK} (d',x')$ iff $\pi(d,x) \prec^{T,Q} \pi(d',x')$.
\end{enumerate}
Representations of $T \otimes Q$ are clearly mutually isomorphic. We shall regard
\begin{displaymath}
T \otimes Q
\end{displaymath}
itself as a ``level $\leq 2$ tree'' whose level-$d$ component is $(\mathbf{t},\tau)$ for which $(d, (\mathbf{t}, \tau)) \in \desc(T, Q, *)$, and whose level-2 component sends $(\mathbf{t}, \tau)$ to $(W, w_m)$ if $(2, (\mathbf{t}, \tau)) \in \desc(T, Q, (W, (w_i)_{i \leq m}))$. 
In this way, $\pi$ is a ``level $\leq 2$ tree isomorphism'' between $X$ and $T \otimes Q$. All the relevant terminologies of level $\leq 2$ trees carry over to $T \otimes Q$ in the obvious ways. In particular, if $W$ is a finite level-1 tree, a $(T \otimes Q, W)$-description takes one of the following forms (recall that $(d,\mathbf{t}, \tau)$ is simply an abbreviation of  $(d, (\mathbf{t}, \tau))$):
\begin{enumerate}
\item $(1,(\mathbf{t},\emptyset), \emptyset)$ for $(1, \mathbf{t}, \emptyset) \in \desc(T, Q, -1)$;
\item $(2, ((\mathbf{t}, \tau), Z, \vec{z}), \psi)$ for $(2, \mathbf{t}, \tau) \in \desc(T, Q, (Z, \vec{z}))$ and $\psi$ factoring $(Z, W)$;
\item $(2, ((\mathbf{t}, \tau) \concat(-1) , Z^{+}, \vec{z}), \psi)$ for $(2, \mathbf{t}, \tau) \in \desc(T, Q, (Z,\vec{z}))$, $\vec{z} = (z_i)_{i \leq l}$, $Z^{+} = Z \cup \se{z_l}$ and $\psi$ factoring $(Z^{+}, W)$. 
\end{enumerate}
$(T \otimes Q ) \otimes W$ is thus regarded as a ``level-1 tree'' whose nodes consists of non-constant $(T \otimes Q, W)$-descriptions. There is a natural isomorphism
\begin{displaymath}
\iota_{T,Q,W}
\end{displaymath}
between ``level-1 trees'' $(T \otimes Q) \otimes W$ and $T \otimes (Q \otimes W)$, defined as follows.
\begin{enumerate}
\item $\iota_{T,Q,W}(1, (\mathbf{t}, \emptyset) , \emptyset) = (1, \mathbf{t}, \emptyset)$. 
\item If $(2, \mathbf{t}, \tau) \in \desc(T,Q,(Z,\vec{z}))$, $\mathbf{t} = (t, S, \vec{s})$, $\psi$ factors $(Z, W)$, define $\iota_{T,Q,W}(2, ((\mathbf{t}, \tau), Z, \vec{z}), \psi) = (2, \mathbf{t}, (Q \otimes \psi) \circ  \tau)$. 
\item If $(2, \mathbf{t}, \tau) \in \desc(T,Q,(Z,\vec{z}))$, $\mathbf{t} = (t, S, \vec{s})$, $\vec{z} = (z_i)_{i \leq l}$, $\vec{s} = (s_i)_{i \leq k}$, $Z^{+} = Z \cup \se{z_l}$,  $\psi$  factors $(Z^{+},W)$,
  \begin{enumerate}
  \item if $t$ is of discontinuous type,
define $\iota_{T,Q,W}(2, ((\mathbf{t}, \tau) \concat (-1), Z^{+}, \vec{z}), \psi) = (2,\mathbf{t} \concat (-1), \psi*_0\tau)$, where $\psi *_0 \tau$ factors $(S\cup \se{s_k}, Q, W)$,    $\psi *_0 \tau$ extends $ (Q \otimes \psi) \circ  \tau$, $\psi *_0 \tau (s_k) = (2,  \mathbf{q}_0  , \sigma)$, $\mathbf{q}_0 = ((-1), \se{(0)}, ((0)))$, $\sigma((0))  = \psi(z_l)$;
\item if $t$ is of continuous type, define $\iota_{T,Q,W} (2, ((\mathbf{t}, \tau) \concat (-1), Z^{+}, \vec{z}), \psi) =  (2, \mathbf{t}, \psi *_1 \tau)$, where $\psi *_1 \tau$ factors $(S, Q, W)$, $\psi *_1 \tau$ extends $ (Q \otimes \psi) \circ  (\tau\res  (S \setminus\se{s_k}))$, $\psi *_1 \tau (s_k) = (2, \mathbf{q} \concat (-1), \sigma^{+})$ where $\tau(s_k) = (2, \mathbf{q}, \sigma)$, $\mathbf{q} = (q, P, (p_i)_{i \leq m})$, $\sigma^{+} $ extends $\sigma$, $\sigma^{+} (p_m) = \psi( z_l)$. 
  \end{enumerate}
\end{enumerate}
The reason why $\iota_{T,Q,W}$ is a surjection is the following. Suppose $\mathbf{C} \in T \otimes (Q \otimes W)$ is of degree 2. Put $\mathbf{C} = (2, \mathbf{t}, \tau)$, $\mathbf{t} = (t, S, \vec{s})$, $\vec{s} = (s_i)_{i < \lh(\vec{s})}$, $k = \lh(t)$. Let $(w_i)_{i < m}$ be the contraction of $(\sign(\tau(s_i)))_{i < k}$. Then $w_0$ is the $<_{BK}$-maximum of $\set{w_i}{ i < m}$. Let $(Z,\vec{z})=(Z,(z_i)_{i < m})$ be the potential partial level $\leq 1$ tower of continuous type and $\psi : Z \to W$ be the level-1 tree isomorphism such that $\psi (z_i) = w_i$ for any $i < m$. If $\mathbf{t}$ is of continuous type, $\tau(s_{k-1})$ is of continuous type, but $w_{m-1}$ does not appear in the contraction of $(\sign(\tau(s_i)))_{i < k-1} \concat (  \sign(\tau(s_{k-1})^{-}))$, then
\begin{displaymath}
\mathbf{C} = \iota_{T,Q,W} (2, ( ( \mathbf{t}, \tau) \concat (-1), Z, \vec{z}  ) , \psi).
\end{displaymath}
Otherwise,
\begin{displaymath}
  \mathbf{C} = \iota_{T,Q,W} (2, ((\mathbf{t}, \tau) , Z, \vec{z} \concat (z_{*})), \psi),
\end{displaymath}
where $(Z,z_{*})$ is a partial level $\leq 1$ tree, $z_{*} = -1$ if $\ucf(S, \vec{s}) = -1$, $z_{*}^{-} = \ucf_1(\tau(\ucf(S,\vec{s})))$ if $\ucf(S,\vec{s}) \neq -1$. $\iota_{T,Q,W}$ justifies the associativity of the $\otimes$ operator acting on level $(\leq 2, \leq 2, 1)$-trees.

\begin{myexample}
  Let $T,Q$ be as in Section~\ref{sec:level-2-2}. $U = T \otimes Q$ has cardinality 160. $\comp{1}{U}$ has cardinality 4. $\comp{2}{U}$ has
  \begin{itemize}
  \item 1 node of length 0, which is of degree 1,
  \item 12 nodes of length 1, 8 of which are of degree 1,
  \item 31 nodes of length 2, 29 of which are of degree 1,
  \item 52 nodes of length 3, all of which are of degree 1,
  \item 45 nodes of length 4, all of which are of degree 1,
  \item 15 nodes of length 5, all of which are of degree 1,
  \end{itemize}
Let $W$ be a level-1 tree of cardinality 4. Then $(T \otimes Q) \otimes W$ has cardinality
\begin{multline*}
  4 + {4 \choose 0} + 12{4 \choose 1} + 31 {4 \choose 2}  + 52 {4 \choose 3}  + 45 {4 \choose 4} \\+ 
{4 \choose 1} + 8{4 \choose 2} + 29 {4 \choose 3}  + 52 {4 \choose 4} -1  = 711.
\end{multline*}
(the constant description does not count for a node in $(T \otimes Q) \otimes W$, so minus one in the end)
On the other hand, $Q \otimes W$ has cardinality 
\begin{displaymath}
 1 +  {4 \choose 0} + {4 \choose 1}+
{4 \choose 1} + {4 \choose 2}  - 1 = 15,
\end{displaymath}
and thus $T \otimes (Q \otimes W)$ has cardinality
\begin{displaymath}
  1 + {15 \choose 0} + 2 {15 \choose 1}  +  {15 \choose 2} +  {15 \choose 1} + {15 \choose 2} + {15 \choose 3}- 1 = 711.
\end{displaymath}
\end{myexample}

The identity function $\id_{T \otimes Q} $ factors $(T \otimes Q, T, Q)$. By definitions and Lemmas~\ref{lem:factor_SQW},~\ref{lem:beta_q_unambiguous},
\begin{displaymath}
  (\id_{T \otimes Q})^{T,Q} (\seed^{(T \otimes Q) \otimes W}_{\mathbf{C}^{*}}) = \seed^{T \otimes (Q \otimes W)}_{\iota_{T,Q,W}(\mathbf{C}^{*})}  
\end{displaymath}
for any $\mathbf{C}^{*} \in (T \otimes Q) \otimes W$. Hence, $(\id_{T \otimes Q})^{T,Q} (u_n) = u_n$ for any $n<\omega$. As $(\id_{T \otimes Q})^{T,Q}$ is elementary from $L_{\kappa_3^x}[j^{T \otimes Q}(T_2),x]$ to $L_{\kappa_3^x}[j^T \circ j^Q(T_2), x]$ for any $x \in \mathbb{R}$, $(\id_{T \otimes Q})^{T,Q}$ is the identity map on $\admistwobold$. 

Suppose $\pi$ factors level $\leq 2$ trees $(X,T)$ and $Q$ is another level $\leq 2$ tree.
\begin{itemize}
\item $ \pi \otimes Q$
factors $(X \otimes Q, T \otimes Q)$, defined as follows: $\pi \otimes Q (d,\mathbf{x}, \tau) = (d, \comp{d}{\pi}(\mathbf{x}) , \tau)$.
\item $Q \otimes \pi$ factors $(Q \otimes X, Q \otimes T)$, defined as follows: $Q \otimes \pi  (d,\mathbf{q}, \tau) = (d, \mathbf{q},  (\pi \otimes W) \circ  \tau)$, where $\tau$ factors $(P, X \otimes W)$.
\end{itemize}

We effectively obtain the following lemma which reduces finite iterations of level $\leq 2$ ultrapowers to a single level $\leq 2$ ultrapower. The proof is in parallel to Lemma~\ref{lem:jQ_move_factor_maps}. 

\begin{mylemma}
  \label{lem:level_2_ultrapower_iteration_reduce}
  Suppose $X, T,Q$ are finite level $\leq 2$ trees. Then
  \begin{enumerate}
  \item $j^T \circ j^Q  =  j^{T \otimes Q}$.
  \item  $\pi$ factors $(X,T,Q)$ iff $\pi$ factors $(X, T \otimes Q)$. If $\pi$ factors $(X,T,Q)$ then $\pi^{T,Q} = \pi^{T \otimes Q}$.
  \item If $\pi$ factors $X,T$, then
    \begin{enumerate}
    \item $j^Q ( \pi ^T \res a )  =  (Q \otimes \pi)^{Q \otimes T} \res j^Q(a)$ for any $a \in \admistwobold$;
    \item  $\pi^T  \res \admistwoboldextra{j^{X \otimes Q}} = (\pi \otimes Q)^{T \otimes Q}$.
    \end{enumerate}
  \end{enumerate}
\end{mylemma}

Suppose $T,Q, U$ are level $\leq 2$ trees. There is a natural ``level $\leq 2$ tree isomorphism''
\begin{displaymath}
\iota_{T,Q,U}
\end{displaymath}
between $(T \otimes Q) \otimes U$ and $T \otimes (Q \otimes U)$ defined as follows. Suppose $\mathbf{B} \in \desc(T \otimes Q, U, *)$.
\begin{enumerate}
\item If $\mathbf{B}= (1,( t, \emptyset), \emptyset)$, $\mathbf{C}= (1, t,  \emptyset) \in \desc(T, Q, -1)$, then $\mathbf{C} \in \desc(T, Q \otimes U, -1)$ and $\iota_{T,Q,U}(\mathbf{B}) = \mathbf{C}$. 
\item If $\mathbf{B} = (2, (  (\mathbf{t}, \tau), Z, \vec{z}) , \psi) \in \desc(T \otimes Q, U, W)$, $\mathbf{C} =(2, \mathbf{t}, \tau)\in \desc(T, Q, (Z, \vec{z}))$, $\mathbf{t} = (t, S, \vec{s})$, $\psi$ factors $(Z, U, W)$, then $\iota_{T,Q,U}(\mathbf{B}) =(2,  \mathbf{t}, \iota^{-1}_{Q,U,W}\circ  (Q \otimes \psi) \circ  \tau)$.   
\item If $ \mathbf{B}= (2, ((\mathbf{t}, \tau)\concat(-1), Z^{+}, \vec{z}), \psi) \in \desc(T \otimes Q, U, W)$, $\mathbf{C} = (2, \mathbf{t}, \tau) \in \desc(T,Q, (Z, \vec{z}))$, $\mathbf{t} = (t, S, (s_i)_{i \leq k})$, $\vec{z} = (z_i)_{i \leq l}$, 
  \begin{enumerate}
  \item if $t$ is of discontinuous type, then $\iota_{T,Q,U}( \mathbf{B}) =(2,  \mathbf{t} \concat(-1), \psi * _0\tau)$, where $\psi*_0 \tau$ factors $(S \cup \se{s_k}, Q \otimes U, W)$, $\psi*_0 \tau$ extends $\iota^{-1}_{Q,U,W}\circ  (Q \otimes \psi) \circ  \tau$, $\psi*_0 \tau (s_k) =\iota_{Q,U,W}^{-1}(2, \mathbf{q}_0, \sigma)$, $\mathbf{q}_0 = ((-1), \se{(0)}, ((0)))$, $\sigma((0)) =\psi( z_l)$.
  \item if $t$ is of continuous type, then $\iota_{T,Q,U}(\mathbf{B}) = (2, \mathbf{t}, \psi *_1 \tau)$, where $\psi *_1 \tau $ factors $(S, Q \otimes U , W)$, $\psi * _1 \tau$ extends $\iota^{-1}_{Q,U,W}\circ  (Q \otimes \psi) \circ   (\tau \res (S \setminus \se{s_k}))$, $\psi *_1 \tau(s_k )= \iota_{Q,U,W}^{-1} (2, \mathbf{q} \concat (-1), \sigma^{+})$ where $\tau(s_k) = (2, \mathbf{q}, \sigma)$, $\mathbf{q} = (q,P, (p_i)_{i \leq m})$, $\sigma^{+}$ extends $\sigma$, $\sigma^{+}(p_m) = \psi(z_l)$. 
  \end{enumerate}
\end{enumerate}
 $\iota_{T,Q,U}$ justifies the associativity of the $\otimes$ operator acting on level $(\leq 2, \leq 2, \leq 2)$ trees. 

 \begin{myexample}
   Let $Q= U = Q^{21}$, $\comp{1}{T} = \se{(0)}$, $\comp{2}{T} = \comp{1}{Q}^{20}$. A representation of $T \otimes Q$ is $X$, and $\pi_X$ is a "level $\leq 2$ tree isomorphism" between $X$ and  $T \otimes Q$, given as follows, where
$    \langle\langle \mathbf{C} \rangle\rangle = \langle \mathbf{C} \rangle$ if $\mathbf{C}$ has degree 1, and 
 for $\mathbf{C} = (2, \mathbf{t},\tau)$,
  \begin{displaymath}
    \langle\langle \mathbf{C} \rangle\rangle = (2,\langle \tau \rangle \oplus \mathbf{t}),
  \end{displaymath}
where $ \langle \tau \rangle (\cdot) = \langle \tau (\cdot)\rangle$.
\begin{itemize}\item $\comp{2}{X}_{\node}(((2)))= -1$

$\langle\langle\pi_X((2))\rangle\rangle=$ (2, ((2, ((0), (0))), (0)))

\item $\comp{2}{X}_{\node}(((1)))= (0, 0)$

$\langle\langle\pi_X((1))\rangle\rangle=$ (2, ((2, ((0), (0))), $-1$))

\item $\comp{2}{X}_{\node}(((1), (0)))= -1$

$\langle\langle\pi_X((1), (0))\rangle\rangle=$ (2, ((2, ((0), (0), (0, 0), $-1$)), (0)))

\item $\comp{2}{X}_{\node}(((0)))= -1$

$\langle\langle\pi_X((0))\rangle\rangle=$ (2, ((2, ((0), $-1$)), (0)))

\item $\comp{1}{X}$ contains (0)

$\langle\langle\pi_X(0)\rangle\rangle=$ (1, ((0)))

\end{itemize}A representation of $Q \otimes U$ is $M$, and $\pi_M$ is a "level $\leq 2$ tree isomorphism" between $M$ and  $Q \otimes U$, given as follows:
\begin{itemize}\item $\comp{2}{M}_{\node}(((3)))= (0, 0)$

$\langle\langle\pi_M((3))\rangle\rangle=$ (2, ((2, ((0), (0))), (0)))

\item $\comp{2}{M}_{\node}(((2)))= (0, 0)$

$\langle\langle\pi_M((2))\rangle\rangle=$ (2, ((2, ((0), (0))), (0), (2, ((0), $-1$)), $-1$))

\item $\comp{2}{M}_{\node}(((2), (0)))= (0, 0, 0)$

$\langle\langle\pi_M((2), (0))\rangle\rangle=$ (2, ((2, ((0), (0))), (0), (2, ((0, 0), (0))), $-1$))

\item $\comp{2}{M}_{\node}(((1)))= (0, 0)$

$\langle\langle\pi_M((1))\rangle\rangle=$ (2, ((2, ((0), (0))), $-1$))

\item $\comp{2}{M}_{\node}(((1), (3)))= (0, 0, 0)$

$\langle\langle\pi_M((1), (3))\rangle\rangle=$ (2, ((2, ((0), (0), (0, 0), $-1$)), (0)))

\item $\comp{2}{M}_{\node}(((1), (2)))= (0, 1)$

$\langle\langle\pi_M((1), (2))\rangle\rangle=$ (2, ((2, ((0), (0), (0, 0), $-1$)), (0), (2, ((0), $-1$)), $-1$))

\item $\comp{2}{M}_{\node}(((1), (2), (1)))= (0, 1, 0)$

$\langle\langle\pi_M((1), (2), (1))\rangle\rangle=$ (2, ((2, ((0), (0), (0, 0), $-1$)), (0), (2, ((0, 1), (0))), $-1$))

\item $\comp{2}{M}_{\node}(((1), (2), (0)))= (0, 0, 0)$

$\langle\langle\pi_M((1), (2), (0))\rangle\rangle=$ (2, ((2, ((0), (0), (0, 0), $-1$)), (0), (2, ((0, 1), (0), (0, 0), $-1$)), $-1$))

\item $\comp{2}{M}_{\node}(((1), (1)))= (0, 0, 0)$

$\langle\langle\pi_M((1), (1))\rangle\rangle=$ (2, ((2, ((0), (0), (0, 0), $-1$)), (0), (2, ((0, 0), (0))), $-1$))

\item $\comp{2}{M}_{\node}(((1), (0)))= (0, 0, 0)$

$\langle\langle\pi_M((1), (0))\rangle\rangle=$ (2, ((2, ((0), (0), (0, 0), $-1$)), (0), (2, ((0, 0), $-1$)), $-1$))

\item $\comp{2}{M}_{\node}(((1), (0), (0)))= (0, 0, 0, 0)$

$\langle\langle\pi_M((1), (0), (0))\rangle\rangle=$ (2, ((2, ((0), (0), (0, 0), $-1$)), (0), (2, ((0, 0, 0), (0))), $-1$))

\item $\comp{2}{M}_{\node}(((0)))= (0, 0)$

$\langle\langle\pi_M((0))\rangle\rangle=$ (2, ((2, ((0), $-1$)), (0)))

\item $\comp{2}{M}_{\node}(((0), (0)))= (0, 0, 0)$

$\langle\langle\pi_M((0), (0))\rangle\rangle=$ (2, ((2, ((0), $-1$)), (0), (2, ((0, 0), (0))), $-1$))

\end{itemize}Both $(T \otimes Q) \otimes U$ and $T \otimes (Q \otimes U)$ can be represented by the tree $N$. Let $\psi$ be the "level $\leq 2$ tree isomorphism" between $N$ and  $X \otimes U$ and let $\psi'$ be the "level $\leq 2$ tree isomorphism" between $N$ and  $T \otimes M$. Then $N$ and $\psi$, $\psi'$ are as follows:
\begin{etaremune}
\item $\comp{2}{N}_{\node}(((8)))= -1$

$\langle\langle\psi((8))\rangle\rangle=$ (2, ((2, ((0), (0))), (2)))

$\langle\langle\psi'((8))\rangle\rangle=$ (2, ((2, ((0), (3))), (0)))

\item $\comp{2}{N}_{\node}(((7)))= (0, 0)$

$\langle\langle\psi((7))\rangle\rangle=$ (2, ((2, ((0), (0))), (1)))

$\langle\langle\psi'((7))\rangle\rangle=$ (2, ((2, ((0), (3))), $-1$))

\item $\comp{2}{N}_{\node}(((7), (0)))= -1$

$\langle\langle\psi((7), (0))\rangle\rangle=$ (2, ((2, ((0), (0))), (1), (2, ((0), (0), (0, 0), $-1$)), (0)))

$\langle\langle\psi'((7), (0))\rangle\rangle=$ (2, ((2, ((0), (3), (0, 0), $-1$)), (0)))

\item $\comp{2}{N}_{\node}(((6)))= -1$

$\langle\langle\psi((6))\rangle\rangle=$ (2, ((2, ((0), (0))), (1), (2, ((0), $-1$)), (0)))

$\langle\langle\psi'((6))\rangle\rangle=$ (2, ((2, ((0), (2))), (0)))

\item $\comp{2}{N}_{\node}(((5)))= (0, 0)$

$\langle\langle\psi((5))\rangle\rangle=$ (2, ((2, ((0), (0))), (1), (2, ((0), $-1$)), $-1$))

$\langle\langle\psi'((5))\rangle\rangle=$ (2, ((2, ((0), (2))), $-1$))

\item $\comp{2}{N}_{\node}(((5), (2)))= -1$

$\langle\langle\psi((5), (2))\rangle\rangle=$ (2, ((2, ((0), (0))), (1), (2, ((0, 0), (0))), (0)))

$\langle\langle\psi'((5), (2))\rangle\rangle=$ (2, ((2, ((0), (2), (0, 0), (0))), (0)))

\item $\comp{2}{N}_{\node}(((5), (1)))= (0, 0, 0)$

$\langle\langle\psi((5), (1))\rangle\rangle=$ (2, ((2, ((0), (0))), (1), (2, ((0, 0), (0))), $-1$))

$\langle\langle\psi'((5), (1))\rangle\rangle=$ (2, ((2, ((0), (2), (0, 0), (0))), $-1$))

\item $\comp{2}{N}_{\node}(((5), (1), (0)))= -1$

$\langle\langle\psi((5), (1), (0))\rangle\rangle=$ (2, ((2, ((0), (0))), (1), (2, ((0, 0), (0), (0, 0, 0), $-1$)), (0)))

$\langle\langle\psi'((5), (1), (0))\rangle\rangle=$ (2, ((2, ((0), (2), (0, 0), (0), (0, 0, 0), $-1$)), (0)))

\item $\comp{2}{N}_{\node}(((5), (0)))= -1$

$\langle\langle\psi((5), (0))\rangle\rangle=$ (2, ((2, ((0), (0))), (1), (2, ((0, 0), $-1$)), (0)))

$\langle\langle\psi'((5), (0))\rangle\rangle=$ (2, ((2, ((0), (2), (0, 0), $-1$)), (0)))

\item $\comp{2}{N}_{\node}(((4)))= -1$

$\langle\langle\psi((4))\rangle\rangle=$ (2, ((2, ((0), (0))), (0)))

$\langle\langle\psi'((4))\rangle\rangle=$ (2, ((2, ((0), (1))), (0)))

\item $\comp{2}{N}_{\node}(((3)))= (0, 0)$

$\langle\langle\psi((3))\rangle\rangle=$ (2, ((2, ((0), (0))), $-1$))

$\langle\langle\psi'((3))\rangle\rangle=$ (2, ((2, ((0), (1))), $-1$))

\item $\comp{2}{N}_{\node}(((3), (8)))= -1$

$\langle\langle\psi((3), (8))\rangle\rangle=$ (2, ((2, ((0), (0), (0, 0), $-1$)), (2)))

$\langle\langle\psi'((3), (8))\rangle\rangle=$ (2, ((2, ((0), (1), (0, 0), (3))), (0)))

\item $\comp{2}{N}_{\node}(((3), (7)))= (0, 0, 0)$

$\langle\langle\psi((3), (7))\rangle\rangle=$ (2, ((2, ((0), (0), (0, 0), $-1$)), (1)))

$\langle\langle\psi'((3), (7))\rangle\rangle=$ (2, ((2, ((0), (1), (0, 0), (3))), $-1$))

\item $\comp{2}{N}_{\node}(((3), (7), (0)))= -1$

$\langle\langle\psi((3), (7), (0))\rangle\rangle=$ (2, ((2, ((0), (0), (0, 0), $-1$)), (1), (2, ((0), (0), (0, 0, 0), $-1$)), (0)))

$\langle\langle\psi'((3), (7), (0))\rangle\rangle=$ (2, ((2, ((0), (1), (0, 0), (3), (0, 0, 0), $-1$)), (0)))

\item $\comp{2}{N}_{\node}(((3), (6)))= -1$

$\langle\langle\psi((3), (6))\rangle\rangle=$ (2, ((2, ((0), (0), (0, 0), $-1$)), (1), (2, ((0), $-1$)), (0)))

$\langle\langle\psi'((3), (6))\rangle\rangle=$ (2, ((2, ((0), (1), (0, 0), (2))), (0)))

\item $\comp{2}{N}_{\node}(((3), (5)))= (0, 1)$

$\langle\langle\psi((3), (5))\rangle\rangle=$ (2, ((2, ((0), (0), (0, 0), $-1$)), (1), (2, ((0), $-1$)), $-1$))

$\langle\langle\psi'((3), (5))\rangle\rangle=$ (2, ((2, ((0), (1), (0, 0), (2))), $-1$))

\item $\comp{2}{N}_{\node}(((3), (5), (4)))= -1$

$\langle\langle\psi((3), (5), (4))\rangle\rangle=$ (2, ((2, ((0), (0), (0, 0), $-1$)), (1), (2, ((0, 1), (0))), (0)))

$\langle\langle\psi'((3), (5), (4))\rangle\rangle=$ (2, ((2, ((0), (1), (0, 0), (2), (0, 1), (1))), (0)))

\item $\comp{2}{N}_{\node}(((3), (5), (3)))= (0, 1, 0)$

$\langle\langle\psi((3), (5), (3))\rangle\rangle=$ (2, ((2, ((0), (0), (0, 0), $-1$)), (1), (2, ((0, 1), (0))), $-1$))

$\langle\langle\psi'((3), (5), (3))\rangle\rangle=$ (2, ((2, ((0), (1), (0, 0), (2), (0, 1), (1))), $-1$))

\item $\comp{2}{N}_{\node}(((3), (5), (3), (0)))= -1$

$\langle\langle\psi((3), (5), (3), (0))\rangle\rangle=$ (2, ((2, ((0), (0), (0, 0), $-1$)), (1), (2, ((0, 1), (0), (0, 1, 0), $-1$)), (0)))

$\langle\langle\psi'((3), (5), (3), (0))\rangle\rangle=$ (2, ((2, ((0), (1), (0, 0), (2), (0, 1), (1), (0, 1, 0), $-1$)), (0)))

\item $\comp{2}{N}_{\node}(((3), (5), (2)))= -1$

$\langle\langle\psi((3), (5), (2))\rangle\rangle=$ (2, ((2, ((0), (0), (0, 0), $-1$)), (1), (2, ((0, 1), (0), (0, 0), $-1$)), (0)))

$\langle\langle\psi'((3), (5), (2))\rangle\rangle=$ (2, ((2, ((0), (1), (0, 0), (2), (0, 1), (0))), (0)))

\item $\comp{2}{N}_{\node}(((3), (5), (1)))= (0, 0, 0)$

$\langle\langle\psi((3), (5), (1))\rangle\rangle=$ (2, ((2, ((0), (0), (0, 0), $-1$)), (1), (2, ((0, 1), (0), (0, 0), $-1$)), $-1$))

$\langle\langle\psi'((3), (5), (1))\rangle\rangle=$ (2, ((2, ((0), (1), (0, 0), (2), (0, 1), (0))), $-1$))

\item $\comp{2}{N}_{\node}(((3), (5), (1), (0)))= -1$

$\langle\langle\psi((3), (5), (1), (0))\rangle\rangle=$ (2, ((2, ((0), (0), (0, 0), $-1$)), (1), (2, ((0, 1), (0), (0, 0, 0), $-1$)), (0)))

$\langle\langle\psi'((3), (5), (1), (0))\rangle\rangle=$ (2, ((2, ((0), (1), (0, 0), (2), (0, 1), (0), (0, 0, 0), $-1$)), (0)))

\item $\comp{2}{N}_{\node}(((3), (5), (0)))= -1$

$\langle\langle\psi((3), (5), (0))\rangle\rangle=$ (2, ((2, ((0), (0), (0, 0), $-1$)), (1), (2, ((0, 1), $-1$)), (0)))

$\langle\langle\psi'((3), (5), (0))\rangle\rangle=$ (2, ((2, ((0), (1), (0, 0), (2), (0, 1), $-1$)), (0)))

\item $\comp{2}{N}_{\node}(((3), (4)))= -1$

$\langle\langle\psi((3), (4))\rangle\rangle=$ (2, ((2, ((0), (0), (0, 0), $-1$)), (1), (2, ((0, 0), (0))), (0)))

$\langle\langle\psi'((3), (4))\rangle\rangle=$ (2, ((2, ((0), (1), (0, 0), (1))), (0)))

\item $\comp{2}{N}_{\node}(((3), (3)))= (0, 0, 0)$

$\langle\langle\psi((3), (3))\rangle\rangle=$ (2, ((2, ((0), (0), (0, 0), $-1$)), (1), (2, ((0, 0), (0))), $-1$))

$\langle\langle\psi'((3), (3))\rangle\rangle=$ (2, ((2, ((0), (1), (0, 0), (1))), $-1$))

\item $\comp{2}{N}_{\node}(((3), (3), (0)))= -1$

$\langle\langle\psi((3), (3), (0))\rangle\rangle=$ (2, ((2, ((0), (0), (0, 0), $-1$)), (1), (2, ((0, 0), (0), (0, 0, 0), $-1$)), (0)))

$\langle\langle\psi'((3), (3), (0))\rangle\rangle=$ (2, ((2, ((0), (1), (0, 0), (1), (0, 0, 0), $-1$)), (0)))

\item $\comp{2}{N}_{\node}(((3), (2)))= -1$

$\langle\langle\psi((3), (2))\rangle\rangle=$ (2, ((2, ((0), (0), (0, 0), $-1$)), (1), (2, ((0, 0), $-1$)), (0)))

$\langle\langle\psi'((3), (2))\rangle\rangle=$ (2, ((2, ((0), (1), (0, 0), (0))), (0)))

\item $\comp{2}{N}_{\node}(((3), (1)))= (0, 0, 0)$

$\langle\langle\psi((3), (1))\rangle\rangle=$ (2, ((2, ((0), (0), (0, 0), $-1$)), (1), (2, ((0, 0), $-1$)), $-1$))

$\langle\langle\psi'((3), (1))\rangle\rangle=$ (2, ((2, ((0), (1), (0, 0), (0))), $-1$))

\item $\comp{2}{N}_{\node}(((3), (1), (2)))= -1$

$\langle\langle\psi((3), (1), (2))\rangle\rangle=$ (2, ((2, ((0), (0), (0, 0), $-1$)), (1), (2, ((0, 0, 0), (0))), (0)))

$\langle\langle\psi'((3), (1), (2))\rangle\rangle=$ (2, ((2, ((0), (1), (0, 0), (0), (0, 0, 0), (0))), (0)))

\item $\comp{2}{N}_{\node}(((3), (1), (1)))= (0, 0, 0, 0)$

$\langle\langle\psi((3), (1), (1))\rangle\rangle=$ (2, ((2, ((0), (0), (0, 0), $-1$)), (1), (2, ((0, 0, 0), (0))), $-1$))

$\langle\langle\psi'((3), (1), (1))\rangle\rangle=$ (2, ((2, ((0), (1), (0, 0), (0), (0, 0, 0), (0))), $-1$))

\item $\comp{2}{N}_{\node}(((3), (1), (1), (0)))= -1$

$\langle\langle\psi((3), (1), (1), (0))\rangle\rangle=$ (2, ((2, ((0), (0), (0, 0), $-1$)), (1), (2, ((0, 0, 0), (0), (0, 0, 0, 0), $-1$)), (0)))

$\langle\langle\psi'((3), (1), (1), (0))\rangle\rangle=$ (2, ((2, ((0), (1), (0, 0), (0), (0, 0, 0), (0), (0, 0, 0, 0), $-1$)), (0)))

\item $\comp{2}{N}_{\node}(((3), (1), (0)))= -1$

$\langle\langle\psi((3), (1), (0))\rangle\rangle=$ (2, ((2, ((0), (0), (0, 0), $-1$)), (1), (2, ((0, 0, 0), $-1$)), (0)))

$\langle\langle\psi'((3), (1), (0))\rangle\rangle=$ (2, ((2, ((0), (1), (0, 0), (0), (0, 0, 0), $-1$)), (0)))

\item $\comp{2}{N}_{\node}(((3), (0)))= -1$

$\langle\langle\psi((3), (0))\rangle\rangle=$ (2, ((2, ((0), (0), (0, 0), $-1$)), (0)))

$\langle\langle\psi'((3), (0))\rangle\rangle=$ (2, ((2, ((0), (1), (0, 0), $-1$)), (0)))

\item $\comp{2}{N}_{\node}(((2)))= -1$

$\langle\langle\psi((2))\rangle\rangle=$ (2, ((2, ((0), $-1$)), (2)))

$\langle\langle\psi'((2))\rangle\rangle=$ (2, ((2, ((0), (0))), (0)))

\item $\comp{2}{N}_{\node}(((1)))= (0, 0)$

$\langle\langle\psi((1))\rangle\rangle=$ (2, ((2, ((0), $-1$)), (1)))

$\langle\langle\psi'((1))\rangle\rangle=$ (2, ((2, ((0), (0))), $-1$))

\item $\comp{2}{N}_{\node}(((1), (2)))= -1$

$\langle\langle\psi((1), (2))\rangle\rangle=$ (2, ((2, ((0), $-1$)), (1), (2, ((0, 0), (0))), (0)))

$\langle\langle\psi'((1), (2))\rangle\rangle=$ (2, ((2, ((0), (0), (0, 0), (0))), (0)))

\item $\comp{2}{N}_{\node}(((1), (1)))= (0, 0, 0)$

$\langle\langle\psi((1), (1))\rangle\rangle=$ (2, ((2, ((0), $-1$)), (1), (2, ((0, 0), (0))), $-1$))

$\langle\langle\psi'((1), (1))\rangle\rangle=$ (2, ((2, ((0), (0), (0, 0), (0))), $-1$))

\item $\comp{2}{N}_{\node}(((1), (1), (0)))= -1$

$\langle\langle\psi((1), (1), (0))\rangle\rangle=$ (2, ((2, ((0), $-1$)), (1), (2, ((0, 0), (0), (0, 0, 0), $-1$)), (0)))

$\langle\langle\psi'((1), (1), (0))\rangle\rangle=$ (2, ((2, ((0), (0), (0, 0), (0), (0, 0, 0), $-1$)), (0)))

\item $\comp{2}{N}_{\node}(((1), (0)))= -1$

$\langle\langle\psi((1), (0))\rangle\rangle=$ (2, ((2, ((0), $-1$)), (1), (2, ((0, 0), $-1$)), (0)))

$\langle\langle\psi'((1), (0))\rangle\rangle=$ (2, ((2, ((0), (0), (0, 0), $-1$)), (0)))

\item $\comp{2}{N}_{\node}(((0)))= -1$

$\langle\langle\psi((0))\rangle\rangle=$ (2, ((2, ((0), $-1$)), (0)))

$\langle\langle\psi'((0))\rangle\rangle=$ (2, ((2, ((0), $-1$)), (0)))

\item $\comp{1}{N}$ contains (0)

$\langle\langle\psi(0)\rangle\rangle=$ (1, ((0)))

$\langle\langle\psi'(0)\rangle\rangle=$ (1, ((0)))

\end{etaremune}

 \end{myexample}


\begin{mylemma}
  \label{lem:XT_division_Q}
  Suppose $X,T$ are level $\leq 2$ trees, $\theta : \rep(X) \to \rep(T)$ is a function in $\mathbb{L}$, order-preserving and continuous. Then there exists a triple
  \begin{displaymath}
    (Q,\pi,\vec{\gamma})
  \end{displaymath}
such that $Q$ is a level $\leq 2$ tree, $\pi$ factors $(X,T,Q)$, $\vec{\gamma}$ respects $Q$, and
\begin{displaymath}
  \forall h \in \omega_1^{T \uparrow} ~ h^{T,Q}_{\pi}(\vec{\gamma}) = [h \circ \theta]^X.
\end{displaymath}
\end{mylemma}
\begin{proof}
For $d \in \se{1,2}$, let $A^{d} = \set{x \in \comp{1}{X}}{ \theta (1, (x) ) \in \se{d} \times \rep(\comp{d}{T})}$. By order preservation and continuity of $\theta$, $A^{1}$ is a $\prec^{\comp{1}{X}}$-initial segment of $\comp{1}{X}$. 
For $x \in A^{1}$, let $t^{1}_x \in \comp{1}{T}$ be such that
\begin{displaymath}
\theta(1, (x)) = (1, (t^{1}_x)).
\end{displaymath}
The existence of $t^{1}_x$ follows from the fact that $(1,(x))$ has cofinality $\omega$ in $\rep(X)$. For $x \in A^{2}$, let $\mathbf{t}^{2}_x = (t^{2}_x, S^{2}_x, \vec{s}^2 _x)\in \desc(\comp{2}{T})$, $\vec{s}^2_x = (s^2_{x, i})_{i < \lh(\vec{s}^2_x)}$ and $\vec{\beta}^{2}_x = ({\beta}^{2}_{x,s})_{s \in S^{2}_x\cup \se{\emptyset}}$ be such that
\begin{displaymath}
\theta(1, (x)) = (2, \vec{\beta}_x^{2}\oplus_{\comp{2}{T}} t^{2}_x).
\end{displaymath} 
For $x \in \dom(\comp{2}{X})$, let $\comp{2}{X}(x)  = (W_x, w_{x})$. 
By order preservation and continuity of $\theta$, we can find $\mathbf{t}_x = (t_x,S_x, \vec{s}_x) \in \desc(T)$  and $\theta_x \in \mathbb{L}$ such that for $\mu^{W_x}$-a.e.\ $\vec{\alpha}$,
  \begin{displaymath}
    \theta(2, \vec{\alpha} \oplus_{\comp{2}{X}} x) = (2,\theta_x (\vec{\alpha}) \oplus_{\comp{2}{T}} t_x).
  \end{displaymath}
Let $\vec{s}_x = (s_{x,i})_{i < \lh(\vec{s}_x)}$, $s_x = s_{x, \lh(\vec{s}_x)-1}$. 
Let $[\theta_x]_{\mu^{W_x}} = \vec{\beta}_x = ({\beta}_{x,s})_{s\in S_x \cup \se{\emptyset}}$, $\theta_x(\vec{\alpha}) = ({\theta}_{x,s}(\vec{\alpha}))_{s\in S_x \cup \se{\emptyset}}$, so that ${\beta}_{x,s} = [{\theta}_{x,s}]_{\mu^{W_x}}$. In particular, ${t}_{\emptyset} = \emptyset$, $\beta_{\emptyset,\emptyset} = \omega_1$, and $t_x \neq \emptyset$ when $x \neq \emptyset$. Fixing $x$, the map $s \mapsto \beta_{x,s}$ is order preserving with respect to $\prec^{S_x}$ and $<$. 
Let
\begin{displaymath}
B_x = \set{s \in {S}_x}{\beta_{x,s} < \omega_1}.
\end{displaymath}
So $B_x$ is closed under $\prec^{S_x}$. 
For $s \in S_x \setminus B_x$, let $(P_{x,s}, \vec{p}_{x,s})$ be the potential partial level $\leq 1$ tower induced by $\beta_{x,s}$,  $\vec{p}_{x,s} = (p_{x,s,i})_{i < \lh(\vec{p}_{x,s})}$, $p_{x,s} = p_{x,s, \lh(\vec{p}_{x,s}-1)}$,
 let $(\seed^{W_x}_{w_{x,s,i}})_{i < v_{x,s}}$ be the signature of $\beta_{x,s}$, let $(\gamma_{x,s,i})_{i \leq v_{x,s}}$ be the approximation sequence of $\beta_{x,s}$, and let $\cf^{\mathbb{L}}(\beta_{x,s}) = \seed^{W_x}_{w_{x,s}}$ if $\cf^{\mathbb{L}}(\beta_{x,s}) >\omega$.   Let $\sigma_{x,s}$ factor $(P_{x,s}, W_x)$, where $\sigma_{x,s}(p_{x,s,i}) = w_{x,s,i}$ for $i < v_{x,s}$. 
Let
\begin{align*}
D_x  &  = \set{s \in {S}_x \setminus B_x}{{\beta}_{x,s} \text{ is essentially continuous}}, \\
  E_x  &  = S_x \setminus (B_x\cup  D_x).
\end{align*}
Thus, $v_{x,s} = \card(P_{x,s})$. For $s \in D_x$, $v_{x,s} = \lh(\vec{p}_{x,s})$;  for $s \in E_x$, $v_{x,s} = \lh(\vec{p}_{x,s})-1$. 

By order preservation and continuity of $\theta$, we can see that for $x \in \dom(\comp{2}{X})$, 
\begin{enumerate}
\item If $t_x$ is of continuous type, then $\theta_{x,s_x}$ has uniform cofinality $\ucf(\comp{2}{X}[x])$.
\item If $t_x$ is of continuous type and $x = \emptyset \vee w_{x^{-}}$ does not appear in $\sign(\theta_{x,s})$ for any $s \in S_x \setminus \se{s_x}$, then $\theta_{x,s_x}$ is essentially discontinuous and thus $s_x \in E_x$.
\item If $t_x$ is of discontinuous type then
  \begin{enumerate}
  \item if $w_x=-1$, then $s_x = -1$;
  \item if $w_x \neq -1$, then $s_x \neq -1$, $\theta_{x, s_x^{-}}$ has uniform cofinality $w_x^{-}$, and thus $w_{x,s_x^{-}} = w_x^{-}$.
  \end{enumerate}
\end{enumerate}



\begin{myclaim}\label{claim:XTQ_factor_cont}
  Suppose $x, x' \in \dom(\comp{2}{X})$, $x = (x')^{-}$,  $t_x $ is of continuous type,  and the contraction of $((w_{x,s_{x,j},i})_{i < v_{x,s_{x,j}}})_{j< \lh(t_x)}$ is $(w_{x \res i})_{i < \lh(x)}$. 
Then
  \begin{enumerate}
  \item $t_x = t_{x'}$.  
  \item For any $s \in S_x \setminus \se{s_x}$, $\beta_{x,s} = \beta_{x',s}$.
  \item $(w_{x,s_x,i}, \gamma_{x,s_x,i})_{i < v_{x,s_x}}$ is a proper initial segment of $(w_{x',s_x,i}, \gamma_{x',s_x,i})_{i < v_{x',s_x}}$. Hence,  $P_{x,s_x}$ is a proper subtree of $P_{x', s_x}$ and $\vec{p}_{x,s_x}$ is an initial segment of $\vec{p}_{x', s_x}$.
    \item  $\sigma_{x', s_x} (p_{x,s_x}) = w_x$.  In particular,  the contraction of $((w_{x', s_{x',j},i})_{i < v_{x', s_{x',j}}})_{j<\lh(t_x)}$ is $(w_{x \res i})_{i \leq  \lh(x)}$. 
  \end{enumerate}
\end{myclaim}
\begin{proof}
By order preservation and  continuity  of $\theta$, $\mathbf{t}_x = \mathbf{t}_{x'}$ and  for $\mu^{W_x}$-a.e.\ $\vec{\alpha}$,
  \begin{enumerate}
  \item for any $s \in {S}_x \setminus \se{s_x}$, if $\vec{\alpha}' \in [\omega_1]^{W_{x'}\uparrow }$ extends $\vec{\alpha}$ then ${\theta}_{x,s}(\vec{\alpha}) = {\theta}_{x', s}(\vec{\alpha}')$;
  \item $ {\theta}_{x,s_x}(\vec{\alpha}) = \sup\set{{\theta}_{x',s_x}(\vec{\alpha}')}{ \vec{\alpha}' \in [\omega_1]^{W_{x'} \uparrow} \text{ extends } \vec{\alpha}}$.
  \end{enumerate}
Thus, $\beta_{x,s} = \beta_{x',s} $ for any $s \in S_x \setminus \se{s_x}$, and $j_{\sup}^{W_x,W_{x'}}(\beta_{x,s_x}) \leq \beta_{x' , s_x} <  j^{W_x,W_{x'}}(\beta_{x,s_x})$.   As $t_x = t_{x'}$ is of continuous type and $w_x$ does not appear in $\sign(\theta_{x',s})$ for any $s \in S_{x'} \setminus \se{s_{x'}}$, $\theta_{x', s_x}$ is essentially discontinuous, giving $j_{\sup}^{W_x,W_{x'}} (\beta_{x,s_x} )  \neq  \beta_{x' , s_x}$. 
We can then apply Lemma~\ref{lem:ordinal_division_blocks} to show that the partial finite level $\leq 1$ tower induced by $\beta_{x,s_x}$ is a proper initial segment of that induced by $\beta_{x', s_x}$, and $ w_{x', s_x, v_{x,s}} = w_x $. 
\end{proof}

\begin{myclaim}
  \label{claim:XTQ_factor_discont}
  Suppose $x, x' \in \dom(\comp{2}{X})$, $x = (x')^{-}$, $t_x$ is of discontinuous type,  and the contraction of $((w_{x,s_{x,j},i})_{i < v_{x,s_{x,j}}})_{j< \lh(t_x)}$ is $(w_{x \res i})_{ i < \lh(x)}$. Then
  \begin{enumerate}
  \item  $t_x \subsetneq t_{x'}$.
  \item for any $s \in S_x$, $\beta_{x,s} = \beta_{x',s}$.
  \item $(w_{x,s_x^{-},i}, \gamma_{x,s_x^{-},i})_{i < v_{x,s_x^{-}}}$ is a proper initial segment of $(w_{x',s_x,i}, \gamma_{x',s_x,i})_{i < v_{x',s_x}}$. Hence,  $P_{x,s_x^{-}}$ is a proper subtree of $P_{x', s_x}$ and $\vec{p}_{x,s_x^{-}}$ is an initial segment of $\vec{p}_{x', s_x}$.
  \item $\sigma_{x', s_x} (p_{x, s_x^{-}}) = w_x$. 
In particular,  the contraction of $((w_{x', s_{x',j},i})_{i < v_{x', s_{x',j}}})_{j<\lh(t_{x'})}$ is $(w_{x \res i})_{i \leq  \lh(x)}$. 
  \end{enumerate}
\end{myclaim}
\begin{proof}
    By order preservation and  continuity  of $\theta$, $t_{x } \subsetneq t_{x'}$ and for $\mu^{W_x}$-a.e.\ $\vec{\alpha}$, 
  \begin{enumerate}
  \item for any $s \in {S}_x$, if $\vec{\alpha}'$ extends $\vec{\alpha}$ then ${\theta}_{x,s}(\vec{\alpha}) = {\theta}_{x', s}(\vec{\alpha}')$;
  \item $ {\theta}_{x,s_x^{-}}(\vec{\alpha}) = \sup\set{{\theta}_{x',s_x}(\vec{\alpha})}{ \vec{\alpha}' \text{ extends } \vec{\alpha}}$.
  \end{enumerate}
The rest is similar to the proof of Claim~\ref{claim:XTQ_factor_cont}. 
\end{proof}

Let
\begin{displaymath}
\phi^1: \set{{\beta}^{2}_{x,s}}{x \in A^{2}, s \in{S}^{2}_x}  \cup \set{\beta_{x,s}}{x \in \dom(\comp{2}{X}), s \in B_x}\to Z^1
\end{displaymath}
be a bijection such that $Z^1$ is a level-1 tree and $v< v' \eqiv \phi^1(v)  \prec^{Z^1} \phi^1(v')$.  
Let
\begin{displaymath}
\phi^{2} : \{ (w_{x,s,i}, \gamma_{x,s,i})_{i < l} : x \in \dom(\comp{2}{X}),  s \in D_x \cup E_x,    l < \lh(\vec{p}_{x,s}) \}\to Z^{2}\cup \se{\emptyset}
\end{displaymath}
be a bijection such that $Z^{2}$ is a tree of level-1 trees and $v  \subseteq v' \eqiv \phi^{2} (v) \subseteq \phi^{2}(v')$,  $v <_{BK}  v' \eqiv \phi^{2}(v) <_{BK} \phi^{2}(v')$. Let
\begin{displaymath}
Q=(\comp{1}{Q},\comp{2}{Q}),
\end{displaymath}
where $\comp{1}{Q}=Z^1$,  $\comp{2}{Q}$ is a level-2 tree, $\dom(\comp{2}{Q})=Z^2$,
\begin{align*}
\comp{2}{Q} [\phi^{2} ((w_{x,s,i}, \gamma_{x,s,i})_{i < \lh(\vec{p}_{x,s})-1} ) \concat (-1)] & =(P_{x,s},\vec{p}_{x,s})\text{ for }s \in D_x,\\
\comp{2}{Q} [\phi^{2} ((w_{x,s,i}, \gamma_{x,s,i})_{i <\lh(\vec{p}_{x,s})-1} )]& =(P_{x,s},\vec{p}_{x,s}) \text{ for }s \in E_x.
\end{align*}
Let
\begin{displaymath}
\vec{\gamma} = (\comp{d}{\gamma}_q)_{(d,q) \in \dom(Q)}
\end{displaymath}
where $\comp{1}{\gamma}_q = (\phi^1)^{-1}(q)$, $\comp{2}{\gamma}_{\emptyset} = \omega_1$, 
$\comp{2}{\gamma}_q = \gamma_{x,s,l}$ when $q= \phi^{2}( (w_{x,s,i}, \gamma_{x,s,i})_{i \leq l} ) $. For $x \in A^{1}$, let
\begin{displaymath}
\pi(1, x) = (1 , t^{1}_x, \emptyset).
\end{displaymath}
 For $x \in A^{2}$, let
 \begin{displaymath}
\pi(1,x) = (2 , \mathbf{t}^{2}_x, \tau^{2}_x),
\end{displaymath}
where $\tau^{2}_x$ factors $(S^{2}_x, Q,  \emptyset)$,  $\tau^{2}_x (1, s) = (1 , \phi^1( {\beta}^{2}_{x,s}) , \emptyset)$.
For $x \in \dom(\comp{2}{X})$, let
\begin{displaymath}
\pi(2, x) = (2, \mathbf{t}_x, \tau_x),
\end{displaymath}
where $\tau_x$ factors $(S_x, Q, W_x)$, defined as follows:
\begin{displaymath}
  \tau_x(s) =
  \begin{cases}
    (1 , \phi^{1}(\beta_{x,s}), \emptyset ) & \text{ if } s \in B_x ,\\
    (2, (\phi^{2}((w_{x,s,i},\gamma_{x,s,i})_{i <\lh(\vec{p}_{x,s})-1} )\concat (-1), P_{x,s}, \vec{p}_{x,s}), \sigma_{x,s}) & \text{ if } s \in D_x ,\\
 (2,( \phi^{2} (({w_{x,s,i}}, \gamma_{x,s,i})_{i <\lh(\vec{p}_{x,s})-1} ) , P_{x,s} , \vec{p}_{x,s}), \sigma_{x,s} ) & \text{ if } s \in E_x.
  \end{cases}
\end{displaymath}
It is easy to check that $(Q,\pi,\vec{\gamma})$ works for the lemma.
\end{proof}

Note that if $\pi$ factors $\Pi^1_2$-wellfounded trees $(X,T)$, then $\llbracket d,x \rrbracket_X \leq \llbracket \pi(d,x) \rrbracket_T$ for any $(d,x) \in \dom(X)$. We say that $\pi$ \emph{minimally factors} $(X,T)$ iff $\pi$ factors $(X,T)$, $X,T$ are both $\Pi^1_2$-wellfounded and $\llbracket d,x \rrbracket_X = \llbracket \pi(d,x) \rrbracket_T$ for any $(d,x) \in \dom(X)$. In particular, if $T,Q$ are both $\Pi^1_2$-wellfounded, then $\id_{T,*}$ minimally factors $(T, T \otimes Q)$. 
In the assumption of Lemma~\ref{lem:XT_division_Q}, if $X,T$ are $\Pi^1_2$-wellfounded and the map $\theta$  is a bijection between $\rep(X)$ and $\rep(T)$, its proof constructs $\pi$ which minimally factors $(X, T \otimes Q)$.
 This entails the comparison theorem between $\Pi^1_2$-wellfounded trees. 

\begin{mytheorem}
  \label{thm:factor_ordertype_embed_equivalent_lv2}
  Suppose $X$, $T$ are $\Pi^1_2$-wellfounded level $\leq 2$ trees. Then there exists $(Q,\pi)$ such that 
$Q$ is $\Pi^1_2$-wellfounded and $\pi$ minimally factors $(X, T \otimes Q)$.
\end{mytheorem}

We shall see in Section~\ref{sec:03} that the minimally of factoring maps between $\Pi^1_2$-wellfounded trees corresponds exactly to the Dodd-Jensen property of iterations of mice.

Suppose $Q,Q'$ are finite level $\leq 2$ trees, $Q$  is a proper subtree of $Q'$, 
$(W_i,w_i)_{i \leq m'}$ is a partial  level $\leq 1$ tower, $m \leq m'$, 
$\mathbf{C}  \in \desc(T,Q,({W}_m,(w_i)_{i \leq m}))$, 
 $ \mathbf{C}' \in \desc(T,Q',(W_{m'}, (w_i)_{i \leq m'}) ) \setminus \desc(T,Q,(W_m,(w_i)_{i \leq m}) ) $. 
 Define
\begin{displaymath}
  \mathbf{C}=\mathbf{C}'\res (T,Q)
\end{displaymath}
iff  $\mathbf{C}' \prec \mathbf{C}$ and $\bigcup_{m \leq k \leq m'}\set{\mathbf{C}^{*}\in \desc(T,Q,(W_k, (w_i)_{i \leq k}) )}{\mathbf{C}' \prec \mathbf{C}^{*} \prec \mathbf{C}} = \emptyset$. A purely combinatorial argument shows that  $\mathbf{C}= \mathbf{C}'\res (T,Q)$ iff $\mathbf{C}$ and $\mathbf{C}'$ are both of degree 2 and putting  $\mathbf{C} = (2, \mathbf{t}, \tau)$, $\mathbf{C}' = (2,\mathbf{t}', \tau')$, $\mathbf{t} = (t,S,\vec{s})$, $\vec{s} = (s_i)_{i < \lh(\vec{s})}$, $\mathbf{t}' = (t',S',\vec{s}')$, $k=\lh(t)$, then either
\begin{enumerate}
\item $t$ is of continuous type, $\mathbf{C}^{-} \iniseg \mathbf{C}$, $\tau(s_{k-1})=\tau'(s_{k-1}) \res (Q,W_{m'})$, or 
\item $t$ is of discontinuous type, $\mathbf{C} \iniseg \mathbf{C}'$, $\tau(s_k^{-}) = \tau'(s_k) \res (Q, W_{m'})$.  
\end{enumerate}

As a corollary to Lemma~\ref{lem:W_desc_extension} and Lemma~\ref{lem:QW_description_extension_another}, the $\res (T,Q)$ operator inherits the following continuity property.
\begin{mylemma}
  \label{lem:TQW_description_extension_another}
Suppose $Q,Q',W,W'$ are finite, $Q$ is a  level $\leq 2$ proper subtree of $Q'$, $W$ is a (not necessarily proper) level-$1$ subtree of $W'$.  
Suppose $\mathbf{C} = (2,\mathbf{t},\tau) \in \desc(T,Q,W)$, $\mathbf{C}' = (2,\mathbf{t}',\tau') \in \desc(T,Q',W')$, $\mathbf{C}= \mathbf{C}'\res (T,Q)$. 
Suppose $E \in \mu_{\mathbb{L}}$ is a club, $\eta \in E'$ iff $\eta \in E$ and $E \cap \eta$ has order type $\eta$. 
  Then 
    for any $h \in \omega_1^{T \uparrow}$, for any $\vec{\beta} \in [E']^{Q \uparrow}$, 
    \begin{displaymath}
      j^{W,W'}\circ  h_{\mathbf{C}}^{Q} (\vec{\beta})  = {\sup}
      \{ h_{\mathbf{C}'}^{Q'} (\vec{\gamma}) :  \vec{\gamma}\in [E]^{Q' \uparrow}, \vec{\gamma} \text{ extends }\vec{\beta}\} . 
    \end{displaymath}
\end{mylemma}

\subsection{Level-3 description analysis}
\label{sec:level-3-description}

\begin{mydefinition}
  \label{def:extended_R_description}
  Suppose $R$ is a level-3 tree. 
The \emph{constant $R$-description} is $\emptyset$. 
An \emph{$R$-description} is either the constant $R$-description or a triple $(r,Q, \overrightarrow{(d,q,P)})$ such that either $r\in \dom(R) \wedge (Q,\overrightarrow{(d,q,P)}) = R[r]$ or  $r = r^{-} \concat (-1) \wedge r^{-} \in \dom(R)\wedge Q$ is a completion of $R(r^{-}) \wedge (Q,\overrightarrow{(d,q,P)}) = R[r, Q]$. $\desc(R)$ is the set of $R$-descriptions. $(r,Q,\overrightarrow{(d,q,P)})$ is \emph{of discontinuous type} if $r \in \dom(R)$, \emph{of continuous type} otherwise.
If $\mathbf{r}= (r,Q,\overrightarrow{(d,q,P)})$ is of discontinuous type and $Q^{+}$ is a completion of  $(Q, \overrightarrow{(d,q,P)})$, then
\begin{displaymath}
\mathbf{r}\concat (-1,Q^{+}) = (r \concat (-1), Q^{+}, \overrightarrow{(d,q,P)}).
\end{displaymath}
An \emph{extended $R$-description} is either an $R$-description or a triple $(r, Q, \overrightarrow{(d,q,P)})$ such that $(r\concat (-1), Q, \overrightarrow{(d,q,P)})$ is an $R$-description of continuous type. 
$\exdesc(R)$ is the set of  extended $R$-descriptions. 
An extended $R$-description $\mathbf{r}$ is \emph{regular} iff either $\mathbf{r} \in \desc(R)$ of discontinuous type or $\mathbf{r} \notin \desc(R)$.
A \emph{generalized $R$-description} is either $(\emptyset,\emptyset,\emptyset)$ or of the form
\begin{displaymath}
  \mathbf{A} = (\mathbf{r}, \pi, T)
\end{displaymath}
so that $\mathbf{r} = (r, Q, \overrightarrow{(d,q,P)}) \in \desc(R) \setminus\se{\emptyset}$, $T$ is a finite level $\leq 2$ tree, 
$\pi$ factors $(Q,T)$. 
 $\exexdesc(R)$ is the set of generalized $R$-descriptions.
\end{mydefinition}

Suppose $(Q, {(d,q,P)})$ is a partial level $\leq 2$ tree. We define
\begin{displaymath}
  \ucf^{*}(Q, {(d,q,P)}) =
  \begin{cases}
    (0,-1,\emptyset) & \text{if } \ucf(Q,{(d,q,P)}) = (0,-1) ,\\
    (1, q^{*},\emptyset) & \text{if } \ucf(Q,{(d,q,P)}) = (1, q^{*}), \\ 
    (2, \mathbf{q}^{*}, \id_{\comp{2}{Q}_{\tree}(q^{*})}) & \text{if } \ucf(Q, {(d,q,P)}) = (2, \mathbf{q}^{*}),  \\
    &\quad\mathbf{q}^{*} = (q^{*}, P^{*}, \vec{p}^{*}).\\
  \end{cases}
\end{displaymath}
Thus, $\ucf^{*}(Q, {(d,q,P)}) \in \se{(0,-1,\emptyset)} \cup \desc(Q, P)$, and 
$\cf(Q, (d,q,P)) = 1$ iff $\ucf^{*}(Q, (d,q,P)) = \min(\prec^{Q, P})$. If $\cf(Q, (d,q,P)) = 2$, let
\begin{displaymath}
  \ucf^{-} (Q, (d,q,P)) = \pred_{\prec^{Q, P}} ( \ucf^{*} (Q, (d,q,P))).
\end{displaymath}
$\ucf^{-}(Q,(d,q,P))$ can be computed in the following concrete way. If $d=1$, then $\ucf^{-}(Q,(1,q,\emptyset)) = (1, \pred_{\prec ^{\comp{1}{Q} \cup \se{q}}}(q), \emptyset)$; if $d=2$, then $\ucf^{-}(Q,(2,q,P)) = (2, \mathbf{q}', \id_P)$, where $\mathbf{q}'=(q',P, \vec{p}) \in \desc(Q)$, $q'$ is the $<_{BK}$-maximum of $\comp{2}{Q}\se{q,-}$.
If $Q^{*}$ is a completion of $(Q, (d,q,P))$ and $\mathbf{D} = (1, q,\emptyset)$ if $d=1$, $\mathbf{D} = (2, (q, P, \vec{p}), \id_P)$ if $d=2 \wedge \comp{2}{Q}^{*}[q] = (P, \vec{p})$, then
\begin{displaymath}
 \mathbf{D}  = \pred_{\prec^{Q^{*}, P}}( \ucf^{*}(Q, (d,q,P)) )
\end{displaymath}
and
\begin{displaymath}
    \ucf^{-}(Q, (d,q,P)) = \pred_{\prec^{Q^{*},P}}(\mathbf{D}).
\end{displaymath}

Suppose $\mathbf{r} = (r, Q,\overrightarrow{(d,q,P)}) \in \exdesc(R)$, $\lh(r) = k$, $\overrightarrow{(d,q,P)} = (d_i,q_i,P_i)_{1 \leq i \leq \lh(\vec{q})}$. For $F \in (\bolddelta{3})^{R \uparrow}$, define $F_{\mathbf{r}}$ to be a function on $[\omega_1]^{Q \uparrow}$: if  $\mathbf{r} \in \desc(R)$, then $F_{\mathbf{r}} = F_r$; if $\mathbf{r} \notin \desc(R)$, then  $F_{\mathbf{r}} ( \vec{\beta} ) = F_r(\vec{\beta}\res R_{\tree}(r))$.  
If $\vec{\gamma} =(\gamma_r)_{r \in \dom(R)}\in [\bolddelta{3}]^{R \uparrow}$, put $\gamma_{\mathbf{r}} = [F_{\mathbf{r}}]_{\mu^Q}$. If $\mathbf{r} \in \desc(R)$ and $\mathbf{A} = (\mathbf{r}, \pi, T) \in \exexdesc(R)$, put $\gamma_{\mathbf{A}} = \pi^T(\gamma_{\mathbf{r}})$. Put $\gamma_{\emptyset} = \gamma_{(\emptyset,\emptyset,\emptyset)} = \bolddelta{3}$.  Thus, if $\mathbf{r} \in \desc(R)$ is of discontinuous type, then $\gamma_{\mathbf{r}} = \gamma_r$; if $\mathbf{r}\notin \desc(R)$, then $\gamma_{\mathbf{r}} = j^{R_{\tree}(r), Q} (\gamma_r) = \gamma_{(\mathbf{r}, \id_{R_{\tree}(r)}, Q)}$. 
The next lemma computes the remaining case when $\mathbf{r}\in \desc(R)$ is of continuous type, justifying that $\gamma_{\mathbf{r}}$ does not depend on the choice of $F$. 

\begin{mylemma}
  \label{lem:gamma_r_continuous_type}
  Suppose $R$ is a level-3 tree, $\vec{\gamma}  \in [\bolddelta{3}]^{R \uparrow}$, $\mathbf{r} = (r, Q, \overrightarrow{(d,q, P)}) \in \desc(R)$ is of continuous type. Then $\gamma_{\mathbf{r}} = j_{\sup}^{R_{\tree}(r^{-}), Q} (\gamma_{r^{-}})$. 
\end{mylemma}
\begin{proof}
Suppose $\vec{\gamma} = [F]^R$, $F \in (\bolddelta{3})^{R \uparrow}$. 
  Put $\lh(r) = k+1$, $\overrightarrow{(d,q,P)} = (d_i,q_i, P_i)_{1 \leq i \leq k}$. 
We prove the case when $\cf(R(r^{-})) = 2$, the other case being similar. 
Put $R(r^{-}) = (Q^{-}, (d, q, P))$, $\pi^{-} = \pi \res \dom(Q^{-})$, so that $Q$ is a completion of $R(r^{-})$, $(d,q,P) = (d_k, q_k, P_k)$. 
Put $\ucf(R(r^{-})) = (d^{*}, \mathbf{q}^{*})$, $\ucf^{-}(R(r^{-})) = (e , \mathbf{z}, \id_P)$.

We firstly show the $\geq$ direction. Suppose $\delta = [G]_{\mu^{Q^{-}}}< \gamma_{r^{-}}$, $G \in \admistwobold$. By \Los{}, for $\mu^{Q^{-}}$-a.e.\ $\vec{\beta}$, $G(\vec{\beta}) < F_{r^{-}}(\vec{\beta}) = \sup_{\xi < \comp{d^{*}}{\beta}_{\mathbf{q}^{*}}}F_r(\vec{\beta} \concat (\xi))$, where $\vec{\beta} \concat (\xi)$ is a tuple extending $\vec{\beta}$ whose entry indexed by $(d,q)$ is $\xi$. Let $H(\vec{\beta})$ be the least $\xi <\comp{d}{\beta}_{\mathbf{q}^{*}}$ satisfying $G(\vec{\beta}) < F_r(\vec{\beta} \concat (\xi))$. 
By Lemmas~\ref{lem:level_2_desc_cofinal_in_next} and~\ref{lem:level_2_desc_order}, there is $h : \omega_1 \to \omega_1$ such that $h \in \mathbb{L}$ and for $\mu^{Q^{-}}$-a.e.\ $\vec{\beta}$, $H(\vec{\beta}) < j^P(h) (\comp{e}{\beta}_{\mathbf{z}}) < \comp{d}{\beta}_q$. Thus, for $\mu^Q$-a.e.\ $\vec{\beta}$, $G(\vec{\beta} \res \dom(Q^{-})) < F_r(\vec{\beta})$. Thus, $j^{Q^{-},Q}(\delta) < [F_{\mathbf{r}}]_{\mu^Q}$.  

We secondly show the $\leq$ direction. 
Suppose $\delta =[G]_{\mu^Q} < [F_{\mathbf{r}}]_{\mu^Q}$, $G \in \admistwobold$. Then for $\mu^Q$-a.e.\ $\vec{\beta}$, $G(\vec{\beta}) < F_r(\vec{\beta}) = \sup_{\xi < \comp{d}{\beta}_q} F_r (\vec{\beta} \res \dom(Q^{-}) \concat (\xi))$. Let $H(\vec{\beta})$ be the least $\xi < \comp{2}{\beta}_v$ satisfying $G(\vec{\beta}) < F_r (\vec{\beta} \res \dom(Q^{-}) \concat (\xi))$. By Lemmas~\ref{lem:factor_SQW} and~\ref{lem:level_2_desc_order} again, there is $h: \omega_1 \to \omega_1$ such that $h \in \mathbb{L}$ and for $\mu^Q$-a.e.\ $\vec{\beta}$, $H(\vec{\beta}) < j^P(h)(\comp{e}{\beta}_{\mathbf{z}}) < \comp{d}{\beta}_q$. Thus, for $\mu^Q$-a.e.\ $\vec{\beta}$, $G(\vec{\beta}) < j^{Q^{-},Q}(\eta)$, where $\eta$ is represented modulo $\mu^{Q^{-}}$ by the function $\vec{\beta} \mapsto F_r(\vec{\beta} \concat j^P(h)(\comp{e}{\beta}_{\mathbf{z}}))$. Since $\eta < \gamma_{{r}^{-}}$, we have $\delta < j^{Q^{-},Q}_{\sup}(\gamma_{{r}^{-}})$. 
\end{proof}

Define
\begin{displaymath}
  C^{*} = \set{ \xi < \bolddelta{3}}{\text{for any finite level $\leq 2$ tree }Q,  j^Q_{\sup} (\xi) = \xi  }.
\end{displaymath}
Assuming $\boldDelta{2}$-determinacy, Lemma~\ref{lem:level_2_embedding_bounded_by_delta13} implies that $C^{*}\cap \kappa_3^x$ has order type $\kappa_3^x$, and hence $C^{*}$ has order type $\bolddelta{3}$. A tuple $\vec{\gamma}$ is said to \emph{strongly respect}  $R$ iff $\vec{\gamma} \in [C^{*}]^{R \uparrow}$.
In most applications, we are only concerned with $\vec{\gamma}$ strongly respecting $R$. 
In that case, the techniques in Section~\ref{sec:more-level-2} helps to decide the ordering of $\gamma_{\mathbf{r}}$ for different $\mathbf{r} \in \exdesc(R)$. The results are in parallel to Lemma~\ref{lem:level_2_beta_q_order}.

Define $\corner{(\emptyset,\emptyset,\emptyset)} = \corner{\emptyset} =  \emptyset$. For $\mathbf{A} = (\mathbf{r}, \pi, T) \in \exexdesc(R)$, $\mathbf{r} = (r,Q, \overrightarrow{(d,q,P)})$, $\lh(r) = k$, define 
\begin{displaymath}
  \corner{ \mathbf{A} } =
  \begin{cases}
    (r(0), \llbracket \pi(d_1,q_1) \rrbracket_T, r(1), \dots, \llbracket \pi(d_{k-2},q_{k-2}) \rrbracket_T, r(k-2), -1) \\
 \qquad\text{ if } r \text{ is of continuous type}, (\pi,T) \text{ is continuous at } (d_{k-1},q_{k-1}),\\
    (r(0), \llbracket \pi(d_1,q_1) \rrbracket_T, r(1), \dots, \llbracket \pi(d_{k-2},q_{k-2}) \rrbracket_T, r(k-2), \llbracket \pred(\pi,T, (d_{k-1},q_{k-1})) \rrbracket_T) \\
 \qquad\text{ if } r \text{ is of continuous type}, (\pi,T) \text{ is discontinuous at } (d_{k-1},q_{k-1}),\\
    (r(0), \llbracket \pi(d_1,q_1) \rrbracket_T, r(1), \dots, \llbracket \pi(d_{k-1},q_{k-1}) \rrbracket_T, r(k-1), -1) \\
 \qquad\text{ if } r \text{ is of discontinuous type}, (\pi,T) \text{ is continuous at } \ucf(R(r)),\\
 (r(0), \llbracket \pi(d_1,q_1) \rrbracket_T, r(1), \dots, \llbracket \pi(d_{k-1},q_{k-1}) \rrbracket_T, r(k-1), \llbracket \pred(\pi,T, \ucf(R(r))) \rrbracket_T) \\
 \qquad\text{ if } r \text{ is of discontinuous type}, (\pi,T) \text{ is discontinuous at } \ucf(R(r)).\\
  \end{cases}
\end{displaymath}
and define $\corner{\mathbf{r}} = \corner{(\mathbf{r}, Q, \id_Q)}$. If $\mathbf{r}$ is of discontinuous type and $Q^{+}$ is a completion of $Q$, define $\corner{ (r, Q^{+}, \overrightarrow{(d,q,P)}) } = \corner{(\mathbf{r}, Q^{+}, \id_Q)}$. 
 For $\mathbf{A},\mathbf{A}' \in \exexdesc(R)$, define 
\begin{displaymath}
\mathbf{A} \prec \mathbf{A}'
\end{displaymath}
iff $\corner{\mathbf{A}} <_{BK} \corner{\mathbf{A}'}$; define
\begin{displaymath}
\mathbf{A} \sim \mathbf{A}'
\end{displaymath}
iff $\corner{\mathbf{A}} = \corner{\mathbf{A}'}$. 
For $\mathbf{r}, \mathbf{r}' \in \exdesc(R)$, define 
\begin{displaymath}
\mathbf{r} \prec \mathbf{r}'
\end{displaymath}
iff $\corner{\mathbf{r}} <_{BK} \corner{\mathbf{r}'}$; define
\begin{displaymath}
\mathbf{r} \sim \mathbf{r}'
\end{displaymath}
iff $\corner{\mathbf{r}} = \corner{\mathbf{r}'}$.  All relations are effective. Define $\prec^R_{*} = \prec \res \exexdesc(R)$, $\sim^R_{*} = \sim \res \exexdesc(R)$
 $\prec^R = \prec \res \exdesc(R)$, $\sim^R = \sim \res \exdesc(R)$. For $r , r' \in \dom(R)$, define $r \prec^R r'$ iff $(r) \concat R[r] \prec (r') \concat R[r']$.

\begin{mylemma}
  \label{lem:gamma_r_order}
  Suppose $R$ is a level-3 tree, $\mathbf{A}, \mathbf{A}' \in \exdesc(R)$,  $\vec{\gamma}$ strongly respects $R$.  Then ${\mathbf{A}} \prec^R{\mathbf{A}}'$ iff $\gamma_{\mathbf{A}} < \gamma_{\mathbf{A}'}$; ${\mathbf{A}} \sim^R{\mathbf{A}}'$ iff $\gamma_{\mathbf{A}} = \gamma_{\mathbf{A}'}$.
\end{mylemma}
\begin{proof}
Put $\mathbf{A} = (\mathbf{r}, T, \pi)$, $\mathbf{A}' = (\mathbf{r}', T', \pi')$. 
Recall our convention that $\gamma_{(\emptyset,\emptyset,\emptyset)} = \bolddelta{3}$. The lemma is trivial if $\mathbf{r} = \emptyset$ or $\mathbf{r}' = \emptyset$. Assume now $\mathbf{r}, \mathbf{r}' \neq \emptyset$.
Put $\mathbf{r}= (r, Q, \overrightarrow{(d,q,P)})$, $\overrightarrow{(d,q,P)} = (d_i,q_i, P_i)_{1 \leq i \leq  \lh(\vec{q})}$,  $k=\lh(r)$, $\mathbf{r}'= (r',Q',\overrightarrow{(d',q',P')})$, $\overrightarrow{(d',q',P')} = (d'_i,q'_i,P'_i)_{1 \leq i \leq \lh(\vec{q}')}$, $k' = \lh(r')$. Assume $\vec{\gamma} = [F]^R$, $F \in (C^{*})^{R \uparrow}$. 

Firstly, we  prove that $\mathbf{A} \sim \mathbf{A}'$ implies $\gamma_{\mathbf{A}} = \gamma_{\mathbf{A}'}$. 

Case 1: $r = r'$ is of continuous type. 

Put $Q^{-} = R_{\tree}(r^-)$.

Subcase 1.1: $(\pi,T)$ is continuous at $(d_{k-1}, q_{k-1})$. 

Then $\llbracket \pi(d,q) \rrbracket_T = \llbracket \pi'(d,q) \rrbracket_{T'}$ for any $(d,q ) \in \dom(Q^{-})$. 
Put $\tau = \pi \res \dom(Q^{-})$, $\tau' = \pi' \res \dom(Q^{-})$. 
By Lemma~\ref{lem:gamma_r_continuous_type}, $\gamma_{\mathbf{A}} = \pi^T_{\sup}(\gamma_{\mathbf{r}}) = \tau^T_{\sup} (\gamma_{r^{-}})$ and $\gamma_{\mathbf{A}'} =(\tau')^{T'}_{\sup}(\gamma_{r^{-}})$. 
Given $\delta  = [G]_{\mu^{Q^{-}}}< \gamma_{r^{-}}$, we need to show that $\tau^T (\delta) < \gamma_{\mathbf{A}'}$.
By Theorem~\ref{thm:factor_ordertype_embed_equivalent_lv2}, there exist $X$ and $\psi$ minimally factoring $(T, T' \otimes X)$. So $\psi \circ \pi (d,q) = \id_{T',*} \circ \pi' (d,q)$ for any $(d,q ) \in \dom(Q^{-})$. We shall actually show that $ \psi^{T',X} \circ \tau^T(\delta) < \gamma_{\mathbf{A}'}$, i.e., $(\psi \circ \tau)^{T',X}(\delta) < \gamma_{\mathbf{A}'}$. By \Los{}, it suffices to show that for $\mu^{T'}$-a.e.\ $\vec{\beta}$, $j^X(G) ( \id^{T',X}_{\psi \circ \tau}(\vec{\beta}) ) < F_{r'}(\vec{\beta}_{\tau'})$. The minimality of $\psi$ implies that $\id^{T',X}_{\psi \circ \tau} (\vec{\beta}) = j^X(\vec{\beta}_{\tau})$. It suffices to show that for $\mu^{T'}$-a.e.\ $\vec{\beta}$, $j^X(G(\vec{\beta}_{\tau})) < F_{r'}(\vec{\beta}_{\tau'})$. Hence, it suffices to show that $\mu^{Q^{-}}$-a.e.\ $\vec{\beta}$, $j^X(G(\vec{\beta})) < F_{r'}(\vec{\beta})$.
As $F_{r'}(\vec{\beta}) \in C^{*}$, this inequality is a consequence of $G(\vec{\beta}) < F_{r'}(\vec{\beta})$, which holds true for $\mu^{Q^{-}}$-a.e.\ $\vec{\beta}$ by assumption.

Subcase 1.2: $(\pi,T)$ is discontinuous at $(d_{k-1}, q_{k-1})$. 

Then $\llbracket \pi(d,q) \rrbracket_T = \llbracket \pi'(d,q) \rrbracket_{T'}$ for any $(d,q) \in \dom(Q)$. 
Let $\tau$ factor $(Q,T)$ where $\tau$ and $\pi$ agree on $\dom(Q^{-})$ and $\tau(d_{k-1},q_{k-1}) = \pred(\pi, T, (d_{k-1}, q_{k-1}))$, and likewise define $\tau'$ which factors $(Q',T)$. By Lemma~\ref{lem:continuous_decomp_level_2},  $\gamma_{\mathbf{A}} = \tau^T_{\sup} \circ j^{Q^{-}, Q} (\gamma_{r^{-}})$ and $\gamma_{\mathbf{A}'} = (\tau')^T_{\sup} \circ j^{Q^{-},Q} (\gamma_{r^{-}})$. Work with $X$ and $\psi$ minimally factoring $(T,T' \otimes X)$ and argue similarly to Subcase 1.1. 

Case 2: $r = r'$ is of discontinuous type. 

Subcase 2.1: $(\pi,T)$ is continuous at $\ucf(R(r))$. 

Then $\llbracket \pi(d,q) \rrbracket_T = \llbracket \pi'(d,q) \rrbracket_{T'}$ for any $(d,q) \in \dom(Q)$ 
and $\gamma_{\mathbf{A}} = \pi^T_{\sup}(\gamma_r)$, $\gamma_{\mathbf{A}'} = (\pi')^T_{\sup} (\gamma_r)$. Argue similarly to Case 1. 

Subcase 2.2: $(\pi,T)$ is discontinuous at $\ucf(R(r))$. 

Let $Q^{+}$ be a completion of $R(r)$ and let $\tau$ factor $(Q^{+},T)$ so that $\tau$ extends $\pi$, 
$\tau(d_{k},q_{k}) = \pred(\pi, T, (d_k,q_k))$, and likewise define $\tau'$ which factors $(Q^{+},T)$. 
Then $\llbracket \tau(d,q) \rrbracket_T = \llbracket \tau'(d,q) \rrbracket_{T'}$ for any $(d,q) \in \dom(Q^{+})$. 
By Lemma~\ref{lem:continuous_decomp_level_2_another}, $\gamma_{\mathbf{A}} = \tau^T_{\sup} \circ j^{Q,Q^{+}}(\gamma_r)$ and $\gamma_{\mathbf{A}'} = (\tau')^{T'}_{\sup} \circ j^{Q,Q^{+}} (\gamma_r)$. Argue similarly to Case 1. 

Case 3: $r \neq r'$. 
Assume $r = r' \concat (-1)$. 

Subcase 3.1: $(\pi,T)$ is continuous at $(d_{k-1},q_{k-1})$. 

It follows from Subcase 1.1 and Subcase 2.1 that  $\gamma_{\mathbf{A}} = \pi^T_{\sup} (\gamma_{r'})$ and $\gamma_{\mathbf{A}'} = (\pi')^{T'}_{\sup} (\gamma_{r'})$. Argue similarly as before.

Subcase 3.2: $(\pi,T)$ is discontinuous at $(d_{k-1}, q_{k-1})$. 

Use a combination of Subcase 1.2 and Subcase 2.2.

Secondly, we prove that $\mathbf{A} \prec \mathbf{A}'$ implies $\gamma_{\mathbf{A}} < \gamma_{\mathbf{A}'}$. 

Case 1: $\corner{\mathbf{A}'}$ is a proper initial segment of $\corner{\mathbf{A}}$. 

Then $\corner{\mathbf{A}'}$ does not end with $-1$. 
We prove the typical case when $r'$ is of discontinuous type. So $r' \subsetneq r$. Let $(Q')^{+}$ be a completion of $R(r')$ and let $\tau'$ factor $((Q')^{+}, T')$ so that $\tau'$ extends $\pi'$, $\tau'(d_k,q_k) = \pred(\pi,T, \ucf(R(r')))$. Then $(Q')^{+} = R_{\tree}(r \res k')$.
We get $\psi$ minimally factoring $(T,T' \otimes X)$ so that $\psi \circ \pi(d,q) = \id_{T',*} \circ \tau'(d,q)$ for any $(d,q) \in \dom( (Q')^{+})$.
 We shall actually show that $\psi^{T',X}(\gamma_{\mathbf{A}}) < \gamma_{\mathbf{A}'}$. 
By \Los{}, it suffices to show that for $\mu^{T'}$-a.e.\ $\vec{\beta}$, $j^X(F_{\mathbf{r}})(\id^{T',X}_{ \psi \circ \pi}(\vec{\beta})) < F_{r'}(\vec{\beta}_{\pi'} )$. The minimality of $\psi$ implies that $\id^{T',X}_{\psi \circ \pi}(\vec{\beta})$ agrees with $j^X(\vec{\beta}_{\tau'})$ on $\dom((Q')^{+})$. 
It suffices to show that for $\mu^Q$-a.e.\ $\vec{\beta}$, $j^X(F_{\mathbf{r}} (\vec{\beta}) ) < F_{r'} (\vec{\beta} \res \dom(Q'))$. As $\ran(F) \subseteq C^{*}$, this would be a consequence of $F_{\mathbf{r}} (\vec{\beta}) < F_{r'}(\vec{\beta} \res \dom(Q'))$, which follows from order preservation of $F$.

Case 2: $\corner{\mathbf{A}'}$ is not a proper initial segment of $\corner{\mathbf{A}}$. 

Similar to Case 1, using the following fact: Suppose $X,X'$ are level $\leq 2$ trees and $\llbracket d_i,x_i \rrbracket_X = \llbracket d_i',x_i' \rrbracket_{X'}$ for $1 \leq i < n$, $\llbracket d_n,x_n \rrbracket_X < \llbracket d_n',x_n' \rrbracket_{X'}$. Then there exist $U$ and $\psi$ minimally factoring $(X, X' \otimes U)$, which implies that for any $\vec{\beta} \in [\omega_1]^{X \uparrow}$, if $\id^{X',U}_{\psi} (\vec{\beta}) = \vec{\delta}$, then $\delta$ and $j^U (\vec{\beta})$ agree on $\set{(d_i,x_i)}{ 1 \leq i < n}$ and  $\comp{d_n}{\delta}_{x_n} < j^U ( \comp{d_n'}{\beta}_{x_n'})$.
\end{proof}

\subsection{Factoring maps between level-3 trees}
\label{sec:fact-betw-level}

Put $\pi \oplus \emptyset = \emptyset$.
Suppose $Y$ is a level-3 tree, $\mathbf{y}=(y,X,\overrightarrow{(e,x,W)} ) \in \desc({Y})$, $\lh(y)=k$, $\overrightarrow{(e,x,W)}=(e_i,x_i,W_i)_{1 \leq i \leq  \lh(\vec{y})}$, 
$\pi$ is a function whose domain contains $\dom(X)$, we put
\begin{displaymath}
 \pi \oplus \mathbf{y} =  \pi \oplus_Y y = ( y(0), \pi(e_1,x_1), y(1),  \ldots,\pi(e_{k-1},x_{k-1}), y(k-1)).
\end{displaymath}
If $l < \lh(y)$, then $\mathbf{y} \res l = (y \res l, Y_{\tree}(y \res l) , (e_i,x_i,W_i)_{1 \leq i \leq l})$.

\begin{mydefinition}
  \label{def:description_TQY}
  Suppose $Y$ is a level-3 tree, $T$ is a level $\leq 2$ tree. 
The  only \emph{$(Y,T,\emptyset)$-description} is $(\emptyset,\emptyset)$, which is called the \emph{constant $(Y,T,*)$-description}. Suppose  $({Q},\overrightarrow{(d,q,P)}) = (Q,(d_i,q_i,P_i))_{1 \leq  i \leq k} $ is a potential partial level $\leq 2$ tower of discontinuous type. 
A 
\emph{$(Y,T,(Q,\overrightarrow{(d,q,P)}))$-description} is of the form
  \begin{displaymath}
    \mathbf{B} =(\mathbf{y},\pi) 
  \end{displaymath}
with the following properties:
\begin{enumerate}
\item $\mathbf{y} \in \desc(Y) \setminus \se{\emptyset}$. Put $\mathbf{y} = (y,X, \overrightarrow{(e,x,W)})$,  $\lh(y) = l$, $\overrightarrow{(e,x,W)} = (e_i,x_i,W_i)_{1 \leq i \leq  \lh(\vec{x})}$. 
\item $\pi$ factors $(X,T,Q)$.
\item The contraction of $(\sign(\pi(e_i,x_i)))_{1 \leq i <  l}$ is $((d_i,q_i))_{1 \leq i < k}$. 
\item If $k > 1$,  $y$ is of continuous type and $(d_{k-1},q_{k-1})$ does not appear in the contraction of $(\sign(\pi(e_i,x_i)))_{1 \leq i < l} \concat ( \sign (\pi(e_{l-1},x_{l-1})^{-}))$, then $\pi(e_{l-1}, x_{l-1})$ is of discontinuous type. 
\item Put $\ucf(X,\overrightarrow{(e,x,W)}) = (e_{*},\mathbf{x}_{*})$. 
  \begin{enumerate}
  \item If $e_{*} = 0$ then $d_k = 0$.
  \item If $e_{*} = 1$ then $\ucf(\pi(1, \mathbf{x}_{*})) = \ucf (Q, \overrightarrow{(d,q,P)})$.
  \item If $e_{*} = 2$, $\mathbf{x}_{*} = (x_{*}, W_{*},\vec{w}_{*}) \in \desc(X)$, then $\ucf(\pi(2, x_{*})) = \ucf(Q, \overrightarrow{(d,q,P)})$. 
  \item If $e_{*} = 2$, $\mathbf{x}_{*} = (x_{*}, W_{*},\vec{w}_{*}) \notin \desc(X)$, then $\ucf^{+}(\pi(2, x_{*})) = \ucf(Q, \overrightarrow{(d,q,P)})$. 
  \end{enumerate}
\end{enumerate}
A $(Y,T,Q)$-description is a $(Y,T,({Q},\overrightarrow{(d',q',P')}))$-description for some potential partial level $\leq 2$ tower $({Q},\overrightarrow{(d',q',P')})$ of discontinuous type. A $(Y,T,*)$-description is a $(Y,T,Q')$-description for some level $\leq 2$ tree $Q'$ or $Q' = \emptyset$. $\desc(Y,T,({Q},\overrightarrow{(d,q,P)}))$, $\desc(Y,T,Q)$, $\desc(Y,T,*)$ denote the sets of relevant descriptions. 
\end{mydefinition}

Similarly to Definition~\ref{def:description_TQW}, if $\mathbf{B} \in \desc(Y,T,Q)$,  then there is at most one $(Q,\overrightarrow{(d,q,P)})$  for which  $\mathbf{B} \in \desc(Y,T, (Q,\overrightarrow{(d,q,P)}))$.
Suppose that $\mathbf{B} = (\mathbf{y},\pi)$ is a $(Y,T,({Q},\overrightarrow{(d,q,P)}))$-description,
$F \in( \bolddelta{3})^{Y \uparrow}$. Then
\begin{displaymath}
  F^T_{\mathbf{B}} : [\omega_1]^{T \uparrow} \to \bolddelta{3}
\end{displaymath}
is the function that sends $[h]^T$ to $[F_{\mathbf{y}} \circ h_{\pi}^{Q}]_{\mu^{Q}}$. Note that $F_{\mathbf{y}}\circ h^Q_{\pi}$ has signature $\sign(Q, \overrightarrow{(d,q,P)})$, is essentially discontinuous, and has uniform cofinality $\ucf(Q, \overrightarrow{(d,q,P)})$.
Of course, $F^T_{\mathbf{B}}$ is meaningful only when $T$ is $\Pi^1_2$-wellfounded.

Assuming $\boldpi{3}$-determinacy, the $\mathbb{L}[T_3]$-measure $\mu^{Y}$ will be defined, and $[F]^Y \to [F_{\mathbf{B}}^{T}]_{\mu^T}$ will represent an element in $\mathbb{L}[j^Y(T_3)]$ modulo $\mu^Y$. Such kind of results related to level-3 ultrapowers are parallel to Section~\ref{sec:more-level-2}. They will be handled in Section~\ref{sec:boldface-level-3_sharp}.


Given a $(Y,T,*)$-description $\mathbf{B} = (\mathbf{y},\pi)$,
define
\begin{displaymath}
  \corner{\mathbf{B}} = \pi \oplus  \mathbf{y}. 
\end{displaymath}
Define
\begin{displaymath}
  \mathbf{B} \prec \mathbf{B}'
\end{displaymath}
iff $\corner{\mathbf{B}} <_{BK} \corner{\mathbf{B}'}$, the ordering on coordinates in $\desc(T,Q,*)$ for some $T,Q$ again according to $\prec$.  The constant $(Y,T,*)$-description $\mathbf{B}_0$ is the $\prec$-greatest, and we have   $\corner{\mathbf{B}_0} = \emptyset$.  Define $\prec^{Y,T}  = \prec \res \desc(Y,T,*)$. 
As a corollary to  Lemma~\ref{lem:TQW_desc_order}, $\prec^{Y,T}$ inherits the following ordering property on $F^{T}_B$. 
\begin{mylemma}
  \label{lem:B_desc_order}
Suppose $(Q_i, (d_i,q_i, P_i))_{1 \leq i \leq m}$ is a partial level $\leq 2$ tower, $\mathbf{B} \in \desc(Y,T,Q_k)$, $\mathbf{B}' \in \desc(Y,T, Q_{k'})$, $k \leq m$, $k' \leq m$,  $\mathbf{B} \prec^{Y,T} \mathbf{B}'$.   Then for any $F \in (\bolddelta{3})^{Y \uparrow}$, for any $\vec{\beta} \in [\omega_1]^{T \uparrow}$, $j^{Q_k, Q_m}\circ F_{\mathbf{B}}^{T}(\vec{\beta}) < j^{Q_{k'},Q_m}\circ F_{\mathbf{B}'}^{T}(\vec{\beta})$.
\end{mylemma}

Suppose $(\vec{Q},\overrightarrow{(d,q,P)})=(Q_i,(d_i,q_i,P_i))_{1 \leq i \leq k}$ is a potential partial level $\leq 2$ tower and $\mathbf{B} = (\mathbf{y}, \pi) \in \desc(Y,T, (Q_k, \overrightarrow{(d,q,P)}) )$.
Define
\begin{displaymath}
  \mathbf{B} \res (Y,T,\emptyset) = \text{the constant $(Y,T,*)$-description. }
\end{displaymath}
Suppose   $\mathbf{y} = (y,X,\overrightarrow{(e,x,W)})$, $0 < \bar{k}<k$. Then
\begin{displaymath}
  \mathbf{B} \res (Y,T,Q_{\bar{k}}) \in \desc(Y,T, (Q_{\bar{k}}, (d_i,q_i,P_i)_{1 \leq i \leq \bar{k}}))
\end{displaymath}
is defined by the following: letting $l$ be the least such that $\pi(e_l,x_l) \notin \desc(T, Q_{\bar{k}}, *)$, $\mathbf{C} \in \desc(T, Q_{\bar{k}}, *)$ be such that $\mathbf{C} = \pi(e_l,x_l) \res (T, Q_{\bar{k}})$, letting $(e_{*},\mathbf{x}_{*}) = \ucf (Y[y \res l])$, $\mathbf{x}_{*} = x_{*}$ if $e_{*} = 1$, $\mathbf{x}_{*} = (x_{*},\dots)$ if $e_{*} = 2$, then
\begin{enumerate}
\item if $\mathbf{C} \neq \pi(e_*,x_*)$, then $\mathbf{B} \res (Y,T,Q_{\bar{k}})  = (\mathbf{y} \res l \concat (-1,X),  \bar{\pi})$, where $\bar{\pi}$ and $\pi$ agree on $Y_{\tree}(y \res l)$, $\bar{\pi}(e_l,x_l) = \mathbf{C}$, and $\bar{\pi}$ factors $(X, T, Q_{\bar{k}})$;
\item if $\mathbf{C} = \pi(e_*,x_*)$, then $\mathbf{B} \res (Y,T,Q_{\bar{k}}) =  (\mathbf{y} \res l,  \pi \res Y_{\tree}(y \res l))$. 
\end{enumerate}

As a corollary to Lemma~\ref{lem:TQW_description_extension_another}, the $\mathbf{B} \res (Y,T,Q_{\bar{k}}) $ operator inherits the following  continuity property. 
\begin{mylemma}
  \label{lem:YTQ_description_extension}
Suppose $Y$ is a level-3 tree, $T$ is a level $\leq 2$ tree, 
$Q$ is a level $\leq 2$ proper subtree of $Q'$. 
Suppose $\mathbf{B} = ( \mathbf{y}, \pi) \in \desc(Y,T,Q)$  and $\mathbf{B}' = ( \mathbf{y}', \pi') \in \desc(Y,T,Q')$, 
 $\mathbf{B} = \mathbf{B}\res (Y,T,Q) $. Suppose $E \in \mu_{\mathbb{L}}$ is a club, $\eta \in E'$ iff $\eta \in E$ and $E \cap \eta$ has order type $\eta$. 
  Then for any $F \in (\bolddelta{3})^{Y \uparrow}$, for any $h \in \omega_1^{T \uparrow}$, for any $\vec{\beta} \in [E']^{Q\uparrow}$, 
  \begin{displaymath}
    F_{\mathbf{y}} \circ h^{Q}_{\pi} (\vec{\beta}) = \sup \{ F_{\mathbf{y}'} \circ h^{Q'}_{\pi'}(\vec{\gamma}) : \vec{\gamma} \in [E]^{Q'\uparrow},\vec{\gamma} \text{ extends }\vec{\beta} \}.
  \end{displaymath}
Hence, the signature and approximation sequence of $F_{\mathbf{y}} \circ h^{Q}_{\pi}$ are proper initial segments of those of $F_{\mathbf{y}'}\circ h^{Q'}_{\pi'}$ respectively.
\end{mylemma}

\begin{mydefinition}
  \label{def:factoring_4}
   Suppose $R,Y$ are level-3 trees, $T$ is a level $\leq 2$ tree. Suppose $\rho : \dom(R) \cup \se{\emptyset}\to \desc(Y,T,*)$ is a function. 
 $\rho$ \emph{factors} $(R,Y,T)$ iff
  \begin{enumerate}
  \item $\rho(\emptyset)$ is the constant $(Y,T,*)$-description. 
  \item For any $r \in \dom(R)$, $\rho(r) \in \desc(Y,T,R[r])$. 
  \item For any $r\concat (a), r \concat (b) \in \dom(R)$, if $a < _{BK} b$ and $R_{\tree}(r\concat (a))  = R_{\tree}(r \concat (b))$ then $\rho(r\concat (a)) \prec \rho(r \concat (b))$. 
  \item For any $r \in \dom(R) $, $\rho(r^{-}) =  \rho(r) \res (Y, T, R_{\tree}(r^{-}))$. 
\end{enumerate}
\end{mydefinition}

\begin{myexample}
     Let $R,Y,T$ be as in Section~\ref{sec:level-3-3}. There is a $(R,Y,T)$-factoring map $\rho$ characterizing the $R$-equivalence class of $\theta_{RY}$. We have for instance,
   \begin{itemize}
   \item $\rho (((0),(1))) = (Y[((1),(0))], \pi)$, where $\pi(2,((0))) = (2, (-1,\se{(0)},((0))), \tau)$, $\tau((0)) = (2, ( (0), \se{(0)}, ((0)) ), \sigma)$, $\sigma((0)) = (0)$.
   \item $\rho(((0),(0),(0))) = (  Y[((1), -1), X], \pi')$, $X$ is the completion of $Y((1))$ by assigning the extra node $(2, ((0)))$ a partial level $\leq 1$ tree of degree 0, $\pi'(2,((0))) = (2, \comp{2}{T}[((0),(0))], \tau')$, $\tau'((0)) = (2, ((0), \se{(0)}, ((0))), \sigma') $, $\sigma' ((0))=(0)$,  $\tau'((0,0)) = (1, ((0)))$.
   \end{itemize}
 \end{myexample}

If $Y$ is a level-3 tree, then
\begin{displaymath}
  \id_{Y,*}
\end{displaymath}
factors $(Y, Y, Q^0)$ where $\id_{Y,*}(y) =((y, X, \overrightarrow{(e,x,W)}), \id_{*, X})$ for $Y[y] = (X, \overrightarrow{(e,x,W)})$.

For level-$3$ trees $R,Y$, we say that $\rho: \dom(R) \to \dom(Y)$ \emph{factors} $(R,Y)$ iff
\begin{enumerate}
\item If $r \in \dom({R})$ then ${R}(r) = {Y}(\rho(r))$.
\item If $r ,r' \in \dom({R})$ and $r \subseteq r'$, then  $ \rho(r) \subseteq \rho(r')$.
\item If  $R_{\tree}(r \concat (a)) = R_{\tree}(r \concat (b))$ and $a <_{BK} b$, then $\rho(r\concat (a))<_{BK} \rho(r \concat (b))$.
\end{enumerate}
If in addition, $\rho$ is onto $\dom(Y)$, then $\rho$ is called a \emph{level-3 tree isomorphism} between $R$ and $Y$. If $\rho$ factors $(R,Y)$ and $\vec{\gamma} = (\gamma_y)_{y \in \dom(Y)} \in [\bolddelta{3}]^{Y \uparrow}$, let $\vec{\gamma}_{\rho}  =  (\gamma_{\rho,r})_{r \in \dom(R)}\in [\bolddelta{3}]^{R \uparrow}$ where $\gamma_{\rho,r} = \gamma_{\rho(r)}$.

If $\rho$ factors $(R,Y,T)$ and $F \in (\bolddelta{3})^{Y\uparrow}$, let
\begin{displaymath}
  F^{T}_{\rho} : [\omega_1]^{T \uparrow} \to [\bolddelta{3}]^{R \uparrow} 
\end{displaymath}
be the function that sends $\vec{\xi}$ to $(F^{T}_{\rho(r)}(\vec{\xi}))_{r \in \dom(R)}$. The fact that $F^T_{\rho} (\vec{\xi}) \in [\bolddelta{3}]^{R \uparrow}$ follows from Lemmas~\ref{lem:B_desc_order}-\ref{lem:YTQ_description_extension}.

Suppose $Y$ is a level-3 tree, $T$ is a level $\leq 2$ tree. \emph{A representation of } $Y \otimes T$ is a pair $(R, \rho)$ such that
\begin{enumerate}
\item $R$ is a level-3 tree;
\item $\rho $ factors $(R,Y,T)$;
\item $\ran(\rho) = \desc(Y,T,*)$;
\end{enumerate}
Representations of $Y  \otimes T$ are clearly mutually isomorphic. As before, we shall regard
\begin{displaymath}
  Y \otimes T
\end{displaymath}
itself as a ``level-3 tree'' whose domain is the set of non-constant $(Y,T,*)$-descriptions and sends $\mathbf{B} \in \desc(Y,T, (Q, (d_i,q_i,P_i)_{1 \leq i \leq k}))$ to $(Q, (d_k,q_k,P_k))$. If $Q$ is a level $\leq 2$ tree, then $(Y \otimes T) \otimes Q $ is a ``level-3 tree'' whose domain consists of non-constant $(Y \otimes T, Q, *)$-descriptions. There is a natural isomorphism
\begin{displaymath}
  \iota_{Y,T,Q} 
\end{displaymath}
between ``level-3 trees'' $(Y \otimes T) \otimes Q$ and $Y \otimes (T \otimes Q)$, defined as follows:
\begin{enumerate}
\item If $\mathbf{A} = ( (\mathbf{B}, Z, \overrightarrow{(d,z,N)}) , \psi) \in \desc( Y \otimes T, Q, U)$, $\mathbf{B} = (\mathbf{y}, \pi) \in \desc(Y, T, (Z, \overrightarrow{(d,z,N)}))$, $\mathbf{y} = (y, X, \overrightarrow{(e,x,W)})$, then $\iota_{Y,T,Q} (\mathbf{A})  =  (\mathbf{y}, \iota^{-1}_{T,Q,U} \circ ( T \otimes \psi) \circ  \pi)$. 
\item If $\mathbf{A} = ((\mathbf{B} \concat (-1,Z^{+}), Z^{+}, \overrightarrow{(d,z,N)}), \psi) \in \desc(Y \otimes T, Q, U)$, $\overrightarrow{(d,z,N)} = (d_i,z_i,N_i)_{1 \leq i \leq l}$,  $\mathbf{B}  =  (\mathbf{y}, \pi) \in \desc(Y, T, (Z, \overrightarrow{(d,z,N)}))$, $\mathbf{y} = (y,X, (e_i,x_i,W_i)_{1 \leq i \leq k})$, then
  \begin{enumerate}
  \item if $y$ is of discontinuous type, then $\iota_{Y,T,Q}(\mathbf{A}) = (\mathbf{y} \concat (-1,X^{+}), \psi *_0 \pi)$, where $\psi*_0 \pi$ factors $(X^{+}, T \otimes Q, U)$, $X^{+}$ is a completion of $(X, (e_i,x_i))$, $\psi *_0 \pi$ extends $ \iota^{-1}_{T,Q,U} \circ ( T \otimes \psi) \circ  \pi$, $\psi *_0 \pi (e_k , x_k) = \iota_{T,Q,U}^{-1} (2, \mathbf{t}_0, \tau)$, $\mathbf{t}_0 = ((-1), \se{(0)}, ((0)))$, $\tau((0)) = \psi (d_l,z_l)$;
  \item if $y$ is of continuous type, then $\iota_{Y,T,Q}(\mathbf{A}) = (\mathbf{y} , \psi *_1 \pi)$, where $\psi*_1 \pi$ factors $(X, T \otimes Q, U)$,  $\psi *_1 \pi$ extends $ \iota^{-1}_{T,Q,U} \circ ( T \otimes \psi) \circ \pi \res (\dom(X) \setminus \se{(e_k,x_k)})$, $\psi *_1 \pi (e_k , x_k) = \iota_{T,Q,U}^{-1} (2, \mathbf{t} \concat (-1), \tau^{+})$,  where $\pi(e_k,x_k) = (2, \mathbf{t}, \tau)$, $\mathbf{t} = (t, S, (s_i)_{i \leq m})$, $\tau^{+}$ extends $\tau$, $\tau^{+}(s_m) = \psi (d_l,z_l)$.
  \end{enumerate}
\end{enumerate}
$\iota_{Y,T,Q}$ justifies the associativity of the $\otimes$ operator acting on level $(3, \leq 2, \leq 2)$ trees.

\begin{myexample}
  Let $Y = R^2$, $T = Q^{21}$, $Q = Q^{20}$. A representation of $Y \otimes T$ is $R$, and $\rho_R$ is a "level-3 tree isomorphism" between $R$ and  $Y \otimes T$, given as follows, where for $\mathbf{B} = (\mathbf{y},\pi)$,
  \begin{displaymath}
    \langle\langle \mathbf{B} \rangle\rangle = \langle \pi \rangle \oplus \mathbf{y},
  \end{displaymath}
where $ \langle \pi \rangle (\cdot) = \langle \pi (\cdot)\rangle$.
\begin{itemize}\item $R(((1)))$  has degree 2

$\langle\langle\rho_R((1))\rangle\rangle=$ ( (0) )

\item $R_{\tree}(((1), (3)))=$ the completion of $R(((1)))$ that sends $(2, ((0)))$ to the completion of $R_{\tree}(((1)))(2, \emptyset)$  whose node component is $(0, 0)$

\noindent$R_{\node}((1), (3))= (2, ((0, 0)))$

$\langle\langle\rho_R((1), (3))\rangle\rangle=$ ( (0), (2, ((2, ((0), (0))), (0))), $-1$ )

\item $R_{\tree}(((1), (2)))=$ the completion of $R(((1)))$ that sends $(2, ((0)))$ to the completion of $R_{\tree}(((1)))(2, \emptyset)$  whose node component is $(0, 0)$

\noindent$R_{\node}((1), (2))= (1, (0))$

$\langle\langle\rho_R((1), (2))\rangle\rangle=$ ( (0), (2, ((2, ((0), (0))), (0), (2, ((0), $-1$)), $-1$)), $-1$ )

\item $R_{\tree}(((1), (1)))=$ the completion of $R(((1)))$ that sends $(2, ((0)))$ to the completion of $R_{\tree}(((1)))(2, \emptyset)$  whose node component is $-1$

\noindent$R_{\node}((1), (1))= (2, ((0, 0)))$

$\langle\langle\rho_R((1), (1))\rangle\rangle=$ ( (0), (2, ((2, ((0), (0))), (0))), $-1$ )

\item $R_{\tree}(((1), (0)))=$ the completion of $R(((1)))$ that sends $(2, ((0)))$ to the completion of $R_{\tree}(((1)))(2, \emptyset)$  whose node component is $-1$

\noindent$R_{\node}((1), (0))= (1, (0))$

$\langle\langle\rho_R((1), (0))\rangle\rangle=$ ( (0), (2, ((2, ((0), (0))), (0), (2, ((0), $-1$)), $-1$)), $-1$ )

\item $R(((0)))$  has degree 1

$\langle\langle\rho_R((0))\rangle\rangle=$ ( (0), (2, ((2, ((0), $-1$)), (0))), $-1$ )

\end{itemize}A representation of $T \otimes Q$ is $M$, and $\pi_M$ is a "level $\leq 2$ tree isomorphism" between $M$ and  $T \otimes Q$, given as follows:
\begin{itemize}\item $\comp{2}{M}_{\node}(((3)))= -1$

$\langle\langle\pi_M((3))\rangle\rangle=$ (2, ((2, ((0), (0))), (0)))

\item $\comp{2}{M}_{\node}(((2)))= (0, 0)$

$\langle\langle\pi_M((2))\rangle\rangle=$ (2, ((2, ((0), (0))), (0), (2, ((0), $-1$)), $-1$))

\item $\comp{2}{M}_{\node}(((2), (0)))= -1$

$\langle\langle\pi_M((2), (0))\rangle\rangle=$ (2, ((2, ((0), (0))), (0), (2, ((0, 0), (0))), $-1$))

\item $\comp{2}{M}_{\node}(((1)))= -1$

$\langle\langle\pi_M((1))\rangle\rangle=$ (2, ((2, ((0), (0))), $-1$))

\item $\comp{2}{M}_{\node}(((0)))= (0, 0)$

$\langle\langle\pi_M((0))\rangle\rangle=$ (2, ((2, ((0), $-1$)), (0)))

\item $\comp{2}{M}_{\node}(((0), (0)))= -1$

$\langle\langle\pi_M((0), (0))\rangle\rangle=$ (2, ((2, ((0), $-1$)), (0), (2, ((0, 0), (0))), $-1$))

\end{itemize}Both $(Y \otimes T) \otimes Q$ and $Y \otimes (T \otimes Q)$ can be represented by the tree $N$. Let $\psi$ be the "level-3 tree isomorphism" between $N$ and  $R \otimes Q$ and let $\psi'$ be the "level-3 tree isomorphism" between $N$ and  $Y \otimes M$. Then $N$ and $\psi$, $\psi'$ are as follows:
\begin{etaremune}
\item $N(((4)))$  has degree 2

$\langle\langle\psi((4))\rangle\rangle=$ ( (1) )

$\langle\langle\psi'((4))\rangle\rangle=$ ( (0) )

\item $N_{\tree}(((4), (15)))=$ the completion of $N(((4)))$ that sends $(2, ((0)))$ to the completion of $N_{\tree}(((4)))(2, \emptyset)$  whose node component is $(0, 0)$

\noindent$N_{\node}((4), (15))= (0, -1)$

$\langle\langle\psi((4), (15))\rangle\rangle=$ ( (1), (2, ((2, ((0), (0))), (0))), (1) )

$\langle\langle\psi'((4), (15))\rangle\rangle=$ ( (0), (2, ((2, ((0), (0))), (3))), $-1$ )

\item $N_{\tree}(((4), (14)))=$ the completion of $N(((4)))$ that sends $(2, ((0)))$ to the completion of $N_{\tree}(((4)))(2, \emptyset)$  whose node component is $(0, 0)$

\noindent$N_{\node}((4), (14))= (2, ((0, 0)))$

$\langle\langle\psi((4), (14))\rangle\rangle=$ ( (1), (2, ((2, ((0), (0))), (0))), (1), (2, ((2, ((0), (0))), $-1$)), $-1$ )

$\langle\langle\psi'((4), (14))\rangle\rangle=$ ( (0), (2, ((2, ((0), (0))), (2))), $-1$ )

\item $N_{\tree}(((4), (14), (1)))=$ the completion of $N(((4), (14)))$ that sends $(2, ((0, 0)))$ to the completion of $N_{\tree}(((4), (14)))(2, \emptyset)$  whose node component is $(0, 0)$

\noindent$N_{\node}((4), (14), (1))= (0, -1)$

$\langle\langle\psi((4), (14), (1))\rangle\rangle=$ ( (1), (2, ((2, ((0), (0))), (0))), (1), (2, ((2, ((0), (0, 0))), (0))), $-1$ )

$\langle\langle\psi'((4), (14), (1))\rangle\rangle=$ ( (0), (2, ((2, ((0), (0))), (2), (2, ((0), (0, 0))), (0))), $-1$ )

\item $N_{\tree}(((4), (14), (0)))=$ the completion of $N(((4), (14)))$ that sends $(2, ((0, 0)))$ to the completion of $N_{\tree}(((4), (14)))(2, \emptyset)$  whose node component is $-1$

\noindent$N_{\node}((4), (14), (0))= (0, -1)$

$\langle\langle\psi((4), (14), (0))\rangle\rangle=$ ( (1), (2, ((2, ((0), (0))), (0))), (1), (2, ((2, ((0), (0, 0))), (0))), $-1$ )

$\langle\langle\psi'((4), (14), (0))\rangle\rangle=$ ( (0), (2, ((2, ((0), (0))), (2), (2, ((0), (0, 0))), (0))), $-1$ )

\item $N_{\tree}(((4), (13)))=$ the completion of $N(((4)))$ that sends $(2, ((0)))$ to the completion of $N_{\tree}(((4)))(2, \emptyset)$  whose node component is $(0, 0)$

\noindent$N_{\node}((4), (13))= (0, -1)$

$\langle\langle\psi((4), (13))\rangle\rangle=$ ( (1), (2, ((2, ((0), (0))), (0))), (1), (2, ((2, ((0), $-1$)), (0))), $-1$ )

$\langle\langle\psi'((4), (13))\rangle\rangle=$ ( (0), (2, ((2, ((0), (0))), (2), (2, ((0), $-1$)), (0))), $-1$ )

\item $N_{\tree}(((4), (12)))=$ the completion of $N(((4)))$ that sends $(2, ((0)))$ to the completion of $N_{\tree}(((4)))(2, \emptyset)$  whose node component is $(0, 0)$

\noindent$N_{\node}((4), (12))= (1, (0))$

$\langle\langle\psi((4), (12))\rangle\rangle=$ ( (1), (2, ((2, ((0), (0))), (0))), (0) )

$\langle\langle\psi'((4), (12))\rangle\rangle=$ ( (0), (2, ((2, ((0), (0))), (2), (2, ((0), $-1$)), $-1$)), $-1$ )

\item $N_{\tree}(((4), (12), (0)))=$ the unique completion of $N(((4), (12)))$

\noindent$N_{\node}((4), (12), (0))= (0, -1)$

$\langle\langle\psi((4), (12), (0))\rangle\rangle=$ ( (1), (2, ((2, ((0), (0))), (0))), (0), (2, ((0), (0))), $-1$ )

$\langle\langle\psi'((4), (12), (0))\rangle\rangle=$ ( (0), (2, ((2, ((0), (0))), (2), (1, ((0))), (0))), $-1$ )

\item $N_{\tree}(((4), (11)))=$ the completion of $N(((4)))$ that sends $(2, ((0)))$ to the completion of $N_{\tree}(((4)))(2, \emptyset)$  whose node component is $(0, 0)$

\noindent$N_{\node}((4), (11))= (0, -1)$

$\langle\langle\psi((4), (11))\rangle\rangle=$ ( (1), (2, ((2, ((0), (0))), (0))), $-1$ )

$\langle\langle\psi'((4), (11))\rangle\rangle=$ ( (0), (2, ((2, ((0), (0))), (1))), $-1$ )

\item $N_{\tree}(((4), (10)))=$ the completion of $N(((4)))$ that sends $(2, ((0)))$ to the completion of $N_{\tree}(((4)))(2, \emptyset)$  whose node component is $(0, 0)$

\noindent$N_{\node}((4), (10))= (2, ((0, 0)))$

$\langle\langle\psi((4), (10))\rangle\rangle=$ ( (1), (2, ((2, ((0), (0))), $-1$)), (3) )

$\langle\langle\psi'((4), (10))\rangle\rangle=$ ( (0), (2, ((2, ((0), (0))), (0))), $-1$ )

\item $N_{\tree}(((4), (10), (1)))=$ the completion of $N(((4), (10)))$ that sends $(2, ((0, 0)))$ to the completion of $N_{\tree}(((4), (10)))(2, \emptyset)$  whose node component is $(0, 0)$

\noindent$N_{\node}((4), (10), (1))= (0, -1)$

$\langle\langle\psi((4), (10), (1))\rangle\rangle=$ ( (1), (2, ((2, ((0), (0))), $-1$)), (3), (2, ((2, ((0), (0, 0))), (0))), $-1$ )

$\langle\langle\psi'((4), (10), (1))\rangle\rangle=$ ( (0), (2, ((2, ((0), (0))), (0), (2, ((0), (0, 0))), (0))), $-1$ )

\item $N_{\tree}(((4), (10), (0)))=$ the completion of $N(((4), (10)))$ that sends $(2, ((0, 0)))$ to the completion of $N_{\tree}(((4), (10)))(2, \emptyset)$  whose node component is $-1$

\noindent$N_{\node}((4), (10), (0))= (0, -1)$

$\langle\langle\psi((4), (10), (0))\rangle\rangle=$ ( (1), (2, ((2, ((0), (0))), $-1$)), (3), (2, ((2, ((0), (0, 0))), (0))), $-1$ )

$\langle\langle\psi'((4), (10), (0))\rangle\rangle=$ ( (0), (2, ((2, ((0), (0))), (0), (2, ((0), (0, 0))), (0))), $-1$ )

\item $N_{\tree}(((4), (9)))=$ the completion of $N(((4)))$ that sends $(2, ((0)))$ to the completion of $N_{\tree}(((4)))(2, \emptyset)$  whose node component is $(0, 0)$

\noindent$N_{\node}((4), (9))= (0, -1)$

$\langle\langle\psi((4), (9))\rangle\rangle=$ ( (1), (2, ((2, ((0), (0))), $-1$)), (3), (2, ((2, ((0), $-1$)), (0))), $-1$ )

$\langle\langle\psi'((4), (9))\rangle\rangle=$ ( (0), (2, ((2, ((0), (0))), (0), (2, ((0), $-1$)), (0))), $-1$ )

\item $N_{\tree}(((4), (8)))=$ the completion of $N(((4)))$ that sends $(2, ((0)))$ to the completion of $N_{\tree}(((4)))(2, \emptyset)$  whose node component is $(0, 0)$

\noindent$N_{\node}((4), (8))= (1, (0))$

$\langle\langle\psi((4), (8))\rangle\rangle=$ ( (1), (2, ((2, ((0), (0))), $-1$)), (2) )

$\langle\langle\psi'((4), (8))\rangle\rangle=$ ( (0), (2, ((2, ((0), (0))), (0), (2, ((0), $-1$)), $-1$)), $-1$ )

\item $N_{\tree}(((4), (8), (0)))=$ the unique completion of $N(((4), (8)))$

\noindent$N_{\node}((4), (8), (0))= (0, -1)$

$\langle\langle\psi((4), (8), (0))\rangle\rangle=$ ( (1), (2, ((2, ((0), (0))), $-1$)), (2), (2, ((0), (0))), $-1$ )

$\langle\langle\psi'((4), (8), (0))\rangle\rangle=$ ( (0), (2, ((2, ((0), (0))), (0), (1, ((0))), (0))), $-1$ )

\item $N_{\tree}(((4), (7)))=$ the completion of $N(((4)))$ that sends $(2, ((0)))$ to the completion of $N_{\tree}(((4)))(2, \emptyset)$  whose node component is $-1$

\noindent$N_{\node}((4), (7))= (0, -1)$

$\langle\langle\psi((4), (7))\rangle\rangle=$ ( (1), (2, ((2, ((0), (0))), (0))), (1) )

$\langle\langle\psi'((4), (7))\rangle\rangle=$ ( (0), (2, ((2, ((0), (0))), (3))), $-1$ )

\item $N_{\tree}(((4), (6)))=$ the completion of $N(((4)))$ that sends $(2, ((0)))$ to the completion of $N_{\tree}(((4)))(2, \emptyset)$  whose node component is $-1$

\noindent$N_{\node}((4), (6))= (2, ((0, 0)))$

$\langle\langle\psi((4), (6))\rangle\rangle=$ ( (1), (2, ((2, ((0), (0))), (0))), (1), (2, ((2, ((0), (0))), $-1$)), $-1$ )

$\langle\langle\psi'((4), (6))\rangle\rangle=$ ( (0), (2, ((2, ((0), (0))), (2))), $-1$ )

\item $N_{\tree}(((4), (6), (1)))=$ the completion of $N(((4), (6)))$ that sends $(2, ((0, 0)))$ to the completion of $N_{\tree}(((4), (6)))(2, \emptyset)$  whose node component is $(0, 0)$

\noindent$N_{\node}((4), (6), (1))= (0, -1)$

$\langle\langle\psi((4), (6), (1))\rangle\rangle=$ ( (1), (2, ((2, ((0), (0))), (0))), (1), (2, ((2, ((0), (0, 0))), (0))), $-1$ )

$\langle\langle\psi'((4), (6), (1))\rangle\rangle=$ ( (0), (2, ((2, ((0), (0))), (2), (2, ((0), (0, 0))), (0))), $-1$ )

\item $N_{\tree}(((4), (6), (0)))=$ the completion of $N(((4), (6)))$ that sends $(2, ((0, 0)))$ to the completion of $N_{\tree}(((4), (6)))(2, \emptyset)$  whose node component is $-1$

\noindent$N_{\node}((4), (6), (0))= (0, -1)$

$\langle\langle\psi((4), (6), (0))\rangle\rangle=$ ( (1), (2, ((2, ((0), (0))), (0))), (1), (2, ((2, ((0), (0, 0))), (0))), $-1$ )

$\langle\langle\psi'((4), (6), (0))\rangle\rangle=$ ( (0), (2, ((2, ((0), (0))), (2), (2, ((0), (0, 0))), (0))), $-1$ )

\item $N_{\tree}(((4), (5)))=$ the completion of $N(((4)))$ that sends $(2, ((0)))$ to the completion of $N_{\tree}(((4)))(2, \emptyset)$  whose node component is $-1$

\noindent$N_{\node}((4), (5))= (0, -1)$

$\langle\langle\psi((4), (5))\rangle\rangle=$ ( (1), (2, ((2, ((0), (0))), (0))), (1), (2, ((2, ((0), $-1$)), (0))), $-1$ )

$\langle\langle\psi'((4), (5))\rangle\rangle=$ ( (0), (2, ((2, ((0), (0))), (2), (2, ((0), $-1$)), (0))), $-1$ )

\item $N_{\tree}(((4), (4)))=$ the completion of $N(((4)))$ that sends $(2, ((0)))$ to the completion of $N_{\tree}(((4)))(2, \emptyset)$  whose node component is $-1$

\noindent$N_{\node}((4), (4))= (1, (0))$

$\langle\langle\psi((4), (4))\rangle\rangle=$ ( (1), (2, ((2, ((0), (0))), (0))), (0) )

$\langle\langle\psi'((4), (4))\rangle\rangle=$ ( (0), (2, ((2, ((0), (0))), (2), (2, ((0), $-1$)), $-1$)), $-1$ )

\item $N_{\tree}(((4), (4), (0)))=$ the unique completion of $N(((4), (4)))$

\noindent$N_{\node}((4), (4), (0))= (0, -1)$

$\langle\langle\psi((4), (4), (0))\rangle\rangle=$ ( (1), (2, ((2, ((0), (0))), (0))), (0), (2, ((0), (0))), $-1$ )

$\langle\langle\psi'((4), (4), (0))\rangle\rangle=$ ( (0), (2, ((2, ((0), (0))), (2), (1, ((0))), (0))), $-1$ )

\item $N_{\tree}(((4), (3)))=$ the completion of $N(((4)))$ that sends $(2, ((0)))$ to the completion of $N_{\tree}(((4)))(2, \emptyset)$  whose node component is $-1$

\noindent$N_{\node}((4), (3))= (0, -1)$

$\langle\langle\psi((4), (3))\rangle\rangle=$ ( (1), (2, ((2, ((0), (0))), (0))), $-1$ )

$\langle\langle\psi'((4), (3))\rangle\rangle=$ ( (0), (2, ((2, ((0), (0))), (1))), $-1$ )

\item $N_{\tree}(((4), (2)))=$ the completion of $N(((4)))$ that sends $(2, ((0)))$ to the completion of $N_{\tree}(((4)))(2, \emptyset)$  whose node component is $-1$

\noindent$N_{\node}((4), (2))= (2, ((0, 0)))$

$\langle\langle\psi((4), (2))\rangle\rangle=$ ( (1), (2, ((2, ((0), (0))), $-1$)), (1) )

$\langle\langle\psi'((4), (2))\rangle\rangle=$ ( (0), (2, ((2, ((0), (0))), (0))), $-1$ )

\item $N_{\tree}(((4), (2), (1)))=$ the completion of $N(((4), (2)))$ that sends $(2, ((0, 0)))$ to the completion of $N_{\tree}(((4), (2)))(2, \emptyset)$  whose node component is $(0, 0)$

\noindent$N_{\node}((4), (2), (1))= (0, -1)$

$\langle\langle\psi((4), (2), (1))\rangle\rangle=$ ( (1), (2, ((2, ((0), (0))), $-1$)), (1), (2, ((2, ((0), (0, 0))), (0))), $-1$ )

$\langle\langle\psi'((4), (2), (1))\rangle\rangle=$ ( (0), (2, ((2, ((0), (0))), (0), (2, ((0), (0, 0))), (0))), $-1$ )

\item $N_{\tree}(((4), (2), (0)))=$ the completion of $N(((4), (2)))$ that sends $(2, ((0, 0)))$ to the completion of $N_{\tree}(((4), (2)))(2, \emptyset)$  whose node component is $-1$

\noindent$N_{\node}((4), (2), (0))= (0, -1)$

$\langle\langle\psi((4), (2), (0))\rangle\rangle=$ ( (1), (2, ((2, ((0), (0))), $-1$)), (1), (2, ((2, ((0), (0, 0))), (0))), $-1$ )

$\langle\langle\psi'((4), (2), (0))\rangle\rangle=$ ( (0), (2, ((2, ((0), (0))), (0), (2, ((0), (0, 0))), (0))), $-1$ )

\item $N_{\tree}(((4), (1)))=$ the completion of $N(((4)))$ that sends $(2, ((0)))$ to the completion of $N_{\tree}(((4)))(2, \emptyset)$  whose node component is $-1$

\noindent$N_{\node}((4), (1))= (0, -1)$

$\langle\langle\psi((4), (1))\rangle\rangle=$ ( (1), (2, ((2, ((0), (0))), $-1$)), (1), (2, ((2, ((0), $-1$)), (0))), $-1$ )

$\langle\langle\psi'((4), (1))\rangle\rangle=$ ( (0), (2, ((2, ((0), (0))), (0), (2, ((0), $-1$)), (0))), $-1$ )

\item $N_{\tree}(((4), (0)))=$ the completion of $N(((4)))$ that sends $(2, ((0)))$ to the completion of $N_{\tree}(((4)))(2, \emptyset)$  whose node component is $-1$

\noindent$N_{\node}((4), (0))= (1, (0))$

$\langle\langle\psi((4), (0))\rangle\rangle=$ ( (1), (2, ((2, ((0), (0))), $-1$)), (0) )

$\langle\langle\psi'((4), (0))\rangle\rangle=$ ( (0), (2, ((2, ((0), (0))), (0), (2, ((0), $-1$)), $-1$)), $-1$ )

\item $N_{\tree}(((4), (0), (0)))=$ the unique completion of $N(((4), (0)))$

\noindent$N_{\node}((4), (0), (0))= (0, -1)$

$\langle\langle\psi((4), (0), (0))\rangle\rangle=$ ( (1), (2, ((2, ((0), (0))), $-1$)), (0), (2, ((0), (0))), $-1$ )

$\langle\langle\psi'((4), (0), (0))\rangle\rangle=$ ( (0), (2, ((2, ((0), (0))), (0), (1, ((0))), (0))), $-1$ )

\item $N(((3)))$  has degree 0

$\langle\langle\psi((3))\rangle\rangle=$ ( (1), (2, ((2, ((0), $-1$)), (0))), (1) )

$\langle\langle\psi'((3))\rangle\rangle=$ ( (0), (2, ((2, ((0), $-1$)), (3))), $-1$ )

\item $N(((2)))$  has degree 1

$\langle\langle\psi((2))\rangle\rangle=$ ( (1), (2, ((2, ((0), $-1$)), (0))), (0) )

$\langle\langle\psi'((2))\rangle\rangle=$ ( (0), (2, ((2, ((0), $-1$)), (2))), $-1$ )

\item $N_{\tree}(((2), (0)))=$ the unique completion of $N(((2)))$

\noindent$N_{\node}((2), (0))= (0, -1)$

$\langle\langle\psi((2), (0))\rangle\rangle=$ ( (1), (2, ((2, ((0), $-1$)), (0))), (0), (2, ((0), (0))), $-1$ )

$\langle\langle\psi'((2), (0))\rangle\rangle=$ ( (0), (2, ((2, ((0), $-1$)), (2), (1, ((0))), (0))), $-1$ )

\item $N(((1)))$  has degree 0

$\langle\langle\psi((1))\rangle\rangle=$ ( (1), (2, ((2, ((0), $-1$)), (0))), $-1$ )

$\langle\langle\psi'((1))\rangle\rangle=$ ( (0), (2, ((2, ((0), $-1$)), (1))), $-1$ )

\item $N(((0)))$  has degree 1

$\langle\langle\psi((0))\rangle\rangle=$ ( (0) )

$\langle\langle\psi'((0))\rangle\rangle=$ ( (0), (2, ((2, ((0), $-1$)), (0))), $-1$ )

\item $N_{\tree}(((0), (0)))=$ the unique completion of $N(((0)))$

\noindent$N_{\node}((0), (0))= (0, -1)$

$\langle\langle\psi((0), (0))\rangle\rangle=$ ( (0), (2, ((0), (0))), $-1$ )

$\langle\langle\psi'((0), (0))\rangle\rangle=$ ( (0), (2, ((2, ((0), $-1$)), (0), (1, ((0))), (0))), $-1$ )

\end{etaremune}

\end{myexample}

The identity map $\id_{Y \otimes T}$ factors $(Y \otimes T, Y, T)$.  $\rho$ factors $(R,Y,T)$ iff $\rho $ factors $(R , Y \otimes T)$. If $y \in \dom(Y)$, $\mathbf{y} = (y, X, \overrightarrow{(e,x,W)}) \in \desc(Y)$, 
\begin{displaymath}
  Y \otimes_y T
\end{displaymath}
is the level-3 subtree of $Y \otimes T$ whose domain is $\dom(Y \otimes Q^0)$ plus all the $(Y,T,*)$-descriptions of the form $(\mathbf{y}, \tau)$. 
If $\pi$ factors level $\leq 2$ trees $(T,Q)$, then
\begin{displaymath}
  Y \otimes \pi
\end{displaymath}
factors $(Y \otimes T, Y \otimes Q)$, where $Y \otimes \pi ( \mathbf{y}, \psi )  =  (\mathbf{y}, ( \pi \otimes U ) \circ \psi)$ for $(\mathbf{y}, \psi) \in \desc(Y, T, U)$.

If $\rho$ factors finite trees $(R,Y,T)$, then $\rho$ induces
\begin{displaymath}
  \tilde{\rho}^T : \exexdesc(R ) \to \exexdesc(Y)
\end{displaymath}
as follows:
\begin{enumerate}
\item If $\mathbf{A} = (\emptyset,\emptyset,\emptyset)$, then $\tilde{\rho}^T(\mathbf{A}) = \mathbf{A}$.
\item If $\mathbf{A} = (\mathbf{r}, \psi, U)$, $\mathbf{r} = (r, Q, \overrightarrow{(d,q,P)})$ is of discontinuous type, $\rho(r) = (\mathbf{y}, \pi)$, then $\tilde{\rho}^T (\mathbf{A}) = (\mathbf{y}, (T \otimes \psi) \circ \pi )$. 
\item If $\mathbf{A} = (\mathbf{r}, \psi, U)$, $\mathbf{r} = (r, Q, \overrightarrow{(d,q,P)})$ is of continuous type, $\overrightarrow{(d,q,P)} = (d_i,q_i,P_i)_{1 \leq i \leq l}$, 
$\rho(r^{-}) = (\mathbf{y}, \pi)$, $\mathbf{y}  = (y, X, (e_i,x_i,W_i)_{1 \leq i \leq k})$, 
  \begin{enumerate}
  \item if $y$ is of discontinuous type, then $\tilde{\rho}^T(\mathbf{A}) = ( \mathbf{y} \concat (-1,X^{+}), \psi *_0 \pi)$, where $\psi *_0 \pi$ factors $(X^{+}, T \otimes U)$, $\psi *_0 \pi$ extends $(T \otimes \psi) \circ \pi$, $\psi *_0 \pi (e_k,x_k) = (2, \mathbf{t}_0, \tau)$, $\mathbf{t}_0 = ((-1), \se{(0)}, ((0)))$, $\tau((0)) = \psi (d_l,p_l)$; 
  \item if $y$ is of continuous type, then $\tilde{\rho}^T(\mathbf{A}) = (\mathbf{y} , \psi *_1 \pi)$, where $\psi *_1 \pi$ factors $(X, T \otimes U)$, $\psi *_1 \pi$ extends $(T \otimes \psi) \circ (\pi \res \dom(X) \setminus \se{(e_k,x_k)})$, $\psi *_1 \pi (e_k,x_k) = (2, \mathbf{t} \concat (-1), \tau^{+})$, where $\pi(e_k,x_k) = (2, \mathbf{t}, \tau)$, $\mathbf{t} = (t,S, (s_i)_{i \leq m})$, $\tau^{+}$ extends $\tau$, $\tau^{+}(s_m) = \psi(d_l,p_l)$.
  \end{enumerate}
\end{enumerate}
$\mathbf{A} \prec^R_{*} \mathbf{A}'$ iff $\tilde{\rho}^T(\mathbf{A}) \prec^Y_{*} \tilde{\rho}^T(\mathbf{A}')$; $\mathbf{A} \sim^R_{*} \mathbf{A}'$ iff $\tilde{\rho}^T(\mathbf{A}) \sim^Y_{*} \tilde{\rho}^T(\mathbf{A}')$. A purely combinatorial argument shows that if $R = Y \otimes T$, then for any $\mathbf{B} \in \exexdesc(Y)$ there is $\mathbf{A} \in \exexdesc(R)$ such that $\tilde{\rho}^T(\mathbf{A}) \sim^Y_{*} \mathbf{B}$. 

\begin{mylemma}
  \label{lem:jQ_move_continuity}
 Suppose $Q$ is a finite level $\leq 2$ tree, $W$ is a finite level-1 tree,
    $\theta: [\omega_1]^{Q \uparrow} \to j^W(\omega_1)$ is a function in $\admistwobold$. Suppose  $\cf^{\mathbb{L}}([\theta]_{\mu^Q}) = \seed^{Q,W}_{\mathbf{D}}$,  $\mathbf{D} = (d, \mathbf{q}, \sigma) \in \desc(Q,W)$. 
  \begin{enumerate}
  \item The uniform cofinality of $\theta$ is $\ucf_{*}^W(\mathbf{D})$. 
  \item  $\ucf(\mathbf{D}) = -1$ iff   $\cf^{\mathbb{L}}(\theta(\vec{\xi})) = \omega$ for $\mu^Q$-a.e.\ $\vec{\xi}$. 
  \item  Fix $w \in W$. Then $\ucf(\mathbf{D}) = w$ iff   $\cf^{\mathbb{L}}(\theta(\vec{\xi})) = \seed^W_w$ for $\mu^Q$-a.e.\ $\vec{\xi}$. 
  \end{enumerate}
\end{mylemma}
\begin{proof}
 Let $g \in \mathbb{L}$ be a strictly increasing function from $\seed^{Q,W}_{\mathbf{D}}$ to $[\theta]_{\mu^Q}$ cofinally. 
Find $G \in \mathbb{L}$ such that $[G]_{\mu^Q} =g$. We have $[\theta]_{\mu^Q}= \sup {[G]_{\mu^Q}} '' \seed^{Q,W}_{\mathbf{D}}$. By \Los{}, for $\mu^Q$-a.e.\ $\vec{\xi}$, $\theta(\vec{\xi}) = \sup G(\vec{\xi})'' ({\sigma}^W ( \comp{d}{\xi}_{\mathbf{q}}))$. (Recall our convention that $\emptyset^W = j^W$.) This shows part 1. 
 Also, for $\mu^Q$-a.e.\ $\vec{\xi}$, $\cf^{\mathbb{L}}(\theta(\vec{\xi})) = \cf^{\mathbb{L}} ( {\sigma}^W(\comp{d}{\xi}_{\mathbf{q}}))$, which equals to $\omega$ when $\ucf(\mathbf{D}) = -1$, equals to $ \seed^W_{\ucf(\mathbf{D})} $ otherwise. This shows parts 2-3.
\end{proof}

\begin{mylemma}
  \label{lem:RYT_factor_from_division}
  Suppose $R,Y$ are level-3 trees, $\theta : \rep(R) \to \rep(Y)$ is continuous and order preserving, $\theta \in \admistwobold$. Then there exists a triple
  \begin{displaymath}
    (T, \rho, \vec{\delta})
  \end{displaymath}
such that $T$ is a level $\leq 2$ tree, $\rho$ factors $(R,Y,T)$, $\vec{\delta}$ respects $T$, and
\begin{displaymath}
  \forall F \in (\bolddelta{3})^{Y \uparrow} ~ F^{T}_{\rho} (\vec{\delta}) = [F \circ \theta]^R. 
\end{displaymath}
\end{mylemma}
\begin{proof}
 For $r \in \dom(R)$, let $R(r)  =  (Q_r,(d_r,q_r, P_r))$. For $q \in \dom(\comp{2}{Q}_r)$, let $\comp{2}{Q}_r(q) = (P_{r,q}, p_{r,q})$. Thus, when  $d_r = 2$,   $P_r$ is the completion of $(P_{r,q_r^{-}}, p_{r,q_r^{-}})$.
Let $E \in \mu_{\mathbb{L}}$, $\mathbf{y}_r = (y_r, X_r, \overrightarrow{(e_r,x_r,W_r)} )\in \desc(Y)$ and $\theta_r \in \admistwobold$ be such that for any $\vec{\beta} \in [E]^{Q_r\uparrow}$, $\theta_r(\vec{\beta}) \in [\omega_1]^{X_r \uparrow}$ and 
 \begin{displaymath}
   \theta(\vec{\beta} \oplus_R r) = \theta_r (\vec{\beta} ) \oplus_Y y_r.
 \end{displaymath}
Let $\overrightarrow{(e_r,x_r,W_r)}= (e_{r,i}, x_{r,i}, W_{r,i})_{1 \leq i  \leq  \lh( \vec{x}_r)}$. 
 For $x \in \dom(\comp{2}{X}_r)$, let $\comp{2}{X}_r(x) = (W_{r,x}, w_{r,x})$. Thus,  when  $e_{r,i}=2$,  $W_{r,i}$ is the completion of $(W_{r,x_r^{-}}, w_{r,x_r^{-}})$. 
Let $[\theta_r]_{\mu^{Q_r}} = \vec{\gamma}_r = (\comp{e}{\gamma}_{r,x})_{(e,x) \in \dom(X_r)}$, $\theta_r(\vec{\beta}) = (\comp{e}{\theta}_{r,x}(\vec{\beta}))_{(e,x) \in \dom(X_r)}$. So $\comp{e}{\gamma}_{r,x} = [ \comp{e}{\theta}_{r,x}]_{\mu^{Q_r}}$.  For $e \in \se{1,2}$, let
\begin{displaymath}
  B^e_x = \set{x \in \dom(\comp{e}{X}_r)}{ \comp{e}{\gamma}_{r,x} < \omega_1}. 
\end{displaymath}
So $B^1_r$ is closed under $\prec^{\comp{1}{X}_r}$ and $B^2_x = \emptyset$. For $x \in \comp{e}{X}_r \setminus B^e_r$, let $(S^e_{r,x}, \vec{s}^e_{r,x})$ be the potential partial level $\leq 1$ tower induced by $\comp{e}{\gamma}_{r,x}$, $\vec{s}^e_{r,x} = (s^e_{r,x,i})_{i < \lh(\vec{s}^e_{r,x})}$, $s_{r,x}^e = s^e_{r,x,\lh(\vec{s}^e_{r,x})-1}$, let $(\seed^{{Q}_r, W_{r,x}}_{\mathbf{D}^e_{r,x,i}})_{i < v^e_{r,x}}$
be the signature of $\comp{e}{\gamma}_{r,x}$, let $( \delta^e_{r,x,i})_{ i \leq v^e_{r,x}}$ be the approximation sequence of $\comp{e}{\gamma}_{r,x}$, 
and let $\cf^{\mathbb{L}}(\comp{e}{\gamma}_{r,x}) = \seed^{{Q}_r, W_{r,x}}_{\mathbf{D}^e_{r,x}}$ if $\cf^{\mathbb{L}}(\comp{e}{\gamma}_{r,x}) > \omega$. The existence of $(\mathbf{D}^e_{r,x,i})_{i < v^e_{r,x}}$ and $\mathbf{D}^e_{r,x}$ follows from Lemma~\ref{lem:factor_SQW}.
Let $\mathbf{D}^e_{r,x,i} = (c_{r,x,i}^e, \mathbf{q}^e_{r,x,i}, \sigma_{r,x,i}^e)$, $\mathbf{D}^e_{r,x} = (c_{r,x}^e, q^e_{r,x}, \sigma_{r,x}^e)$.
Let
\begin{displaymath}
\tau^e_{r,x}
\end{displaymath}
factor $(S^e_{r,x}, Q_r, *)$, where $\tau^e_{r,x}(s^e_{r,x,i}) = \mathbf{D}^e_{r,x,i}$ for $i<v^e_{r,x}$.
 Let
\begin{align*}
  D^e_r &= \set{ x \in \comp{e}{X}_r \setminus B^e_r  }{ \comp{e}{\gamma}_{r,x} \text{ is essentially continuous}}, \\
  E^e_r & = \dom(\comp{e}{X}_r ) \setminus  (B^e_r \cup D^e_r). 
\end{align*}
Thus, $v^e_{r,x}  =  \card(S^e_{r,x})$. For $x \in D^e_r$, $v^e_{r,x} = \lh(\vec{s}^e_{r,x})$; for $x \in E^e_r$, $v^e_{r,x} = \lh(\vec{s}^e_{r,x})-1$. 

Put $\ucf(R[r]) = (d_r^{*}, \mathbf{q}_r^{*})$, $\ucf(X_r, \overrightarrow{(e_r,x_r,W_r)}) = (e_r^{*}, \mathbf{x}_r^{*})$, if $e_r^{*}=2$ then put $\mathbf{x}_r^{*} = (x_r^{*},W_r^{*}, \vec{w}_r^{*})$.

By order preservation and continuity of $\theta$, we can see that for $r \in \dom(R)$,
\begin{enumerate}
\item if $y_r$ is of continuous type, then $\comp{e_r}{\theta}_{r,x_r}$ has uniform cofinality $\ucf(R[r])$;
\item if $y_r$ is of continuous type and $\lh(r) = 1 \vee (d_{r^{-}},q_{r^{-}})$ does not appear in the contraction of $(\sign_{*}^{Q_r} (\comp{e}{\theta}_{r,x}))$ for any $(e,x) \in \dom(X_r) \setminus \se{(e_r,x_r)}$, then $\comp{e_r}{\theta}_{r,x_r}$ is essentially discontinuous;
\item if $y_r$ is of discontinuous type, 
  \begin{enumerate}
  \item if $d_r ^{*}= 0$, then $e_r^{*} = 0$;
  \item if $d_r^{*} = 1$, then $e_r^{*} = 1$, $\comp{1}{\theta}_{r, \mathbf{x}_r^{*}}$ has uniform cofinality $(1, \mathbf{q}_r^{*})$, and thus by Lemma~\ref{lem:jQ_move_continuity}, $\mathbf{D}^1_{r,x_r^{*}} = (1, \mathbf{q}_r^{*},\emptyset)$;
  \item if $d_r^{*} = 2$ and $\mathbf{q}_r^{*} \in \desc(\comp{2}{Q}_r)$,   then $e_r^* = 2$ and $\mathbf{x}_r^{*} \in \desc(\comp{2}{X}_r)$,  $\comp{2}{\theta}_{r, x_r^{*}}$ has uniform cofinality $(2, \mathbf{q}_r^{*})$, and thus by Lemma~\ref{lem:jQ_move_continuity}, 
$\ucf_{*}^{W_{r}^{*}}(\mathbf{D}^2_{r,x_r^{*}}) = (2, \mathbf{q}_r^{*}) $;
  \item if $d_r^{*} = 2$ and $\mathbf{q}_r^{*} \notin \desc(\comp{2}{Q}_r)$, then $e_r^{*} = 2$ and $\mathbf{x}_r^{*} \in \desc(\comp{2}{X}_r)$,  
$ j^{W_{r,x_r^{*}}, W_{r}^{*}} (\comp{2}{\theta}_{r, x_r^{*}})$ has uniform cofinality $(2, \mathbf{q}_r^{*})$, and thus  by Lemma~\ref{lem:jQ_move_continuity}, 
$\ucf_{*}^{W_{r}^{*}}(\mathbf{D}^2_{r,x_r^{*}}) = (2, \mathbf{q}_r^{*}) $.
  \end{enumerate}
\end{enumerate}

\begin{myclaim}
  \label{claim:YTQ_division}
  Suppose $r \in \dom(R)$, $x, x' \in \dom(\comp{2}{X}_r)$, $x = (x')^{-}$. Suppose the contraction of $(\sign(\mathbf{D}^2_{r,x,i}))_{ i < v^2_{r,x}}$ is $(w_{r, x \res i})_{i < \lh(x)}$. Then
  \begin{enumerate}
  \item For any $i < v^2_{r,x}$, $\delta^2_{r,x,i} = \delta^2_{r, x',i}$.
  \item $(\mathbf{D}^2_{r,x,i}, \delta^2_{r,x,i})_{i < v^2_{r,x}}$ is a proper initial segment of $(\mathbf{D}^2_{r,x',i}, \delta^2_{r,x',i})_{i < v^2_{r,x'}}$. 
 Hence, $S^2_{r,x}$ is a proper subtree of $S^2_{r, x'}$ and $\vec{s}^2_{r,x}$ is an initial segment of $\vec{s}^2_{r,x'}$. 
  \item $\sign(\mathbf{D}^2_{r, x', v^2_{r,x}}) = w_{r,x}$. In particular, the contraction of $(\sign(\mathbf{D}^2_{r,x',i}))_{i < v^2_{r,x'}}$  is $(w_{r,x \res i})_{i \leq \lh(x)}$.
  \item $\mathbf{D}^2_{r,x} = \mathbf{D}^2_{r,x'} \res (Q_r, W_{r,x})$.
  \item If $x \in D^2_r \cup E^2_r$, $x \concat (c) , x \concat (d) \in \dom(\comp{2}{X})$, $c<_{BK}d$, then $\delta^2_{r,x\concat(c), v^2_{r,x}} < \delta^2_{r,x \concat (d), v^2_{r,x}}$.
  \item If $x \in D^2_r \cup E^2_r$, $[h]_{\mu^{S^2_{r,x}}} = \delta^2_{r,x,v^2_{r,x}}$, then  for any $g \in E^{Q_r\uparrow}$,
    \begin{displaymath}
      [h \circ g^{Q_r, W_{r,x}}_{\tau^2_{r,x}}]_{\mu^{W_{r,x}}} = \theta_r ( [ g]^{Q_r}).
    \end{displaymath}
  \end{enumerate}
\end{myclaim}
\begin{proof}
By Lemma~\ref{lem:jQ_move_factor_maps}, $j^{Q_r}(j^{W_{r,x}, W_{r, x'}} \res j^{W_{r,x}}(\omega_1+1)) = j^{Q_r \otimes W_{r,x}, Q_r \otimes W_{r,x'}} \res j^{Q_r \otimes W_{r,x}}(\omega_1+1)$ and $j^{Q_r}_{\sup}(j^{W_{r,x}, W_{r, x'}} \res j^{W_{r,x}}(\omega_1+1)) = j^{Q_r \otimes W_{r,x}, Q_r \otimes W_{r,x'}}_{\sup} \res j^{Q_r \otimes W_{r,x}}(\omega_1+1)$.  Since $\theta_r$ takes values in $[\omega_1]^{X_r \uparrow}$ on a $\mu^{Q_r}$-measure one set, for $\mu^{Q_r}$-a.e.\ $\vec{\xi}$, we have
\begin{displaymath}
 j^{W_{r,x}, W_{r,x'}} (\comp{2}{\theta}_{r,x} (\vec{\xi})) < \comp{2}{\theta}_{r,x'} (\vec{\xi})< j^{ W_{r,x},W_{r,x'}}(\comp{2}{\theta}_{r,x} (\vec{\xi}))
\end{displaymath}
and
\begin{displaymath}
  \cf^{\mathbb{L}}(\comp{2}{\theta}_{r,x}(\vec{\xi}) ) = \seed^{W_{r,x}}_{w_{r,x}^{-}}.
\end{displaymath}
Hence by \Los{},
\begin{displaymath}
  j^{Q_r \otimes W_{r,x}, Q_r \otimes W_{r,x'}}_{\sup} (\comp{2}{\gamma}_{r,x}) < \comp{2}{\gamma}_{r,x'} <  j^{Q_r \otimes W_{r,x}, Q_r \otimes W_{r,x'}}(\comp{2}{\gamma}_{r,x})
\end{displaymath}
and by Lemma~\ref{lem:jQ_move_continuity},
\begin{displaymath}
  \ucf (\mathbf{D}_{r,x}) = w_{r,x}^{-}. 
\end{displaymath}
We are in a position to apply Lemma~\ref{lem:ordinal_division_blocks} with
\begin{displaymath}
  A = \set{l}{ \exists \mathbf{D} \in \desc(Q_r,W_{r,x}) ~ u_l = \seed_{\mathbf{D}}^{Q_r, W_{r,x'}}},
\end{displaymath}
leading to parts 1-4. Part 5 also follows from Lemma~\ref{lem:ordinal_division_blocks}, using the fact that $\gamma_{r, x \concat (c)} < \gamma_{r, x\concat (d)}$. We now prove part 6. Note that $ \tau^2_{r,x}$ factors $(S^2_{r,x}, Q_r \otimes W_{r,x})$ and in fact, $( \tau_{r,x}^2)^{Q_r \otimes W_{r,x}} (\delta^2_{r,x,v_{r,x}}) = \gamma^2_{r,x}$. Suppose we are given $h$ with $[h]_{\mu^{S^2_{r,x}}} = \delta^2_{r,x,v^2_{r,x}}$. Define $h_{*}$ on $[E]^{Q_r\uparrow}$ by $h_{*} ([g]^{Q_r}) = [h \circ g^{Q_r,W_{r,x}}_{\tau^2_{r,x}}]_{\mu^{W_{r,x}}} $. By \Los{}, it suffices to show that $[h_{*}]_{\mu^{Q_r}} = \gamma^2_{r,x}$. But this follows from  Lemma~\ref{lem:factor_SQW}. 
This finishes the proof of Claim~\ref{claim:YTQ_division}.
\end{proof}

In parallel to Claim~\ref{claim:XTQ_factor_cont}, we have
\begin{myclaim}
  \label{claim:YTQ_division_cont}
  Suppose $r, r' \in \dom(R)$, $r = (r')^{-}$, $y_r$ is of continuous type, and the contraction of $((\sign_{*}(\mathbf{D}^{e_{r,j}}_{r,x_{r,j},i}))_{i < v^{e_{r,j}}_{r,x_{r,j}}})_{1 \leq j < \lh(x_r)}$ is $((d_{r \res i}, q _{r \res i}))_{1 \leq i < \lh(r)}$. Then
  \begin{enumerate}
  \item $y_r = y_{r'}$. 
  \item For any $(e,x) \in \dom(X_r) \setminus \se{(e_r,x_r)}$, $\comp{e}{\gamma}_{r,x} = \comp{e}{\gamma}_{r', x}$.
  \item $(\mathbf{D}^{e_r}_{r,x_r,i}, \delta^{e_r}_{r,x_r,i})_{i < v^{e_r}_{r,x_r}}$ is a proper initial segment of $(\mathbf{D}^{e_{r}}_{r',x_{r},i}, \delta^{e_{r}}_{r',x_{r},i})_{i < v^{e_{r}}_{r',x_{r}}}$. Hence $S_{r, x_r}$ is a proper subtree of $S_{r ', x_r}$, and $\vec{s}_{r,x_r}$ is an initial segment of $\vec{s}_{r', x_r}$. 
  \item The signature of $ \tau^{e_r}_{r', x_r}(s_{r,x_r})$ is $((1, q_r))$ if $d_r =1$, $((2, q_{r \res i}))_{1 \leq i \leq \lh(q_r)}$ if $d_r = 2$. In particular, the contraction of $((\sign_{*}(\mathbf{D}^{e_{r',j}}_{r',x_{r',j},i}))_{i < v^{e_{r',j}}_{r',x_{r',j}}})_{1 \leq j < \lh(x_r)}$ is $((d_{r \res i}, q _{r \res i}))_{1 \leq i \leq \lh(r)}$.
  \end{enumerate}
\end{myclaim}
\begin{proof}
  By order preservation and  continuity  of $\theta$, $y_r =y _{r'}$ and  for $\mu^{Q_r}$-a.e.\ $\vec{\beta}$,
  \begin{enumerate}
  \item for any $(e,x) \in \dom(X_r) \setminus \se{(e_r,x_r)}$, if $\vec{\beta}'$ extends $\vec{\beta}$ then $\comp{e}{\theta}_{r,x} (\vec{\beta}) = \comp{e}{\theta}_{r',x}(\vec{\beta}')$;
  \item $ {\comp{e_r}{\theta}}_{r,x_r}(\vec{\beta}) = \sup\set{{\comp{e_r}{\theta}}_{r',x_r}(\vec{\beta}')}{ \vec{\beta}' \text{ extends } \vec{\beta}}$.
  \end{enumerate}
Thus, $\comp{e}{\gamma}_{r,x} = \comp{e}{\gamma}_{r',x} $ for any $(e,x) \in \dom(X_r) \setminus \se{(e_r,x_r)}$, and $j_{\sup}^{Q_r,Q_{r'}}(\comp{e}{\gamma}_{r,x} ) \leq \comp{e}{\gamma}_{r',x} < j^{Q_r,Q_{r'}}(\comp{e}{\gamma}_{r,x} ) $.   As $t_x = t_{x'}$ is of continuous type and $(d_r,q_r)$ does not appear in $\sign(\comp{e}{\theta}_{r',x})$ for any $(e,x) \in \dom(X_{r'}) \setminus \se{(e_{r'}, x_{r'})}$,  $\theta_{r' , x_r}$ is essentially discontinuous, giving $j_{\sup}^{Q_r,Q_{r'}} (\comp{e}{\gamma}_{r,x} )  \neq  \comp{e}{\gamma}_{r',x}$. 
With the help of Lemma~\ref{lem:factor_SQW} again, we can find level-1 trees $M_r, M_{r'}$ such that $M_r$ is a subtree of $M_{r'}$ and $j_{\sup}^{M_r,M_{r'}}(\comp{e}{\gamma}_{r,x} ) < \comp{e}{\gamma}_{r',x} < j^{M_r,M_{r'}}(\comp{e}{\gamma}_{r,x} ) $.   
The claim then follows from Lemma~\ref{lem:ordinal_division_blocks}.  
\end{proof}

In parallel to Claim~\ref{claim:XTQ_factor_discont}, we have
\begin{myclaim}
  \label{claim:YTQ_division_discont}
  Suppose $r, r' \in \dom(R)$, $r = (r')^{-}$, $y_r$ is of discontinuous type, and the contraction of $((\sign_{*}(\mathbf{D}^{e_{r,j}}_{r,x_{r,j},i}))_{i < v^{e_{r,j}}_{r,x_{r,j}}})_{1 \leq j < \lh(x_r)}$ is $((d_{r \res i}, q _{r \res i}))_{1 \leq i < \lh(r)}$. Put $x_r^{*} = x_r^{-}$ if $\ucf(R(r))\notin \desc(Q_r)$, $x_r^{*} = x_r^{-}\concat( (x_r(\lh(x_r)-1))^{-})$ if $\ucf(R(r))\in \desc(Q_r)$. 
 Then
  \begin{enumerate}
  \item $y_r \subsetneq y_{r'}$. 
  \item For any $(e,x) \in \dom(X_r) $, $\comp{e}{\gamma}_{r,x} = \comp{e}{\gamma}_{r', x}$.
  \item $(\mathbf{D}^{e_r}_{r,x_r,i}, \delta^{e_r}_{r,x_r,i})_{i < v^{e_r}_{r,x_r}}$ is a proper initial segment of $(\mathbf{D}^{e_{r}}_{r',x_{r}^{*},i}, \delta^{e_{r}}_{r',x_{r}^{*},i})_{i < v^{e_{r}}_{r',x_{r}^{*}}}$. 
The signature and approximation sequence of $\comp{e_r}{\gamma}_{r, x_r^{*}}$ are proper initial segments of those of $\comp{e_r}{\gamma}_{r', x_r}$. Hence $S_{r, x_r^{*}}$ is a proper subtree of $S_{r', x_r}$, and $\vec{s}_{r,x_r^{*}}$ is an initial segment of $\vec{s}_{r', x_r}$. 
  \item The signature of $ \tau^{e_r}_{r', x_r}(s_{r,x_r})$ is $((1, q_r))$ if $d_r =1$, $((2, q_{r \res i}))_{1 \leq i \leq \lh(q_r)}$ if $d_r = 2$. In particular, the contraction of $((\sign_{*}(\mathbf{D}^{e_{r',j}}_{r',x_{r',j},i}))_{i < v^{e_{r',j}}_{r',x_{r',j}}})_{1 \leq j < \lh(x_{r'})}$ is $((d_{r \res i}, q _{r \res i}))_{1 \leq i \leq \lh(r)}$.
  \end{enumerate}
\end{myclaim}
\begin{proof}
    By order preservation and  continuity  of $\theta$, $y_r \subsetneq y_{r'}$ and for $\mu^{Q_r}$-a.e.\ $\vec{\beta}$,
    \begin{enumerate}
  \item for any $(e,x) \in \dom(X_r) $, if $\vec{\beta}'$ extends $\vec{\beta}$ then $\comp{e}{\theta}_{r,x} (\vec{\beta}) = \comp{e}{\theta}_{r',x}(\vec{\beta}')$;
  \item if $\ucf(R(r)) \notin \desc(Q_r)$ then  $ j^{X_r, X_r^{*}}({\comp{e_r}{\theta}}_{r,x_r^{*}}(\vec{\beta}) )= \sup\set{{\comp{e_r}{\theta}}_{r',x_r}(\vec{\beta}')}{ \vec{\beta}' \text{ extends } \vec{\beta}}$, where $X_r^{*}= Y_{\tree}(y_{r'} \res \lh(y_r)+1)$;
  \item if $\ucf(R(r)) \in \desc(Q_r)$ then  $ {\comp{e_r}{\theta}}_{r,x_r^{*}}(\vec{\beta}) = \sup\set{{\comp{e_r}{\theta}}_{r',x_r}(\vec{\beta}')}{ \vec{\beta}' \text{ extends } \vec{\beta}}$.
    \end{enumerate}
The rest is similar to the proof of Claim~\ref{claim:YTQ_division_cont}. 
\end{proof}

Let
\begin{displaymath}
  \phi^1: \set{\comp{1}{\gamma}_{r,x}}{r \in \dom(R), x \in B^1_r} \to Z^1
\end{displaymath}
be a bijection such that $Z^1$ is a level-1 tree and $v < v' \eqiv \phi^1(v) \prec^{Z^1} \phi^1(v')$. Let
\begin{displaymath}
  \phi^2 :  \set{(   \mathbf{D}^{e}_{r,x,i} , \delta^e_{r,x,i} )_{i < l} }{r \in \dom(R), e \in \se{1,2}, x \in D^e_r\cup E^e_r, l < \lh(\vec{s}_{r,x}) } \to Z^2 \cup \se{\emptyset}
\end{displaymath}
be a bijection such that $Z^2$ is a tree of level-1 trees and $v \subseteq v' \eqiv \phi^2(v) \subseteq \phi^2(v')$, $v <_{BK} v'$ iff $\phi^2(v) <_{BK} \phi^2(v)$, where the ordering of subcoordinates $\mathbf{D}^e_{r,x,i}$ is according to $\prec$. Let
\begin{displaymath}
  T = (\comp{1}{T},\comp{2}{T})
\end{displaymath}
where $\comp{1}{T} = Z^1$, $\comp{2}{T}$ is a level-2 tree, $\dom(\comp{2}{T}) = Z^2$,
\begin{align*}
  \comp{2}{T}  [ \phi^2 ( (\mathbf{D}^e_{r,x,i}, \delta^e_{r,x,i})_{i < \lh(\vec{s}_{r,x})-1} )\concat (-1 ) ] & = (S_{x,r},\vec{s}_{x,r}) \text{ for } x \in D_r^e,\\
  \comp{2}{T}  [ \phi^2 ( (\mathbf{D}^e_{r,x,i}, \delta^e_{r,x,i})_{i < \lh(\vec{s}_{r,x})-1}  ) ] & = (S_{x,r},\vec{s}_{x,r}) \text{ for } x \in E_r^e.
\end{align*}
Let
\begin{displaymath}
  \vec{\delta} = (\comp{c}{\delta}_t)_{(c,t) \in \dom(T)}
\end{displaymath}
where $\comp{1}{\delta}_t = (\phi^1)^{-1}(t)$, $\comp{2}{\delta}_{\emptyset} = \omega_1$, $\comp{2}{\delta}_t = \delta^e_{r,x,l}$ where $t = \phi^2 ( (\mathbf{D}^e_{r,x,i}, \delta^e_{r,x,i})_{i \leq l} )$. For $r \in \dom(R)$, let
\begin{displaymath}
  \rho(r) = ( \mathbf{y}_r  , \pi_r)
\end{displaymath}
where $\pi_r$ factors $(X_r, T, Q_r)$, defined as follows: 
\begin{displaymath}
  \pi_r(e,x) =
  \begin{cases}
    (1, \phi^1(\gamma^e_{r,x}), \emptyset) & \text{ if } x \in B^e_r, \\
    (2, ( \phi^2 ( (\mathbf{D}^e_{r,x,i}, \delta^e_{r,x,i})_{i < \lh(\vec{s}_{r,x})-1}) \concat (-1)  , S^e_{r,x}, \vec{s}^e_{r,x}  ), \tau^e_{r,x}  ) & \text{ if } x \in D^e_r , \\
    (2, ( \phi^2 ( (\mathbf{D}^e_{r,x,i}, \delta^e_{r,x,i})_{i < \lh(\vec{s}_{r,x})-1}), S^e_{r,x}, \vec{s}^e_{r,x}  ), \tau^e_{r,x}  ) & \text{ if } x \in E^e_r.
  \end{cases}
\end{displaymath}
It is easy to check that $(T,\rho, \vec{\delta})$ works for the lemma. 
\end{proof}

Put $\llbracket \emptyset \rrbracket_R = \ot(<^R)$.
For $\mathbf{r}=(r,Q, \overrightarrow{(d,q,P)}) \in \exdesc(R)$, put
\begin{displaymath}
 \llbracket \mathbf{r} \rrbracket_{R} = [\vec{\beta} \mapsto \wocode{\vec{\beta}\oplus_R r}_{<^R}]_{\mu^Q}.
\end{displaymath}
If $\mathbf{r}\in \desc(R)$ is of discontinuous type, put $\llbracket r \rrbracket_{R} = \llbracket \mathbf{r} \rrbracket_{R}$. 
Note that if $\rho$ factors $\Pi^1_3$-wellfounded trees $(R,Y)$, then $\llbracket r \rrbracket_R \leq \llbracket \rho(r) \rrbracket_Y$ for any $ r \in \dom(R)$. We say that $\rho$ \emph{minimally factors} $(R,Y)$ iff $\rho$ factors $(R,Y)$, $R,Y$ are both $\Pi^1_3$-wellfounded and $\llbracket r \rrbracket_R = \llbracket \rho(r) \rrbracket_Y$ for any $r \in \dom(R)$. In particular, if $Y$ is $\Pi^1_3$-wellfounded and $T$ is $\Pi^1_2$-wellfounded, then $\id_{Y,*}$ minimally factors $(Y, Y \otimes T)$. 
In the assumption of Lemma~\ref{lem:RYT_factor_from_division}, if $R,Y$ are $\Pi^1_3$-wellfounded and $\ran(\theta)$ is a $<^Y$-initial segment of $\rep(Y)$, its proof constructs $\rho$ which minimally factors $(R, Y \otimes T)$. This entails the comparison theorem between $\Pi^1_3$-wellfounded trees.

\begin{mytheorem}
  \label{thm:factor_tower_order_type_equivalent}
  Suppose $R,Y$ are $\Pi^1_3$-wellfounded level-3 trees and $\llbracket \emptyset \rrbracket_R \leq \llbracket \emptyset \rrbracket_Y$. Then there exists $(T, \rho)$ such that $T$ is $\Pi^1_2$-wellfounded and $\rho$ minimally factors $(R,Y \otimes T)$. Furthermore, if $\llbracket \emptyset \rrbracket_R < \llbracket \emptyset \rrbracket_Y$, we further obtain $\mathbf{B} \in \dom(Y \otimes T)$ such that $\lh(\mathbf{B}) = 1$ and $\llbracket \emptyset \rrbracket_R = \llbracket \mathbf{B} \rrbracket_{Y \otimes T}$.
\end{mytheorem}

The goal of the remaining of this section is to prove Lemma~\ref{lem:order_type_realizable}, which states that every $\vec{\gamma} $ respecting a finite level-3 tree $R$ is ``representable''.   Lemma~\ref{lem:order_type_realizable} will essentially be a strengthening of~  \cite[Theorem 5.3]{jackson_handbook}. 

The next lemma is an easy corollary of Lemma~\ref{lem:Q_respecting}. In its statement, $(T, \vec{\gamma})$ is the ``amalgamation'' of $(Q, \vec{\beta})$ and $(Q', \vec{\beta}')$. 
\begin{mylemma}
  \label{lem:level_2_amalgamation}
  Suppose $Q,Q'$ are level $\leq 2$ trees, $\vec{\beta}  = (\comp{d}{\beta}_q)_{(d,q ) \in \dom(Q)}$ respects $Q$, $\vec{\beta}' = (\comp{d}{\beta}'_q)_{(d,q) \in \dom(Q')}$ respects $Q'$. Then there are a level $\leq 2$ tree $T$, a tuple $\vec{\gamma} = (\comp{d}{\gamma}_t)_{(d,t) \in \dom(T)}$ and maps $\pi,\pi'$ factoring $(Q, T)$, $(Q',T)$ respectively such that $\dom(T) = \ran(\pi) \cup \ran( \pi')$,  $\comp{d}{\gamma}_{\comp{d}{\pi}(q)} = \comp{d}{\beta}_q$ for any $(d,q) \in \dom(Q)$, 
$\comp{d}{\gamma}_{\comp{d}{\pi}'(q)} = \comp{d}{\beta}'_q$ for any $(d,q) \in \dom(Q')$. 
\end{mylemma}

Amalgamation of level-3 trees is similar, using Lemma~\ref{lem:R_respect} instead. 

\begin{mylemma}
  \label{lem:level_3_amalgamation}
  Suppose $R,R'$ are level-$3$ trees, $\vec{\gamma}  = ({\gamma}_r)_{r\in \dom(R)}$ respects $R$, $\vec{\gamma}' = ({\gamma}'_r)_{r \in \dom(R')}$ respects $R'$. Then there are a level-3 tree $Y$, a tuple $\vec{\delta} = ({\delta}_y)_{y \in \dom(Y)}$ and maps $\rho,\rho'$ factoring $(R, Y)$, $(R',Y)$ respectively such that $\dom(Y) = \ran(\rho) \cup \ran( \rho')$,  ${\delta}_{\rho(r)} = {\gamma}_r$ for any $r \in \dom(R)$, 
${\delta}_{\rho'(r)} = {\gamma}'_r$ for any $r \in \dom(R')$.
\end{mylemma}

\begin{mylemma}
  \label{lem:level_2_tree_cofinal_in_u_2}
For any $a \in \omega^{<\omega}$,   $\{  \llbracket 2, (a) \rrbracket_Q  : Q $ is a $\Pi^1_2$-wellfounded level $\leq 2$ tree, $(a) \in \dom({Q})\}$ is a cofinal subset of $u_2$. 
\end{mylemma}
\begin{proof}
Note that $u_2 = \bolddelta{2}$ is the sup of ranks of $\boldsigma{2}$ wellfounded relations on $\mathbb{R}$. 
 Given  $< ^{*}$, a $\boldsigma{2}$ wellfounded on $\mathbb{R}$, we need to find a $\Pi^1_2$-wellfounded level $\leq 2$ tree $Q$ such that $\rank(<^{*}) \leq \llbracket 2,  ((0)) \rrbracket_Q$. This suffices for the Lemma by rearranging the nodes in a level $\leq 2$ tree in a suitable way.  Put $x<^{*}x' \eqiv \exists y~ x \oplus x' \oplus y \in A$, where $A$ is $\boldpi{1}$. 
Let $(P_s)_{s \in \omega^{<\omega}}$ be a regular level-1 system such that $P_{x \oplus x' \oplus y}$ is $\Pi^1_1$-wellfounded iff $  x \oplus x' \oplus y \in A$. 
Fix  an effective bijection $\phi: \omega^{<\omega} \eqiv (\omega^{<\omega})^{<\omega}$. If $(W_n)_{n<\omega}$ is a sequence of nonempty level-1 trees, their join is $\oplus_{n<\omega} W_n = \set{(n) \concat w}{w \in W_n}$.
Let $Q^{*}$ be an infinite level-2 tree whose domain is $\set{((0)) \concat q}{q \in (\omega^{<\omega})^{<\omega}}$, and for any real $v$, $\cup_{n<\omega} Q^{*}_{\tree}(((0))\concat \phi(v\res n)) =\oplus_{n<\omega} P_{ (v)_{2n+2} \oplus (v)_{2n} \oplus (v)_{2n+1}}$. 
Then $Q^{*}$ is $\Pi^1_2$-wellfounded. Let $Q = (\emptyset, Q^{*})$. 
The proof of Kunen-Martin shows that $\rank(<^{*}) \leq \llbracket 2, ((0)) \rrbracket_Q$. 
\end{proof}

\begin{mylemma}
  \label{lem:level_2_tree_cofinal_general}
  Suppose $Q$ is a $\Pi^1_2$-wellfounded level $\leq 2$ tree, $q^{*} \in \dom(\comp{2}{Q})$, $P^{*}$ is the completion of $\comp{2}{Q}(q^{*})$. 
Then $\{\llbracket 2, q' \rrbracket_{Q'}: Q'$ is $ \Pi^1_2$-wellfounded,  $\llbracket 2,q^{*} \rrbracket_Q = \llbracket 2,(q')^{-} \rrbracket_{Q'}$,  $\comp{2}{Q}'_{\tree}(q') = P^{*}\}$ is a cofinal subset of $j^{\comp{2}{Q}_{\tree}(q^{*}), P^{*}} (\llbracket 2, q^{*} \rrbracket_Q)$. 
\end{mylemma}
\begin{proof}
If $q^{*} = \emptyset$,  
we are reduced to Lemma~\ref{lem:level_2_tree_cofinal_in_u_2}. Suppose now $q^{*}\neq \emptyset$. Put $\comp{2}{Q}(q^{*}) = (P^{-}, p^{*})$, so $P^{*}$ is the completion of $(P^{-},p^{*})$.

Let $p^{**} = \pred_{\prec^{P^{*}}}(p^{*})$. 
By remarkability of the level-1 sharps, letting $f (\beta) = [\vec{\alpha} \mapsto \wocode{(2, \vec{\alpha} \res P^{-}\concat g(\alpha_{p^{*}}) \oplus_{\comp{2}{Q}} q^{*} \concat (-1))}_{<^Q}]_{\mu^{P^{-}}}$ for $\beta = [g]_{\mu_{\mathbb{L}}}< u_2$, then $\sup f'' u_2 = j^{P^{-}, P^{*}} (\llbracket 2, q^{*} \rrbracket_Q)$. Fix $\beta = [g]_{\mu_{\mathbb{L}}}< u_2$, and we try to find $Q', q'$ 
such that $\comp{2}{Q}'[(q')^{-}] = \comp{2}{Q}[q^{*} ]$, $\comp{2}{Q}'_{\tree}(q') = P^{*}$,  and $ f(\beta) < \llbracket 2, q' \rrbracket_{Q'}$. Let $U$ be a $\Pi^1_2$-wellfounded level $\leq 2$ tree obtained by Lemma~\ref{lem:level_2_tree_cofinal_in_u_2} such that $\beta < \llbracket 2, ((0)) \rrbracket_U $. 
Let $(X,\pi)$ be a representation of $Q \otimes U$,   and let $\theta : \rep(X) \to \rep(Q)$ be the order preserving bijection. Let $\mathbf{C} = (2, \mathbf{q}, \tau) \in \desc(Q, U, *)$, where $\mathbf{q}= (q^{*}\concat (-1), P^{*}, \vec{p})$, $\tau$ extends $\id_{*,P^{-}}$,  $\tau(p^{*}) = (2, ((0), \se{(0)}, ((0))), \sigma)$, $\sigma((0)) = p^{**}$. 
Let $(2, x) = \pi^{-1}(\mathbf{C})$. 
Then 
 for $\mu^{P}$-a.e.\ $\vec{\alpha}$, $\theta(2, \vec{\alpha} \oplus_{\comp{2}{X}} x) = (2, \vec{\alpha}\res P^{-}\concat (g (\alpha_{p^{*}}) ) \oplus_{\comp{2}{Q}} q^{*}\concat (-1) )$. Therefore, $\llbracket 2, x \rrbracket_Q = f( \llbracket 2, ((0))\rrbracket_U )>f(\beta)$.  $(X, x)$ plays the role of the desired $(Q',q')$. 
\end{proof}

Suppose $Q$ is a level $\leq 2$ tree and $\vec{\epsilon} = (\comp{d}{\epsilon}_t)_{(d,t) \in \dom(Q)}$ is a tuple of ordinals indexed by $\dom(Q)$. We say that \emph{$\vec{\epsilon}$ is represented by $Q'$} iff $Q$ is a subtree of $Q'$, $Q'$ is $\Pi^1_2$-wellfounded and $\vec{\epsilon} =( \llbracket d,t \rrbracket_{Q'})_{(d,t) \in \dom(Q')}$. 

\begin{mylemma}
  \label{lem:order_type_realizable_level_2}
  Suppose $Q$ is a finite level $\leq 2$ tree and $\vec{\beta} = (\comp{d}{\beta}_q)_{(d,q) \in \dom(Q)}$ respects $Q$. Then $\vec{\beta}$ is represented by some level $\leq 2$ tree $Q'$.
\end{mylemma}
\begin{proof}
By rearranging the nodes in $\dom(Q')$ in a suitable way, it suffices to produce a level $\leq 2$ tree $Q'$ and a map $\pi$ factoring $(Q,Q')$ such that for any $(d,q) \in \dom(Q)$, $\comp{d}{\beta}_q = \llbracket \pi(d,q) \rrbracket_{Q'}$.  By a repeated application of  Lemma~\ref{lem:level_2_amalgamation}, it suffices to show that for any $(d^{*}, q^{*} \concat (a)) \in \dom(Q)$, 
  \begin{enumerate}
  \item if $d^{*}=1$, then there is a $\Pi^1_2$-wellfounded level $\leq 2$ tree $Q'$  and $q' \in \comp{1}{Q}'$ such that $\comp{1}{\beta}_{q^{*}\concat (a)} = \llbracket 1, q' \rrbracket_{Q'}$.
  \item if $d^{*}=2$ and $P^{*} = \comp{2}{Q}_{\tree}(q^{*} \concat (a))$,  then there is a $\Pi^1_2$-wellfounded level $\leq 2$ tree $Q'$ and $q' \in \dom(\comp{2}{Q}')$ such that  $\comp{2}{\beta}_{q^{*}\concat (a)} = \llbracket 2, q' \rrbracket_{Q'}$, $\comp{2}{Q}'[(q')^{-}] = \comp{2}{Q}[q^{*}]$, $\comp{2}{Q}'_{\tree}(q') = P^{*}$.
  \end{enumerate}
The case $d^{*}=1$ is obvious. We assume now $d^{*}=2$. 



Lemma~\ref{lem:level_2_tree_cofinal_general} gives us a $\Pi^1_2$-wellfounded level $\leq 2$ tree $T$ and $t \in \dom(\comp{2}{T})$ such that $\llbracket 2, t \rrbracket_T \geq \comp{2}{\beta}_{q^{*}\concat (a)}$, $\comp{2}{T}[t^{-}] = \comp{2}{Q}[q^{*}], \comp{2}{T}_{\tree}(t) = P^{*}$. 
Minimizing $\llbracket 2, t \rrbracket_T$, we may further assume that for any $\Pi^1_2$-wellfounded $T'$ and any $t' $ such that $\llbracket 2, t' \rrbracket_{T'}\geq \comp{2}{\beta}_{q^{*}\concat (a)}$, $\comp{2}{T}'[(t')^{-}] = \comp{2}{Q}[q^{*}], \comp{2}{T}'_{\tree}(t') = P^{*}$,  we have $\llbracket 2, t' \rrbracket_{T'} \geq \llbracket 2, q^{*}\concat (a) \rrbracket_T$. 
We claim that $\llbracket 2, t \rrbracket_{T} = \comp{2}{\beta}_{q^{*} \concat (a)}$. Suppose otherwise. Put $p^{*} = \comp{2}{T}_{\node}(t)$.

Case 1: $\cf^T(2,t) = 0$. 

If $\comp{2}{T}\se{t,-}$ has a $<_{BK}$-maximum $t'$, then $\llbracket 2,t \rrbracket_T = \llbracket 2,t' \rrbracket_T + \omega$. 
So $\comp{2}{\beta}_{q^{*}\concat (a)} \leq \llbracket 2, t' \rrbracket_T < \llbracket 2, t \rrbracket_T$, contradicting the minimization assumption. 
If $\comp{2}{T}\se{t,-}$ has $<_{BK}$-limit order type, then 
$\llbracket 2, t \rrbracket_T = \sup \set{ \llbracket 2, t' \rrbracket_T}{t' \in \comp{2}{T}\se{t,-}}$, 
so there is $t'$ satisfying $\comp{2}{\beta}_{q^{*}\concat (a)} \leq \llbracket 2, t' \rrbracket_T < \llbracket 2, t \rrbracket_T$, contradiction again.

Case 2: $\cf^T(2,t) = 1$.

For $\beta  = \omega_1$, put $f(\beta) = [\vec{\alpha} \mapsto \wocode{(2, \vec{\alpha} \concat (\beta) \oplus_{\comp{2}{T}} t \concat (-1))}_{<^T} ]_{\mu^{P^{*}}}$. Then $\llbracket 2, t \rrbracket_T = \sup\set{f(\beta)}{\beta < \omega_1}$.
  For each limit $\beta<\omega_1$, we shall find a $\Pi^1_2$-wellfounded $T'$ and a node $t'$ such that $\llbracket 2, t'\rrbracket_{T'} = f(\beta)$, $\comp{2}{T}'[(t')^{-}] = \comp{2}{Q}[q^{*}], \comp{2}{T}'_{\tree}(t') = P^{*}$, contradicting to the minimization assumption. Fix a limit ordinal $\beta < \omega_1$. Let $U$ be a $\Pi^1_2$-wellfounded level $\leq 2$ tree such that $ \llbracket 1, (0) \rrbracket_U =\beta$. 
Let $(X,\pi)$ be a representation of $T \otimes U$  and let $\theta : \rep(X) \to \rep(T)$ be the order preserving bijection. Let $\mathbf{C} = (2, \mathbf{t}, \tau) \in \desc(T,U,*)$, $\mathbf{t} = (t\concat (-1),S,\vec{s})$, $\tau$ extends $\id_{*,S}$, $\tau(s_{\lh(\vec{s})-1}) = (1, (0), \emptyset)$. Let $(2, x) = \pi^{-1}(\mathbf{C})$. Then  
for $\mu^{P^{*}}$-a.e.\ $\vec{\alpha}$, $\theta(2, \vec{\alpha} \oplus_{ \comp{2}{X}} x) = (2, \vec{\alpha} \concat (\beta) \oplus_{\comp{2}{T}} t \concat(-1))$. Therefore, $\llbracket 2, x \rrbracket_X = f(\beta)$.
$(X, x)$ plays the role of the desired $(T',t')$. 

Case 3: $\cf^T(2,t) = 2$. 

Let  $p^{**}$ be the $<_{BK}\res P^{*} \cup \se{p^{*}}$-predecessor of $p^{*}$. 
For $\beta  = [g]_{\mu_{\mathbb{L}}} < u_2$, put $f(\beta) = [\vec{\alpha} \mapsto \wocode{(2, \vec{\alpha} \concat g(\alpha_{p^{**}}) \oplus_{\comp{2}{T}} t \concat (-1))}_{<^T} ]_{\mu^{P^{*}}}$.  Then $\llbracket 2, t \rrbracket_T = \sup\set{f(\beta)}{\beta < u_2}$. For each limit $\omega_1 < \beta < u_2$, we shall find a $\Pi^1_2$-wellfounded $T'$ and a node $t'$ such that $\llbracket 2, t'\rrbracket_{T'} = f(\beta)$, $\comp{2}{T}'[(t')^{-}] = \comp{2}{Q}[q^{*}], \comp{2}{T}'_{\tree}(t') = P^{*}$. Fix a limit ordinal $\omega_1< \beta < u_2$. By Case 1 and Case 2 of this lemma applied to $(2,((0))) $ in place of $(d^{*}, q^{*} \concat (a))$, we can find a $\Pi^1_2$-wellfounded level $\leq 2$ tree $U$ such that $\llbracket 2, (0) \rrbracket_U = \beta$. Let $(X, \pi)$ be a representation of $T \otimes U$ and let $\theta: \rep(X) \to \rep(T)$ be the order preserving bijection. Let $\mathbf{C} = (2, \mathbf{t}, \tau)$, $\mathbf{t} = (t\concat (-1), S, \vec{s})$, $\tau $ extends $\id_{*,S}$, $\tau(s_{\lh(\vec{s})-1}) = (2, ( (0), \se{(0)}, ((0)) ), \sigma)$, $\sigma((0)) = p^{**}$. Let $(2, x)= \pi^{-1}(\mathbf{C})$. $(X,x)$ plays the role of the desired $(T',t')$. 
\end{proof}

The level-3 version of Lemmas~\ref{lem:level_2_tree_cofinal_in_u_2}-\ref{lem:order_type_realizable_level_2} are similarly proved.

\begin{mylemma}
  \label{lem:level_3_tree_cofinal_in_delta13}
For any $a \in \omega^{<\omega}$,   $\{  \llbracket (a) \rrbracket_R  : R $ is a $\Pi^1_3$-wellfounded level-3 tree, $(a) \in \dom({R})\}$ is a cofinal subset of $\bolddelta{3}$. 
\end{mylemma}
\begin{proof}
  It is possible to imitate the proof of Lemma~\ref{lem:level_2_tree_cofinal_in_u_2}. We give an alternative proof using the prewellordering property of the pointclass $\boldpi{3}$.  Let $G$ be a good universal $\Pi^1_3$-set and let $(R_s)_{s \in \omega^{<\omega}}$ be an effective level-3 system satisfying
 $x \in G$ iff $R_x \DEF \cup_{n<\omega} R_{x \res n}$ is $\Pi^1_3$-wellfounded. $G$ is equipped with the  $\Pi^1_3$-norm $\varphi(x) = \llbracket \emptyset \rrbracket_{R_x}$, the complexity from Kechris-Martin.  By Moschovakis \cite[4C.14]{mos_dst}, $\ot(\ran(\varphi)) = \bolddelta{3}$. The rest of the proof is simple.
\end{proof}

\begin{mylemma}
  \label{lem:level_3_tree_cofinal_general}
  Suppose $R$ is a $\Pi^1_3$-wellfounded level-3 tree, $r^{*} \in \dom(R)$, $Q^{*}$ is a completion of ${R}(r^{*})$. 
Then $\{\llbracket r' \rrbracket_{R'}: R'$ is $ \Pi^1_3$-wellfounded, ${R}'[(r')^{-}] = {R}[r^{*} ]$, ${R}'_{\tree}(r') = Q^{*}\}$ is a cofinal subset of $j^{{R}_{\tree}(r^{*}), Q^{*}} (\llbracket r^{*} \rrbracket_R)$. 
\end{mylemma}
\begin{proof}
Put $R(r^{*}) = (Q^{-}, (d^{*}, q^{*}, P^{*}))$, $R[r^{*}] = (Q^{-}, \overrightarrow{(d, q, P)})$. 

Case 1: $\cf(R(r^{*})) = 1$.

By Lemma~\ref{lem:level_2_desc_cofinal_in_next}, letting $f(\xi) = [\vec{\beta} \mapsto \wocode{\vec{\beta} \concat (\xi)  \oplus_R r^{*}}_{<^R}]_{\mu^{Q^{-}}}$ for $\xi < \omega_1$, then $\sup f''\omega_1 = j^{Q^{-}, Q^{*}} ( \llbracket r^{*} \rrbracket_R)$. Fix $\beta < \omega_1$, and we try to find $R'$ and $r'$ such that $R'[(r')^{-}] = R[r^{*}]$, $R'(r') = Q^{*}$, and $f(\beta) < \llbracket r' \rrbracket_{R'}$. Let $U$ be a $\Pi^1_2$-wellfounded level $\leq 2$ tree such that $\beta < \llbracket 1, (0) \rrbracket_U $. Let $(Z, \rho)$ be a representation of $R \otimes U$, and let $\theta : \rep(Z) \to \rep(U)$ be the order preserving bijection. Let $\mathbf{B} = (\mathbf{r}, \pi) \in \desc(R, U, *)$, where $\mathbf{r} = (r^{*} \concat (-1), Q^{*}, \overrightarrow{(d,q, P)})$, $\pi$ extends $\id_{*,Q^{-}}$, $\pi(d^{*}, q^{*}) = (1, (0), \emptyset)$. Let $z = \rho^{-1}(\mathbf{B})$. Similarly to Case 1 of the proof of Lemma~\ref{lem:level_2_tree_cofinal_general}, $(Z,z)$ plays the role of the desired $(R',r')$.

Case 2: $\cf(R(r^{*})) = 2$.

Put $\mathbf{E}=(e, \mathbf{z}, \id_{P^{*}}) = \ucf^{-}(R(r^{*}))$. 
By Lemma~\ref{lem:level_2_desc_cofinal_in_next}, letting $f(\xi) = [\vec{\beta} \mapsto \wocode{\vec{\beta} \concat (j^{P^{*}}(g) (\comp{e}{\beta}_{\mathbf{z}}))  \oplus_R r^{*}}_{<^R}]_{\mu^{Q^{-}}}$ for $\xi = [g]_{\mu_{\mathbb{L}}} < u_2$, then $\sup f'' u_2 = j^{Q^{-}, Q^{*}} ( \llbracket r^{*} \rrbracket_R)$. Fix $\beta < u_2$ and we try to find $R', r'$ as in Case 1. Let $U$ be a $\Pi^1_2$-wellfounded level $\leq 2$ tree such that $\beta < \llbracket 2, (0) \rrbracket_U $, obtained by Lemma~\ref{lem:level_2_tree_cofinal_in_u_2}. Let $(Z, \rho, \theta)$ be as in Case 1. Let $\mathbf{B} = (\mathbf{r}, \pi) \in \desc(R, U, *)$, where $\mathbf{r} = (r^{*}\concat (-1,Q^{*}), Q^{*}, \overrightarrow{(d,q,P)})$, $\pi$ extends $\id_{*,Q^-}$, $\pi(d^{*}, q^{*}) = (2, ((0), \se{(0)}, ((0))), \tau)$, $\tau(0) = \mathbf{E}$. Let $z = \rho^{-1}(\mathbf{B})$. $(Z,z)$ plays the role of the desired $(R',r')$.
\end{proof}

\begin{mylemma}
  \label{lem:order_type_realizable}
  Suppose $R$ is a finite level-3 tree and $\vec{\gamma}=(\gamma_r)_{r \in \dom(R)}$ respects $R$. Then there is a $\Pi^1_3$-wellfounded level-3 tree $R'$ such that $R \subseteq R'$ and for any $r \in \dom(R)$,  $\gamma_r = \llbracket r \rrbracket_{R'}$.
\end{mylemma}
\begin{proof}
It suffices to produce a level-3 tree $R'$ and a map $\rho$ factoring $(R,R')$ such that for any $r \in \dom(R)$, $\gamma_r = \llbracket r \rrbracket_{R'}$. By Lemma~\ref{lem:level_3_amalgamation}, it suffices to show that for any $r^{*} \concat (a) \in \dom(R)$, letting $Q^{*} = R_{\tree}(r^{*} \concat (a))$, there is a $\Pi^1_3$-wellfounded level-3 tree $R'$ and $r' \in \dom(R')$ such that $\gamma_r = \llbracket r' \rrbracket_{R'}$, $R' [ (r')^{-}] = R[r^{*}]$, $R_{\tree}(r') = Q^{*}$.



Lemma~\ref{lem:level_2_tree_cofinal_general} gives us a $\Pi^1_3$-wellfounded level-3 tree $Y$ and $y \in \dom(Y)$ such that $\llbracket  y \rrbracket_Y \geq \gamma_{r^{*}\concat (a)}$, $Y[y^{-}] = R[r^{*}], Y_{\tree}(y) = Q^{*}$. 
Minimizing $\llbracket  y \rrbracket_Y$, we may further assume that for any $\Pi^1_2$-wellfounded $Y'$ and any $y'$ such that  $\llbracket  y' \rrbracket_{Y'}\geq \gamma_{r^{*}\concat (a)}$, $Y'[(y')^{-}] = R[r^{*}], Y'_{\tree}(y') = Q^{*}$,  we have $\llbracket  y' \rrbracket_{T'} \geq \llbracket  r^{*}\concat (a) \rrbracket_Y$. 
We claim that $\llbracket  t \rrbracket_Y = \gamma_{r^{*} \concat (a)}$. Suppose otherwise. Put $Y(y) = (Q^{*}  , (d^{*}, q^{*}, P^{*}))$.  

Case 1: $\cf(Y(y)) = 0$. 


Argue as in Case 1 in the proof of Lemma~\ref{lem:order_type_realizable_level_2} to obtain a contradiction.


Case 2: $\cf(Y(y)) = 1$. 

For $\beta< \omega_1$, put $f(\beta) =[\vec{\alpha} \mapsto  \wocode{  \vec{\alpha} \concat (\beta) \oplus_{\comp{2}{Q}^{*}} y \concat (-1)}_{<^Y} ]_{\mu^{Q^{*}}}$. 
So $\llbracket t \rrbracket_Y = \sup \set{f(\beta)}{\beta<\omega_1}$. For each limit ordinal $\beta < \omega_1$, we shall find a $\Pi^1_3$-wellfounded $Y'$ and a node $y' \in  \dom(Y')$ such that $\llbracket y' \rrbracket_{Y'} = f(\beta)$, contradicting to the minimization assumption. 
 Fix a limit ordinal $\beta<\omega_1$. Let $U$ be a $\Pi^1_2$-wellfounded level $\leq 2$ tree such that $ \llbracket 1, (0) \rrbracket_U =\beta$. Let $(Z, \rho)$ be a representation of $Y \otimes U$ and let $\theta : \rep(Z) \to \rep(Y)$ be the order preserving bijection. Let $\mathbf{B} = (\mathbf{y}, \pi) \in \desc(Y, U, *)$, 
 $\mathbf{y} = (y \concat (-1), X, \overrightarrow{(e,x,W)})$, $\pi$ extends $\id_{*,X}$, $\pi(e_{\lh(\vec{x})}, x_{\lh(\vec{x})}) = (1, (0), \emptyset)$.
Let $z = \rho^{-1}(\mathbf{B})$.  Similarly to Case 2 of the proof of Lemma~\ref{lem:order_type_realizable_level_2}, $(Z,z)$ constitutes a counterexample. 


Case 3:  $\cf(Y(y)) = 2$. 

Let $\mathbf{E} = (e,\mathbf{z}, \id_{P^{*}}) = \ucf^{-}(Y(y))$. 
For $\beta = [g]_{\mu^{\mathbb{L}}} < u_2$, put $f(\beta) =[\vec{\alpha} \mapsto  \wocode{  \vec{\alpha} \concat j^{P^{*}}(g)(\comp{e}{\alpha}_{\mathbf{z}}) \oplus_{\comp{2}{Q}^{*}} y \concat (-1)}_{<^Y} ]_{\mu^{Q^{*}}}$.  So $\llbracket t \rrbracket_Y = \sup \set{f(\beta)}{\beta<u_2}$. For each limit ordinal $\omega_1 < \beta < u_2$, we shall find a $\Pi^1_3$-wellfounded $Y'$ and a node $y' \in \dom(Y')$ such that $\llbracket y' \rrbracket_{Y'} = f(\beta)$. 
Let $U$ be a $\Pi^1_2$-wellfounded level $\leq 2$ tree such that $\llbracket 2, ((0)) \rrbracket_U = \beta$. Let $Z,\rho, \theta$ be as in Case 2. Let $\mathbf{B} = (\mathbf{y}, \pi)$, where $\mathbf{y} = (y \concat (-1), X, \overrightarrow{(e,x,W)})$, $\pi$ extends $\id_{*,X}$, $\pi(e_{\lh(\vec{x})}, x_{\lh(\vec{x})}) = (2, ( (0), \se{(0)}, ((0)) ), \tau)$, $\tau((0)) =\mathbf{E}$. Let $z = \rho^{-1}(\mathbf{B})$. Similarly to Case 3 of the proof of Lemma~\ref{lem:order_type_realizable_level_2}, $(Z,z)$ constitutes a counterexample. 
\end{proof}



\section{On $0^{3\#}$}
\label{sec:03}

Recall that  $\mathbb{L}[T_3] = \bigcup_{x \in \mathbb{R}}L[T_3,x]$,  $\mathbb{L}_{\bolddelta{3}}[T_3] = \bigcup_{x \in \mathbb{R}} L_{\bolddelta{3}}[T_3,x]$.  Recall Steel's computation of $L[T_{2n+1}]$:

\begin{mytheorem}[Steel  \cite{steel_hodlr_1995}]
  \label{thm:steel}
  Assuming $\boldpi{2n+1}$-determinacy. Assume $z$ is a real.
  \begin{enumerate}
  \item $\bolddelta{2n+1}$ is the least $<\delta_{2n,\infty}^z$-strong cardinal of $M_{2n,\infty}^{\#}(z)$, where $\delta_{2n,\infty}^z$ is the least Woodin cardinal of $M_{2n,\infty}^{\#}(z)$.
  \item $M_{2n,\infty}^{-}(z) = L_{\bolddelta{2n+1}}[T_{2n+1},z]$.
  \end{enumerate} 
\end{mytheorem}
By Theorem~\ref{thm:steel}, $\mathbb{L}_{\bolddelta{3}}[T_3] = V_{\bolddelta{3}} \cap \mathbb{L}[T_3]$.

Recall the temporary definition of $0^{3\#}$ in \cite{sharpII}. There is a club $C \subseteq \bolddelta{3}$ such that for a finite level-3 tree $R$, 
\begin{displaymath}
0^{3\#}(R) = \set{\varphi}{  (M_{2,\infty}^{-}; \vec{\gamma}) \models \varphi \text{ for any $\vec{\gamma} \in [C]^{R \uparrow}$}}.
\end{displaymath}

In fact, the complexity of $0^{3\#}(R)$ relies only on $\llbracket \emptyset \rrbracket_R$.

 \begin{mylemma}
   \label{lem:level-2_ultrapower_maps_definable}
If $Q$ is a finite level $\leq 2$ tree, then 
   $j^Q(M_{2,\infty}^{-}), j^Q \res M_{2,\infty}^{-}$ are definable over $M_{2,\infty}^{-}$, uniformly in $Q$. 
If $X$ is another finite level $\leq 2$ trees and $\pi$ factors $(X,Q)$,  then $\pi^{Q} \res j^X(M_{2,\infty}^{-})$ is definable over $M_{2,\infty}^{-}$, uniformly in $(\pi,X,Q)$. 
 \end{mylemma}
 \begin{proof}
By Theorem~\ref{thm:steel},  $j^Q(M_{2,\infty}^{-}) = L[j^Q(T_3)]$, and every $\Sigma^1_4$ subset of $\bolddelta{3}$ is definable over $M_{2,\infty}^{-}$. It suffices to show that $j^Q(T_3)$, $j^Q \res \bolddelta{3}$, $\pi^{Q} \res \bolddelta{3}$ are all $\Sigma^1_4$ in the codes. 

Let $G$ be a good universal $\Pi^1_3$ set and let $\varphi : G \to \bolddelta{3}$ be a regular $\Pi^1_3$ norm. Then $x \in G \wedge y \in G \wedge j^Q(\varphi(x)) = \varphi(y)$ iff there exists $z \in \mathbb{R}$ such that $z=M_1^{\#}(x,y)$ and $\admistwo{z} \models j^Q(\varphi(x)) = \varphi(y)$. Here, the statement $j^Q(\varphi(x)) = \varphi(y)$ is $\Delta_1$ over $\admistwo{z}$ from $\se{T_2,z}$ by Corollary~\ref{coro:Delta13_pwo_computable_in_LT2} and Lemma~\ref{lem:level_2_embedding_bounded_by_delta13}. Similarly, using Lemma~\ref{lem:jQ_move_factor_maps_another},  $\pi^{Q} \res \kappa_3^x$ and $j^Q(T_3) \res \kappa_3^x$ are $\Delta_1$-definable  over $\admistwo{x}$ from $\se{T_2,x}$, uniformly in $x$. So $\pi^{Q} \res \bolddelta{3}$ and $j^Q(T_3)$ are $\Sigma^1_4$ using similar arguments. 
 \end{proof}

Based on Theorem~\ref{thm:steel} and $\Sigma^1_4$-absoluteness of iterates of $M_2^{\#}$, a $\Sigma^1_4$ set $ A \subseteq \bolddelta{3}$ has the following alternative definition: $\alpha \in A$ iff  $M_{2,\infty}^{-}$ satisfies that
\begin{quote}
  for any $\xi > \alpha$ cardinal cutpoint, in the $\coll(\omega,\xi)$-generic extension, $\pi_{K | \xi , \infty}(\alpha) \in A$.
\end{quote}

We introduce the following informal symbols arising from the proof of Lemma~\ref{lem:level-2_ultrapower_maps_definable} that will occur in $\mathcal{L}$-formulas or $\mathcal{L}^R$-formulas for a level-3 tree $R$: 
\begin{enumerate}
\item If $Q$ is a finite level $\leq 2$ tree, $\underline{j^Q}$ is the informal symbol so that $\underline{j^Q}( a ) = b$ iff for any $\xi $ cardinal cutpoint such that $\se{a,b} \in K|\xi$, the $\coll(\omega,\xi)$-generic extension satisfies $j^Q(\pi_{K|\xi, \infty} ( a) )= \pi_{K|\xi, \infty}(b) $.
\item If $\pi$ factors finite level $\leq 2$ trees $(X, T)$, $\underline{\pi^T}$ is the informal symbol so that $\underline{\pi^T}(a) = b$ iff for any $\xi $ cardinal cutpoint such that $\se{a,b} \in K|\xi$, the $\coll(\omega,\xi)$-generic extension satisfies $\pi^T(\pi_{K|\xi, \infty} ( a) )=  \pi_{K|\xi, \infty}(b)  $.
\item If $Q$ is a level $\leq 2$ subtree of $Q'$, $Q'$ is finite, then $\underline{j^{Q,Q'}} = \underline{(\id_Q)^{Q,Q'}}$, where $\id_Q$ factors $(Q,Q')$, $\id_Q(d,q) = (d,q)$. 
\item Corresponding to items 1-3, $  \underline{j^Q_{\sup}}, \underline{\pi^{T}_{\sup}}, \underline{j^{Q,Q'}_{\sup}}$ stand for functions on ordinals that send $\alpha$ to $\sup (\underline{j^Q})''\alpha$, $\sup (\underline{\pi^T})'' \alpha$, $\sup (\underline{j^{Q,Q'}})'' \alpha$ respectively.
\item $\underline{S_3}$ is the informal symbol such that  $(\emptyset,\emptyset) \in \underline{S_3}$ and $ ( (R_i)_{ i\leq n}, (\alpha_i)_{i \leq n}  ) \in \underline{S_3}$ iff $\vec{R}$ is a finite regular level-3 tower and letting 
$r_i \in \dom(R_{i+1}) \setminus \dom(R_i)$, then $r_k = (r_l)^{-} \to \alpha_k < \underline{j^{(R_n)_{\tree}(r_k), (R_n)_{\tree}(r_l)}} (\alpha_l)$.
\item For $1 \leq n \leq \omega$, $\underline{u_n}$ is the symbol so that for any $\xi > \underline{u_n}$ cardinal cutpoint, the $\coll(\omega,\xi)$-generic extension satisfies $\pi_{K|\xi, \infty} (\underline{u_n}) = u_n$.
\item  Suppose $T$ is a finite level $\leq 2$ tree. If $\mathbf{D} \in \desc(T,U)$, $\wocode{ \mathbf{D}}_{\prec^{T,U}} = n$, then $\underline{\seed^{T,U}_{\mathbf{D}}}=\underline{u_{n+1}}$. If $(1,t) \in \dom(T)$, then $\underline{\seed^T_{(1,t)}} = \underline{\seed^{T,\emptyset}_{(1,t,\emptyset)}}$. If $(2,t) \in \dom(T)$, and $\comp{2}{T}[t] = (S,\vec{s})$,  then $\underline{\seed^T_{(2,t)}} = \underline{\seed^{T,S}_{(2, (t, S,\vec{s}), \id_S)}}$. $\underline{\seed^T} = (\underline{\seed^T_{(d,t)}})_{(d,t) \in \dom(T)}$. 
\item If $k$ is a definable class function and $W$ is a definable class, then $k(W) = \bigcup \set{k(W \cap V_{\alpha})}{\alpha \in \ord}$. 
\item If  $X , T, T'$ are finite level $\leq 2$ trees,  $T$ is a subtree of $T'$, $a \in \underline{j^X}(V)$, $d \in \se{1,2}$,  then 
  \begin{enumerate}
  \item $\underline{B^T_{X, a}} = \{\underline{\pi^{T \otimes Q}} (a) : Q$ finite level $\leq 2$ tree,  $\pi $ factors $(X,T \otimes Q)\}$;
  \item $\underline{H^T_{X, a}}$ is the transitive collapse of the Skolem hull of $\underline{B^T_{X, a}} \cup \ran(\underline{j^T})$ in $\underline{j^T}(V)$ and  $\underline{\phi^T_{X, a}}: \underline{H^T_{X, a}} \to \underline{j^T}(V)$ is the associated elementary embedding;
  \item $\underline{j^T_{X, a}} = (\underline{\phi^T_{X, a}})^{-1} \circ \underline{j^T}$;
  \item $\underline{j^{T,T'}_{X, a}} = (\underline{\phi^{T'}_{X, a}})^{-1} \circ \underline{j^{T,T'}} \circ \underline{\phi^T_{X, a}}$;
  \item $\underline{B^T_{1,a}} = \underline{B^T_{Q^0,a} } \cup \underline{B^T_{Q^1,a}}$, $\underline{B^T_{2, a}} =  \underline{B^T_{Q^0,a} } \cup \underline{B^T_{Q^{20},a}} \cup  \underline{B^T_{Q^{21},a}}$;
  \item  $\underline{H^T_{d,a}}$ is the transitive collapse of the Skolem hull of $\underline{B^T_{d,a}} \cup \ran( \underline{j^T} )$ in $\underline{j^T} (V) $ and $\underline{\phi^T_{d, a}}: \underline{H^T_{d, a}} \to \underline{j^T}(V)$ is the associated elementary embedding;
  \item $\underline{j^T_{d, a}} = (\underline{\phi^T_{d, a}})^{-1} \circ \underline{j^T}$;
  \item $\underline{j^{T,T'}_{d, a}} = (\underline{\phi^{T'}_{d, a}})^{-1} \circ \underline{j^{T,T'}} \circ \underline{\phi^T_{d, a}}$.
 \end{enumerate}
\item Suppose $R$ is a level-3 tree. 
  \begin{enumerate}
  \item 
If $\mathbf{r} = (r,Q, \overrightarrow{(d,q,P)}) \in \exdesc(R)$, $\underline{c_{\mathbf{r}}}$ is the informal $\mathcal{L}^{R}$-symbol whose interpretation is
 \begin{displaymath}
   \underline{c_{\mathbf{r}}} =
   \begin{cases}
\underline{j_{\sup}^{R_{\tree}(r^{-}),Q}} (\underline{c_{r^{-}}}) & \text{if }  \mathbf{r} \in \desc(R) \text{ of continuous type} ,\\
\underline{c_r} & \text{if } \mathbf{r} \in \desc(R) \text{ of discontinuous type} ,\\
\underline{j^{R_{\tree}(r), Q}}(\underline{c_r}) & \text{if } \mathbf{r} \notin \desc(R).
   \end{cases}
 \end{displaymath}
\item If  $T,U$ are finite level $\leq 2$ trees and $\mathbf{B} = (\mathbf{r}, \pi) \in \desc(R, T, U)$, $\mathbf{r} \neq \emptyset$, then $  \underline{c_{\mathbf{B}}^T}$ is the informal $\mathcal{L}^{R}$-symbol which stands for $\underline{\pi^{T,U}} ( \underline{c_{\mathbf{r}}})$.
\item If $\mathbf{A} = (\mathbf{r}, \pi, T) \in \exexdesc(R)$, $\mathbf{r} \neq \emptyset$, then $\underline{c_{\mathbf{A}}}$ is the informal $\mathcal{L}^R$-symbol which stands for $\underline{\pi^T}(\underline{c_{\mathbf{r}}})$.
  \end{enumerate}
\end{enumerate}

Put $\mathcal{L}^{\underline{x}} = \se{\underline{\in}, \underline{x}}$;  for a level-3 tree $R$, put $\mathcal{L}^{R,\underline{x}} = \mathcal{L}^R \cup \se{\underline{x}}$, where $\underline{x}$ is a constant symbol. All of the above informal symbols work in $\mathcal{L}^{\underline{x}}$ or $\mathcal{L}^{R,\underline{x}}$, which are intended to be interpreted in $M_{2,\infty}^{-}(x)$ for $x \in \mathbb{R}$.  In particular, for $x \in \mathbb{R}$, 
\begin{displaymath}
  (L[ \underline{S_3}])^{M_{2,\infty}^{-} (x)} = M_{2,\infty}^{-}.
\end{displaymath}

\begin{mylemma}
  \label{lem:3sharp_R_complexity}
Assume $\boldpi{3}$-determinacy. 
  Suppose $R,Y$ are finite level-3 trees, $T$ is a finite level $\leq 2$ tree,  $\rho$ factors $(R,Y,T)$. Then 
   \begin{displaymath}
      \gcode{\varphi(\underline{c_{r_1}},\ldots,\underline{c_{r_n}})} \in 0^{3\#}(R)
   \end{displaymath}
iff
\begin{displaymath}
        \gcode{\underline{j^T}(V) \models  \varphi (\underline{c_{\rho(r_1)}^T},\ldots,\underline{c_{\rho(r_n)}^T})}\in 0^{3\#}(Y).
\end{displaymath}
\end{mylemma}
 \begin{proof}
Put $\rho(r) = (\mathbf{y}_r,\pi_r)$. Put $R_{\tree}(r)=Q_r$. Suppose $\varphi(v_1,\ldots,v_n)$ is an $\mathcal{L}$-formula, $r_1,\ldots,r_n \in \dom(R)$, and 
   \begin{displaymath}
      \gcode{\varphi(\underline{c_{r_1}},\ldots,\underline{c_{r_n}})} \in 0^{3\#} (R). 
   \end{displaymath}
Let $C$ be  a firm set of potential level-3 indiscernibles. 
Suppose $F \in C^{Y\uparrow}$. By Lemma~\ref{lem:R_respect},  $F^{T}_{\rho}$ is a function from $\omega_1^{T \uparrow}$ to $ [C]^{R \uparrow}$. Recall that $F^{T}_{\rho}(\vec{\delta}) = (F^{T}_{\rho(r)} (\vec{\delta}) ) _{r \in \dom(R)}$.  Hence for any $\vec{\delta} \in \omega_1^{T \uparrow}$, 
\begin{displaymath}
  M_{2,\infty}^{-} \models \varphi (  F _{\rho(r_1)}^{T}(\vec{\delta}) ,\ldots,  F _{\rho(r_n)}^{T}(\vec{\delta})).
\end{displaymath}
For each $r \in \dom(R)$, by definition of ${\pi_r}^{T,Q_r}$,  ${\pi_r}^{T,Q_r} ( [F]^{Y}_{\mathbf{y}_r}) = [F^{T}_{\rho(r)}]_{\mu^T}$. 
Hence by \Los{},
\begin{displaymath}
  j^T (M_{2,\infty}^{-}) \models \varphi  (  {\pi_{r_1}}^{T,Q_{r_1}}( [F]^Y_{\mathbf{y}_{r_1}}) , \ldots, {\pi_{r_n}}^{T,Q_{r_n}}( [F]^Y_{\mathbf{y}_{r_n}}) ).
\end{displaymath}
Finally, by Lemma~\ref{lem:gamma_r_continuous_type}, for $\mathbf{y} = (y,X) \in \desc(Y)$, if $y$ is of discontinuous type then $[F]^Y_{\mathbf{y}} = [F]^Y_y$; if $y$ is of continuous type then $[F]^Y_{\mathbf{y}} = j_{\sup}^{Y_{\tree}(y^{-}),X}([F]^Y_{y^{-}})$. 
Hence, 
\begin{displaymath}
        \gcode{\underline{j^T}(V) \models  \varphi (\underline{c_{\rho(r_1)}^T},\ldots,\underline{c_{\rho(r_n)}^T})}\in 0^{3\#}(Y).
\end{displaymath}
 \end{proof}

As a corollary to Lemma~\ref{lem:3sharp_R_complexity}, Lemma~\ref{lem:level-2_ultrapower_maps_definable} and Theorem~\ref{thm:factor_tower_order_type_equivalent}, we obtain:
 \begin{mycorollary}
   \label{coro:3sharp_R_invariant}
 Assume $\boldpi{3}$-determinacy.   Suppose $R$ and $Y$ are finite level-3 trees and $\llbracket \emptyset \rrbracket_R = \llbracket \emptyset \rrbracket_Y$. Then $0^{3\#}(R) \equiv_m 0^{3\#}(Y)$.
 \end{mycorollary}

\subsection{Syntactical properties of $0^{3\#}$}
 \label{sec:syntactical_3sharp}

Suppose $\mathcal{M},\mathcal{N}$ are countable $\Pi^1_3$-iterable mice. 
 A map $\pi: \mathcal{M} \to \mathcal{N}$ is \emph{essentially an iteration map} iff there are $\mathcal{P}$ and iteration maps $\psi_{\mathcal{M}}: \mathcal{M} \to \mathcal{P}$, $\psi_{\mathcal{N}}: \mathcal{N} \to \mathcal{P}$  such that $\psi_{\mathcal{M}} = \psi_{\mathcal{N}} \circ \pi$. For $\alpha \in \mathcal{M}$, $\beta \in \mathcal{N}$, say that $(\mathcal{M},\alpha) <_{DJ} (\mathcal{N}, \beta)$ iff either $\mathcal{M}<_{DJ} \mathcal{N}$ or there exist $\mathcal{P}$ and iteration maps $\psi_{\mathcal{M}} : \mathcal{M} \to \mathcal{P}$, $\psi_{\mathcal{N}}: \mathcal{N} \to \mathcal{P}$ such that $\psi_{\mathcal{M}}(\alpha) < \psi_{\mathcal{N}}(\beta)$.

\begin{mydefinition}[Level-3 EM blueprint]
\label{def:pre_level_3_EM_blueprint}
  A \emph{pre-level-3 EM blueprint} is a function $\Gamma$ sending any finite level-3 tree $Y$ to   a complete consistent $\mathcal{L}^Y$-theory $\Gamma(Y)$ which contains all of the following axioms:
  \begin{enumerate}
    \item\label{item:EM_ZFC} $\ZFC+$ there is no inner model with two Woodin cardinals $+ V=K +$ there is no strong cardinal $+ V$ is closed under the $M_1^{\#}$-operator. 
 \item \label{item:EM_XTQ_factor} 
   \label{item:EM_commutativity} 
Suppose $X,T, Q,Z$ are finite level $\leq 2$ trees, $\pi$ factors $(X,T)$, $\psi$ factors $(T,Z)$. 
    \begin{enumerate}
    \item  $\underline{j^T}: V \to \underline{j^T}(V)$ is $\mathcal{L}$-elementary. $\underline{j^{Q^0}} $ is the identity map on $V$.
    \item  $\underline{\pi^T}:  \underline{j^X}(V) \to \underline{j^T}(V)$ is $\mathcal{L}$-elementary. $\underline{j^{Q^0,T}} = \underline{j^T}$. $\underline{j^{T,T}}$ is the identity map on $\underline{j^T}(V)$.
     \item $\underline{(\psi \circ \pi)^Z}  = \underline{\psi^Z} \circ \underline{\pi^T} $.
     \item  $\underline{j^T} \circ \underline{j^Q} = \underline{j^{T \otimes Q}}$. 
     \item $ \underline{j^Q} (\underline{\pi^T}) = \underline{(Q \otimes \pi)^{Q\otimes T}}$.
     \item $\underline{\pi^T} \res \underline{j^{X \otimes Q}}(V) = \underline{ (\pi \otimes Q)^{T \otimes Q}}$. 
\end{enumerate}
\item  \label{item:level-2-embedding_invariance} 
If $\xi$ is a cardinal and strong cutpoint, then $V^{\coll(\omega,\xi)}$ satisfies the following: 
If ${U}$ is a $\Pi^1_2$-wellfounded level $\leq 2$ tree, then $K|\xi$ and $(\underline{j^{{U}}})^K(K|\xi) $ are countable $\Pi^1_3$-iterable mice and $(\underline{j^{{U}}})^K\res (K|\xi) $ is essentially an iteration map from $K|\xi$ to $(\underline{j^{{U}}})^K(K|\xi)$. Here $(\underline{j^{{U}}})^K$ stands for the direct limit map of $(\underline{j^{Z,Z'}})^K$ for $Z,Z'$ finite subtrees of $U$, $Z$ a finite subtree of $Z'$.
\item \label{item:EM_c_t_are_ordinals} For any $y\in \dom(Y)$, ``$\underline{c_y} \in \ord$'' is an axiom. 
  \item\label{item:EM_order} If $\mathbf{y} \prec^Y \mathbf{y}'$, then ``$\underline{c_{\mathbf{y}}} < \underline{c_{\mathbf{y}'}}$'' is an axiom; if $\mathbf{y} \sim^Y \mathbf{y}'$, then ``$\underline{c_{\mathbf{y}}} = \underline{c_{\mathbf{y}'}}$'' is an axiom.
  \end{enumerate}
A level-3 EM blueprint is a pre-level-3 EM blueprint satisfying the \emph{coherency property}:  if $R,Y,T$ are finite,  $\rho$ factors $(R,Y,T)$, then for each $\mathcal{L}$-formula $\varphi(v_1,\ldots,v_n)$, for each  $r_1,\ldots,r_n \in \dom(R)$, 
   \begin{displaymath}
      \gcode{\varphi(\underline{c_{r_1}},\ldots,\underline{c_{r_n}})} \in \Gamma(R)
   \end{displaymath}
iff
\begin{displaymath}
 \gcode{\underline{j^T}(V) \models  \varphi (\underline{c_{\rho(r_1)}^T},\ldots,\underline{c_{\rho(r_n)}^T})}\in \Gamma(Y).
\end{displaymath}
\end{mydefinition}

 In particular, if $\Gamma$ is a level-3 EM blueprint, $\rho_0$ factors $(R,Y)$, 
then $\id_{Y,*} \circ  \rho_0$ factors $(R,Y,Q^0)$,  so by coherency, $     \gcode{\varphi(\underline{c_{r_1}},\ldots)} \in \Gamma(R)
$ iff $      \gcode{\varphi(\underline{c_{\rho_0(r_1)}},\ldots)} \in \Gamma(Y)
$. This degenerates to the usual indiscernability of the (level-1) EM blueprint.

\begin{mylemma}
  \label{lem:pre_level_2_correct}
  Assume $\boldpi{3}$-determinacy. Then $0^{3\#}$ is a level-3 EM blueprint.
\end{mylemma}
\begin{proof}
We verify Axioms~\ref{item:EM_ZFC}-\ref{item:EM_order} in Definition~\ref{def:pre_level_3_EM_blueprint}.
  Axiom~\ref{item:EM_ZFC}  follows from Theorem~\ref{thm:steel}. Axiom~\ref{item:EM_commutativity} follows from Lemma~\ref{lem:level_2_ultrapower_iteration_reduce}. 
Axioms~\ref{item:EM_c_t_are_ordinals}-\ref{item:EM_order} follow from Lemma~\ref{lem:gamma_r_order}.

Axiom~\ref{item:level-2-embedding_invariance}  is shown as follows. Let $\xi$ be a cardinal strong cutpoint of $M_{2,\infty}^{-}$. Let $\mathcal{P} \in \mathcal{F}_2$ and $\eta \in \mathcal{P}$ such that $\pi_{\mathcal{P}, \infty}(\eta) = \xi$. Let $g$ be $\coll(\omega, \eta)$-generic over $\mathcal{P}$. 
 Suppose $T$ is a $\Pi^1_2$-wellfounded tree in $\mathcal{P}[g]$. The direct limit of $j^T$ is wellfounded by \cite[Proposition 4.1]{sharpII}. 
We need to show that in $\mathcal{P}[g]$, $(\underline{j^{T}})^{\mathcal{P}}(\mathcal{P}|\eta) : \mathcal{P}|\eta \to (\underline{j^{T}})^{\mathcal{P}}(\mathcal{P}|\eta)$ is essentially an iteration map, where $(\underline{j^T})^{\mathcal{P}}$ is the direct limit map of $(\underline{j^{T'}})^{\mathcal{P}}$ for finite subtrees $T'$ of $T$. Since $\mathcal{P}[g]$ is $\Sigma^1_4$-correct, we need to show the same fact in $V$. 

Note that $M_{2,\infty}^{-}$ is definable over $M_{2,\infty}^{-}(g)$. In fact, $M_{2,\infty}^{-} = (L[ \underline{S_3}])^{M_{2,\infty}^{-}(g)}$. Let $\mathcal{Q} \in \mathcal{F}_{2, g}$ and $\nu$ so that $\pi_{Q, \infty}(\nu) = \xi$. The maps from $\mathcal{P}$ to $\pi_{\mathcal{Q},\infty}^{-1}(M_{2,\infty}^{-}| \xi)$ and from $\pi_{\mathcal{Q},\infty}^{-1}(M_{2,\infty}^{-}| \xi)$ to $L_{\xi}[S_3]$ plus Dodd-Jensen implies that $\mathcal{P}|\eta \sim_{DJ} \pi_{\mathcal{Q},\infty}^{-1}(M_{2,\infty}^{-}| \xi)$. By $\Sigma^1_4$-correctness of set-generic extensions of $\mathcal{Q}$, $\mathcal{Q}$ thinks that  ``$\mathcal{P}|\eta \sim_{DJ} L_{\nu}[S_3]$''. By elementarity, 
$(\underline{j^T}(V))^{\mathcal{Q}}$ thinks ``$\mathcal{P}|\eta \sim_{DJ} \underline{j^T}(L_{\nu}[ \underline{S_3} ])$''. We claim that $ (\underline{j^T}(V))^{\mathcal{Q}}$ is also $\Sigma^1_3$-correct in set-generic extensions. To see this, it suffices to show $p[(\underline{j^T}(\underline{S_3}))^{\mathcal{Q}}] \subseteq p[S_3]$. We know that $(\underline{j^T}(\underline{S_3}))^{\mathcal{Q}}$ embeds into $j^T(S_3)$, so $p[(\underline{j^T}(\underline{S_3}))^{\mathcal{Q}}] \subseteq p[j^T(S_3)]$. But $x \in p[j^T(S_3)]$ implies $x \in p[S_3]$ by absoluteness of wellfoundedness and elementarity of $j^T$ acting on $L[S_3,x]$. Hence, in reality we have   $\mathcal{P}|\eta \sim_{DJ} (\underline{j^T}(L_{\nu}[ \underline{S_3} ]))^{\mathcal{Q}}$. 
But $(\underline{j^{T}})^{\mathcal{P}}(\mathcal{P}|\eta)$ embeds into $(\underline{j^T}(L_{\nu}[ \underline{S_3} ]))^{\mathcal{Q}}$, implying that  $(\underline{j^{T}})^{\mathcal{P}}(\mathcal{P}|\eta) \leq _{DJ} \mathcal{P}|\eta$. 

Of course, $\mathcal{P}|\eta \leq_{DJ} (\underline{j^{T}})^{\mathcal{P}}(\mathcal{P}|\eta)$.  So $ \mathcal{P} \sim_{DJ} (\underline{j^{T}})^{\mathcal{P}}(\mathcal{P}|\eta)$.  A similar argument shows that for any $\alpha \in \mathcal{P}$, $(\mathcal{P}, \alpha) \sim_{DJ} ((\underline{j^{T}})^{\mathcal{P}}(\mathcal{P}|\eta), (\underline{j^{T}})^{\mathcal{P}}(\alpha))$. This finishes verifying Axiom~\ref{item:level-2-embedding_invariance} of Definition~\ref{def:pre_level_3_EM_blueprint}. 

Finally, the coherency property of $0^{3\#}$ is a consequence of Lemma~\ref{lem:3sharp_R_complexity}.
\end{proof}

We say that the \emph{upward closure} of $A \subseteq (\omega^{<\omega})^{<\omega}$ is
\begin{displaymath}
  \set{ r \in (\omega^{<\omega})^{<\omega}}{ \exists a \in A (r \subseteq a)}.
\end{displaymath}
The upward closure does not apply to subcoordinates of $a \in A$. For instance, $b \subsetneq a(\lh(a)-1)$ does not imply that $a^{-} \concat (b)$ is in the upward closure of $A$.
For a level-3 tree $R$ and nodes $s_1,\ldots,s_n,s_1',\ldots,s_n'$ in $\dom(R)$,
\begin{center}
  $\vec{s'}$ is an \emph{$R$-shift} of $\vec{s}$
\end{center}
iff there are a level-3 tree $S$ and maps $\rho,\rho'$ factoring $(S,R)$ such that $\ran(\rho)$ is the upward closure of $\vec{s}$, $\ran(\rho')$ is the upward closure of $\vec{s'}$, and $\rho^{-1}(s_i) = (\rho')^{-1}(s_i')$ for any $i$.

\begin{mylemma}[Level-3 indiscernability]
  \label{lem:coherent_EM_blueprint_is_indiscernible_to_shifts}
  Suppose $\Gamma$ is a level-3 EM blueprint. Suppose $R$ is a level-3 tree and $s_1,\ldots,s_n,s_1',\ldots,s_n'$ are nodes in $\dom(R)$. Suppose that $\vec{s'}$ is a {shift} of $\vec{s}$ with respect to $R$. Then for each formula $\varphi$, $\Gamma(R)$ contains the formula
  \begin{displaymath}
    \varphi ( \underline{c_{s_1}},\ldots,\underline{c_{s_n}}) \eqiv \varphi (\underline{c_{s_1'}},\ldots,\underline{c_{s_n'}}).
  \end{displaymath}
\end{mylemma}
\begin{proof}
  Let $S$ be a level-3 tree and $\rho ,\rho'$ both factor $(S,R)$  such that $\ran(\rho)$ is upward closure of $\se{s_1,\ldots,s_n}$, $\ran(\rho')$ is the upward closure of $\se{s_1', \ldots,s_n'}$.  Let $\rho^{-1}(s_i) = t_i = (\rho')^{-1}(s_i')$. Applying coherency of $\Gamma$ to $\rho,\rho'$,
  \begin{align*}
\gcode{ \varphi (\underline{c_{s_1}},\ldots,\underline{c_{s_n}}) }  \in \Gamma(R) & \eqiv     
\gcode{ \varphi (\underline{c_{t_1}},\ldots,\underline{c_{t_n}}) }  \in \Gamma(S) \\
& \eqiv \gcode{ \varphi (\underline{c_{s_1'}},\ldots,\underline{c_{s_n'}}) }  \in \Gamma(R) .
  \end{align*}
\end{proof} 

Of course, there is extra information in the coherency property beyond Lemma~\ref{lem:coherent_EM_blueprint_is_indiscernible_to_shifts}.

As with the usual treatment of $0^{\#}$, a level-3 EM blueprint $\Gamma$ admits an $\mathcal{L}$-Skolemized conservative extension. That means, since $\ZFC+ V = K$ is a part of the axioms, so is ``there is a $\Sigma_1^{\mathcal{L}}$-definable wellordering of the universe''. Thus, to each $\mathcal{L}$-formula $\varphi(v,w_1,\ldots,w_n)$ we may attach a definable $\mathcal{L}$-Skolem term $\tau_{\varphi}(w_1,\ldots,w_n)$ so that the formula $\forall w_1 \ldots w_n ~(\exists v~ \varphi(v,w_1,\ldots,w_n )\to \varphi(\tau_{\varphi}(w_1,\ldots,w_n), w_1,\ldots,w_n))$ belongs to $\Gamma(R)$, for any $R$.

If $Y$ is an infinite level-3 tree, put
\begin{displaymath}
\Gamma(Y) = \bigcup \{ \Gamma (R) : R \text{ is a finite level-3 subtree of } Y \}.
\end{displaymath}
By coherency, $\Gamma(R) \subseteq \Gamma(R')$ whenever $R\subseteq R'$ are finite. Hence by compactness, $\Gamma(Y)$ is a complete consistent $\mathcal{L}^Y$-theory. The usual argument of EM models with order indiscernibles carries over to obtain a unique up to isomorphism  $\mathcal{L}^Y$-structure 
\begin{displaymath}
 \mathcal{M}_{\Gamma,Y} = (M; \underline{\in}^M,\underline{c_t}^M: {t \in \dom(Y)}).
\end{displaymath}
such that $\mathcal{M}_{\Gamma,Y}$ is $\mathcal{L}$-Skolem generated by $\set{\underline{c_t}^M}{t \in \dom(Y)}$, and 
\begin{displaymath}
 \mathcal{M}_{\Gamma,Y} \models \Gamma(Y).
\end{displaymath}
 $\mathcal{M}_{\Gamma,Y}$ is called the \emph{EM model} associated to $\Gamma$ and $Y$.  When $\underline{\in}^{\mathcal{M}_{\Gamma,Y}}$ is wellfounded,  $\mathcal{M}_{\Gamma,Y}$ is identified with its transitive collapse. Since $\mathcal{M}_{\Gamma,Y}$ is a model of $V=K$, the extender sequence on $K^{\mathcal{M}_{\Gamma,Y}}$ is definable over $\mathcal{M}_{\Gamma,Y}$, this allows us to sometimes treat $\mathcal{M}_{\Gamma,Y}$ as a structure in the language of premice.

If $\mathcal{L}^{*}$ is a first-order  language expanding $\mathcal{L}$,  $\mathcal{N}$ is an $\mathcal{L}^{*}$-structure  satisfying axioms~\ref{item:EM_ZFC}-\ref{item:level-2-embedding_invariance} in Definition~\ref{def:pre_level_3_EM_blueprint}, we make the following notations:
\begin{enumerate}
\item If $T$ is a finite level $\leq 2$ tree, then $j^T_{\mathcal{N}} = (\underline{j^T})^{\mathcal{N}}$,  $\mathcal{N}^T = (\underline{j^T}(V))^{\mathcal{N}}$ is an $\mathcal{L}^{*}$-structure so that $j^T_{\mathcal{N}}: \mathcal{N} \to \mathcal{N}^T$ is $\mathcal{L}^{*}$-elementary. 
\item If $\pi$ factors finite level $\leq 2 $ trees $(X,T)$, then $\pi^T_{\mathcal{N}} = (\underline{\pi^T})^{\mathcal{N}}$. If $T,T'$ are finite level $\leq 2$ trees, $T$ is a subtree of $T'$, then $j^{T,T'}_{\mathcal{N}} = (\underline{j^{T,T'}})^{\mathcal{N}}$. 
\item If $T$ is a level $\leq 2$ tree, then $\mathcal{N}^T$ is the direct limit of $(\mathcal{N}^{T'},j^{T',T''}_{\mathcal{N}}: T',T''$ finite subtrees of $T$, $T'$ a finite subtree of $T'')$ and $j^T_{\mathcal{N}}: \mathcal{N} \to \mathcal{N}^{T}$ is the direct limit map; if $T'$ is a finite subtree of $T$, then $j^{T',T}_{\mathcal{N}}: \mathcal{N}^{T'} \to \mathcal{N}^T$ is the tail of the direct limit map.
  The wellfounded part of $\mathcal{N}^T$ is always assumed to be transitive. 
\item If $\pi$ factors level $\leq 2$ trees $(X,T)$, then $\pi^T_{\mathcal{N}} : \mathcal{N}^X \to \mathcal{N}^T$ is the factor map between direct limits.
\item If $X$ is a finite level $\leq 2$ tree, $a \in \mathcal{N}^X$, $d \in \se{1,2}$
  \begin{enumerate}
  \item if $T,T'$ are finite level $\leq 2$ trees, $T$ is a subtree of $T'$, then $\mathcal{N}^T_{X, a} = (\underline{H^T_{X,a}})^{\mathcal{N}}$, $j^T_{X,a, \mathcal{N}} = (\underline{j^T_{X,a}})^{\mathcal{N}}$, $\phi^T_{X,a,\mathcal{N}} = (\underline{\phi^T_{X,a}})^{\mathcal{N}}$, $j^{T,T'}_{X,a,\mathcal{N}} = (\underline{j^{T,T'}_{X,a}})^{\mathcal{N}}$,  $\mathcal{N}^T_{d, a} = (\underline{H^T_{d,a}})^{\mathcal{N}}$, $j^T_{d,a, \mathcal{N}} = (\underline{j^T_{d,a}})^{\mathcal{N}}$, $\phi^T_{d,a,\mathcal{N}} = (\underline{\phi^T_{d,a}})^{\mathcal{N}}$, $j^{T,T'}_{d,a,\mathcal{N}} = (\underline{j^{T,T'}_{d,a}})^{\mathcal{N}}$;
  \item if $T$ is a level $\leq 2$ tree, then $\mathcal{N}^T_{X,a}$ is the natural direct limit, $j^T_{X,a,\mathcal{N}} : \mathcal{N} \to \mathcal{N}^T_{X,a}$ is the direct limit map, $\phi^T_{X,a,\mathcal{N}} : \mathcal{N}^T_{X,a} \to \mathcal{N}^T$ is the natural factoring map between direct limits; if $T'$ is a finite subtree of $T$, then $j^{T',T}_{X,a,\mathcal{N}} : \mathcal{N}^{T'}_{X,a} \to \mathcal{N}^T_{X,a}$ is the tail of the direct limit map; similarly define $\mathcal{N}^T_{d,a}$, $j^T_{d,a,\mathcal{N}}$, $\phi^T_{d,a,\mathcal{N}}$, $j^{T',T}_{d,a,\mathcal{N}}$. 
  \end{enumerate}
\end{enumerate}
If $\Gamma$ is a level-3 EM blueprint and $R$ is $\Pi^1_3$-wellfounded, we make further notations: 
\begin{enumerate}
\item If $T$ is a level $\leq 2$ tree, then $\mathcal{M}_{\Gamma,Y}^T = (\mathcal{M}_{\Gamma,Y})^T$, $j^T_{\Gamma,Y} = j^T_{\mathcal{M}_{\Gamma,Y}}$. 
\item If $T$ is a finite subtree of $T'$, then $j^{T,T'}_{\Gamma,Y} = j^{T,T'}_{\mathcal{M}_{\Gamma,Y}}$.
\item If $\pi$ factors $(X,T)$, then $\pi^T_{\Gamma,Y} = \pi^T_{\mathcal{M}_{\Gamma,Y}}$.
\item If $T$ is  a finite subtree of $T'$, $y \in \dom(Y)$, $X = Y_{\tree}(y)$, $d \in \se{1,2}$, then $c_{\Gamma,Y,y} = (\underline{c_y})^{\mathcal{M}_{\Gamma,Y}}$,  $\mathcal{M}^T_{\Gamma,Y,y} = (\mathcal{M}_{\Gamma,Y})^T_{X, c_{\Gamma,Y,y}}$, $j^{T}_{\Gamma,Y,y} = {j^T_{X,c_{\Gamma,Y,y}, \mathcal{M}_{\Gamma,Y} }}$, 
 $\phi^{T}_{\Gamma,Y,y} = {\phi^T_{X,c_{\Gamma,Y,y}, \mathcal{M}_{\Gamma,Y} }}$, $j^{T,T'}_{\Gamma,Y,y} = {j^{T,T'}_{X,c_{\Gamma,Y,y}, \mathcal{M}_{\Gamma,Y} }}$, $\mathcal{M}^T_{\Gamma,R^d,*} = (\mathcal{M}_{\Gamma,R^d})^T_{d, c_{\Gamma,R^d,((0))}}$, $j^{T}_{\Gamma,R^d,*} = {j^T_{d,c_{\Gamma,R^d,((0))}, \mathcal{M}_{\Gamma,R^d} }}$, 
 $\phi^{T}_{\Gamma,R^d,*} = {\phi^T_{d,c_{\Gamma,R^d,((0))}, \mathcal{M}_{\Gamma,R^d} }}$, $j^{T,T'}_{\Gamma,R^d,*} = {j^{T,T'}_{d,c_{\Gamma,R^d,((0))}, \mathcal{M}_{\Gamma,R^d} }}$.
  \item If $\mathbf{B} \in \desc(Y,T',*)$ and $T'$ is a finite subtree of $T$, then $c_{\Gamma,Y, \mathbf{B}}^T = j^{T',T}_{\Gamma,Y}(\underline{c_{\mathbf{B}}^{T'}})^{\mathcal{M}_{\Gamma,Y}}$. 
\end{enumerate}

By coherency, if $\rho$ factors $(R,Y,T)$, then $\rho$ induces an elementary embedding
\begin{displaymath}
  \rho_{\Gamma}^{Y,T} : \mathcal{M}_{\Gamma,R} \to \mathcal{M}_{\Gamma,Y}^T
\end{displaymath}
where 
\begin{displaymath}
\rho^{Y,T}_{\Gamma} ( \tau^{\mathcal{M}_{\Gamma,R}} ( c_{\Gamma,R,r_1},\ldots ) ) = \tau^{\mathcal{M}_{\Gamma,Y}^T} (c_{\Gamma,Y, \rho(r_1)}^T,\ldots).
\end{displaymath}
If $\rho$ factors $(R,Y)$, then $\rho$ induces
\begin{displaymath}
  \rho^Y_{\Gamma}: \mathcal{M}_{\Gamma,R} \to \mathcal{M}_{\Gamma,Y}
\end{displaymath}
where $\rho^{Y}_{\Gamma} ( \tau^{\mathcal{M}_{\Gamma,R}} ( c_{\Gamma,R,r_1},\ldots ) ) = \tau^{\mathcal{M}_{\Gamma,Y}} (c_{\Gamma,Y, \rho(r_1)},\ldots).
$

Recall that wellfoundedness of a (level-1) EM blueprint is a $\Pi^1_2$ condition, stating that for every countable ordinal $\alpha$, the EM model generated by order indiscernibles of order type $\alpha$ is wellfounded. Its higher level analog is called iterability, which is a $\Pi^1_4$ condition. 

\begin{mydefinition}
  \label{def:iterable-EM}
Let $\Gamma$ be a level-3 EM blueprint.  $\Gamma$ is \emph{iterable} iff for any  $\Pi^1_3$-wellfounded level-3 tree $Y$, 
$\mathcal{M}_{\Gamma,Y}$ is a $\Pi^1_3$-iterable mouse. 
\end{mydefinition}

\begin{mylemma}
  \label{lem:3sharp_is_iterable}
  Assume $\boldpi{3}$-determinacy. Then $0^{3\#}$ is iterable.
\end{mylemma}
\begin{proof}
  Let $Y$ be any $\Pi^1_3$-wellfounded level-3 tree.  Let $F \in [C]^{Y \uparrow}$, where $C$ is a firm set of potential level-3 indiscernibles for $M_{2,\infty}^{-}$. 
Then $\mathcal{M}_{0^{3\#},Y}$ elementarily embeds into $M_{2,\infty}^{-}$, the map being generated by $c_{0^{3\#}, Y, s} \mapsto  [F]^Y_s$. Therefore, $\mathcal{M}_{0^{3\#},Y}$ is iterable. 
\end{proof}

\begin{mylemma}
  \label{lem:iterable_level_2_invariance_absolute_in_V}
 Assume $\boldDelta{2}$-determinacy.
 \begin{enumerate}
 \item Suppose $\mathcal{N}$ is a countable $\Pi^1_3$-iterable mouse satisfying Axioms~\ref{item:EM_ZFC}-\ref{item:level-2-embedding_invariance} in Definition~\ref{def:pre_level_3_EM_blueprint}.
   \begin{enumerate}
   \item If $T$ is a $\Pi^1_2$-wellfounded level $\leq 2$ tree, then $\mathcal{N}^T$ is a $\Pi^1_3$-iterable mouse and $j^T_{\mathcal{N} } : \mathcal{N} \to \mathcal{N}^T$ is essentially an iteration map.
   \item If $\psi$ minimally factors level $\leq 2$ trees $(T,X)$, then $\psi^X_{\mathcal{N}} : \mathcal{N}^T \to \mathcal{N}^X$ is essentially an iteration map. 
   \end{enumerate}
 \item 
Suppose $\Gamma$ is an iterable level-3 EM blueprint and $Y$ is a $\Pi^1_3$-wellfounded level-3 tree.  If $\psi$ minimally factors level-3 trees $(Y,R)$ and $\llbracket \emptyset \rrbracket_Y = \llbracket \emptyset \rrbracket_R$, then $\psi^R_{\Gamma} : \mathcal{M}_{\Gamma,Y} \to \mathcal{M}_{\Gamma,R}$ is essentially an iteration map.
 \end{enumerate}

\end{mylemma}
\begin{proof}
1(a).  
By Axiom~\ref{item:EM_ZFC} in Definition~\ref{def:pre_level_3_EM_blueprint}, there are cofinally many cardinal strong cutpoints in $\mathcal{N}$. 
 $j^T_{\mathcal{N}}$ is cofinal in $\mathcal{N}^T$ by definition and a direct limit argument. By Dodd-Jensen, it suffices to show that for any cardinal strong cutpoint $\xi$ of $\mathcal{N}$, $j^T_{\mathcal{N}}(\mathcal{N}|\xi)$ is $\Pi^1_3$-iterable and $j^T_{\mathcal{N}} \res (\mathcal{N} | \xi) $ is essentially an iteration map from $\mathcal{N} | \xi$ to $\mathcal{N}^T | j^T_{\mathcal{N}}(\xi)$. Fix such $\xi$. Let $g$ be $\coll(\omega,\xi)$-generic over $\mathcal{N}$. 
The statement ``for any $\Pi^1_2$-wellfounded level $\leq 2$ tree $T'$,  $j^{T'}_{\mathcal{N}} (\mathcal{N}|\xi) $ is a $\Pi^1_3$-iterable mouse, and $j^{T'}_{\mathcal{N}} \res (\mathcal{N}|\xi) $ is essentially an iteration map from $\mathcal{N}|\xi$ to $j^{T'}_{\mathcal{N}} (\mathcal{N}|\xi) $'' is $\Pi^1_3$ in a real $z \in \mathcal{N}[g]$ coding $(\mathcal{N}|\xi , j^{T'}_{\mathcal{N}}\res (\mathcal{N}|\xi ))_{T' \text{  finite level $\leq 2$ tree}}$. This statement is true in $\mathcal{N}[g]$ by Level $\leq 2$ ultrapower invariance axiom in Definition~\ref{def:pre_level_3_EM_blueprint}. 
It suffices to show  $\mathcal{N}[g]\elem_{\Sigma^1_3} V$. 
But by Axiom~\ref{item:EM_ZFC} in Definition~\ref{def:pre_level_3_EM_blueprint}, $\mathcal{N}\models $``I am closed under the $M_1^{\#}$-operator''. Since $\mathcal{N}$ is a $\Pi^1_3$-iterable mouse, the $M_1^{\#}$-operators are correctly computed in $\mathcal{N}$. Using genericity iterations \cite{steel-handbook}, $M_1^{\#}$-operators figure out $\Sigma^1_3$-truth. Hence, $\mathcal{N}[g]\elem_{\Sigma^1_3} V$.

1(b).  By Theorem~\ref{thm:factor_ordertype_embed_equivalent_lv2}, there is a $\Pi^1_2$-wellfounded $Q$ and $\pi$ minimally factoring $(X,T \otimes Q)$. So $\id_{T,*} = \pi \circ \psi$. By Axiom~\ref{item:EM_commutativity} in Definition~\ref{def:pre_level_3_EM_blueprint},  $ j^Q_{\mathcal{N}^T}= \pi^{T \otimes Q}_{\mathcal{N}} \circ \psi^X_{\mathcal{N}}$, which is essentially an iteration from $\mathcal{N}^T$ to $\mathcal{N}^{T \otimes Q}$ by part 1(a).  By Dodd-Jensen, $\psi^X_{\mathcal{N}}$ is essentially an iteration map. 

2. By Theorem~\ref{thm:factor_tower_order_type_equivalent}, there is a $\Pi^1_2$-wellfounded $T$ and $\rho$ minimally factoring $(R, Y \otimes T)$.  So $\id_{Y,*} =  \rho \circ \psi$. By Axiom~\ref{item:EM_commutativity} in Definition~\ref{def:pre_level_3_EM_blueprint} and part 1, $j^T_{\Gamma,Y} = \rho^{Y , T}_{\Gamma} \circ \psi ^R _{\Gamma}$ is essentially an iteration from $\mathcal{M}_{\Gamma,Y}$ to $\mathcal{M}_{\Gamma,Y}^T$. By Dodd-Jensen, $\psi^R_{\Gamma}$ is essentially an iteration map. 
\end{proof}

We start to introduce the remarkability property of a level-3 EM blueprint.

For $r ,s \in \omega^{<\omega}$, define $r <_0 s $ iff $r(0) <_{BK} s(0)$, $r \leq ^R_0 s$ iff $r(0) \leq _{BK} s(0)$. If $\vec{r} = (r_i)_{1 \leq i \leq n}$ is a tuple of nodes in $\omega^{<\omega}$, define $\vec{r} < _0 s$ iff $r_i <_0 s$ for any $i$. Similarly define $\vec{r} \leq _0 s$, $\vec{r} <_0 \vec{s}$, etc.\ 

 \begin{mydefinition}[Unboundedness]
   \label{def:level_3_EM_blueprint_unbounded}
   A level-3 EM blueprint $\Gamma$  is \emph{unbounded} iff for any level-3 tree $R$, if  $\tau$ is an  $\mathcal{L}$-Skolem term,  $\se{ t, r_1,\dots,r_m } \subseteq \dom(R)$,  $\vec{r} <_0 t$,  then $\Gamma(R)$ contains the formula
     \begin{displaymath}
       \tau (\underline{c_{r_1}}, \ldots, \underline{c_{r_m}}) \in \ord \to  \tau (\underline{c_{r_1}}, \ldots, \underline{c_{r_m}}) < \underline{c_t}.
     \end{displaymath}
 \end{mydefinition}

 \begin{mylemma}
   \label{lem:3sharp_is_unbounded}
Assume $\boldpi{3}$-determinacy. Then $0^{3\#}$ is unbounded.
 \end{mylemma}
 \begin{proof}
Let $C$ be a firm club of potential level-3 indiscernibles for $M^{-}_{2,\infty}$. 
Let $\eta \in D$ iff $C \cap \eta$ has order type $u_{\omega} \xi$ for some ordinal $\xi$. 

We may further assume that $\dom(R)$ is the upward closure of $\vec{r} \cup \se{t}$ and $R^{-} \DEF R \res ($the upward closure of $\vec{r})$ is a level-3 subtree of $R$. The reason is because we can find level-3 trees $S^{-},S$ ,   $\rho^{-}$ factoring $(S^{-},R)$, $\rho$ factoring $(S,R)$ so that $S^{-}$ is a subtree of  $S$, $\rho^{-} = \rho \res S^{-}$, $\ran(\rho^{-})$ is the upward closure of $\vec{r}$,  $\ran( \rho)$ is the upward closure of $\vec{r} \cup \se{t}$.  We then work with $S$ and $\rho^{-1}(\vec{r},t)$ instead, and finally apply the coherency of $0^{3\#}$.

 Suppose $\Gamma(R)$ contains the formula ``$\tau(\underline{c}_{r_1},\ldots,\underline{c}_{r_m}) \in \ord $''. Then for any $\vec{\gamma} \in [C]^{R\uparrow }$,
  \begin{displaymath}
\tau^{M_{2,\infty}^{-}(x)} (\gamma_{r_1},\ldots,\gamma_m) <\bolddelta{3}.
\end{displaymath}
Our assumption $\vec{r} <_0 t$ ensures the existence of $\vec{\delta} \in [D]^{R \uparrow}$ extending $\vec{\gamma} \res \dom(R^{-})$ such that   
\begin{displaymath}
  \tau^{M_{2,\infty}^{-}} (\gamma_{r_1},\dots, \gamma_{r_m}) <  \delta_t.
\end{displaymath}
Hence,  $\Gamma(R)$ contains the formula ``$\tau(\underline{c_{r_1}},\ldots,\underline{c_{r_m}})< \underline{c_{{t} }}$''.
 \end{proof}

 \begin{mydefinition}[Weak remarkability]
   \label{def:level_3_EM_blueprint_weakly_remarkable}
   A level-3 EM blueprint $\Gamma$  is \emph{weakly remarkable} iff $\Gamma$ is unbounded and for any level-3 tree $R$,
    if $\tau$ is an  $\mathcal{L}$-Skolem term, $\vec{r}\cup \vec{s} \cup \vec{s'}\cup \se{t} \subseteq \dom(R)$, $\vec{r} <_0 t \leq _0 \vec{s} \leq_0 \vec{s'}$,  $\vec{s'}$ is an $R$-shift of $\vec{s}$, 
$\lh(t) = 1$, then  $\Gamma(R)$ contains the formula    
 \begin{multline*} \tau(\underline{c_{r_1}},\ldots,\underline{c_{r_m}}, \underline{c_{s_1}},\ldots,\underline{c_{s_n}}) < \underline{c_t }\to\\
\tau(\underline{c_{r_1}},\ldots,\underline{c_{r_m}}, \underline{c_{s_1}},\ldots,\underline{c_{s_n}}) =
\tau(\underline{c_{r_1}},\ldots,\underline{c_{r_m}}, \underline{c_{s_1'}},\ldots,\underline{c_{s_n'}}) .
     \end{multline*}
 \end{mydefinition}

 \begin{mylemma}
   \label{lem:3sharp_is_weakly_remarkable}
Assume $\boldpi{3}$-determinacy.   Then $0^{3\#}$ is weakly remarkable.
 \end{mylemma}
 \begin{proof}
Again, we may assume that $\dom(R)$ is the upward closure of $\vec{r} \cup \vec{s} \cup \vec{s'} \cup \se{ t}$. 

Suppose $0^{3\#}(R)$ contains the formula ``$\tau(\underline{c_{r_1}},\ldots, \underline{c_{s_1}},\ldots) < \underline{c_{t}} $''. 
We need to show that $0^{3\#}(R)$ contains the formula ``$ \tau(\underline{c_{r_1}},\ldots, \underline{c_{s_1}},\ldots) = \tau(\underline{c_{r_1}},\ldots, \underline{c_{s_1'}},\ldots)$''. By Axiom~\ref{item:EM_order} in Definition~\ref{def:pre_level_3_EM_blueprint}, we may further assume that $t$ is in the upward closure of $\vec{s}$.
Let $S$ be a level-3 tree and $\rho,\rho'$ both  factor $(S,R)$ such that $\ran(\rho)$ is the upward closure of $\vec{r} \cup \vec{s}$, $\ran(\rho')$ is the upward closure of $\vec{r} \cup \vec{s'} $. 
Put $\rho^{-1}(r_i,s_j,t)= (\bar{r}_i,\bar{s}_j,\bar{t}) = (\rho')^{-1}(r'_i,s'_j,t')$. 

Let $C $ be a firm set of potential level-3 indiscernibles for $M^{-}_{2,\infty}$.  
 Let $C= \bigcup _{\xi < \bolddelta{3}} C_{\xi}$ be a disjoint partition of $C$ such that for any $\xi<\bolddelta{3}$,  $\ot(C_{\xi}) = u_{\omega}$, and for any $\xi<\eta<\bolddelta{3}$,  any member of $C_{\xi}$ is smaller than any member of $C_{\eta}$. Let $D$ be a club in $\bolddelta{3}$ where $\nu \in D$ iff $\sup \bigcup_{\xi<\nu} C_{\xi} = \nu$.
 As $C$ is firm, $D$ has order type $\bolddelta{3}$. 

If $X,Y$ are subsets of ordinals, define $X \sqsubseteq Y$ iff $X \subseteq Y$ and $X = Y \cap \alpha$ for some $\alpha$.
For each $0 < \xi < \bolddelta{3}$, let $F^{\xi} \in D^{S \uparrow}$ so that $(F^{\xi})'' \rep( U )  \sqsubseteq C_0$, $(F^{\xi})'' (\rep(S) \setminus \rep(U)) \sqsubseteq C_{\xi}$. Define $\vec{\gamma}^{\xi} = (\gamma^{\xi}_x)_{x \in \dom(S)}=  [F^{\xi}]^S$. 
Define
\begin{displaymath}
  \epsilon_{\xi}  = \tau^{M_{2,\infty}^{-}} ( \gamma^{\xi}_{\bar{r}_1}, \dots, \gamma^{\xi}_{\bar{s}_1}, \dots ).
\end{displaymath}
Hence,
\begin{displaymath}
  \epsilon_{\xi} < \min(C_1).
\end{displaymath}

For $0 < \eta < \xi < \bolddelta{3}$, define $\vec{\gamma}^{\eta\xi}  = ( \gamma^{\eta\xi}_y)_{y \in \dom(R)}$ where $\gamma^{\eta\xi}_{\rho ( x) } = \gamma^{\eta}_x$ and $\gamma^{\eta\xi}_{\rho' (x) } = \gamma^{\xi}_x$ for any $x \in \dom(S)$.
By Lemma~\ref{lem:R_respect}, $\gamma^{\eta\xi} \in [D]^{R \uparrow}$.
Suppose towards a contradiction. 

Case 1:  $0^{3\#}(R)$ contains the formula `` $\tau(\underline{c_{r_1}},\ldots,\underline{c_{s_1}},\ldots) > \tau(\underline{c_{r_1}},\ldots,\underline{c_{s_1'}},\ldots) $''.

Then $\vec{\gamma}^{\xi\eta}$ witnesses that $\epsilon_{\eta}>\epsilon_{\xi}$ whenever $0 < \eta < \xi <\bolddelta{3}$. This is an infinite descending chain of ordinals.

Case 2: $0^{3\#}(R)$ contains the formula `` $\tau(\underline{c_{r_1}},\ldots,\underline{c_{s_1}},\ldots) < \tau(\underline{c_{r_1}},\ldots,\underline{c_{s_1'}},\ldots) $''.

Then $\epsilon_{\eta}<\epsilon_{\xi}$ whenever $ 0 < \eta < \xi < \bolddelta{3}$, contradicting to $\epsilon_{\xi} < \min(C_1)$. 
 \end{proof}

If $R$ is a level-3 tree, $t \in \dom(R)$, $\lh(t) = 1$, let 
\begin{displaymath}
  R \res t = R \res \set{r \in \dom(R)}{ r <_0 t}.
\end{displaymath}

 \begin{mylemma}
   \label{lem:weakly_remarkable_strengthening}
   Suppose $\Gamma$ is a weakly remarkable level-3 EM blueprint. Suppose $R$ is a level-3 tree, $t \in \dom(R)$, $\lh(t) = 1$.  
   \begin{enumerate}
   \item If  $\tau$ is an $\mathcal{L}$-Skolem term, $\vec{r}\cup \vec{s} \cup \vec{s'}\cup \se{t} \subseteq \dom(R)$, $\vec{r} <_0 t \leq _0 \vec{s} \concat \vec{s'}$,  $\vec{s'}$ is an $R$-shift of $\vec{s}$, then  $\Gamma(R)$ contains the formula    
 \begin{multline*} \tau(\underline{c_{r_1}},\ldots,\underline{c_{r_m}}, \underline{c_{s_1}},\ldots,\underline{c_{s_n}}) < \underline{c_t }\to\\
\tau(\underline{c_{r_1}},\ldots,\underline{c_{r_m}}, \underline{c_{s_1}},\ldots,\underline{c_{s_n}}) =
\tau(\underline{c_{r_1}},\ldots,\underline{c_{r_m}}, \underline{c_{s_1'}},\ldots,\underline{c_{s_n'}}) .
     \end{multline*}
   \item  $\Gamma(R)$ contains the scheme ``$K|\underline{c_{t}} \elem V$''. In particular, $\Gamma(R)$ contains the formula ``$\underline{c_{t}}$ is inaccessible and there are cofinally many cardinal strong cutpoints below $\underline{c_{t}}$''.
   \end{enumerate}
 \end{mylemma}
 \begin{proof}
1.  Assume without loss of generality that $\dom(R)$ is the upward closure of $\vec{r} \cup \vec{s} \cup \vec{s}' \cup \se{t}$. 
Suppose $\Gamma(R)$ contains the formula ``$\tau(\underline{c_{r_1}},\ldots, \underline{c_{s_1}},\ldots) < \underline{c_{t}} $''. 
Expand $R$ to the level-3 tree $S$ where $\dom(S)$ is the upward closure of $\dom(R)\cup \se{s_i'':1 \leq i \leq n}$, each $s_i'' \notin \dom(R)$, $\vec{s''}$ is an $R$-shift of $\vec{s}$, $\vec{s} <_0 \vec{s''}$.  By coherency and weak remarkability, $\Gamma(S)$ contains the formula $  \tau(\underline{c_{r_1}},\ldots, \underline{c_{s_1}},\ldots) =\tau(\underline{c_{r_1}},\ldots, \underline{c_{s_1''}},\ldots) $. But $\vec{r}\concat \vec{s} \concat \vec{s''}$ is a shift of $\vec{r} \concat \vec{s'} \concat \vec{s''} $. By indiscernability,  $\Gamma(S)$ contains the formula $\tau(\underline{c_{r_1}},\ldots, \underline{c_{s_1'}},\ldots) =\tau(\underline{c_{r_1}},\ldots, \underline{c_{s_1''}},\ldots) $.  
Hence, $\Gamma(S)$ contains the formula $\tau(\underline{c_{r_1}},\ldots, \underline{c_{s_1}},\ldots) =\tau(\underline{c_{r_1}},\ldots, \underline{c_{s_1'}},\ldots) $.

2. 
Put $\mathcal{N} = \mathcal{M}_{\Gamma,R}$. 
By coherency of $\Gamma$, we may assume that $A = \set{s \in \dom(R)}{s <_0 t }$ has $<_{BK}$-limit order type.  
By Tarski's criterion, we need to show that if  $w = \tau^{\mathcal{N}} (z_1,\dots, z_k)  \in \ord$,
$z_1,\dots,z_k < \underline{c_t}^{\mathcal{N}}$, then $w < \underline{c_t}^{\mathcal{N}}$. 
To save notations, let $k=1$, $z _1= 
  \sigma^{\mathcal{N}}(\underline{c_{r_1}}^{\mathcal{N}}, \dots, \underline{c_{s_1}}^{\mathcal{N}},\dots)< \underline{c_t}^{\mathcal{N}}$, $\vec{r} <_0 t \leq _0 \vec{s}$.
Pick $t^{*}$ of length 1 such that $\vec{r} <_0 t^{*} <_0  t $. 
 Build a level-3 tree $S$ that extends $R$ in which there are nodes $t', \vec{s'} \in \dom(S)$ such that $\vec{r} <_0 (t' )\concat \vec{s'} <_0 t^{*}$ and $(t' )\concat \vec{s'}$ is an $S$-shift of $(t )\concat \vec{s}$. Put $\mathcal{P} = \mathcal{M}_{\Gamma,S}$. 
By weakly remarkability,    
\begin{displaymath}
\sigma^{\mathcal{P}}(\underline{c_{r_1}}^{\mathcal{P}}, \dots, \underline{c_{s_1}}^{\mathcal{P}}) = \sigma^{\mathcal{P}} (\underline{c_{r_1}}^{\mathcal{P}}, \dots, \underline{c_{s_1'}}^{\mathcal{P}}, \dots ).
\end{displaymath}
By unboundedness of $\Gamma$,
\begin{displaymath}
  \tau^{\mathcal{P}} ( \sigma^{\mathcal{P}} (\underline{c_{r_1}}^{\mathcal{P}}, \dots, \underline{c_{s_1'}}^{\mathcal{P}}, \dots )) <  \underline{c_{t^{*}}}^{\mathcal{P}}.
\end{displaymath}
By coherency of $\Gamma$,  $w< \underline{c_{t^{*}}}^{\mathcal{N}}$. By Axiom~\ref{item:EM_order} in Definition~\ref{def:pre_level_3_EM_blueprint}, $\underline{c_{t^{*}}}^{\mathcal{N}} < \underline{c_t}^{\mathcal{N}}$. 
 \end{proof}

A level-3 tree $R$ is said to be \emph{universal above $t$} iff $t \in \dom(R)$, $\lh(t) = 1$, and 
 for any level-3 tree $S$, if $S \res {t}' $ is isomorphic to $R \res {t}$   via $\pi$ 
and $\dom(S) \setminus \dom(S \res {t}')$ is finite, then there is a map $\rho$ factoring $(S,R)$ that extends $\pi$.  
Clearly, for any $R$, there is $(R',{t})$ such that  $R' \res{t}$ is isomorphic to $R$ and $R'$ is universal above $t$. If $R$ is $\Pi^1_3$-wellfounded, we may further demand that $R'$ is $\Pi^1_3$-wellfounded.

\begin{mylemma}\label{lem:universal_agree_on_EM_model}
Suppose  $\Gamma$ is a weakly remarkable level-3 EM blueprint, $R$ is universal above $t$,  $R'$ is universal above $t'$,   $R \res {t} $ is isomorphic to $ R'  \res {t}'$. 
Then
\begin{displaymath}
  (K|{\underline{c_{{t}}}}) ^{\mathcal{M}_{\Gamma,R}} \cong  (K|{\underline{c_{{t}'}}}) ^{\mathcal{M}_{\Gamma,R'}}.
\end{displaymath}
\end{mylemma}
\begin{proof}
To begin with, we build an isomorphism $\psi : (\underline{c_t})^{\mathcal{M}_{\Gamma,R}} \to  (\underline{c_{t'}})^ {\mathcal{M}_{\Gamma,R'}}$  which preserves membership relations in the respective EM models. Given $a \in \mathcal{M}_{\Gamma,R}$ such that $\mathcal{M}_{\Gamma,R} \models a < \underline{c_t}$, find a Skolem term $\tau$ and nodes $\vec{r}, \vec{s}$ such that $\vec{r} \prec^R t \preceq^R \vec{s}$ and 
\begin{displaymath}
 a = (\tau(\underline{c_{r_1}}, \ldots, \underline{c_{s_1}},\ldots))^{\mathcal{M}_{\Gamma,R} }.
\end{displaymath}
Let $S$ be a level-3 tree and $\rho$ factor $(S,R)$ such that $\ran(S)$ is the upward closure of $\dom( R \res t) \cup \vec{s} \cup \se{t}$. By universality, pick  $\rho'$ factoring $(S,R')$ which extends $\pi$. 
By coherency of $\Gamma$, $(\tau( \underline{c_{\pi(r_1)}}, \ldots, \underline{c_{\rho'\circ \rho^{-1}(s_1)}},\ldots))^{\mathcal{M}_{\Gamma,R'}} <^{\mathcal{M}_{\Gamma,R'}} \underline{c_{t'}}$. 
Define
\begin{displaymath}
\psi(a) = ( \tau( \underline{c_{\pi(r_1)}}, \ldots, \underline{c_{\rho'\circ \rho^{-1}(s_1)}},\ldots))^{\mathcal{M}_{\Gamma,R'}} .
\end{displaymath}

$\psi$ is well-defined and preserves membership. For this, we firstly show that $\psi(a)$ does not depend on the choice of $\rho'$. Suppose $\rho''$ is another candidate for $\rho'$. Then $\rho'' \circ \rho^{-1}(\vec{s})$ is an $R'$-shift of $\rho'\circ \rho^{-1}(\vec{s})$. By Lemma~\ref{lem:weakly_remarkable_strengthening}, $\mathcal{M}_{\Gamma,R'} \models \tau( \underline{c_{\pi(r_1)}}, \ldots, \underline{c_{\rho'\circ \rho^{-1}(s_1)}},\ldots) =\tau( \underline{c_{\pi(r_1)}}, \ldots, \underline{c_{\rho''\circ \rho^{-1}(s_1)}},\ldots) $. Secondly, the reason why  $\psi(a)$ does not depend on the choice of $\tau$ and $\vec{r},\vec{s}$ is because of coherency of $\Gamma$. In the same spirit, we can show that $\psi$ preserves membership. A completely symmetrical argument gives  $\psi' : (\underline{c_{t'}})^{\mathcal{M}_{\Gamma,R'}} \to (\underline{c_{t}})^{\mathcal{M}_{\Gamma,R}} $. By Lemma~\ref{lem:coherent_EM_blueprint_is_indiscernible_to_shifts}, $\psi\circ \psi' $ and $ \psi' \circ \psi $ are both identity functions. So $\psi$ is an isomorphism between $(\underline{c_t})^{\mathcal{M}_{\Gamma,R}}$ and $(\underline{c_{t'}})^{\mathcal{M}_{\Gamma,R'}}$.

 $\mathcal{M}_{\Gamma,R}$ is a model of $V=K$. Working in $\mathcal{M}_{\Gamma,R}$,  $ K | \underline{c_t}$ has a canonical wellordering of order type $\omega \underline{c_t}$, and similarly for $\mathcal{M}_{\Gamma,R'}$. $\psi$ extends to $\psi^{*}$, acting on $(K| \underline{c_t})^{\mathcal{M}_{\Gamma,R}}$ according to these canonical wellorderings. Using the same argument as before, $\psi^{*}$ is an isomorphism from $(K|{\underline{c_t}}) ^{\mathcal{M}_{\Gamma,R}}$ to $(K|{\underline{c_{t'}}}) ^{\mathcal{M}_{\Gamma,R'}}$.
\end{proof}

A level-3 tree $R$ is \emph{universal based on $Y$} iff there is 
$t \in \dom(R)$ such that $\lh( t ) = 1$, $R$ is universal above $t$  and  $ R \res t$ is isomorphic to $Y$. 
Suppose $\Gamma$ is a weakly remarkable level-3 EM blueprint. 
For a level-3 tree $Y$, if $R$ is universal based on $Y$,  $t \in \dom(R)$, $\lh(t ) = 1$,  $R \res t$ is isomorphic to $Y$, put
\begin{displaymath}
  \mathcal{M}^{*}_{\Gamma,Y} =  (K|{\underline{c_t}})^{\mathcal{M}_{\Gamma,R}}.
\end{displaymath}
$\mathcal{M}^{*}_{\Gamma,Y}$ is well-defined up to an isomorphism.
Its wellfounded part is transitivized. By Lemma~\ref{lem:weakly_remarkable_strengthening}, there are cofinally many cardinal strong cutpoints in $\mathcal{M}^{*}_{\Gamma,Y}$. Similarly, for a level $\leq 2$ tree $T$, define 
\begin{displaymath}
  \mathcal{M}^{*,T}_{\Gamma,Y} =  (K|{\underline{c_t}})^{\mathcal{M}_{\Gamma,R}^T}.
\end{displaymath}
Hence, $\mathcal{M}^{*,T}_{\Gamma,Y} = (\mathcal{M}^{*}_{\Gamma,Y})^T$. If $\rho$ factors $(Y,Y')$, $R'$ is universal above $R$, then $\rho^{*,Y'}_{\Gamma} = \rho^{R'}_{\Gamma} \res \mathcal{M}^{*}_{\Gamma,Y}$. If $\rho$ factors $(Y,Y',T)$, $R'$ is universal above $R$, then 
 $\rho^{*,Y',T}_{\Gamma} = \rho^{Y',T}_{\Gamma}  \res \mathcal{M}^{*}_{\Gamma,Y}$.


A $\Pi^1_3$-iterable mouse $\mathcal{P}$ is \emph{full} iff for any strong cutpoint $\eta$ of $\mathcal{P}$, for any $\Pi^1_3$-iterable mouse $\mathcal{Q}$ extending $\mathcal{P}|\eta$ which is sound and projects to $\eta$, $\mathcal{Q} \inisegeq \mathcal{P}$. 

\begin{mylemma}
  \label{lem:iterable_implies_mouse_order}
Assume $\boldDelta{2}$-determinacy.   Suppose $\Gamma$ is an iterable, weakly remarkable level-3 EM blueprint.
  \begin{enumerate}
 \item 
   Suppose $Y,Y'$ are  $\Pi^1_3$-wellfounded level-3 trees. Then 
$\llbracket \emptyset \rrbracket_Y= \llbracket \emptyset \rrbracket_{Y'}$ iff $\mathcal{M}_{\Gamma,Y} \sim_{DJ} \mathcal{M}_{\Gamma,Y'}$; $\llbracket \emptyset \rrbracket_Y< \llbracket \emptyset \rrbracket_{Y'}$ iff $\mathcal{M}_{\Gamma,Y} <_{DJ} \mathcal{M}_{\Gamma,Y'}$.
 \item Suppose $Y$ is a $\Pi^1_3$-wellfounded level-3 tree. Then $\mathcal{M}^{*}_{\Gamma,Y}$ is full.  
 \item 
   Suppose $Y,Y'$ are  $\Pi^1_3$-wellfounded level-3 trees. Then $\llbracket \emptyset \rrbracket_Y= \llbracket \emptyset \rrbracket_{Y'}$ iff $\mathcal{M}^{*}_{\Gamma,Y} \sim_{DJ} \mathcal{M}^{*}_{\Gamma,Y'}$; $\llbracket \emptyset \rrbracket_Y< \llbracket \emptyset \rrbracket_{Y'}$ iff $\mathcal{M}^{*}_{\Gamma,Y} <_{DJ} \mathcal{M}^{*}_{\Gamma,Y'}$.
 \end{enumerate}
\end{mylemma}
\begin{proof}
1. If $\llbracket \emptyset \rrbracket_Y \leq \llbracket \emptyset \rrbracket_{Y'}$, by Theorem~\ref{thm:factor_tower_order_type_equivalent}, there exist a $\Pi^1_3$-wellfounded $Z$ and $\rho$ minimally factoring $(Y,Z)$, $\rho'$ minimally factoring $(Y',Z)$ so that $\llbracket \emptyset \rrbracket_Y = \llbracket \emptyset \rrbracket_Z$. By Lemma~\ref{lem:iterable_level_2_invariance_absolute_in_V}, $\mathcal{M}_{\Gamma,Y}  \leq_{DJ}  \mathcal{M}_{\Gamma,Z} \sim_{DJ} \mathcal{M}_{\Gamma,Y'}$.

If $\llbracket \emptyset \rrbracket_Y < \llbracket \emptyset \rrbracket_{Y'}$, we further obtain $t \in \dom(R)$ so that $\lh(t) = 1$ and $\llbracket \emptyset \rrbracket_Y = \llbracket t \rrbracket_Z$. By unboundedness of $\Gamma$,  $\ran(\rho_{\Gamma}^{Y',T}) \subseteq \mathcal{M}_{\Gamma,Z} | c_{\Gamma,Z,t}. $
Hence,  $\rho_{\Gamma}^{Y,T}$ is  $\Sigma_1$-elementary from $\mathcal{M}_{\Gamma,Y}$ into $\mathcal{M}_{\Gamma,Z} | c_{\Gamma,Z,t}$. Hence $\mathcal{M}_{\Gamma,Y} <_{DJ}\mathcal{M}_{\Gamma,Z}$.

2. 
Recall that there are cofinally many cardinal strong cutpoints in $\mathcal{M}^{*}_{\Gamma,Y}$. 
  Suppose $\eta$ is a strong cutpoint of $\mathcal{M}^{*}_{\Gamma,Y}$ and $\mathcal{M}^{*}_{\Gamma,Y}|\eta \iniseg \mathcal{P}$, $\mathcal{P}$ is a sound $\Pi^1_3$-iterable mouse,  $\rho_{\omega}(\mathcal{P}) \leq\eta$. Let $Y'$ be a $\Pi^1_3$-wellfounded level-3 tree such that $\wocode{\mathcal{P}}_{DJ} < \llbracket \emptyset \rrbracket_{Y'}$ and $Y'$ is universal based on $Y$. By part 2, $\mathcal{P} <_{DJ} \mathcal{M}_{\Gamma,Y'}$. Since $\mathcal{M}^{*}_{\Gamma,Y}|\eta \iniseg \mathcal{M}_{\Gamma,Y'}$ and $\eta$ is a strong cutpoint of $\mathcal{M}_{\Gamma,Y'}$, the comparison between $\mathcal{P}$ and $\mathcal{M}_{\Gamma,Y'}$ is above $\eta$. It follows that $\mathcal{P} \iniseg \mathcal{M}_{\Gamma,Y'}$. Hence $\mathcal{P} \iniseg \mathcal{M}^{*}_{\Gamma,Y}$.

3. By parts 1-2 and remarkability of $\Gamma$. 
\end{proof}

Assume $\boldDelta{2}$-determinacy. 
Suppose $\Gamma$ is an iterable, weakly remarkable level-3 EM blueprint.  Suppose $Y$ is a $\Pi^1_3$-wellfounded level-3 tree. 

For $s \in \dom(Y)$, let
\begin{displaymath}
  c^{*}_{\Gamma,Y,s} = \underline{c_s}^{\mathcal{M}^{*}_{\Gamma,Y}}
\end{displaymath}
and 
\begin{displaymath}
c_{\Gamma,Y,s , \infty} = \pi_{\mathcal{M}^{*}_{\Gamma,Y},\infty} (   c^{*}_{\Gamma,Y,s} ) .
\end{displaymath}
In fact, $c_{\Gamma,Y,s,\infty}$ depends only on $(\llbracket s \rrbracket_Y, Y_{\tree}(s))$, shown as follows.
Suppose $Y'$ is another $\Pi^1_3$-wellfounded level-3 tree and  $(\llbracket s \rrbracket_Y, Y_{\tree}(s)) = (\llbracket s'  \rrbracket_{Y'}, Y'_{\tree}(s'))$. 
 By Lemma~\ref{lem:unique_level_3_tree_represent}, $Y[s] = Y'[s']$.
By Theorem~\ref{thm:factor_tower_order_type_equivalent}, we can find $\Pi^1_3$-wellfounded  level-3 trees $R,R'$ which are universal based on $Y,Y'$ respectively, a $\Pi^1_3$-wellfounded $Z$ and $\rho$ minimally factoring $(R,Z)$, $\rho'$ minimally factoring $(R',Z)$. In particular, $\rho(s) = \rho'(s')$. By Lemma~\ref{lem:iterable_level_2_invariance_absolute_in_V},  $\rho^Z_{\Gamma}: \mathcal{M}_{\Gamma,R} \to \mathcal{M}_{\Gamma,Z}$ is essentially an iteration map, sending $c^{*}_{\Gamma,Y,s}$ to $c_{{\Gamma,Z, \rho(s)}}$, and similarly on the $\rho'$-side. Hence $c_{\Gamma,Y,s,\infty} = c_{\Gamma,Y',s',\infty}$.  We can safely define
\begin{displaymath}
  c_{\Gamma,Q, \gamma} = c_{\Gamma,Y,s,\infty}
\end{displaymath}
for $Y_{\tree}(s) = Q$ and $\gamma = \llbracket s  \rrbracket_Y$.

If $(Q,\overrightarrow{(d,q,P)}) = (Q,(d_i,q_i, P_i)_{1 \leq i \leq k})$ is a potential partial level $\leq 2$ tower, let $F \in B^{(Q,\overrightarrow{(d,q,P)}) \uparrow}$ iff $F : [\omega_1]^{Q \uparrow} \to B$ is an order preserving function and
\begin{enumerate}
\item if $(Q,\overrightarrow{(d,q,P)}) $ is of continuous type, then the signature of  $F$ is  $(d_i,q_i)_{1 \leq i \leq  k}$, $F$ is essentially continuous;
\item if $(Q,\overrightarrow{(d,q,P)}) $ is of discontinuous type, then the signature of  $F$ is  $(d_i,q_i)_{1 \leq i <  k}$, $F$ is essentially discontinuous, $F$ has uniform cofinality $\ucf(Q,\overrightarrow{(d,q,P)})$.
\end{enumerate}
Let $\gamma \in [B]^{(Q,\overrightarrow{(d,q,P)}) \uparrow}$ iff $\gamma  = [F]_{\mu^Q}$ for some $F \in B^{(Q,\overrightarrow{(d,q,P)})\uparrow}$. 
${\gamma}$ is said to \emph{respect} $(Q,\overrightarrow{(d,q,P)})$ iff $\gamma \in (\bolddelta{3})^{(Q,\overrightarrow{(d,q,P)}) \uparrow}$. $\gamma$ is said to \emph{respect} $Q$ if $\gamma$ respects some potential partial level $\leq 2$ tower $(Q, \overrightarrow{(d',q',P')})$. 
By Lemma~\ref{lem:order_type_realizable}, $\gamma$ respects $Q$ iff there is a $\Pi^1_3$-wellfounded $Y$ and $s$ such that $Y_{\tree}(s) = Q$ and $\gamma = \llbracket s \rrbracket_Y$. Hence, $c_{\Gamma,Q, \gamma}$ is defined whenever $\gamma$ respects $Q$ and the map $\gamma \mapsto c_{\Gamma,Q,\gamma}$ is order preserving. 
Define
\begin{displaymath}
  c_{\Gamma,\gamma} = c_{\Gamma,Q^0, \gamma} .
\end{displaymath}
$c_{\Gamma,\gamma}$ is defined whenever $\gamma< \bolddelta{3}$ is a limit ordinal. 
Remarkability will ensure that the map $\gamma \mapsto c_{\Gamma,\gamma}$ is continuous. Assuming $\boldDelta{3}$-determinacy, define
\begin{align*}
  c_{Q,\gamma}^{(3)} & = c_{0^{3\#}, Q, \gamma}, \\
  c_{\gamma}^{(3)} &= c_{0^{3\#}, \gamma}, \\
I^{(3)} & = \set{c^{(3)}_{Q, \gamma}}{\gamma \text{ respects }Q}. 
\end{align*}
$I^{(3)}$ is the higher analog of Silver indiscernibles for $L$. 

\begin{mylemma}
  \label{lem:3sharp_indiscernible_continuous}
  Assume $\boldpi{3}$-determinacy. Then there is a club $C \subseteq \bolddelta{3}$ such that $C \in L[T_3,0^{3\#}]$ and for any potential partial level $\leq 2$ tree $(Q,\overrightarrow{(d,q,P)})$, for any $\gamma \in [C]^{(Q, \overrightarrow{(d,q,P)})\uparrow}$,
  \begin{displaymath}
    \gamma= c_{Q, \gamma}^{(3)}. 
  \end{displaymath}
\end{mylemma}
\begin{proof}
  Let $D$ be a firm set of potential level-3 indiscernibles for $M_{2,\infty}^{-}$ and let $\eta \in C$ iff $\eta \in D$ and $D \cap \eta$ has order type $\eta$. $C$ works for the lemma.
\end{proof}

 Recall  Definition~\ref{def:typical_level_3_tree} for the definition of $R^{d}$. An ordinal $\alpha<\omega_1$ is \emph{$\omega_1$-represented by $T$} iff $(1,(0)) \in \dom(T)$ and $\llbracket 1,(0) \rrbracket_T = \alpha$. $\alpha < u_2$ is \emph{$u_2$-represented by $T$} iff $(2, ((0))) \in \dom(T)$ and $\llbracket 2, ((0)) \rrbracket_T = \alpha$. 

\begin{mydefinition}[Remarkability]
  \label{def:EM_remarkability}
 A weakly remarkable level-3 EM blueprint $\Gamma$ is \emph{remarkable} iff 
\begin{enumerate}
\item $\Gamma(R^{0})$ contains the axiom ``$\underline{c_{((0))}}$ is not measurable''.
  \item  $\Gamma(R^{1})$ contains the following axiom:  
      if $\xi$ is a cardinal and strong cutpoint, $c = \underline{c_{((0))}}$, $b = (\underline{\phi^{Q^1}_{1 , c}})^{-1}(c)$, then $V^{\coll(\omega,\xi)}$ satisfies the following: 
     \begin{enumerate}
\item  If $\alpha$ is $\omega_1$-represented by both $T$ and $T'$, then 
$(  (\underline{j^T_{1,c}})^K(K|\xi) , ( (\underline{j^{Q^1,T}_{1,c}})^K (b)  ) \sim_{DJ} (  (\underline{j^{T'}_{1,c}})^K(K|\xi) , ( (\underline{j^{Q^1,T'}_{1,c}})^K (b)  ) $. 
Here
 $(\underline{j^U_{1,c}})^K$ stands for the direct limit of $(\underline{j^{Z,Z'}_{1,c}})^K$ for $Z,Z'$ finite subtrees of $U$, $Z$ a finite subtree of $Z'$, and $(\underline{j^{Q^1,U}_{1,c}})^K$ stands for the tail of the direct limit map from $(\underline{j^{Q^1}_{1,c}})^K (K)$ to $(\underline{j^U_{1,c}})^K (K)$.
\item 
 Let $F({\alpha}) = \pi_{(\underline{j^T_{1,c}})^K(K|\xi), \infty}( ( \underline{j^{Q^1,T}_{1,c}})^K  (b)) $ for ${\alpha}$ represented by $T$. Then $\sup_{\alpha < \omega_1} F(\alpha) = \pi_{K|\xi , \infty }(c)$.   
 \end{enumerate}
 \item  $\Gamma(R^{2})$ contains the following axiom:  
      if $\xi$ is a cardinal and strong cutpoint, $e \in \se{0,1}$, $c = \underline{c_{((0))}}$, $b = (\underline{\phi^{Q^{2e}}_{1 , c}})^{-1}(c)$, then $V^{\coll(\omega,\xi)}$ satisfies the following: 
     \begin{enumerate}
\item  If $\alpha$ is $u_2$-represented by both $T$ and $T'$, then 
$(  (\underline{j^T_{2,c}})^K(K|\xi) , ( (\underline{j^{Q^{2e},T}_{2,c}})^K (b)  ) \sim_{DJ} (  (\underline{j^{T'}_{2,c}})^K(K|\xi) , ( (\underline{j^{Q^{2e},T'}_{2,c}})^K (b)  ) $. 
Here
 $(\underline{j^U_{2,c}})^K$ stands for the direct limit of $(\underline{j^{Z,Z'}_{2,c}})^K$ for $Z,Z'$ finite subtrees of $U$, $Z$ a finite subtree of $Z'$, and $(\underline{j^{Q^{2e},U}_{2,c}})^K$ stands for the tail of the direct limit map from $(\underline{j^{Q^{2e}}_{2,c}})^K (K)$ to $(\underline{j^U_{2,c}})^K (K)$.
\item 
 Let $F({\alpha}) = \pi_{(\underline{j^T_{2,c}})^K(K|\xi), \infty}( ( \underline{j^{Q^{2e},T}_{2,c}})^K  (b)) $ for ${\alpha}$ represented by $T$. Then $\sup_{\alpha < u_2} F(\alpha) = \pi_{K|\xi , \infty }(c)$.   
 \end{enumerate}
  \end{enumerate}
\end{mydefinition}

In the next lemma, we denote $\mathbf{y}^1 =( ((0) ,-1) , Q^1, ((1, (0), \emptyset))   ) \in \desc(R^{1}) $, $\mathbf{B}^1 = (\mathbf{y}^1, \id_{Q^1,*}) \in \desc(R^{1}, Q^1, Q^0)$,  $\mathbf{y}^{2e} = (((0),-1), Q^{2e}, ((2, ((0)), \se{(0)}))) \in \desc(R^{2})$, $\mathbf{B}^{2e} = (\mathbf{y}^{2e}, \id_{Q^{2e}, *}) \in \desc(R^{2}, Q^{2e}, Q^0)$ for $e \in \se{0,1}$.
Note that if $\Gamma$ is a level-3 EM blueprint then $\Gamma(R^{1})$ contains the axiom
\begin{displaymath}
  \underline{c_{((0))}} = \underline{c_{\mathbf{y}^1}} = \underline{c_{\mathbf{B}^1}^{Q^1}}
\end{displaymath}
and $\Gamma(R^{2})$ contains the axiom
\begin{displaymath}
  \underline{c_{((0))}} = \underline{ c_{\mathbf{y}^{20}} } = \underline{c_{\mathbf{y}^{21}}} = \underline{c_{\mathbf{B}^{20}}^{Q^{20}}} = \underline{c_{\mathbf{B}^{21}}^{Q^{21}}}.
\end{displaymath}

\begin{mylemma}
  \label{lem:EM_model_ultrapower_equal_normal}
  Suppose $\Gamma$ is a level-3 EM blueprint. Suppose
 $d \in \se{1,2}$,   $T$ is a finite level $\leq 2$ tree. 
  \begin{enumerate}
  \item $ \mathcal{M}_{\Gamma,R^{d}, * }^T = \mathcal{M}_{\Gamma, R^{d} \otimes T}$, $ \phi^T_{\Gamma,R^{d},*} = (\id_{R^{d} \otimes T})^{R^{d},T}_{\Gamma} $.
  \item If $Q^1$ is a subtree of $T$, then $j^{Q^1,T}_{\Gamma, R^{1}} \circ ( \phi^{Q^1}_{\Gamma, R^{1}, *})^{-1}(c_{\Gamma, R^{1}, ((0)) }) =  c_{\Gamma, R^{1} \otimes T, \mathbf{B}^1} $.
  \item For $e \in \se{0,1}$, if $Q^{2e}$ is a subtree of $T$, then $j^{Q^{2e},T}_{\Gamma, R^{2}} \circ ( \phi^{Q^{2e}}_{\Gamma, R^{2}, *})^{-1}(c_{\Gamma, R^{2}, ((0)) }) =  c_{\Gamma, R^{2} \otimes T, \mathbf{B}^{2e}} $.
  \end{enumerate}
\end{mylemma}
\begin{proof}
1. Put $Y = R^{d}$,  $c = c_{\Gamma, Y, ((0))}$, $R = Y \otimes  T$, $\rho = \id_R$ factoring $(R, Y, T)$, $\psi = \id_{Y, *}$ factoring $(Y, R)$. We only prove the typical case when $d=2$. Put $\mathbf{y} = (((0)), Q^0, \overrightarrow{(e,x,W)}) \in \desc(Y)$. 
We have to show that
  \begin{displaymath}
    \ran( \phi^T_{\Gamma, Y, *}) = \ran( \rho^{Y,T} _{\Gamma}). 
  \end{displaymath}

The $\subseteq $ direction: If $a \in \mathcal{M}_{\Gamma,Y}$, then $j^T_{\Gamma,Y} (a) = \rho_{\Gamma}^{Y,T} \circ \psi_{\Gamma}^R(a)$. 
If $Q$ is finite, $\pi$ factors $(Q^0, T \otimes Q)$, then $\pi ' \DEF \id_{T \otimes Q, *} \circ \pi$ factors $(Q^0, (T \otimes Q) \otimes Q^0)$ and $(\mathbf{y}, \pi') \in \dom( Y \otimes (T \otimes Q))$. 
Hence,
\begin{displaymath}
    ( \underline{\pi^{T \otimes Q}} ( \underline{c_{((0))}})  )^{\mathcal{M}_{\Gamma,Y}}  = c_{\Gamma,Y,(\mathbf{y}, \pi')} = \rho^{Y,T}_{\Gamma}  (   c_{\Gamma,R, \iota^{-1}_{Y,T,Q} (\mathbf{y}, \pi')}).
\end{displaymath}
If $\pi$ factors $(Q^{2e}, T \otimes Q)$,  $e \in \se{0,1}$, then $\pi' \DEF \id_{T \otimes Q, *} \circ \pi$ factors $(Q^{2e}, (T \otimes Q) \otimes Q^{2e})$ and $(\mathbf{y}^{2e}, \pi') \in \dom(Y \otimes (T \otimes Q))$. Argue similarly. 

The $\supseteq$ direction: By definition. 

2,3. Simple computation. 
\end{proof}

\begin{mylemma}
  \label{lem:EM_remarkability}
  Assume $\boldDelta{2}$-determinacy. Suppose $\Gamma$ is an iterable, weakly remarkable level-3 EM blueprint. The following are equivalent:
  \begin{enumerate}
  \item $\Gamma$ is remarkable.
  \item The map $\gamma \mapsto c_{\Gamma,\gamma}$ is continuous.
  \item There exist $\gamma_0, \gamma_1, \gamma_2$ such that for $d \in \se{0,1,2}$, $\cf^{\admistwobold}(\gamma_d) = u_d$ and
    \begin{displaymath}
      c_{\Gamma,\gamma_d} = \set{c_{\Gamma,\beta}} {\beta < \gamma_d}.
    \end{displaymath}
  \end{enumerate}
In particular, if $\boldpi{3}$-determinacy holds, then $0^{3\#}$ is remarkable, and hence the map $\gamma \mapsto c_{\gamma}^{(3)}$ is continuous.
\end{mylemma}
\begin{proof}
  1 $\Rightarrow$ 2: Suppose $\gamma < \bolddelta{3}$ is a limit of limit ordinals. By Lemma~\ref{lem:order_type_realizable}, there exists a $\Pi^1_3$-wellfounded tree $Y$ such that $\gamma = \llbracket ((0)) \rrbracket_Y$.

Case 1: $\cf^{\admistwobold} (\gamma) = \omega$. 

Then $R^{0}$ is a subtree of $Y$ and $A \DEF \set{a \in \omega^{<\omega}}{a <_{BK} ((0)), (a) \in \dom(Y)}$ has limit order type. By indiscernability, $c^{*}_{\Gamma, Y, ((0))} = \sup_{a \in A} c^{*}_{\Gamma, Y, (a)}$. By weak remarkability, for $a \in A$, $\mathcal{M}^{*}_{\Gamma, Y} \res c^{*}_{\Gamma,Y, (a)}= \mathcal{M}^{*}_{\Gamma,Y \res (a)}$. By remarkability, $\pi_{\mathcal{M}_{\Gamma,Y}^{*}, \infty}$ is continuous at $c^{*}_{\Gamma,Y,((0))}$. It follows that $c_{\Gamma,\gamma} = \sup_{\beta<\gamma} c_{\Gamma, \beta}$.

Case 2: $\cf^{\admistwobold} (\gamma) = u_d$, $d \in \se{1,2}$. 

Then $R^{d}$ is a subtree of $Y$. Let $F(\alpha) = \llbracket \mathbf{B}^1 \rrbracket_{Y \otimes T}$ for $\alpha<u_d$ represented by $T$. Then $\sup_{\alpha < u_d} F(\alpha) = \gamma$. By remarkability,  Lemma~\ref{lem:EM_model_ultrapower_equal_normal} and absoluteness, $\sup_{\alpha < u_d} c_{\Gamma, F(\alpha)} = c_{\gamma}$.

2 $\Rightarrow$ 3: Trivial.

3 $\Rightarrow$ 1: 
By Lemma~\ref{lem:order_type_realizable}, there exist $\Pi^1_3$-wellfounded trees $Y^d$ for $d \in \se{0,1,2}$ such that $\gamma_d = \llbracket ((0)) \rrbracket_Y$. Reverse the argument in  1 $\Rightarrow$ 2.
\end{proof}

\begin{mydefinition}[Level $\leq 2$ correctness]
  \label{def:remarkable}
     A level-3 EM blueprint $\Gamma$ is \emph{level $\leq 2$ correct} iff for each finite level-3 tree $Y$, for each $y \in \dom(Y)$, putting 
$X = Y_{\tree}(y)$,  
     $\Gamma(Y)$ contains the following axiom: 

If  $c = \underline{c_y}$, $b = (\underline{\phi_{X,c}^X})^{-1}(c)$, $\xi>c$ is a cardinal and strong cutpoint, then $V^{\coll(\omega,\xi)}$ satisfies the following: 
     \begin{enumerate}
\item \label{item:level_2_ordinal_invariance} If $\vec{\alpha} = (\comp{d}{\alpha}_x)_{(d,x) \in \dom(X)}$ is represented by both $T$ and $T'$, then \\
$(  (\underline{j^T_{X,c}})^K(K|\xi) , ( \underline{j^{X,T}_{X,c}})^K  (b) ) \sim_{DJ} ( (\underline{j^{T'}_{X,c}})^K(K|\xi) , ( \underline{j^{X,T'}_{X,c}})^K  (b)  )$. 
Here
 $(\underline{j^U_{X,c}})^K$ stands for the direct limit of $(\underline{j^{Z,Z'}_{X,c}})^K$ for $Z,Z'$ finite subtrees of $U$, $Z$ a finite subtree of $Z'$, and $(\underline{j^{X,U}_{X,c}})^K$ stands for the tail of the direct limit map from $(\underline{j^X_{X,c}})^K (K)$ to $(\underline{j^U_{X,c}})^K (K)$.
\item \label{item:level_2_ordinal_correct} Let $F(\vec{\alpha}) = \pi_{(\underline{j^T_{X,c}})^K(K|\xi), \infty}( ( \underline{j^{X,T}_{X,c}})^K  (b)) $ for $\vec{\alpha}$ represented by $T$. Then $[F]_{\mu^X} = \pi_{K|\xi , \infty }(c)$.
     \end{enumerate}
\end{mydefinition}

\begin{mylemma}
  \label{lem:EM_model_ultrapower_equal}
  Suppose $\Gamma$ is a level-3 EM blueprint, $Y$ is a finite level-3 tree, $T$ is a finite level $\leq 2$ tree, $y \in \dom(Y)$, $\mathbf{y} = (y,X, \overrightarrow{(e,x,W)}) \in \desc(Y)$, $\mathbf{B} = (\mathbf{y} , \id_{X, *} )\in \desc(Y,X,Q^0)$. Then
  \begin{enumerate}
  \item $      \mathcal{M}_{\Gamma,Y, y }^T  = \mathcal{M}_{\Gamma, Y \otimes_y T}$, $ \phi^T_{\Gamma,Y,y} = (\id_{Y \otimes_y T})^{Y,T}_{\Gamma} $, where $\id_{Y \otimes_y T}$ factors $(Y \otimes_y T, Y,T)$.
  \item If $X$ is a subtree of $T$, then $ j^{X,T}_{\Gamma, Y} \circ   (\phi^X_{\Gamma,Y,y})^{-1} (c_{\Gamma,Y,y})  = c_{\Gamma,Y \otimes_y X, \mathbf{B}} $.
  \end{enumerate}
\end{mylemma}
\begin{proof}
1. 
  Put $c = c_{\Gamma,Y,y}$,  $R = Y \otimes_y T$, $\rho = \id_R$, $\psi = \id_{Y,*}$ factoring $(Y, R)$,  $Y[y] = (X,\overrightarrow{(e,x,W)})$, $\mathbf{y} = (y, X, \overrightarrow{(e,x,W)})$. 
 We have to show that
  \begin{displaymath}
    \ran(\phi^T_{\Gamma,Y,y}) = \ran( \rho_{\Gamma}^{Y,T}).
  \end{displaymath}

The $\subseteq $ direction: If $a \in \mathcal{M}_{\Gamma,Y}$, then $j^T_{\Gamma,Y} (a) = \rho_{\Gamma}^{Y,T} \circ \psi_{\Gamma}^R(a)$. If $Q$ is finite, $\pi$ factors $(X, T \otimes Q)$, then $\pi ' \DEF \id_{T \otimes Q, *} \circ \pi$ factors $(X, (T \otimes Q) \otimes X)$ and $(\mathbf{y}, \pi') \in \dom( Y \otimes (T \otimes Q))$. 
$\iota_{Y,T,Q}^{-1}(\mathbf{y}, \pi')$ is of the form $(\mathbf{B}, \tau)$ where $\mathbf{B} = (\mathbf{y}, \varphi)  \in \dom(Y \otimes_y T)$. Hence,
\begin{displaymath}
    ( \underline{\pi^{T \otimes Q}} ( \underline{c_y})  )^{\mathcal{M}_{\Gamma,Y}}  = c_{\Gamma,Y,(\mathbf{y}, \pi')} = \rho^{Y,T}_{\Gamma}  (   c_{\Gamma,R, (\mathbf{B}, \tau)}).
\end{displaymath}

The $\supseteq$ direction: If $\mathbf{B}  \in \dom(Y \otimes Q^0 )$, then $c^T_{\Gamma,Y, \mathbf{B}} =  j^T_{\Gamma,Y}(   c_{\Gamma,Y,\psi^{-1}(\mathbf{B})})$. If $\mathbf{B} = (\mathbf{y}, \pi) \in \dom(R)$, then $c_{\Gamma,Y, \mathbf{B}}^T  \in \ran( \phi^T_{\Gamma,Y,y})$ by definition.

2. Set $X=T$ in part 1.
\end{proof}

It is straightforward to compute that if $Y,y, \mathbf{y}, \mathbf{B}$ are as in the assumption of Lemma~\ref{lem:EM_model_ultrapower_equal}, then 
\begin{enumerate}
\item if  $\vec{\alpha}  =  (\comp{d}{\alpha}_x)_{(d,x) \in \dom(X)}$ is represented by both $T$ and $T'$, then $\llbracket \mathbf{B} \rrbracket_{Y \otimes_y T} = \llbracket \mathbf{B} \rrbracket_{Y \otimes_y T'}$;
\item letting $G(\vec{\alpha}) = \llbracket \mathbf{B} \rrbracket_{Y \otimes_y T}$ for $\vec{\alpha}$ represented by $T$, then $\llbracket y \rrbracket_Y = [G]_{\mu^X}$.
\end{enumerate}
From Lemmas~\ref{lem:EM_model_ultrapower_equal},~\ref{lem:3sharp_indiscernible_continuous} and absoluteness, we conclude:
\begin{mylemma}
  \label{lem:level_2_correct}
Assume $\boldDelta{2}$-determinacy. Suppose $\Gamma$ is a remarkable iterable level-3 EM blueprint. Then the following are equivalent.
\begin{enumerate}
\item $\Gamma$ is level $\leq 2$ correct.
\item For any potential partial level $\leq 2$ tower $(X,\overrightarrow{(e,x,W)})$ of continuous type, if $F \in (\bolddelta{3})^{(X, \overrightarrow{(e,x,W)}) \uparrow}$, then
  \begin{displaymath}
    c_{\Gamma,X,[F]_{\mu^X}} = [  \vec{\alpha} \mapsto c_{\Gamma,F(\vec{\alpha})}  ]_{\mu^X}.
  \end{displaymath}
\item For any potential partial level $\leq 2$ tower $(X,\overrightarrow{(e,x,W)})$ of continuous type, there exists $F \in (\bolddelta{3})^{(X, \overrightarrow{(e,x,W)}) \uparrow}$ satisfying
  \begin{displaymath}
    c_{\Gamma,X,[F]_{\mu^X}} = [  \vec{\alpha} \mapsto c_{\Gamma,F(\vec{\alpha})}  ]_{\mu^X}.
  \end{displaymath}
\end{enumerate}
In particular, if $\boldpi{3}$-determinacy holds, then $0^{3\#}$ is level $\leq 2$ correct, and hence, if $F \in (\bolddelta{3})^{(X, \overrightarrow{(e,X,W)}) \uparrow}$, then
\begin{displaymath}
c_{X,[F]_{\mu^X}}^{(3)} =  [ \vec{\alpha} \mapsto   c_{F(\vec{\alpha})}^{(3)}]_{\mu^X}.
\end{displaymath}
\end{mylemma}

\begin{mytheorem}
  \label{thm:unique_level_3_EM_blueprint}
Assume $\boldpi{3}$-determinacy. Then $0^{3\#}$ is the unique iterable, remarkable, level $\leq 2$ correct level-3 EM blueprint.
\end{mytheorem}
\begin{proof}
  It remains to show uniqueness. 
  Suppose $\Gamma, \Gamma'$ are both iterable remarkable level-3 EM blueprints. We carry out a ``comparison'' between $\Gamma$ and $\Gamma'$. 
By Corollary~\ref{coro:Delta13_pwo_computable_in_LT2}, the function $\gamma \mapsto (c_{\Gamma,\gamma}, c_{\Gamma',\gamma})$ is $\Sigma^1_4(\Gamma,\Gamma')$ in the codes, and hence belongs to $L[T_3, \Gamma,\Gamma']$. By Lemma~\ref{lem:EM_remarkability}, there is a club $C \in L[T_3, \Gamma,\Gamma']$ such that $\gamma = c_{\Gamma,\gamma} = c_{\Gamma',\gamma}$ for any $\gamma \in C$. By Lemma~\ref{lem:level_2_correct}, if $\gamma \in [C]^{(Q, (d,q,P)) \uparrow}$, then $\gamma =c_{\Gamma,Q, \gamma} = c_{\Gamma' ,Q, \gamma }$. 

 Suppose $R$ is a finite level-3 tree. Let $\vec{\gamma} \in [C]^{R\uparrow}$. 
By Lemma~\ref{lem:order_type_realizable}, we can find a $\Pi^1_3$-wellfounded $Y$ extending 
 $R$ so that $\llbracket \emptyset \rrbracket_Y \in C$ and for any $r \in \dom(R)$, $\gamma_r = \llbracket r \rrbracket_Y$. 
Then $(\mathcal{M}^{*}_{\Gamma,Y})_{\infty} = (\mathcal{M}^{*}_{\Gamma',Y})_{\infty} = M_{2,\infty}^{-} | c^{(3)}_{\llbracket \emptyset \rrbracket_Y+\omega}$ and for any $r \in \dom(R)$, $\pi_{\mathcal{M}^{*}_{\Gamma,Y}, \infty }  ( c^{*}_{\Gamma,Y,r}) = \pi_{\mathcal{M}^{*}_{\Gamma',Y}, \infty }  ( c^{*}_{\Gamma',Y,r})= \gamma_r$. 
This ensures that $(\mathcal{M}^{*}_{\Gamma,Y}; ( c^{*}_{\Gamma,Y,r} )_{r \in \dom(R)})  $ is elementarily equivalent to $ (\mathcal{M}^{*}_{\Gamma',Y}; ( c^{*}_{\Gamma',Y,r})_{r \in \dom(R)})  $. 
Hence, $\Gamma(R) = \Gamma'(R)$. 
\end{proof}

The existence of an iterable, remarkable, level $\leq 2$ correct level-3 EM blueprint is a purely syntactical definition of a large cardinal. The minimum background assumption to make sense of it is $\boldDelta{2}$-determinacy. 
 However, its existence and uniqueness is proved under boldface $\boldpi{3}$-determinacy. It is unclear if the assumption of boldface $\boldpi{3}$-determinacy can be weakened, at least to $\boldDelta{3}$-determinacy$+\Pi^1_3$-determinacy. To draw a complete analogy with the level-1 sharp, one would naturally ask
\begin{myquestion}
 \label{ques:3_sharp_existence}
Assume $\boldDelta{2}$-determinacy. Are the following equivalent?
\begin{enumerate}
\item There is an iterable, remarkable, level $\leq 2$ correct level-3 EM blueprint. 
\item There is an $(\omega_1,\omega_1)$-iterable $M_2^{\#}$.
\item $\Pi^1_3$-determinacy.
\end{enumerate}
\end{myquestion}

\begin{mytheorem}
  \label{thm:existence_3_sharp_equivalence}
  Assume $\boldDelta{2}$-determinacy. If there is an iterable, remarkable, level $\leq 2$ correct level-3 EM blueprint, then $\Pi^1_3$-determinacy holds.
\end{mytheorem}
\begin{proof}
Let $\Gamma$ be an  iterable, remarkable, level $\leq 2$ correct level-3 EM blueprint. 
  Suppose $A\subseteq \mathbb{R}$ is $\Pi^1_3$ and $G$ is the game on $\omega$ with payoff set $A$. Let $(R_s)_{s \in \omega^{<\omega}}$ be an effective regular level-3 system such that $x \in A \eqiv R_x \DEF \cup_{n<\omega} R_{x \res n}$ is $\Pi^1_3$-wellfounded. By iterability of $\Gamma$, $\mathcal{M}^{*}_{\Gamma,R^0}$ is a $\Pi^1_3$-iterable mouse. 
Working in $\mathcal{M}^{*}_{\Gamma,R^0}$, define the auxiliary game $H(\underline{c_{((0))}})$ where in rounds $2n$ and $2n+1$, I plays $x(2n) \in \omega, \gamma_n \in \underline{c_{((0))}}$, II plays $x(2n+1)$. Player I is said to follow the rules at stage $k$ iff letting $r_n \in \dom(R_{x \res n+1}) \setminus \dom(R_{x \res n})$ for $n<k$, then  for any $n<m<k$, $ r_n = (r_m)^{-} \to \gamma_m<  \underline{j^{(R_{x \res n+1})_{\tree}(r_n),(R_{x \res m+1})_{\tree}(r_m)}} (\gamma_n)$. Players I wins $H(\underline{c_{((0))}})$ iff he follows the rules at every finite stage $k$. 
 $H(\underline{c_{((0))}})$  is a closed game for Player I, hence determined in $\mathcal{M}^{*}_{\Gamma,R^0}$.

Case 1: $\mathcal{M}^{*}_{\Gamma,R^0}\models $``$\sigma$ is a winning strategy for Player I in $H(\underline{c_{((0))}})$''.

Let $\sigma^{*}$ be the strategy for Player I in $G$ obtained by following $\sigma$ and ignoring the auxiliary moves $\gamma_n$. If $x$ is (in $V$) a complete run according to $\sigma^{*}$, then $R_x \in p[\underline{S_3}^{\mathcal{M}^{*}_{\Gamma,R^0}}]$. By $\Sigma^1_4$-correctness of set-generic extensions of $\mathcal{M}^{*}_{\Gamma,Y}$,  $p[\underline{S_3}^{\mathcal{M}^{*}_{\Gamma,R^0}}] \subseteq p[S_3]$.
Hence $x \in A$. This shows $\sigma^{*}$ is winning for Player I. 

Case 2: $\mathcal{M}^{*}_{\Gamma,R^0} \models $``$\sigma$ is a winning strategy for Player II in $H(\underline{c_{((0))}})$''.

 We define a strategy $\sigma^{*}$ for II in $G$ as follows: if $\lh(s) = 2n+1$, then $ \sigma( s ) = a $ iff  the formula
\begin{displaymath}
 \sigma( (s(0),  \underline{c_{r_1}}), s(1), \dots, (s(2n), \underline{c_{r_{n}}})) = a
\end{displaymath}
belongs to $  \Gamma(R_{s \res n+1})$, where $r_k \in \dom(R_{s \res k+1}) \setminus \dom(R_{s \res k})$. 
We claim that $\sigma^{*}$ is a winning strategy for Player II. Suppose otherwise and $x$ is a complete run according to $\sigma^{*}$ but $x \in A$. Then $R_x$ is $\Pi^1_3$-wellfounded. 
Let $\mathcal{N} = \mathcal{M}^*(\Gamma, R_x ^{+})$, where $R_x^{+}$ extends $R_x$, $\dom(R_x^{+}) = \dom(R_x) \cup \se{((1))}$, $R_x^{+}((1))$ has degree 0. By iterability of $\Gamma$, $\mathcal{N}$ is a $\Pi^1_3$-iterable mouse.  By coherency of $\Gamma$, $\mathcal{N} \models $``$\sigma$ is a winning strategy for Player II in $(H(\underline{c_{((1))}}))^{\mathcal{N}}$''.
However, $x \oplus (\underline{c_{r_k}}^{\mathcal{N}})_{k<\omega}$ is a complete run according to $\sigma$ which is legal according to the rules of $(H(\underline{c_{((1))}}))^{\mathcal{N}}$. In $V$, the tree of attempts of building a complete run according to $\sigma$ which is legal according to the rules of $(H(\underline{c_{((1))}}))^{\mathcal{N}}$ is illfounded. By absoluteness of wellfoundedness, $\mathcal{N}$ can see such a complete run. Contradiction.
\end{proof}

 Theorem~\ref{thm:existence_3_sharp_equivalence} is a generalization of Martin's theorem that $0^{\#}$ implies $\Pi^1_1$-determinacy.  It proves $1 \Rightarrow 3$ in Question~\ref{ques:3_sharp_existence}. For a real $x$, a level-3 EM blueprint over $x$ is the obvious generalization of Definition~\ref{def:pre_level_3_EM_blueprint}, i.e., a function $\Gamma$ that sends $R$ to $\Gamma(R)$, a complete consistent $\mathcal{L}^{\underline{x},R}$-theory containing the additional axioms ``$\underline{x} \in \mathbb{R}$'' and ``$\underline{x}(i) = j$'' when $x(i) = j$. Assume $\boldpi{3}$-determinacy, $x^{3\#}$ is the unique  iterable, remarkable, level $\leq 2$ correct level-3 EM blueprint over $x$.  Thus, in combination with Neeman \cite{nee_opt_I,nee_opt_II} and Woodin \cite{SUW}, we reach an 
affirmative answer to the boldface version of Question~\ref{ques:3_sharp_existence}. 

\begin{mytheorem}
 \label{thm:3_sharp_existence_boldface}
Assume $\boldDelta{2}$-determinacy. The following are equivalent.
\begin{enumerate}
\item For all $x \in \mathbb{R}$, there is an iterable, remarkable, level $\leq 2$ correct level-3 EM blueprint over $x$.
\item For all $x \in \mathbb{R}$, there is an $(\omega_1,\omega_1)$-iterable $M_2^{\#}(x)$.
\item $\boldpi{3}$-determinacy.
\end{enumerate}
\end{mytheorem}

Recall the basic fact that $L$ is the Skolem hull of the class of Silver indiscernibles. 
We exploit its higher level analog. Clearly, $M_{2,\infty}^{-}$ is not the Skolem hull of $I^{(3)}$, as by remarkability, the Skolem hull contains only countably many ordinals below $c^{(3)}_{\omega}$. 
The missing part will be generated by ordinals below $u_{\omega}$ in a specific way.

\begin{mylemma}
  \label{lem:iterates_dominated_by_level_2_ultrapower}
  Suppose $\mathcal{N}$ is $\Pi^1_3$-iterable and satisfies $0^{3\#}(\emptyset)$. Then for any limit ordinal $\alpha \in \mathcal{N}$,
  \begin{displaymath}
    \pi_{\mathcal{N},\infty}( \alpha ) = \sup \set{\pi_{\mathcal{N}^T, \infty} (\beta) }{ T \text{ is $\Pi^1_2$-wellfounded, } \beta < j^T_{\mathcal{N}}(\alpha)}. 
  \end{displaymath}
\end{mylemma}
\begin{proof}
Let the \emph{universality of level $\leq 2$ ultrapowers axiom} be the following:
\begin{quote}
    If $\alpha$ is a limit ordinal and $\xi > \alpha$ is a cardinal and cutpoint, then $V^{\coll(\omega,\xi)}$ satisfies $\pi_{K | \xi, \infty} (\alpha) = \sup \{ \pi_{(\underline{j^T})^K (K | \xi),\infty} (\beta): T $ is $\Pi^1_2$-wellfounded, $\beta < (\underline{j^T})^K (\alpha)\}$, where $(\underline{j^T})^K$ denotes the direct limit of $(\underline{j^{T'}})^K$ for $T'$ a finite subtree of $T$.
\end{quote}
By elementarity and absoluteness, it suffices to show that $M_{2,\infty}^{-}  $ is a model of this axiom.
Fix $\alpha < \bolddelta{3}$.

  Case 1: $\cf^{\admistwobold}( \alpha ) = \omega$. 

Then $M_{2,\infty}^{-} \models ``\cf(\alpha)$ is not measurable''. So when $\alpha < \xi < \bolddelta{3}$,  $(M_{2,\infty}^{-})^{\coll(\omega,\xi)} \models ``\pi_{K | \xi, \infty}$ is continuous at $\alpha$''.

Case 2:  $\cf^{\admistwobold}( \alpha) = u_1$. 

Let $F : u_1 \to \alpha$ be order preserving and cofinal, $F \in \admistwobold$. Let $z \in \mathbb{R}$ so that $F$ is $\Delta_1$-definable over $\admistwo{z}$ from $\se{T_2,z}$. Let $\mathcal{P} \in \mathcal{F}_{2, z}$ and $\bar{\alpha}, \bar{F} \in \mathcal{P} $ so that $\pi_{\mathcal{P}, \infty}(\bar{\alpha}, \bar{F}) = (\alpha, F)$. Let $\mathcal{Q} = (L[ \underline{S_3}])^{\mathcal{P}}$. So for any $\Pi^1_2$-wellfounded $T$, $\mathcal{Q}^T =  ( L [ \underline{S_3}])^{\mathcal{P}^T}$ and $j^T_{\mathcal{Q}} = j^T_{\mathcal{P}} \res \mathcal{Q}^T$. By absoluteness, it suffices to show that
\begin{displaymath}
  \alpha = \sup \set{ \pi_{\mathcal{P}^T, \infty} (\beta) }{ T \text{ is $\Pi^1_2$-wellfounded, } \beta < j^T_{\mathcal{P}}(\alpha)}.
\end{displaymath}
This would follow from
\begin{displaymath}
  u_1 = \sup \set{ \pi_{\mathcal{P}^T, \infty} (\beta) }{ T \text{ is $\Pi^1_2$-wellfounded, } \beta < (\underline{u_1})^{\mathcal{P}^T}}.
\end{displaymath}
The last equality is because $ \{ \pi_{\mathcal{P}^T,\infty}  \circ j^{T',T}_{\mathcal{P}}(   (\underline{\seed^{T'}_{(1,(0))}})^{\mathcal{P}}) : T$ is $\Pi^1_2$-wellfounded, $T'$ is a finite subtree of $T$, $(0) \in \comp{1}{T}'\}$ is a subset of the right hand side and has order type $\omega_1$. 

Case 3: $\cf^{\admistwobold}( \alpha) = u_2$.

Similar to Case 2.
\end{proof}

\begin{mydefinition}
  \label{def:g_N}
  If $\mathcal{N}$ is a structure that satisfies Axioms~\ref{item:EM_ZFC}-\ref{item:level-2-embedding_invariance} in Definition~\ref{def:pre_level_3_EM_blueprint} and the universality of level $\leq 2$ ultrapowers axiom, then
  \begin{displaymath}
    \mathcal{G}_{\mathcal{N}}
  \end{displaymath}
is the direct system consisting of models $\mathcal{N}^T$ for which $T$ is a $\Pi^1_2$-wellfounded level $\leq 2$ tree and maps $\pi^{T,T'}_{\mathcal{N}} : \mathcal{N}^T \to \mathcal{N}^{T'}$ for $\pi$ minimally factoring $T,T'$. Define
\begin{align*}
  \mathcal{N}_{\infty}& = \dirlim \mathcal{G}_{\mathcal{N}},\\
\pi_{\mathcal{N}^T,\mathcal{N}_{\infty}}&: \mathcal{N}^T \to \mathcal{N}_{\infty} \text{ is tail of the direct limit map.}
\end{align*}
\end{mydefinition}

If in addition, $\mathcal{N}$ is countable $\Pi^1_3$-iterable mouse, then  $\mathcal{G}_{\mathcal{N}}$ is a subsystem of $\mathcal{I}_{\mathcal{N}}$.
By Lemma~\ref{lem:iterates_dominated_by_level_2_ultrapower}, $\mathcal{G}_{\mathcal{N}}$ is dense in $\mathcal{I}_{\mathcal{N}}$, so there is no ambiguity in the notation $\mathcal{N}_{\infty}$: 

\begin{mylemma}
  \label{lem:EM_model_ultrapower_dense_in_iterates}
Suppose $\mathcal{N}$ is  a countable $\Pi^1_3$-iterable mouse and 
satisfies Axioms~\ref{item:EM_ZFC}-\ref{item:level-2-embedding_invariance} in Definition~\ref{def:pre_level_3_EM_blueprint} and the universality of level $\leq 2$ ultrapowers axiom.
If $ \pi : \mathcal{N} \to \mathcal{P}$ is an iteration map, then there exist a $\Pi^1_2$-wellfounded $T$ and $\psi : \mathcal{P} \to \mathcal{N}^T$ such that $\psi \circ \pi = j^T_{\mathcal{N}}$ and $\psi$ is essentially an iteration map. 
\end{mylemma}

The direct system $\mathcal{G}_{\mathcal{N}}$ is useful even when $\mathcal{N}$ is not $\Pi^1_3$-iterable. In the proof of the level-4 Kechris-Martin theorem in Section~\ref{sec:level-4-sharp}, we will inevitably have to deal with partially iterable level-3 EM blueprints. The structure $\mathcal{N}$ will be the EM model built from a partially iterable level-3 EM blueprint. The advantage of the (possibly illfounded) direct limit $\mathcal{N}_{\infty}$ is that the order type of its ordinals is easily codable by a subset of $u_{\omega}$.  If $X$ is a finite level $\leq 2$ tree, $a \in{\mathcal{N}}$, $\vec{\beta} = (\comp{d}{\beta}_x)_{(d,x) \in \dom(X)}$ is represented by both $T$ and $T'$, then $\pi_{\mathcal{N}^T,\infty}  \circ j^{X,T}_{\mathcal{N}}(a) =\pi_{\mathcal{N}^{T'},\infty}  \circ j^{X,T'}_{\mathcal{N}}(a) $.  We can define 
\begin{displaymath}
  \pi_{\mathcal{N}, X, \vec{\beta}, \infty} (a) = \pi_{\mathcal{N}^T, \infty} \circ j^{X,T}_{\mathcal{N}} (a)
\end{displaymath}
for $\vec{\beta}$ represented by $T$. So
\begin{displaymath}
  \mathcal{N}_{\infty} = \set{\pi_{\mathcal{N},X,\vec{\beta},\infty}(a)}{a \in \mathcal{N},X \text{ finite level $\leq 2$ tree}, \vec{\beta}\in [\omega_1]^{X \uparrow}} .
\end{displaymath}
Essentially, the inner model theoretic comparison between mice is replaced by the comparison between $\Pi^1_2$-wellfounded level $\leq 2$ trees in Theorem~\ref{thm:factor_ordertype_embed_equivalent_lv2}.

A \emph{level $\leq 3$ code} for an ordinal in $\bolddelta{3}$ is of the form
\begin{displaymath}
  (R, \vec{\gamma}, X, \vec{\beta}, 
  \gcode{\sigma})
\end{displaymath}
such that 
$R$ is a finite level-3 tree, $\vec{\gamma}$ respects $R$, $X$ is a finite level $\leq 2$ tree, $\vec{\beta}$
respects $X$, and  $\sigma$ is an $\mathcal{L}^{R}$-Skolem term for an ordinal. It codes the ordinal
\begin{displaymath}
      \sharpcode{ (R, \vec{\gamma}, X, \vec{\beta}, \gcode{\sigma})}  =
  \pi_{\mathcal{M}^{*}_{0^{3\#},R}, X, \vec{\beta}, \infty}    
(
 \sigma^{\mathcal{M}^{*}_{0^{3\#},R}} ( (\underline{c_r})_{r \in \dom(R)})).
\end{displaymath}
By Lemmas~\ref{lem:EM_model_ultrapower_dense_in_iterates},  
every ordinal in $\bolddelta{3}$ has a level $\leq 3$ code. The evaluation function on level $\leq 3$ codes is $\Sigma^1_4(0^{3\#})$, and hence definable over $M_{2,\infty}^{-}(0^{3\#})$. 

\subsection{Level-3 indiscernibles}
\label{sec:level-3-indisc}

If $\vec{\gamma}$ respects a level-3 tree $R$, define
\begin{displaymath}
  c_{\vec{\gamma}} = (c_{R_{\tree}(r),\gamma_r}^{(3)})_{r \in \dom(R)}
\end{displaymath}
which strongly respects $R$. Combined with Lemma~\ref{lem:gamma_r_order}, this leads to the order of level-3 indiscernibles for $M_{2,\infty}^{-}$: $c^{(3)}_{Q, \gamma} < c^{(3)}_{Q',\gamma'}$ iff letting $(\gamma_i)_{1 \leq i \leq k}$ be the $Q$-approximation sequence of $\gamma$ and $(\gamma'_i)_{1 \leq i \leq k'}$ be the $Q'$-approximation sequence of $\gamma'$, then $(\gamma_i)_{1 \leq i \leq k} \concat (-1) < _{BK} (\gamma_i')_{1 \leq i \leq k'} \concat (-1)$.
We prove the general remarkability property of $0^{3\#}$ based on this order.

\begin{mylemma}[General remarkability]
  \label{lem:remarkable_general}
  Suppose $\vec{\gamma}$ and $\vec{\gamma}'$ both respect a finite level-3 tree $R$. Suppose $r \in \dom(R)$ and for any $s \prec^R r$, $\gamma_s = \gamma'_s$. Then for any $\mathcal{L}$-Skolem term $\tau$,
  \begin{displaymath}
    M_{2,\infty}^{-} \models \tau(c_{\vec{\gamma}}) < c_{R_{\tree}(r), \gamma_r}^{(3)} \to \tau(c_{\vec{\gamma}}) = \tau(c_{\vec{\gamma}'}).
  \end{displaymath}
\end{mylemma}
\begin{proof}
Assume $\tau^{M_{2,\infty}^{-}}(c_{\vec{\gamma}}) < c^{(3)}_{R_{\tree}(r), \gamma_r}$.
If $s \in \dom(R)$ and $r \preceq^R s$ let $l(s)$ be the largest $l$ so that $\corner{r \res l}^R = \corner{s \res l}^R$. It is easy to find $\vec{\gamma}^l \in [\bolddelta{3}]^{R \uparrow}$ for $l \leq \lh( r )$ so that $\vec{\gamma}^0 = \vec{\gamma}$, $\vec{\gamma}^{\lh(r)}  =  \vec{\gamma}'$ and $\gamma^{l}_s \neq \gamma^{l+1}_s \to (r \preceq^R s \wedge  l(s) = l)$.
Thus, we may assume a fixed $l_0$ so that $\gamma_s \neq \gamma'_s$ implies $l(s) = l_0$. The case $l_0 = 0$ is just Lemma~\ref{lem:weakly_remarkable_strengthening}. Assume now $l_0 > 0$. 
 Note that $l(s) = l_0$ also implies that $R_{\tree}(s \res l_0) = R_{\tree} (r \res l_0)$.  A sliding argument similar to Lemma~\ref{lem:weakly_remarkable_strengthening} reduces to the special case that $(\lh(s) = \lh(s') =   l_0 +1 \wedge l(s)=l(s')  = l_0) \to \gamma_s < \gamma_{s'}$. Let $(Y,\rho,\rho')$ be the amalgamation obtained by Lemma~\ref{lem:level_3_amalgamation} so that $\rho,\rho'$ both factor $(R,Y)$ and if $\vec{\delta}, \vec{\delta}' \in [\bolddelta{3}]^{R \uparrow}$ and $(\lh(s) = \lh(s') =   l_0 +1 \wedge l(s)=l(s')  = l_0) \to \delta_s < \delta_{s'}$, then $\vec{\delta} \oplus \vec{\delta}' \in [\bolddelta{3}]^{Y \uparrow} $, where $\vec{\delta} \oplus \vec{\delta}' = \vec{\epsilon}$, $ \epsilon_{\rho(r)} = \delta_r$, $\epsilon_{\rho'(r)} = \delta_r'$.
Put $\eta \in D$ iff $c_{\eta} ^{(3)}= \eta$. By indiscernability, we may assume that $\vec{\gamma}, \vec{\gamma}' \in [D]^{R \uparrow}$, so that $c_{\vec{\gamma}} = \vec{\gamma}$, $c_{\vec{\gamma}'} = \vec{\gamma}'$. It is easy to construct  $\vec{\delta}^{\xi} \in [\bolddelta{3}]^{R \uparrow}$ for $\xi < \gamma_r$ so that $\vec{\gamma} \oplus \vec{\delta}^{\xi} \in [\bolddelta{3}]^{Y \uparrow}$ and $\eta < \xi < \gamma_r \to \vec{\delta}^{\eta} \oplus \vec{\delta}^{\xi} \in [\bolddelta{3}]^{Y \uparrow}$. Put $\epsilon^{\xi} = \tau^{M_{2,\infty}^{-}} (c_{\vec{\delta}^{\xi}})$. By indiscernability, it suffices to show that $\epsilon^{\eta} = \epsilon^{\xi}$ for some (or equivalently, for any) $\eta < \xi < \gamma_r$. Suppose otherwise. By indiscernability again, either $\eta < \xi < \gamma_r \to \epsilon^{\eta} > \epsilon^{\xi} $ or $\eta < \xi < \gamma_r \to \epsilon^{\eta} < \epsilon^{\xi} $. The former gives a descending chain of ordinals.  The latter implies that $\gamma_r \leq \tau^{M_{2,\infty}^{-}}(c_{\vec{\gamma}})$, contracting to our assumption.
\end{proof}

Recall that if $c<c'$ are consecutive $L[x]$-indiscernibles, then $L[x^{\#}] \models c'<c^{+}$. The level-3 version is similar. 

\begin{mylemma}
  \label{lem:next-level-3-indis}
Assume $\boldpi{3}$-determinacy. For any $c_{\omega}^{(3)} <\xi \in I^{(3)}$, there is an $\mathcal{L}$-Skolem term $\tau$ such that 
 $M_{2,\infty}^{-}(0^{3\#}) \models `` \tau(\sup(I^{(3)} \cap \xi) , \cdot)$ is a surjection from $ \sup (I^{(3)} \cap \xi)$ onto $\xi$''. For any $\xi < c^{(3)}_{\omega}$, there is an $\mathcal{L}$-Skolem term $\tau$ such that 
 $M_{2,\infty}^{-}(0^{3\#}) \models `` \tau(\underline{u_{\omega}} , \cdot)$ is a surjection from $u_{\omega}$ onto $\xi$''. 
\end{mylemma}
\begin{proof}
The evaluation function on level $\leq 3$ codes is $\Sigma^1_4(0^{3\#})$, and hence is definable over $M_{2,\infty}^{-}(0^{3\#})$. 
If $(R, \vec{\gamma}, X, \vec{\beta} , \gcode{\sigma})$ and $(R, \vec{\gamma}', X, \vec{\beta} , \gcode{\sigma})$ are both level $\leq 3$ codes an ordinal below $\xi$ and $\forall r ~ (\gamma_r < \xi \to \gamma_r = \gamma'_r)$, then by Lemma~\ref{lem:remarkable_general} they must code the same ordinal. This easily defines a surjection from $\sup(I^{(3)} \cap \xi)$ onto $\xi$ in $M_{2,\infty}^{-}(0^{3\#})$ when $\xi > c^{(3)}_{\omega}$, from $u_{\omega}$ onto $\xi$ when $\xi < c^{(3)}_{\omega}$.
\end{proof}

By Lemma~\ref{lem:gamma_r_order}, for any finite level-3 tree $R$,  if $\mathbf{A} \prec \mathbf{A}'$  then ``$\underline{c_{\mathbf{A}}} < \underline{c_{\mathbf{A}'}}$'' is true in $0^{3\#}(R)$,  if $\mathbf{A} \sim \mathbf{A}'$  then ``$\underline{c_{\mathbf{A}}} = \underline{c_{\mathbf{A}'}}$'' is true in $0^{3\#}(R)$. 
For notational convenience, if $X$ is a finite level $\leq 2$ tree and $\gamma = [F]_{\mu^X}$ is a limit ordinal, define $ c^{(3)} _{X, \gamma} = [\vec{\alpha} \mapsto c^{(3)}_{F(\vec{\alpha})}]_{\mu^X}$; define $c^{(3)}_{\emptyset, \bolddelta{3}} = \bolddelta{3}$. Ordinals of the form  $c^{(3)}_{X, \gamma}$ when $ X \neq \emptyset$  are definable from elements in $ I^{(3)}$ over $M_{2,\infty}^{-}$:
 If the $X$-approximation sequence of $\gamma$ is $(\gamma_i)_{1 \leq i \leq k}$, $(Q, (d_i,q_i,P_i)_{i < \lh(\vec{q})})$ is the $X$-potential partial level $\leq 2$ tower induced by $\gamma$, 
$\pi : Q \to X$ is the induced level-2 factoring map, then
\begin{enumerate}
\item if $\gamma$ is of $X$-discontinuous type, then  $c^{(3)}_{X, \gamma} = \pi^X (c^{(3)}_{Q, \gamma_k})$;
\item if $\gamma$ is of $X$-continuous type, $Q^{-} $ is the subtree of $Q$ obtained by removing $(d_k, q_k)$, then $c^{(3)}_{X, \gamma} = \pi^X \circ j^{Q^{-},Q}_{\sup} (c^{(3)}_{Q, \gamma_{k-1}})$. 
\end{enumerate}

Define $\bar{I}^{(3)}=$the closure of $I^{(3)}$ under the order topology.
 Every ordinal in $\bar{I}^{(3)}$ is of the form $c^{(3)}_{X, \gamma}$ where $X$ is finite and $\gamma < \bolddelta{3}$ is a limit. 
Thus, if $\mathbf{A} = (\mathbf{r}, \pi,T) \in \exexdesc(R)$ and $\vec{\gamma} $ strongly respects $R$, then $c^{(3)}_{T, \gamma_{\mathbf{A}}} \in \bar{I}^{(3)}$ and is a limit point of $I^{(3)}$.

Given $\gamma_0,\dots,\gamma_n,\gamma_0',\dots \gamma_n' \in I^{(3)}$, 
 $\vec{\gamma}$ is a \emph{shift} of $\vec{\gamma}'$ iff there exist a level-3 tree $R$, nodes $r_0,\dots,r_n \in \dom(R)$, $\vec{\delta}, \vec{\delta}' $ both respecting $R$ such that $\gamma_i = c^{(3)}_{R_{\tree}(r_i), \delta_i}$, $\gamma_i '= c^{(3)}_{R_{\tree}(r_i), \delta'_i}$ for any $i \leq n$. By indiscernability, if $\vec{\gamma}$ is a shift of $\vec{\gamma}'$, then for any $\mathcal{L}$-formula $\varphi$,
 \begin{displaymath}
   M_{2,\infty}^{-} \models \varphi (\vec{\gamma})  \eqiv \varphi(\vec{\gamma}').
 \end{displaymath}

\begin{mylemma}
  \label{lem:R_indiscernible_description_next}
Suppose $R$ is a finite level-3 tree,  $\tau$ is an $\mathcal{L}$-Skolem term, $\vec{\gamma} $ strongly respects $R$. Suppose $\mathbf{A} = (\mathbf{r}, \pi, T)\in \exexdesc(R)$. Then
\begin{multline*}
      \tau^{M_{2,\infty}^{-}} (c_{\vec{\gamma}}) <  c^{(3)}_{T, \gamma_{\mathbf{A}}} \to \\ \tau^{M_{2,\infty}^{-}} (c^{(3)}_{\vec{\gamma}}) < \min (  I^{(3)}  \setminus \sup
\set{c^{(3)}_{T', \gamma_{\mathbf{A}'}}}{\mathbf{A}' = (\mathbf{r}',\pi',T') \prec^R_{*} \mathbf{A}}).
\end{multline*}
\end{mylemma}
\begin{proof}
Suppose $\tau^{M_{2,\infty}^{-}}(c_{\vec{\gamma}}) < c^{(3)}_{T, \gamma_{\mathbf{A}}}$.
Let $\delta = \min (  I^{(3)}  \setminus \sup
\set{c^{(3)}_{T', \gamma_{\mathbf{A}'}}}{\mathbf{A}' \prec^R_{*} \mathbf{A}})$. We shall show that  $\{\delta': c_{\vec{\gamma}} \concat (\delta')$ is a shift of $c_{\vec{\gamma}} \concat (\delta)\}$ is cofinal in $\gamma_{\mathbf{A}}$. 
From this and indiscernability, $\tau^{M_{2,\infty}^{-}}(c_{\vec{\gamma}}) < \delta$.

If $\mathbf{r} = \emptyset$, then $\delta = c^{(3)}_{\gamma_{(a)}+\omega}$ where $a = \max_{<_{BK}} R \se{\emptyset}$.  
So $\{\delta': c_{\vec{\gamma}} \concat (\delta')$ is a shift of $c_{\vec{\gamma}} \concat (\delta)\}$ is cofinal in $\gamma_{\mathbf{A}}=\bolddelta{3}$. 

Suppose now $\mathbf{r} = (r, Q, \overrightarrow{(d,q,P)})\neq \emptyset$, $\lh(r) = k$, $\ucf(R(r)) = (d^{*},\mathbf{q}^{*})$, and if $d^{*}=1$ put $q^{*} = \mathbf{q}^{*}$,  
if $d^{*}=2$ put $\mathbf{q}^{*} = (q^{*}, P^{*}, \vec{p}^{*})$. 

Case 1: $\mathbf{r}$ is of discontinuous type,  $\corner{\mathbf{A}}$ ends with $-1$. 

Let $s = \max_{<_{BK}}R \se{r,-}$ and $\mathbf{s} = (s, Q, \overrightarrow{(d,q,P)})$. Then $\delta = c^{(3)}_{T, \pi^T(\gamma_{\mathbf{s}})+\omega}$. 
It is easy to compute that   $\{\delta': c_{\vec{\gamma}} \concat (\delta')$ is a shift of $c_{\vec{\gamma}} \concat (\delta)\}$ is cofinal in $\gamma_{\mathbf{A}}$. 

Case 2: $\mathbf{r}$ is of discontinuous type, $\corner{\mathbf{A}}$ ends with an ordinal. 

If either $d_k=1$ or $d_k = d^{*} = 2 \wedge \mathbf{q}^{*} \in \desc(\comp{2}{Q})$, let $\tau$ factor $(Q,T)$ so that $\tau$ agrees with $\pi$ on $\dom(Q) \setminus \se{(d^{*},q^{*})}$, $\tau(d^{*},q^{*}) = \pred(\pi,T, (d^{*},\mathbf{q}^{*}))$. 
Then $\delta = c^{(3)}_{T, \tau^T(\gamma_r) + \omega}$. 
Otherwise, let $U$ be the subtree of $T$ obtained by removing $\pred(\pi,T, (d^{*}, \mathbf{q}^{*}))$ from its domain. Then $\delta = c^{(3)}_{U, \pi^U(\gamma_r) +\omega}$. 
In either case, $\{\delta': c_{\vec{\gamma}} \concat (\delta')$ is a shift of $c_{\vec{\gamma}} \concat (\delta)\}$ is cofinal in $\gamma_{\mathbf{A}}$. 

Case 3: $\mathbf{r}$ is of continuous type. 

Similar to Cases 1 and 2. 
\end{proof}

\begin{mylemma}
  \label{lem:level_3_indis_cardinal}
  Suppose $R$ is a finite level-3 tree and $\mathbf{A}  = (\mathbf{r}, \pi, T)\in \exexdesc(R)$, $\mathbf{r} \neq \emptyset$, $\vec{\gamma}$ strongly respects $R$. Then $c^{(3)}_{T,\gamma_{\mathbf{A}}}$ is a cardinal in $M_{2,\infty}^{-}$.
\end{mylemma}
\begin{proof}
  Otherwise, $\tau ^{M_{2,\infty}^{-}}( c_{\vec{\gamma}}) $ is a wellordering on $\alpha \DEF \card^{M_{2,\infty}^{-}}(c^{(3)}_{T,\gamma_{\mathbf{A}}})$ of order type $c^{(3)}_{T,\gamma_{\mathbf{A}}}$ and $\alpha < c^{(3)}_{T,\gamma_{\mathbf{A}}}$. Put $\beta =  \min (  I^{(3)}  \setminus \sup
\set{c^{(3)}_{T', \gamma_{\mathbf{A}'}}}{\mathbf{A}' = (\mathbf{r}',\pi',T') \prec^R_{*} \mathbf{A}}) $. 
By Lemma~\ref{lem:R_indiscernible_description_next}, $\alpha <\beta$. By Lemma~\ref{lem:remarkable_general}, if $\vec{\delta}$ respects $R$ and $\forall s ~(\delta_s < \beta \to \delta_s = \gamma_s)$, then $\tau^{M_{2,\infty}^{-}}(c_{\vec{\gamma}}) = \tau^{M_{2,\infty}^{-}}(c_{\vec{\delta}})$, and hence $\gamma_{\mathbf{A}} = \delta_{\mathbf{A}}$. 
However, it is easy to find such $\vec{\delta}$ satisfying $\delta_{\mathbf{A}} > \gamma_{\mathbf{A}}$.
\end{proof}

\section{The boldface level-3 sharp}
\label{sec:boldface-level-3_sharp}

From now on, we assume $\boldpi{3}$-determinacy. Recall that $\mathbb{L}[T_3] = \bigcup_{x \in \mathbb{R}} L[T_3, x]$. Every subset of $\bolddelta{3}$ in $\mathbb{L}[T_3]$ is definable over ${M_{2,\infty}^{-}(x)}$ for some $x \in \mathbb{R}$.
 All the results in Section~\ref{sec:03} relativize to any given real $x$.  If $R$ is a $\Pi^1_3$-wellfounded level-3 tree, $\mathcal{M}_{x^{3\#},R}$ is the EM model built from $x^{3\#}(R)$. $\mathcal{M}^{*}_{x^{3\#},R}$, $\mathcal{M}^{*,T}_{x^{3\#},R}$,  $c_{x,Q,\gamma}^{(3)}$, $c^{(3)}_{x,\gamma}$, $c_{x,\vec{\gamma}}$, $I^{(3)}_x$,  $\bar{I}^{(3)}_x$ have obvious meanings. Fixing $x$, the function $( Q, \gamma) \mapsto c^{(3)}_{x,Q,\gamma}$ is $\Sigma^1_4(x^{3\#})$ in the codes and hence is definable over $M_{2,\infty}^{-}(x^{3\#})$.

\subsection{Homogeneity properties of $S_3$}
\label{sec:homog-prop-s_3}

For a level $\leq 2$ tree $Q$, $j^Q$ was defined as the ultrapower map on $\mathbb{L}_{\bolddelta{3}}[T_2]$. 
The following lemma allow $j^Q$ to act on $\mathbb{L}[T_3]$ for a level $\leq 2$ tree $Q$. 
\begin{mylemma}
  \label{lem:LT2_LT3_agree_on_power_set_of_u_omega}
  Assume $\boldpi{3}$-determinacy.
  \begin{enumerate}
  \item  $\admistwobold = \mathbb{L}_{\bolddelta{3}}[T_3]$.
  \item Every subset of $\bolddelta{3}$ in $\mathbb{L}_{\bolddelta{3}}[T_3]$ is definable over $M_{2,\infty}^{-}(x)$ from $\se{x}$ for some $x \in \mathbb{R}$. 
  \end{enumerate}
\end{mylemma}
\begin{proof}
1. 
  $T_2$ is a $\Delta^1_3$ subset of $u_{\omega}$, and of course a $\Sigma^1_4$ subset of $\bolddelta{3}$, and hence $T_2 \in L[T_3]$ by Becker-Kechris \cite{becker_kechris_1984}. This gives the $\subseteq$ inclusion. 

If $A \in  L_{\bolddelta{3}}[T_3,x]$, by Theorem~\ref{thm:steel}, there must be $\xi < \bolddelta{3}$ such that $A \in L_{\xi}[T_3\res \xi,x]$.  Pick $y \geq_T x$ such that $\kappa_3^y > \xi$. Then $A \in \admistwo{y}$ by Lemma~\ref{lem:level_2_embedding_bounded_by_delta13}. This gives the $\supseteq$ inclusion.

2. By Theorem~\ref{thm:steel}, every subset of $\bolddelta{3}$ in $\mathbb{L}_{\bolddelta{3}}[T_3]$ is in $M_{2,\infty}^{\#}(x)$ for some $x \in \mathbb{R}$.  If $A\subseteq \bolddelta{3}$ is definable over $M_{2,\infty}^{\#}(x)$ from $\se{\gamma,x}$, $\gamma < \bolddelta{3}$, letting $y \geq_T M_2^{\#}(x)$ such that $\gamma$ is definable over $\admistwo{y}$, then $A$ is definable over $M_{2,\infty}^{-}(y)$ from $\se{y}$.
\end{proof}

Caution that Lemma~\ref{lem:LT2_LT3_agree_on_power_set_of_u_omega} does not give a real $x$ for which $T_3 \in L[T_2,x]$.  $\admistwo{x}$ computes a proper initial segment of $T_3$, and by varying $x$, these proper initial segments are cofinal in $T_3$. However, there is not a single $x$ with $T_3 \in L[T_2,x]$.

A \emph{level $\leq 3$ tree} is of the form $R = (\comp{1}{R}, \comp{2}{R},\comp{3}{R})$ so that $\comp{\leq 2}{R} \DEF (\comp{1}{R},\comp{2}{R})$ is a level $\leq 2$ tree and $\comp{3}{R}$ is a level-3 tree. If $T$ is a level $\leq 2$ tree and $Y$ is a level-3 tree then $T \oplus Y$ denotes the level $\leq 3$ tree $(\comp{1}{T}, \comp{2}{T}, Y)$. $R $ is $\Pi^1_3$-wellfounded iff $\comp{\leq 2}{R}$ is $\Pi^1_2$-wellfounded and $\comp{3}{R}$ is $\Pi^1_3$-wellfounded. 
 
Suppose  $R$ is a level $\leq 3$-tree.  Define $\dom(R) = \cup_d \se{d} \times \dom(\comp{d}{R})$,  $\desc(R) = \cup_d \se{d} \times \desc(\comp{d}{R})$. 
If $\vec{\beta}$ respects $\comp{\leq 2}{R}$ and $\vec{\gamma}$ respects $\comp{3}{R}$, define $\vec{\beta} \oplus \vec{\gamma} = \vec{\delta} = (\comp{d}{\delta}_t)_{(d,r) \in \dom(R)}$ where $\comp{d}{\delta}_r = \comp{d}{\beta}_r$ for $(d,r) \in \dom(\comp{\leq 2}{R})$ and $\comp{3}{\delta}_r = \gamma_r $ for $r \in \dom(\comp{3}{R})$.
Define $A \oplus B = \set{\vec{\beta} \oplus \vec{\gamma}}{\vec{\beta} \in A, \vec{\gamma} \in B}$. 
If $E \subseteq \omega_1$ and $C \subseteq \bolddelta{3}$, define
$[E,C]^{R \uparrow } = [E]^{\comp{\leq 2}{R} \uparrow} \oplus [C]^{\comp{3}{R} \uparrow}$.  $\vec{\delta}$ respects $R$ iff $\vec{\delta} \in [\omega_1, \bolddelta{3}]^{R \uparrow}$.
A finite level $\leq 3$ tree $R$  induces a filter $\mu^R$ on finite tuples in $\bolddelta{3}$, originated from the weak partition property of $\bolddelta{3}$ under AD. $\mu^R$ is the higher level analog of the $n$-fold product of the club filter on $\omega_1$.

 \begin{mydefinition}\label{def:level-3-measure}
Assume $\boldpi{3}$-determinacy. Let $R$ be a finite level $\leq 3$ tree. We say
\begin{displaymath}
A \in \mu^{ R}
\end{displaymath}
iff there are clubs $E \subseteq \omega_1$, $C\subseteq \bolddelta{3}$ such that $E, C \in \mathbb{L}[T_3]$ and
 \begin{displaymath}
[E, C]^{R \uparrow} \subseteq A.
\end{displaymath}
If $Y$ is a finite level-3 tree, put $A \in \mu^Y$ iff $[\omega_1]^{Q^0 \uparrow} \oplus  A \in \mu^{Q^0 \oplus Y}$.
 \end{mydefinition}
$\mu^R$ is an $\mathbb{L}[T_3]$-measure, the reason being as follows. Every $A \in \mathbb{L}[T_3]$ is definable over $M_{2,\infty}^{-}(x)$ from $\se{x}$ for some real $x$.  By indiscernability and remarkability, the section  $A^{*} \DEF \set{\vec{\beta}}{\vec{\beta} \oplus c^{(3)}_{x,\vec{\gamma}} \in A}$ is invariant in $\vec{\gamma} \in [\bolddelta{3}]^{R \uparrow}$. So $C=\set{c^{(3)}_{x,\xi}}{\xi < \bolddelta{3}}$ and some $E$ deciding the $\mu^{\comp{\leq 2}{R}}$-measure of $A^{*}$ works.
$\mu^R$ is the product measure on $\mathbb{L}[T_3]$ of $\mu^{\comp{\leq 2}{R}}$ and $\mu^{\comp{3}{R}}$.  
 Let $j^{R}=j^{\mu^{R}}_{\mathbb{L}[T_3]}$ be the ultrapower map from $\mathbb{L}[T_3]$ to $\mathbb{L}[j^{R}(T_3)]$. For any $x \in \mathbb{R}$, $j^{R}$ is elementary from $L[T_3,x]$ to $L[j^{R}(T_3), x]$. 
By indiscernability and remarkability again, if $\alpha < \bolddelta{3}$ and $F : [\omega_1,\bolddelta{3}]^{R \uparrow} \to \alpha$, $F \in \mathbb{L}[T_3]$, then there is $x \in \mathbb{R}$ and $G\in \mathbb{L}_{\bolddelta{3}}[T_3]$.  such that for any $\vec{\gamma} \in [\bolddelta{3}]^{\comp{3}{R}\uparrow}$, for any $\vec{\beta} \in [\omega_1]^{\comp{\leq 2}{R} \uparrow}$, $F(\vec{\beta} \oplus  \vec{\gamma}) = G(\vec{\beta})$. By Lemma~\ref{lem:LT2_LT3_agree_on_power_set_of_u_omega}, $G\in \admistwobold$. Therefore,
\begin{displaymath}
j^R \res \bolddelta{3} = j^{\comp{\leq 2}{R}} \res \bolddelta{3}.
\end{displaymath}
 For $(d,r) \in \dom(R) \cup \se{(3,\emptyset)}$, let
 \begin{displaymath}
\seed^R_{(d,r)}
\end{displaymath}
be the element represented modulo $\mu^R$ by the projection map $\vec{\gamma} \mapsto \comp{d}{\gamma}_r$. If $R$ is $\Pi^1_3$-wellfounded, the direct limit of $j^{R',R''}$ for $R'$ a finite subtree of $R''$ and $R''$ a finite subtree of $R$ is wellfounded, and we let $j^R: \mathbb{L}[T_3] \to \mathbb{L}[j^R(T_3)]$ be the direct limit map; if $R'$ is a finite subtree of $R$ then $j^{R',R}: \mathbb{L}[j^{R'}(T_3)] \to \mathbb{L}[j^R(T_3)]$ is the tail of the direct limit map. If $(d,r) \in \dom(R')$, $R'$ is a finite subtree of $R$, then $\seed^R_{(d,r)} = j^{R',R} (\seed^{R'}_{(d,r)})$. Let
\begin{displaymath}
  \seed^R = (\seed^R_{(d,r)})_{(d,r) \in \dom(R)}. 
\end{displaymath}
If $Y$ is a $\Pi^1_3$-wellfounded level-3 tree, then $j^Y = j^{Q^0 \oplus Y }$, $\seed^Y_y = \seed^{Q^0 \oplus Y}_{(3,y)}$, $\seed^Y = (\seed^Y_y)_{y \in \dom(Y)}$. 

In particular, $\seed^R_{(3,\emptyset)} = j^R(\bolddelta{3})$, $\seed^R_{(d,r)} = \seed^{\comp{\leq 2}{R}}_{(d,r)}$ 
 when $d \leq 2$. 
If $\mathbf{A} \in \exexdesc(\comp{3}{R}')$, $\mathbf{r} \in \exdesc(\comp{3}{R}')$, $R'$ finite subtree of $R$, let $\seed^R_{(3, \mathbf{A})} = j^{R',R}( [\vec{\gamma} \mapsto \comp{3}{\gamma}_{\mathbf{A}}]_{\mu^{R'}})$,  $\seed^R_{(3, \mathbf{r})} = j^{R',R}( [\vec{\gamma} \mapsto \comp{3}{\gamma}_{\mathbf{r}}]_{\mu^{R'}})$. 
By Lemma~\ref{lem:gamma_r_order}, $\seed^R_{(3,\mathbf{A})} < \seed^R_{(3,\mathbf{A}')}$ iff $\mathbf{A} \prec \mathbf{A}'$; $\seed^R_{(3,\mathbf{A})} = \seed^R_{(3,\mathbf{A}')}$ iff $\mathbf{A} \sim \mathbf{A}'$. $\seed^R_{(3,\mathbf{A})}$ for finite $R$ is the higher level analog of uniform indiscernibles. 
If $Y$ is a $\Pi^1_3$-wellfounded level-3 tree and $\mathbf{A} \in \exexdesc(Y)$, $\mathbf{y} \in \exdesc(Y)$, let $\seed^Y_{\mathbf{A}} = \seed^{Q^0\oplus Y}_{(3,\mathbf{A})}$, $\seed^Y_{\mathbf{y}} = \seed^{Q^0\oplus Y}_{(3,\mathbf{y})}$.
We will show in Section~\ref{sec:level-3-uniform-indis} that $\seed^R_{(3,\mathbf{A})} = \seed^{\comp{3}{R}}_{\mathbf{A}}$ for $\mathbf{A} \in \exexdesc(R)$. 


Under full AD, the set of $\seed^R_{(3,\emptyset)}$ for finite $R$ is exactly $\set{\aleph_{\xi+1}}{\omega\leq  \xi < \omega^{\omega^{\omega}}}$ by Martin \cite[Theorem 4.17]{jackson_handbook} and Jackson \cite{jackson_delta15}. 
The set of $\seed^R_{(3,\emptyset)}$ for finite $R$ and their limit points will be level-3 indiscernibles. 
The rest of Section~\ref{sec:homog-prop-s_3} will contain a thorough analysis of the structure of level-3 uniform indiscernibles. 



\subsection{Level-3 uniform indiscernibles}
\label{sec:level-3-uniform-indis}


\begin{mydefinition}
  \label{def:level_3_uniform_indiscernibles}
If $R$ is a finite level $\leq 3$ tree, $\alpha$ is an \emph{$R$-uniform indiscernible} iff $\alpha \in \bigcap_{x \in \mathbb{R}} j^R ( \bar{I}^{(3)}_x)$. 
\end{mydefinition}
By Lemma~\ref{lem:R_indiscernible_description_next}, the set of $R$-uniform indiscernibles is the closure of $\set{\seed^R_{(3,\mathbf{A})}}{\mathbf{A} \in \exexdesc(\comp{3}{R})}$, which has order type $\xi+1$ if $\llbracket \emptyset \rrbracket_{\comp{3}{R}} = \widehat{\xi}$. By Lemmas~\ref{lem:next-level-3-indis}-\ref{lem:level_3_indis_cardinal}, $\alpha$ is an $R$-uniform indiscernible iff $\alpha \geq \bolddelta{3}$ is a cardinal in $\mathbb{L}[j^{{R}}(T_3)]$.
In particular, the least $R$-uniform indiscernible is $ \bolddelta{3}$.  

Recall that if $R$ is a level-3 tree, $s \in \dom(R)$, then $R \res s$ the subtree of $R$ whose domain consists of $r$ for which $\corner{r} <_{BK} \corner{s}$. 
If $R$ is a level-3 tree and $\mathbf{A} \in \exexdesc(R)$, we let $R \res \mathbf{A}$ be the subtree of $R$ whose domain consists of $r$ for which $\corner{r} <_{BK} \corner{\mathbf{A}}$. 
If $R$ is a level $\leq 3$ tree and $\mathbf{A} \in \exexdesc(\comp{3}{R})$, $s \in \dom(R)$, let $R \res (3,\mathbf{A}) = \comp{\leq 2}{R} \oplus (\comp{3}{R}\res \mathbf{A})$, $R \res (3,s) = \comp{\leq 2}{R}\oplus (\comp{3}{R}\res s)$. 

\begin{mylemma}
  \label{lem:ordinal_in_below_seed_use_bounded_indis}
  Assume $\boldpi{3}$-determinacy. Suppose $R$ is $\Pi^1_3$-wellfounded level $\leq 3$ tree and $\mathbf{A} \in \exexdesc(\comp{3}{R})$. Then $j^{R \res (3,\mathbf{A}), R}$ is the identity on $\mathbb{L}_{j^{R \res (3,\mathbf{A})}(\bolddelta{3})}[ j^{R \res (3,\mathbf{A})} (T_3)]$. Furthermore, if $s \in \dom(\comp{3}{R})$ and $\lh(s) = 1$, then $j^{R \res (3,s)}(\bolddelta{3}) = \seed^R_{(3,s)}$.
%
\end{mylemma}
\begin{proof}
Using a direct limit argument, it suffices to prove the case when $R$ is finite. 
We prove that $j^{R \res (3,\mathbf{A})}(\bolddelta{3})$ is contained in the range of $j^{R \res (3,\mathbf{A}),R}$.
Suppose $\alpha = [G]_{\mu^R}< \seed^R_{(3,\mathbf{A})}$,  $x \in \mathbb{R}$, $\tau$ is an $\mathcal{L}$-Skolem term such that  $G(\vec{\gamma}) = \tau^{M_{2,\infty}^{-}(x)}(x, \vec{\gamma})$ for any $\vec{\gamma} \in [\omega_1, \bolddelta{3}]^{R \uparrow}$ and $G(\vec{\gamma}) < \comp{3}{\gamma}_{\mathbf{A}}$ for $\mu^R$-a.e.\ $\vec{\gamma}$. 
By Lemma~\ref{lem:remarkable_general}, if $\vec{\beta}$ respects $\comp{\leq 2}{R}$, 
$\vec{\delta}$ and $\vec{\delta}'$ both strongly respects $\comp{3}{R}$ and $\forall r~( (3,r)\in \dom(R \res (3,\mathbf{A})) \to \comp{3}{\delta}_r = \comp{3}{\delta}_r')$, then 
 $\tau^{M_{2,\infty}^{-}(x)}(x, \vec{\beta} \oplus c^{(3)}_{x,\vec{\delta}}) = \tau^{M_{2,\infty}^{-}(x)}(x, \vec{\beta} \oplus c^{(3)}_{x,\vec{\delta}'})$. 
Using the fact that $(Q,\gamma) \mapsto c^{(3)}_{x,Q,\gamma}$ is definable over $M_{2,\infty}^{-}(x^{3\#})$, we can find an $\mathcal{L}$-Skolem term $\sigma$  such that for $\mu^R$-a.e.\ $\vec{\gamma}$, 
\begin{displaymath}
  G(\vec{\gamma})= \sigma^{M_{2,\infty}^{-}(x^{3\#})}(x^{3\#}, (\comp{d}{\gamma}_r)_{(d,r) \in \dom(R \res (3,\mathbf{A}))}). 
\end{displaymath}
Hence, $\alpha = j^{R \res (3,\mathbf{A}), A } (\beta)$ where $\beta = \sigma^{j^{R \res (3,\mathbf{A})}(M_{2,\infty}^{-}(x^{3\#}))}(x^{3\#}, \seed^{R \res (3,\mathbf{A})})$. 
This also implies that $j^{R \res (3,\mathbf{A})}(\bolddelta{3}) \geq \seed^R_{(3,\mathbf{A})}$.
The ``furthermore'' part is due to unboundedness of level-3 sharps. 
\end{proof}

Suppose $Y$ is a level $\leq 3$ tree, $T$ is a level $\leq 2$ tree. 
A \emph{$(Y,T, *)$-description} is of the form $\mathbf{B} = (d, (\mathbf{y}, \pi))$ so that either $d=3 \wedge (\mathbf{y}, \pi) \in \desc(\comp{3}{Y},T, *)$ or $d \leq 2 \wedge (d, ( \mathbf{y}, \pi)) \in \desc(\comp{\leq 2}{Y}, T, *)$. As usual, $\mathbf{B} = (d, (\mathbf{y},\pi))$ is abbreviated by $(d, \mathbf{y}, \pi)$. 
If $Q$ is a finite level $\leq 2$ tree, a $(Y,T,Q)$-description is $(3,(\mathbf{y}, \pi))$ so that $(\mathbf{y}, \pi) \in \desc(\comp{3}{Y},T,Q)$. If $P$ is a finite level-1 tree, a $(Y,T,P)$-description is $(2, (\mathbf{y}, \pi))$ so that $(2,(\mathbf{y}, \pi)) \in \desc(\comp{\leq 2}{Y}, T, P)$. A $(Y,T,-1)$-description is $(1, (\mathbf{y}, \emptyset))$ so that $\mathbf{y} \in \dom(\comp{1}{Y})$.
$\desc(Y,T,*)$, $\desc(Y,T,Q)$, etc.\ denote the sets of relevant descriptions. 
If $Y,T$ are finite,
\begin{displaymath}
  \seed_{\mathbf{B}}^{Y,T} \in \mathbb{L}( j^Y \circ j^T (T_3))
\end{displaymath}
is the element represented modulo $\mu^Y$ by $\id^{Y,T}_{\mathbf{B}}$. 

Suppose that $R,Y$ are level $\leq 3$ trees and $T$ is a level $\leq 2$ tree. $\rho$ \emph{factors $(R,Y,T)$} if $\rho$ is a function on $\dom(R)$, $\comp{\leq 2}{\rho} \DEF \rho \res \dom(\comp{\leq 2}{R})$ factors $(\comp{\leq 2}{R},\comp{\leq 2}{Y},T)$ and $\comp{3}{\rho} \DEF \rho \res \dom(\comp{3}{R})$ factors $(\comp{3}{R}, \comp{3}{Y},T)$. $\rho$ \emph{factors $(R,Y)$} iff $\comp{\leq 2}{\rho}$ factors $(\comp{\leq 2}{R}, \comp{\leq 2}{Y})$ and $\comp{3}{\rho}$ factors $(\comp{3}{R}, \comp{3}{Y})$. 
Suppose that $\rho$ factors $(R,Y,T)$. If $F \in (\omega_1, \bolddelta{3})^{Y \uparrow}$, then
\begin{displaymath}
  F^T_{\rho} : [\omega_1]^{T \uparrow } \to [\omega_1, \bolddelta{3}]^{R \uparrow}
\end{displaymath}
is the function that sends $\vec{\xi}$ to $F^T_{\comp{\leq 2}{\rho}}(\vec{\xi})\oplus F^T_{\comp{3}{\rho}}(\vec{\xi})$. If $T$ is finite, 
\begin{displaymath}
  \id_{\rho}^{Y,T}
\end{displaymath}
is the function $[F]^Y \mapsto [F^T_{\rho}]_{\mu^T}$.  If $Y$ is also finite,
\begin{displaymath}
  \seed_\rho^{Y,T} = [\id^{Y,T}_{\rho}]_{\mu^Y}  \in \mathbb{L}( j^Y \circ j^T (T_3)).
\end{displaymath}
=By \Los{} and Lemmas~\ref{lem:TQW_desc_order},\ref{lem:TQW_description_extension},\ref{lem:B_desc_order},\ref{lem:YTQ_description_extension},\ref{lem:Q_respecting},\ref{lem:R_respect}, for any   $A \in \mu^R$, $\seed_{\rho}^{Y,T} \in j^Y \circ j^T(A)$. We can unambiguously define
\begin{displaymath}
  \rho^{Y,T} : \mathbb{L}(j^R(T_3)) \to \mathbb{L} (j^Y \circ j^T(T_3))
\end{displaymath}
by sending $j^R(F)(\seed^R)$ to $j^Y\circ j^T(F) (\seed_{\rho}^{Y,T})$. 
In general, if $Y,T$ are $\Pi^1_3$-wellfounded and $T$ is $\Pi^1_2$-wellfounded, then $\rho^{Y,T} \circ j^{R',R} = j^{Y',Y} \circ j^{(Y,(T',T))} \circ (\rho')^{Y',T'}$ for $R',Y',T'$ finite subtrees of $R,Y,T$ respectively and $\rho' = \rho \res \dom(R')$ factoring $(R',Y',T')$, where $j^{(Y,(T',T))} = \cup_{x \in \mathbb{R}} j^Y (j^{T',T} \res L[j^{T'}(T_3),x])$.
%
In particular, $\rho^{Y,T}\circ j^R(\bolddelta{3}) = j^Y(\bolddelta{3})$. If $\mathbf{A} \in \exexdesc(\comp{3}{R})$, then $\rho^{Y,T} (\seed^R_{(3, \mathbf{A})}) = \seed^Y_{(3, \tilde{\rho}^T(\mathbf{A}))}$.



If $Y$ is a level $\leq 3$ tree and $T$ is a level $\leq 2$ tree, then
\begin{displaymath}
  Y \otimes T = (\comp{\leq 2}{Y} \otimes T) \oplus (\comp{3}{Y} \otimes T)
\end{displaymath}
is (modulo an isomorphism) a level $\leq 3$ tree. The domain of $Y \otimes T$ consists of $\mathbf{B} = (d, (\mathbf{y},\pi)) \in \desc(Y,T,*)$. 
So $\rho$ factors $(R,Y,T)$ iff $\rho$ factors $(R, Y \otimes T)$. The identity map $\id_{Y \otimes T}: \mathbf{B} \mapsto \mathbf{B}$ factors $(Y \otimes T, Y,T)$.

\begin{mylemma}
  \label{lem:RYT_factor_identity}
  Suppose $Y$ is a $\Pi^1_3$-wellfounded level $\leq 3$ tree and $T$ is a $\Pi^1_2$-wellfounded level $\leq 2$ tree. Then 
$\mathbb{L}_{j^{Y \otimes T} (\bolddelta{3})}[j^{Y \otimes T}(T_3)] = \mathbb{L}_{j^Y(\bolddelta{3})}[j^Y \circ j^T(T_3)]$ and 
$(\id_{Y \otimes T})^{Y,T}$ is the identity map on $\mathbb{L}_{j^{Y \otimes T} (\bolddelta{3})}[j^{Y \otimes T}(T_3)] $.
\end{mylemma}
\begin{proof}
Assume without loss of generality that $Y,T$ are finite. 
Put $R = Y \otimes T$ and $\rho = \id_{Y \otimes T}$. Then $\rho^{Y,T} (\seed^R_{(3, \mathbf{A})}) = \seed^Y_{(3, \tilde{\rho}^T(\mathbf{A}))}$ for any $\mathbf{A} \in \exexdesc(\comp{3}{R})$ and $\forall \mathbf{B} \in \exexdesc(Y) ~ \exists \mathbf{A} \in \exexdesc(R)~ \tilde{\rho}^T(\mathbf{A}) \sim^Y_{*} \mathbf{B}$. Recall that the set of $\mathbb{L}[j^R(T_3)]$-cardinals in the interval $[\bolddelta{3}, j^R(\bolddelta{3})]$ is exactly the closure of $\set{\seed^R_{(3,\mathbf{A})}}{\mathbf{A} \in \exexdesc(\comp{3}{R})}$. For any $x \in \mathbb{R}$, $\rho^{Y,T} \res L[j^R(T_3),x]$ is elementary from $L[j^R(T_3),x]$ to $L[j^Y \circ j^T (T_3), x]$. Hence, it suffices to show that $\rho^{Y,T} \res \bolddelta{3} $ is the identity and whenever $\wocode{\mathbf{A}}_{\prec^{\comp{3}{R}}_{*}}$ is a successor cardinal, then $\rho^{Y,T}$ is continuous at $\seed^R_{(3,\mathbf{A})}$.

By Lemma~\ref{lem:level_2_ultrapower_iteration_reduce} and Corollary~\ref{coro:delta13_bounds_ultrapowers},  $j^R \res \bolddelta{3} = j^{\comp{\leq 2}{R}} \res \bolddelta{3} = j^{\comp{\leq 2}{Y}} \circ j^T \res \bolddelta{3} = j^Y \circ j^T \res \bolddelta{3}$. By indiscernability and remarkability, $\rho^{Y,T} \res \bolddelta{3} = \comp{\leq 2}{\rho}^{\comp{\leq 2}{Y},T} \res \bolddelta{3} $ is the identity map.

Suppose $\wocode{\mathbf{A}}_{\prec^{\comp{3}{R}}_{*}}$ is a successor cardinal and we prove that $\rho^{Y,T}$ is continuous at $\seed^R_{(3,\mathbf{A})}$. 
Suppose $\alpha < \seed^Y_{(3, \tilde{\rho}^T(\mathbf{A}))}$. There is $x \in \mathbb{R}$ such that
\begin{displaymath}
\alpha < \min (j^Y(\bar{I}^{(3)}_x) \setminus \sup\set{\seed^Y_{(3,\mathbf{B})}}{\mathbf{B} \prec^Y_{*} \tilde{\rho}^T(\mathbf{A})}).
\end{displaymath}
Let $\beta = \min(j^R(I^{(3)}_x) \setminus \sup \set{\seed^R_{(3,\mathbf{B})}}{\mathbf{B} \prec^Y_{*} \mathbf{A}})$. Then $\beta < \seed^R_{(3,\mathbf{A})}$ and by induction and elementarity,
\begin{displaymath}
\rho^{Y,T} (\beta) = \min (j^Y \circ j^T(\bar{I}^{(3)}_x) \setminus \sup\set{\seed^Y_{(3,\mathbf{B})}}{\mathbf{B} \prec^Y_{*} \tilde{\rho}^T(\mathbf{A})}) \geq \alpha,
\end{displaymath}
the last inequality from $j^T(\bar{I}^{(3)}_x) \subseteq \bar{I}^{(3)}_x$ and elementarity of $j^Y$.
\end{proof}

\begin{mylemma}
  \label{lem:level_3_uniform_indis_compare}
  Suppose $Y,Y'$ are $\Pi^1_3$-wellfounded level $\leq 3$ trees. 
Then $\llbracket \emptyset \rrbracket_{\comp{3}{Y}}= \llbracket \emptyset \rrbracket_{\comp{3}{Y}'}$ iff $j^Y(\bolddelta{3}) = j^{Y'}(\bolddelta{3})$,  $\llbracket \emptyset \rrbracket_{\comp{3}{Y}} < \llbracket \emptyset \rrbracket_{\comp{3}{Y}'}$ iff $j^Y(\bolddelta{3}) < j^{Y'}(\bolddelta{3})$. 
\end{mylemma}
\begin{proof}
If $\llbracket \emptyset \rrbracket_{\comp{3}{Y}} \leq \llbracket \emptyset \rrbracket_{\comp{3}{Y}'}$ then by Theorems~\ref{thm:factor_ordertype_embed_equivalent_lv2} and~\ref{thm:factor_tower_order_type_equivalent}, there exist a finite $T$ and $\rho$ that factors $(Y, Y',T)$. So $\rho^{Y',T} \circ j^Y(\bolddelta{3}) = j^{Y'}(\bolddelta{3})$, yielding $j^Y(\bolddelta{3}) \leq j^{Y'}(\bolddelta{3})$. So $\llbracket \emptyset \rrbracket_{\comp{3}{Y}}= \llbracket \emptyset \rrbracket_{\comp{3}{Y}'}$ implies $j^Y(\bolddelta{3}) = j^{Y'}(\bolddelta{3})$. 
If $\llbracket \emptyset \rrbracket_{\comp{3}{Y}}< \llbracket \emptyset \rrbracket_{\comp{3}{Y}'}$, we further obtain $\mathbf{B} \in \desc(\comp{3}{Y}',T,*)$ such that $\lh(\mathbf{B}) = 1$ and $\llbracket \mathbf{B}  \rrbracket_{\comp{3}{Y} ' \otimes T} = \llbracket \emptyset \rrbracket_{\comp{3}{Y}}$.
Put $Z = ({Y}' \otimes T )\res (3,\mathbf{B})$. Then $\rho$ factors $(Y,Z)$. Hence $j^Y(\bolddelta{3}) \leq j^Z(\bolddelta{3})$. By Lemma~\ref{lem:ordinal_in_below_seed_use_bounded_indis}, the factor map $j^{Z, Y' \otimes T} $ is the identity on $\mathbb{L}_{j^Z(\bolddelta{3})}[j^Z(\bolddelta{3})]$ and $j^Z(\bolddelta{3}) = \seed^{Y' \otimes T}_{(3, \mathbf{B})}$. By Lemma~\ref{lem:RYT_factor_identity}, $\seed^{Y' \otimes T}_{(3,\mathbf{B})} < j^{Y'}\circ j^T(\bolddelta{3}) = j^{Y'}(\bolddelta{3})$.
\end{proof}

\begin{mylemma}
  \label{lem:level_3_descriptions_represent_uniform_indiscernibles}
  Suppose $Y$ is a finite level $\leq 3$ tree and $\mathbf{A} \in \exexdesc(\comp{3}{Y})$. Suppose $\wocode{\mathbf{A}}_{\prec^{\comp{3}{Y}}_{*}} = \xi $ and $R$ is a finite level $\leq 3$ tree such that $\llbracket \emptyset \rrbracket_{\comp{3}{R}} = \widehat{\xi}$.  Then $\seed^Y_{(3,\mathbf{A})} = j^R(\bolddelta{3})$.
\end{mylemma}
\begin{proof}
  If $\mathbf{A} = (\emptyset,\emptyset,\emptyset)$, this is exactly Lemma~\ref{lem:level_3_uniform_indis_compare}. Suppose $\mathbf{A} \neq (\emptyset,\emptyset,\emptyset)$. Let $T$ be a finite level $\leq 2$ tree and let $\rho$ minimally factor $(R,Y,T)$. Let $\mathbf{B} \in \desc(\comp{3}{Y},T,*)$ such that $\lh(\mathbf{B}) = 1$ and $\llbracket \mathbf{B} \rrbracket_{\comp{3}{Y} \otimes T} = \widehat{\xi}$. 
Put $\mathbf{B} = (\mathbf{y}, \pi) \in \desc(\comp{3}{Y},T,Q)$. 
A routine computation gives $\mathbf{A} \sim^{\comp{3}{Y}}_{*} (\mathbf{y}, \pi, T \otimes Q)$. So $\seed^{Y \otimes T}_{(3,\mathbf{B})} = \seed^{Y \otimes T}_{(3, \mathbf{A})}$. 
Put $Z = Y \otimes T \res (3, \mathbf{B})$. Then  $\llbracket \emptyset \rrbracket_{\comp{3}{Z}} = \llbracket \emptyset \rrbracket_{\comp{3}{R}}$.   By Lemma~\ref{lem:ordinal_in_below_seed_use_bounded_indis},  $j^Z(\bolddelta{3}) = \seed^{Y \otimes T} _{(3,\mathbf{B})}$. By Lemma~\ref{lem:level_3_uniform_indis_compare}, $j^R(\bolddelta{3}) = j^Z(\bolddelta{3})$, and we are done.
\end{proof}

\begin{mydefinition}
  \label{def:level_3_uniform_indiscernibles_final}
In view of Lemma~\ref{lem:level_3_descriptions_represent_uniform_indiscernibles}, we define the \emph{level-3 uniform indiscernibles}:
\begin{enumerate}
\item $u^{(3)}_{\xi+1} = j^R(\bolddelta{3})$ when $\xi < \omega^{\omega^{\omega}}$, $R$ is a $\Pi^1_3$-wellfounded level $\leq 3$ tree and $\llbracket \emptyset \rrbracket_{\comp{3}{R}} = \widehat{\xi}$. 
\item If $0 < \xi  \leq \omega^{\omega^{\omega}}$ is a limit, then $u^{(3)}_{\xi} = \sup _{\eta<\xi}u^{(3)}_{\eta}$.
\end{enumerate}
\end{mydefinition}
If $R$ is a finite level $\leq 3$ tree and $\llbracket \emptyset \rrbracket_{\comp{3}{R}} = \widehat{\xi}$, then the set of $R$-uniform indiscernibles is $\set{u^{(3)}_{\eta}}{0 < \eta \leq \xi+1}$ and we have $\seed^R_{(3, \mathbf{A})} = u^{(3)}_{\eta+1}$ for $\wocode{\mathbf{A}}_{\prec^{\comp{3}{R}}_{*}} = \eta$.

The next lemma is the higher level analog of $\bolddelta{2} = u_2$. 
\begin{mylemma}
  \label{lem:delta14}
  Assume $\boldpi{3}$-determinacy. Then $\bolddelta{4} = u^{(3)}_2$.
\end{mylemma}
\begin{proof}
  If $W$ is a $\Sigma^1_4(x)$ wellfounded relation on $\mathbb{R}$, then $W$ is $\bolddelta{3}$-Suslin via a tree in $L[T_3,x]$, so by Kunen-Martin and Lemma~\ref{lem:next-level-3-indis}, $\rank(W) < ((\bolddelta{3})^{+})^{L[T_3,x]} < \min (j^{R^0} (I ^{(3)} _{x^{3\#}}) \setminus (\bolddelta{3}+1)) < u^{(3)}_2$ as $\llbracket \emptyset \rrbracket_{R^0} = \omega = \widehat{1}$. If $\alpha < u^{(3)}_2$, pick $x$ such that $\alpha <  \min (j^{R^0} (I ^{(3)} _{x}) \setminus (\bolddelta{3}+1))$. 
 Lemma~\ref{lem:next-level-3-indis} gives a surjection  $f: \bolddelta{3} \twoheadrightarrow \alpha $ which is definable over $j^{R^0}(M_{2,\infty}^{-}(x))$ from $\se{\bolddelta{3},x}$. From $f$ we can define a $\Delta^1_4(x^{3\#})$ prewellordering of length $\alpha$.
\end{proof}

\subsection{The level-4 Martin-Solovay tree}
\label{sec:level-4-Martin_Solovay_tree}


Let $R^{\infty}$ be the unique (up to an isomorphism) level-3 tree such  that
\begin{enumerate}
\item  for any finite level-3 tree $Y$, there exists $\rho$ which minimally factors $(Y,R)$;
\item if $r \in \dom(R^{\infty})$ then there exist a finite $Y$ and $\rho$ which minimally factors $(Y,R)$ such that $r \in \dom(\rho)$.
\end{enumerate}
In other words, $R^{\infty}$ is the minimum $\Pi^1_3$-wellfounded level-3 tree that is universal for finite level-3 tree in terms of minimal factorings. 
We fix the following representation of $R^{\infty}$, whose domain consists of finite tuples of ordinals in $\omega^{\omega^{\omega}}$:
\begin{enumerate}
\item $(\xi_1) \in \dom(R^{\infty})$ iff $0<\xi_1< \omega^{\omega^{\omega}}$. $R^{\infty}((\xi_1))$ is the $Q^0$-partial level $\leq 2$ tree induced by $\widehat{\xi_1}$. 
\item If $r = (\xi_1,\dots,\xi_{k-1}) \in \dom(R^{\infty})$, then $r \concat (\xi_k) \in \dom(R^{\infty})$ iff ${\xi_k} < \omega^{\omega^{\omega}}$ and there exists a completion $Q^{+}$ of $R^{\infty}(r)$ such that 
the $Q^{+}$-approximation sequence of $\widehat{\xi_k}$ is $(\widehat{\xi_i})_{1 \leq i \leq k}$;  if $r \concat(\xi_k) \in \dom(R^{\infty})$ and $Q^{+}$ is the unique such completion, then $R^{\infty}(r \concat (\xi_k))$ is the $Q^{+}$-partial level $\leq 2$ tree induced by $\widehat{\xi_k}$. 
\end{enumerate}
Therefore, $\llbracket \emptyset \rrbracket_R = u_{\omega}$ and if $r=(\xi_1,\dots,\xi_k) \in \dom(R^{\infty})$, then $\llbracket r \rrbracket_R = \widehat{\xi_k}$. If $Y$ is a finite level-3 tree, then the map $y \mapsto r_y$ minimally factors $(Y,R)$, where if $(\llbracket y \res i \rrbracket_Y) _{1 \leq i \leq \lh(y)}=  (\widehat{\xi_1},\dots, \widehat{\xi_{\lh(y)}})$ then $r_y = (\xi_1,\dots,\xi_{\lh(y)})$. 
If $0 < \xi < \omega^{\omega^{\omega}}$, let $R^{\infty}_{\xi} = R^{\infty } \res (\xi)$.
By Lemma~\ref{lem:ordinal_in_below_seed_use_bounded_indis}, if $0<\xi \leq  \omega^{\omega^{\omega}}$, then the factoring map $j^{R^{\infty} _{\xi}, R^{\infty}}$ is the identity on $\mathbb{L}_{j^{R^{\infty}_{\xi}}(\bolddelta{3})}[j^{R^{\infty} _{\xi}}(T_3)]$ and $j^{R^{\infty} _{\xi}}(\bolddelta{3}) = \seed^{R^{\infty}}_{(\xi)} = u^{(3)}_{\xi}$. In particular, $j^{R^{\infty}}(\bolddelta{3}) = u \ooo$. 

$R^{\infty}$ will be the tree based on which level-3 sharp codes for ordinals below $u\ooo$ are defined. 

Recall Kunen's $\Delta^1_3$ coding of subsets of $u_{\omega}$. We fix a $\Delta^1_3$ surjection
\begin{displaymath}
v \mapsto X_v
\end{displaymath}
from $\mathbb{R}$ onto $\power((V_{\omega}\cup u_{\omega})^{<\omega})$.
$\LO^{(3)} $ is the set of $v \in \mathbb{R}$ such that $X_v$ is a linear ordering of $u_{\omega}$. 
$\WO^{(3)} = \WO^{(3)}_0$ is the set of $v \in \LO^{(3)}$ such that $X_v$ is a wellordering of $u_{\omega}$. $\WO^{(3)}_0$ is $\Pi^1_3$. For $v \in \WO^{(3)}$, put $\wocode{v} = \ot(X_v)$. Every ordinal in $\bolddelta{3}$ is of the form $\wocode{v}$ for some $v \in \WO^{(3)}$.

A \emph{level-3 sharp code} is a pair $\corner{ \gcode{\tau}, x^{3\#}}$ where $\tau$ is an $\mathcal{L}^{{\underline{x}},R^{\infty}}$-Skolem term for an ordinal without free variables. For $0 < \xi \leq \omega^{\omega^{\omega}}$, $\WO^{(3)}_{\xi}$ is the set of level-3 sharp codes $\corner{ \gcode{\tau}, x^{3\#}}$ such that $\tau$ is an $\mathcal{L}^{\underline{x},R^{\infty} _{\xi}}$-Skolem term. 
 $\WO^{(3)}_{\xi}$ is $\Pi^1_4$ for $0 < \xi \leq \omega^{\omega^{\omega}}$.
The ordinal coded by $\corner{ \gcode{\tau}, x^{3\#}}$ is
\begin{displaymath}
  \sharpcode{\corner{\gcode{\tau}, x^{3\#}}} = \tau^{(j^{R^{\infty}}(M^{-}_{2,\infty}(x)); \seed^{R^{\infty}})}.
\end{displaymath}
For each $\xi$, $\WO^{(3)}_{\xi}$ is $\Pi^1_4$. By Lemma~\ref{lem:ordinal_in_below_seed_use_bounded_indis}, if $\corner{ \gcode{\tau}, x^{3\#}} \in \WO^{(3)}_{\xi}$ and $\tau = \sigma(\underline{x}, \underline{c_{r_1}},\dots)$, $\sigma$ is an $\mathcal{L}$-Skolem term, then 
\begin{displaymath}
  \sharpcode{\corner{\gcode{\tau}, x^{3\#}}} = \sigma^{j^{R^{\infty}_{\xi}}(M^{-}_{2,\infty})} (x, \seed^{R^{\infty}_{\xi}}_{r_1},\dots).
\end{displaymath}
\begin{mylemma}
  \label{lem:level_3_sharp_code_equivalence}
  The relations $v,w \in \WO\ooo \wedge \sharpcode{v} = \sharpcode{w}$ and $v,w \in \WO\ooo \wedge \sharpcode{v} < \sharpcode{w}$ are both $\Delta^1_5$.
\end{mylemma}
\begin{proof}
    $ \sharpcode{\corner{\gcode{\tau}, x^{3\#}}} =  \sharpcode{\corner{ \gcode{\tau'}, (x')^{3\#}}}$ iff $\tau=\sigma(\underline{x}, \underline{c_{r_1}},\dots, \underline{c_{r_n}})$, $\tau'=\sigma'(\underline{x}, \underline{c_{r_1'}},\dots, \underline{c_{r'_{n'}}})$, $\sigma,\sigma'$ are $\mathcal{L}$-Skolem terms, and for some finite level-3 tree $Y$ and some $\rho$ factoring $(Y, R^{\infty})$ such that $\vec{r} \concat \vec{r}' \subseteq \ran(\rho)$, ``$\sigma ( (\underline{x})_{\text{left}}, \underline{c_{\rho^{-1}(r_1)},}\dots) =  \sigma'( (\underline{x})_{\text{right}}, \underline{c_{\rho^{-1}(r'_1)}},\dots)$'' is true in $(x \oplus x')^{3\#} (Y)$.
\end{proof}
Recall that $\WO_{\omega}$ is the set of (level-1) sharp codes for ordinals below $u_{\omega}$. 
The connection between level-3 sharp codes and level-1 sharp codes or $\WO$ is also $\Delta^1_5$. For instance, the relation ``$v \in \WO\ooo \wedge w \in \WO_{\omega} \wedge \sharpcode{v} = \sharpcode{w}$'' is $\Delta^1_5$. 

If $\Gamma$ is a pointclass, say that $A \subseteq  u^{(3)}_{\omega^{\omega^{\omega}}} \times  \mathbb{R} $ is in $\Gamma$ iff $\set{(v, x)}{v \in \WO_{\omega^{\omega^{\omega}}}^{(3)} \wedge (\sharpcode{v},x) \in A}$ is in $\Gamma$. $\Gamma $ acting on subsets of product spaces is defined in the obvious way. 

\begin{mylemma}
  \label{lem:sharp_code_factor_complexity}
  \begin{enumerate}
  \item Suppose that $\xi\leq \eta < \omega^{\omega^{\omega}}$ and $\rho$ factors $(R^{\infty}_{\xi}, R^{\infty}_{\eta})$. Then $\rho^{R^{\infty}_{\eta}} \res u^{(3)}_{\xi}$ is $\Delta^1_5$, uniformly in $(\xi,\eta,\rho)$.
  \item Suppose that $\xi < \omega^{\omega^{\omega}}$ and $Q,Q'$ are finite level $\leq 2$ trees, $Q$ is a subtree of $Q'$. Then $j^{(R^{\infty}_{\xi}, (Q,Q'))} \res u^{(3)}_{\xi}$ is $\Delta^1_5$, uniformly in $(\xi,Q,Q')$. 
  \end{enumerate}
\end{mylemma}
\begin{proof}
 1.  $\alpha < u^{(3)}_{\xi} \wedge \rho^{R^{\infty}_{\eta}} (\alpha) = \beta$ iff there exist $x \in \mathbb{R}$, an $\mathcal{L}$-Skolem term $\tau$ and $ r_1,\dots,r_n$ such that $r_i \in \dom(R^{\infty}_{\xi})$ for any $i$ and $\alpha = \corner{\gcode{\tau(\underline{x},\underline{c_{r_1}},\dots, \underline{c_{r_n}} )}, x^{3\#}}$,  $\beta = \corner{\gcode{\tau(\underline{x},\underline{c_{\rho(r_1)}},\dots, \underline{c_{\rho(r_n)}} )}, x^{3\#}}$.

2.  $\alpha < u^{(3)}_{\xi} \wedge j^{(R^{\infty}_{\xi}, (Q,Q'))} (\alpha) = \beta$ iff there exist $x \in \mathbb{R}$, an $\mathcal{L}$-Skolem term $\tau$ and $ r_1,\dots,r_n$ such that $r_i \in \dom(R^{\infty}_{\xi})$ for any $i$ and $\alpha = \corner{\gcode{\tau(\underline{x},\underline{c_{r_1}},\dots, \underline{c_{r_n}} )}, x^{3\#}}$,  $\beta = \corner{\gcode{\underline{j^{Q,Q'}}(\tau(\underline{x},\underline{c_{(r_1)}},\dots, \underline{c_{(r_n)}} ))}, x^{3\#}}$.
\end{proof}

By Lemma~\ref{lem:next-level-3-indis}, the set of uncountable $\mathbb{L}[j^{R^{\infty}}(T_3)]$-cardinals $\leq u\ooo$ is the closure of 
\begin{displaymath}
  \set{u_n}{1 \leq n < \omega} \cup \set{\seed^{R^{\infty}}_{\mathbf{A}}}{\mathbf{A} \in \exexdesc(R^{\infty})}.
\end{displaymath}
By Lemma~\ref{lem:RYT_factor_identity}, if $\mathbf{A} = (\mathbf{r},\pi, T) \in \exexdesc(R^{\infty})$, $\mathbf{r} = (r,Q,\overrightarrow{(d,q,P)})$, $r = (\xi_i)_{1 \leq i \leq k}$ and $\seed^{R^{\infty}}_{\mathbf{A}} > \bolddelta{3}$ is a successor cardinal in $\mathbb{L}[j^{R^{\infty}}(T_3)]$, then $r$ is of discontinuous type, $\xi_k$ is a successor ordinal,  and letting $r' = (\xi_i)_{1 \leq i < k} \concat (\xi_k-1)$, $\mathbf{r}' = (r',Q,\overrightarrow{(d,q,P)})$, $\mathbf{A}' = (\mathbf{r}',\pi,T)$, then
\begin{displaymath}
  \set{ \corner{x^{3\#},\gcode{\tau^{\underline{j^T}(V)} (\underline{x},  \underline{c_{\mathbf{A}'}})}} }{ x\in \mathbb{R},\tau\text{ is an }\mathcal{L}\text{-Skolem term for an ordinal}}
\end{displaymath}
is a cofinal subset of $\seed^{R^{\infty}}_{\mathbf{A}}$.

A level-3 EM blueprint over a real $\Gamma$ is completely decided by $\Gamma(R^{\infty})$. $\Gamma$ is coded into the real $z_{\Gamma} \in 2^{\omega}$ where $z (k) = 0 \eqiv k \in \Gamma(R^{\infty})$. We shall identify $\Gamma$ with $z_{\Gamma}$ when no confusion occurs. 
We define the level-4 Martin Solovay tree $T_4$ which projects to $\set{x^{3\#}}{ x \in \mathbb{R}}$. 
  $T_4$ will be $\Delta^1_5$ as a subset of $(\omega \times u\ooo )^{<\omega}$, the complexity based on Lemma~\ref{lem:sharp_code_factor_complexity}. 

Let $T$ be a recursive tree so that $z \in [T]$ iff $z$ is a level $\leq 2$-correct, remarkable level-3 EM blueprint over a real. Let $(r_i)_{1 \leq i < \omega}$ be an effective enumeration of $\dom(R^{\infty})$ and let $(\tau_{k})_{k<\omega}$ be an effective enumeration of all the $\mathcal{L}$-Skolem terms for an ordinal, where $\tau_k$ is $f(k)+1$-ary.   
 $T_4$ is the tree on $2 \times u^{(3)}_{\omega^{\omega^{\omega}}}$ where
\begin{displaymath}
  (t,\vec{\alpha}) \in T_4
\end{displaymath}
iff $t \in T$ and 
\begin{enumerate}
\item if $\xi \leq \eta < \omega^{\omega^{\omega}}$, $r_1,\dots,r_{f(k)} \in \dom(R^{\infty} _{\xi})$, $r_1,\dots,r_{f(l)} \in \dom(R^{\infty}_{\eta})$,  $\rho$ factors $(R^{\infty} _{\xi}, R^{\infty} _{\eta})$,
  \begin{enumerate}
  \item if
``$\tau_k ( \underline{x},\underline{c_{\rho(r_1)}},\dots, \underline{c_{\rho(r_{f(k)})}}) = \tau_l (\underline{x}, \underline{c_{r_1}},\dots, \underline{c_{r_{f(l)}}})$'' is true in $t$,  then $\rho^{R^{\infty}_{\eta}}(\alpha_k) = \alpha_l$;
  \item if
``$\tau_k ( \underline{x},\underline{c_{\rho(r_1)}},\dots, \underline{c_{\rho(r_{f(k)})}}) < \tau_l (\underline{x}, \underline{c_{r_1}},\dots, \underline{c_{r_{f(l)}}})$'' is true in $t$,  then $\rho^{R^{\infty}_{\eta}}(\alpha_k) < \alpha_l$;
\end{enumerate}
\item if $\xi < \omega^{\omega^{\omega}}$, $r_1,\dots,r_{\max(f(k),f(l))} \in \dom(R^{\infty} \res \xi)$, 
$Q,Q'$ are finite level $\leq 2$ trees, $Q$ is a subtree of $Q'$, ``$\underline{j^{Q,Q'}}(\tau_{k}(\underline{x},\underline{c_{r_1}},\dots, \underline{c_{r_{f(k)}}})) = \tau_{l}(\underline{x}, \underline{c_{r_1}},\dots, \underline{c_{r_{f(l)}}})$'' is true in $t$, then $j^{(R^{\infty}_{\xi}, (Q,Q'))} (\alpha_k)  = \alpha_l$. 
\end{enumerate}

\begin{mytheorem}
  \label{thm:level_4_MS_tree}
  $p[T_4] = \set{x^{3\#}}{x \in \mathbb{R}}$. Furthermore, for any $x \in \mathbb{R}$,  $(\tau_{k}^{(j^{R^{\infty}}(M_{2,\infty}^{-}(x)))}(x, \seed^{R^{\infty}}_{r_1},\dots,\seed^{R^{\infty}}_{r_{f(k)}}))_{k<\omega}$ is the honest leftmost branch of $(T_4)_{x^{3\#}}$. 
\end{mytheorem}
\begin{proof}
By definition, for any $x$,
\begin{displaymath}
(x^{3\#},(\tau_{k}^{(j^{R^{\infty}}(M_{2,\infty}^{-}(x)))}(x, \seed^{R^{\infty}}_{r_1},\dots,\seed^{R^{\infty}}_{r_{f(k)}}))_{k<\omega} ) \in [T_4].
\end{displaymath}

  Suppose now $(z , \vec{\beta}) \in p[T_4]$. Let $x$ be a real so that $z$ codes  a remarkable level-3 EM blueprint $\Gamma$ over $x$. We need to show that $z$ is iterable and for any  $k$,  $\tau_{k}^{(j^{R^{\infty}}(M_{2,\infty}^{-}(x)))}(z, \seed^{R^{\infty}}_{r_1},\dots,\seed^{R^{\infty}}_{r_{f(k)}}) \leq \beta_{k}$. 
For each $k$, pick a finite subtree $Y_k$ of $R^{\infty}$ and $F_{k} : [\bolddelta{3}]^{Y_k \uparrow} \to \bolddelta{3}$ such that $\se{r_1,\dots,r_{f(k)}} \subseteq \dom(Y_k)$, $F_{k} \in \mathbb{L}[T_3]$ and $j^{Y_k, R^{\infty}}([F_{k}]_{\mu^{Y_k}}) = \beta_{{k}}$. By $\mathbb{L}[T_3]$-countable completeness of the club filter on $\bolddelta{3}$, we can find a club $C$ in $\bolddelta{3}$ such that $C \in \mathbb{L}[T_3]$ and
\begin{enumerate}
\item if $\xi \leq \eta < \omega^{\omega^{\omega}}$, $Y_k$ is a subtree of $R^{\infty}_{\xi}$, $Y_l$ is a subtree of $R^{\infty}_{\eta}$,  $\rho$ factors $(R^{\infty}_{\xi}, R^{\infty}_{\eta})$, 
  \begin{enumerate}
  \item if ``$\tau_k ( \underline{x},\underline{c_{\rho(r_1)}},\dots, \underline{c_{\rho(r_{f(k)})}}) = \tau_l (\underline{x}, \underline{c_{r_1}},\dots, \underline{c_{r_{f(l)}}})$'' is true in $\Gamma(R^{\infty})$, $\vec{\gamma} \in [C]^{R^{\infty}_{\eta} \uparrow}$,   then $F_k (\vec{\gamma}_{\rho} \res \dom(Y_k)) = F_l (\vec{\gamma} \res \dom(Y_l))$;
  \item if ``$\tau_k ( \underline{x},\underline{c_{\rho(r_1)}},\dots, \underline{c_{\rho(r_{f(k)})}}) < \tau_l (\underline{x}, \underline{c_{r_1}},\dots, \underline{c_{r_{f(l)}}})$'' is true in $\Gamma(R^{\infty})$, $\vec{\gamma} \in [C]^{R^{\infty}_{\eta} \uparrow}$,   then $F_k (\vec{\gamma}_{\rho} \res \dom(Y_k)) < F_l (\vec{\gamma} \res \dom(Y_l))$.
 \end{enumerate}
\item if $\xi < \omega^{\omega^{\omega}}$, $Y_k,Y_l$ are subtrees of $R^{\infty}_{\xi}$, $Q,Q'$ are finite level $\leq 2$ trees, $Q$ is a subtree of $Q'$,  ``$\underline{j^{Q,Q'}}(\tau_{k}(\underline{x},\underline{c_{r_1}},\dots, \underline{c_{r_{f(k)}}})) = \tau_{l}(\underline{x}, \underline{c_{r_1}},\dots, \underline{c_{r_{f(l)}}})$'' is true in $\Gamma(R^{\infty})$, $\vec{\gamma} \in [C]^{R^{\infty}_{\xi}\uparrow}$, 
 then $j^{Q,Q'} (F_k(\vec{\gamma} \res \dom(Y_k)))  = F_l (\vec{\gamma} \res \dom(Y_l))$. 
\end{enumerate}

Suppose $S$ is a $\Pi^1_3$-wellfounded level-3 tree. We show that $\mathcal{M}_{\Gamma,S}$ is a $\Pi^1_3$-iterable $x$-mouse. Put $\mathcal{N} = \mathcal{M}_{\Gamma,S}$. Put $\eta \in C_0$ iff $C \cap \eta$ has order type $\eta$, $\eta \in D$ iff $C_0 \cap \eta$ has order type $\eta$. Fix $\vec{\gamma} \in [D]^{S \uparrow}$. We define an embedding
\begin{displaymath}
  \theta : \ord^{\mathcal{N}} \to \bolddelta{3}
\end{displaymath}
as follows. 
If $\sigma$ is an $\mathcal{L}$-Skolem term, $s_1,\dots,s_n \in \dom(R)$, $R$ is a finite subtree of $S$, $a = ( \sigma ( \underline{c_{s_1}},\dots, \underline{c_{s_n}} ))^{\mathcal{N}}$, 
$\rho$ factors $(R,Y_k)$, $\tau_k (\underline{c_{r_1}}, \dots, \underline{c_{r_{f(k)}}})$ is logically equivalent to $\sigma( \underline{c_{\rho(s_1)}},\dots, \underline{c_{\rho(s_n)}} )$, $\vec{\delta} \in [C_0]^{Y_k \uparrow}$, $\delta_{\rho(s)} = \gamma_s$ for any $s \in \dom(R)$, we put
\begin{displaymath}
  \theta(a) = F_k  ( \vec{\delta}).
\end{displaymath}
$\theta$ is well defined: Suppose $\sigma'$ is another $\mathcal{L}$-Skolem term, $s_1',\dots,s_{n'}' \in \dom(R')$, $R'$ is a finite subtree of $S$, ``$\sigma(\underline{x}, \underline{c_{s_1}},\dots, \underline{c_{s_n}}) = \sigma'( \underline{x},\underline{c_{s_1'}},\dots, \underline{c_{s_{n'}'}}) $'' is true in $\Gamma(S)$, 
$\rho'$ factors $(R', Y_{k'})$,  $\tau_{k'} ( \underline{x},\underline{c_{r_1}},\dots, \underline{c_{r_{f(k')}}})$ is logically equivalent to $\sigma' (\underline{x}, \underline{c_{\rho'(s_1')}},\dots, \underline{c_{\rho'(s_{n'}')}})$, $\vec{\delta}' \in [C_0]^{Y_{k'} \uparrow}$, $\delta_{\rho'(s)} = \gamma_s$ for any $s \in \dom(R')$. Pick $\xi,\xi' < \omega^{\omega^{\omega}}$ such that $Y_k$ is a subtree of $R^{\infty}_{\xi}$, $Y_{k'}$ is a subtree of $R^{\infty}_{\xi'}$. 
Let $(Y^{*}, \psi,\psi', \vec{\epsilon})$ be the amalgamation of $(Y_k, \vec{\delta})$ and $(Y_{k'}, \vec{\delta}')$, obtained by Lemma~\ref{lem:level_3_amalgamation}. That is, $Y^{*}$ is a finite level-3 tree, $\psi$ factors $(Y_k, Y^{*})$, $\psi'$ factors $(Y_{k'},Y^{*})$, $\vec{\epsilon} \in [C_0]^{Y^{*}\uparrow}$, $\epsilon_{\psi(y)} = \delta_y$ for $y \in \dom(Y_k)$, $\epsilon_{\psi'(y)} = \delta'_y$ for  $y \in \dom(Y_{k'})$. So ``$\sigma(\underline{x},\underline{c_{\psi \circ \rho(s_1)}},\dots) = \sigma'(\underline{x}, \underline{c_{\psi' \circ \rho'(s_1')}},\dots)$'' is true in $\Gamma(Y^{*})$. 
Pick $\eta < \omega^{\omega^{\omega}}$ large enough so that there exist $\phi,\phi'$ factoring 
$(R^{\infty}_{\xi}, R^{\infty}_{\eta})$, $(R^{\infty}_{\xi'}, R^{\infty}_{\eta})$ respectively and $\phi^{*}$ factoring $(Y^{*}, R^{\infty}_{\eta})$ such that $\phi^{*} \circ \psi = \phi \res \dom(Y_k)$, $\phi^{*} \circ \psi' = \phi' \res \dom(Y_{k'})$. 
 So ``$\sigma(\underline{x},\underline{c_{\phi \circ \rho(s_1)}},\dots) = \sigma'(\underline{x}, \underline{c_{\phi' \circ \rho'(s_1')}},\dots)$'' is true in $\Gamma(R^{\infty})$. Let $\tau_l(\underline{x}, \underline{c_{r_1}},\dots, \underline{c_{r_{f(l)}}}) $ be logically equivalent to $\sigma (\underline{x}, \underline{c_{\phi \circ \rho (s_1)}},\dots)$. So ``$\tau_k (\underline{x}, \underline{c_{\phi(r_1)}},\dots) = \tau_l (\underline{x}, \underline{c_{r_1}},\dots)$'' and ``$\tau_{k'} (\underline{x}, \underline{c_{\phi'(r_1)}},\dots) = \tau_l (\underline{x}, \underline{c_{r_1}},\dots)$'' are both true in $\Gamma(R^{\infty})$. 
We can find $\vec{\alpha} \in [C]^{R^{\infty}_{\eta} \uparrow}$ so that $\alpha_{\phi^{*}(y)} = \epsilon_y$ for any $y \in \dom(Y^{*})$. 
 By assumption, $F_k (\vec{\delta}) = F_l (\vec{\alpha}) = F_{k'} (\vec{\delta}')$.

 Similar arguments show that $\theta$ is order preserving and $\theta'' ((\underline{S_3})^{\mathcal{N}}) \subseteq S_3$.  So $\mathcal{N}$ is wellfounded and $p[(\underline{S_3})^{\mathcal{N}}] \subseteq p[S_3]$. We then show that $\mathcal{N}$ is $\Pi^1_3$-iterable. By Lemma~\ref{lem:weakly_remarkable_strengthening}, $\mathcal{N}$ has cofinally many cardinal strong cutpoints. For each cardinal strong cutpoint $\xi$ of $\mathcal{N}$, by definition of $\underline{S_3}$, $\mathcal{N}^{\coll(\omega,\xi)} \models $``$\mathcal{N}|\xi$ is $\Pi^1_3$-iterable $\wedge ~\underline{S_3}$ projects to the set of $\Pi^1_3$-wellfounded level-3 towers''. The fact that $p[(\underline{S_3})^{\mathcal{N}}] \subseteq p[S_3]$ implies that $\mathcal{N}^{\coll(\omega,\xi)}$ is $\Sigma^1_3$-correct. So $\mathcal{N}|\xi$ is genuinely $\Pi^1_3$-iterable. By varying $\xi$, $\mathcal{N}$ is $\Pi^1_3$-iterable. 
 
Next, we show that for any $k$,   $ \tau_{k}^{(j^{R^{\infty}}(M_{2,\infty}^{-}(x)); \seed^{R^{\infty}})} \leq \beta_k$. We define an embedding
\begin{displaymath}
  \theta : \set{\tau_{k}^{(M_{2,\infty}^{-}(x); \vec{\gamma})}}{ \vec{\gamma} \in [D]^{R^{\infty} \uparrow}, k<\omega} \to \bolddelta{3}
\end{displaymath}
by $\theta (\tau_{k}^{(M_{2,\infty}^{-}(x);\vec{\gamma})}) = F_{k}(\vec{\gamma} \res \dom(Y_k))$. A similar argument shows that $\theta$ is well defined and order preserving. 
%
In particular, for any $\vec{\gamma} \in [D]^{R^{\infty}\uparrow}$, $\tau_{k}^{(M_{2,\infty}^{-}(x);\vec{\gamma})}  \leq F_{k}(\vec{\gamma} \res \dom(Y_k))$. Hence,  $ \tau_{k}^{(j^{R^{\infty}}(M_{2,\infty}^{-}(x)); \seed^{R^{\infty}})} \leq \beta_k$.
\end{proof}

\section{The level-4 sharp}
\label{sec:level-4-sharp}

\subsection{The level-4 Kechris-Martin theorem}
\label{sec:level-4-kechris}

For a countable structure $\mathcal{P}$ in the language of premice that satisfies Axioms~\ref{item:EM_ZFC}-\ref{item:level-2-embedding_invariance} in Definition~\ref{def:pre_level_3_EM_blueprint} and the universality of level $\leq 2$ ultrapowers axiom,
the direct limit $\mathcal{P}_{\infty}$ is defined in Definition~\ref{def:g_N}. The wellfounded part $\mathcal{P}_{\infty}$ is always transitive. 
Recall that every ordinal in $\mathcal{P}_{\infty}$ is of the form $\pi_{\mathcal{P},X,\vec{\beta}, \infty}(a)$ where $a \in \ord^{\mathcal{N}}$, $X$ is a finite level $\leq 2$ tree, $\vec{\beta} \in [\omega_1]^{X \uparrow}$. The relation ``$v \in \LO^{(3)}$ codes the order type of $\ord^{\mathcal{P}_{\infty}}$'' is uniformly $\boldDelta{3}$ in the code of $\mathcal{P}$.

For $x \in \mathbb{R}$, a putative $x$-3-sharp is a level $\leq 2$ correct, remarkable level-3 EM blueprint over $x$ that satisfies the universality of level $\leq 2$ ultrapowers axiom. Suppose $x^{*}$ is a putative $x$-3-sharp.  
For any limit ordinal $\alpha < \bolddelta{3}$, we can build an EM model
\begin{displaymath}
\mathcal{M}^{*}_{x^{*},\alpha}
\end{displaymath}
as follows. Let $R$ be a level-3 tree such that $\llbracket \emptyset \rrbracket_R = \alpha$. Then $\mathcal{M}^{*}_{x^{*},\alpha} = (\mathcal{M}^{*}_{x^{*},R})_{\infty}$. This definition is independent of the choice of $R$: Suppose $R'$ is another level-3 tree and $\llbracket \emptyset \rrbracket_{R'} = \alpha$, and suppose without loss of generality that $\rho $ minimally factors $(R,R')$ by Theorem~\ref{thm:factor_tower_order_type_equivalent}. Then $\rho^{*, R'}_{x^{*}}: \mathcal{M}^{*}_{x^{*},R} \to \mathcal{M}^{*}_{x^{*},R'}$ induces a canonical embedding $\phi: (\mathcal{M}^{*}_{x^{*},R})_{\infty} \to (\mathcal{M}^{*}_{x^{*},R'})_{\infty} $. Let $T$ be $\Pi^1_2$-wellfounded and let $\psi$ minimally factor $(R',R \otimes T)$. Then $\psi^{*, R ,T }_{x^{*}}: \mathcal{M}^{*}_{x^{*},R'} \to \mathcal{M}^{*,T}_{x^{*},R}$ induces a canonical embedding $\phi' : (\mathcal{M}^{*}_{x^{*},R'})_{\infty} \to (\mathcal{M}^{*}_{x^{*},R})_{\infty} $. By coherency, $\phi' \circ \phi = \id$ and hence $\phi = \phi' = \id$. 
We say that $x^{*}$ is \emph{$\alpha$-iterable} iff $\alpha$ is in the wellfounded part of $\mathcal{M}^{*}_{x^{*},\alpha}$. 

A \emph{putative level-3 sharp code for an increasing function} is $w = \corner{\gcode{\tau}, x^{*}}$ such that $x^{*}$ is a putative $x$-3-sharp, $\tau$ is a unary $\mathcal{L}^{\underline{x}}$-Skolem term and
\begin{displaymath}
``\forall v,v'( (v , v'\in \ord \wedge v < v' )\to  (\tau(v) \in \ord  \wedge \tau(v) < \tau(v')))"
\end{displaymath}
is true in $x^{*}(\emptyset)$. In addition, when $x^{*} = x^{3\#}$, $\corner{\gcode{\tau}, x^{*}}$ is called a \emph{(true) level-3 sharp code for an increasing function}. 
If $x^{*}$ is a putative $x$-$3$-sharp and $\alpha < \bolddelta{3}$, $\llbracket \emptyset \rrbracket_R = \alpha$, there is a canonical bijection between $u_{\omega}$ and $\mathcal{M}^{*}_{x^{*},\alpha}$ which is  $\Delta_1$-definable over $L_{\kappa_3^{x,R}}[T_2,x,R]$ from $\se{T_2,x,R}$. 
 The statement `` $\corner{\gcode{\tau}, x^{*}}$ is a putative level-3 sharp code for an increasing function, $x^{*}$ is $\alpha$-iterable, $X_r$ codes the order type of $\tau^{\mathcal{M}_{x^{*}, \alpha}}(\alpha)$'' about  $(\corner{\gcode{\tau}, x^{*}},r)$  is $\boldsigma{3}$ in the code of $\alpha$.

\begin{mylemma}
  \label{lem:KM_bound}
Assume $\boldDelta{4}$-determinacy. Suppose $ \kappa \leq u \ooo$ is an uncountable cardinal in $\mathbb{L}[j^{R^{\infty}}(T_3)]$.
If $A$ is a $\Sigma^1_{5}(x)$ subset of $\kappa$ and $\sup A < \kappa$, then $\exists w \in \Delta^1_5(x) \cap \WO\ooo~( \sup A < \sharpcode{w} < \kappa)$. 
\end{mylemma}
\begin{proof}
Let $x = 0$. The lemma is trivial if $\kappa$ is a limit cardinal in $\mathbb{L}[j^{R^{\infty}}(T_3)]$.  Suppose now $\kappa $ is a successor cardinal in $\mathbb{L}[j^{R^{\infty}}(T_3)]$. 
Let $B$ be $\Pi^1_4$ such that $w \in \WO\ooo \wedge \sharpcode{w} \in A$ iff $\exists y~ (w,y) \in B$. 

Case 1: $\omega_1 \leq \kappa < u_{\omega}$. 

The lower level proof in \cite{Kechris_Martin_II} 
carries over almost verbatim, except the game becomes $\Sigma^1_4$ for the winner and hence a $\Delta^1_5$ winning strategy can be found by Moschovakis third periodicity \cite{mos_dst}. 

Case 2: $\kappa = \bolddelta{3} = u^{(3)}_1$. 

Suppose $A \subseteq \bolddelta{3}$ is $\Sigma^1_5$ and $\sup A < \bolddelta{3}$. Let $B$ be $\Pi^1_4$ such that $w \in \WO\ooo \wedge \sharpcode{w} \in A$ iff $\exists y~ (w,y) \in B$. 
Consider the game in which I produces $v$, II produces $(w,y)$. II wins either $v \notin \WO^{(3)}$ or  $v,w \in \WO^{(3)}  \wedge \wocode{v} < \wocode{w}  \wedge (w,y) \in B $. I has a winning strategy, and so has a $\Delta^1_5$ winning strategy $\tau$ by Moschovakis third periodicity. By boundedness, $\set{\wocode{\tau*x}}{x \in \mathbb{R}}$ has a $\Delta^1_3(\tau)$ bound, hence has a $\Delta^1_5$ bound. 

Case 3:  $\kappa = \seed^{R^{\infty}}_{\mathbf{A}}> \bolddelta{3}$, $\mathbf{A} \in \exexdesc(R^{\infty})$. 

Put $\mathbf{A} = (\mathbf{r}, \pi, T)$, $\mathbf{r} = (r,Q,\overrightarrow{(d,q,P)})$, $r = (\xi_i)_{1 \leq i \leq k}$. Then $r$ is of discontinuous type and $\xi_k$ is a successor ordinal. Put $r' = (\xi_i)_{1 \leq i < k} \concat (\xi_k-1)$, $\mathbf{r}' = (r',Q,\overrightarrow{(d,q,P)})$, $\mathbf{A}' = (\mathbf{r}', \pi,T)$. 

Consider the game in which I produces  $\corner{\gcode{\tau}, a^{*}}$, II produces $(\corner{\gcode{\sigma},b^{*}}, y)$. II wins iff
\begin{enumerate}
\item If  $\corner{\gcode{\tau}, a^{*}}$ is a putative level-3 sharp code for an increasing function, then so is $\corner{\gcode{\sigma},b^{*}}$. Moreover, for any $\eta < \bolddelta{3}$, if
  \begin{displaymath}
    a^{*} \text{ is $\eta$-iterable }\wedge \tau^{\mathcal{M}_{a^{*},\eta}}(\eta) \in \wfp(\mathcal{M}_{a^{*},\eta})
  \end{displaymath}
then
\begin{displaymath}
      b^{*} \text{ is $\eta$-iterable }\wedge \sigma^{\mathcal{M}_{b^{*},\eta}}(\eta) \in \wfp(\mathcal{M}_{b^{*},\eta}) 
\wedge  \tau^{\mathcal{M}_{a^{*},\eta}}(\eta) <\sigma^{\mathcal{M}_{b^{*},\eta}}(\eta) .
\end{displaymath}
\item If $\corner{\gcode{\tau}, a^{*}} $ is a true level-3 sharp code for an increasing function, $a^{*} = a^{3\#}$,   
  then $( \corner{ \gcode{\sigma ^{\underline{j^T}(V)}(\underline{c_{\mathbf{A}'}})}, b^{*} } , y)  \in B$. 
\end{enumerate}
This game is $\Sigma^1_4$ for Player I. Player I has a winning strategy, and so has a $\Delta^1_5$ winning strategy $f$. Let $\sigma $ be the $\mathcal{L}^{\underline{x},R^{\infty}}$-Skolem term for $c^{(3)}_{y, T, \underline{\pi^T} (\underline{c_{r'}})+\omega}$ where $\underline{x}=y^{3\#}$. 
Let
\begin{displaymath}
 w= \corner{\gcode{\sigma}, (\tau^{3\#})^{3\#}}
\end{displaymath}
 So $w \in \WO\ooo$ is $\Delta^1_5$ and $\sharpcode{w} < \seed^R_{\mathbf{A}}$.
 We show that $\sup A < \sharpcode{w}$ using a boundedness argument. For each $\eta < \bolddelta{3}$, Let $Z_{\eta}$ be the set of $r \in \LO^{(3)}$ such that there are putative level-3 sharp codes for increasing function on ordinals $\corner{\gcode{\tau}, a^{*}}$, $\corner{\gcode{\sigma}, b^{*}}$ and an ordinal $\beta \leq \eta$ such that
 \begin{enumerate}
 \item $\corner{ \gcode{\tau}, a^{*} } = f * \corner{ \gcode{\sigma}, b^{*}}$;
 \item for any $\bar{\beta}<\beta$, $b^{*}$ is $\bar{\beta}$-iterable, $\sigma^{M_{b^{*},\eta}(\vec{\beta})} (\bar{\beta}) \in \wfp(\mathcal{M}_{b^{*},\eta})$, $\sigma^{\mathcal{M}_{b^{*},\eta}}(\bar{\beta}) \leq \eta$;
 \item $a^{*}$ is $\beta$-iterable,
 \item the order type of $\tau^{\mathcal{M}_{a^{*},\eta}}(\beta)$ is coded in $X_r$.
 \end{enumerate}
$Z_{\eta}$ is a $\boldsigma{3}$ set in the code of $\eta$. Since $f$ is a winning strategy for I, $Z_{\eta} \subseteq \WO^{(3)}$. By $\boldsigma{3}$-boundedness in \cite[Corollary 5.1]{sharpII}, $\set{ \wocode{r}}{r \in Z_{\eta}}$ is bounded by $c^{(3)}_{f,\eta + \omega}$.
Hence,  if $\corner{\gcode{\sigma}, b^{*}}$ is a true level-3 sharp code for an increasing function $g$ and $\corner{ \gcode{\tau}, a^{*} } = f * \corner{ \gcode{\sigma}, b^{*}}$, then $\corner{ \gcode{\tau}, a^{*} } $ is a true level-3 sharp code for an increasing function $h$ and 
for any $\eta < \bolddelta{3}$ such that $g''\eta \subseteq \eta$, $h(\eta) < c^{(3)}_{f,\eta+\omega}$. Let $\eta \in C$ iff $g''\eta \subseteq \eta$. 
By Lemma~\ref{lem:level_2_correct}, for any $\gamma \in j^T(C)$,  $j^T(h) ( \gamma) < c^{(3)}_{f, T, \gamma+\omega} $.
Hence, $\sup A < \sharpcode{w}$.
\end{proof}

Based on Lemma~\ref{lem:KM_bound}, the proof of the following theorem is completely in parallel to the level-2 Kechris-Martin theorem in \cite{Kechris_Martin_II} or \cite{km_neeman}. 
It is proved by induction on the $\mathbb{L}[j^{R^{\infty}}(T_3)]$-cardinality of $\sup(A)$. A key step uses the following observation: by Lemma~\ref{lem:next-level-3-indis},  if  $\sharpcode{w} < \seed^{R^{\infty}}_{\mathbf{A}}$ then there is a $\Delta^1_5(w)$ surjection from the $\mathbb{L}[j^{R^{\infty}}(T_3)]$-predecessor of $\seed_{\mathbf{A}}^{R^{\infty}}$ onto $\sharpcode{w}$.

\begin{mytheorem}
  \label{thm:KM_basis}
Assume $\boldDelta{4}$-determinacy. If $A$ is a nonempty $\Pi^1_5(x)$ subset of $u\ooo$, then $\exists w \in \Delta^1_5(x) ~ (\sharpcode{w} \in A)$. So the pointclass $\Pi^1_5$ is closed under quantification over $u^{(3)}_{\omega^{\omega^{\omega}}}$.
\end{mytheorem}







\begin{mydefinition}\label{def:kappa_5}
  $\kappa_5^x$ is the least $(T_4,x)$-admissible ordinal.
\end{mydefinition}

Using Theorem~\ref{thm:KM_basis}, we obtain the level-4 version of Becker-Kechris-Martin. The proof is parallel to \cite{Kechris_Martin_II} and \cite{becker_kechris_1984}, using Moschovakis set induction in one direction and the Becker-Kechris game in the other direction. 

\begin{mytheorem}
  \label{thm:level-4_KM}
    Assume $\boldDelta{4}$-determinacy. Then for each $A \subseteq u\ooo \times \mathbb{R}$, the following are equivalent.
  \begin{enumerate}
  \item $A$ is $\Pi^1_5$.
  \item There is a $\Sigma_1$ formula $\varphi$ such that $(\alpha,x) \in A$ iff $\admisfour{x}\models \varphi(T_4,\alpha,x)$.
  \end{enumerate}
\end{mytheorem}

The ordinal $\kappa_5^x$ is defined in a different way in \cite{Q_theory}:
\begin{align*}
  \lambda_5^x & = \sup \set{ \absvalue{W} }{W \text{ is a } \Delta^1_5(x) \text{ prewellordering on } \mathbb{R}} ,\\
    \kappa_5^x &= \sup \set{\lambda_5^{x,y}}{ M_3^{\#}(x) \nleq_{\Delta^1_5} (x,y)}.
\end{align*}
In parallel to \cite{Kechris_Martin_II}, these two definitions are equivalent, and in fact, 
\begin{align*}
  \lambda_5^x &=  \sup \set{ \xi < \kappa_5^x}{ \xi \text{ is } \Delta_1\text{-definable over } {L_{\kappa_5^x}[T_4,x]} \text{ from } \se{T_4,x}},\\
    \kappa_5^x &= \sup \set{\ot(W)}{W \text{ is a $\Delta^1_5(x,\lessthanshort{u\ooo})$ prewellordering on $\mathbb{R}$}}.
\end{align*}
Moreover,
\begin{displaymath}
  \forall \alpha <u\ooo ~\exists w \in \WO\ooo ~(\sharpcode{w} = \alpha \wedge \lambda_5^{x,w} < \kappa_5^x) .
\end{displaymath}
\subsection{The equivalence of $x^{4\#}$ and $M_3^{\#}(x)$}
\label{sec:level-4-equivalence}

Suppose $x$ is a real and $\beta \leq u\ooo$. 
A subset $A \subseteq \mathbb{R}$ is $\beta$-$\Pi^1_5(x)$ iff there is a $\Pi^1_5(x)$ set $B \subseteq u\ooo \times \mathbb{R}$ such that $A = \Diff B$. 
$\beta$-$\Pi^1_5(x)$ acting on product spaces of $\omega$ and $\mathbb{R}$ is defined in the obvious way. Lightface $\beta$-$\Pi^1_5$ and boldface $\beta$-$\boldpi{5}$ have the obvious meanings. 

In parallel to \cite{sharpI}, we have
\begin{mylemma}
  \label{lem:diff_Pi_1_5_to_game}
  Assume $\boldDelta{4}$-determinacy. Suppose $ \xi < \omega^{\omega^{\omega}}$ and $m<\omega$. If $A$ is $(u^{(3)}_{\xi+1})^m$-$\Pi^1_5(x)$, then $A$ is $\game^2((\widehat{\xi+1})\text{-}\Pi^1_3(x))$. 
\end{mylemma}

If $S$ is a finite regular level-3 tree, let $S^{+}$ be the level-3 tree extending $S$ where $\dom(S^{+}) = \dom(S) \cup \se{((1))}$ and $\cf(S((1))) = 0$. Thus, $\llbracket \emptyset \rrbracket_{S^{+}}= \llbracket \emptyset \rrbracket_S + \omega$. If $\vec{\xi} = (\xi_s)_{s \in \dom(S^{+})}$ respects $S^{+}$, put $\vec{\xi}^{-} = (\xi_s)_{s \in \dom(S)}$. 

If $x \in \mathbb{R}$ and $\alpha < \bolddelta{3}$, let $\mathcal{N}_{\alpha,\infty}(x) = \mathcal{P}_{\infty}$ where $\wocode{\mathcal{P}}_{<_{DJ(x)}} = \alpha $. In particular, $\mathcal{N}_{c^{(3)}_{x,\alpha}, \infty}(x) = \mathcal{M}^{-}_{2,\infty}(x) | c^{(3)}_{x,\alpha}$.

\begin{mylemma}
  \label{lem:game-to-diff-Pi-1-5}
Assume $\boldDelta{4}$-determinacy. 
Let $\xi<\omega^{\omega^{\omega}}$. If $A$ is $\game^2(\widehat{\xi} \textnormal{-}\Pi^1_3(x))$, then $A$ is $u^{(3)}_{\xi+2}\textnormal{-}\Pi^1_5(x)$.
\end{mylemma}
\begin{proof}
  Let $S$ be a regular level-3 tree such that $\llbracket \emptyset \rrbracket_S= \widehat{\xi}$. 
By \cite{sharpII}, if $(y,r) \in \mathbb{R}^2$, $C \subseteq \mathbb{R}$ is $\widehat{\xi}$-$\Pi^1_3(y,r)$, then we can effectively find a formula $\varphi$ such that Player I has a winning strategy in $G(C)$ iff
  \begin{displaymath}
     \gcode{\varphi(y,r)} \in (y,r)^{3\#}(S). 
  \end{displaymath}
Suppose $A = \game B$, where $B \subseteq \mathbb{R}^2$ is $\game(\widehat{\xi}\text{-}\Pi^1_3)$. Suppose $\varphi$ is an $\mathcal{L}$-formula such that
\begin{displaymath}
  (y,r) \in B \eqiv      \gcode{\varphi(y,r, (\underline{c_s})_{s \in \dom(S)})} \in (y,r)^{3\#}(S). 
\end{displaymath}
For ordinals $\vec{\xi}$ respecting $S^{+}$, say that $M$ is a Kechris-Woodin non-determined set with respect to $(y, \vec{\xi})$ iff 
\begin{enumerate}
\item $M$ is a countable subset of $\mathbb{R}$.
\item $M$ is closed under join and Turing reducibility.
\item $\forall \sigma \in M ~ \exists v \in M ~ \mathcal{N}_{\xi_{((1))},\infty}(y,\sigma\otimes v) \models \neg\varphi(y,\sigma\otimes  v,\vec{\xi}^{-})$.
\item $\forall \sigma \in M ~ \exists v \in M ~ \mathcal{N}_{\xi_{((1))},\infty}(y,v\otimes \sigma)  \models \varphi(y, v\otimes \sigma,\vec{\xi}^{-})$.
\end{enumerate}
Say that $z$ is $(y, \vec{\xi})$-stable iff  $z$ is not contained in any Kechris-Woodin non-determined set with respect to $(y,\vec{\xi})$.  $z$ is $y$-stable iff $z$ is $(y,\vec{\xi})$-stable for all $\vec{\xi}$ respecting $S^{+}$. The set of $(y,z)$ such that $z$ is $y$-stable is $\Pi^1_4$. By the proof of Kechris-Woodin \cite{KW_det_transfer}, for all $y\in \mathbb{R}$, there is $z \in \mathbb{R}$ which is $y$-stable.
Let $<^{\vec{\xi}}_y$ be the following wellfounded relation on the set of $z$ which is $(y,\vec{\xi})$-stable: 
\begin{align*}
  z' <^{\vec{\xi}} _{y} z \eqiv~&  z \text{ is $(y,\vec{\xi})$-stable}  \wedge z \leq_T z' \wedge \\
&  \forall \sigma \leq_T z~ \exists v \leq_T z' ~\mathcal{N}_{\xi_{((1))},\infty}(y,\sigma\otimes v) \models \neg\varphi(y,\sigma\otimes  v,\vec{\xi}^{-})\\
&  \forall \sigma \leq_T z~ \exists v \leq_T z' ~\mathcal{N}_{\xi_{((1))},\infty}(y,v\otimes \sigma) \models \varphi(y,v\otimes \sigma,\vec{\xi}^{-}).
\end{align*}
 If $z$ is $y$-stable,  let $f^z_y$ be the function that sends $\vec{\xi}$ to the rank of $z$ in $<^{\vec{\xi}}_{y}$. Then $f^z_y$ is a function into $\bolddelta{3}$. By $\Sigma^1_4$-absoluteness between $V$ and $\mathcal{N}^{\coll(\omega,\eta)}$, where $\mathcal{N} \in \mathcal{F}_{2,\infty}(y,z)$ and $\pi_{\mathcal{N},\infty}(\eta) = \xi_{((1))}$,  we can see $f^z_y\res \{\vec{\xi} \in [\bolddelta{3}]^{S \uparrow}:\xi_{((1))}$ is a cardinal cutpoint of $M_{2,\infty}^{-}(y,z)\}$ is definable over $M_{2,\infty}^{-}(y,z)$ in a uniform way, so there is a $\mathcal{L}^S$-Skolem term $\tau$ such that for all $(y,z) \in \mathbb{R}^2$, if $z$ is $y$-stable, $\vec{\xi}$ respects $S$, $\xi_{((1))}$ is a cardinal cutpoint of $M_{2,\infty}^{-}(y,z)$,  then
 \begin{displaymath}
f_y^z(\vec{\xi}) = \tau^{M_{2,\infty}^{-}(y,z)}(y,z,\vec{\xi}).
\end{displaymath}
Let
\begin{displaymath}
  \beta^z_y = \tau^{j^{S^{+}}(M_{2,\infty}^{-}(y,z))}(y,z, \seed^{R^{+}}).
\end{displaymath}
The function
\begin{displaymath}
  (y,z) \mapsto \beta^z_y
\end{displaymath}
 is $\Delta^1_5$ in the level-3 sharp codes. The rest is in parallel to the proof of \cite[Lemma 3.2]{sharpI}. 
\end{proof}

Lemma~\ref{lem:diff_Pi_1_5_to_game} and Lemma~\ref{lem:game-to-diff-Pi-1-5} are concluded in a simple equality between pointclasses.
\begin{mytheorem}
  \label{thm:pointclass_game_exchange_5}
  Assume $\boldDelta{4}$-determinacy. Then for $x \in \mathbb{R}$, 
  \begin{displaymath}
\game^2 ( <\! u_{\omega} \text{-}\Pi^1_3(x)) = \:<\!u\ooo \text{-}\Pi^1_5(x).
\end{displaymath}
Hence by \cite[Theorem 3.2]{sharpI},
\begin{displaymath}
  \game^4 ( <\! \omega^2 \text{-}\Pi^1_1(x)) = \:<\!u\ooo \text{-}\Pi^1_5(x).
\end{displaymath}
\end{mytheorem}

The level-4 sharp is defined in parallel to \cite{sharpI}. 
\begin{mydefinition}\label{def:OT4x}
  \begin{displaymath}
    \mathcal{O}^{T_4,x} = \{ (\gcode{\varphi},\alpha) : \varphi \text{ is a $\Sigma_1$-formula}, \alpha<u\ooo, \admisfour{x} \models \varphi(T_4,x,\alpha).\}
  \end{displaymath}
\end{mydefinition}


\begin{mydefinition}
\begin{displaymath}
  x^{4\#}_{\xi} = \{(\gcode{\varphi},\gcode{\psi}) :  \exists \alpha<u^{(3)}_{\xi} ( (\gcode{\varphi},\alpha) \notin \mathcal{O}^{T_4,x}\wedge \forall \eta<\alpha (\gcode{\psi},\eta) \in \mathcal{O}^{T_4,x} )\}.
\end{displaymath}
\begin{displaymath}
  x^{4\#} = \{(n,\gcode{\varphi},\gcode{\psi}) :  n<\omega \wedge (\gcode{\varphi},\gcode{\psi}) \in x^{4\#}_{\omega^{\omega^n}}\}.
\end{displaymath}
\end{mydefinition}

Applying Theorem~\ref{thm:pointclass_game_exchange_5}  to the space $ \omega$, in combination with Theorem~\ref{thm:level-4_KM}, we reach the equivalence between $x^{4\#}$ and $M_3^{\#}(x)$. 

\begin{mytheorem}\label{thm:4sharp-equivalent-to-M3sharp}
 Assume $\boldDelta{4}$-determinacy.  Then $x^{4\#}$ is many-one equivalent to $M_3^{\#}(x)$, the many-one reductions being independent of $x$.
\end{mytheorem}

\section{Level $\leq 2$ indiscernibles for $M_{2,\infty}^{-}$}
\label{sec:coding-subsets}

In order to proceed with the induction into higher levels of the projective hierarchy, we will need a $\Delta^1_5$ coding of subsets of $u\ooo$. For this, we need to analyze all the $\mathbb{L}_{\bolddelta{5}}[T_4]$-measures on $u\ooo$.  The analysis of  level-3 indiscernibles for $M_{2,\infty}^{-}$ \emph{above} $u_{\omega}$ in Section~\ref{sec:level-3-indisc} will not be good enough for this purpose. We establish the level $\leq 2$ indiscernibles \emph{below} $u_{\omega}$ in this section. The analysis of $\mathbb{L}_{\bolddelta{5}}[T_4]$-measures on $u\ooo$ and $\Delta^1_5$ coding of subsets of $u\ooo$ will be in the next part of this series. 







In the rest of this section, we set
\begin{displaymath}
  \mathcal{N} = \mathcal{M}^{*}_{0^{3\#}, R^0}.
\end{displaymath}
Therefore, $\mathcal{N}_{\infty} = M_{2,\infty}^{-} | c^{(3)}_{\omega}$. We would like to investigate how ordinals below $u_{\omega}$ are generated in $\mathcal{N}_{\infty}$. 
Define: 
  \begin{enumerate}
\item If $T,Q$ are level $\leq 2$ trees, then $\mathcal{N}^{(T,Q)} = (\mathcal{N}^T)^Q$, $j^{(T,Q)}_{\mathcal{N}} = j^Q_{\mathcal{N}^T}$.  So $j^{(T,Q)}_{\mathcal{N}}: \mathcal{N}^T \to \mathcal{N}^{(T,Q)}$ is elementary.
\item If $T$ is a finite level $\leq 2$ tree, $U$ is a finite level-1 tree,  $\mathbf{D} \in \desc(T,U)$, then $b^{T,U}_{\mathbf{D},\mathcal{N}} = (\underline{\seed^{T,U}_{\mathbf{D}}})^{\mathcal{N}}$; if $T$ is a subtree of $T'$, then $b^{T',U}_{\mathbf{D},\mathcal{N}} = j^{T,T'}_{\mathcal{N}} ( b^{T,U}_{\mathbf{D},\mathcal{N}})$, which is an element of $\mathcal{N}^{T'}$.
\item If $T$ is a level $\leq 2$ tree,
  \begin{enumerate}
  \item if $(1,t) \in \dom(T)$, then $b^{T}_{(1,t),\mathcal{N}} = b^{T,\emptyset}_{(1, t, \emptyset), \mathcal{N}}$;
  \item if $(2,t) \in \dom(T)$ and $\comp{2}{T}[t] = (S,\vec{s})$, then $b^{T}_{(2,t), \mathcal{N}} = b^{T,S}_{(2, (t,S,\vec{s}), \id_S), \mathcal{N}}$;
 \end{enumerate}
\item If $T$ is a $\Pi^1_2$-wellfounded level $\leq 2$ tree and $(d,t) \in \dom(T)$, then $      b^T_{(d,t)}   = \pi_{\mathcal{N}^T, \infty} (b^T_{(d,t), \mathcal{N}}).$
\end{enumerate}

We claim that the value of $b^T_{(d,t)}$  depends only on $\llbracket d,t \rrbracket_T$, shown as follows. Note that $\llbracket d,t \rrbracket_T < \omega_1$ iff $d=1$. 
  The $d=1$ case is a simple variant of the $d=2$ case. 
 Let us assume $d=2$. Suppose $T'$ is another $\Pi^1_2$-wellfounded level $\leq 2$ tree and $\llbracket 2,t \rrbracket_T = \llbracket 2,t' \rrbracket_{T'}$. By Theorem~\ref{thm:factor_ordertype_embed_equivalent_lv2}, we can find a $\Pi^1_2$-wellfounded level tree $Q$ and a map $\phi$ minimally factoring $(T,T' \otimes Q)$ such that $\phi(2,t) = (2, \id_{T',*} (2,t'))$. The map $\phi^{T',Q} _{\mathcal{N}}: \mathcal{N}^T \to \mathcal{N}^{(T,Q)}$ is elementary. We have  argued in Lemma~\ref{lem:iterable_level_2_invariance_absolute_in_V} that $j^{(T',Q)}_{\mathcal{N}} : \mathcal{N}^{T'} \to \mathcal{N}^{(T',Q)}$ is essentially an iteration map, sending $b^{T'}_{(2,t'),\mathcal{N}}$ to $\phi^{T',Q}_{\mathcal{N}} (b^T_{(2,t), \mathcal{N}})$. By Dodd-Jensen, $b^T_{(2,t)} \leq b^{T'}_{(2,t)}$, and so $b^T_{(2,t)} = b^{T'}_{(2,t)}$ by symmetry. 

We call an ordinal $\beta$ \emph{level $\leq 2$ respecting} iff either $\beta$ is a countable limit ordinal or $\beta$ respects some partial level $\leq 1$ tower. 
We are thus safe to define
\begin{displaymath}
      b_{\beta}   = b^T_{(d,t)}.
\end{displaymath}
for $\beta = \llbracket d,t \rrbracket_T$ level $\leq 2$ respecting.  A similar argument shows that the map $\beta \mapsto b_{\beta}$ is order preserving. 
We let
\begin{displaymath}
  I^{(\leq 2)} = \set{b_{\beta}}{ \beta \text{ is level $\leq 2$ respecting}}
\end{displaymath}
be the set of level $\leq 2$ indiscernibles for $M_{2,\infty}^{-}$. 
By the level $\leq 2$ correctness of $0^{3\#}$, we have for a finite level-1 tree $P$ and $\beta = [f]_{\mu^P}<j^P(\omega_1)$,
\begin{displaymath}
   b_{\beta} =  [\vec{\alpha} \mapsto b_{f(\alpha)}]_{\mu^P}.  
\end{displaymath}
By remarkability of $0^{3\#}$, the function $\beta \mapsto b_{\beta}$ is continuous. Therefore, there is a club $E \subseteq \omega_1$ such that $E \in \mathbb{L}$ and $b_{\beta}$ is the $\beta$-th element of $j^P(E)$ when $P$ is a finite level-1 tree such that $\beta < j^P(\omega_1)$.


\begin{remark}
\label{rem:I2}
In fact, $E$ is just the critical sequence of the $\omega_1$-iteration of $M_2^{\#}$ at its bottom measurable cardinal.
\end{remark}

The level $\leq 3$ indiscernibles for $M_{2,\infty}^{-}$ is 
\begin{displaymath}
I^{(\leq 3)} =   I^{(\leq 2)} \cup \se{u_{\omega}} \cup I^{(3)}.
\end{displaymath}
$I^{(\leq 3)}$ is a club in $\bolddelta{3}$. 
Recall that for $\gamma_0,\dots,\gamma_n,\gamma_0',\dots \gamma_n' \in I^{(3)}$, 
 $\vec{\gamma}$ is said to be a shift of $\vec{\gamma}'$ iff there exist a level-3 tree $R$, nodes $r_0,\dots,r_n \in \dom(R)$, $\vec{\delta}, \vec{\delta}' $ both respecting $R$ such that $\gamma_i = c^{(3)}_{R_{\tree}(r_i), \delta_i}$, $\gamma_i '= c^{(3)}_{R_{\tree}(r_i), \delta'_i}$ for any $i \leq n$. Now for $\beta_0,\dots,\beta_m,\beta_0', \dots, \beta_m' \in I^{(\leq 2)}$, we say that $\vec{\beta}$ is a \emph{shift} of $\vec{\beta}'$ iff there is a level $\leq 2$ tree $Q$, nodes $(d_0,q_0),\dots,(d_m,q_m) \in \dom(Q)$, $\vec{\delta},\vec{\delta}'$ both respecting $Q$ such that 
$\beta_i = b_{\delta_i}$, $\beta_i' = b_{\delta_i' }$ for any $i \leq m$. 
 \begin{mylemma}[Level $\leq 3$ indiscernability]
   \label{lem:level_leq_3_indis}
   Suppose $\vec{\gamma} , \vec{\gamma}' \in I^{(3)}$, $\vec{\beta}, \vec{\beta}' \in I^{(\leq 2)}$, and $\vec{\gamma}$ is a shift of $\vec{\gamma}'$, $\vec{\beta}$ is a shift of $\vec{\beta}'$. Then for any $\mathcal{L}$-formula $\varphi$, 
   \begin{displaymath}
     M_{2,\infty}^{-}  \models \varphi(\vec{\gamma}, \vec{\beta}) \eqiv \varphi (\vec{\gamma}', \vec{\beta}').
   \end{displaymath}
 \end{mylemma}
 \begin{proof}
   Suppose without loss of generality that $\vec{\gamma} = (\gamma_r)_{r \in \dom(R)}$ respects a finite level-3 tree $R$ and $\vec{\beta}= (\beta_{(d,q)})_{(d,q) \in \dom(Q)}$ respects a finite level $\leq 2$ tree $Q$. Then
   \begin{displaymath}
     M_{2,\infty}^{-} \models \varphi (\vec{\gamma}, \vec{\beta})
    \end{displaymath}
iff $0^{3\#}(R)$ contains the formula
\begin{displaymath}
  \underline{j^Q(V)} \models \varphi ( (\underline{c_r})_{r \in \dom(R)}, (\underline{\seed^Q_{(d,q)}})_{(d,q) \in \dom(Q)}).
\end{displaymath}
 \end{proof}

 \begin{mylemma}
   \label{lem:generate_below_level_2_indis}
   Suppose that $\delta$ is  level $\leq 2$ respecting and $\alpha < b_{\delta}$. Then there is a level $\leq 3$ code $(R,\vec{\gamma}, X, \vec{\beta}, \gcode{\sigma})$ such that 
   \begin{displaymath}
     \sharpcode{(R,\vec{\gamma}, X, \vec{\beta}, \gcode{\sigma})} = \alpha
   \end{displaymath}
and each entry of $\vec{\beta}$ is $\leq \delta$. 
 \end{mylemma}
 \begin{proof}
Let $T$ be a $\Pi^1_2$-wellfounded level $\leq 2$ tree and $(d,t ) \in \dom(T)$ such that  $\llbracket d,t \rrbracket_T = \delta$. Then $b_{\delta} = \pi_{\mathcal{N}^T,\infty} ( b_{(d,t),\mathcal{N}}^T)$. Pick a $\Pi^1_2$-wellfounded $Q$ and $\xi < j^{(T,Q)}_{\mathcal{N}} (b^T_{(d,t),\mathcal{N}})$ such that $\pi_{\mathcal{N}^{(T,Q)}, \infty} (\xi) = \alpha$. Pick a finite tree $X$, a map $\pi$ factoring $(X, T \otimes Q)$, a node $(d_0,x_0) \in \dom(X)$ and $\bar{\xi}   \in \mathcal{N}^X$ such that $ \pi^{T \otimes Q}_{\mathcal{N}} (\bar{\xi})  =  \xi$ and  $\pi(d_0,x_0) = \id_{T,*}(d,t) $. Pick a finite level-3 tree $R$ and an $\mathcal{L}^R$-Skolem term $\sigma$ such that
\begin{displaymath}
  \bar{\xi} = \sigma^{\mathcal{M}^{*}_{0^{3\#}, R}} (  (\underline{c_r})_{r \in \dom(R)} ).
\end{displaymath}
Then for any $\vec{\gamma}$ respecting $R$, 
   \begin{displaymath}
     \sharpcode{(R,\vec{\gamma}, X, \vec{\beta}, \gcode{\sigma})} = \alpha,
   \end{displaymath}
where $\vec{\beta} = (\llbracket \pi(d,x) \rrbracket_{T \otimes Q})_{(d,x) \in \dom(X)}$. 
If every entry of $\beta$ is $\leq\delta$, we are done. Otherwise, let $X'$ be the ``subtree'' of $X$ (modulo a rearrangement of the domain) whose domain consists of entries $(d,x)$ satisfying  $\llbracket \pi(d,x) \rrbracket_{T \otimes Q}\leq  \delta$. Equivalently, $(d,x) \in \dom(X)$ iff $\llbracket d,x \llbracket_X \leq \llbracket d_0,x_0 \llbracket_X $. Then the results in Section~\ref{sec:level-2-description} imply that $j^{X',X} _{\mathcal{N}}  \res  \seed_{(d_0,x_0)}^X$ is the identity. Then for any $\vec{\gamma}$ respecting $R$, 
   \begin{displaymath}
     \sharpcode{(R,\vec{\gamma}, X', \vec{\beta}', \gcode{\sigma})} = \alpha,
   \end{displaymath}
where $\vec{\beta} '= (\llbracket \pi(d,x) \rrbracket_{T \otimes Q})_{(d,x) \in \dom(X')}$. 
 \end{proof}

Lemma~\ref{lem:next-level-3-indis} has its analog for level $\leq 2$ indiscernibles. The proof is similar  but uses Lemma~\ref{lem:generate_below_level_2_indis}.
\begin{mylemma}
  \label{lem:next-level-3-indis-level-2}
Assume $\boldpi{3}$-determinacy. For any $\min I^{(\leq 2)} <\xi \in I^{(\leq 2)}$, there is an $\mathcal{L}$-Skolem term $\tau$ such that 
 $M_{2,\infty}^{-}(0^{3\#}) \models `` \tau(\sup(I^{(\leq 2)} \cap \xi) , \cdot)$ is a surjection from $ \sup (I^{(\leq 2)} \cap \xi)$ onto $\xi$''. For any $\xi < \min I^{(\leq 2)}$, there is an $\mathcal{L}$-Skolem term $\tau$ such that 
 $M_{2,\infty}^{-}(0^{3\#}) \models `` \tau( \cdot)$ is a surjection from $\omega$ onto $\xi$''. 
\end{mylemma}


\section*{Acknowledgements}
The breakthrough ideas of this paper were obtained during the AIM workshop on Descriptive inner model theory, held in Palo Alto, and the Conference on Descriptive Inner Model Theory, held in Berkeley, both in June, 2014. 
The author greatly benefited from conversations with Rachid Atmai and Steve Jackson that took place in these two conferences. 
The final phase of this paper was completed whilst the author was a visiting fellow at the Isaac Newton Institute for Mathematical Sciences in the programme `Mathematical, Foundational and Computational Aspects of the Higher Infinite' (HIF) in August and September, 2015 funded by NSF Career grant DMS-1352034 and EPSRC grant EP/K032208/1.

Part of the computations in this paper were carried out by the author's C++11 program at \url{https://github.com/zhuyizheng/Higher-sharp}.

\bibliography{sharp}{}
\bibliographystyle{plain}

\end{document}